\documentclass[11pt,a4paper]{article}
\usepackage[pdfusetitle]{hyperref}
\usepackage{mymacros}

\usepackage[titles]{tocloft}
\newcommand{\GFF}{\mathrm{GFF}}
\newcommand{\SG}{\mathrm{SG}}
\newcommand{\FF}{\mathrm{FF}}
\usepackage{slashed}
\newcommand{\Dirac}{\slashed{\partial}}
\newcommand{\wick}[1]{\mathopen{:}#1\mathclose{:}}
\newcommand{\LL}{\bm{L}}

\newcommand{\loc}{{\rm loc}}
\newcommand{\dnorm}[1]{|\!|\!|#1|\!|\!|}

\renewcommand{\H}{\mathcal H}
\newcommand{\1}{\mathbf 1}
\newcommand{\msmj}{\hat{m}}

\newcommand{\subproof}[1]{\smallskip\emph{Proof of #1.}}
\newcommand{\subproofx}[1]{\smallskip\emph{#1.}}

\usepackage{tikz}
\usetikzlibrary {arrows.meta}

\author{Roland Bauerschmidt\footnote{Courant Institute of Mathematical Sciences, NYU. E-mail: {\tt bauerschmidt@cims.nyu.edu}.}
  \and Scott Mason\footnote{Courant Institute of Mathematical Sciences, NYU. E-mail: {\tt sm12814@cims.nyu.edu}.}
  \and Christian Webb\footnote{University of Helsinki, Department of Mathematics and Statistics. E-mail: {\tt christian.webb@helsinki.fi}.}}

\title{Twisted Dirac operators and fractional correlations of \\ the massless sine-Gordon model at the free fermion point}
\date{\vspace*{-2em}}

\begin{document}
\maketitle
\begin{abstract}
  For the massless sine-Gordon model at the free fermion point, in infinite volume,
  we define the fractional (charge or vertex operator) correlation functions from the probabilistic path integral
  and prove that they are given by renormalized determinants of massive twisted Dirac operators.
  The fractional correlation functions are 
  the moments of the imaginary multiplicative chaos, a random generalized function that we construct with respect to the infinite-volume massless
  sine-Gordon measure.
  The renormalized determinants are the tau functions of Sato--Miwa--Jimbo as identified by Palmer.

  The construction and a priori control of the imaginary multiplicative chaos combines methods from stochastic analysis (of singular SPDE flavor) for short-scale regularity
  with qualitative input from integrability for large-scale control.
  The exact identification of the correlation functions with the renormalized determinants relies on finite-volume approximation, regularity estimates for the mass perturbation, and analytic continuation in the coupling constant.
  
  The combination of existing results for tau functions with our identification implies various predictions for the sine-Gordon model
  such as that
  the fractional two-point functions are expressed as Fredholm determinants and
  satisfy certain PDEs as predicted by Bernard--LeClair.
  Using asymptotics of Fredholm determinants of Basor--Tracy and mixing of the massless sine-Gordon
  model at the free fermion point, which we prove,
  we further derive the exact formula for the one-point function
  predicted by Lukyanov--Zamolodchikov (at the free fermion point).
\end{abstract}

\setcounter{tocdepth}{1}
\tableofcontents

\section{Introduction and main results}\label{sec:intro}

\subsection{Introduction}
\label{sec:intro-intro}

The massless sine-Gordon model is a fascinating topic in integrable systems. Its classical version is the sine-Gordon equation
$\partial_t^2\varphi -\partial_x^2\varphi = - 2\sin(\varphi)$, an integrable hyperbolic PDE,
with explicit soliton and breather solutions and many further interesting features \cite{MR1995460}.
The sine-Gordon QFT is expected to share the interesting behavior of the classical model,
along with further intricate properties such as its expected exact mapping  to the Massive Thirring Model \cite{PhysRevD.11.2088}.
It has connections to various apparently unrelated problems in two-dimensional statistical physics,
such as the Coulomb gas \cite{MR0434278,MR649810},
the Ising model \cite{MR4149524}, the dimer model \cite{2209.11111}, and massive SLE \cite{2203.15717},
and its instance at the free fermion point (the focus of this paper) even features in Polyakov's instanton heuristics for
the mass gap of the $O(3)$ nonlinear sigma model \cite[Section~6.1]{MR1122810}.

The probabilistic Euclidean field theory point of view is essential for these connections.
From this point of view, 
the massless sine-Gordon model at $\beta \in (0,8\pi)$ should be a probability measure on $\cS'(\R^2)$ 
formally given by the path integral
\begin{equation} \label{e:SG-pathintegral}
  \exp\qa{- 2\int_{\R^2}\partial\varphi \bar\partial\varphi \,dx + 2z\int_{\R^2}\wick{\cos(\sqrt{\beta}\varphi)} \, dx} \, d\varphi,
  \qquad \partial = \frac12(-i\partial_0 + \partial_1),
\end{equation}
where $z\in\R$, $z\neq 0$ is a coupling constant and $\wick{\cdot}$ denotes an infinite multiplicative renormalization,
see \eqref{e:SGdef} below for the precise definition.
The nonzero value of $z$ is unimportant for the massless model
and could be chosen to be $1$ by scaling -- as in the classical equation.
On the other hand, the value of $\beta$ is important and does not have a classical counterpart.

Due to the noncompact $\Z$-symmetry $\varphi \to \varphi+\frac{2\pi}{\sqrt{\beta}} n$, $n\in \Z$ of the sine-Gordon action,
the construction of this massless sine-Gordon model as a probability measure must rely on spontaneous breaking of this symmetry and is only
possible on the infinite space $\R^2$. This symmetry breaking and the expected associated  mass generation
are not understood mathematically for general values of $\beta$.
Mass generation refers to the phenomenon that the ``massless'' model (where massless refers to the absence of a \emph{bare} mass term 
$\frac12 m^2 \varphi^2$ in the action)
should have a strictly positive \emph{physical} mass gap (manifesting itself by exponential correlation decay
and other properties such as nonvanishing of the nonneutral charge correlation functions studied in the present paper).
In addition to these long distance difficulties,
for $\beta\geq 4\pi$ including $\beta=4\pi$,
ultraviolet singularities (the measure is not locally absolutely continuous relative to the free field) complicate the analysis.

For the special value $\beta=4\pi$, \emph{the free fermion point},
the massless infinite-volume sine-Gordon measure was constructed in \cite{MR4767492},
by relying on its description in terms of massive free fermions -- Bosonization, or Coleman correspondence \cite{PhysRevD.11.2088} --
the most fundamental instance of the predicted integrability of the model.
The expectation of this measure is denoted $\avg{\cdot}_{\SG(4\pi,z)}$,
see Section~\ref{sec:intro-SG} for the precise definition.
For $\beta=4\pi$, Bosonization is the statement that the bosonic path integral
\eqref{e:SG-pathintegral} is `equivalent' to a massive Dirac field, given by the fermionic path integral
\begin{equation}
  \label{e:FF-pathintegral}
  \exp\qa{-\int_{\R^2} \psi \Dirac \bar\psi \,dx - \mu \int_{\R^2}  (\psi_1\bar\psi_1+\psi_2\bar\psi_2) \, dx } %
  \, d_\psi d_{\bar\psi},
  \qquad
  \Dirac = \begin{pmatrix} 0 & 2\bar\partial \\ 2\partial & 0 \end{pmatrix},
\end{equation}
with the mass $\mu$ proportional to the coupling constant $z$, 
and the correspondence of correlation functions given, up to multiplicative constants, by
\begin{align}
  \label{e:corr1}
  \wick{e^{+i\sqrt{{4\pi}}\varphi}}
  &\quad\leftrightarrow\quad
    \bar\psi_1\psi_1
    =
    \frac{1}{2} {\bar\psi ({\bf 1}+\gamma^5) \psi},
  \\
  \label{e:corr2}
  \wick{e^{-i\sqrt{{4\pi}}\varphi}}
  &\quad\leftrightarrow\quad
      \bar\psi_2\psi_2
    =
    \frac{1}{2} {\bar\psi ({\bf 1}-\gamma^5) \psi},
  \\
  \label{e:corr3}
  -i \partial \varphi
  &\quad\leftrightarrow\quad
    \bar\psi_2\psi_1
    =
    \frac{1}{2} {\bar\psi (i\gamma^0+\gamma^1) \psi},
  \\
  \label{e:corr4}
  +i \bar\partial \varphi
  &\quad\leftrightarrow\quad
    \bar\psi_1\psi_2
    =
    \frac{1}{2} {\bar\psi (-i\gamma^0+\gamma^1) \psi},
\end{align}
where $\gamma^\mu$ are Euclidean $\gamma$-matrices, see \cite{MR4767492}  for further details on this and the rigorous implementation of the Coleman correspondence at $\beta=4\pi$.

Our goal is to construct and then compute the correlation functions of $\wick{e^{i\sqrt{4\pi}\alpha\varphi}}$ for $\alpha \in (-\frac12,\frac12)$
from the path integral.
The former exponentials are also known as vertex operators.
Since $\alpha \not\in\Z$, these fractional correlation functions are not covered by the Bosonization dictionary
\eqref{e:corr1}--\eqref{e:corr4}.
Instead, they turn out to be related to \emph{twisted} free fermions (also known as \emph{branched} fermions or fermions with \emph{winding}).
The winding will correspond to $\alpha$ in the fractional correlation functions.
Our main result establishes this  correspondence from the path integral.

Assuming the correspondence between fractional correlation functions and twisted fermions,
the fractional correlation functions were formally studied by Bernard and LeClair \cite{MR1297289,MR1462303}.
They argued that the fermionic correlations functions can be expressed as a Fredholm determinant and that they satisfy certain PDEs.
The Fredholm determinant appearing in the analysis of \cite{MR1297289} has also been studied rigorously and it is known as
the \emph{tau function} for the twisted Dirac operator (or the branched Dirac operator)
-- see e.g. \cite{MR1233355} and \cite{MR555666,MR566086}.

Again assuming the correspondence of fractional correlations and twisted fermions,
Lukyanov and Zamolodchikov formally computed the fractional one-point function %
exactly  \cite{MR1453266}:
\begin{equation}
  \langle\wick{e^{i\sqrt{4\pi}\alpha\varphi(0)}}\rangle_{\SG(4\pi,z)}
  = \left(\frac{\mu}{2}\right)^{\alpha^2}
  \exp\left(\int_0^\infty \frac{dt}{t}\left[\frac{\sinh^2(\alpha t)}{\sinh^2t}-\alpha^2e^{-2t}\right]\right),
  \qquad \mu \propto z.
\end{equation}
Inspired by this expression for $\beta=4\pi$ and similar looking asymptotics as $\beta\to 0$,
they further conjectured an explicit expression for the fractional one-point functions for all values of $\beta \in (0,8\pi)$,
see \eqref{e:LZ} below.
In Corollary~\ref{cor:LZ}, we prove the above Lukyanov--Zamolodchikov formula from the Euclidean path integral,
for all $\alpha\in(-\frac12,\frac12)$ when $\beta=4\pi$. More precisely, we define the correlation functions
as the moments of the \emph{imaginary multiplicative chaos}
(see Section~\ref{sec:intro-corr})
which we construct in infinite volume for the massless sine-Gordon model at $\beta=4\pi$.

Finally, we mention that the massless sine-Gordon model also has an equivalent representation as a Coulomb gas,
in which the correlation functions of the imaginary exponentials $\wick{e^{+i\sqrt{4\pi}\varphi}}$ and $\wick{e^{-i\sqrt{4\pi}\varphi}}$ correspond to correlations of $+$ and $-$ charges.
The exponentials  $\wick{e^{i\alpha \sqrt{4\pi}\varphi}}$ with $\alpha \not\in\Z$
correspond to noninteger charges in the Coulomb gas representation (and are therefore also called fractional).
This explains the terminology fractional charge correlation function.

\subsection{Twisted Dirac operators and main result}\label{sec:introferm}

As discussed above, our main results establish the
equivalence of fractional sine-Gordon correlation functions with those of twisted free fermions,
which we define next.

\begin{figure}
\centering
\begin{tikzpicture}
    \coordinate (A) at (0, 0);
    \coordinate (B) at (-1.5, 2.5);
    \coordinate (C) at (3, -1);
    \coordinate (D) at (4, 1.5);
 
    \draw[-,thick] (A) -- (10,0); 
    \draw[-,thick] (B) -- (10,2.5); 
    \draw[-,thick] (C) -- (10,-1); 
    \draw[-,thick] (D) -- (10,1.5); 

    \fill (A) circle (2pt);
    \fill (B) circle (2pt);
    \fill (C) circle (2pt);
    \fill (D) circle (2pt);

    \node[below] at (A) {$x_1$};
    \node[below] at (B) {$x_2$};
    \node[below] at (C) {$x_3$};
    \node[below] at (D) {$x_4$};

    \draw[-{Latex},thick,draw=red]([shift={(4:0.7)}]A) arc[start angle=4, end angle=356, radius=0.7] node[below right]{$\color{red}\alpha_1$};
    \draw[-{Latex},thick,draw=red]([shift={(4:0.7)}]B) arc[start angle=4, end angle=356, radius=0.7] node[below right]{$\color{red}\alpha_2$};
    \draw[-{Latex},thick,draw=red]([shift={(4:0.7)}]C) arc[start angle=4, end angle=356, radius=0.7] node[below right]{$\color{red}\alpha_3$};
    \draw[-{Latex},thick,draw=red]([shift={(4:0.7)}]D) arc[start angle=4, end angle=356, radius=0.7] node[below right]{$\color{red}\alpha_4$};
\end{tikzpicture}
  \caption{Illustration of the branch points $x_i$, the windings $\alpha_i$, and the branch cuts $\Gamma$.\label{fig:branches}}
\end{figure}
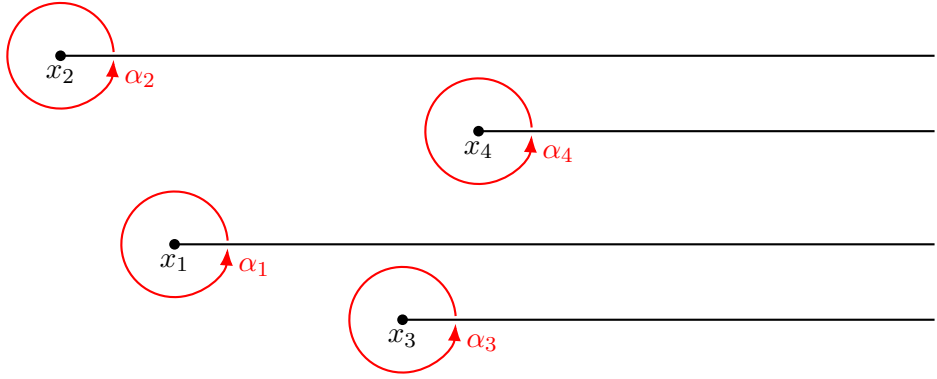

It is natural here (and in many other places) to identify $\R^2$ with $\C$.
Given points $x_1,\dots, x_n \in \C$ (which we refer to as branch or insertion points or punctures of $\C$
-- depending on the context)
and winding numbers $\alpha_1, \dots, \alpha_n \in \R$ (which are also referred to as the monodromy data),
the twisted Dirac operator acts on a complex line bundle over $\C\setminus\{x_1,\dots,x_n\}$  with flat connection 
or on complex-valued functions on the universal cover of $\C \setminus \{x_1,\dots,x_n\}$
with prescribed winding around the points $x_i$.
For our purposes, it is most convenient and concrete to fix branch cuts on $\C$
rather than to work directly with the universal cover.
The branching structure will be encoded in terms of a function $\rho: \C\setminus \Gamma \to \C$ defined by 
\begin{equation}\label{eq:rho}
\rho(z)=\prod_{j=1}^n (z-x_j)^{\alpha_j}=\prod_{j=1}^n e^{\alpha_j\log(z-x_j)},
\end{equation}
where $\Gamma \subset \C \cong \R^2$ is the union of branch cuts, see Figure~\ref{fig:branches}.
The exact choice of the branches for the logarithm is not important.
To connect with Palmer~\cite{MR1233355},
we later choose the branch cut of the logarithm to be on the positive real axis, and fix the branch by requiring the argument to be in $[0,2\pi)$ and assume that $\mathrm{Im}(x_i)\neq \mathrm{Im}(x_j)$ for $i\neq j$.
The twisted Dirac operator $\Dirac_\rho$ on $\C$ with this branching structure is formally given by
\begin{equation} \label{e:Dirac-rho}
  \Dirac_\rho 
  = \begin{pmatrix}
      0 & 2 \rho^{-1} \bar\partial \rho \\
      2 \bar\rho \partial  \bar\rho^{-1} & 0
    \end{pmatrix},
\end{equation}
see Sections~\ref{sec:massless-bosonization}--\ref{sec:Green} for details,
and can, again formally, also be interpreted as a Dirac operator with a singular gauge potential $\Dirac - i\slashed{A}$ supported along the
branch cuts.
As reviewed in Appendix~\ref{app:det}, it can be argued that
\begin{equation}
  \text{``} \frac{\det(\Dirac_\rho)}{\det(\Dirac)} \text{''}  \propto
   \prod_{1\leq r< s\leq n} |x_r-x_s|^{2\alpha_r\alpha_s},
\end{equation}
where $\Dirac$ denotes $\Dirac_\rho$ with $\rho=1$ (no twisting), and the left-hand side would require a suitable interpretation (which we however do not need).
The tau functions of Sato--Miwa--Jimbo \cite{MR566086} were interpreted by Palmer~\cite{MR1233355} as the corresponding renormalized determinants of the massive twisted Dirac operator $\Dirac_\rho+\mu$, i.e.,
\begin{equation}
  \tau_\rho(\mu) \propto \text{``} \frac{\det(\Dirac_\rho+\mu)}{\det(\Dirac+\mu)} \text{''}.
\end{equation}
The tau functions do not depend on the choice of branch cuts $\Gamma$ and are continuous in the distinct branch points $x_1,\dots,x_n$.
Thus despite our notation $\tau_\rho(\mu)$ in which their dependence on the branch points and windings is indicated through $\rho$,
the tau functions should be viewed as a function of the points $x_1,\dots, x_n$ and the windings $\alpha_1,\dots,\alpha_n$.

For $\mu\neq 0$, the function $\tau_\rho(\mu)$ is not explicit but can be seen as a special function itself and is related to
differential equations of Painlev\'e type -- see e.g.~the results of Sato--Miwa--Jimbo who introduced and studied the tau
functions from a related but slightly different perspective \cite{MR499666,MR533348,MR555666,MR566086,MR594916}.
To make our paper most accessible, we only use Palmer's treatment  \cite{MR1233355}
as our reference for tau functions
(except for relatively general regularity statements for solutions to the twisted Dirac equation).
Compared with Palmer's notation, we write $\mu$ instead of $-m$ (which instead refers
to an additional bare mass term regularization used in the construction of the sine-Gordon model),
$\alpha_i$ instead of $\lambda_i$, and $x_i$ instead of $a_i$,
see Section~\ref{sec:Palmer-notation} for a precise translation of the conventions.

\medskip

Our main result identifies the fractional correlation functions of the massless sine-Gordon model at the free fermion point with these tau functions.
The  massless sine-Gordon measure with expectation $\avg{\cdot}_{\SG(4\pi,z)}$ and its fractional correlation functions defined in terms of the imaginary multiplicative chaos $M_\alpha$
appearing in the statement are defined in detail
in Section~\ref{sec:intro-SG}--\ref{sec:intro-corr} below.
In essence, there are random variables $M_\alpha(f)$ with respect to $\avg{\cdot}_{\SG(4\pi,z)}$ defined by
\begin{equation} \label{e:Malpha-main}
  M_{\alpha}(f) = \lim_{\epsilon\to 0} \epsilon^{-\alpha^2} \int e^{i\sqrt{4\pi}\alpha (\eta_\epsilon * \varphi)(x)} f(x)\, dx, \qquad f\in C_c^\infty(\R^2),
\end{equation}
where $\eta \in C_c^\infty(\R^2)$ is a radial mollifier and $\eta_\epsilon(x)=\epsilon^{-2}\eta(x/\epsilon)$.
This definition is independent of the mollifier $\eta$ except for a multiplicative constant,
see Section~\ref{sec:intro-corr}.

\begin{theorem} \label{thm:main-tau}
  Let $\alpha_1,\dots,\alpha_n \in (-\frac12,\frac12)$ with $\sum_i \alpha_i = 0$, and let $x_1,\dots, x_n\in \C$.
  Then for all test functions $f_1,\dots,f_n\in C_c^\infty(\C)$ with disjoint supports,
  \begin{equation} \label{e:main-tau}
    \avg{\prod_{j=1}^n M_{\alpha_j}(f_j)}_{\SG(4\pi, z)}
    \propto
    \int dx_1\cdots dx_n\; f_1(x_1)\cdots f_n(x_n) \tau_\rho(\mu),
  \end{equation}
  where $z\in \R$ and $\mu =Az$ with $A>0$ a regularization dependent constant
  (which with our choice is $4\pi e^{-\gamma/2}$ where $\gamma$ is the Euler--Mascheroni constant).
\end{theorem}

In the statement of the theorem (and throughout the paper) $\propto$ denotes proportionality with a constant
that depends on the specific regularization used to define the correlation functions on the left-hand side
(and thus depends on the $\alpha_i$ and $z$);
see however Section~\ref{sec:intro-corr}
and also the discussion around \eqref{e:2pt-LZ-norm} for a canonical a posteriori normalization of the correlation functions.

\begin{remark}
  The neutrality assumption $\sum_i \alpha_i=0$ can be removed on the left-hand side by
  using the mixing property of the massless sine-Gordon measure which we establish in Section~\ref{sec:mixing}.
  This is illustrated in Corollary~\ref{cor:onepoint-twopoint} in the instance of the one-point function corresponding to $n=1$.

  Moreover, we expect it is possible to show the multiplicative chaos for the sine-Gordon measure
  is analytic in $\alpha$
  in a suitable complex domain (and in a suitable space); we postpone this question to future work,
  but see Remark~\ref{e:GMC-analyticity} for further discussion.
  This would allow to analytically extend the classical definition of the tau functions
  (as distributions),
  which is restricted to monodromies $\alpha_i\in(-\frac12,\frac12)$ in \cite{MR555666,MR1233355},
  to a wider range of possibly complex monodromies.
\end{remark}

\begin{remark}
  Part of the proof of Theorem~\ref{thm:main-tau} is a finite-volume version of the correspondence,
  involving finite-volume versions of the tau functions that we define in Section~\ref{sec:bosonization}.
  For the massless limit of the finite-volume sine-Gordon model, the neutrality condition $\sum_i \alpha_i=0$ is important, and
  the correct extension to the infinite-volume model requires the correct handling of the spontaneous symmetry
  breaking of the massless sine-Gordon model.
\end{remark}

\begin{remark}
  We expect our arguments could be extended to obtain mixed correlations involving combinations of
  fractional exponentials,
  integer ones, and derivatives of the massless sine-Gordon field, and
  more precisely, that the Bosonization dictionary \eqref{e:corr1}--\eqref{e:corr4}
  holds if bosonic correlations are evaluated using the massless sine-Gordon measure weighted by fractional exponentials
  (instances of the imaginary multiplicative chaos)
  and fermionic correlations are evaluated as massive twisted free fermions with propagator as in
  Proposition~\ref{pr:ivlim}.
\end{remark}

\subsection{Applications of Theorem~\ref{thm:main-tau}}

Next we demonstrate some applications of Theorem~\ref{thm:main-tau} in combination with existing results for tau functions.
We make the standing assumption that $z\in \R$ and $\mu = A|z|$ with the constant $A$ as in Theorem~\ref{thm:main-tau}.

In particular, we discuss the representation of the fractional correlation functions in terms of Fredholm and Hilbert--Schmidt determinants,
their short- and long-distance asymptotics, the computation of the fractional one-point function given by the formula of Lukyanov--Zamolodchikov,
and the PDEs for the fractional two-point function of Bernard--LeClair.

For distinct $x_1,\dots, x_n$,
the fractional correlation functions $\avg{\wick{e^{i\sqrt{4\pi}\alpha_i(x_1)}}\cdots \wick{e^{i\sqrt{4\pi}\alpha_n(x_n)}}}_{\SG(4\pi,z)}$ are defined such that
for $f_1,\dots,f_n\in C_c^\infty(\R^2)$ with disjoint supports,
\begin{equation} \label{e:corr-M}
  \avg{\prod_{j=1}^n M_{\alpha_j}(f_j)}_{\SG(4\pi, z)}
  =
  \int dx_1\cdots dx_n\; f_1(x_1)\cdots f_n(x_n) 
  \avg{\prod_{j=1}^n\wick{e^{i\sqrt{4\pi}\alpha_j \varphi(x_j)}}}_{\SG(4\pi, z)}.
\end{equation}
It follows from Theorem~\ref{thm:main-tau} and continuity of the tau functions that
they are continuous functions for distinct $x_1,\dots, x_n$.

\subsubsection{Determinantal formulas and Basor--Tracy asymptotics}

By combining Theorem~\ref{thm:main-tau} with Palmer's analysis of the tau functions \cite{MR1233355},
we obtain the representation of the fractional two-point function as a Fredholm determinant 
stated in Corollary~\ref{cor:2ptFredholm} below.

\begin{corollary} \label{cor:2ptFredholm}
For $\alpha \in (-\frac12,\frac12)$, the fractional two-point function has the following representation as a Fredholm determinant:
\begin{equation} \label{e:2ptFredholm}
\langle \wick{e^{i\sqrt{4\pi}\alpha\varphi(x)}} \wick{e^{-i\sqrt{4\pi}\alpha\varphi(y)}}\rangle_{\SG(4\pi,z)} \propto \det(1-K),
\end{equation}
where $K$ is an integral operator on $L^2(0,\infty)$ with kernel
\begin{equation}
K(a,b)=K_{x,y,\alpha}(a,b)=-\frac{\sin^2(\pi \alpha)}{\pi^2} \int_0^\infty du \frac{e^{-\frac{|x-y|}{2}\omega(a)-\frac{|x-y|}{2}\omega(b)-|x-y|\omega(u)} }{(a+u)(b+u)}\left(\frac{ab}{u^2}\right)^\alpha,
\end{equation}
where 
$\omega(a)=\frac{\mu}{2}(a+a^{-1})$.
\end{corollary}

\begin{proof}
  Using the identification of the fractional two-point function as a tau function from \eqref{e:main-tau},
  the result follows from \cite[p.~332]{MR1233355}.
  As translated in detail in Section~\ref{sec:Palmer-notation},
  our $\alpha$ and $\mu$ correspond to $\lambda_i$ and $m$ in that reference
  according to $\lambda_1=-\lambda_2=\alpha$ and $m = |\mu|$, 
  and $\omega$ is defined in \cite[p.~299]{MR1233355}.
  In particular, note the factors $2$ in the definition
  of the Dirac operator in \cite{MR1233355} are compensated by an additional factor $2$ in the definition
  of the Green's function, 
  and a factor $\frac12$ in front of the mass term, see \eqref{e:Palmer-Dirac} and \eqref{e:Palmer-green} in Section~\ref{sec:Palmer-notation}.
\end{proof}

The asymptotics of the Fredholm determinants \eqref{e:2ptFredholm} 
have actually been analyzed by Basor--Tracy~\cite{MR1187544}.
The implications of their results in our context are stated in Corollary~\ref{cor:BasorTracy} below.
Together with the mixing property of the massless sine-Gordon model which we establish, see Theorem~\ref{thm:SG4pi} and Section~\ref{sec:mixing},
Corollaries~\ref{cor:2pt-LZ}--\ref{cor:LZ} and the discussion following these then relate this to the
Lukyanov--Zamolodchikov formula for the fractional one-point correlation functions \cite{MR1453266}.

\begin{corollary} \label{cor:BasorTracy}
As $|x-y|\to \infty$ respectively $|x-y|\to 0$,
\begin{align}
  \det(1-K_{x,y,\alpha}) &\to 1 \qquad (|x-y|\to\infty)
  \\
  \det(1-K_{x,y,\alpha}) & \sim  \left(\frac{\mu}{2}|x-y|\right)^{-2\alpha^2}G(1+\alpha)^2G(1-\alpha)^2\qquad (|x-y|\to 0),
\end{align}
where $G$ is the Barnes $G$-function with the following integral representation for $|z|<1$:
\begin{equation}
  \log G(1+z)
=\frac{z}{2}\log(2\pi)+\int_0^\infty \frac{dt}{t}\left[\frac{1-e^{-2zt}}{4 \sinh^2 t}+\frac{z^2}{2}e^{-2t}-\frac{z}{2t}\right].
\end{equation}
\end{corollary}

\begin{proof}
  This follows from the results of Basor--Tracy \cite{MR1187544}, explained in detail in Section~\ref{sec:BasorTracy}.
\end{proof}

\begin{remark}
Generalizations of the formula \eqref{e:2ptFredholm}
for general higher point fractional correlation functions
exist in terms of Hilbert--Schmidt determinants, see \cite[p.~330]{MR1233355}.
\end{remark}

\subsubsection{Lukyanov--Zamolodchikov formula}

Corollaries~\ref{cor:2ptFredholm}--\ref{cor:BasorTracy}
imply that the fractional two-point function is asymptotic to a multiple of $|x-y|^{-2\alpha^2}$ as $|x-y|\to 0$.
As in \cite{MR1453266}, we may therefore normalize the correlation functions such that this proportionality constant is $1$, i.e., we can include an appropriate
multiplicative constant in the definition of $\wick{e^{i\sqrt{4\pi}\alpha \varphi}}$ so that
\begin{equation} \label{e:2pt-LZ-norm}
  \langle \wick{e^{i\sqrt{4\pi}\alpha\varphi(x)}} \wick{e^{-i\sqrt{4\pi}\alpha\varphi(y)}}\rangle_{\SG(4\pi,z)}
  \sim |x-y|^{-2\alpha^2} \qquad (|x-y|\to 0).
\end{equation}

\begin{corollary} \label{cor:2pt-LZ}
  Normalizing the correlation functions as in \eqref{e:2pt-LZ-norm},
  the fractional two-point function is given by
  \begin{align} \label{e:2pt-LZ}
    &\langle \wick{e^{i\sqrt{4\pi}\alpha\varphi(x)}} \wick{e^{-i\sqrt{4\pi}\alpha\varphi(y)}}\rangle_{\SG(4\pi,z)}\nnb
  &= \pa{\frac{\mu}{2}}^{2\alpha^2} \frac{1}{G(1+\alpha)^2G(1-\alpha)^2}
  \det(1-K_{x,y,\alpha})
  \nnb
  &=
  \qa{\left(\frac{\mu}{2}\right)^{\alpha^2}\exp\left(\int_0^\infty \frac{dt}{t}\left[\frac{\sinh^2(\alpha t)}{\sinh^2t}-\alpha^2e^{-2t}\right]\right)}^2
    \det(1-K_{x,y,\alpha}).
\end{align}
\end{corollary}

\begin{proof}
The first equality in \eqref{e:2pt-LZ} follows from Corollary~\ref{cor:BasorTracy}.
The second equality follows from the following identity for the Barnes function:
\begin{align}
-\log G(1+z)-\log G(1-z)&=\int_0^\infty \frac{dt}{t}\left[\frac{e^{-2zt}-2+e^{2z t}}{4\sinh^2 t}-z^2 e^{-2t}\right]\nnb
&=\int_0^\infty \frac{dt}{t}\left[\frac{\sinh^2(zt)}{\sinh^2 t}-z^2 e^{-2t}\right],
\end{align}
completing the proof.
\end{proof}

In Section~\ref{sec:intro-corr}, we show that the correlation functions of the massless sine-Gordon model at the free fermion point
factorize in the large distance limit. In particular, also using symmetry, as $|x-y|\to\infty$,
  \begin{equation}
    \avga{ \wick{e^{i\sqrt{4\pi}\alpha\varphi(x)}} \wick{e^{-i\sqrt{4\pi}\alpha\varphi(y)}} }_{\SG(4\pi,z)}
    \to
    \avga{ \wick{e^{i\sqrt{4\pi}\alpha\varphi(0)}}}_{\SG(4\pi,z)}^2 %
    ,
\end{equation}%
see Corollary~\ref{cor:onepoint-twopoint}.  
Since the determinant $\det(1-K_{x,y,\alpha})$ tends to $1$ as $|x-y|\to\infty$,
we therefore obtain
from Corollary~\ref{cor:2pt-LZ} that the one-point function is given as follows.

\begin{corollary} \label{cor:LZ}
Normalizing the correlation functions as in \eqref{e:2pt-LZ-norm}, the one-point function equals,
for $\alpha \in (-\frac12,\frac12)$,
\begin{equation} \label{e:LZ-cor}
  \langle\wick{e^{i\sqrt{4\pi}\alpha\varphi(0)}}\rangle_{\SG(4\pi,z)}
  = \left(\frac{\mu}{2}\right)^{\alpha^2}
  \exp\left(\int_0^\infty \frac{dt}{t}\left[\frac{\sinh^2(\alpha t)}{\sinh^2t}-\alpha^2e^{-2t}\right]\right).
\end{equation}
\end{corollary}

The formula \eqref{e:LZ-cor} is the conjecture of Lukyanov--Zamolodchikov in the special case $\beta=4\pi$, see \cite[(19)]{MR1453266}.
For all $\beta\in(0,8\pi)$ and $|\re \alpha| < \sqrt{4\pi/\beta}$, they conjecture
\begin{multline} \label{e:LZ}
  \langle\wick{e^{i\sqrt{4\pi}\alpha\varphi(0)}}\rangle_{\SG(\beta,z)}
  = \left(\frac{\mu \Gamma(\frac12 + \frac{\xi}{2})\Gamma(1-\frac{\xi}{2})}{4\sqrt{\pi}}\right)^{\alpha^2}
  \\
  \times \exp\left(\int_0^\infty \frac{dt}{t}\left[\frac{\sinh^2(\sqrt{\frac{\beta}{4\pi}}\alpha t)}{2\sinh(t\frac{\beta}{8\pi})\sinh(t)\cosh((1-\frac{\beta}{8\pi})t)}-\alpha^2e^{-2t}\right]\right),
\end{multline}
where
\begin{equation} \label{e:xi-m}
  \xi = \frac{\beta}{8\pi-\beta}, \qquad \mu = 2M \sin(\frac{\pi}{2}\xi),
\end{equation}
and $M$ is proportional to $|z|^{1/(2-\beta/4\pi)}$ with another explicit $\beta$-dependent proportionality constant, see \cite[(12)]{MR1453266},
where what we denote by $\mu$ is denoted $m$ there, our $z$ is proportional to their $\mu$,
our $\alpha$ corresponds to their $\sqrt{2}a$, and our $\beta/4\pi$ corresponds to their $2\beta^2$.
See also \cite[Section~2]{MR4258290} for an overview of predictions on the sine- and related sinh-Gordon models.

The special case $\beta=4\pi$ of the formula (which we prove) is an important ingredient in arriving at the above
conjecture for general $\beta$.
It is always assumed above that the massless sine-Gordon model exists, see Section~\ref{sec:ref}
for a summary of some mathematical results on this.
An important step would be to prove that, for general $\beta$, the one-point function is strictly positive or that there is a strictly positive physics mass.
For the physics predictions that the physical mass is explicitly given by the formula for $\mu$ stated in \eqref{e:xi-m},
see the discussion in \cite{MR1453266} and also \cite{10.1142/S0217751X9500053X,MR1059830}.

That the one-point function and the physical mass (assuming they are strictly positive)
are proportional to $(|z|^{1/(2-\beta/4\pi)})^{\alpha^2}$ respectively  $|z|^{1/(2-\beta/4\pi)}$
(with the normalization \eqref{e:Malpha-main} for the one-point function)
simply follows from scaling, as already observed in \cite[(4.20)]{MR0434278} for the physical mass
and discussed for the one-point function in \eqref{e:onepoint-scaling} below;
however, that the proportionality constants are nonzero does not (and is in general difficult to establish).

\begin{remark} \label{e:GMC-analyticity}
  Using analyticity, we expect that Corollary~\ref{cor:LZ} can be extended to suitable $\alpha \in \C$.
  For example, in \cite{MR4149524}, it is shown that the complex GMC $M_\alpha$ is an analytic function of $\alpha$ with values in $H^{-2}_{\loc}(\R^2)$,
  for $\alpha \in \C$ corresponding to the $L^1$ phase of the complex GMC. For our methods, analyticity of the complex GMC in a space of
  distributions that can be multiplied with functions in $C^r$, any $r<1$, would be useful.
  The Besov spaces $C^{-s}$, $s<r$, have this property, but they do not seem to be the correct setting for the real GMC \cite{1905.12027}.
  The space of positive measures also have this property, but they are not the correct setting for the imaginary GMC.
  We expect that a suitable setting can be found, but postpone this question about the GMC to future work.
\end{remark}

\subsubsection{Bernard--LeClair differential equation}

Palmer's analysis also implies
a version of the differential equation for the fractional two-point function
predicted by Bernard--LeClair \cite{MR1297289,MR1462303}.

\begin{corollary} \label{cor:main-diffeq}
  Let $\alpha \in (-\frac12,\frac12)$.
  Then 
  $\Sigma(x)=\log \avg{\wick{e^{i\sqrt{4\pi}\alpha\varphi(x)}} \wick{e^{-i\sqrt{4\pi}\alpha\varphi(0)}}}_{\SG(4\pi,z)}$
  satisfies
  \begin{equation}
    \Delta \Sigma = \frac12 \mu^2(1-\cos(2\psi)) %
    ,
  \end{equation}
  for $x\neq 0$, and
  where $\psi=\psi(r)$ with $r=|x|$  solves the differential equation
  \begin{equation}
    \big(\partial_r^2+\frac{1}{r}\partial_r \big)\psi = \frac12 \mu^2 \sin(2\psi)
    + \frac{(2\alpha)^2}{r^2}\tan(\psi)(1+\tan^2(\psi)).
  \end{equation}
  The mass parameter $\mu$ is $\mu=A|z|$ with the same constant $A$ as in Theorem~\ref{thm:main-tau}.
\end{corollary}

\begin{proof}
  This follows from the results of Palmer \cite{MR1233355}, see Section~\ref{sec:BernardLeclair}
  for details.
\end{proof}

We remark that Bernard--LeClair \cite{MR1297289,MR1462303} obtain the above equations with
trigonometric functions replaced by hyperbolic functions.
This difference was already observed in \cite[p.~341]{MR1233355}.
The versions with hyperbolic functions can be obtained by replacing $\psi$ with $i\psi$.

\subsection{The massless sine-Gordon measure}
\label{sec:intro-SG}

We now state the definition and main properties of the Euclidean massless sine-Gordon measure at the free fermion point $\beta=4\pi$,
denoted $\nu^{\SG(4\pi,z)}$ and constructed in \cite{MR4767492}. Further details are given in Section~\ref{sec:mixing}.

In general, 
the Euclidean massless sine-Gordon model with coupling constants $\beta\in (0,8\pi)$ and $z\in \R$ 
should be defined in terms of limits $\epsilon\to 0$, $m\to 0$, $\Lambda\to\R^2$ of the regularized probability measures 
\begin{equation} \label{e:SGdef}
  \nu^{\SG(\beta,z|\Lambda,m,\epsilon)}(d\varphi) \propto 
  \exp\qa{2z \int_{\Lambda}\epsilon^{-\beta/4\pi} \cos(\sqrt{\beta}\varphi(x)) \, dx} \nu^{\GFF(m,\epsilon)} (d\varphi),
\end{equation}
where $\nu^{\GFF(m,\epsilon)}$ is the Gaussian free field (GFF) on $\R^2$  with mass $m>0$  and small scale regularization $\epsilon>0$.
The precise choice of the regularizations is not crucial, but
for concreteness, we take 
$\nu^{\GFF(m,\epsilon)}$ to be the Gaussian measure supported on $C^\infty(\R^2)$ with covariance kernel
\begin{equation}
   \int_{\epsilon^2}^\infty ds\, e^{-s(-\Delta +m^2)}(x,y).
\end{equation}
The expectation with respect to the measure $\nu^{\SG(\beta,z|\Lambda,m,\epsilon)}$ will be denoted by $\avg{\cdot}_{\SG(\beta,z|\Lambda,m,\epsilon)}$,
and omission of an argument means that the corresponding limit has been taken (implying it exists).
The measure $\nu^{\SG(\beta,z|\Lambda,m)}$ is the massive finite-volume sine-Gordon measure (with what is known as free boundary conditions in this context).

\begin{theorem} \label{thm:SG4pi}
  For $z\in \R$, $z\neq 0$,
  there is a unique probability measure $\nu^{\SG(4\pi,z)}$ on $\cS'(\R^2)$ that is symmetric under $\varphi \mapsto -\varphi$, mixing,
  and satisfies, for any $f \in C_c^\infty(\R^2)$ with $\int f\, dx = 0$,
  \begin{equation} \label{e:SG4pi}
    \avg{e^{i(\varphi,f)}}_{\SG(4\pi,z)} = \lim_{\Lambda \to \R}\lim_{m\to 0} \lim_{\epsilon\to 0} \avg{e^{i(\varphi,f)}}_{\SG(4\pi,z|\Lambda,m,\epsilon)}.
  \end{equation}
  This measure has superexponential moments, is Euclidean invariant, is scale covariant, and does not depend on the sign of $z \in \R$.
  For the definitions of mixing, Euclidean invariance, and scale covariance,
  see \eqref{e:intro-transinv}--\eqref{e:intro-mixing} below.
\end{theorem}

The proof of the theorem is deduced from the results of \cite{MR4767492},
where however mixing, superexponential moments, and uniqueness are not established, and  is given in Section~\ref{sec:mixing}.

We emphasize that the main difficulty in proving the theorem is that it applies to the massless model
and that $\nu^{\SG(4\pi,z)}$ is a measure on $\cS'(\R^2)$
rather than $\cS'(\R^2)/\text{constants}$.
The massive model or the massless measure on  $\cS'(\R^2)/\text{constants}$ can be constructed for all $\beta$.
As mentioned before, the existence of the massless measure on $\cS'(\R^2)$ corresponds to the spontaneous breaking of the
noncompact symmetry $\varphi \mapsto \varphi + \frac{2\pi}{\sqrt{\beta}}n$, $n\in \Z$.
Symmetry breaking does not take place for $z=0$:
the massless Gaussian free field can only be constructed modulo constants on $\R^2$.

Euclidean invariance means that the measure is invariant under spatial translations and rotations: for all integrable $F$,
\begin{equation} \label{e:intro-transinv}
  \avg{T_xF}_{\SG(\beta,z)} =     \avg{F}_{\SG(\beta,z)},
\end{equation}
where $T_x$ is the action of translation by $x\in \R^d$, and analogously for rotations.
By scale covariance for the massless sine-Gordon model we mean that
\begin{equation} \label{e:intro-scaleinv}
  \avg{F}_{\SG(\beta,z)}
  =
  \avg{R_sF}_{\SG(\beta,zs^{2-\beta/4\pi})},
\end{equation}
where $R_s$ is the action of spatial rescaling by $s>0$:
$R_sF(\varphi)= F(R_s\varphi)$ and $R_s\varphi(x) = \varphi(x/s)$ in the distributional sense, 
i.e., $(R_s\varphi, f)= s^{2}(\varphi,R_{s^{-1}}f)$ where $R_sf(x) = f(x/s)$.
Mixing means that,  for all $F,G\in L^2(\nu^{\SG(4\pi,z)})$,
\begin{equation} \label{e:intro-mixing}
  \avg{(T_xF)G}_{\SG(\beta,z)} \to      \avg{F}_{\SG(\beta,z)} \avg{G}_{\SG(\beta,z)}, \qquad (|x|\to\infty).
\end{equation}

\subsection{Imaginary multiplicative chaos and fractional correlation functions}
\label{sec:intro-corr}

For $\alpha_1,\dots,\alpha_n \in \R$ and $x_1,\dots, x_n\in \R^2$ distinct, the fractional correlation functions should formally be given by
\begin{equation} \label{e:corrdef}
  \avg{\prod_{i=1}^n\wick{e^{i \sqrt{4\pi} \alpha_i\varphi(x_i)}}}_{\SG(4\pi, z)}.
\end{equation}
Since the random field $\varphi$ is not defined pointwise, this definition must be interpreted in a renormalized and distributional sense.
Since the sine-Gordon measure (with $\beta\geq 4\pi$) is not locally absolutely continuous with respect to the
Gaussian free field, one also cannot simply reduce the definition to the Gaussian case, where the meaning of the correlation functions is well understood.

To define the correlation functions, we actually construct the \emph{imaginary multiplicative chaos} (IMC for short) $M_\alpha$,
which is formally given by $\wick{e^{i\sqrt{4\pi}\alpha\varphi}}$ defined as
a random generalized function (up to a choice of normalization constant).
This IMC is defined with respect to the infinite-volume massless sine-Gordon measure $\nu^{\SG(4\pi,z)}$.
Its analogue for the Gaussian free field is well studied, see \cite{MR3339158,MR4149524}.
It can be obtained as a limit of its regularized version:
\begin{equation} \label{e:Meps}
  M_\alpha^\epsilon(f) = \wick{e^{i\sqrt{4\pi}\alpha \varphi_\epsilon}}_\epsilon(f) =  \epsilon^{-\alpha^2} \int e^{i\sqrt{4\pi}\alpha\varphi_\epsilon(x)} f(x)\, dx
  ,
\end{equation}
where, for a smooth, nonnegative, compactly supported mollifier $\eta_\epsilon(x) = \epsilon^{-d} \eta(x/\epsilon)$ with
$\int \eta_\epsilon \, dx =1$, we denote the smoothed field by
\begin{equation}
  \varphi_\epsilon = \eta_\epsilon* \varphi
  .
\end{equation}
The limiting Gaussian IMC is independent of the mollifier $\eta$ except for a multiplicative constant,
resulting from the dependence of $\lim_{\epsilon \to 0} [\var(\varphi_\epsilon(0)) - \frac{1}{2\pi} \log \epsilon^{-1}]$ on $\eta$,
see Propositions~\ref{prop:GMC}--\ref{prop:hkimc}.

Our result for the infinite-volume massless sine-Gordon model is the following theorem.
Its proof also includes a finite-volume version with strong convergence estimates.

\begin{theorem} \label{thm:Mexists}
  Let $\alpha \in (-1,1)$. Then 
  in probability  with respect to the infinite-volume massless sine-Gordon measure $\nu^{\SG(4\pi,z)}$,
  the generalized function $M_\alpha^\epsilon$ converges to a random limit $M_\alpha$  as $\epsilon \to 0$
  that is independent of the mollifier $\eta$ except for a deterministic multiplicative constant.
  For any $s>\alpha^2$, the limit is an element of the Besov-H\"older space $C^{-s}_{\rm loc}$,
  the convergence takes place in that space, and $\chi M_\alpha$, $\chi\in C_c^\infty$ has moments in $C^{-s}$ norm of all order.
\end{theorem}

\begin{remark}
  We will sometimes abuse notation slightly by also including an additional deterministic multiplicative constant in the definition of
  $M_\alpha$ without emphasizing this explicitly,
  for example to fix the short-distance asymptotics \eqref{e:2pt-LZ-norm}.
\end{remark}

The extension of the existence of the IMC from the Gaussian free field to
the massless sine-Gordon measure is nontrivial for two reasons.

The first reason is that the sine-Gordon measure with $\beta\geq 4\pi$ (in particular for $\beta=4\pi$)
is not absolutely continuous with respect to law of the free field
-- even with a mass and a finite-volume sine-Gordon interaction.
Our construction of the IMC uses the optimal Besov-H\"older regularity of the IMC
with respect to the Gaussian free field and the  (presumably) optimal H\"older regularity of the difference
between the sine-Gordon field and the Gaussian free field in a probabilistic coupling.

This decomposition is addressed in Section~\ref{sec:coupling-finvol}.
Such decompositions in regularity are familiar in the field of singular SPDEs, where global in time solutions
typically yield such decompositions, see for example \cite{MR3693966} (and the recent references
\cite{MR4528966,2401.13648,2402.05544,2410.15493}
for results applying to finite-volume and massive sine-Gordon models).
However, for our purposes to eventually identify the correlation functions with the tau functions
(see Theorem~\ref{thm:main-tau}),
particular attention is needed for a construction with \emph{complex} coupling constants $z$
(for which the naive definition of the sine-Gordon measure would not be a probability measure -- in fact not even a complex measure).

The second reason is that the existence of the IMC with respect to the massless infinite-volume sine-Gordon measure relies on spontaneous symmetry
breaking of the latter, which seems difficult (if not impossible) to establish directly using stochastic analytic or SPDE methods
as in the finite-volume or massive part of the analysis.
This is addressed in Section~\ref{sec:coupling-infvol}, by a combination of the stochastic analytic methods with additional input from integrability
(Bosonization) to establish the symmetry breaking.
Particular attention is also needed for the convergence of the IMC with respect to finite-volume sine-Gordon measure
to the massless infinite-volume version.

In terms of the IMC for the massless sine-Gordon measure constructed in Theorem~\ref{thm:Mexists}, the fractional sine-Gordon correlation functions can now be defined as follows.

\begin{definition} \label{defn:corr-chaos}
For $\alpha_1,\dots,\alpha_n \in (-1,1)$ and $f_1,\dots, f_n \in C_c^\infty(\R^2)$, the smeared fractional correlation function is
defined as
\begin{equation}
  \avg{M_{\alpha_1}(f_1) \cdots M_{\alpha_n}(f_n)}_{\SG(4\pi,z)}.
\end{equation}
For distinct points $x_1,\dots,x_n \in \R^2$ the pointwise correlation functions \eqref{e:corrdef} are then defined by
\begin{equation} \label{e:corrfunc-formal-def}
      \int dx_1\cdots dx_n\; f_1(x_1)\cdots f_n(x_n) 
    \avg{\prod_{j=1}^n\wick{e^{i\sqrt{4\pi}\alpha_j \varphi(x_j)}}}_{\SG(4\pi, z)}.
\end{equation}
The definitions are possibly up to a positive multiplicative constant which can be fixed as discussed above Corollary~\ref{cor:2pt-LZ}.
\end{definition}

Using that the massless sine-Gordon measure $\nu^{\SG(4\pi,z)}$ is translation invariant and mixing,
as discussed in \eqref{e:intro-transinv} and \eqref{e:intro-mixing} and established in Section~\ref{sec:mixing},
one can in fact recover the correlation functions with $\sum_i \alpha_i \neq 0$ from those with $\sum_i \alpha_i=0$.
For illustration we consider the computation of the one-point function from the neutral two-point function.
Mixing implies
\begin{equation}
  \avg{M_\alpha(f) M_{-\alpha}(T_x g)}_{\SG(4\pi,z)} \to \avg{M_\alpha(f)}_{\SG(4\pi,z)} \avg{M_{-\alpha}(g)}_{\SG(4\pi,z)} \qquad (|x|\to\infty),
\end{equation}
where $T_x$ is the action of translation as in \eqref{e:intro-transinv}.
Using the symmetry $\varphi \mapsto -\varphi$ of the sine-Gordon measure,
$\avg{M_{-\alpha}(g)}_{\SG(4\pi,z)}=\avg{M_{\alpha}(g)}_{\SG(4\pi,z)}$ on the right-hand side,
and it follows from correlation inequalities that this one-point function is in fact nonnegative,
see the discussion above \eqref{e:SG4pi-ffx}.
Translation invariance implies that the one-point function is simply a constant, i.e.,
\begin{equation}
  \avg{M_{\alpha}(g)}_{\SG(4\pi,z)} = \avg{\wick{e^{i\sqrt{4\pi}\alpha \varphi(0)}}}_{\SG(4\pi,z)} \int g\, dx,
\end{equation}
for some constant $\avg{\wick{e^{i\alpha \sqrt{4\pi}\varphi(0)}}}_{\SG(4\pi,z)} \geq 0$ which is the definition of the one-point function.
Therefore we obtained the following representation of the one-point function.

\begin{corollary} \label{cor:onepoint-twopoint}
  For $\alpha\in (-1,1)$, the one-point function is given by
  \begin{equation}
    \avg{\wick{e^{i\sqrt{4\pi}\alpha \varphi(0)}}}_{\SG(4\pi,z)} 
    = \avg{M_{\alpha}(g)}_{\SG(4\pi,z)} = \lim_{|x|\to\infty} \sqrt{  \avg{M_\alpha(g) M_{-\alpha}(T_x g)}_{\SG(4\pi,z)}},
\end{equation}
where $g\in C_c^\infty(\R^2)$ is arbitrary with $\int g \, dx =1$.
\end{corollary}

The $\varphi \mapsto -\varphi$ symmetry and scale covariance \eqref{e:intro-scaleinv} of the massless sine-Gordon measure also imply that
the dependence of the one-point function on $z$ is given by
\begin{equation} \label{e:onepoint-scaling}
  \avg{\wick{e^{i\alpha \sqrt{4\pi}\varphi(0)}}}_{\SG(\beta,z)} =
  C_{\alpha,\beta} \pB{|z|^{1/(2-\beta/4\pi)}}^{\alpha^2},
\end{equation}
if $\beta=4\pi$ (and in general if the measure exist) as in \eqref{e:LZ},
but without determining the constant $C_{\alpha,\beta} =   \avg{\wick{e^{i\alpha \sqrt{4\pi}\varphi(0)}}}_{\SG(\beta,1)}$.
Indeed, by \eqref{e:intro-scaleinv} and assuming $z>0$ (without loss of generality),
\begin{equation}
  F_{\alpha,\beta} (z)
  =  \avg{\wick{e^{i\alpha \sqrt{4\pi}\varphi(0)}}}_{\SG(\beta,z)}
  = s^{\alpha^2} \avg{\wick{e^{i\alpha \sqrt{4\pi}\varphi(0)}}}_{\SG(\beta,s^{2-\beta/4\pi}z)}
  = s^{\alpha^2}F_{\alpha,\beta}(s^{2-\beta/4\pi}z),
\end{equation}
which gives $F_{\alpha,\beta}(z) = (z^{1/(2-\beta/4\pi)})^{\alpha^2} F_{\alpha,\beta}(1)$.

\begin{remark} \label{rk:ssb}
  We emphasize that, for $\alpha \neq 0$,
\begin{equation}
  \avg{\wick{e^{i\sqrt{4\pi}\alpha \varphi(0)}}}_{\SG(4\pi,z)} \neq \lim_{L\to\infty}\lim_{m\to 0}   \avg{\wick{e^{i\sqrt{4\pi}\alpha \varphi(0)}}}_{\SG(4\pi,z|L,m)} = 0,
\end{equation}
reflecting the spontaneous symmetry breaking of the sine-Gordon measure.
That the left-hand side is nonzero follows from Corollary~\ref{cor:LZ}.
We omit the proof that the right-hand side vanishes.
(Since it vanishes already without the limit $L\to\infty$, one could prove it using Section~\ref{sec:finvol-massless}.)
\end{remark}

\subsection{Further references}
\label{sec:ref}

Aside from the references mentioned already in the previous subsections, we mention the following ones.

\paragraph{Multiplicative chaos}
For Gaussian log-correlated random fields, the imaginary multiplicative chaos (and more generally the complex multiplicative chaos)
were defined and studied in a number of references,  see in particular \cite{MR3339158,MR4047992,MR4149524,MR4507934}
for examples that seem most closely related to our work.
The imaginary multiplicative chaos
is an ingredient in our construction and regularity of the fractional correlation functions of the sine-Gordon model.
We do not attempt to survey the vast literature on the real Gaussian multiplicative chaos and instead refer to \cite{MR3274356} for an overview.

\paragraph{Tau functions}
One of the motivations for the introduction of the tau functions by Sato--Miwa--Jimbo is the understanding of the spin correlation
functions of the two-dimensional near-critical Ising model, see
\cite{MR499666,MR594916,MR695532} and references in these,
and related ideas have also recently played a role in the analysis of the finite domain correlation functions of the Ising model,
see for example \cite{MR3296821,1811.06636}.

There is a large literature on the general theory of tau functions
associated with monodromy preserving deformations of ordinary differential equations related to Painlev\'e equations,
starting with \cite{MR555666,MR0575994},
that has turned out to have applications to various problems in integrable systems, integrable probability, and integrable quantum field theory.
We will not attempt to review these, but mention a few specific applications in contexts that seem most closely related to ours.

In the context of the dimer model (with Gaussian height function limit),
twisted fermions and the corresponding fractional correlations (``electric correlators'')
were studied by Dubédat~\cite{MR3369909}.
This is essentially a discrete version of the massless case for us (free field instead of sine-Gordon field).
The sine-Gordon model at the free fermion point is the limiting height function of a near-critical dimer model~\cite{2209.11111,BMR25},
and we expect a result analogous to Dubédat's result also to hold in this context (with the fractional free field correlations replaced by the sine-Gordon
ones constructed and computed in this paper).
A proof of this (which is already difficult in the above work with Gaussian limit) would be very interesting.

For tau functions in the context of the double dimer model (which also has a Gaussian height function limit
but more interesting monodromy),
see \cite{MR3880204,MR4269427}.

Painlev\'e equations (of a different type) are also important in the description of correlation functions of
noncompact conformal field theories such as Liouville CFT \cite{MR3322384}.
The recent probabilistic construction of the latter (see \cite{2403.12780} for a review) enabled a probabilistic
description of the conformal blocks \cite{MR4748792}.
The difficulties in this probabilistic realization are different to those in our context:
The base space is compact and the measures involved are absolutely continuous, while the lack of these properties
and the associated spontaneous symmetry breaking are the sources of the some of the main difficulties in our context
of the sine-Gordon model.
On the other hand, the tau functions have different algebraic structure, resulting in different challenges from this.

As opposed to our context (which is ``massive'') the above examples all correspond to ``massless'' (or conformal)
field theories.

\paragraph{Massive Thirring model}
For general values of the coupling constants $(\beta,z) \in (0,8\pi) \times \R$,
the massless sine-Gordon model is predicted to be in correspondence with the massive Thirring model,
see Coleman's original paper \cite{PhysRevD.11.2088} and also \cite[Section~1.3]{MR4767492} for a high-level summary.
In particular, at the free fermion point $\beta=4\pi$ studied in this paper, the massive Thirring model becomes massive free fermions.
This correspondence is in the sense of the Bosonization dictionary \eqref{e:corr1}--\eqref{e:corr4}
and does not immediately extend to the fractional correlations studied in this paper.

The Massive Thirring model with small coupling constant was constructed from the fermionic path integral in \cite{MR2308750} and a
finite-volume Bosonization result proved in \cite{MR2461991}. In the conjectured Coleman correspondence
the small coupling would correspond to $\beta$ near $4\pi$ and a rigorous proof of this correspondence in the infinite-volume limit
would be very interesting
(extending the result \cite{MR4767492} for $\beta=4\pi$).
It would be even more interesting to extend the ``fermionization'' of the fractional correlation functions
to  $\beta \neq 4\pi$.

\paragraph{Sine-Gordon model}
While our results focus on the large scale behavior of the sine-Gordon model,
the thorough understanding of its small scale behavior (the ultraviolet problem) is a prerequisite.
Since we are interested in  the free fermion point $\beta=4\pi$, there is a nontrivial ultraviolet problem 
and the sine-Gordon measure is not locally absolutely continuous with respect to the Gaussian free field.
This results in various difficulties for us as we need  strong estimates to
connect with the results from exact integrability, but the solution is essentially the same for $\beta \in [4\pi,6\pi)$.
To handle the short-distance singularities,
our analysis uses generalizations and extensions of the solution of the Polchinski equation
from \cite{MR914427}.
Other recent treatments motivated by SPDE techniques (in finite volume or with external mass)
related to methods we use include
\cite{MR4303014,MR4399156,2410.15493,2401.13648,2402.05544}.
For the resolution of the ultraviolet problem for all $\beta\in (0,8\pi)$, we mention 
the reference \cite{MR849210} (which is based on \cite{MR814849,MR649810})
and \cite{MR1777310} (based on \cite{MR1240586,MR1048698}),
and the recent work \cite{2508.13778} (based on \cite{MR2523458,MR2917175}).

The above works all consider the massive or finite-volume version of the sine-Gordon model,
not the more subtle massless model.
The existence of the massless infinite-volume limit is related to the existence of a mass gap for the formally massless model (see
the discussion in Section~\ref{sec:intro-intro}), a problem that is mathematically open for general $\beta \in (0,8\pi)$.
Aside from the results for $\beta=4\pi$ from \cite{MR4767492} which we heavily rely on and extend in this paper in various aspects,
the existence of a mass gap is essentially known for small $\beta>0$ (or could be established from these methods),
see \cite{MR574172,MR923850} and also \cite{MR818828}.

\subsection{Outline}

This paper has two major parts and a global difficulty arises from the necessity to design both in such a way they can ultimately be glued together
(and complete each other).

\paragraph{Part I (Sections~\ref{sec:mixing}--\ref{sec:coupling-infvol})}
The first part concerns the construction and regularity (analyticity, infinite-volume convergence)
of the fractional correlation functions of the sine-Gordon model at the free fermion point. Much of this involves
stochastic analysis of similar flavor as in singular SPDEs and applies to general values of $\beta$,
but the extension of these methods to the massless infinite-volume
model at $\beta=4\pi$ requires combination with input from the integrability at the free fermion point
(which we use to establish the spontaneous symmetry breaking).

\paragraph{Part II (Sections~\ref{sec:Green}--\ref{sec:Palmer})}
The second main part concerns the construction and regularity of renormalized determinants of twisted Dirac operators.
The infinite-volume version of these determinants are the tau functions of Sato--Miwa--Jimbo and Palmer, but to establish
the connection with the sine-Gordon model it is essential for us to have
suitable finite-volume versions (which we also refer to as renormalized partition functions)
with good analyticity properties.

To define the finite-volume tau functions, we first develop a theory for the Green's functions of twisted Dirac operators with finite-volume mass. It is based on PDE techniques and functional analysis
(such as Fredholm theory),
in contrast to, for example, Palmer's use of geometric-algebraic tools like Grassmannians and determinant bundles or SMJ's
integrable systems framework in the infinite-volume situation.
Using the Green's functions, we define our finite-volume versions of the tau functions and prove various regularity properties
(which are false in infinite volume) on the one hand, and
on the other hand, connect them to the usual tau functions in the infinite-volume limit, from which the exact results
in our applications follow.

\paragraph{Connection}
The high level strategy for the proof of Theorem~\ref{thm:main-tau}, connecting the two parts above,
is inspired by that of the Coleman correspondence in \cite{MR4767492},
but the details and execution of this strategy are different and more difficult in all parts.
The starting point is an analogous statement for $z=\mu=0$ (the massless case)
-- and in this situation it is actually elementary, see Section~\ref{sec:massless-bosonization}.
To show Theorem~\ref{thm:main-tau} when $z\neq 0$, we approximate both sides of the identity \eqref{e:main-tau} by finite-volume versions,
corresponding to the two parts discussed above.
The approximation of the left-hand side requires special care because of the spontaneous symmetry breaking of the sine-Gordon measure,
see Sections~\ref{sec:imc} and~\ref{sec:coupling-infvol},
while the approximation of the tau functions on the right-hand side requires special care because of the singularity of the 
formal determinants, see Section~\ref{sec:Palmer}.

For the finite-volume versions of the two sides of \eqref{e:main-tau} we show,
with different arguments, analyticity in $z$ respectively $\mu$,
both in a complex neighborhood of $\R$ (that does depend on the volume but always contains the full real axis),
see Section~\ref{sec:imc} and \ref{sec:coupling-finvol} for the sine-Gordon correlation functions and
Section~\ref{sec:Green} and \ref{sec:bosonization} for the tau functions.
We identify the series expansions of both sides and use analytic continuation to extend the result to all coupling constants.

We emphasize that analyticity does not hold in the infinite-volume limit, necessitating this involved approximation scheme.

\paragraph{Overview}
In more detail, the article is outlined as follows.

In Section~\ref{sec:mixing}, we summarize the definition and properties of the infinite-volume massless sine-Gordon measure $\nu^{\SG(4\pi,z)}$,
and prove that this measure has superexponential moments and is mixing.

In Section~\ref{sec:Gauss}, we include some background and notation for the Gaussian free field, its scale decomposition,
its fractional correlation functions,
and the associated Gaussian imaginary multiplicative chaos.

In Section~\ref{sec:imc}, we construct the imaginary multiplicative chaos with respect to the sine-Gordon field, both in the massive finite-volume case
and in the massless infinite-volume case. The construction relies on the decompositions of the sine-Gordon field established in Section~\ref{sec:coupling-finvol}
in finite volume and in Section~\ref{sec:coupling-infvol} in the massless infinite-volume situation.
We further prove that the massless limit of the fractional correlation functions of the finite-volume sine-Gordon measure
exists and that these have a series expansion in terms of free field correlations.

In Section~\ref{sec:massless-bosonization}, we discuss the Bosonization formulas for massless twisted free fermions.
These are elementary, but important input for the massive Bosonization.

In Section~\ref{sec:Green}, we establish existence and various estimates of the Green's function of a twisted Dirac operator with a finite-volume mass term.
In particular, we establish analyticity in the mass $\mu$ in a neighborhood of the real axis.

In Section~\ref{sec:bosonization}, we combine the analyticity estimates from Sections~\ref{sec:coupling-finvol} and \ref{sec:Green},
together with the massless Bosonization identities from Section~\ref{sec:massless-bosonization}, to extend the Bosonization to the finite-volume situation.
As part of this, in particular we define our candidate for the finite-volume renormalized determinants of the twisted Dirac operator with finite-volume mass (or finite-volume tau function).

In Section~\ref{sec:Palmer}, we show that the finite-volume renormalized determinants indeed converge, in the infinite-volume limit, to the tau functions.

In Section~\ref{sec:main-proofs}, we finally prove the main results by combining the ingredients from the previous sections. 

In Appendix~\ref{app:det}, we provide a formal motivation of the definition of renormalized determinants of the twisted Dirac operators without mass.
It is not needed for our results, but we hope that it provides helpful intuition.

In Appendices~\ref{app:GMCGFF} and \ref{app:Gauss}, we include proofs of some estimates of relatively standard flavor for the Gaussian multiplicative chaos,
the Gaussian free field and its scale decomposition.

In Appendices~\ref{app:fracrenorm} and \ref{app:renormpart}, we provide a construction of the renormalized potential
for multiple periodic interactions with different scales and use it to derive
integrability estimates for fractional free correlations. %
This provides a priori bounds used in particular in Sections~\ref{sec:Gauss-corr} and~\ref{sec:finvol-massless}.
Our analysis extends \cite[Sections 4--5]{MR4767492} (which in turn is based on \cite{MR914427}).

In Appendix~\ref{app:SMJ}, using the framework of \cite{MR555666,MR1233355},
we show that the exponential decay of certain solutions to the massive twisted Dirac equation is locally uniform in the branch points.

\subsection{Notation}
\label{sec:notation}

Our convention for the Fourier transform is 
\begin{equation}
\hat f(p)=\int_{\R^2}dx\, f(x) e^{-ip\cdot x} \qquad \text{and} \qquad f(x)=\frac{1}{(2\pi)^2}\int_{\R^2} dp \, e^{ip\cdot x} \hat f(p).
\end{equation}
Throughout we identify $\R^2$ and $\C$ and use whichever interpretation is most convenient.

\subsubsection{Function spaces}

We use various function spaces throughout the paper.
We use both global versions of the spaces on $\R^2$, weighted versions with a polynomial weight $\rho$ with exponent $\sigma$,
\begin{equation} \label{e:weight-def}
  \rho(x)=(1+|x|^2)^{-\sigma/2}
\end{equation}
that permits functions or distributions to grow at infinity when $\sigma>0$, and local versions of these spaces indicated with subscript loc.
In particular, we use standard $L^p$ spaces (both on $\R^2$ and in probability), and write
\begin{equation}
  \|f\|_{L^\infty(\rho)} = \|\rho f\|_{L^\infty(\R^2)}
\end{equation}
for the weighted norm and $L^\infty_\loc(\R^2)$ for the space of locally bounded measurable functions on $\R^2$.
The space $C^k$ of $k$-times continuously differentiable functions has weighted norm (and similarly without weight)
\begin{equation}
  \|f\|_{C^k(\rho)} = \sum_{i=0}^k \|\rho \nabla^k f\|_{L^\infty(\R^2)}, \qquad \|f\|_{C^k(\R^2)} = \sum_{i=0}^k \|\nabla^k f\|_{L^\infty(\R^2)}.
\end{equation}
The H\"older space $C^r$ with $r\in (0,1)$ has norm (and similarly for global and local version,
as well as for $r>1$ which we however do not need)
\begin{equation}
  \|f\|_{C^r(\rho)} = \|\rho f\|_{L^\infty(\R^2)} + [f]_{C^r(\rho)}, \qquad [f]_{C^r(\rho)} =\sup_{\substack{x,y\in \R^2\\ |x-y|\leq 1}} \rho(x) \frac{|f(x)-f(y)|}{|x-y|^r}.
\end{equation}
For the Besov--H\"older spaces $C^{-s}$ with $s\in(0,1)$ we use a variant of the $B^{-s}_{\infty,\infty}$ norm %
given by
\begin{equation} \label{e:Besov-def}
  \|f\|_{C^{-s}(\rho)} = \sup_{R\leq 1} R^s \|\Psi_R * f\|_{L^\infty(\rho)},
\end{equation}
where $\Psi \in C_c^\infty(\R^2)$ is a fixed test function with support in $B_1(0)$
and $\int \Psi \, dx = 1$ which we can pick radially symmetric and nonnegative,
and set $\Psi_R(x) = R^{-2}\Psi(x/R)$.
For a summary of the main properties of these norms sufficient for what we need, see \cite[Appendix~A]{2504.08606}.

To ensure that spaces are separable, we define the spaces corresponding to the above norms as the completion of $C_c^\infty(\R^2)$.
As discussed in \cite{MR3693966} these spaces are smaller than the standard ones (of distributions with finite norm),
but ensure separability. For example, $C^0(\R^2)$ corresponds to the space of continuous functions decaying to $0$ as $|x|\to \infty$,
typically denoted by $C_0(\R^2)$.
Separability ensures that the Bochner theory of Banach space valued random variables is applicable with standard properties, e.g.,
continuous functions are measurable.
See \cite[Part I]{MR3236753} for an introduction to integration and stochastic integration on Banach spaces
which we will use.

Finally, $\cS(\R^2)$ denotes the Schwartz space of test functions (with its usual topology)
and $\cS'(\R^2)$ the space of Schwartz distributions,
and $C_c^\infty(\R^2)$ denotes compactly supported test functions and $\cD'(\R^2)$ denotes the space of distributions
which we often refer to as generalized functions to avoid confusion with the distribution of a random variable.
We also write $L_c^\infty(\R^2)$ for the space of bounded measurable functions on $\R^2$ with compact support.

\subsubsection{Expectations}

We generally use $\E$ for the expectation with respect to a probability space on which a cylindrical Brownian motion is defined,
see Section~\ref{sec:Gauss}. In particular, a scale decomposition of the Gaussian free field
and of the finite-volume sine-Gordon measure (and its regularized version) will be realized on this probability space as, respectively,
\begin{equation}
  (\Phi^{\GFF(m)}_{t})_{t}, \qquad   (\Phi^{\SG(\beta,z|\Lambda,m)}_t)_t,
\end{equation}
see Section~\ref{sec:Gauss} for the GFF and Section~\ref{sec:coupling-finvol} for the finite-volume sine-Gordon measure.

In the above, $\Phi^{\GFF(m)}_0$ and $\Phi^{\SG(\beta,z|\Lambda,m)}_0$ correspond to the full realizations of the GFF and the finite-volume sine-Gordon model
(the version with parameter $t$ corresponding to versions in which the short-distance scales from $0$ to $t$ are missing).
We denote the expectation with respect to laws of these distributions, respectively, by
\begin{equation}
  \avg{\cdot}_{\GFF(m)}, \qquad
  \avg{\cdot}_{\SG(\beta,z|\Lambda,m)}.
\end{equation}
More generally, we will use angle brackets $\avg{\cdot}$ for expectations with respect to probability measures referring to the
law of a field on $\R^2$ (typically distribution valued).

Truncated expectations (also called cumulants) will be denoted by a superscript $\sf T$ and semicolons between the ``slots''
of the cumulants. For example, 
\begin{equation}
  \E^{\sf T}[A_1;A_2] = \ddp{}{t_1}\ddp{}{t_2}\Big|_{t_1=t_2=0} \log \E[e^{t_1A_1+t_2A_2}],
\end{equation}
where the expression also makes sense when the exponential moment is defined in terms of a formal power series.
We will use various equivalent combinatorial definitions throughout the paper.
For some basic properties (from which others are easy to derive), we refer to \cite[Appendix~A.1]{MR4767492}.

\section{Definition and mixing of the massless sine-Gordon measure}
\label{sec:mixing}

In this section, we summarize the construction and main properties of the massless sine-Gordon measure at the free fermion point $\nu^{\SG(4\pi,z)}$,
culminating in the proof of Theorem~\ref{thm:SG4pi}.
For most statements, we refer to the relevant statements in \cite{MR4767492}.
In addition, we establish the existence of superexponential moments and the mixing property of the massless sine-Gordon model.

\subsection{Modulo constants sine-Gordon measure}
\label{sec:mixing-modconst}

As discussed in Section~\ref{sec:intro-SG},
given regularization parameters $\epsilon>0$, $m>0$, $\Lambda \subset \R^2$ with $|\Lambda|<\infty$,
we define the regularized sine-Gordon measure with support in fields in $C^\infty(\R^2)$ by
\begin{equation} \label{e:SGdef-bis}
  \nu^{\SG(\beta,z|\Lambda,m,\epsilon)}(d\varphi) \propto 
  \exp\qa{2z \int_{\Lambda}\epsilon^{-\beta/4\pi} \cos(\sqrt{\beta}\varphi(x)) \, dx} \nu^{\GFF(m,\epsilon)} (d\varphi),
\end{equation}
where $\nu^{\GFF(m,\epsilon)}$ is a regularized Gaussian free field on $C^\infty(\R^2)$ which we choose with covariance
\begin{equation}
   \int_{\epsilon^2}^\infty ds\, e^{-s(-\Delta +m^2)}(x,y).
\end{equation}
For $\beta\in (0,6\pi)$ and $z\in \R$, and assuming that $m>0$ and $|\Lambda|<\infty$,
the weak limit $\epsilon \to 0$ of these measures exists as a measure $\nu^{\SG(\beta,z|\Lambda,m)}$ on $\cS'(\R^2)$.
There are various proofs of this result, and we refer to the presentation contained in \cite{MR4767492} which in turn builds on the method of \cite{MR914427}.
The measure $\nu^{\SG(\beta,z|\Lambda,m)}$ will be referred to as the massive finite-volume sine-Gordon measure.

Correlation inequalities of Ginibre type \cite[Corollary~3.2]{MR496191} (see also \cite{MR0475380})
imply that the characteristic functional $\avg{e^{i(\varphi,g)}}_{\SG(\beta,z|\Lambda,m)}$, with $g\in C_c^\infty(\R^2)$ a real-valued test function,
is increasing in $\Lambda \subset \R^2$ and in $z>0$,
as well as in $m>0$, see \cite[Section~3.4]{MR4767492} for detailed references.
Variants of these correlation inequalities also imply that
the moment generating functional $\avg{e^{(\varphi,g)}}_{\SG(\beta,z|\Lambda,m)}$ is decreasing in $\Lambda \subset \R^2$ and $z>0$,
see \cite[Section~5.2]{1711.04720}.

In particular, it follows that the following Gaussian domination bound holds
for any $g \in C_c^\infty(\R^2)$ with $\hat g(0)=0$:
\begin{equation} \label{e:SG-variance-0}
  \avg{e^{(\varphi,g)}}_{\SG(\beta,z|\Lambda,m)} \leq
  \avg{e^{(\varphi,g)}}_{\GFF(m)}
  \leq
  e^{\frac12 \|g\|_{\dot H^{-1}(\R^2)}^2}, \qquad \|g\|_{\dot H^{-1}(\R^2)} = (g, (-\Delta)^{-1}g)^{1/2}.
\end{equation}
By monotonicity of the characteristic functional,
the massless sine-Gordon model can then be constructed modulo constants,
i.e., as a probability measure $\nu^{\SG(\beta,z)}$ on the space $\cS'(\R^2)/\text{constants}$,
see \cite[Proof of Theorem~1.6]{MR4767492} for details.
In other words, the gradient of the field can be constructed, but not the field itself.
It also follows from \eqref{e:SG-variance-0} that
the measure is, in fact, supported on $C^{-s}(\rho)/\text{constants}$ where $\rho(x)=(1+|x|^2)^{-\sigma/2}$ with any $\sigma>0$ and any $s>0$,
i.e., there is a version of the field such that for any $\eta \in \cS(\R^2)$ with $\hat\eta(0)=1$ and any $p>0$ one has
\begin{equation} \label{e:SG-Besov}
  \avg{\|\varphi-(\varphi,\eta)\|_{C^{-s}(\rho)}^p}_{\SG(\beta,z)} < \infty,
\end{equation}
see Proposition~\ref{prop:GFF-Besov} in Appendix~\ref{app:Gauss}.
Finally, monotonicity in $\Lambda$ implies that the resulting measure is Euclidean invariant,
see \cite[Section VIII.6]{MR0489552} for the standard argument.

\subsection{Removal of the modulo constants restriction} %
\label{sec:mixing-withoutmod}

Assuming that for $g\in C_c^\infty(\R^2)$ with $\hat g(0)=0$,
\begin{equation} \label{e:SG-variance}
  \avg{(\varphi,g)^2}_{\SG(\beta,z)} \leq C(\beta,z) \|g\|_{L^2(\R^2)}^2,
\end{equation}
the massless sine-Gordon measure can be extended in a canonical way from $\cS'(\R^2)/\text{constants}$ to $\cS'(\R^2)$  as we explain now
(as in the proof of \cite[Theorem~1.3]{MR4767492}).
For the free fermion point $\beta=4\pi$ the variance bound \eqref{e:SG-variance} was derived in \cite{MR4767492} from Bosonization.
Indeed, it is immediate from  \cite[Theorem~1.3]{MR4767492} since $\hat C_\mu(p)$ is bounded in $p$.

Extension of the measure $\nu^{\SG(\beta,z)}$ from $\cS'(\R^2)/\text{constants}$
to $\cS'(\R^2)$ using \eqref{e:SG-variance} proceeds as follows.
Let $\eta \in \cS(\R^2)$ be any Schwartz function with $\hat\eta(0)=1$ and define
$P_\eta\varphi$ by
\begin{equation} \label{e:Peta-def}
  P_\eta\varphi = \varphi-(\varphi,\eta), \quad \text{i.e.,} \quad
  (P_\eta\varphi,f) = (\varphi,f_\eta), \qquad f_\eta = f-\hat f(0)\eta.
\end{equation}
Thus the regularized field $P_\eta\varphi$ is defined under $\nu^{\SG(\beta,z)}$ restricted to $\cS'(\R^2)/\text{constants}$.
To simplify notation, we also write $P_T$ for $P_{\eta_T}$ when $\hat\eta_T(p) = \hat\eta(T p)$ and the Schwartz function $\eta$ is fixed.
Thus $\hat\eta_T(p) \to 0$ as $T \to \infty$ whenever $p\neq 0$.
The two-point function bound \eqref{e:SG-variance} implies that the distribution of $P_T\varphi$ has a weak limit as $T \to \infty$
that is independent of the function $\eta$, see the proof of \cite[Theorem~1.3]{MR4767492}, i.e.,
there exists a probability measure $\nu^{\SG(\beta,z)}$ on $\cS'(\R^2)$ whose expectation we denote by $\avg{\cdot}_{\SG(\beta,z)}$
such that, without the restriction $\hat g(0)=0$,
\begin{equation} \label{e:SG-def}
  \avg{e^{i(\varphi,g)}}_{\SG(\beta,z)} = \lim_{T \to\infty} \avg{e^{i(P_T \varphi,g)}}_{\SG(\beta,z)}.
\end{equation}
Moreover, since $\|\eta_T\|_{L^2} = T^{-d/2}\|\eta\|_{L^2} = O(T^{-d/2})$, the two-point function bound \eqref{e:SG-variance} implies
\begin{align}
  \avg{(P_T\varphi,g)^2}_{\SG(\beta,z)}
  \leq C(\beta,z) \|g-\hat g(0)\eta_T\|_{L^2}^2
  &\leq C(\beta,z) \qB{\|g\|_{L^2} + |\hat g(0)| \|\eta_T\|_{L^2}}^2
    \nnb
  &  \leq C(\beta,z) \qB{\|g\|_{L^2} + O(T^{-d/2}) |\hat g(0)|}^2.
\end{align}
Thus, also without the restriction $\hat g(0)=0$, by \cite[Theorem~3.4]{MR1700749},
\begin{equation} \label{e:SG-variance-limit}
  \avg{(\varphi,g)^2}_{\SG(\beta,z)} \leq \liminf_{T\to\infty} \avg{(P_T\varphi,g)^2}_{\SG(\beta,z)} \leq C(\beta,z)\|g\|_{L^2}^2.
\end{equation}
To guarantee that the extension has moments of any order we further assume that there is a norm on $\cS(\R^2)$
such that the following moment bound holds:
for any $p>0$ and $g\in C_c^\infty(\R^2)$,
\begin{equation} \label{e:SG-moments}
  \liminf_{T\to\infty} \avg{|(\varphi,g- \hat g(0)\eta_T)|^p}_{\SG(\beta,z)} \leq C_p(\beta,z)\|g\|^p.
\end{equation}
It then follows that the limiting infinite-volume measure satisfies, for all $g\in C_c^\infty(\R^2)$,
\begin{equation} \label{e:SG-moments-limit}
  \avg{|(\varphi,g)|^p}_{\SG(\beta,z)} \leq C_p(\beta,z)\|g\|^p.
\end{equation}
For $\beta=4\pi$,
the moment bound \eqref{e:SG-moments}
is derived in Proposition~\ref{prop:moments} or alternatively Corollary~\ref{cor:superexp-moments} below from the results of \cite{MR4767492}.

It also follows from the construction that the measure defined by \eqref{e:SG-def} is Euclidean invariant:
\begin{equation} \label{e:Euclinv}
  \avg{F(\varphi)}_{\SG(\beta,z)} = \avg{F(T\varphi)}_{\SG(\beta,z)},
\end{equation}
where $T$ is any rotation or translation acting on $\varphi \in \cS'(\R^2)$ and $f\in \cS(\R^2)$ in the standard way:
\begin{equation}
  T\varphi(f)) = \varphi(T^*f), \qquad T^*f(x) = f(Tx).
\end{equation}
Indeed, for the restriction of the field to mean $0$ test functions, Euclidean invariance is discussed above. %
By the construction of the extension of the measure to $\cS'(\R^2)$ it is clear that this property carries over to $\nu^{\SG(\beta,z)}$.

\begin{proof}[Proof of Theorem~\ref{thm:SG4pi}]
  The measure is defined by \eqref{e:SG-def}.
  It follows from the definition that \eqref{e:SG4pi} holds and that $\nu^{\SG(4\pi,z)}$ is symmetric under $\varphi\mapsto -\varphi$
  because its restriction     modulo constants has this property.
  Euclidean invariance follows from the discussion around \eqref{e:Euclinv}.
  Superexponential moments are established in Corollary~\ref{cor:superexp-moments} below,
  and mixing is proved in Corollary~\ref{cor:mixing} below.

  To see that the measure is unique, let $f\in C_c^\infty(\R^2)$ be arbitrary.
  The symmetry $\varphi \mapsto -\varphi$ implies that $\avg{e^{i(\varphi,f)}}_{\SG(4\pi,z)} = \avg{e^{-i(\varphi,f)}}_{\SG(4\pi,z)} \in \R$
  and therefore also $\avg{e^{i(\varphi,f)}}_{\SG(4\pi,z)} \geq 0$ by \cite[Theorem~3.1]{MR496191}.
  Mixing and translation invariance thus imply, with the definition $f_x(y) = f(y-x)$,
  \begin{equation} \label{e:SG4pi-ffx}
    \avg{e^{i(\varphi,f)}}_{\SG(4\pi,z)}^2 = \lim_{x\to\infty} \avg{e^{i(\varphi,f-f_x)}}_{\SG(4\pi,z)}.
  \end{equation}
  Since $g=f-f_x\in C_c^\infty(\R^2)$ satisfies $\int g\, dx =0$,
  the right-hand side is determined by \eqref{e:SG4pi}.
  Since the left-hand side determines $\nu^{\SG(4\pi,z)}$ uniquely (as $\avg{e^{i(\varphi,f)}}_{\SG(4\pi,z)} \geq 0$), this concludes the proof of uniqueness.
  
  Finally, we consider scale covariance. Let $R_s\varphi(x) = \varphi(x/s)$.
  The regularized GFF satisfies
  \begin{equation}
    \avg{F(\varphi)}_{\GFF(m,\epsilon)} = \avg{F(R_s\varphi)}_{\GFF(ms, \epsilon/s))}.
  \end{equation}
  Since, clearly,
  \begin{equation}
    \epsilon^{-\beta/4\pi} \cos(\sqrt{\beta}\varphi) = s^{-\beta/4\pi} \times  (\epsilon/s)^{-\beta/4\pi} \cos(\sqrt{\beta}\varphi),
  \end{equation}
  and thus
  \begin{equation}
    \epsilon^{-\beta/4\pi} \int_{\Lambda} \cos(\sqrt{\beta}\varphi) \, dx=
    s^{2-\beta/4\pi} \int_{\Lambda/s}  (\epsilon/s)^{-\beta/4\pi} \cos(\sqrt{\beta}R_{1/s}\varphi) \, dx,
  \end{equation}
it follows that
  \begin{equation}
    \avg{F(\varphi)}_{\SG(\beta,z|\Lambda,m,\epsilon)}
    =
    \avg{F(R_s\varphi)}_{\SG(\beta,zs^{2-\beta/4\pi}|\Lambda/s,ms,\epsilon/s)}.
  \end{equation}
  Taking the limit $\epsilon\to 0$, $m\to 0$, $\Lambda\to\R^2$, for $F(\varphi)= e^{i(\varphi,f)}$, $f\in C_c^\infty(\R^2)$ with $\int f\, dx = 0$,
  \begin{equation}
    \avg{F(\varphi)}_{\SG(\beta,z)}
    =
    \avg{F(R_s\varphi)}_{\SG(\beta,zs^{2-\beta/4\pi})}.
  \end{equation}
  This applies for any $\beta>0$ (for which the $\epsilon\to 0$ limit in finite volume exists). For $\beta=4\pi,$
  using that the measure without the restriction $\int f\, dx =0$ is characterized by
  \eqref{e:SG4pi-ffx}, the analogous statement extends to the full infinite-volume measure.

  That the measure does not depend on the sign of $z$ follows for example from \eqref{e:CThat} below,
  which only depends on $\mu^2 = (Az)^2$.
  More generally, it follows (in finite volume, and as a consequence in particular in the limit)
  from the Coulomb gas representation which only involves neutral charge configurations
  (and thus only depends on $z^2$).   Since we do not need this, we omit the details.
\end{proof}

\subsection{Moments and mixing at the free fermion point}

We now prove the moment bound  \eqref{e:SG-moments} for the massless sine-Gordon model with $\beta=4\pi$
as well as the mixing property \eqref{e:intro-mixing}.
These extend the analysis of the two-point function from \cite{MR4767492}.

\begin{proposition} \label{prop:moments}
  For any functions $f_1,\dots, f_n \in \cS(\R^2)$, %
  the truncated moments (cumulants) of the
  massless sine-Gordon measure $\avg{\cdot}_{\SG(4\pi,z)}$ vanish if $n$ is odd and are given for even $n$ by
  \begin{multline}
    \avg{\varphi(f_1); \cdots ;\varphi(f_n)}^{\sf T}_{\SG(4\pi,z)}
    = \int _{\C^{n-1}} dp \, \widehat f_1(p_1)\cdots \widehat f_n(-p_1-\cdots -p_{n-1}) \hat C_{\mu}^{\sf T}(p_1,\dots,p_{n-1})
  \end{multline}
  where $\mu=Az$ (with a constant $A>0$), and 
  \begin{multline} \label{e:CThat}
    \hat C_{\mu}^{\sf T}(p_1,\dots,p_{n-1}) = 
    -\left(\frac{-i\sqrt{\pi}}{2\pi^2}\right)^n(n-1)! \frac{1}{p_1\cdots p_{n-1}(-p_1-\cdots-p_{n-1})}
    \\\times \int_{\C}dq  \frac{q(q+p_1)\cdots (q+p_1+\cdots +p_{n-1})}{(|q|^2+\mu^2)\cdots (|q+p_1+\cdots p_{n-1}|^2+\mu^2)}.
  \end{multline}
  Moreover, $\hat C_\mu^{\sf T}(p_1,\dots,p_{n-1})$ is locally integrable and there is a constant $C=C(z)$ and a norm $\|\cdot\|$ on $\cS(\R^2)$
  satisfying  $\|\eta_T\|\to 0$ as $T\to\infty$ %
  such that
  \begin{equation} \label{e:momentsT-bd}
    \avg{\varphi(f_1); \cdots ;\varphi(f_n)}^{\sf T}_{\SG(4\pi,z)}
    \leq C^n n!  \prod_{i=1}^n \|f_i\|.
  \end{equation}
  For the norm, one can take
  \begin{equation} \label{e:moments-norm}
    \|f\|=\|(1+|p|)|p|^{-1} \hat f\|_{L^{3/2}} + \|\hat f\|_{L^2} .
  \end{equation}
\end{proposition}

\begin{remark} \label{rk:etaT}
Note $\|\eta_T\|_{L^2} = T^{-1} \|\eta\|_{L^2} \to 0$
and that the norm \eqref{e:moments-norm}
also satisfies $\|\eta_T\|\to 0$: Denoting by $\|\cdot\|'$ the first term in the definition of the norm,
\begin{align} \label{e:etaT-norm}
  \|\eta_T\|'
  &= \pa{\int (1+|p|)^{3/2}|p|^{-3/2} |\hat \eta(T p)|^{3/2} \, dp}^{2/3}\nnb
  &= \pa{T^{-2+3/2} \int (1+|p|/T)^{3/2}|p|^{-3/2} |\hat \eta(p)|^{3/2} \, dp}^{2/3} \leq CT^{-1/3},
\end{align}
for $T\geq 1$.
\end{remark}

The proof of the proposition is given in Section~\ref{sec:pf-moments} below.
The bounds \eqref{e:momentsT-bd} imply that the moment generating function
exists at least for $w\in \C$ such that $2 C(z) |w| \|f\| \leq 1$ and then
\begin{equation} \label{e:expmoment-small}
\avg{e^{w\varphi(f)}}_{\SG(4\pi,z)} \leq e.
\end{equation}
By combining the bounds with the Gaussian domination bound \eqref{e:SG-variance-0}, one can actually deduce almost Gaussian moment bounds,
see Corollary~\ref{cor:superexp-moments} below.
First we show that local integrability of $\hat C_\mu^{\sf T}$ also implies the following mixing property at long distances.

\begin{corollary} \label{cor:mixing}
  The sine-Gordon measure with $\beta=4\pi$ is clustering (mixing):
  for all $F,G \in L^2$,
  \begin{equation}
    \avg{(T_xF)G}_{\SG(4\pi,z)} \to \avg{F}_{\SG(4\pi,z)}\avg{G}_{\SG(4\pi,z)}, \qquad (|x|\to\infty),
  \end{equation}
  where $T_x$ is the action of translation by $x\in\R^2$.
\end{corollary}

\begin{proof}
  By Proposition~\ref{prop:moments},
  the correlation kernels $\hat C_{\mu}^{\sf T}(p_1,\dots, p_{n-1})$ are locally integrable
  (and also integrable after multiplication by Schwartz functions).
  The Riemann--Lebesgue lemma  and the moments to cumulants formula therefore imply the
  claim when $F$ and $G$ are cylinder monomials, i.e., with $f_i \in \cS(\R^2)$ and $g_i\in\cS(\R^2)$,
  \begin{equation}
    F(\varphi) = (\varphi,f_1) \cdots (\varphi,f_n), \qquad
    G(\varphi) = (\varphi,g_1) \cdots (\varphi,g_m).
  \end{equation}
  By linearity the claim extends to cylinder polynomials.
  The proof is completed by the fact that cylinder polynomials are dense in $L^2(\nu^{\SG(4\pi,z)})$.
  This is essentially a consequence of the existence of exponential moments.
  Indeed, let $\cF$ be the subspace of $L^2$ which is orthogonal to all cylinder polynomials.
  Fix $F\in \cF$ and $g\in \cS(\R^2)$, and consider the function 
  \begin{equation}
    \psi(w)=\E(F(\varphi)e^{w(\varphi,g)}).
  \end{equation}
  By \eqref{e:expmoment-small},
  this function $\psi$ is well defined and analytic for $|\re w| \leq \epsilon = 1/(2C(z)\|g\|)$,
  and for $|w|\leq \epsilon$ the series expansion of the exponential implies that $\psi(w)=0$.
  By unique analytic continuation, this implies $\psi(w)=0$ for all $|\re w|\leq \epsilon$.
  In particular, $\psi(it)=0$ for all $t\in \R$,
  so $\mathcal F$ is orthogonal to $\mathrm{span}\{e^{i(\varphi,g)}: g\in \mathcal S(\R^2)\}$.
  Since the exponentials are dense in $L^2$, it follows that $\mathcal F=\{0\}$. This means that also cylinder polynomials are dense in $L^2$.

  For completeness,
  we also include the argument that the exponentials $e^{i(\varphi,g)}$ with $g\in \cS(\R^2)$ are dense in $L^2$.
  Indeed, for every Schwartz cylinder functional,
  \begin{equation}
    F(\varphi) = \tilde F((\varphi,f_1), \cdots, (\varphi, f_n)),
  \end{equation}
  denoting by $\hat F$ the Fourier transform of $\tilde F$,
  we have
  \begin{equation}
    F(\varphi) = \frac{1}{(2\pi)^n}
    \int \hat F(s_1,\dots, s_n) e^{i(\varphi,s_1f_1+\cdots + s_n f_n)} \, ds_1\cdots ds_n,
  \end{equation}
  and it suffices to show that Schwartz cylinder functionals are dense.
  Since any bounded Borel cylinder functional can be approximated by Schwartz
  cylinder functionals, it suffices to show that bounded Borel cylinder functionals are dense.
  The monotone class theorem implies that bounded Borel cylinder functionals are dense
  in the class of all bounded cylinder-measurable functionals. But every $F \in L^2$
  is the monotone limit of the bounded $F \wedge N$ as $N\to \infty$.
\end{proof}

By combining the cumulant bound \eqref{e:momentsT-bd} with the Gaussian domination bound
\eqref{e:SG-variance-0}, one can also obtain almost Gaussian moment bounds in a suitable norm.
This norm is defined by
\begin{equation}
  \dnorm{f} = \|f-\hat f(0)\eta\|_{\dot H^{-1}} + |\hat f(0)|,
\end{equation}
where $\eta \in C_c^\infty(\R^2)$ with $\hat\eta(0)=1$ is fixed.

\begin{corollary} \label{cor:superexp-moments}
  The sine-Gordon measure with $\beta=4\pi$ satisfies, 
  for any $f \in \cS(\R^2)$, $p \geq 2$, %
  \begin{equation} \label{e:superexp-moment1}
    \avg{|(\varphi,f)|^p}_{\SG(\beta,z)}
    \leq  C^p (p \log p)^{p/2} \dnorm{f}^p 
  \end{equation}
  and
  \begin{equation} \label{e:superexp-moment2}
    \avg{|(\varphi,f-\hat f(0)\eta_T)|^p}_{\SG(\beta,z)}
    \leq C^p \qB{ (p\log p)^{p/2} + T^{-p/3} p^p } \dnorm{f}^p.
  \end{equation}
  In particular, \eqref{e:superexp-moment1} implies that $\nu^{\SG(4\pi,z)}$ has superexponential (almost Gaussian) moments.
\end{corollary}

\begin{proof}
  We first observe that \eqref{e:expmoment-small} and \eqref{e:etaT-norm} imply that
  (using that $\avg{e^{tX}} \leq e$ for $t=1/\sigma$ implies
  $\avg{X^p} \leq \frac{p!}{t^p} \avg{e^{tX}} \leq (C\sigma)^pp^p$), for any $T \geq 1$ and $p\geq 1$,
  \begin{equation}
    \avg{|(\varphi,\eta_T)|^p}_{\SG(4\pi,z)} \leq C^p T^{-p/3} p^p.
  \end{equation}
  Thus \eqref{e:superexp-moment2} follows from \eqref{e:superexp-moment1}.
  To show the latter bound, let $f\in \cS(\R^2)$ and write 
  \begin{equation}
    f = (f -\hat f(0)\eta_1) + \hat f(0) (\eta_1-\eta_T) + \hat f(0)\eta_T,
  \end{equation}
  with the choice $T=p^{3/2}$.
  Then the third term is bounded by the above by:
  \begin{equation}
    \avg{|(\varphi,\hat f(0)\eta_T)|^p}_{\SG(4\pi,z)} \leq (CT^{-1/3}|\hat f(0)|)^p p^p \leq  C^p \dnorm{f}^p p^{p/2}.
  \end{equation}
  For the first and middle term we use \eqref{e:SG-variance-0}  which gives (using the fact that $\avg{e^{tX}} \leq e^{\frac12 t^2 \sigma^2}$ implies with $t=\sqrt{p}/\sigma$ that $\avg{X^p} \leq \frac{p!}{t^p} \avg{e^{tX}} \leq (C\sigma)^p p^{p/2}$)
  \begin{equation}
    \avg{|(\varphi,f-\hat f(0)\eta_1)|^p}_{\SG(4\pi,z)} \leq C^p \|f-\hat f(0)\eta_1\|_{\dot H^{-1}}^p p^{p/2} \leq C^p\dnorm{f}^p p^{p/2}.
  \end{equation}
  and 
  \begin{align}
    \avg{|(\varphi,\hat f(0)(\eta_1-\eta_T))|^p}_{\SG(4\pi,z)}
    &\leq C^p |\hat f(0)|^p\|\eta_1-\eta_T\|_{\dot H^{-1}}^p p^{p/2}
      \nnb
    &\leq C^p \dnorm{f}^p (\log T)^{p/2} p^{p/2} \leq C^p  \dnorm{f}^p(p \log p)^{p/2},
  \end{align}
  where we used that
  \begin{equation}
    \|\eta_T -\eta_1\|_{\dot H^{-1}}^2 = \int \frac{dp}{(2\pi)^2} \frac{1}{|p|^2}|\hat\eta(Tp)-\hat\eta(p)|^2 = O(\log T).
  \end{equation}
  This completes the proof.
\end{proof}

\subsection{Proof of Proposition~\ref{prop:moments}}
\label{sec:pf-moments}

For notational convenience, we assume $\mu>0$.
For $n=2$, a much more explicit version of the claim was already shown in \cite[Theorem~1.3]{MR4767492}
which implies that
\begin{equation}
  |\hat C^{\sf T}_\mu(p)|  \leq C(\mu) . %
\end{equation}
Thus 
\begin{equation} \label{e:twopointpf}
  \avg{\varphi(f_1); \varphi(f_2)}_{\SG(\beta,z)}^{\sf T}
  \leq C(\mu) \|\hat f_1\|_{L^2} \|\hat f_2\|_{L^2}
  .
\end{equation}
We therefore assume that $n\geq 3$ in the following.
We first assume that $h_1, \dots, h_n\in C_c^\infty(\C)$ are fixed and that $f_1, \dots, f_n$ are given by
\begin{equation}
  f_j = \bar\partial h_j,
\end{equation}
where we identify $\R^2$ with $\C$ from now on.
Under this assumption,
we will now compute cumulants of the random variables $\varphi(f_j)$.
We start with the following claim
for the truncated correlation functions of bilinears of free fermions with mass $\mu$.
These are given by (see \cite[Equation (1.10)]{MR4767492}):
\begin{equation}
  \langle \bar \psi_1\psi_2(x_1);\cdots; \bar\psi_1\psi_2(x_n)\rangle^\mathsf T=(-1)^{n+1}\sum_{\sigma}\left(-\frac{1}{\pi}\right)^n
  \prod_{i=1}^n \bar\partial_{x_{\sigma^i(1)}} K_0(\mu |x_{\sigma^i(1)}-x_{\sigma^{i+1}(1)}|),
\end{equation}
where the sum is over all cyclic permutations (with no fixed points allowed) and $K_0$ is the
modified Bessel function of the second kind.
For $n$ odd, each term in the product above is odd, and hence in the sum the terms corresponding to $\sigma$ and $\sigma^{-1}$ cancel.

\begin{lemma} \label{lem:hint3}
Let $n\geq 3$. For any $h_1,\dots, h_n$ as above, the following integral is absolutely convergent:
\begin{equation} \label{e:hint3}
  \int_{\C^n}dx_1\, \cdots \, dx_n\, h_1(x_1)\cdots h_n(x_n) \langle \bar \psi_1\psi_2(x_1);\dots;\bar\psi_1\psi_2(x_n)\rangle^\mathsf T
  .
\end{equation}
\end{lemma}

\begin{proof}
Note that we have the bound $|\bar\partial_xK_0(\mu|x-y|)|\leq C_\mu|x-y|^{-1} e^{-\frac{\mu}{2}|x-y|}$, so it is enough to prove that for $g_1\dots,g_n\in L^p(\C)$ and nonnegative for each $p>1$ (including $p=\infty$),  
\begin{equation}
\int_{\C^n}dx\, g_1(x_1)\cdots g_n(x_n)\frac{e^{-\frac{\mu}{2}(|x_1-x_2|+\dots+|x_{n-1}-x_n|+|x_n-x_1|)}}{|x_1-x_2|\cdots |x_{n-1}-x_n| |x_{n}-x_1|}<\infty.
\end{equation}
For $j\geq 2$, let us make the change of variables $x_j=x_1+\sum_{k=2}^j u_k$ so that the integral becomes 
\begin{align}
\int_{\C^n}dx_1\, du_2\cdots du_n\, g_1(x_1) g_2(x_1+u_2)\cdots g_n(x_1+u_2+\cdots +u_n) \prod_{j=2}^n \frac{e^{-\frac{\mu}{2}|u_j|}}{|u_j|} \frac{e^{-\frac{\mu}{2}|\sum_{j=2}^nu_j|}}{|\sum_{j=2}^n u_j|}.
\end{align}
To bound this, we use that by the Cauchy-Schwarz inequality,
\begin{align}
\int dx_1\, g_1(x_1)g_2(x_1+u_2)\cdots g_n(x_1+u_2+\dots+u_n)\leq \|g_1\|_2\|g_2\|_2 \prod_{j=3}^n \|g_j\|_\infty.
\end{align}
Thus it remains to control 
\begin{equation}
\int_{\C^{n-1}}du_2\cdots du_n \prod_{j=2}^n \frac{e^{-\frac{\mu}{2}|u_j|}}{|u_j|} \frac{e^{-\frac{\mu}{2}|\sum_{j=2}^nu_j|}}{|\sum_{j=2}^n u_j|}=(\rho*\cdots *\rho)(0),
\end{equation}
where the right-hand side is the $(n-1)$-fold convolution of the function $\rho(u)=\frac{e^{-\frac{\mu}{2}|u|}}{|u|}$ with itself.
It can be bounded by a generalization of Young's convolution inequality which states that if $\sum_{i=1}^k \frac{1}{p_i}=k-1$ for some $p_1,\dots,p_k\geq 1$, then 
\begin{equation}
|\rho_1*\cdots *\rho_k(0)|\leq \|\rho_1\|_{p_1}\cdots \|\rho_k\|_{p_k}. 
\end{equation}
We apply this inequality with $k=n$ and $p_i=p=\frac{n}{n-1}$.
For $n\geq 3$, one has $p<2$ so that $\|\rho\|_p<\infty$. This concludes the proof.
\end{proof}

\begin{lemma}
  Let $n\geq 3$.
  For any $h_1, \dots, h_n \in C_c^\infty(\C)$ and  $f_i = \bar\partial h_i$, %
  \begin{equation} \label{e:SG-CThat}
    \avg{\varphi(f_1); \cdots ;\varphi(f_n)}^{\sf T}_{\SG(4\pi,z)}
    = \int _{\C^{n-1}} dp\, \widehat f_1(p_1)\cdots \widehat f_n(-p_1-\cdots -p_{n-1}) \hat C_{\mu}^{\sf T}(p_1,\dots,p_{n-1})
  \end{equation}
  where $\mu=Az$ and $\hat C_\mu^{\sf T}$ is as in \eqref{e:CThat}.
\end{lemma}

\begin{proof}
By the Bosonization identity \cite[Theorem~1.1]{MR4767492}
and using that $\varphi(f_i) = \varphi(\bar\partial h_i) = -\bar\partial \varphi(h_i)$ in the sense of distributions,
the left-hand side of \eqref{e:SG-CThat} equals  \eqref{e:hint3} up a multiplicative constant $(-iB)^n = (-i\sqrt{\pi})^n$.
(The published version of \cite{MR4767492} contains an unimportant sign error in the Bosonization identity; the correct
sign can be found in the arXiv version.)
We compute  \eqref{e:hint3} in terms of the Fourier transform.
Since
\begin{equation}
  \frac{1}{2\pi}K_0(\mu|x-y|)=(-\Delta+\mu^2)^{-1}(x,y) =
  \int \frac{dq}{(2\pi)^2} \, \frac{e^{i q \cdot (x-y)}}{|q|^2+\mu^2},
\end{equation}
first write
\begin{align}
  &\avg{\bar\psi_1\psi_2(x_1);\dots.;\bar\psi_1\psi_2(x_n)}^\mathsf T\nnb
  &=(-1)^{n+1}\sum_\sigma \left(-\frac{1}{\pi}\right)^n \prod_{i=1}^n \bar \partial_{x_{\sigma^i(1)}}\frac{1}{2\pi}\int dq_i \frac{e^{i q_i (x_{\sigma^i(1)}-x_{\sigma^{i+1}(1)})}}{|q_i|^2+\mu^2}\nnb
  &=(-1)^{n+1}\sum_\sigma \left(-\frac{i}{4\pi^2}\right)^n \prod_{i=1}^n \int dq_i \frac{q_i}{|q_i|^2+\mu^2}e^{i q_i (x_{\sigma^i(1)}-x_{\sigma^{i+1}(1)})}
    .
\end{align}
Using Lemma~\ref{lem:hint3}, we thus have
\begin{align}
&\int_{\C^n}dx \, h_1(x_1)\cdots h_n(x_n) \avg{\bar\psi_1\psi_2(x_1);\dots.;\bar\psi_1\psi_2(x_n)}^\mathsf T\nnb
&=(-1)^{n+1}\left(-\frac{i}{4\pi^2}\right)^n\sum_{\sigma}\int_{\C^n} dx\, \frac{1}{(2\pi)^{2n}}\int _{\C^n} dp \,\widehat h_1(p_1)\cdots \widehat h_1(p_n) e^{i\sum_{j=1}^n p_j x_j}\nnb
&\qquad \qquad \times \int_{\C^n}dq\, \prod_{j=1}^n \frac{q_j}{|q_j|^2+\mu^2} e^{i\sum_{j=1}^n q_j(x_{\sigma^j(1)}-x_{\sigma^{j+1}(1)})} \nnb
&=(-1)^{n+1}\left(-\frac{i}{4\pi^2}\right)^n\sum_{\sigma} \frac{1}{(2\pi)^{2n}}\int _{\C^n} dp\, \widehat h_1(p_1)\cdots \widehat h_1(p_n) \nnb
  &\qquad \qquad \times \int_{\C^n}dq\, \prod_{j=1}^n \frac{q_j}{|q_j|^2+\mu^2}(2\pi)^{2n}\prod_{j=1}^n \delta(p_j+q_{\sigma^{-j}(1)}-q_{\sigma^{-j+1}(1)}).
\end{align}
Renaming the $q$-variables according to $\xi_j=q_{\sigma^{-j}(1)}$, the delta-function constraint becomes
\begin{equation}
  \xi_2=p_1+\xi_1, \quad \xi_3=p_2+\xi_2=p_1+p_2+\xi_1, \quad \cdots, \quad \xi_n=\sum_{j=1}^{n-1}p_j+\xi_1,
\end{equation}
and for $j=n$, we find $\delta(\sum_{j=1}^n p_j)$. As nothing depends on $\sigma$ anymore and there are $(n-1)!$ cyclic permutations, this becomes 
\begin{align}
&(-1)^{n+1}\left(-\frac{i}{4\pi^2}\right)^n(n-1)! \int _{\C^{n-1}} dp \, \widehat h_1(p_1)\cdots \widehat h_n(-p_1-\cdots -p_{n-1}) \nnb
&\qquad \qquad \times \int_{\C}dq_1 \, \frac{q_1(q_1+p_1)\cdots (q_1+p_1+\cdots +p_{n-1})}{(|q_1|^2+\mu^2)\cdots (|q_1+p_1+\cdots p_{n-1}|^2+\mu^2)}.
\end{align}
Applying this with $\widehat f_j(p)=\widehat{\bar\partial h_j}(p)=\frac{i}{2}p\widehat h_j(p)$, we can write this as 
\begin{align}
&(-1)^{n+1}\left(-\frac{1}{2\pi^2}\right)^n(n-1)! \int _{\C^{n-1}} dp\, \widehat f_1(p_1)\cdots \widehat f_n(-p_1-\cdots -p_{n-1}) \nnb
&\qquad \times \frac{1}{p_1\cdots p_{n-1}(-p_1-\cdots-p_{n-1})} \int_{\C}dq_1  \,\frac{q_1(q_1+p_1)\cdots (q_1+p_1+\cdots +p_{n-1})}{(|q_1|^2+\mu^2)\cdots (|q_1+p_1+\cdots p_{n-1}|^2+\mu^2)}.
\end{align}
This is the desired claim.
\end{proof}

Together with the $n=2$ case already handled by \eqref{e:twopointpf},
the following lemma completes the proof of Proposition~\ref{prop:moments} for test functions $f_i = \bar\partial h_i$ with $h_i \in C_c^\infty(\R^2)$.

\begin{lemma} \label{lem:momentn3}
For $n\geq 3$,
\begin{equation}
  |\hat C_{\mu}^{\sf T}(p_1,\dots,p_{n-1})|
  \leq  \frac{C(\mu)^n n!}{|p_1|\cdots |p_{n-1}| |p_1+\cdots+p_{n-1}|}
\end{equation}
and %
\begin{equation}
  \avg{\varphi(f_1); \cdots; \varphi(f_n)}^{\sf T}_{\SG(4\pi,z)} \leq C(\mu)^nn! \prod_{j=1}^n \|f_j\|
,\qquad
  \|f\|= \|(1+|p|)|p|^{-1} \hat f\|_{L^{3/2}}.
\end{equation}
\end{lemma}

\begin{proof}
By Hölder's inequality,  for $n \geq 3$ and $\mu \neq 0$,
\begin{align}
&\int_{\C}dq  \left|\frac{q(q+p_1)\cdots (q+p_1+\cdots +p_{n-1})}{(|q|^2+\mu^2)\cdots (|q+p_1+\cdots p_{n-1}|^2+\mu^2)}\right|\nnb
&\leq \left(\int_{\C}dq \, \frac{|q|^n}{(|q|^2+\mu^2)^n}\right)^{1/n}\cdots \left(\int_{\C}dq\, \frac{|q+p_1+\cdots +p_{n-1}|^n}{(|q+p_1+\cdots +p_{n-1}|^2+\mu^2)^n}\right)^{1/n} \nnb
&=\int_{\C}dq \, \frac{|q|^n}{(|q|^2+\mu^2)^n}
  \leq \int_{|q|\leq 1}dq \, \frac{|q|^n}{\mu^{2n}} + \int_{|q|>1} dq \, \frac{1}{|q|^{n}}
  \leq C(\mu)^n.
\end{align}
Therefore,
\begin{equation}
  |\hat C_{\mu}^{\sf T}(p_1,\dots,p_{n-1})|
  \leq  \frac{C(\mu)^n n!}{|p_1|\cdots |p_{n-1}| |p_1+\cdots+p_{n-1}|}.
\end{equation}
Moreover, again by Young's convolution inequality, for any $f_1,\dots, f_n \in C_c^\infty(\R^2)$,
with the notation $\hat g(p) = \hat f(p)/|p|$,
\begin{align}
  &\int dp_1\, \cdots \, dp_{n-1} \, \frac{|\widehat f_1(p_1)|}{|p_1|}\cdots \frac{|\widehat f_n(p_1+\dots+p_{n-1})|}{|p_1+\cdots +p_{n-1}|}
    \nnb
  &= (\hat g_1* \cdots * \hat g_n)(0) \leq \prod_{i=1}^n \|\hat g_i\|_{L^{n/(n-1)}} %
    \leq \prod_{i=1}^{n} \left(\int dp \frac{|\widehat f_i(p)|^{\frac{n}{n-1}}}{|p|^{\frac{n}{n-1}}}\right)^{\frac{n-1}{n}}
    ,
\end{align}
which is finite for $n\geq 3$. It follows that (with a different constant)
\begin{equation}
  \avg{\varphi(f_1); \cdots; \varphi(f_n)}^{\sf T}_{\SG(4\pi,z)} \leq C(\mu)^nn! \prod_{j=1}^n \|f_j\|
\end{equation}
with
\begin{equation}
  \|f\|= \|(1+|p|)|p|^{-1} \hat f\|_{L^{3/2}}.
\end{equation}
Indeed, for $n=3$ this bound is immediate from the bound above.
For $n>3$, H\"older's inequality with $p=3(n-1)/(2n)$ and $1/q = 1-2n/(3(n-1)) =(n-3)/(3n-3)$ gives
\begin{align}
  \pa{\int dp\, |\hat g|^{n/(n-1)}}^{(n-1)/n}
  &= \pa{\int dp\, (1+|p|)^{-n/(n-1)} (1+|p|)^{n/(n-1)} |\hat g|^{n/(n-1)}}^{(n-1)/n} \nnb
  &\leq \pa{\int dp\, (1+|p|)^{3/2}|\hat g|^{3/2} \, dp}^{2/3} \pa{\int (1+|p|)^{-r}}^{1/r}
\end{align}
where
\begin{equation}
   r = \frac{n}{n-1} \frac{3n-3}{n-3}  \geq 3
\end{equation}
and
\begin{align}
  \pa{\int dp\, (1+|p|)^{-r}}^{1/r}
  \leq C
  \pa{1 + \frac{1}{r-2}}^{1/r}
  =
   C \pa{\frac{r-1}{r-2}}^{1/r} \leq C.
\end{align}
This completes the proof.
\end{proof}

To complete the proof, we now extend the statement to all test functions $f_1, \dots, f_n \in \cS(\C)$ with $\int f_i \, dx = 0$.
Since $C_c^\infty(\C)$ is dense in $\cS(\C)$ and $f_i \to 0$ in $\cS(\C)$ implies $\|f_i\| \to 0$,
it suffices to fix $f_1, \dots, f_n \in C_c^\infty(\C)$ with $\int f_j \, dx =0$.
We can write such an $f_j$ as 
\begin{equation}
f_j(x)=\bar \partial h_j(x),
\end{equation}
where  $h_j$ is the Cauchy transform of $f_j$:
\begin{equation}
h_j(x)=\frac{1}{\pi}\int_{\C}du \, \frac{f_j(u)}{x-u}.
\end{equation}
The Cauchy transform is not a Schwartz function (and certainly not compactly supported as we assumed above), but we have the following fact.

\begin{lemma}
  Let $f \in C_c^\infty(\C)$ with $\int f\, dx =0$. Then its Cauchy transform $h$ satisfies, for each $k\geq 0$,
  \begin{equation}
    \nabla^kh(x) = O(|x|^{-2-k}), \qquad |x|\to\infty,
  \end{equation}
  (where the implied constant may depend on $k$) and   $h\in C^\infty(\C)\cap L^p(\C)$ for any $p \in (1,\infty]$.
\end{lemma}

\begin{proof}
Since $\int_{\C}du \, f(u)=0$, as $x\to\infty$,
\begin{equation}
h(x)=\frac{1}{\pi x}\int_{\C}du f(u)+O(|x|^{-2})=O(|x|^{-2}).
\end{equation}
In fact, we have a convergent series expansion in $1/x$ for large enough $x$, and the bound on $\nabla^k h$ follows by termwise differentiation.
Smoothness of $h$ follows from the fact that
\begin{equation}
h(x)=\partial_x \frac{2}{\pi}\int_{\C}du f(u)\log |x-u|,
\end{equation}
so by induction (on the order of the differential operator), one can prove that for any differential operator $D$ we have 
\begin{equation}
Dh(x)=\frac{2}{\pi}\int_{\C}du\, D_u \partial_u f(u)\log |x-u|.
\end{equation}
Combining smoothness with our decay estimate at infinity, we get the claim.
\end{proof}

\begin{proof}[Proof of Proposition~\ref{prop:moments}]
By \eqref{e:SG-moments-limit}, it suffices to assume
$f_1, \dots, f_n \in C_c^\infty(\R^2)$ satisfy $\int f_i \, dx = 0$.
Let $\chi_R \in C_c^\infty(\R^2)$ be equal to $1$ in $B_{R}(0)$, vanish outside $B_{2R}(0)$,
and satisfy $|\nabla^k\chi_R|\leq C_k R^{-k}$.
Given the Cauchy transforms $h_1, \dots, h_n$ of $f_1, \dots, f_n$,
we can apply Lemma~\ref{lem:momentn3} and the discussion preceding it
with $h_j$ replaced by $h_j \chi_R\in C_c^\infty(\R^2)$.
For the general statement, we then need to show that
\begin{equation}
  \avg{\varphi(f_1)\cdots \varphi(f_n)}_{\SG(\beta,z)} =  \lim_{R\to\infty} \avg{\varphi(\bar\partial (h_1\chi_R))\cdots \varphi(\bar\partial (h_n\chi_R))}_{\SG(\beta,z)}.
\end{equation}
For sufficiently large $R$ (such that $B_R(0)$ contains the support of $f_j$), we have
\begin{equation}
  \bar\partial (h_j \chi_R) = (\bar\partial h_j)\chi_R + h_j \bar\partial \chi_R = f_j + h_j \bar\partial \chi_R.
\end{equation}
Since $\int f_j \, dx =0$ by assumption and $\int \bar\partial(h_j\chi_R)\, dx =0$ also $\int h_j \bar\partial\chi_R \, dx = 0$.
Since, for any $r<1$ (see \cite[Appendix A]{2504.08606} for the definition of the norm, there denoted $|\!|\!|\cdot |\!|\!|_{r,\rho^{-1}}$),
and the definition of the weight \eqref{e:weight-def},
\begin{equation}
  \|h_j \bar\partial \chi_R \|_{B^{r}_{1,1}(\rho^{-1})} \lesssim \int_{B_{2R}(0)\setminus B_R(0)} \rho(x)^{-1} \frac{1}{|x|^2} \frac{1}{|x|} \,dx \lesssim R^{-1+\sigma}
\end{equation}
we have by duality of $C^{-s}(\rho)$ and $B^{r}_{1,1}(\rho^{-1})$ where $r>s$, see \cite[Remark~A.2]{2504.08606},
\begin{equation}
  \varphi(h_j \bar\partial\chi_R)
  \lesssim \|\varphi-(\varphi,\eta)\|_{C^{-s}(\rho)} R^{-1+\sigma} \to 0,
\end{equation}
where we used $\int h_j \bar\partial\chi_R \, dx =0$ to replace $\varphi$ by $\varphi-(\varphi,\eta)$.
The claim thus follows from \eqref{e:SG-Besov}.
\end{proof}

\section{Background on Gaussian free field}
\label{sec:Gauss}

In this section, we collect properties of the Gaussian free field, its scale decomposition,
its correlation functions, and the associated imaginary multiplicative chaos.
Most properties are relatively standard and proofs are given in Appendices~\ref{app:GMCGFF}--\ref{app:renormpart} when no convenient reference is available.

\subsection{Scale decomposition}
\label{sec:Gauss-decomp}

In many places throughout the article, we will need a scale decomposition for the Gaussian free field (GFF).
For concreteness, we fix the heat kernel decomposition, i.e., for $m > 0$, we decompose the covariance of the massive GFF as
\begin{equation}
  (-\Delta_x +m^2)^{-1} = \int_0^\infty \dot C_s \, ds = \int_0^\infty e^{-m^2 s} e^{\Delta_x s} \, ds.
\end{equation}
The decomposed field can be realized as a process
\begin{equation} \label{e:GFF-Wiener}
  \Phi^{\GFF(m)}_t = \int_t^\infty \sqrt{\dot C_s}\, dW_s = \int_t^\infty e^{-\frac12 m^2 s} e^{\frac12 \Delta_x s}\, dW_s
\end{equation}
where $(W_t)$ is a cylindrical Brownian motion on $L^2(\R^2)$ defined on some probability space fixed from now on,
see for example \cite[Chapters 4.1.2 and 4.3]{MR3236753} for background on this, and \cite{MR4798104}
for discussion of such a scale decomposition.
The process $\Phi^{\GFF(m)}$ takes values in $C((0,\infty),C^\infty(\R^2))$,
and various more quantiative regularity estimates we need are included in Appendix~\ref{app:Gauss}.

The full massive Gaussian free field on $\R^2$ is then realized as $\Phi^{\GFF(m)}_0$ and $\Phi_{\epsilon^2}^{\GFF(m)}$ provides
a small scale regularized version of the field (along with a large scale regularization provided by the mass $m>0$).
This regularization will be referred to as the heat kernel regularized GFF.
On the other hand, a large scale (heat kernel) regularized version of the GFF is given by
\begin{equation}
  \Phi_{0,t}^{\GFF(m)}  = \Phi^{\GFF(m)}_0-\Phi_t^{\GFF(m)} = \int_0^t \sqrt{\dot C_s} \, dW_s
\end{equation}
with covariance
\begin{equation}
  \int_0^t e^{-m^2 s} e^{\Delta_x s} \, ds.
\end{equation}
Since the upper bound by $t$ plays an analogous role to the mass scale $1/m^2$, the massless version of $\Phi^{\GFF(m)}_0-\Phi^{\GFF(m)}_t$
where $m=0$ is well defined for $t<\infty$ as well.

In addition to the heat kernel regularized GFF, we will also use 
the convolution regularized GFF (for example in the next subsection),
defined by $\eta_\epsilon * \Phi^{\GFF(m)}_0$ where $\eta \in C_c^\infty(\R^2)$
is a smooth radially symmetric mollifier and $\eta_\epsilon=\epsilon^{-2}\eta(x/\epsilon)$.

\subsection{Gaussian IMC and its Besov regularity}

Let  $\sqrt{2\pi}\varphi$  be a log-correlated Gaussian random field on $\R^2$ which for simplicity of the discussion
we assume to be translation invariant (stationary),
having covariance kernel
\begin{equation}
  \log\frac{1}{|x-y|} + g(x-y),
\end{equation}
with $g$ continuous. In particular, this applies 
with $\varphi = \Phi_0^{\GFF(m)}-\Phi_{t}^{\GFF(m)}$
with $(t,m) \in (0,\infty) \times [0,\infty) \cup (0,\infty] \cup (0,\infty)$,
where the notation is as in Section~\ref{sec:Gauss-decomp}.

The imaginary multiplicative chaos (IMC) is the limit $\epsilon\to 0$ in probability, when it exists, of the random generalized function (distribution)
\begin{equation} \label{e:Meps-bis}
  M_\alpha^\epsilon(f) = \wick{e^{i\alpha \sqrt{4\pi} \varphi_\epsilon}}_\epsilon(f) =  \epsilon^{-\alpha^2} \int e^{i\alpha \sqrt{4\pi}\varphi_\epsilon(x)} f(x)\, dx.
\end{equation}
The regularized field $\varphi_\epsilon$ is defined as $\varphi_\epsilon = \eta_\epsilon*\varphi$ with a nonnegative mollifier $\eta \in C_c^\infty(\R^2)$
and $\eta_\epsilon(x) = \epsilon^{-2} \eta(x/\epsilon)$,
and we assume $\eta$ is rotationally invariant and has support in the unit ball.

The IMC was introduced as the imaginary Gaussian multiplicative chaos in \cite{MR3339158,MR4149524}, but we use the term IMC
which is more natural when the underlying field is non-Gaussian (as our goal is to study its version with respect to the sine-Gordon measure).
In the above Gaussian setting,
the limit $M_\alpha$  exists in $\cD'(\R^2)$ if and only if $\alpha\in (-1,1)$, and the limit is independent of $\eta$ except for a multiplicative constant that could depend on $\eta$.
In the standard definition, the IMC would be normalized by $e^{\frac12 \alpha^2 \var(\sqrt{2\pi}\varphi_\epsilon(0))}$ instead of $\epsilon^{-\alpha^2}$.
The variance incorporates a mollifier dependent constant, which makes the corresponding limit independent of $\eta$,
but for our use it will be better to use the explicit renormalization.
Indeed, the variance regularization is not the correct one in the massless limit -- it would not lead to a nontrivial limit (or require another counterterm $m^{\alpha^2}$).

The Besov--H\"older regularity of the IMC  established in \cite{MR4149524} will play an important role for us,
and we now recall the statements in our context.
For definitions and properties of Besov spaces, see Section~\ref{sec:notation}. 
We write $C^{-s}_{\loc}(\R^2)$ %
for the space defined by the seminorms 
\begin{equation} \label{e:Cs-seminorm}
  \|\chi f\|_{C^{-s}}, \qquad \chi \in C_c^\infty(\R^2).
\end{equation}
The Besov--H\"older norms have the following important multiplication property
which we use for $r \in(0,1)$ and $s \in (0,1)$ with $r-s>0$ (see for example \cite[Proposition~A.3]{2504.08606}):
\begin{equation} \label{e:Besov-mult}
  \|fg\|_{C^{-s}} \lesssim \|f\|_{C^{-s}}\|g\|_{C^r}.
\end{equation}
It suffices to consider a countable family of $\chi$ in \eqref{e:Cs-seminorm}, for example defined such that $\chi_n(x)=1$
for $x\in B_n(0)$ and $n\in \N$, and $C^{-s}_{\loc}$ is metrizable, see for example \cite[Remark~2.20]{MR3724565}.

\begin{proposition} \label{prop:GMC}
  For any $s>\alpha^2$,
  the regularized Gaussian imaginary multiplicative chaos \eqref{e:Meps-bis}
  defined in terms of a log-correlated Gaussian random  field  
  converges 
  as an element of the local Besov--H\"older space $C^{-s}_\loc$ to a limit $M_\alpha \in C^{-s}_{\loc}$ in probability.
  In particular, for any $\delta>0$:
  \begin{equation}
    \lim_{\epsilon \to 0}\P(\|\chi M^\epsilon_\alpha-\chi M_\alpha\|_{C^{-s}}>\delta)=0 \qquad \text{for all $\chi\in C_c^\infty(\R^2)$.}
  \end{equation}
  For any $p<\infty$ and $\chi\in C_c^\infty(\R^2)$, the limit satisfies the moment bounds
  \begin{equation} \label{e:GMC-moments}
    \E\qa{\|\chi M_\alpha\|_{C^{-s}}^p} < \infty.
  \end{equation}
  The law of the limit is independent of the mollifier up to a deterministic multiplicative constant.
\end{proposition}

The proposition is essentially proved in  \cite{MR4149524}, but since the statement is not exactly the same, we discuss
the adaption in Appendix~\ref{app:GMC}.

It will be helpful that $M_{\alpha}$ %
can be constructed through the convolution regularization (as in Proposition~\ref{prop:GMC})
or alternatively through a heat kernel regularization as in Section~\ref{sec:Gauss-decomp}.
This fact is recorded in the following proposition whose proof is provided in Appendix~\ref{app:GMCGFF-cov}.

\begin{proposition}\label{prop:hkimc}
For $f\in C_c^\infty(\R^2)$, $\alpha\in (-1,1)$, let $M_\alpha(f)$ be the limit of \eqref{e:Meps-bis} with $\varphi=\Phi_0^{\GFF(m)}$:
\begin{equation}
M_\alpha(f)=\lim_{\epsilon\to 0} \epsilon^{-\alpha^2}\int_{\R^2}dx\, e^{i\sqrt{4\pi}\alpha(\eta_\epsilon * \Phi_0^{\GFF(m)})(x)}f(x),
\end{equation}
with convergence in probability. 
Then this convergence also holds in $L^p$ for each $p\in [1,\infty)$ and 
\begin{equation}
  M_\alpha(f) = C_{\eta}^{\alpha^2}\lim_{\epsilon\to 0}  \epsilon^{-\alpha^2}\int_{\R^2}dx\, e^{i\sqrt{4\pi}\alpha\Phi_{\epsilon^2}^{\GFF(m)}(x)}f(x),
\end{equation}
with convergence in $L^p$ for each $p\in [1,\infty)$, and  the constant is given by
\begin{equation} \label{e:Ceta}
  C_\eta^{\alpha^2}=e^{-\alpha^2(-\frac{\gamma}{2}+\log 2+\int_{\R^2\times \R^2}dudv\, \eta(u)\eta(v)\log \frac{1}{|u-v|})},
\end{equation}
where $\gamma$ is the Euler-Mascheroni constant.
\end{proposition}

\subsection{Fractional correlation functions for the massless GFF}
\label{sec:Gauss-corr}

For distinct points $x_1,\dots, x_n \in \R^2$ and $\alpha_1,\dots,\alpha_n \in \C$ we define the fractional
charge (or vertex operator) correlation functions of the massless free field by
\begin{align} \label{e:gffcorr}
  &  \avg{\wick{e^{i\alpha_1  \sqrt{4\pi} \varphi(x_1)}} \cdots \wick{e^{i\alpha_n \sqrt{4\pi}  \varphi(x_n)}}}_{\GFF(0)}
  \nnb
  &=
  \lim_{m \to 0}\lim_{\epsilon \to 0}
  \E\left[\wick{e^{i\alpha_1 \sqrt{4\pi}  \Phi_{\epsilon^2}^{\GFF(m)}(x_1)}}_\epsilon \cdots \wick{e^{i\alpha_n \sqrt{4\pi} \Phi_{\epsilon^2}^{\GFF(m)}(x_n)}}_\epsilon\right]
\end{align}
where
\begin{equation}\label{eq:fracwick}
  \wick{e^{i\alpha \sqrt{4\pi}  \Phi_{\epsilon^2}^{\GFF(m)}(x)}}_\epsilon = \epsilon^{-\alpha^2} e^{i\alpha \sqrt{4\pi} \Phi_{\epsilon^2}^{\GFF(m)}(x)}.
\end{equation}
The existence of this limit (and an explicit formula for it) is proven in Lemma \ref{le:gfffrac}. We also define analogous correlation functions involving trigonometric functions by writing (within correlation functions) $\wick{\cos(\sqrt{4\pi}\alpha\varphi(x))}=\frac{1}{2}(\wick{e^{i\sqrt{4\pi}\alpha\varphi(x)}}+\wick{e^{-i\sqrt{4\pi}\alpha\varphi(x)}})$ and similarly for the regularized field.
Truncated correlation functions (cumulants) can be defined in the usual way and are denoted
\begin{equation}
  \avg{\wick{e^{i\alpha_1  \sqrt{4\pi} \varphi(x_1)}} ; \cdots ; \wick{e^{i\alpha_n \sqrt{4\pi}  \varphi(x_n)}}}_{\GFF(0)}^{\sf T}.
\end{equation}
Our definition of  \eqref{e:gffcorr} is in terms of the heat kernel regularized correlation functions;
an alternative definition in terms of the convolution regularized GFF would involve the factors \eqref{e:Ceta}.
In particular, for $\alpha_i \in (-1,1)$ for which the imaginary multiplicative chaos $M_{\alpha_i}$ is defined as a random distribution,
it is not difficult to show that for $f_1,\dots, f_n \in C_c^\infty(\R^2)$,
\begin{multline}
  \lim_{m\to 0}\avg{M_{\alpha_1}(f_1) \cdots M_{\alpha_n}(f_n)}_{\GFF(m)}
  \\
  =  C_\eta^{\alpha_1^2} \cdots C_\eta^{\alpha_n^2}
  \int f_1(x_1) \cdots f_n(x_n)
  \avg{\wick{e^{i\alpha_1  \sqrt{4\pi} \varphi(x_1)}} \cdots \wick{e^{i\alpha_n \sqrt{4\pi}  \varphi(x_n)}}}_{\GFF(0)} \, dx_1 \dots dx_n
  .
\end{multline}
We will need an extension of this statement in which some of the $\alpha_i$ are not in $(-1,1)$,
and the imaginary multiplicative chaos does not exist.
The precise statement
that we need is the following proposition, proved in Appendix~\ref{app:renormpart}. It is essential that this statement
is for truncated correlation functions (as the one-point function can diverge for $\alpha_i \not\in (-1,1)$, see \cite[Proposition~1.4]{MR4767492}).

\begin{proposition}\label{th:fracapp-bis}
  Let $\beta \in (0,6\pi)$ and $\alpha_1,\dots,\alpha_n \in \R$ with $\alpha_i^2 < \min\{1,2-\beta/4\pi\}$. Then for $p\geq 0$,
  \begin{equation}
    \avg{\wick{e^{i\alpha_1 \sqrt{4\pi} \varphi(x_1)}}\cdots \wick{e^{i\alpha_n \sqrt{4\pi} \varphi(x_n)}}; \wick{\cos(\sqrt{\beta}\varphi(y_1))}; \cdots; \wick{\cos(\sqrt{\beta}\varphi(y_p))}}_{\GFF(0)}^{\mathsf T}\label{e:finvol-corfunc-deriv-explicit-bis}
  \end{equation}
  is locally integrable in $x_1,\dots, x_n, y_1, \dots, y_p\in \C$. Moreover,
  \begin{multline}\label{eq:gffregcf}
    \E^\mathsf T\bigg[\wick{e^{i\alpha_1\sqrt{4\pi}\Phi_{\delta^2}^{\GFF(m)}(x_1)}}_\delta\cdots \wick{e^{i\alpha_n\sqrt{4\pi}\Phi_{\delta^2}^{\GFF(m)}(x_n)}}_\delta;
      \\
      \wick{\cos(\sqrt{\beta}\Phi_{\epsilon^2}^{\GFF(m)}(y_1))}_\epsilon;\dots; \wick{\cos(\sqrt{\beta}\Phi_{\epsilon^2}^{\GFF(m)}(y_p))}_\epsilon \bigg]%
  \end{multline}
  is locally uniformly integrable (with uniformity in $0<\delta<\epsilon<1$ and $0 < m<1$)
  and for $f_1,\dots,f_n\in C_c^\infty(\R^2)$ with disjoint supports, and $\Lambda\subset \R^2$ compact 
  \begin{align} \label{eq:gffregcfconv}
 &\lim_{m\to 0}\lim_{\epsilon\to 0}\lim_{\delta\to 0}\E^\mathsf T\bigg[\wick{e^{i\alpha_1\sqrt{4\pi}\Phi_{\delta^2}^{\GFF(m)}}}_\delta(f_1)\cdots \wick{e^{i\alpha_n\sqrt{4\pi}\Phi_{\delta^2}^{\GFF(m)}}}_\delta (f_n);\nnb
&\qquad \qquad\qquad \qquad \wick{\cos(\sqrt{\beta}\Phi_{\epsilon^2}^{\GFF(m)})}_\epsilon(\1_\Lambda);\dots; \wick{\cos(\sqrt{\beta}\Phi_{\epsilon^2}^{\GFF(m)})}_\epsilon (\1_\Lambda) \bigg]  \nnb
    &=\int_{\R^{2n}\times \Lambda^p}dxdy\, f_1(x_1)\cdots f_n(x_n)\avg{\wick{e^{i\alpha_1\sqrt{4\pi}\varphi(x_1)}}\cdots \wick{e^{i\alpha_n\sqrt{4\pi}\varphi(x_n)}};
      \nnb
      &\qquad\qquad\qquad\qquad\wick{\cos(\sqrt{\beta}\varphi(y_1))};\dots; \wick{\cos(\sqrt{\beta}\varphi(y_p))}}_{\GFF(0)}^\mathsf T\nnb
      &=:    \avg{\wick{e^{i\alpha_1 \sqrt{4\pi} \varphi}}(f_1)\cdots \wick{e^{i\alpha_n \sqrt{4\pi} \varphi}}(f_n); \wick{\cos(\sqrt{\beta}\varphi)}(\1_\Lambda); \cdots; \wick{\cos(\sqrt{\beta}\varphi)}(\1_\Lambda)}_{\GFF(0)}^{\mathsf T}.
  \end{align}
  Here $\E^\mathsf T$ denotes a cumulant and there are $p$ cosine-terms in the cumulants.
\end{proposition}

Note that all of the $\alpha$-entries are in the same slot in the above cumulants (see Section~\ref{sec:notation}),
while each of the $\beta$-entries is in its own slot.

We also mention that we expect that the restrictions on the parameters could be relaxed to some degree, but as this is the setting we need, we do not explore this question further.

\section{Construction and regularity of imaginary multiplicative chaos}
\label{sec:imc}

In this section, we
construct the imaginary multiplicative chaos (IMC) with respect to the sine-Gordon field,
both in the massive finite-volume case and in the massless infinite-volume case,
and establish the important properties we need -- in particular, analyticity in finite volume,
its series expansion, and convergence to the infinite-volume version.
The proofs rely on the decomposition results for the sine-Gordon measures proved
in Sections~\ref{sec:coupling-finvol}--\ref{sec:coupling-infvol} and summarized here.

\subsection{Finite- and infinite-volume sine-Gordon IMC}

The massive finite-volume sine-Gordon measure $\nu^{\SG(\beta,z|\Lambda,m)}$  and its massless infinite-volume limit $\nu^{\SG(\beta,z)}$
were defined in Sections~\ref{sec:intro-SG} and~\ref{sec:mixing}.

We will prove that the IMC, the limit of \eqref{e:Meps-bis} as a random distribution (generalized function), exists with respect to both
$\varphi \sim \nu^{\SG(\beta,z|\Lambda,m)}$ and $\varphi\sim \nu^{\SG(\beta,z)}$ and establish regularity properties of the corresponding limits that
are essential for the proof of the main results.
In particular, we prove analyticity in $z$ in a neighborhood $I$ of the real axis of the moments of $M_\alpha$ when $\Lambda \subset \R^2$ is bounded,
and we prove convergence of the law of $M_\alpha$ in the massless infinite-volume limit.

We begin with the finite-volume statement that we need. We only require the case $\beta=4\pi$, but the proof is the same for all $\beta\in (0,6\pi)$.
The assumption on $\Lambda$ is more general than the case we require
but could easily be relaxed further.

\begin{theorem} \label{thm:sg-finvol}
  Let $\beta\in (0,6\pi)$,
  let $\Lambda \subset \R^2$ be a bounded domain,
  and denote by $\avg{\cdot}_{\SG(\beta,z|\Lambda,m)}$
  the expectation of the massive finite-volume sine-Gordon measure with mass $m>0$.
  There exists a complex neighborhood $I$ of $\R$ (which may depend on $\Lambda \subset \R^2$ and $\beta \in (0,6\pi)$ but is independent of $m>0$)
  such that the following hold for any $m>0$:
  \smallskip
  
  \noindent
  (i) For $m>0$, $z\in \R$, and $\alpha\in \R:$ $\alpha^2 < \min\{1,2-\beta/4\pi\}$,
  the imaginary multiplicative chaos $M_\alpha$ exists as a random distribution
  with respect to the measure $\nu^{\SG(\beta,z|\Lambda,m)}$,
  obtained as the $\epsilon \to 0$ limit in probability of its convolution regularized version \eqref{e:Meps-bis}.
  The limit does not depend on the mollifier $\eta$ except for a multiplicative constant (the same as in Proposition~\ref{prop:hkimc}).

  \smallskip\noindent
  (ii)
  $M_\alpha$ has moments of all order: for any $p<\infty$, any $s>\alpha^2$, and any $\chi\in C_c^\infty(\R^2)$,
  \begin{equation} \label{e:MSG-moments}
    \avg{\|\chi M_\alpha\|_{C^{-s}}^p}_{\SG(\beta,z|\Lambda,m)} < \infty.
  \end{equation}

  \smallskip\noindent
  (iii)
  For $\alpha_1,\dots, \alpha_n \in \R$ satisfying $\alpha_i^2 < \min\{1,2-\beta/4\pi\}$ 
  and any $f_1,\dots, f_n \in C_c^\infty(\R^2)$, %
  the smeared correlation functions
  \begin{equation} \label{e:coupling-finvol-corfunc}
    \avg{M_{\alpha_1}(f_1)\cdots M_{\alpha_n}(f_n)}_{\SG(\beta,z|\Lambda,m)}
  \end{equation}
  extend to analytic functions in $z\in I$ and
  satisfy bounds that are uniform in $m>0$.

  \smallskip\noindent
  (iv)
  The same holds for the correlation functions of the regularized measure
  $\avg{\cdot}_{\GFF+\SG(\beta,z|\Lambda,m,\epsilon)}$ defined precisely in \eqref{e:SGeps-2} below,
  uniformly in $\epsilon \in (0,1]$. Moreover,
  for any $\rho>0$,
  \begin{equation} \label{e:Malpha-uniform-bd}
    \sup_{z\in I \cap B_\rho(0)}\sup_{\epsilon,m\in(0,1)}|\avg{M_{\alpha_1}(f_1)\cdots M_{\alpha_n}(f_n)}_{\GFF+\SG(\beta,z|\Lambda,m,\epsilon)}|<\infty,
  \end{equation}
  and
  \begin{equation} \label{e:corrfunc-as-limit}
    \avg{M_{\alpha_1}(f_1)\cdots M_{\alpha_n}(f_n)}_{\SG(\beta,z|\Lambda,m)}
    =
    \lim_{\epsilon \to 0} %
    \avg{M_{\alpha_1}(f_1)\cdots M_{\alpha_n}(f_n)}_{\GFF+\SG(\beta,z|\Lambda,m,\epsilon)}.
  \end{equation}

\end{theorem}

The most essential property that we need later for the identification of the sine-Gordon correlation functions with
their fermionic counterparts is the analyticity of the correlation functions in the coupling constant $z$.
This relies on $|\Lambda|<\infty$ (and is false in the infinite-volume limit).

\medskip

The infinite-volume statement given in the next theorem is stated only for $\beta=4\pi$.
The proof could be extended to $\beta\in (0,6\pi)$ under the condition \eqref{e:SG-moments},
see Section~\ref{sec:coupling-infvol},
which we have however only verified when $\beta=4\pi$ (using Bosonization).
The theorem implies Theorem~\ref{thm:Mexists}.

\begin{theorem} \label{thm:sg-infvol}
  For $z\in \R$, let $\avg{\cdot}_{\SG(4\pi, z)}$ denote the expectation with respect to
  the massless infinite-volume sine-Gordon measure at the free fermion point, see Theorem~\ref{thm:SG4pi}.

  \smallskip\noindent
  (i) For $\alpha\in (-1,1)$, the imaginary multiplicative chaos $M_\alpha$ exists as a random distribution
  with respect to the measure $\nu^{\SG(\beta,z)}$,
  again obtained as the $\epsilon \to 0$ limit in probability of its convolution regularized version \eqref{e:Meps-bis}.
  The limit does not depend on the mollifier $\eta$ except for a multiplicative constant (the same as in Proposition~\ref{prop:hkimc}).

  \smallskip\noindent
  (ii)
  $M_\alpha$ has moments of all order: for any $p<\infty$, any $s>\alpha^2$, and any $\chi\in C_c^\infty(\R^2)$,
  \begin{equation} \label{e:infvol-moments}
    \avg{\|\chi M_\alpha\|_{C^{-s}}^p}_{\SG(4\pi,z)} < \infty.
  \end{equation}

  \smallskip\noindent
  (iii)
  $M_\alpha$ is translation and rotation invariant, i.e.,
  $T M_\alpha$ has the same distribution as $M_\alpha$ where $T$ is the action of a deterministic translation or rotation.

  \smallskip\noindent
  (iv)
  For any $\alpha_1,\dots, \alpha_n \in (-1,1)$ with $\sum_i \alpha_i=0$ and $f_1,\dots, f_n \in C_c^\infty(\R^2)$,
 \begin{equation} \label{e:MSG-conv-massless}
    \avg{M_{\alpha_1}(f_1)\cdots M_{\alpha_n}(f_n)}_{\SG(4\pi,z)}
    = \lim_{\Lambda\to\R^2}\lim_{m\to 0}     \avg{M_{\alpha_1}(f_1)\cdots M_{\alpha_p}(f_p)}_{\SG(4\pi,z|\Lambda,m)},
  \end{equation}
  where the limit on the right-hand side is along suitable subsequences $\Lambda \to \R^2$ and $m \to 0$.
\end{theorem}

\subsection{Decompositions of the sine-Gordon field}

The proofs of Theorems~\ref{thm:sg-finvol} and~\ref{thm:sg-infvol} rely on decompositions of the sine-Gordon field
and a strong convergence statement,
proved in Sections~\ref{sec:coupling-finvol} and~\ref{sec:coupling-infvol}.
In both situations, finite and infinite volume, we construct the sine-Gordon field and the Gaussian free field (GFF) on the same probability space as
\begin{equation}
\varphi = Z+ \tilde \varphi  
\end{equation}
where $Z$ is essentially a massive Gaussian free field and $\tilde\varphi$ is a H\"older continuous random field.

For $z\in \R$ and $|\Lambda|<\infty$ or $m>0$,
this construction is similar to the one in \cite{MR4399156}
and the closely related constructions such as \cite{2401.13648} or the SPDE based constructions \cite{2410.15493}.
For our purposes, the important novelties are the analytic extension of such a decomposition to $z\not\in \R$, in finite volume,
and a version for the massless infinite-volume case if $z\in \R$, along with the convergence estimates as $m\to 0$ and $\Lambda \to \R^2$
that are sufficiently strong to guarantee convergence of the IMC.
The infinite-volume result relies crucially on input from the Bosonization results derived in Section~\ref{sec:mixing}.

\paragraph{Finite-volume statement}
We begin with the statement for the decomposition of the massive finite-volume sine-Gordon field.

\begin{proposition} \label{prop:finvol-coupling} 
  Let $\beta\in(0,6\pi)$, and
  let $\Lambda \subset \R^2$ be a bounded domain.
  Then there is a complex neighborhood $I$ of the real axis and a function $\rho\in [0,\infty) \mapsto \eta(\rho) \in [0,\infty)$,
  both depending on $\Lambda$ and $\beta$ but independent of $m \in (0,1]$,
  and a probability space on which
  random variables $Z$ and $\tilde \varphi = \tilde \varphi(z,m)$ with $z\in I$ and $m\in (0,1]$ are defined such that
  the following hold.

  \smallskip
  \noindent
  (i) For $z\in \R$, the massive finite-volume sine-Gordon field can be decomposed as
  \begin{equation}
    Z+\tilde \varphi \sim \nu^{\SG(\beta,z|\Lambda,m)}.
  \end{equation}

  \smallskip
  \noindent
  (ii)
  The random field $Z$ is a real-valued log-correlated Gaussian random field satisfying $Z \in C^{-s}_\loc$ almost surely
  for any $s>0$. It is independent of $z \in I$.

  \smallskip
  \noindent
  (iii)
  The random field  $\tilde \varphi$   is locally H\"older continuous with $|\im \tilde\varphi| \leq \eta(\rho)$ if $|z|\leq \rho$
  and $\im \tilde\varphi=0$ if $z\in \R$ and the following moment bound holds for any $p>0$   and $\chi\in C_c^\infty(\R^2)$
  and for any fixed $r<2-\beta/4\pi$ and any fixed $m\in (0,1]$:
  \begin{equation} \label{e:finvol-coupling-moment}
    \E\qB{\|\chi \tilde \varphi\|_{C^{r}}^p} \leq C_p(\beta,z,\Lambda,m) < \infty,
  \end{equation}
  and the H\"older seminorm is bounded uniformly in $m\in (0,1]$:
  \begin{equation}  \label{e:finvol-coupling-moment-unif-m}
    \E\qB{ [\chi\tilde\varphi]_{C^r}^p} \leq C_p(\beta,z,\Lambda) < \infty.
  \end{equation}

  \smallskip
  \noindent
  (iv)
  The map $z\in I \mapsto \chi\tilde\varphi(z) \in C^{r}(\R^2)$ is complex analytic almost surely.

  \smallskip
  \noindent
  (v) The same holds with $\tilde\varphi$ replaced by a regularized version $\tilde\varphi_\epsilon$
  such that $Z+\tilde\varphi_\epsilon \sim \nu^{\GFF+\SG(\beta,z|\Lambda,m,\epsilon)}$
  when $z\in \R$,
  defined in \eqref{e:SGeps-2} below, with all bounds uniform in $\epsilon$
  and convergence in $C^r$ as $\epsilon \to 0$.
\end{proposition}

For $z\not \in \R$, the proposition extends the finite-volume massive sine-Gordon measure as
a probability measure on complex-valued fields rather than as a complex measure on real-valued fields as one might expect
from \eqref{e:SGdef-bis}.
Nonetheless we show that expectations of analytic functionals under this measure agree with those of the
sine-Gordon expectation defined as in \eqref{e:SGdef-bis} with complex $z$ in the density.
The latter is \emph{not} a positive (or probability) measure for $z\not\in\R$.
Its small scale regularized version would be a complex measure (but supported on real-valued fields),
but we actually expect that for $\beta \geq 4\pi$ the $\epsilon\to 0$ limit does not even exist as a complex measure (while our probabilistic version does).
For analytic observables  both constructions provide the unique analytic continuation of the sine-Gordon measure with real coupling constants.
Under our above decomposition, the difference between the sine-Gordon field and the GFF is H\"older continuous and in $C^{r}_{\rm loc}$ for any $r<2-\beta/4\pi$, even when $z$ is complex.

\paragraph{Infinite-volume statements}
The infinite-volume statements involve weighted H\"older spaces, whose definition we recall from Section~\ref{sec:notation}.
We also recall the following moment bound established in Section~\ref{sec:mixing} when $\beta=4\pi$: for any $g\in C_c^\infty(\R^2)$,
\begin{equation} \label{e:SG-moments-bis2}
  \sup_{T\geq 1}
  \avg{|(\varphi,g-\hat g(0)\eta_T)|^p}_{\SG(\beta,z)} \leq C_p(\beta,z)\|g\|^p.
\end{equation}

\begin{proposition}  \label{prop:coupling-infvol}
  Let  $\beta\in (0,6\pi)$, $z\in \R$, and assume the moment bound
  \eqref{e:SG-moments-bis2} holds with $p\geq 2$.
  Then there is a probability space and random variables $Z$ and $\tilde \varphi$ such that the following hold.

  \smallskip
  \noindent
  (i) The massless sine-Gordon field is decomposed as
  \begin{equation}
    Z+\tilde\varphi \sim \nu^{\SG(\beta,z)}.
  \end{equation}
  (ii) The random field $Z$ has the same properties as in Proposition~\ref{prop:finvol-coupling}.
  
  \smallskip\noindent
  (iii)  For any $r<2-\beta/4\pi$ and weight of the form $\rho(x)=(1+|x|^2)^{-\sigma/2}$ for large enough $\sigma$,
  the field $\tilde\varphi$ is a random element in $C^r(\rho)$ satisfying (for the same $p$ as in the assumption)
  \begin{equation}
    \E\qa{\|\tilde\varphi\|_{C^r(\rho)}^p} \leq C_{p,r}(\beta, z).
  \end{equation}
\end{proposition}

The next proposition addresses a strong form of convergence of the finite-volume sine-Gordon  measure to the infinite-volume one (strong enough to imply the convergence of the IMC).

\begin{proposition} \label{prop:coupling-infvol-convergence}
  Under the same assumptions as in Proposition~\ref{prop:coupling-infvol}, one can find sequences $m_i \to 0$ and $\Lambda_j \to \R^2$ and $T_k \to \infty$
  and a probability space and random variables $Z$ and $\tilde \varphi_{i,j,k}$ and $\tilde \varphi$ such that the following hold.

  \smallskip\noindent
  (i)
  The distribution $(P_{T_k})_\#\nu^{\SG(\beta,z|\Lambda_j,m_i)}$  of $P_{T_k}\varphi$ where $\varphi \sim \nu^{\SG(\beta,z|\Lambda_j,m_i)}$ and $P_T$ is defined in \eqref{e:Peta-def}
  coincides with the distribution of $Z + \tilde \varphi_{i,j,k}$:
  \begin{equation}
     \varphi_{i,j,k}= Z + \tilde \varphi_{i,j,k} \sim (P_{T_k})_\#\nu^{\SG(\beta,z|\Lambda_j,m_i)}.
  \end{equation}
  \smallskip\noindent
  (ii) The random field $Z$ has the same properties as in Proposition~\ref{prop:finvol-coupling}
  and is independent of $i,j,k$.

  \smallskip
  \noindent
  (iii) The random field  $\tilde \varphi_{i,j,k}$ is a random element in $C^r(\rho)$, and there are real-valued random variables $X_{i,j,k}$
  and a random element $\tilde\varphi$ in $C^r(\rho)$ such that
  \begin{equation}
    \lim_{k\to\infty}\lim_{j\to\infty}\lim_{i\to\infty}  \|\tilde \varphi_{i,j,k} - \tilde \varphi  + X_{i,j,k}\|_{C^r(\rho)} \to 0,
  \end{equation}
  almost surely,
  as well as
  \begin{equation}
    \limsup_{i,j,k}\E\qa{\|\tilde\varphi_{i,j,k}\|_{C^r(\rho)}^p} \leq C_{p,r}(\beta, z).
  \end{equation}

  \smallskip\noindent
  (iv)
  In particular, the law of $Z+\tilde\varphi$ agrees with that of $\nu^{\SG(\beta,z)}$ modulo constants.
\end{proposition}

Thus the $(Z,\tilde\varphi)$ in the two propositions have the same law possibly up to a random constant.
The two propositions differ as follows: The first proposition only applies to the infinite-volume limit (but not simultaneously to the finite-volume
regularized versions of the fields).
The second proposition provides convergence in a sufficiently strong topology, but the convergence is only modulo constants
(and hence the limit as well). Therefore both versions will be useful.

\subsection{Sine-Gordon IMC -- proofs of Theorem~\ref{thm:sg-finvol} (i)--(iii) and Theorem~\ref{thm:sg-infvol}}
\label{sec:sg-chaos}

To prove Theorem~\ref{thm:sg-finvol}, we start from the probability space of Proposition~\ref{prop:finvol-coupling}.
We denote the Gaussian IMC associated with the Gaussian field $Z$ according to Proposition~\ref{prop:GMC} by $M^{\GFF}_\alpha$
to distinguish it from the IMC associated with the sine-Gordon field $Z+ \tilde \varphi \sim \nu^{\SG(\beta,z|\Lambda,m)}$ which
which will be  denoted $M^{\SG}_\alpha$ and is defined by
\begin{equation} \label{e:MSG-def-finvol}
  M_\alpha^{\SG}(f)
    = M_\alpha^\GFF(e^{i\sqrt{4\pi}\alpha \tilde\varphi} f)
  .
\end{equation}
The right-hand side is well-defined
when $\alpha^2 < \min\{1,2-\beta/4\pi\}$
in the sense of the multiplication properties of Besov spaces \eqref{e:Besov-mult}.
This is precisely stated in the following proposition and proof.

\begin{proof}[Proof of Theorem~\ref{thm:sg-finvol} (i)--(iii)]
\subproof{items (i)--(ii)}
Let $1>r>s>0$ and fix $\chi \in C_c^\infty$ and denote by $K \subset \R^2$ its support.
The multiplication properties of Besov--H\"older norms
\eqref{e:Besov-mult}
imply
\begin{equation} \label{e:MSG-Besov-mult}
  \|M_\alpha^{\SG}\chi \|_{C^{-s}}
  \lesssim  \|M_\alpha^{\GFF}\chi\|_{C^{-s}} \|e^{i\sqrt{4\pi}\alpha\tilde\varphi}\|_{C^{r}(K)}
  .
\end{equation}
We now use the following elementary inequality for $f: \R^2 \to \C$ with $|\im f|\leq h$:
\begin{equation} \label{e:eif-Holder}
  \|e^{if}\|_{C^r(K)} = \|e^{if}\|_{L^\infty(K)} + \sup_{x,y\in K} \frac{|e^{if(x)}-e^{if(y)}|}{|x-y|^r} \leq e^{h}\pB{1+ [f]_{C^r(K)}}.
\end{equation}
Indeed, if $if=ia+b$ with $a,b$ real and $|b|\leq h$,
\begin{align} \label{e:Holder-exp}
  |e^{ia(x)+b(x)}-e^{ia(y)+b(y)}|
  &=\sqrt{ e^{2b(x)}+e^{2b(y)}-2e^{b(x)+b(y)}\cos(a(x)-a(y))}
    \nnb
  &=\sqrt{ (e^{b(x)}-e^{b(y)})^2+2e^{b(x)+b(y)}(1-\cos(a(x)-a(y)))}
    \nnb
  &\leq e^h\sqrt{ (1-e^{-|b(y)-b(x)|})^2+|a(x)-a(y)|^2}
    \nnb
  &\leq e^h \sqrt{|b(x)-b(y)|^2+|a(x)-a(y)|^2} = e^h |f(x)-f(y)|
    .
\end{align}
For any $\rho>0$ there exists $\eta = \eta(\rho)$ such that for $z\in B_\rho(0) \cap I$, one has
$|\im \tilde\varphi|\leq \eta$ almost surely. 
It follows that $M^\SG_\alpha$ is well-defined and the moment bounds \eqref{e:MSG-moments} are also a direct consequence
from their Gaussian versions \eqref{e:GMC-moments} and from the moment bound on the H\"older seminorm \eqref{e:finvol-coupling-moment-unif-m}.

To show that $M^{\SG,\epsilon}_\alpha$ defined by \eqref{e:Meps}, with $\varphi_\epsilon$ now the convolution regularized
sine-Gordon field, converges to $M^{\SG}_\alpha$ in $C^{-s}_{\rm loc}$ in probability as $\epsilon \to 0$,
we start from the Gaussian version of this statement, which holds by Proposition~\ref{prop:GMC}:
$M^{\GFF,\epsilon}_\alpha \to M^\GFF_\alpha$ in $C^{-s}_{\rm loc}$ in probability.
Since $\tilde\varphi$ is in $C^{r}(K)$, it follows that
$\tilde\varphi_\epsilon = \eta_\epsilon * \tilde\varphi \to \tilde\varphi$ in $C^{r}(K)$ almost surely.
For $f$ compactly supported,
\begin{align}
  M^{\SG,\epsilon}_\alpha(f)
  &= \int \epsilon^{-\alpha^2} e^{i\alpha \sqrt{4\pi}\varphi_\epsilon(x)} f(x) \, dx
    \nnb
  &= \int \epsilon^{-\alpha^2} e^{i\alpha \sqrt{4\pi}(Z_\epsilon(x)+\tilde\varphi_\epsilon(x))} f(x) \, dx
  = M^{\GFF,\epsilon}_\alpha(e^{i\alpha \sqrt{4\pi}\tilde\varphi_\epsilon} f).
\end{align}
Since $|\im \tilde\varphi| \leq \eta$,
it follows from \eqref{e:Holder-exp} and $\tilde\varphi_\epsilon \to \tilde\varphi$ in $C^r(K)$
that $e^{i\alpha \sqrt{4\pi}\tilde\varphi_\epsilon}$ is also bounded in $C^r(K)$ by $C(1+[\tilde\varphi]_{C^r})$.
Similarly, it follows that it converges to $e^{i\alpha \sqrt{4\pi}\tilde\varphi}$ in $C^r(K)$:
\begin{equation}
  \|e^{i\alpha \sqrt{4\pi}\tilde\varphi}-e^{i\alpha \sqrt{4\pi}\tilde\varphi_\epsilon}\|_{C^r(K)} \lesssim \|\tilde\varphi-\tilde\varphi_\epsilon\|_{C^r(K)}.
\end{equation}
The multiplicative inequality for Besov spaces \eqref{e:Besov-mult} thus implies
\begin{align}
  \|(M^{\SG}_\alpha-M^{\SG,\epsilon}_\alpha)\chi\|_{C^{-s}}
  &\leq
    \|(M^{\GFF}_\alpha-M^{\GFF,\epsilon}_\alpha)\chi\|_{C^{-s}}\|e^{i\alpha \sqrt{4\pi}\tilde\varphi_\epsilon}\|_{C^r(K)}
    \nnb
    &\qquad + \|M^{\GFF}_\alpha\chi\|_{C^{-s}}\|e^{i\alpha \sqrt{4\pi}\tilde\varphi}-e^{i\alpha \sqrt{4\pi}\tilde\varphi_\epsilon}\|_{C^r(K)} \to 0,
\end{align}
where the convergence is in probability, as desired.
It is also clear from this construction that dependence on the mollifier $\eta$ is the same as in the Gaussian version
of the IMC.

\subproof{item (iii)}
By Proposition~\ref{prop:finvol-coupling},
the map $z \in I \mapsto \tilde\varphi(z) \in C^{r}(K,\C)$
and therefore also  $e^{i\alpha \sqrt{4\pi} \tilde\varphi}$ is bounded and analytic from $z\in I$ into $C^r(K,\C)$.
This implies that $M_\alpha^\SG$ is almost surely analytic in $z\in I$.
Using Morera's and Fubini's theorem and that $\|\chi M_\alpha^\SG\|_{C^{-s}}$ has moments of all orders,
since
\begin{equation}
  \avg{M_{\alpha_1}(f_1)\cdots M_{\alpha_n}(f_n)}_{\SG(\beta,z|\Lambda,m)}
  = \E\qB{M_{\alpha_1}^{\SG}(f_1)\cdots M_{\alpha_n}^{\SG}(f_n)}
\end{equation}
this gives the existence and analyticity of the correlation functions.
The uniform in $m>0$ bounds on the correlation functions follow from the uniform bound on the moments of $\|\chi M_\alpha\|_{C^{-s}}$
stated in item (ii).
\end{proof}

The proofs of items (i) and (ii) of Theorem~\ref{thm:sg-infvol} are essentially identical
to the corresponding statements in Theorem~\ref{thm:sg-finvol},
with the only change that we now start from the probability space constructed in Proposition~\ref{prop:coupling-infvol},
which however has the same properties,
and again define
\begin{equation} \label{e:MSG-def}
  M_\alpha^{\SG}(f)
  = M_\alpha^\GFF(e^{i\sqrt{4\pi}\alpha \tilde\varphi} f) .
\end{equation}
The following proof therefore focuses on the other properties.

\begin{proof}[Proof of Theorem~\ref{thm:sg-infvol}]
  \subproof{items (i) and (ii)}
  As already mentioned, the existence and construction of $M_\alpha$ with respect to $\nu^{\SG(\beta,z)}$
  is completely analogous to the massive version.
  The proofs that $M_\alpha^{\SG}$ is given as the $\epsilon\to 0$ limit of \eqref{e:Meps-bis} and finiteness of moments are also identical.

  \subproof{item (iii)}
  For Euclidean invariance, we use the fact that the measure $\nu^{\SG(\beta,z)}$ is Euclidean invariant, see \eqref{e:Euclinv}.
  Since we chose our mollifier in \eqref{e:Meps-bis} to be invariant under rotations,
  this means that the law of $\varphi^\epsilon$ is invariant under rotations and translations,
  so the same is true of the law of the regularized IMC from \eqref{e:Meps-bis},
  and thus also for the limiting law $M_\alpha^\SG$.

  \subproof{item (iv)}  
  To show the massless convergence of the correlation functions  \eqref{e:MSG-conv-massless},
  we start from the probability space given by Proposition~\ref{prop:coupling-infvol-convergence} on which
  \begin{equation}
    \|\tilde\varphi_{i,j,k}-\tilde\varphi -X_{i,j,k}\|_{C^r(\rho)} \to 0.
  \end{equation}
  Therefore, if $\sum_i\alpha_i=0$, and with $\Lambda=\Lambda_j$ and $m=m_i$,
  \begin{align}
  \avg{\prod_{l=1}^n M_{\alpha_l}(f_l)}_{\SG(\beta,z|\Lambda,m)}
  &= \E\qa{\prod_{l=1}^n M_{\alpha_l}^\GFF(e^{i\sqrt{4\pi}\alpha_l \tilde\varphi_{i,j,k}} f_l)}
  \\
  &= \E\qa{\prod_{l=1}^n M_{\alpha_l}^\GFF(e^{i\sqrt{4\pi}\alpha_l (\tilde\varphi_{i,j,k}+X_{i,j,k})} f_l)}
    \label{e:M-integrand}
    .
\end{align}
Using that the moment bound follows as in \eqref{e:MSG-Besov-mult}--\eqref{e:eif-Holder}:
\begin{equation}
  \sup_{i,j,k}\E\qa{|M_{\alpha}^\GFF(e^{i\sqrt{4\pi}\alpha (\tilde\varphi_{i,j,k}+X_{i,j,k})} f)|^p}
  = \sup_{i,j,k}\E\qa{|M_{\alpha}^\GFF(e^{i\sqrt{4\pi}\alpha (\tilde\varphi_{i,j,k})} f)|^p}
  < \infty,
\end{equation}
the integrand in \eqref{e:M-integrand} is uniformly integrable and \eqref{e:M-integrand} converges to
\begin{align}
  \E\qa{\prod_{l=1}^n M_{\alpha_l}^\GFF(e^{i\sqrt{4\pi}\alpha_l \tilde\varphi} f_l)}
  =
  \avg{\prod_{l=1}^n M_{\alpha_l}(f_l)}_{\SG(\beta,z)},
\end{align}
where the limits $m\to 0$ and $\Lambda\to \R^2$ are possibly along the subsequences of Proposition~\ref{prop:coupling-infvol-convergence}.
\end{proof}

\subsection{Proof of Theorem~\ref{thm:sg-finvol} (iv)}

To prove item (iv) of Theorem~\ref{thm:sg-finvol} we first define the ultraviolet
regularized measure appearing in the statement with  expectation denoted $\avg{F(\varphi)}_{\GFF+\SG(\beta,z|\Lambda,m,\epsilon)}$.
For motivation,
recall that $\nu^{\SG(\beta,z|\Lambda,m,\epsilon)}$ is defined in \eqref{e:SGdef} by
\begin{align} \label{e:SGeps-1}
  \avg{F}_{\SG(\beta,z|\Lambda,m,\epsilon)}
  &= \frac{1}{Z(\beta,z|\Lambda,m,\epsilon)}\E\qB{e^{-V_{\epsilon^2}(\Phi^{\GFF}_{\epsilon^2},z|\Lambda,m,\epsilon)} F(\Phi^{\GFF}_{\epsilon^2})},
\end{align}
where
the decomposed GFF is defined in Section~\ref{sec:Gauss-decomp},
$V_{\epsilon^2}(\varphi,z|\Lambda,m,\epsilon)$ denotes the term in the exponential in \eqref{e:SGdef},
and $Z(\beta,z|\Lambda,m,\epsilon)$ is the partition function (normalization constant).

The small scales below $\epsilon^2$ are absent
in the above ultraviolet regularization.
The expectation with index $\GFF+\SG(\beta,z|\Lambda,m,\epsilon)$ next defined
is a variant of this regularized measure in which the small scales are Gaussian
(rather than missing).
This has the advantage that we can define the IMC  for this regularized version of the sine-Gordon measure
using the same counterterm $\epsilon^{-\alpha^2}$ as for the full measure.
Concretely, for any $z\in \C$ for which the partition function is nonzero, $|\Lambda|<\infty$ and $m,\epsilon>0$, define
\begin{align} \label{e:SGeps-2}
  \avg{F(\varphi)}_{\GFF+\SG(\beta,z|\Lambda,m,\epsilon)}
  &= \E_{C_{\epsilon^2}} \qB{ \avg{F(\varphi+\zeta)}_{\SG(\beta,z|\Lambda,m,\epsilon)}}\nnb
  &= \frac{1}{Z(\beta,z|\Lambda,m,\epsilon)}\E_{C_{\epsilon^2}}\qB{\E_{C_\infty-C_{\epsilon^2}}\qB{e^{-V_{\epsilon^2}(\varphi,z|\Lambda,m,\epsilon)} F(\varphi+\zeta)}}
    \nnb
  &= \frac{1}{Z(\beta,z|\Lambda,m,\epsilon)}\E\qB{e^{-V_{\epsilon^2}(\Phi^{\GFF}_{\epsilon^2},z|\Lambda,m,\epsilon)} F(\Phi^{\GFF}_{0})},
\end{align}
where on the first two lines the Gaussian expectation with covariance $C_{\epsilon^2}$ applies to the variable $\zeta$ and the second expectation to $\varphi$.
Note that for $\epsilon>0$ fixed, the random field $\varphi$
respectively $\Phi_{\epsilon^2}^\GFF$
is smooth whereas  $\zeta$ respectively $\Phi^\GFF_{0,\epsilon^2}$ has the regularity of a GFF (and is thus only a generalized function).
The partition functions in \eqref{e:SGeps-1} and \eqref{e:SGeps-2} are the same.

\begin{proof}[Proof of Theorem~\ref{thm:sg-finvol} (iv)]
  Proposition~\ref{prop:finvol-coupling} (v) implies that, for $z\in \R$,
  \begin{equation}
      \avg{F(\varphi)}_{\GFF+\SG(\beta,z|\Lambda,m,\epsilon)}
      = \E\qB{F(Z+\tilde\varphi_\epsilon(z,m))}.
  \end{equation}
  The analyticity, estimates, and convergence of the IMC moments now
  all follow in the same way as the statements in the proof of Theorem~\ref{thm:sg-finvol} (i)--(iii).
\end{proof}

\subsection{Series expansion of the finite-volume correlation functions}
\label{sec:finvol-massless}

Using Theorem~\ref{thm:sg-finvol} for the massive sine-Gordon model as input,
this section establishes that
the massless limit of the fractional sine-Gordon correlation functions exists in finite
volume, that the limiting correlation functions remain analytic in $z$, and
that their series expansion is given in terms of free field correlations.

\begin{theorem} \label{thm:finvol-massless-corr}
Under the assumptions of Theorem~\ref{thm:sg-finvol},
the $m\to 0$ limit of the smeared correlation function \eqref{e:coupling-finvol-corfunc} exists,
is also analytic in $z\in I$, and the derivatives at $z=0$ are given by
\begin{multline}\label{e:finvol-corfunc-deriv}
    \pa{\prod_{i=1}^n C_{\eta}^{-\alpha_i^2}}
    \frac{1}{2^p}\ddp{^{p}}{z^{p}} \avg{M_{\alpha_1}(f_1)\cdots M_{\alpha_n}(f_n)}_{\SG(\beta,z|\Lambda,0)} \bigg|_{z=0}
    \\
    =
    \avg{\wick{e^{i\alpha_1 \sqrt{4\pi} \varphi}}(f_1)\cdots \wick{e^{i\alpha_n \sqrt{4\pi} \varphi}}(f_n); \wick{\cos(\sqrt{\beta}\varphi)}(\1_{\Lambda}); \cdots; \wick{\cos(\sqrt{\beta}\varphi)}(\1_{\Lambda})}_{\GFF(0)}^{\mathsf T},
  \end{multline}
  where $C_\eta$ is the constant from Proposition~\ref{prop:hkimc}, the right-hand side is defined in Proposition \ref{th:fracapp-bis}, and there are $p$ cosine terms in the cumulant. %
\end{theorem}

Before turning to the proof of the theorem, we emphasize that (because of the zero mode)
the probability measure $\avg{\cdot}_{\SG(\beta,z|\Lambda,m)}$ itself does not have a $m\to 0$ limit if $|\Lambda|<\infty$,
but that the fractional correlation functions \eqref{e:coupling-finvol-corfunc} still do,
though they could vanish (which happens for nonneutral correlations).
This is different from the correlation functions of  the infinite-volume measure in Theorem~\ref{thm:sg-infvol},
which do not vanish for nonneutral correlations.
As discussed in Remark~\ref{rk:ssb}, the nonequivalence of limits is a consequence of spontaneous symmetry breaking.

\begin{proof}[Proof of Theorem~\ref{thm:finvol-massless-corr}]
Let $f_1,\dots, f_n \in C_c^\infty(\R^2)$ and $\alpha_1,\dots, \alpha_n \in \R$ with $\alpha_i^2 < \min\{1,2-\beta/4\pi\}$.
By Theorem~\ref{thm:sg-finvol}, the function
\begin{equation} \label{e:correxp-pf2}
  z \in I \mapsto   \avg{M_{\alpha_1}(f_1)\cdots M_{\alpha_n}(f_n)}_{\SG(\beta,z|\Lambda,m)}
\end{equation}
is analytic and bounded uniformly in $z\in I \cap B_\rho(0)$ and bounded uniformly in $m\in (0,1]$.
By Montel's theorem, this family of functions depending on $m\in (0,1]$ is therefore relatively compact,
i.e., there is an analytic function (at this point possibly not unique)
\begin{equation} \label{e:correxp-pf3}
  z \in I \mapsto   \avg{M_{\alpha_1}(f_1)\cdots M_{\alpha_n}(f_n)}_{\SG(\beta,z|\Lambda,0)}
\end{equation}
and a subsequence $m \to 0$ such  the functions \eqref{e:correxp-pf2} converge to \eqref{e:correxp-pf3}
along this subsequence, uniformly on compact subsets of $z\in I$.

In the remainder of the proof,
we show that  the series expansion of \eqref{e:correxp-pf3} at $z=0$ has coefficients given by \eqref{e:finvol-corfunc-deriv}.
In particular, this expansion is independent of the subsequence $m \to 0$. Since every subsequential limit
\eqref{e:correxp-pf3} has the same series expansion in a neighborhood of $z=0$ and is analytic in $I$,
uniqueness of analytic continuation then implies that the limit is also unique and the convergence holds as $m\to 0$
without passing to a subsequence.

To show the series expansion, we begin with the corresponding expansion for the regularized expectation  \eqref{e:SGeps-2} with $\epsilon>0$.
In this case, since $\wick{\cos(\sqrt{\beta}\Phi_{\epsilon^2}^{\GFF(m)})}_{\epsilon}$ is bounded, it is clear that
\begin{equation} \label{e:correxp-pf1}
  z \in I \mapsto   \avg{M_{\alpha_1}(f_1)\cdots M_{\alpha_n}(f_n)}_{\GFF+\SG(\beta,z|\Lambda,m,\epsilon)}
\end{equation}
has a series expansion around $z=0$ given by:
\begin{multline} \label{e:correxp-pf4}
  \avg{M_{\alpha_1}(f_1)\cdots M_{\alpha_n}(f_n)}_{\GFF+\SG(\beta,z|\Lambda,m,\epsilon)}
  \\
  =
  \sum_{p=0}^\infty 2^p\frac{z^p}{p!}
  \E^\mathsf T \qa{M_{\alpha_1}(f_1)\cdots M_{\alpha_n}(f_n); \wick{\cos(\sqrt{\beta}\Phi_{\epsilon^2}^{\GFF(m)})}_\epsilon({\bf 1}_\Lambda); \cdots ; \wick{\cos(\sqrt{\beta}\Phi_{\epsilon^2}^{\GFF(m)})}_\epsilon({\bf 1}_\Lambda)},
\end{multline}
where $\E^\mathsf T$ denotes a cumulant, and there are $p$ entries of the cosine terms in the cumulant.

By Theorem~\ref{thm:sg-finvol} (iv), the function \eqref{e:correxp-pf1}
is analytic and bounded uniformly in $z\in I \cap B_\rho(0)$ and uniformly in $\epsilon,m\in (0,1]$, and it converges locally uniformly to
\eqref{e:correxp-pf2} as $\epsilon\to 0$,
and as then $m\to 0$ along a suitable subsequence, we have local uniform convergence to \eqref{e:correxp-pf3}.
This implies that the Taylor coefficients of \eqref{e:correxp-pf3} are the $\epsilon\to 0$ and then $m\to 0$ limits (along a subsequence) of the series coefficients of \eqref{e:correxp-pf1}.

On the other hand, using Proposition~\ref{prop:hkimc},  the coefficients on the right-hand side of \eqref{e:correxp-pf4} are equal to
\begin{multline}
2^p  C_\eta^{\alpha_1^2} \cdots   C_\eta^{\alpha_n^2}
  \lim_{\delta\to 0}
  \E^{\sf T}\Big[
  \wick{e^{i\alpha_1\sqrt{4\pi} \Phi_{\delta^2}^{\GFF(m)}}}_\delta (f_1)\cdots \wick{e^{i\alpha_n \sqrt{4\pi} \Phi_{\delta^2}^{\GFF(m)}}}_\delta (f_n);
    \\ \wick{\cos(\sqrt{\beta}\Phi_{\epsilon^2}^{\GFF(m)})}_\epsilon({\bf 1}_\Lambda); \cdots ; \wick{\cos(\sqrt{\beta}\Phi_{\epsilon^2}^{\GFF(m)})}_\epsilon({\bf 1}_\Lambda)
    \Big],
\end{multline}
which by Proposition \ref{th:fracapp-bis} converges to the desired limit as $\epsilon \to 0$ and then $m\to 0$ (without any subsequence assumption).
Since the series expansion of the analytic function \eqref{e:correxp-pf3} does not depend on which subsequence we picked, the function cannot either, so we have convergence as $m\to 0$. This concludes the proof.
\end{proof}

\section{Decomposition of massive sine-Gordon field -- Proposition~\ref{prop:finvol-coupling}}
\label{sec:coupling-finvol}

The goal of this section is to prove Proposition~\ref{prop:finvol-coupling} which we restate for convenience below.
The constructions of this section also provide
the starting point for the proofs of Propositions~\ref{prop:coupling-infvol}--~\ref{prop:coupling-infvol-convergence} concerning
the massless infinite-volume model in Section~\ref{sec:coupling-infvol},
where similar constructions for small scales are combined with input from integrability for large scales.

\begin{proposition} \label{prop:finvol-coupling-bis} 
  Let $\beta\in(0,6\pi)$, and
  let $\Lambda \subset \R^2$ be a bounded domain.
  Then there is a complex neighborhood $I$ of the real axis and a function $\rho\in [0,\infty) \mapsto \eta(\rho) \in [0,\infty)$,
  both depending on $\Lambda$ and $\beta$ but independent of $m \in (0,1]$,
  and a probability space on which
  random variables $Z$ and $\tilde \varphi = \tilde \varphi(z,m)$ with $z\in I$ and $m\in (0,1]$ are defined such that
  that the following hold.

  \smallskip
  \noindent
  (i) For $z\in \R$, the massive finite-volume sine-Gordon field can be decomposed as
  \begin{equation}
    Z+\tilde \varphi \sim \nu^{\SG(\beta,z|\Lambda,m)}.
  \end{equation}

  \smallskip
  \noindent
  (ii)
  The random field $Z$ is a real-valued log-correlated Gaussian random field satisfying $Z \in C^{-s}_\loc$ almost surely
  for any $s>0$. It is independent of $z \in I$.

  \smallskip
  \noindent
  (iii)
  The random field  $\tilde \varphi$   is locally H\"older continuous with $|\im \tilde\varphi| \leq \eta(\rho)$ if $|z|\leq \rho$
  and $\im \tilde\varphi=0$ if $z\in \R$ and the following moment bound holds for any $p>0$   and $\chi\in C_c^\infty(\R^2)$
  and for any fixed $r<2-\beta/4\pi$ and any fixed $m\in (0,1]$:
  \begin{equation} \label{e:finvol-coupling-moment-bis}
    \E\qB{\|\chi \tilde \varphi\|_{C^{r}}^p} \leq C_p(\beta,z,\Lambda,m) < \infty,
  \end{equation}
  and the H\"older seminorm is bounded uniformly in $m\in (0,1]$:
  \begin{equation}  \label{e:finvol-coupling-moment-unif-m-bis}
    \E\qB{ [\chi\tilde\varphi]_{C^r}^p} \leq C_p(\beta,z,\Lambda) < \infty.
  \end{equation}

  \smallskip
  \noindent
  (iv)
  The map $z\in I \mapsto \chi\tilde\varphi(z) \in C^{r}(\R^2)$ is complex analytic almost surely.

  \smallskip
  \noindent
  (v) The same holds with $\tilde\varphi$ replaced by a regularized version $\tilde\varphi_\epsilon$
  such that $Z+\tilde\varphi_\epsilon \sim \nu^{\GFF+\SG(\beta,z|\Lambda,m,\epsilon)}$
  when $z\in \R$,
  defined in \eqref{e:SGeps-2}, with all bounds uniform in $\epsilon$
  and convergence in $C^r$ as $\epsilon \to 0$.
\end{proposition}

\subsection{Renormalized potential}\label{sec:renpot}

As the first preliminary ingredient of the proof of Theorem~\ref{thm:sg-finvol}, we begin with the definition of the renormalized potential
for the sine-Gordon model.
Its construction we use goes back to \cite{MR914427} and we use the precise set-up given in \cite[Section 4]{MR4767492}.
For bounded $\Lambda\subset \R^2$, $\varphi:\Lambda\to \C$ bounded and continuous, and $\epsilon>0$,  define the bare potential
\begin{equation}
  V_{\epsilon^2}(\varphi,z|\Lambda,\epsilon) = \int_\Lambda 2z\epsilon^{-\beta/4\pi} \cos(\sqrt{\beta}\varphi(x)) \, dx.
\end{equation}
For $t>\epsilon^2$ and $m \geq 0$, the renormalized potential is then defined by
\begin{equation} \label{e:Vtdef}
  V_t(\varphi,z|\Lambda,m,\epsilon)
  = -\log \E[e^{-V_{\epsilon^2}(\varphi+\Phi^\GFF_{\epsilon^2,t},z|\Lambda,\epsilon)}]
  = -\log \E_{C_t-C_{\epsilon^2}}[e^{-V_{\epsilon^2}(\varphi+\zeta,z|\Lambda,\epsilon)}],
\end{equation}
where, for $z$ or $\varphi$ complex, we use the standard branch cut for the logarithm to define \eqref{e:Vtdef}.
The notation $\Phi^{\GFF}$ is defined in Section~\ref{sec:Gauss} and (as in \cite{MR4767492} and illustrated in the
second equality above) we sometimes alternatively use
$\E_{C}$ for the expectation of the Gaussian measure with mean zero and covariance $C$ applying in the random variable $\zeta$.
In \eqref{e:Vtdef} the covariance is
\begin{equation}
  C_t-C_{\epsilon^2} = \int_{\epsilon^2}^t e^{-m^2 s} e^{\Delta s}\, ds,
\end{equation}
in which case this measure is realized as a probability measure supported on $C^\infty(\R^2)$.
Somewhat informally we will sometimes write this as $\zeta \sim \cN(0,C_t-C_{\epsilon^2})$.
Even though $\Phi_{\epsilon^2,t}^\GFF$ respectively $\zeta$ are defined on all of $\R^2$, only their restrictions to $\Lambda$ are relevant
for the finite-volume renormalized potential defined above.

We will need the following facts about the renormalized potential. Except for the statements concerning complex $z$ and $\varphi$
and analyticity,
the proof of these properties is very similar to those in \cite{MR4767492,MR4399156} (see also \cite{MR914427,MR4303014}), to which we will often refer.
The statements involve the space of continuous functions $C(\bar\Lambda,\C)$ with the $L^\infty$ norm
(see Section~\ref{sec:notation} for the precise definition)
and the space
of bounded Lipschitz continuous functions ${\rm Lip}(\Lambda,\C)$ with the norm $\|f\|_{L^\infty(\Lambda)}+\|\nabla f\|_{L^\infty(\Lambda)}$.
Given $\eta>0$, we define the open subset $B_\eta \subset C(\bar\Lambda,\C)$ given by
\begin{equation} \label{e:BI}
  B_{\eta} =\{ \varphi \in C(\bar\Lambda,\C): |\im \varphi(x)| < \eta \text{ for all } x\in \Lambda\},
\end{equation}
and the open unit disk of radius $\rho$ in $\C$ is denoted
\begin{equation}
  B_\rho(0) = \{z \in \C: |z|<\rho \}.
\end{equation}

For some background on analyticity of functions on Banach spaces that is sufficient for us, we refer to~\cite[Appendix~A]{MR0894477}.

\begin{proposition} \label{prop:Vt}
  Let $\beta\in (0,6\pi)$. For every bounded domain $\Lambda \subset \R^2$,
  there exists a complex neighborhood of the real axis $I$ and $\eta(\rho)>0$ for any $\rho>0$,
  such that the following  estimates hold
  for any $z\in I \cap B_\rho(0)$ and $\eta=\eta(\rho)$ where $\rho>0$ is arbitrary.
  
\smallskip\noindent(i)
For every %
$\varphi \in B_{\eta(\rho)}$ that is Lipschitz continuous,
  the limit
  \begin{equation} \label{e:Vteps-conv}
    V_t(\varphi,z|\Lambda,m)-V_t(0,z|\Lambda,m) = \lim_{\epsilon\to 0} \qB{V_t(\varphi,z|\Lambda,m,\epsilon) - V_t(0,z|\Lambda,m,\epsilon)}
  \end{equation}
  exists, and the limit
  \begin{equation}
    V_t(\varphi,z|\Lambda,0)-V_t(0,z|\Lambda,0)= \lim_{m\to 0} \qb{V_t(\varphi,z|\Lambda,m)-V_t(0,z|\Lambda,m)}
  \end{equation}
  also exists.
  (Thus we do not define $V_t(\varphi,z|\Lambda,m)$, only 
  differences of $V_t$ with different $\varphi$.)

\smallskip\noindent(ii)
For any $t>0$, the map
\begin{equation}
  (\varphi,z) \in (B_{\eta(\rho)}\cap {\rm Lip}(\Lambda,\C)) \times (I \cap B_\rho(0)) \mapsto V_t(\varphi,z|\Lambda,m)-V_t(0,z|\Lambda,m),
\end{equation}
where $B_{\eta(\rho)} \cap {\rm Lip}(\Lambda,\C)$ is equipped with the norm $\|f\|_{L^\infty}+\|\nabla f\|_{L^\infty}$,
is Fréchet differentiable with derivative $\nabla_\varphi V_t(\varphi,z|\Lambda,m) \in C(\bar\Lambda)$
for any Lipschitz continuous $\varphi \in B_{\eta(\rho)}$.

\smallskip\noindent(iii)
  For 
  $\varphi,\varphi' \in B_{\eta(\rho)}\cap {\rm Lip}(\Lambda,\C)$,
  the gradient $\nabla_\varphi V_t(\varphi,z|\Lambda,m)$ continued as $0$ outside $\bar\Lambda$
  satisfies, for any $k =0,1,\dots$,
  \begin{gather} \label{e:nablaV-bd-finvol}
    \|(\sqrt{t}\nabla_x)^ke^{\Delta_x t} \nabla_\varphi V_t(\varphi,z,\cdot |\Lambda,m)\|_{L_x^\infty}
    \leq C_k(\beta,\rho,\Lambda)
    (1+t^{-\beta/8\pi})
    (1+t)^{-1}
    ,
    \\
    \label{e:nablaV-imbd}
    \|\im e^{\Delta_x t} \nabla_\varphi V_t(\varphi,z,\cdot |\Lambda,m)\|_{L_x^\infty}
    \leq C(\beta,\rho,\Lambda) (|\im z|+\|\im \varphi\|_{L^\infty_x})
    (1+t^{-\beta/8\pi})
    ,
  \end{gather}
  where the $L_x^\infty$ norms are over all of $\R^2$, and
  \begin{multline}\label{e:nablaV-Lipschitz-finvol}
    \|e^{\Delta_x t}\nabla_\varphi V_{t}(\varphi,z,x|\Lambda,m)-e^{\Delta_x t}\nabla_{\varphi'} V_{t}(\varphi',z,x|\Lambda,m)\|_{L^\infty_x}
    \\
    \leq C(\beta,\rho,\Lambda) (1+t^{-\beta/8\pi}) \|\varphi-\varphi'\|_{L^\infty_x}.
  \end{multline}
  The bounds do not depend on the Lipschitz constants of $\varphi$ and $\varphi'$ and $e^{\Delta_xt}\nabla_\varphi V_t(\varphi,z|\Lambda,m)$
  can be extended continuously to $\varphi \in B_{\eta(\rho)}$.
  Moreover, as $\epsilon \to 0$, uniformly in $\varphi \in B_{\eta(\rho)}$,
  \begin{equation} \label{e:nablaVteps-conv}
    \|(\sqrt{t}\nabla_x)^ke^{\Delta_x t}\nabla_\varphi V_t(\varphi,z|\Lambda,m,\epsilon)
    -
    (\sqrt{t}\nabla_x)^ke^{\Delta_x t}\nabla_\varphi V_t(\varphi,z|\Lambda,m)\|_{L_x^\infty} \to 0.
  \end{equation}

\smallskip\noindent(iv)
For any $t>0$, the map
\begin{equation}
  (\varphi,z) \in B_{\eta(\rho)} \times (I\cap B_\rho(0)) \mapsto e^{\Delta_x t} \nabla_\varphi V_t(\varphi,z|\Lambda,m) \in L^\infty(\R^2)
\end{equation}
is analytic.
\end{proposition}

It is possible to obtain additional estimates on the renormalized potential that hold uniformly in $\Lambda$.
These are not needed in this section and are stated in Section~\ref{sec:coupling-infvol} where they are used.

\subsection{Proof of Proposition~\ref{prop:Vt} -- small $t$}\label{sec:smallt}

It is convenient to prove Proposition~\ref{prop:Vt} in two separate cases: $t$ is small and $t$ is large. We begin with the case where $t$ is small.

Following the method of \cite{MR914427} as presented in \cite[Section 4]{MR4767492},
recall that for all $\varphi: \Lambda \to \C$ and $z\in \C$ for which the series converges,
the regularized renormalized potential has the representation
\begin{equation} \label{e:Vt-series}
  V_t(\varphi,z|\Lambda,m,\epsilon)= \sum_{n=1}^\infty \frac{z^n}{n!} \int_{(\Lambda \times \{\pm 1\})^n} d\xi_1\cdots d\xi_n \tilde V_t^n(\xi_1,\dots,\xi_n|m,\epsilon) e^{i\sqrt{\beta}\sum_j \sigma_j \varphi(x_j)}
\end{equation}
where $\xi_i=(x_i,\sigma_i)$ and
\begin{equation}
  \int_{\Lambda\times \{\pm 1\}}d\xi=\sum_{\sigma\in \{\pm 1\}}\int_\Lambda dx.
\end{equation}

The following lemma summarizes the main properties of the kernels $\tilde V_t^n$ we need. We suppress the dependence on various parameters when convenient.
The coefficients $\tilde V_t^n(\xi_1,\dots,\xi_n|m,\epsilon)$ can be defined by an explicit recursion, see \cite[(4.8)]{MR4767492},
but we do not make use of this here (only of the consequences of this).
Strictly speaking, it may have only been proven in the literature that \eqref{e:Vt-series} holds for real-valued fields $\varphi$.
For a justification that the extension to complex valued fields is allowed, see
Proposition~\ref{pr:rpexp}
in Appendix~\ref{app:fracrenorm} (in which a more general setting is considered).

\begin{lemma} \label{lem:tildeVn}
  Let $\beta\in(0,6\pi)$.
  Then $\tilde V_t^n(\xi_1,\dots,\xi_n|m,\epsilon)$ is permutation invariant, independent of $\Lambda$, and
  and there are functions $h_t^n$ such that
  \begin{equation}
    |\tilde V_t^n(\xi_1,\dots,\xi_n|m,\epsilon)| \leq h_t^n(\xi_1,\dots,\xi_n)
  \end{equation}
  and 
  the following estimates hold
  for $\epsilon^2 < t < 1< m^{-2}$
  with  a constant $C_\beta$ depending only on $\beta$.

  \smallskip\noindent (i)
  The $n=1$ term satisfies
  \begin{equation} \label{e:tildeV1}
    h_t^1(\xi_1) \leq C_{\beta}t^{-\beta/8\pi}.
\end{equation}

  \smallskip\noindent (ii)
For $\beta<4\pi$ or $\beta <6\pi$ and $n\neq 2$,
\begin{equation} \label{e:tildeVn-bd}
  \sup_{\xi_1}\|h_t^n(\xi_1,\dots,\xi_n)\|_{L^1((\Lambda \times \{\pm 1\})^{n-1})} \leq n^{n-2} t^{-1} (C_\beta t^{1-\beta/8\pi})^n,
\end{equation}

\smallskip\noindent (iii)
The $n=2$ term $\tilde V_t^2(\xi_1,\xi_2|m,\epsilon)$ is given by
 \begin{equation} \label{eq:n2explicit}
   \tilde V_t^2(\xi_1,\xi_2|m,\epsilon)
   = \tilde V_t^1(\xi_1|m,\epsilon)\tilde V_t^1(\xi_2|m,\epsilon)(1-e^{-\beta\sigma_1\sigma_2 (C_t(x_1,x_2)-C_{\epsilon^2}(x_1,x_2))}).
 \end{equation}
 For $\beta\in [4\pi,6\pi)$, the same estimate as in (ii) holds for $\tilde V_t^2(\xi_1,\xi_2|m,\epsilon)\1_{\sigma_1+\sigma_2\neq 0}$ (the charged part of $\tilde V_t^2$)
whereas $\tilde V_t^2(\xi_1,\xi_2|m,\epsilon)\1_{\sigma_1+\sigma_2= 0}$ (the neutral part of $\tilde V_t^2$) is bounded by
\begin{equation}\label{eq:V2bd}
  \sup_{\xi_1}\int_{\Lambda\times \{\pm 1\}}d\xi_2\, \frac{|x_1-x_2|}{\sqrt{t}} h_t^2(\xi_1,\xi_2)\1_{\sigma_1+ \sigma_2=0}
  \leq  t^{-1} (C_\beta t^{1-\beta/8\pi})^2.
\end{equation}

\smallskip\noindent
For any distinct $x_1,\dots, x_n$,
the kernels $\tilde V_t^n(\xi_1,\dots,\xi_n|m,\epsilon)$ converge
as $\epsilon\to 0$ and $m \to 0$.
\end{lemma}

\begin{proof}
  These kinds of estimates are proved in \cite[Section~4]{MR4767492} (and in a more general setting also in Appendix~\ref{app:fracrenorm}).  
  Permutation invariance follows by induction using the recursion formula \cite[(4.8)]{MR4767492}, that the sum appearing there is over all bipartitions of $[n]$ and that the function $\dot{u}_s^{m^2}$ appearing there is symmetric. The independence of $\Lambda$ is clear by inspection of \cite[(4.8)]{MR4767492}.
  For the estimates with $n\neq 2$, see \cite[Proposition~4.1]{MR4767492} and in particular, \cite[(4.73)]{MR4767492}.

  For $\beta\geq 4\pi$, the $n=2$ term is singular but explicitly given by
  \eqref{eq:n2explicit}, see \cite[(4.60)]{MR4767492}.
  The estimates for the charged and neutral part of $\tilde V_t^2$ follow
  by a direct calculation from this expression. In particular, for the neutral part
  (see \cite[(4.47)]{MR4767492}),
 \begin{align}
   |\tilde V_t^2(\xi_1,\xi_2|m,\epsilon)|
   &\lesssim \tilde V_t^1(\xi_1|\epsilon)\tilde V_t^1(\xi_2|\epsilon) (\frac{|x_1-x_2|}{\sqrt{t}})^{\sigma_1\sigma_2(\beta/2\pi)}
     e^{-c|x_1-x_2|/\sqrt{t}}
     \nnb
   &\lesssim (t^{-\beta/8\pi})^2 (\frac{|x_1-x_2|}{\sqrt{t}})^{\sigma_1\sigma_2(\beta/2\pi)}
     e^{-c|x_1-x_2|/\sqrt{t}},\label{eq:n2bound}
 \end{align}
 where $c$ and the implied constants are independent of $\Lambda,m,\epsilon$.

 The pointwise convergence as $\epsilon \to 0$ and then $m\to 0$
 follows from the recursive representation for the kernels,
 see for example \cite[(4.8)]{MR4767492} or \eqref{eq:nnrec} in Appendix~\ref{app:fracrenorm}.
\end{proof}

The following lemma gives the small-$t$ version of Proposition~\ref{prop:Vt}.

\begin{lemma}\label{le:vtsmall}
  For any $\rho>0$, there are $t_0=t_0(\beta,\rho)>0$ and $\eta=\eta(\beta,\rho)>0$
  such that the assertions of Proposition~\ref{prop:Vt} hold for $t\leq t_0$ and $|z|\leq \rho$,
  and in this case the constants are in fact independent of $\Lambda$. More precisely, concerning (iv),
  the mapping 
  \begin{equation}
    B_\eta\times B_\rho(0) \ni (\varphi,z)\mapsto e^{\Delta_x t}\nabla_\varphi V_t(\varphi,z|\Lambda,m) .
  \end{equation}
  is analytic when $t \leq t_0(\beta,\rho)$.
\end{lemma}

\begin{proof}
The proofs here are very similar to (and rely on) those of \cite[Sections~4 and~5]{MR4767492},
but as the statements are not exactly ones found in the literature, we offer the required arguments here to fill in the details. In this part of the proof, it turns out that we can take the neighborhood $I$ to be a strip $I_\eta=\{z\in \C: |\im z|<\eta\}$,
 where the parameter $\eta>0$ is  chosen sufficiently small depending on $\beta$ and $\rho$ and $\Lambda$ at various (but finitely many) steps of the proof
  (but not depending on $m>0$).
We will also write $B_\eta$ instead of $B_{I_\eta}$ in the remainder of this section.

\subproof{item (i)}
The role of $t_0$ is that we see from \eqref{e:tildeVn-bd} that for any $\beta\in(0,6\pi)$ and $z\in \C$, $\eta>0$ there exists a $t_0=t_0(\beta,z,\eta)$ such that for $t\leq t_0$ and $\varphi\in B_\eta$, the series expansion representation for $V_t$ stated in \eqref{e:Vt-series} converges.
More precisely, we have by \eqref{e:tildeVn-bd},
\begin{align}\label{temp46uj}
  |V_t(\varphi,z|\Lambda,m,\epsilon)|&\leq \sum_{n=1}^\infty \frac{|z|^n}{n!} \int_{(\Lambda \times \{\pm 1\})^n} d\xi_1\cdots d\xi_n |\tilde V_t^n(\xi_1,\dots,\xi_n)| |e^{i\sqrt{\beta}\sum \sigma_j \varphi(x_j)}|
  \\
&\leq   \sum_{n\neq 2}^\infty |z|^n e^{\sqrt{\beta}\eta n} \frac{1}{t} (C_\beta t^{1-\beta/8\pi})^n |\Lambda|+\frac{|z|^2}{2}\int_{(\Lambda\times \{\pm 1\})^2}d\xi_1 d\xi_2 |\tilde V_t^2(\xi_1,\xi_2)| e^{2\sqrt{\beta}\eta}. \nonumber
\end{align}
Thus for $t_0>0$ such that $|z| e^{\sqrt{\beta}\eta}C_\beta t_0^{1-\beta/8\pi}<1$, we see that for $t\leq t_0$, the series with the $n=2$ terms removed converges uniformly in $\epsilon,m$.
The $n=2$ term is not bounded in the above sense in the $\epsilon,m\to 0$ limit, but for any strictly positive $\epsilon,m$ it is finite.
Let us fix such a $t_0$ now -- we will later take $\eta=\eta(\rho)$ so that $t_0$ will only depend on $\beta,\rho$.

For such a choice of $t_0$, we can then write for $t\leq t_0$,
\begin{multline}\label{eq:Vtdiffsum}
  V_t(\varphi,z|\Lambda,m,\epsilon)-V_t(0,z|\Lambda,m,\epsilon)
  \\
  =\sum_{n=1}^\infty \frac{z^n}{n!} \int_{(\Lambda \times \{\pm 1\})^n} d\xi_1\cdots d\xi_n \tilde V_t^n(\xi_1,\dots,\xi_n) (e^{i\sqrt{\beta}\sum_j \sigma_j \varphi(x_j)}-1).
\end{multline}
Note that for $\varphi$ Lipschitz continuous on $\Lambda$, we have for $x,y\in \Lambda$, 
\begin{equation}
  |e^{i\sqrt{\beta}(\varphi(x)-\varphi(y))}-1|\leq \sqrt{\beta}\|\nabla \varphi\|_{L^\infty(\Lambda)}|x-y|.
\end{equation}
Thus with a similar estimate as the one ensuring uniform convergence (now using \eqref{eq:V2bd} as well as the fact that the charged part of $\tilde V^2$ satisfies \eqref{e:tildeVn-bd}), we see that for $\varphi\in B_{\eta}$ that is %
 Lipschitz continuous,
the series \eqref{eq:Vtdiffsum} converges uniformly in $\epsilon$ and $m$. 

To prove that $V_t(\varphi,z|\Lambda,m,\epsilon)-V_t(0,z|\Lambda,m,\epsilon)$ converges as we let $\epsilon,m\to 0$, it is thus sufficient to prove that pointwise, the kernels $\tilde V_t^n(\xi_1,\dots,\xi_n|m,\epsilon)$ converge (at distinct points) when we first let $\epsilon\to 0$ and then $m\to 0$,
as given by Lemma~\ref{lem:tildeVn}.

\subproof{item (ii)} 
We note that we have for $t\leq t_0$, %
\begin{align} \label{e:Vt-diff}
&V_t(\varphi+f,z|\Lambda,m)-V_t(\varphi,z|\Lambda,m)\nnb
&=\sum_{n=1}^\infty \frac{z^n}{n!}\int_{(\Lambda\times \{\pm 1\})^n}d\xi_1\cdots d\xi_n \tilde V_t^n (\xi_1,\dots,\xi_n)e^{i\sqrt{\beta}\sum_j \sigma_j \varphi(x_j)}(e^{i\sqrt{\beta}\sum_j \sigma_j f(x_j)}-1)\nnb
&=i\sqrt{\beta}\sum_{n=1}^\infty \frac{z^n}{n!}\int_{(\Lambda\times \{\pm 1\})^n}d\xi_1\cdots d\xi_n \tilde V_t^n (\xi_1,\dots,\xi_n)e^{i\sqrt{\beta}\sum_j \sigma_j \varphi(x_j)}\sum_{j=1}^n \sigma_j f(x_j)\nnb
  &\qquad +\sum_{n=1}^\infty \frac{z^n}{n!}\int_{(\Lambda\times \{\pm 1\})^n}d\xi_1\cdots d\xi_n \tilde V_t^n (\xi_1,\dots,\xi_n)e^{i\sqrt{\beta}\sum_j \sigma_j \varphi(x_j)} R_f(\xi_1,\dots,\xi_n)
\end{align}
where
\begin{equation}
  R_f(\xi_1,\dots,\xi_n) = e^{i\sqrt{\beta}\sum_j \sigma_j f(x_j)}-i\sqrt{\beta}\sum_j \sigma_j f(x_j)-1.
\end{equation}
Using that $\tilde V_t^n(\xi_1,\dots,\xi_n)$ is permutation invariant (see Lemma~\ref{lem:tildeVn}),
the first line on the right-hand side of \eqref{e:Vt-diff} equals $(f,\nabla_\varphi V_t(\varphi,z))$ with
\begin{multline} \label{e:nablaV-def}
  \nabla_\varphi V_t(\varphi,z,x)
  \\= \sum_{n=1}^\infty \frac{z^n}{(n-1)!} \sum_{\sigma_1} \int_{(\Lambda \times \{\pm 1\})^{n-1}}d\xi_2 \cdots d\xi_n\, \tilde V_t^n((\sigma_1,x),\xi_2,\dots,\xi_n) i\sqrt{\beta}\sigma_1 e^{i\sqrt{\beta}\sum \sigma_j \varphi(x_j)}.
\end{multline}

The function $\nabla_\varphi V_t(\varphi,z)$ is in $C(\bar\Lambda)$ for every Lipschitz continuous $\varphi \in B_\eta$.
Indeed, we again see from our geometric convergence estimates (see \eqref{temp46uj}) that the $n\neq 2$ and charged $n=2$ terms are in $C(\bar\Lambda)$
and that their sum converges for uniformly in $x_1$.
For the $n=2$ term in the neutral case, by \eqref{eq:n2explicit},
\begin{align}
&\left|\sum_{\sigma_1\in \{\pm 1\}}\int_{\Lambda\times \{\pm 1\}} dx_2\, \tilde V_t((x_1,\sigma_1),(x_2,-\sigma_1))\sigma_1e^{i\sqrt{\beta}\sum_j \sigma_j\varphi(x_j)}\right|\nnb
  &\leq  \int_{\Lambda}dx_2 \, |\tilde V_t^2((x_1,1),(x_2,-1))||e^{i\sqrt{\beta}(\varphi(x_1)-\varphi(x_2))}-e^{-i\sqrt{\beta}(\varphi(x_1)-\varphi(x_2))}|\nnb
&\leq 2\sqrt{\beta}\|\nabla\varphi\|_{L^\infty(\Lambda)}e^{2\sqrt{\beta}\eta}\int_{\Lambda}dx_2 \, |\tilde V_t^2((x_1,1),(x_2,-1))||x_1-x_2|.
\end{align}
Thus, by \eqref{eq:n2bound}, our series converges uniformly in $x_1$ and is thus in $C(\bar\Lambda)$.
To bound the second line of \eqref{e:Vt-diff}, for $f\in B_{\eta}$, we note that 
\begin{align}
  |R_f(\xi_1,\dots,\xi_n)|
  \leq \left|\sum_j \sigma_j f(x_j)\right|^2e^{n\sqrt{\beta}\eta}.
\end{align}
Thus when $n\neq 2$ or $n=2$ and we have a charged configuration, we can simply bound this by 
\begin{equation}
    |R_f(\xi_1,\dots,\xi_n)| \leq
n^2\|f\|^2_{L^\infty(\Lambda)}e^{n\sqrt{\beta}\eta}\leq n^2\|f\|^2_{L^\infty(\Lambda)}e^{n\sqrt{\beta}\eta},
\end{equation}
whereas, for the neutral $n=2$ case, we have 
\begin{align}
    |R_f(\xi_1,\xi_2)| \leq 
  |x_1-x_2|^2\|\nabla f\|_{L^\infty(\Lambda)}^2 e^{2\sqrt{\beta}\eta}.
\end{align}
Thus with the geometric convergence for $t\leq t_0$ for our series and \eqref{eq:V2bd} (the integral is finite when $|x_1-x_2|$ is replaced by $|x_1-x_2|^\alpha$ for any $\alpha>\beta/2\pi-2$ -- see e.g. \cite[Proof of Lemma 4.5]{MR4767492}), we see
\begin{multline}
  \sum_{n=1}^\infty \frac{z^n}{n!}\int_{(\Lambda\times \{\pm 1\})^n}d\xi_1\cdots d\xi_n \tilde V_t^n (\xi_1,\dots,\xi_n)e^{i\sqrt{\beta}\sum_j \sigma_j \varphi(x_j)}
  R_f(\xi_1,\dots,\xi_n)
  \\
 =O(\|f\|_{L^\infty(\Lambda)}+\|\nabla f\|_{L^\infty(\Lambda)})^2.
\end{multline}
We conclude that, for $t\leq t_0$, the map $\varphi\mapsto V_t(\varphi,z|\Lambda,m)$ is Fréchet differentiable and  its derivative in the direction
$f$ is given by \eqref{e:nablaV-def}.

\subproof{item (iii)}
It was shown in item~(ii) that the gradient $\nabla_\varphi V_t(\varphi,z)$ given by \eqref{e:nablaV-def} is in $L^\infty$ for any %
 Lipschitz continuous $\varphi$
with bound depending on the Lipschitz constant of $\varphi$. However, as we now observe, its smoothed version $e^{\Delta_x t} \nabla_\varphi V_t(\varphi,z)$
can be bounded uniformly in $\varphi \in B_\eta$ (without the Lipschitz condition). Indeed,
\begin{multline} \label{e:nablaV-def-bis}
  (f,\nabla_\varphi V_t(\varphi,z))
  \\
  = \sum_{n=1}^\infty \frac{z^n}{(n-1)!} \int_{(\Lambda \times \{\pm 1\})^n} d\xi_1\cdots d\xi_n\, f(x_1) \tilde V_t^n((\sigma_1,x_1),\xi_2,\dots,\xi_n) i\sqrt{\beta}\sigma_1 e^{i\sqrt{\beta}\sum \sigma_j \varphi(x_j)}.
\end{multline}
For smooth $f$, this series converges uniformly in $\varphi \in B_\eta$ when $t\leq t_0$. Indeed, by Lemma~\ref{lem:tildeVn},
\begin{align} \label{e:nablaVbd}
  &|(f,\nabla_\varphi V_t(\varphi,z))|\nnb
  &\leq \sum_{n=1}^\infty n \sqrt{\beta}\|f\|_{L^1(\Lambda)} |z|^n e^{\sqrt{\beta} \eta n}\frac{1}{t}(C_\beta t^{1-\beta/8\pi})^n \nnb
&\quad +2\sqrt{\beta}\left|z^2\int_{\Lambda^2}dx_1\, dx_2\, \tilde V_t^2((1,x_1),(-1,x_2)) f(x_1)\sin(\sqrt{\beta}(\varphi(x_1)-\varphi(x_2)))\right|\nnb
&=\sum_{n=1}^\infty n \sqrt{\beta}\|f\|_{L^1(\Lambda)} |z|^n e^{\sqrt{\beta} \eta n}\frac{1}{t}(C_\beta t^{1-\beta/8\pi})^n \nnb
&\quad +\sqrt{\beta}\left|z^2\int_{\Lambda^2}dx_1\, dx_2\, \tilde V_t^2((1,x_1),(-1,x_2)) (f(x_1)-f(x_2))\sin(\sqrt{\beta}(\varphi(x_1)-\varphi(x_2)))\right|.
\end{align}
To estimate the last integral above, we note first of all that for $\varphi\in B_\eta$, $|\sin(\sqrt{\beta}(\varphi(x_1)-\varphi(x_2)))|\leq e^{2\sqrt{\beta}\eta}$. For the next step, let us assume for simplicity that $f$ is defined on $\R^2$ and $\nabla f\in L^1(\R^2)$ %
-- we will shortly be applying our estimate to such a function. Using \eqref{eq:n2bound}, we have (since $\beta<6\pi$) 
\begin{align} \label{e:nablaV2bd}
&\int_{\Lambda^2}dx_1\, dx_2\, |\tilde V_t^2((1,x_1),(-1,x_2))||f(x_1)-f(x_2)|\nnb
&\lesssim t^{-\frac{\beta}{4\pi}}\int_{\Lambda\times\Lambda} dx_1\, dx_2\, \left(\frac{\sqrt{t}}{|x_1-x_2|}\right)^{\beta/2\pi}e^{-c|x_1-x_2|/\sqrt{t}}\int_0^1 du \, |\nabla f(x_1+u(x_2-x_1))||x_2-x_1|\nnb
&= t^{-\frac{\beta}{4\pi}}\sqrt{t}\int_{\R^2}dy\, \left(\frac{\sqrt{t}}{|y|}\right)^{\frac{\beta}{2\pi}-1}e^{-c\frac{|y|}{\sqrt{t}}}\int_0^1 du\, \int_{\Lambda} dx_1 \, |\nabla f(x_1+uy)| {\bf 1}_{x_1+y \in \Lambda}\nnb
& \leq t^{-\frac{\beta}{4\pi}}\sqrt{t}\int_{\R^2}dy\, \left(\frac{\sqrt{t}}{|y|}\right)^{\frac{\beta}{2\pi}-1}e^{-c\frac{|y|}{\sqrt{t}}} \int_{\tilde\Lambda} dx \, |\nabla f(x)| \nnb
&\lesssim \|\nabla f\|_{L^1(\tilde\Lambda)} t^{\frac{3}{2}-\frac{\beta}{4\pi}},
\end{align}
where the implied constants are independent of $t$ and $f$ and $\tilde\Lambda$ denotes the convex hull of $\Lambda$.
(In our eventual application, $\Lambda$ is convex, so $\tilde \Lambda=\Lambda$.)

We now want to apply this to  $f(x_1)=(\sqrt{t}\nabla_x)^k e^{\Delta_xt}(x,x_1)$ (with $k\geq 0$) since
\begin{equation}
(\sqrt{t}\nabla_x)^k  e^{\Delta_x t}\nabla_\varphi V_t(\varphi,z,x) = ((\sqrt{t}\nabla_x)^ke^{\Delta_x t}(x,\cdot), \nabla_\varphi V_t(\varphi,z)).
\end{equation}
For such $f$, $\|f\|_{L^1(\R^2)}$ and $\sqrt{t}\|\nabla f\|_{L^1(\R^2)}$ are bounded in $t$,
and we conclude that for $t\leq t_0$,
\begin{align}
\|(\sqrt{t}\nabla_x)^k e^{\Delta_x t}\nabla_\varphi V_t(\varphi,z,\cdot|\Lambda,m)\|_{L_x^\infty}&\lesssim \sum_{n=1}^\infty n\sqrt{\beta} |z|^n e^{\sqrt{\beta}n}\frac{1}{t} (C_\beta t^{1-\frac{\beta}{8\pi}})^n+t^{1-\frac{\beta}{4\pi}}\nnb
&\lesssim t^{-\frac{\beta}{8\pi}},
\end{align}
where the implied constants depend only on $k$, $\beta$, and $z$.

We remark that these $L_x^\infty$ estimates are independent of $\Lambda$
in this small $t$-regime we considered.
One can also obtain the following $\Lambda$-dependent $L_x^1$ estimates.
Namely, since
\begin{equation}
  \int dx\, \|(\sqrt{t} \nabla_x)^k e^{\Delta_x t}(x,\cdot)\|_{L^1(\tilde\Lambda)}
  \leq \int_{\R^2} dx \int_{\tilde\Lambda} dy\, |(\sqrt{t}\nabla_x)^k \frac{e^{-\frac{|x-y|^2}{4t}}}{4\pi t}|
  \leq C_k|\tilde\Lambda|,
\end{equation}
it also follows from the above that
\begin{equation} \label{e:nablaV-L1}
  \|(\sqrt{t}\nabla_x)^k e^{\Delta_x t}\nabla_\varphi V_t(\varphi,z,\cdot|\Lambda,m)\|_{L_x^1}
  \lesssim |\tilde\Lambda| t^{-\frac{\beta}{8\pi}} = C(\beta,\rho,\Lambda)  t^{-\frac{\beta}{8\pi}}.
\end{equation}

The proof of \eqref{e:nablaV-imbd} is  similar:
we use \eqref{e:nablaV-def-bis} with $f(x_1)=e^{\Delta_x t}(x,x_1)$ and note that 
\begin{align}
|z^n-\bar z^n|=O(n |\im z||z|^{n-1})
\end{align} 
and since $\tilde V_t^n((-\sigma_1,x_1),\dots,(-\sigma_n,x_n))=\tilde V_t^n((\sigma_1,x_1),\dots,(\sigma_n,x_n))$ (this follows by induction from the recursive representation of the kernels \cite[(4.8)]{MR4399156})
\begin{align}
&\sum_{\sigma_1,\dots,\sigma_n\in \{-1,1\}}\tilde V_t^n((\sigma_1,x_1),\dots,(\sigma_n,x_n))i\sqrt{\beta}\sigma_1 \left(e^{i\sqrt{\beta}\sum \sigma_j\varphi(x_j)}+ e^{-i\sqrt{\beta}\sum \sigma_j \overline{\varphi(x_j)}}\right)\nnb
&=\sum_{\sigma_1,\dots,\sigma_n\in \{-1,1\}}\tilde V_t^n((\sigma_1,x_1),\dots,(\sigma_n,x_n))i\sqrt{\beta}\sigma_1 \left(e^{i\sqrt{\beta}\sum \sigma_j\varphi(x_j)}- e^{i\sqrt{\beta}\sum \sigma_j \overline{\varphi(x_j)}}\right).
\end{align}
For $n\neq 2$, the difference is bounded by $\|\im \varphi\|_\infty e^{n\sqrt{\beta}\eta}$, while the $n=2$ term is estimated in a similar manner as before along with noting that  
\begin{align}
\left|\sin(\sqrt{\beta}(\varphi(x_1)-\varphi(x_2)))-\sin(\sqrt{\beta}(\overline{\varphi(x_1)}-\overline{\varphi(x_2)}))\right|\lesssim e^{2\sqrt{\beta}\eta}\|\im \varphi\|_\infty
\end{align}

For \eqref{e:nablaV-Lipschitz-finvol}, the argument is again similar. We use \eqref{e:nablaV-def-bis} and for $n\neq 2$, we write 
\begin{align}
\left|e^{i\sqrt{\beta}\sum \sigma_j\varphi(x_j)}-e^{i\sqrt{\beta}\sum \sigma_j\varphi'(x_j)}\right|\lesssim \|\varphi-\varphi'\|_\infty e^{n\sqrt{\beta}\eta},
\end{align}
while for the $n=2$-term, we proceed as before and write 
\begin{align}
\left|\sin(\sqrt{\beta}(\varphi(x_1)-\varphi(x_2)))-\sin(\sqrt{\beta}(\varphi'(x_1)-\varphi'(x_2)))\right|\lesssim e^{2\sqrt{\beta}\eta}\|\varphi-{\varphi'}\| _\infty.
\end{align}
Again, the constants obtained are independent of $\Lambda$.

Finally, for the convergence \eqref{e:nablaVteps-conv} of $\nabla_x^k e^{\Delta_x t}\nabla_\varphi V_t$ in $L_x^\infty$,
one can for example observe that the series representation \eqref{e:nablaV-def} of $\nabla_\varphi V_t$ converges in $L_x^\infty$
if the neutral part of the $n=2$ term is removed, uniformly in $\varphi$. This is similar to the bound in \eqref{e:nablaVbd}
and the pointwise convergence of the kernels $\tilde V^n_t$ used in the proof of item (i).
Since $(\sqrt{t}\nabla_x)^k e^{\Delta_x t}$ is bounded on $L^\infty_x$ a version of \eqref{e:nablaVteps-conv}
with the neutral part of the $n=2$ term removed follows. The neutral $n=2$ term can be verified by hand, similarly to \eqref{e:nablaV2bd}.

\subproof{item (iv)}
Recall (see \cite[Appendix~A]{MR0894477}) that if $E$ and $F$ are complex Banach spaces and $U\subset E$ is open, then a function $f:U\to F$ is analytic if it is continuously differentiable on $U$. Another fact we need from \cite[Appendix~A]{MR0894477} is that if $(f_n)$ is a sequence of analytic functions on $U$ and $f_n\to f$ uniformly, then $f$ is also analytic (see \cite[Appendix~A: Theorem 2]{MR0894477}).

We will first prove that for each $\rho>0$, there exists a $t_0(\beta,\rho)$ such that for $t\leq t_0$, the mapping 
\begin{equation}
B_\eta\times B_\rho(0) \ni (\varphi,z)\mapsto e^{\Delta_x t}\nabla_\varphi V_t(\varphi,z|\Lambda,m) 
\end{equation}
is analytic.
To see why this is true, note that in the proofs of the claims of item (iii), we saw that  for suitable $t_0$, the series \eqref{e:nablaV-def-bis} with $f(x_1)=e^{t\Delta_x}(x,x_1)$ converges uniformly in $(\varphi,z)\in B_\eta\times B_\rho(0)$.
Thus it is sufficient to prove that each term in the series is analytic. One readily sees that each term in the series is continuously differentiable, so indeed we have analyticity in $B_\eta\times B_\rho(0)$ for $t\leq t_0(\beta,\rho)$.
\end{proof}

\subsection{Proof of Proposition~\ref{prop:Vt} -- large $t$}

In this section, we prove the remainder of Proposition~\ref{prop:Vt},
namely we treat the case of large $t$. We begin with some preliminary considerations before going into the actual proof.

Let us write $V_t(\varphi,z)=V_t(\varphi,z|\Lambda,m,\epsilon)$ and assume $m,\epsilon>0$. The latter two assumptions will be relaxed later. 

We will obtain crude estimates for $V_t(\varphi,z)$ for $t>t_0$ from the following representation. Recall the definition of $V_t$ in \eqref{e:Vtdef}
and note that if $\epsilon^2< t_0<t$ and $\zeta\sim \mathcal{N}(0,C_t-C_{\epsilon^2})$,
by convolution of Gaussians, we can write $\zeta=\zeta_1+\zeta_2$ where $\zeta_1\sim \mathcal{N}(0,C_t-C_{t_0})$ and $\zeta_2\sim \mathcal{N}(0,C_{t_0}-C_{\epsilon^2})$ are independent Gaussian fields. Fubini's theorem therefore gives
\begin{align}
  e^{-V_t(\varphi,z)}=\E_{C_t-C_{\epsilon^2}}[e^{-V_{\epsilon^2}(\varphi+\zeta,z)}]
  &=\E_{C_t-C_{t_0}}\Big[\E_{C_{t_0}-C_{\epsilon^2}}\Big[e^{-V_{\epsilon^2}(\varphi+\zeta_1+\zeta_2,z)}\Big]\Big]
    \nnb
    &=\E_{C_t-C_{t_0}}\Big[e^{-V_{t_0}(\varphi+\zeta_1,z)}\Big],
\end{align} 
and hence, relabelling $\zeta_1$ as $\zeta$, the following representation holds:
\begin{equation} \label{e:Vt-Vt0}
  V_t(\varphi,z)-V_{t_0}(\varphi,z)
  = -\log \E_{C_t-C_{t_{0}}}\qB{e^{-V_{t_0}(\varphi+\zeta,z)+V_{t_0}(\varphi,z)}}
\end{equation}
provided the logarithm on the right-hand side is well defined. This will be established in the next lemma. 
Similarly, 
\begin{align} \label{e:nablaVt-Vt0}
  \nabla_\varphi V_t(\varphi,z)
  &= e^{+V_t(\varphi,z)}\E_{C_t-C_{t_0}}\qB{e^{-V_{t_0}(\varphi+\zeta,z)} \nabla_\varphi V_{t_0}(\varphi+\zeta,z)}
    \nnb
  &= e^{+V_t(\varphi,z)-V_{t_0}(\varphi,z)}\E_{C_t-C_{t_0}}\qB{e^{-V_{t_0}(\varphi+\zeta,z)+V_{t_0}(\varphi,z)} \nabla_\varphi V_{t_0}(\varphi+\zeta,z)}.
\end{align}
The point is that the right-hand sides can be expressed in terms of $V_{t_0}$ which is already controlled.
Our estimates will be valid for any bounded $\Lambda\subset\R^2$ (but unlike the $t\leq t_0$ estimates not uniformly in $\Lambda$).

To bound the right-hand sides of \eqref{e:Vt-Vt0} and \eqref{e:nablaVt-Vt0} we need estimates on $V_{t_0}$ (which follow as in the previous
subsection from the series expansion) and standard tail estimates on the Gaussian field with covariance $C_t-C_{t_0}$. Both are stated in the following two lemmas.
Together these imply rough stability estimates on $V_t$ for $t>t_0$ stated in the third lemma below.

The subtraction of $V_{t_0}(\varphi,z)$ above is not needed for $\beta<4\pi$ for which the renormalized potential itself exists, but it is needed if $\beta\geq 4\pi$
when only differences of the renormalized potential are defined. 

We write $\delta_\varphi F(\zeta) = F(\varphi+\zeta)-F(\varphi)$.

\begin{lemma}\label{le:logest}
  For $z \in I \cap B_\rho(0)$, uniformly in $\varphi \in B_{\eta(\rho)}$,
  for any Lipschitz continuous $\zeta: \Lambda \to \R$,
  \begin{align} \label{e:delta-Vt0}
  |\delta_\varphi V_{t_0}(\zeta,z)|
    &\leq C(\beta,\rho) |\Lambda| (1+\|\nabla_x\zeta\|_{L^\infty(\Lambda)})
    \\
    \label{e:imdelta-Vt0}
    |\im \delta_\varphi V_{t_0}(\zeta,z)|
    &\leq C(\beta,\rho) (|\im z| + \|\im \varphi\|_{L^\infty(\Lambda)}) |\Lambda| (1+\|\nabla_x\zeta\|_{L^\infty(\Lambda)}).
  \end{align}
  The same estimates hold for $\epsilon>0$.
\end{lemma}

\begin{proof}
  These estimates follow from the series expansion
  \begin{multline} 
    \delta_\varphi V_{t_0}(\zeta,z) = V_{t_0}(\varphi+\zeta,z)-V_{t_0}(\varphi,z)
    \\
  =\sum_{n=1}^\infty \frac{z^n}{n!} \int_{(\Lambda \times \{\pm 1\})^n} d\xi_1\cdots d\xi_n \tilde V_{t_0}^n(\xi_1,\dots,\xi_n) e^{i\sqrt{\beta}\sum_j \sigma_j \varphi(x_j)}(e^{i\sqrt{\beta}\sum_j \sigma_j \zeta(x_j)}-1),
  \end{multline}
  as in the proof of Lemma~\ref{le:vtsmall}.
  For the proof of \eqref{e:delta-Vt0},
  see in particular \eqref{eq:Vtdiffsum} and the discussion following it.
  The estimate \eqref{e:imdelta-Vt0} can be obtained as in the proof of \eqref{e:nablaV-imbd} for $t \leq t_0$.
\end{proof}

\begin{lemma}\label{le:gaussianest}
  There is $C=C(\Lambda)>0$ such that for $T\geq T(\Lambda)$ and $p\geq 1$, uniformly in $t>t_0$,
  \begin{align}
    \label{e:Gaussian1}
    \E_{C_t-C_{t_0}}\qB{e^{-p\|\nabla \zeta\|_{L^\infty(\Lambda)}}}
    &\geq e^{-Cp},
      \\
    \label{e:Gaussian2}
      \E_{C_t-C_{t_0}}\qB{e^{+p\|\nabla \zeta\|_{L^\infty(\Lambda)}}\1_{\|\nabla \zeta\|_{L^\infty(\Lambda)} > T}}
    &\leq e^{+C p^2} e^{-T^2/C},
            \\
    \label{e:Gaussian3}
    \E_{C_t-C_{t_0}}\qB{e^{+p\|\nabla \zeta\|_{L^\infty(\Lambda)}}(1+\|\nabla \zeta\|_{L^\infty(\Lambda)})}
    &\leq e^{+C p^2}.      
\end{align}
\end{lemma}

The proof of these Gaussian estimates is given in Appendix~\ref{app:gaussianest}.

\begin{lemma} \label{lem:re-im-eV}
  For any $\rho>0$, there exist $C=C(\beta,\rho,\Lambda)<\infty$ and $\eta=\eta(\rho,\Lambda)>0$ such that for $|\im z| + \|\im \varphi\|_{L^\infty(\Lambda)} \leq \eta$,
  \begin{align}
    \re e^{-V_t(\varphi,z)+V_{t_0}(\varphi,z)}
    &\geq
      e^{-C}\label{temp76ys}
    \\
    \E_{C_t-C_{t_0}}\qB{e^{-\re (V_{t_0}(\varphi+\zeta,z)-V_{t_0}(\varphi,z))}} &\leq e^{+C}
    \\
    |\im e^{-V_{t}(\varphi,z)+V_{t_0}(\varphi,z)}|
    &\leq {e^{+C}}(|\im z|+\|\im \varphi\|_{L^\infty(\Lambda)}).\label{temp65trm}
  \end{align}
  The estimates are uniform in the real part of $\varphi$.
\end{lemma}

\begin{proof}
We remark (but will not use) that, for real $z$ and $\varphi$,
the first estimate would simply follow from Jensen's inequality which gives
\begin{equation}
  \E_{C_t-C_{t_{0}}}[e^{-\delta_\varphi V_{t_0}(\zeta,z|\Lambda,m)}]
  \geq
  e^{-\E_{C_t-C_{t_{0}}}[\delta_\varphi V_{t_0}(\zeta,z|\Lambda,m)]}
\end{equation}
and the right-hand side is bounded from below by a constant depending on $\Lambda$
by \eqref{e:delta-Vt0} and \eqref{e:Gaussian1}.
In general, using that
\begin{equation}
  \re e^{-\delta_\varphi V_{t_0}(\zeta,z)} = e^{-\re \delta_\varphi V_{t_0}(\zeta,z)} \cos(\im \delta_\varphi V_{t_0}(\zeta,z)),
\end{equation}
with $\cos(x) \geq \frac12$ if $|x| \leq \frac12$ and $\cos(x) \geq -1$ for all $x\in \R$,
for $|\im z|+\|\im \varphi\|_{L^\infty(\Lambda)} \leq \eta$, it follows
from \eqref{e:delta-Vt0} and \eqref{e:imdelta-Vt0} that 
\begin{align}
  &\re e^{-\delta_\varphi V_{t_0}(\zeta,z)}\nnb
  &\geq 
    \frac12 e^{-C|\Lambda| (1+\|\nabla\zeta\|_{L^\infty(\Lambda)})} {\bf 1}\qa{C\eta|\Lambda| (1+\|\nabla\zeta\|_{L^\infty(\Lambda)})\leq \frac12}
    \nnb
    &\qquad \qquad - e^{+C|\Lambda|(1+\|\nabla\zeta\|_{L^\infty(\Lambda)})}{\bf 1}\qa{C\eta |\Lambda| (1+\|\nabla\zeta\|_{L^\infty(\Lambda)})> \frac12}
    \nnb
  &\geq \frac12 e^{-C|\Lambda| (1+\|\nabla\zeta\|_{L^\infty(\Lambda)})} - \frac{3}{2}e^{+C|\Lambda|(1+\|\nabla\zeta\|_{L^\infty(\Lambda)})}{\bf 1}\qa{C\eta |\Lambda| (1+\|\nabla\zeta\|_{L^\infty(\Lambda)})> \frac12}
 .
\end{align}
Using the Gaussian estimates \eqref{e:Gaussian1}--\eqref{e:Gaussian2}, there is a another small constant $c>0$ (which can depend on $\Lambda$) such that
\begin{equation}
  \re e^{-V_t(\varphi,z)+V_{t_0}(\varphi,z)}
  =
  \E_{C_t-C_{t_0}}\qB{ \re  e^{-\delta_\varphi V_{t_0}(\zeta,z)}}
  \geq 
  \frac12 e^{-|\Lambda|/c} - e^{+|\Lambda|^2/c} e^{-c/(\eta|\Lambda|)^2} \geq e^{-C(\Lambda)}
\end{equation}
for some $C(\Lambda)>0$, provided $\eta = \eta(\rho,\Lambda)$ is sufficiently small.
The other two estimates are similar. In particular, by \eqref{e:delta-Vt0} and the Gaussian estimate  \eqref{e:Gaussian2},
\begin{equation}
  \E_{C_t-C_{t_0}}\qB{e^{-\re (V_{t_0}(\varphi+\zeta,z)-V_{t_0}(\varphi,z))}}
  \leq
  \E_{C_t-C_{t_0}}\qB{e^{C|\Lambda|(1+\|\nabla \zeta\|_{L^\infty(\Lambda)})}} \leq e^{+C(\Lambda)},
\end{equation}
and from
\begin{equation}
  \im e^{-\delta_\varphi V_{t_0}(\zeta,z)} = e^{-\re \delta_\varphi V_{t_0}(\zeta,z)} \sin (\im \delta_\varphi V_{t_0}(\zeta,z)),
\end{equation}
using \eqref{e:delta-Vt0} and \eqref{e:imdelta-Vt0} and the Gaussian bound \eqref{e:Gaussian3}, 
we conclude
\begin{align}
  |\im e^{-\delta_\varphi V_{t}(\zeta,z)}|
  &\leq \E\qB{e^{|\delta_\varphi V_{t_0}(\zeta,z)|} |\im \delta_\varphi V_{t_0}(\zeta,z)|}
\nnb
  &\leq \E\qB{e^{C|\Lambda|(1+\|\nabla\zeta\|_{L^\infty(\Lambda)})} C(|\im z|+\|\im \varphi\|_{L^\infty(\Lambda)})|\Lambda|(1+\|\nabla\zeta\|_{L^\infty(\Lambda)})}
    \nnb
    &
  \leq C(\beta,z,\Lambda)(|\im z|+\|\im \varphi\|_{L^\infty(\Lambda)}),\label{temp4gcvf}
\end{align}
this inequality, together with \eqref{e:Vt-Vt0}, proves \eqref{temp65trm}.
\end{proof}

\begin{proof}[Proof of Proposition~\ref{prop:Vt}]
  The case $t\leq t_0$ was proven in Lemma \ref{le:vtsmall}.
  We therefore focus on $t>t_0$ in the following (we are assuming that $\eta>0$ is such that Lemma \ref{le:vtsmall} holds for $t_0$).
  
  \subproof{item (i)}
  We begin by writing 
  \begin{align}
  V_t(\varphi,z|\Lambda,m,\epsilon)-V_t(0,z|\Lambda,m,\epsilon)&=V_t(\varphi,z|\Lambda,m,\epsilon)-V_{t_0}(\varphi,z|\Lambda,m,\epsilon)\nnb
  &\quad +V_{t_0}(\varphi,z|\Lambda,m,\epsilon)-V_{t_0}(0,z|\Lambda,m,\epsilon)\nnb
  &\quad -(V_t(0,z|\Lambda,m,\epsilon)-V_{t_0}(0,z|\Lambda,m,\epsilon)).
  \end{align}
  We know that the middle term on the right-hand side converges.
  Since the last term is a special case of the first one,
  it is thus sufficient to prove that the first term converges for $z\in I \cap B_\rho(0)$ and $\varphi\in B_{\eta(\rho)}$.
  We write this term as   
 \begin{equation}
   V_t(\varphi,z|\Lambda,m,\epsilon)-V_{t_0}(\varphi,z|\Lambda,m,\epsilon)=-\log \E_{C_t-C_{t_0}}\left[e^{-V_{t_0}(\varphi+\zeta,z|\Lambda,m,\epsilon)+V_{t_0}(\varphi,z|\Lambda,m,\epsilon)}\right]
   .
 \end{equation}
  Combining Lemma \ref{le:logest} and Lemma \ref{le:gaussianest}, we see that for small enough $\eta>0$ (depending on $\beta,\rho,\Lambda$), the logarithm is well defined and using dominated convergence, we can take the $\epsilon\to 0$ limit using the convergence of $V_{t_0}$ (provided by Lemma \ref{le:vtsmall}). All of the estimates are also uniform in $m$, so the $m\to 0$ limit follows from this argument as well. Let us note that our proof yields the following representation:
  \begin{align}\label{eq:largetrep}
    V_t(\varphi,z)-V_t(0,z)&=-\log \E_{C_t-C_{t_0}}\left[e^{-V_{t_0}(\varphi+\zeta,z)+V_{t_0}(\varphi,z)}\right]\nnb
                             &\quad +V_{t_0}(\varphi,z)-V_{t_0}(0,z)\nnb
  &\quad +\log \E_{C_t-C_{t_0}}\left[e^{-V_{t_0}(\zeta,z)+V_{t_0}(0,z)}\right],
  \end{align}
  where we have again suppressed the dependence on $\Lambda$ (and possibly $m$) in our notation.
  
  This concludes the proof of item (i).
   
  \subproof{item (ii)}
  For the derivative, we note that by the Leibniz rule, we find from \eqref{eq:largetrep} in item (i) and differentiability of $V_{t_0}$ (from Lemma \ref{le:vtsmall}) 
  \begin{equation}\label{eq:nablatrep}
  \nabla_\varphi V_t(\varphi,z)=\frac{\E_{C_t-C_{t_0}}[e^{-V_{t_0}(\varphi+\zeta,z)+V_{t_0}(\varphi,z)}\nabla_\varphi V_{t_0}(\varphi+\zeta,z)]}{\E_{C_t-C_{t_0}}[e^{-V_{t_0}(\varphi+\zeta,z)+V_{t_0}(\varphi,z)}]}.
\end{equation}
  Note that we can also write this as stated in \eqref{e:nablaVt-Vt0}:
    \begin{equation} \label{e:nablaVt-Vt0-bis}
    \nabla_\varphi V_t(\varphi,z)
    = e^{+V_t(\varphi,z)-V_{t_0}(\varphi,z)}\E_{C_t-C_{t_0}}\qB{e^{-V_{t_0}(\varphi+\zeta,z)+V_{t_0}(\varphi,z)} \nabla_\varphi V_{t_0}(\varphi+\zeta,z)}.
  \end{equation}
  To be precise, the integral on the right-hand side is a Bochner integral of the random variable $\nabla_\varphi V_{t_0}(\varphi+\zeta,z)$
  which takes values in the separable Banach space $C(\bar\Lambda)$,
  and for avoidance of doubt we observe that the Leibniz rule can be verified by hand:
  \begin{align}\label{eq:nablatrep-pf}
        &V_t(\varphi,z)-V_t(\varphi+f,z)
          \nnb
        &=\int_0^1 \ddp{}{h} \log \E_{C_t-C_{t_0}}[e^{-V_{t_0}(\varphi+hf+\zeta,z)+V_{t_0}(\varphi+hf,z)}] \, dh
          \nnb
        &=\int_0^1 \frac{\E_{C_t-C_{t_0}}[e^{-V_{t_0}(\varphi+hf+\zeta,z)+V_{t_0}(\varphi+hf,z)}(f,\nabla_\varphi V_{t_0}(\varphi + hf+\zeta,z))]}{\E_{C_t-C_{t_0}}[e^{-V_{t_0}(\varphi+hf+\zeta,z)+V_{t_0}(\varphi+hf,z)}]}\, dh
          \nnb
        &=\frac{\E_{C_t-C_{t_0}}[e^{-V_{t_0}(\varphi+\zeta,z)+V_{t_0}(\varphi,z)}(f,\nabla_\varphi V_{t_0}(\varphi + \zeta,z))]}{\E_{C_t-C_{t_0}}[e^{-V_{t_0}(\varphi+\zeta,z)+V_{t_0}(\varphi,z)}]}
          + O(\|f\|_{L^\infty(\Lambda)}^2+\|\nabla f\|_{L^\infty(\tilde\Lambda)}^2)
          \nnb
        &= (f,\nabla_\varphi V_t(\varphi,z)) + O(\|f\|_{L^\infty(\Lambda)}^2+\|\nabla f\|_{L^\infty(\tilde\Lambda)}^2),
  \end{align}
  using Lemma \ref{le:logest}, Lemma \ref{le:gaussianest}, and our proof of item (ii) in Lemma \ref{le:vtsmall}.
  
  \subproof{item (iii)}
  We start from the representation \eqref{e:nablaVt-Vt0-bis} and consider separately the cases $t\geq 2t_0$ and $t_0\leq t<2t_0$. Our argument relies on the fact that
  \begin{equation}
    \|(\sqrt{t}\nabla)^ke^{\Delta t}f\|_{L^\infty}\leq C_k \|f\|_{L^\infty},
    \qquad
    \|(\sqrt{t}\nabla)^ke^{\Delta t}f\|_{L^\infty}\leq C_k t^{-1}\|f\|_{L^1},
  \end{equation}
  for some constant $C_k$ depending only on $k$ and the fact that $\nabla^k e^{t\Delta}f=e^{t\Delta}\nabla^k f$ for smooth $f$.
  
  In the case $t\geq 2t_0$, we make use of the bound $\|(\sqrt{t}\nabla)^ke^{\Delta t}f\|_{L^\infty}\leq C_k t^{-1}\|f\|_{L^1}$ to deduce that
  (for a possibly different $C_k$ but still depending only on $k$)
  \begin{align}
  \|(\sqrt{t}\nabla_x)^k e^{\Delta_x t}\nabla_\varphi V_{t_0}(\varphi,z)\|_{L_x^\infty}&=\left(\frac{t}{t-t_0}\right)^{k/2}\|(\sqrt{t-t_0}\nabla_x)^k e^{\Delta_x(t-t_0)}e^{\Delta_x t_0}\nabla_\varphi V_{t_0}(\varphi,z)\|_{L_x^\infty}\nnb
    &\leq C_k t^{-1} \|e^{\Delta_x t_0}\nabla_\varphi V_{t_0}(\varphi,z)\|_{L_x^1}
      .
  \end{align}
  For $|\im z|+\|\im \varphi\|_{L^\infty(\Lambda)}$ sufficiently small, the right-hand side is bounded by
  the small-$t$ estimate  \eqref{e:nablaV-L1} with $t=t_0$ (which is a constant) which implies
  \begin{equation}
    \sup_{\varphi\in B_\eta} \|e^{\Delta_x {t_0}}\nabla_\varphi V_{t_0}(\varphi,z)\|_{L_x^1} \leq C_{ k}(\beta,\rho,\Lambda).
  \end{equation}
  Therefore, for $t\geq 2t_0$, by \eqref{e:nablaVt-Vt0-bis},
  \begin{align}
    &\|(\sqrt{t}\nabla_x)^ke^{\Delta_x t}\nabla V_t(\varphi,z)\|_{L_x^\infty}
      \nnb
    &\leq C_{k}(\beta,\rho,\Lambda) t^{-1}|e^{+V_t(\varphi,z)-V_{t_0}(\varphi,z)}| \E_{C_t-C_{t_0}}\qB{|e^{-V_{t_0}(\varphi+\zeta,z)+V_{t_0}(\varphi,z)}|}
      \nnb
    &= C_{k}(\beta,\rho,\Lambda) t^{-1} |e^{+V_t(\varphi,z)-V_{t_0}(\varphi,z)}|
      \E_{C_t-C_{t_0}}\qB{e^{-\re (V_{t_0}(\varphi+\zeta,z)-V_{t_0}(\varphi,z))}}
      .
  \end{align}
  For $z$ and $\varphi$ real, we remark that one would have the simple identity
  \begin{equation}
    |e^{V_t(\varphi,z)-V_{t_0}(\varphi,z)}| \E_{C_t-C_{t_0}}\qB{e^{-\re (V_{t_0}(\varphi+\zeta,z)-V_{t_0}(\varphi,z))}} = 1.
  \end{equation}
  In general, the first factor is bounded (with a $\Lambda$-dependent constant) by Lemma~\ref{lem:re-im-eV} since
  \begin{equation}
    |e^{V_t(\varphi,z)-V_{t_0}(\varphi,z)}|
    \leq  \frac{1}{|\re e^{-V_t(\varphi,z)+V_{t_0}(\varphi,z)}|} \leq e^{+C}.
  \end{equation}  
  The second factor is treated in Lemma \ref{lem:re-im-eV}.
  
  For the $t_0\leq t<2t_0$-case, we write instead 
  \begin{align}
  \|(\sqrt{t}\nabla_x)^k e^{\Delta_x t}\nabla_\varphi V_{t_0}(\varphi,z)\|_{L_x^\infty}&=\left(\frac{t}{t_0}\right)^{k/2}\|e^{\Delta_x(t-t_0)} (\sqrt{t_0}\nabla_x)^k e^{\Delta_x t_0}\nabla_\varphi V_{t_0}(\varphi,z)\|_{L_x^\infty}\\
  &\leq C_k \|(\sqrt{t_0}\nabla_x)^k e^{\Delta_x t_0}\nabla_\varphi V_{t_0}(\varphi,z)\|_{L_x^\infty}
  \end{align}
  and make use of \eqref{e:nablaV-bd-finvol}, which allows completing the argument as in the case of $t\geq 2t_0$.
   
Next we prove \eqref{e:nablaV-Lipschitz-finvol} followed by \eqref{e:nablaV-imbd}.
We use a Cauchy estimate and the analyticity from item (iv),
emphasizing that the proof for the large $t$ regime of item (iv) given below is independent of item (iii).
Pick any $\varphi_0\in B_\eta$ with $\im \varphi_0=0$ and note that
the $L_x^\infty$ ball of radius $\eta$ around $\varphi_0$ lies inside $B_\eta$.
Using \eqref{e:nablaV-bd-finvol} on this ball with $k=0$,
the standard Cauchy estimate in $L_x^\infty$ gives, for $\varphi$ in the $\eta/2$ ball around $\varphi_0$,
\begin{equation}
\|\nabla_\varphi e^{\Delta t}\nabla_\varphi V_t(\varphi,z)\|_{L^\infty_x} \leq \frac{C(\beta,\rho,\Lambda)}{\eta/2}(1+t^{-\beta/8\pi}).
\end{equation}
Hence when we have that for $\varphi,\varphi'\in B_\eta$ in this ball with $\|\varphi-\varphi'\|_{L^\infty_x} \leq \eta/2$,
the claim \eqref{e:nablaV-Lipschitz-finvol} holds by the fundamental theorem of calculus in $L^\infty_x$.
To see \eqref{e:nablaV-imbd}, observe that $\|\im e^{\Delta t}\nabla_\varphi V_t(\varphi,z)\|_{L^\infty_x} =\| \im (e^{\Delta t}\nabla_\varphi V_t(\varphi,z)- e^{\Delta t}\nabla_\varphi V_t(\re\varphi,\re z))\|_{L^\infty_x} $, which is bounded above by
\begin{equation}
  \| e^{\Delta t}\nabla_\varphi V_t(\varphi,z)-e^{\Delta t}\nabla_\varphi V_t(\varphi,\re z)\|_{L^\infty_x}
  +\| e^{\Delta t}\nabla_\varphi V_t(\varphi,\re z)-e^{\Delta t}\nabla_\varphi V_t(\re \varphi,\re z)\|_{L^\infty_x}.
\end{equation}
So \eqref{e:nablaV-imbd} holds by using  \eqref{e:nablaV-Lipschitz-finvol} to bound the second term, and the first term is bounded by applying the standard (one-dimensional) Cauchy estimate and the fundamental theorem of calculus in the $z$ coordinate.
Note that we have shrunk the width $\eta$ of $B_\eta$ by $\eta/2$, but this does not matter.

The proof of \eqref{e:nablaVteps-conv} is analogous to that of \eqref{e:Vteps-conv}.

\subproof{item (iv)}
We begin with the following consequence of \eqref{eq:nablatrep} (justified by Fubini
to exchange the heat kernel integral with the expectation): For every $\rho>0$ and $z\in B_\rho(0)$,
\begin{equation}
  e^{\Delta_x t}\nabla_\varphi V_t(\varphi,z)=  \frac{\E_{C_t-C_{t_0}}[e^{-V_{t_0}(\varphi+\zeta,z)+V_{t_0}(\varphi,z)}e^{\Delta_x t}\nabla_\varphi V_{t_0}(\varphi+\zeta,z)]}{\E_{C_t-C_{t_0}}[e^{-V_{t_0}(\varphi+\zeta,z)+V_{t_0}(\varphi,z)}]},
\end{equation}
where we recall that $t_0 = t_0(\beta, \rho)$.
We will show that the numerator and denominator are both analytic functions
in $(\varphi,z) \in B_{\eta(\rho)} \times B_\rho(0)$,
and that the denominator does not vanish for $z$ in the intersection of $B_\rho(0)$ with a neighborhood of the real axis. 

Indeed, 
the analyticity of the 
quantities within the expectations in $B_{\eta(\rho)} \times B_\rho(0)$
follow from the series expansions used in the proof for $t\leq t_0$
by choice of $t_0$ and $\eta$.
With an application of the Leibniz rule (justified as in \eqref{eq:nablatrep-pf}),
we conclude that the expectations in the numerator and denominator are analytic in this set as well.
It remains to show that there exists a neighborhood $I$ of the real axis such that the denominator does not vanish in $B_{\eta(\rho)}\times (I \cap B_\rho(0))$. This follows from writing the denominator as $e^{-V_t(\varphi,z)+V_{t_0}(\varphi,z)}$ and using Lemma \ref{lem:re-im-eV}.
\end{proof}

\subsection{Coupling and analyticity in finite volume}
\label{sec:coupling}

From Section~\ref{sec:Gauss-decomp}, recall that for 
for $m >0$, we realize the scale decomposition of the Gaussian free field as
\begin{equation} \label{e:GFF-Wiener-bis}
  \Phi^{\GFF(m)}_t = \int_t^\infty \sqrt{\dot C_s}\, dW_s = \int_t^\infty e^{-\frac12 m^2 s} e^{\frac12 \Delta_x s}\, dW_s
\end{equation}
where $(W_t)$ is a cylindrical Brownian motion on $L^2(\R^2)$ defined on some probability space fixed from now on.
In particular, $\Phi^{\GFF(m)}_0$ is distributed
as a Gaussian free field with mass $m >0$ on $\R^2$
and the decomposition $\Phi^{\GFF(m)}$ takes values in $C((0,\infty),C^\infty(\R^2))$.

We will prove that the sine-Gordon field with large scale regularization $m>0$ and $\Lambda \subset \R^2$ a bounded domain can be realized as the $t\to 0$ limit of the solution of the SDE
\begin{equation} \label{e:SDE}
  \Phi_t = -\int_t^\infty e^{-m^2 s} e^{\Delta_x s} \nabla_\varphi V_s(\Phi_s,z,x|\Lambda,m)\, ds + \Phi^{\GFF(m)}_t,
\end{equation}
where $V_t$ is the renormalized potential (see Proposition~\ref{prop:Vt}),
and that the solution extends analytically to $z$ in a complex neighborhood $I$ of the real axis.
This construction is similar to the one from \cite{MR4399156} or \cite{MR4798104}, with the small
difference that we work directly in the continuum, and the more important difference that we allow $z$ to be complex.
One should think of the scale parameter $t \in [0,+\infty]$ to run backwards starting from $t=+\infty$ with initial condition $\Phi_\infty = 0$.

To show this, we first show that the SDE is well posed and has the required regularity in its parameters.
We write $C^r(\R^2)$, $r \in (0,1)$ for the space of H\"older continuous functions on $\R^2$ vanishing at infinity, see Section~\ref{sec:notation}.

It will be useful later to consider the $\epsilon$ regularized version of the equation as well,
in which case $t$ is restricted to $[\epsilon^2,+\infty]$
and $\nabla_\varphi V_s(\Phi_s,z,x|\Lambda,m)$ is replaced by $\nabla_\varphi V_s(\Phi_s,z,x|\Lambda,m,\epsilon)$.
For the continuity in $\epsilon$ in the
following proposition, we extend this solution from $[\epsilon^2,+\infty]$ to $[0,+\infty]$ by
\begin{equation} \label{e:SDE-eps}
  \Phi_t = -\int_{t\vee \epsilon^2} ^\infty e^{-m^2 s} e^{\Delta_x s} \nabla_\varphi V_s(\Phi_s,z,x|\Lambda,m,\epsilon)\, ds + \Phi^{\GFF(m)}_t.
\end{equation}
Thus the scales $t<\epsilon^2$ are defined only in terms of the driving free field term.

\begin{proposition} \label{prop:SDE-existence}
  Let $\beta\in (0,6\pi)$ and $r < 2-\beta/4\pi$.
  Then for any %
  bounded domain $\Lambda \subset \R^2$
  and $m>0$,
  there is a neighborhood $I \subset \C$ containing the real axis (that depends on $\Lambda$ but is independent of $m>0$)
  such that for
  $z\in I$,
  there is a unique strong solution
  \begin{equation}
    (\Phi_t) = (\Phi^{\SG(\beta,z|\Lambda,m)}_t)
  \end{equation}
  to \eqref{e:SDE} such that
  \begin{equation}
    \Phi^{\Delta(\beta,z|\Lambda,m)}=(\Phi^{\SG(\beta,z|\Lambda,m)}_t-\Phi^{\GFF(m)}_t) \in 
    C([0,+\infty],C^r(\R^2)) \cap C((0,+\infty], C^\infty(\R^2)).
  \end{equation}
  The solution $\Phi^{\Delta(\beta,z|\Lambda,m)}$ is analytic as a function
  of $z \in I$ taking values in $C([0,+\infty],C^r(\R^2))$,
  and $|\im \Phi^{\Delta(\beta,z|\Lambda,m)}| \leq \eta$ where $\eta>0$ can be made arbitrarily small (by making $I$ small).
  
  The same statement also holds uniformly in the ultraviolet regularization $\epsilon>0$
  as in \eqref{e:SDE-eps},  and the resulting
  solutions converge uniformly as $\epsilon \to 0$, in these sense that almost surely:
  \begin{equation}
    \sup_{t\geq 0}\|\Phi^{\SG(\beta,z|\Lambda,m)}_t- \Phi^{\SG(\beta,z|\Lambda,m,\epsilon)}_t\|_{C^r(\R^2)} \to 0.
  \end{equation}
\end{proposition}

In the remainder of this section (and in  Section~\ref{sec:coupling-infvol}) we will frequently
denote by $\Phi^{\Delta}_t$ and $\Phi^\Delta_{t,t'}$ the difference fields
\begin{equation}
  \Phi^\Delta_t = \Phi^\Delta_{t,\infty},\qquad
  \Phi^{\Delta}_{t,t'} %
  = -\int_t^{t'} e^{-m^2 s} e^{\Delta_x s} \nabla_\varphi V_s(\Phi_s,z,x|\Lambda,m)\, ds
\end{equation}
appearing in the formulation of the SDE \eqref{e:SDE}.

\begin{proof}
  Throughout the proof, $\Lambda\subset \R^2$ bounded and $m>0$ will be fixed, and we will thus omit them from the notation.
  Pathwise in $\Phi^\GFF$,
  we use a fixed point argument in a weighted space of trajectories which $\Phi^\Delta_t$ will take values in.
  For $u \in D_\eta \subset X$ where
  $X= C^0([0,\infty), C^0(\R^2,\C))$
  and the domain $D_\eta$ is defined below,
  define
  \begin{equation}
    F(u,z)_t = -\int_t^\infty e^{-m^2 s} e^{\Delta_x s} \nabla_\varphi V_s(\Phi^{\GFF}_s+u_s,z) \, ds,
  \end{equation}
  so that the problem is to find a fixed point of $F(\cdot,z)$ regarded as a map acting
  on $D_\eta$. %
  Let $\eta>0$ and $\rho>0$ be fixed below.
  We will impose the conditions $|z|\leq \rho$ and $\|\im u\|_{L^\infty} \leq \eta$ eventually.
  We use the following norm on $u: [0,\infty) \to C^0(\R^2, \C)$: %
  \begin{equation}
    \|u\|_X
    = \sup_{t\geq 0} w(t)^{-1} \|u_t\|_{L^\infty(\R^2)}
  \end{equation}
  with time-dependent weight $w(t)$.
  In terms of a large constant $C(\beta,\rho)>0$ given in terms of the implicit constants in Proposition~\ref{prop:Vt}, we choose the weight by solving
  \begin{equation}
    w(t)^{-1}  C(\beta,\rho) \int_t^\infty w(s)  e^{-m^2 s}(1+s^{-\beta/8\pi}) \, ds = \frac12,
  \end{equation}
  with boundary condition  $w(t) \to 1$ as $t\to\infty$.
  This equation is equivalent to 
  \begin{equation}
    w'(t) = -2C(\beta,\rho) w(t)  e^{-m^2 t}(1+t^{-\beta/8\pi}),
  \end{equation}
  and solved by
  \begin{equation}
    w(t) = \exp\qa{\int_t^\infty 2C(\beta,\rho) e^{-m^2 s} (1+s^{-\beta/8\pi})\, ds}.
  \end{equation}
  Thus this is our definition of the weight.
  The normed space %
  $X=C^0([0,\infty),C^0(\R^2,\C))$ with norm $\|u\|_X$ is a complex Banach space.

  Let $D_\eta = \{ u \in X: \|\re u\|_X \leq 1, \|\im u\|_{X} \leq \eta\}$.
  We will show that $u \mapsto F(u,z)$ maps $D_\eta$ to itself and is a contraction on $D_\eta$.
  Indeed, choosing $C(\beta,\rho)$ sufficiently large,
  \eqref{e:nablaV-bd-finvol} and \eqref{e:nablaV-imbd} imply that for $u \in D_\eta$,
  $|z|\leq \rho$ and $|\im z|\leq \eta$,

  \begin{gather}
    \|\re F(u)\|_X  \leq \sup_t\qa{ w(t)^{-1}C(\beta,\rho)\int_t^\infty w(s) e^{-m^2 s}(1+s^{-\beta/8\pi})ds}\leq \frac 12,
    \\
    \|\im F(u)\|_X  \leq \sup_t\qa{ w(t)^{-1}C(\beta,\rho)\eta \int_t^\infty w(s) e^{-m^2 s}(1+s^{-\beta/8\pi})ds} \leq \frac{\eta}{2},
  \end{gather}
  where we also used that $\Phi^\GFF_s$ is real valued.
  Similarly, \eqref{e:nablaV-Lipschitz-finvol} implies that, if $u,v \in D_\eta$,
  \begin{equation} \label{e:F-contract}
    \|F(u)-F(v)\|_X \leq \qa{  \sup_t w(t)^{-1}  C(\beta,\rho)\int_t^\infty w(s) e^{-m^2 s} (1+s^{-\beta/8\pi}) } \|u-v\|_X \leq \frac12 \|u-v\|_X.
  \end{equation}
  Thus $F(\cdot, z)$ maps $D_\eta$ to itself and is a contraction.
  By the Banach fixed point theorem, there is a fixed point $u=u(z) \in D_\eta$ for the map $u\mapsto F(u,z)$.

We next show that $u(z)$ is analytic in $z \in I$.
The idea for proving analyticity of $u$ (as an element of $X$) %
is to apply the analytic implicit function theorem to the function $(z,u)\mapsto u-F(u,z)$,
see \cite[Appendix~B]{MR0894477} for a Banach space valued version of this theorem.
To apply this theorem, we need to show two things: first of all, we need to show that for each $z\in I$, there exist open sets $U\subset I$ and $V\subset X$ such that $z\in U$, $u(z)\in V$ (where $u(z)$ is our fixed point), and $F:U\times V\to X$ is analytic. %
The second thing we need to show is that $D_u(u-F(u,z))$ evaluated at $(z,u(z))$ (where $u(z)$ is our fixed point)
is a linear isomorphism from $X$ to $X$. %
This follows from \eqref{e:F-contract} which implies that $\|D_uF(u,z)\| \leq \frac12$ so that $D_u(u-F(u,z)) = \id_X - D_uF(u,z)$ is invertible.

For analyticity, we note that by choosing $U\subset I$  small enough, and choosing $V$ to be $\bigcup_{w\in U}B_{X}(u(w),\delta)$ with $\delta$ small enough (to ensure that the imaginary part of $u$ is small enough), one can argue as in the proof of item (iv) of Proposition~\ref{prop:Vt} to see that $(z,u)\mapsto F(u,z)$ is analytic in $U\times V$.

  Now that we know $u$ exists,
  the H\"older continuity estimates follow by integrating the estimates for the renormalized potential.
  More precisely, we show that $F_0$ actually maps into $C^\eta$ and that $F_t$ maps into $C^k$ for any $k$ if $t>0$.
  To see that $F_0$ maps into $C^r(\R^2)$, we note that
  \eqref{e:nablaV-bd-finvol} implies
  \begin{equation} \label{e:Vt-Holder}
    \|e^{\Delta_x s} \nabla_\varphi V_t(\varphi,z)\|_{C^r(\R^2)}
    \leq C_r(\beta,z) (1+t^{-r/2})(1+t^{-\beta/2}).
  \end{equation}
  For example, if $r \in (0,1)$, %
  \begin{align}
    &|e^{\Delta_x s} \nabla_\varphi V_t(\varphi,z,x)- e^{\Delta_x s} \nabla_\varphi V_t(\varphi,z,y)|
      \nnb
    &\leq \min\hB{|x-y| \|\nabla_xe^{\Delta_x s} \nabla_\varphi V_t(\varphi,z)\|_{L_x^\infty}, 2\|e^{\Delta_x s} \nabla_\varphi V_t(\varphi,z)\|_{L_x^\infty}}
      \nnb
    &\lesssim  \min\hB{\frac{|x-y|}{\sqrt{t}}, 1} \leq \frac{|x-y|^r}{t^{r/2}}
      ,
  \end{align}
  and a similar argument would show H\"older estimates with $r>1$ (but we do not actually need these).
  Therefore, for $r < 2-\beta/4\pi$,
  \begin{align}
    \|F_0(u,z)\|_{C^r}
    &\leq \int_0^\infty e^{-m^2 s} \|e^{\Delta_x s} \nabla_\varphi V_s(\Phi^\GFF_s+u_s,z)\|_{C^r} \, ds
      \nnb
    &\lesssim \int_0^\infty e^{-m^2 s} (1+s^{-r/2}) (1+s^{-\beta/8\pi}) \, ds
      \lesssim 1.
  \end{align}
  That $F_t$ maps into $C^k$ is essentially immediate from the fact that $e^{\Delta_x t}\nabla_\varphi V_t$ takes values in $C^k$.

  Analyticity of the solutions follows by observing that the map
  \begin{equation} \label{e:F0-Ft-analytic}
    F_0: D_\eta \times I \to C^r(\R^2) %
  \end{equation}
  is actually analytic for any $t>0$.
  This is again seen as in the proof of item (iv) of Proposition~\ref{prop:Vt}.
  Since $u(z)= F(u(z),z)$ and $z \mapsto u(z)$ is analytic taking values in $D_\eta$
  this implies the analyticity of $z\mapsto u_0(z)$ taking values in $C^r(\R^2)$.
  
  For convergence as $\epsilon\to 0$, first observe that 
  it is clear from the above proof that all estimates apply also if $\epsilon >0$ and are uniform in $\epsilon >0$.
  To show convergence it suffices to show that $(\Phi_t)_{t \geq \delta}$ converges for any fixed $\delta>0$.
  Indeed, as $\delta\to 0$,
 \begin{equation}
   \sup_{t\leq \delta} \|\Phi_0^\Delta-\Phi_t^\Delta\|_{C^r} \lesssim \int_0^\delta s^{-r/2-\beta/8\pi} \, ds \lesssim \delta^{1-\beta/8\pi-r/2} \to 0
 \end{equation}
 uniformly in $\epsilon$. Then, by  \eqref{e:nablaVteps-conv} in Proposition~\ref{prop:Vt}, for any $\varphi \in B_I$,
  \begin{equation} \label{e:nablaVteps-conv-bis}
    \|(\sqrt{t}\nabla_x)^ke^{\Delta_x t}\nabla_\varphi V_t(\varphi,z|\Lambda,m,\epsilon)
    -
    (\sqrt{t}\nabla_x)^ke^{\Delta_x t}\nabla_\varphi V_t(\varphi,z|\Lambda,m)\|_{L_x^\infty} \to 0.
  \end{equation}
  Thus we are essentially in the situation of an SDE with smooth coefficients and
  \begin{align}
    f_k(t,\epsilon)
    &= \|(\sqrt{t}\nabla_x)^k(\Phi_t^{\Delta(\beta,z|\Lambda,m,\epsilon)} -     \Phi_t^{\Delta(\beta,z|\Lambda,m)})\|_{L_x^\infty}
      \nnb
    &\lesssim r_k(t,\epsilon) + \int_t^\infty e^{-m^2 s} (1+s^{-\beta/8\pi})\|\Phi_t^{\Delta(\beta,z|\Lambda,m,\epsilon)} -     \Phi_t^{\Delta(\beta,z|\Lambda,m)}\|_{L_x^\infty} \, ds
      \nnb
      &\lesssim r_k(t,\epsilon) + \int_t^\infty e^{-m^2 s} (1+s^{-\beta/8\pi}) f_0(s,\epsilon) \, ds
  \end{align}
  where
  \begin{equation}
    r_k(t,\epsilon) = \int_t^\infty \|(\sqrt{t}\nabla_x)^k[\nabla_\varphi V_t(\Phi^{\GFF}_s+\Phi^{\Delta}_s|\Lambda,m)-\nabla_\varphi V_t(\Phi^{\GFF}_s+\Phi^{\Delta}_s|\Lambda,m,\epsilon)]\|_{L_x^\infty} \, ds
  \end{equation} 
  tends to $0$ for every $t > 0$ by \eqref{e:nablaVteps-conv-bis} and the dominated convergence theorem. 
  By a Gronwall argument, we conclude that $f_0(t,\epsilon) \to 0$ for every $t>0$ and using that then
  also $f_k(t,\epsilon) \to 0$ for every $k\in \N$ and $t>0$.  
\end{proof}

\begin{proposition} \label{prop:SDE-Ito}
  Under the assumptions of Proposition~\ref{prop:SDE-existence} and in addition $z\in \R$, for $\epsilon >0$,
  \begin{equation}
    \Phi_{\epsilon^2}^{\SG(\beta,z|\Lambda,m,\epsilon)} \sim \avg{\cdot}_{\SG(\beta,z|\Lambda,m,\epsilon)},
  \end{equation}
  and, as a consequence, this statement also applies in the continuum limit $\epsilon \to 0$, i.e.,
  \begin{equation}
    \Phi_{0}^{\SG(\beta,z|\Lambda,m)} \sim \avg{\cdot}_{\SG(\beta,z|\Lambda,m)}.
  \end{equation}
\end{proposition}

\begin{proof}
  It suffices to show the statement with regularization $\epsilon>0$.
  As $\epsilon\downarrow 0$, both sides converge to the respective sides in the claimed statement. Indeed, for $\avg{\cdot}_{\SG(\beta,z|\Lambda,m)}$ this is our
  definition of the sine-Gordon measure, see Section~\ref{sec:mixing},
  and for the SDE it follows from Proposition~\ref{prop:SDE-existence}.

  The proof  of this is an application of It\^o's formula,
  as explained in \cite[Section~4]{MR4798104} for example in a finite-dimensional setting (see also \cite{MR4399156}).
  For the reader's orientation, we remark that
  we use the smooth continuum regularization \eqref{e:SGdef} of the sine-Gordon measure
  that is not finite dimensional (but has other advantages), but that the difference is technical
  (and we could have used a finite-dimensional  lattice regularization instead).

  We now recall this argument.
  First, for $\epsilon>0$, $m>0$, $L<\infty$,
  it follows from its definition \eqref{e:Vtdef} that
  the renormalized potential 
  $V_t(\varphi) = V_t(\varphi,z|\Lambda,m,\epsilon)$ satisfies the Polchinski equation
\begin{equation} \label{e:Polchinski}
  \ddp{V_t}{t}(\varphi) = \frac12 \Delta_{\dot C_t} V_t(\varphi) - \frac12 (\nabla V_t(\varphi), \dot C_t \nabla V_t(\varphi))
\end{equation}
where
\begin{equation}
  \Delta_{\dot C_t} V_t(\varphi) = \int \He V_t(\varphi,x,y) \dot C_t(x-y)\, dx\, dy,
\end{equation}
and the gradient and Hessian of $V_t(\varphi,z|\Lambda,m,\epsilon)$ can be written as
\begin{align} \label{e:nablaVHessVkernels}
    (\nabla V_t(\varphi), f) &= \int_\Lambda \nabla V_t(\varphi,x) f(x) \, dx,
    \nnb
    (f,\He V_t(\varphi)f) &= \int_{\Lambda\times\Lambda} \He V_t(\varphi,x,y) f(x)f(y)\, dx\, dy.
  \end{align}
  Since we are considering the regularized version, the kernels $\nabla_\varphi V_t(\varphi,\cdot)$ and $\He_\varphi V_t(\varphi,\cdot,\cdot)$
  are bounded and supported in $\Lambda$ respectively $\Lambda \times \Lambda$; see for example \cite[Corollary 4.3]{MR4767492}.
  Let $F$ be a cylinder functional with $L^2(\R^2)$ gradient and Hessian analogously to the above satisfying
  \begin{equation}   \label{e:nablaFHessFkernels}
    DF(\varphi;f) = (\nabla F(\varphi), f), \qquad D^2 F(\varphi; f,f)= (f,\He F(\varphi)f).
  \end{equation}
  For example, if $F(\varphi) = e^{i(\varphi,g)}$ with $g\in C_c^\infty(\R^2)$ then
  $\nabla F(\varphi) = i e^{i(\varphi,g)} g \in L^\infty(\Lambda)$.
  Then $F_{s,t}$ defined by
  \begin{equation} \label{e:F-semigroup}
    F_{s,t}(\varphi) = e^{+V_t(\varphi)}\E_{C_{t}-C_s}\qb{e^{-V_s(\varphi+\zeta)} F(\varphi+\zeta)}.
  \end{equation}
  satisfies 
  \begin{equation} \label{e:F-semigroup2}
    (\ddp{}{t}-\LL_t) F_{s,t}(\varphi) = 0
  \end{equation}
  where
  \begin{equation}
    \LL_t F(\varphi) = \frac12 \Delta_{\dot C_t} F(\varphi) - (\dot C_t \nabla V_t(\varphi), \nabla F_t(\varphi)).
  \end{equation}
  This is verified exactly as in \cite[Proposition~3.5]{MR4798104}.

  To apply the It\^o formula in its usual formulation, one may
  reverse the time direction so that the SDE (which is adapted with respect to the backward filtration generated by $(\Phi^\GFF_s)_{s\geq t})$
  becomes a standard forward SDE, see \cite[Section~4.1]{MR4798104}.
  To be concrete, we can use the reparametrization $a(t)=1/t$ and set $\tilde\Phi_t = \Phi_{a(t)}$ %
  and $\tilde W_t=W_{a(t)}$. Then this forward SDE for $\tilde \Phi$ is:
  \begin{equation}
    \tilde \Phi_t = \int_0^t \dot a(s) \dot C_{a(s)}    \nabla V_{a(s)}(\tilde\Phi_s)\, ds
    + \int_0^t \sqrt{\dot C_{a(s)}} d\tilde W_{s}.
  \end{equation}
  As explained in \cite[Section~4.1]{MR4798104}, the It\^o formula  then implies that for suitable $F$:
  \begin{equation}
    \E_{\tilde \Phi_0=0}[F(\tilde\Phi_t)]
    \propto \E_{C_\infty-C_{a(t)}} \qa{e^{-V_{a(t)}(\Phi)} F(\Phi)},
  \end{equation}
  i.e.,
  \begin{equation}
    \E_{\Phi_\infty=0}[F(\Phi_t)]
    \propto \E_{C_\infty-C_{t}} \qa{e^{-V_{t}(\Phi)} F(\Phi)}.
  \end{equation}
  The same argument applies in our setting using the It\^o formula for Hilbert space valued Wiener processes
  from \cite[Theorem 4.32]{MR3236753}.
  To this end, the cylindrical Wiener process $W$ on $L^2(\R^2)$ used in \eqref{e:GFF-Wiener}
  can be realized as an $H$-valued Wiener process with $H$ a suitable larger Hilbert space
  into which $L^2(\R^2)$ is compactly embedded, see \cite[Section~4]{MR3236753}.
  The drift $\dot a(s)\dot C_{a(s)} \nabla V_{a(s)}(\tilde\Phi_s)$ is bounded in $L^2(\R^2)$
  (and thus $H$-valued).
  Provided that the function $F$ and its first two derivatives are uniformly continuous functions on $H$,
  the assumptions of \cite[Theorem~4.32]{MR3236753} are satisfied.
  For our purposes it suffices to consider cylinder functionals $F$ for which this is clear and
  the $H$-derivatives of $F$ in the direction $f \in L^2(\R^2)$ can be written
  as in \eqref{e:nablaFHessFkernels}.
  Using that $\tilde W$ has quadratic variation $-\dot a(s) \dot C_{a(s)}$ (note that $\dot a(s) = -1/s^2<0$),
  the It\^o formula then reads
  \begin{multline}
    F(a(r),\tilde\Phi_r)-F(a(0),\tilde\Phi_0) =
    \int_0^r (\nabla F(a(s),\tilde \Phi_s), \sqrt{\dot C_{a(s)}} d\tilde W_s)
    \\
    +
    \int_0^r \qa{\dot a(s) \ddp{F}{a(s)}(a(s),\tilde\Phi_s) - \frac12 \dot a(s) \Delta_{\dot C_{a(s)}} F(a(s),\tilde\Phi_s) + (\dot a(s)\dot C_{a(s)} \nabla V_{a(s)}(\tilde\Phi_s), \nabla F(a(s),\tilde\Phi_s))} \, ds.
  \end{multline}
  Therefore, if $F(t,\varphi) = F_{s,t}(\varphi)$ is as in \eqref{e:F-semigroup} then by \eqref{e:F-semigroup2} the last line vanishes and
  hence
  \begin{equation}
    \E\qb{F(a(r),\tilde\Phi_r)} = F(a(0),\tilde\Phi_0).
  \end{equation}
  In the original parametrization with $\Phi_\infty =\tilde\Phi_0=0$ and $a(r)=\epsilon^2$,
  \begin{align}
    \E_{\Phi_\infty=0}\qb{F(\Phi_{\epsilon^2})} 
  &= \E_{\Phi_\infty=0}\qb{F_{\epsilon^2,\epsilon^2}(\Phi_{\epsilon^2})}\nnb
   &= F_{\epsilon^2,\infty}(\Phi_\infty) = e^{+V_\infty(0)} \E_{C_\infty-C_{\epsilon^2}}\qa{ e^{-V_{\epsilon^2}(\zeta)} F(\zeta)}.
  \end{align}
  Since the functionals $F(\Phi)=e^{i(\Phi,f)}$ with $f\in C_c^\infty(\R^2)$ characterize the distribution
  this completes the first part of the statement in which $\epsilon>0$.
  Taking $\epsilon\to 0$ as discussed at the beginning of the proof completes the proof.
\end{proof}

\subsection{Proof of Proposition~\ref{prop:finvol-coupling-bis}}

The main properties and consequences of the construction are summarized in
Proposition~\ref{prop:finvol-coupling-bis} (which is a restatement of Proposition~\ref{prop:finvol-coupling})
which we now prove.

\begin{proof}[Proof of Proposition~\ref{prop:finvol-coupling-bis}]
  The probability space is the one from Section~\ref{sec:coupling} on which the cylindrical Brownian motion is defined.
  The neighborhood $I$ and the function $\eta(\rho)$ are as in Proposition~\ref{prop:Vt} and Proposition~\ref{prop:SDE-existence}.
  We recall that $\Phi^\GFF$ denotes the decomposed free field from Section~\ref{sec:Gauss-decomp}
  and that $\Phi^\SG$ is the solution of the SDE \eqref{e:SDE} as in Proposition~\ref{prop:SDE-existence},
  and set
  \begin{equation}
    Z = \Phi^{\GFF}_{0,t_0}, \qquad \tilde \varphi =\Phi^{\Delta}_{0,t_0} + \Phi^{\SG}_{t_0}.
  \end{equation}
  Then, by definition,
  \begin{equation}
    \Phi^{\SG}_0
    = \Phi^{\GFF}_{0,t_0} + \Phi^{\Delta}_{0,\infty} + \Phi^{\GFF}_{t_0,\infty}
    = Z+ \tilde \varphi,
  \end{equation}
  and it follows from  Proposition~\ref{prop:SDE-Ito}  that $Z+\tilde \varphi$ is distributed according to $\nu^{\SG(\beta,z|\Lambda,m)}$ for $z\in \R$,
  proving item (i).
  Clearly,  $Z$ is a Gaussian log-correlated field
  and the Besov--H\"older regularity of $Z$ follows from Proposition~\ref{prop:GFF-Besov},
  proving item (ii).
  
  For item (iii), from \eqref{e:SDE} and the estimates of Proposition~\ref{prop:Vt} we have that, for $\epsilon \geq 0$ (with $0$ included),
  $0 < r < 2-\beta/4\pi$,
  \begin{align}
    \|\Phi^\Delta_{\epsilon^2}\|_{C^r_x}
    &\leq \int_{\epsilon^2}^\infty e^{-m^2 s} \|e^{\Delta_x s} \nabla_\varphi V_s(\Phi_s,z,x|\Lambda,m,\epsilon)\|_{C^r_x} \, ds
      \nnb
    &\leq C_r(\beta,z,\Lambda) \int_0^\infty e^{-m^2 s} (1+s^{-\beta/8\pi-r/2})\, ds \leq C_r(\beta,z,\Lambda,m) <\infty
  \end{align}
  where the integral is finite if $m>0$.
  In particular, for $|\Lambda|<\infty$ and $m>0$, it follows that $\Phi^\Delta_{0}$ is deterministically bounded in $C^{r}_x$.
  If the $L_x^\infty$ norm omitted in we obtain an $s^{-r/2}$ improvement for $s>1$ by  \eqref{e:nablaV-bd-finvol} and thus
  the following bound which is uniform in $m>0$:
  \begin{align}
      [\Phi^\Delta_{\epsilon^2}]_{C^r_x}
    &\leq C_r(\beta,z,\Lambda) \int_{\epsilon^2}^\infty [e^{\Delta_x s} \nabla_\varphi V_s(\Phi_s,z,x|\Lambda,m,\epsilon)]_{C^r_x} \, ds
        \nnb
      &\lesssim C_r(\beta,z,\Lambda) \qa{ \int_0^1 s^{-\beta/8\pi-r/2} \, ds + \int_1^\infty s^{-1-r/2} \, ds} \leq C_r'(\beta,z,\Lambda) < \infty,
  \end{align}
  which is again deterministic.   
  Since $\Phi^{\GFF}_{t_0}$ is a smooth Gaussian field, similar estimates hold locally for $\Phi^{\GFF}_{t_0}$ in the sense of moments by
  standard Gaussian estimates, see Lemma~\ref{le:gaussianest}. Thus
  \begin{equation}
    Z = \Phi^{\GFF(m)}_{0,t_0}, \qquad \tilde \varphi = \Phi^{\Delta(\beta,z|\Lambda,m)}_0 + \Phi_{t_0}^{\GFF(m)}
  \end{equation}
  gives the desired decomposition.
    
  For item (iv), the analyticity is a direct consequence of Proposition~\ref{prop:SDE-existence}.

  Finally, for item (v),
  it suffices to remark that all statements above also hold for the regularized version \eqref{e:SDE-eps}
  whose distribution is $\nu^{\GFF+\SG(\beta,z|\Lambda,m,\epsilon)}$.
\end{proof}

We remark that, since $\Phi^{\GFF}$ is defined with mass $m>0$,
it is insignificant that we defined $Z$ in terms of the field $\Phi^{\GFF}_{0,t_0}$ instead of
the full free field $\Phi^{\GFF}_{0,\infty}$.
This would be somewhat simpler since the moment bound on $\tilde\varphi$ then
becomes deterministic.
We used this division in preparation for the massless infinite-volume limit in Section~\ref{sec:coupling-infvol}
which can then be done using the same definitions.

\subsection{Discussion: analytic continuation and complex sine-Gordon measure}
\label{sec:regexp}

The results of this last subsection are not directly applied in the remainder of the paper, but we include it to clarify
how the analytic continuation works.

For all $z\in I$, the solution defined by Proposition~\ref{prop:SDE-existence} is a random field
on the probability space on which the cylindrical Wiener process in \eqref{e:GFF-Wiener} is defined.
For $z \not\in\R$, this random field takes complex values.
In this section,
we will only need the regularized version with  $\epsilon >0$ (even though the limiting version is well defined).
We denote the corresponding expectation by
\begin{equation} \label{e:tildeSG-def}
  \avg{F(\varphi)}_{\widetilde \SG(\beta,z|\Lambda,m,\epsilon)} = \E\qb{F(\Phi_{\epsilon^2}^{\SG(\beta,z|\Lambda,m,\epsilon)})},
\end{equation}
where $(\Phi^{\SG(\beta,z|\Lambda,m,\epsilon)}_t)_{t}$ is the process defined using the renormalized potential $V_t(\varphi |\Lambda,m,\epsilon)$.
In particular, the expectation $\avg{\cdot}_{\widetilde \SG(\beta,z|\Lambda,m,\epsilon)}$ is the expectation  of a  probability measure.

On the other hand, on can directly extend the definition \eqref{e:SGdef-bis} of the sine-Gordon measure to $z \not\in \R$  with sufficiently small imaginary part (depending on $\Lambda$)
--- but this definition is at best a complex measure on real-valued fields.
Indeed, if $\epsilon>0$ then
\begin{equation}
  V_{\epsilon^2}(\varphi,z|\Lambda,\epsilon) = z V_{\epsilon^2}(\varphi,1|\Lambda,\epsilon), \qquad V_{\epsilon^2}(\varphi,1|\Lambda,\epsilon)= 2\int_{\Lambda} \epsilon^{-\beta/4\pi}\cos(\sqrt{\beta} \varphi)\, dx
\end{equation}
is deterministically bounded.
Hence, for $z\in \C$ with sufficiently small imaginary part (depending on $\Lambda\subset \R^2$ with $|\Lambda|<\infty$ but not on $\epsilon>0$), the definition
\begin{equation} \label{e:SG-def-eps-ratio}
  \avg{F(\varphi)}_{\SG(\beta,z|\Lambda,m,\epsilon)} =
  \frac{\E\qB{e^{zV_{\epsilon^2}(\Phi^{\GFF(m)}_{\epsilon^2,\infty},1|\Lambda,\epsilon)} F(\Phi^{\GFF(m)}_{\epsilon^2,\infty})}}
  {\E\qB{e^{z V_{\epsilon^2}(\Phi^{\GFF(m)}_{\epsilon^2,\infty},1|\Lambda,\epsilon)}}}
\end{equation}
makes sense for all bounded functions $F$ (defined on real-valued fields)
as the expectation of a complex measure.
Indeed, the
numerator and the denominator (partition function)
are entire in $z\in \C$. The partition function is further nonzero for $z$
in a complex neighborhood $I$ containing $\R$ which is independent of $\epsilon>0$, see \cite[Theorem~5.1]{MR4767492},
so that the ratio is well defined.

By Proposition~\ref{prop:SDE-Ito}, the two expectations \eqref{e:tildeSG-def} and \eqref{e:SG-def-eps-ratio}
agree as probability measures when $z\in \R$, but for $z\not\in\R$ the first is a probability measure
while the second a complex measure (at least when $\epsilon>0$).
Nonetheless, when applied to analytic functionals $F$,
the next proposition shows that the two expectations also coincide when $z\not\in\R$.

\begin{proposition} \label{prop:SG-tildeSG-identification}
  Under the assumptions of Proposition~\ref{prop:SDE-existence},
  for all $\epsilon>0$, $m>0$, $z\in I$,
  and $F: B_\eta \to \C$ is bounded and analytic,
  \begin{equation}
    \avg{F}_{\widetilde\SG(\beta,z|\Lambda,m,\epsilon)}
    =
    \avg{F}_{\SG(\beta,z|\Lambda,m,\epsilon)},
  \end{equation}
  and both sides are analytic in $z\in I$.
\end{proposition}

\begin{proof}
  It suffices to show that the left-hand side is analytic in $z\in I$ for some complex neighborhood $I$ of $\R$.
  The right-hand side is analytic in $z \in I$ by the discussion preceding the statement of the proposition.
  Since, by Proposition~\ref{prop:SDE-Ito}, both sides are equal for $z\in \R$,
  the statement follows from uniqueness of the analytic continuation.

  For the analyticity of the left-hand side, note that,
for each realization of $(\Phi^\GFF_t)_t$,
  the assumption is such that $F(\Phi^\GFF_{\epsilon^2}+\Phi^\Delta_{\epsilon^2})$ is analytic in $\Phi^\Delta_{\epsilon^2} \in B_\eta \cap C^k(\R^2)$
  whereas $z \mapsto \Phi^{\Delta}_\epsilon = \Phi^{\Delta(\beta,z|\Lambda,m)}_\epsilon \in B_\eta \cap C^k(\R^2)$ is analytic in $z\in I$
  by Proposition~\ref{prop:SDE-existence}.
  It follows that
  \begin{equation} \label{e:F-extend-unique-pf}
    z\in I \mapsto \avg{F}_{\widetilde\SG(\beta,z|\Lambda,m,\epsilon)}  = \E[F(\Phi^{\GFF(m)}_{\epsilon^2}+\Phi^{\Delta(\beta,z|\Lambda,m)}_{\epsilon^2})]
  \end{equation}
  is analytic.
  Indeed, analyticity can be verified using Morera's theorem whose assumptions are verified using Fubini's theorem
  and the assumed boundedness of $F$ to exchange the expectation with the contour integral.
\end{proof}

\begin{remark}
  The above proposition identifies the expectation of analytic $F$ on complex-valued fields with integration against a complex-valued measure supported on real fields.
  The  following one-dimensional example illustrates this. %
  For a bounded function $F$ defined on an open horizontal strip containing the real axis, consider the complex random variable $\zeta+z$ (in place of the SDE)
  where $\zeta \sim \mathcal{N}(0,1/2)$ and $z$ lies in the same strip. Then
\begin{equation}
  \E[F(\zeta+z)]=\int_\R F(x+z)e^{-x^2}\frac{dx}{\sqrt{\pi}}=\int_{\R+z} F(x)e^{-(x-z)^2}\frac{dx}{\sqrt{\pi}}.
\end{equation}
In the case that $F$ is analytic, we can deform the contour $\R+z$ to $\R$, obtaining an integral of $F$ against a complex measure supported on $\R$.
\end{remark}

\section{Decomposition of massless sine-Gordon field -- Propositions~\ref{prop:coupling-infvol}--\ref{prop:coupling-infvol-convergence}}
\label{sec:coupling-infvol}

The goal of this section is to prove Propositions~\ref{prop:coupling-infvol}--\ref{prop:coupling-infvol-convergence} which we restate for convenience below.
We also recall the moment bounds assumption \eqref{e:SG-moments-bis2}, which is that there is a norm $\|\cdot\|$ on $\cS(\R^2)$ such that for any $g\in C_c^\infty(\R^2)$,
\begin{equation} \label{e:SG-moments-bis}
  \sup_{T\geq 1}
  \avg{|(\varphi,g-\hat g(0)\eta_T)|^p}_{\SG(\beta,z)} \leq C_p(\beta,z)\|g\|^p.
\end{equation}
For $\beta=4\pi$ and any $p\geq 1$, this assumption is verified in Section~\ref{sec:mixing},
see Corollary~\ref{cor:superexp-moments}.
For the statement of the proposition, also recall the definition of weighted function spaces from Section~\ref{sec:notation},
and the definition of the (infinite-volume) massless sine-Gordon measure from Section~\ref{sec:mixing}.

\begin{proposition}  \label{prop:coupling-infvol-bis}
  Let  $\beta\in (0,6\pi)$, $z\in \R$, and assume the moment bound
  \eqref{e:SG-moments-bis} holds with $p\geq 2$.
  Then there is a probability space and random variables $Z$ and $\tilde \varphi$ such that the following hold.

  \smallskip
  \noindent
  (i) The massless sine-Gordon field is decomposed as
  \begin{equation}
    Z+\tilde\varphi \sim \nu^{\SG(\beta,z)}.
  \end{equation}
  (ii) The random field $Z$ has the same properties as in Proposition~\ref{prop:finvol-coupling-bis}.
  
  \smallskip\noindent
  (iii)  For any $r<2-\beta/4\pi$ and weight of the form $\rho(x)=(1+|x|^2)^{-\sigma/2}$ for large enough $\sigma$,
  the field $\tilde\varphi$ is a random element in $C^r(\rho)$ satisfying
  \begin{equation}
    \E\qa{\|\tilde\varphi\|_{C^r(\rho)}^p} \leq C_{p,r}(\beta, z).
  \end{equation}
\end{proposition}

The previous proposition gives a decomposition of the massless infinite-volume sine-Gordon measure.
The next proposition gives strong convergence of the massive finite-volume model.

\begin{proposition} \label{prop:coupling-infvol-convergence-bis}
  Under the same assumptions as in Proposition~\ref{prop:coupling-infvol-bis}, one can find sequences $m_i \to 0$ and $\Lambda_j \to \R^2$ and $T_k \to \infty$
  and a probability space with random variables $Z$ and $\tilde \varphi_{i,j,k}$ and $\tilde \varphi$ such that the following hold.

  \smallskip\noindent
  (i)
  The distribution $(P_{T_k})_\#\nu^{\SG(\beta,z|\Lambda_j,m_i)}$  of $P_{T_k}\varphi$ where $\varphi \sim \nu^{\SG(\beta,z|\Lambda_j,m_i)}$ and $P_T$ is defined in \eqref{e:Peta-def}
  coincides with the distribution of $Z + \tilde \varphi_{i,j,k}$:
  \begin{equation}
     \varphi_{i,j,k}= Z + \tilde \varphi_{i,j,k} \sim (P_{T_k})_\#\nu^{\SG(\beta,z|\Lambda_j,m_i)}.
  \end{equation}
  \smallskip\noindent
  (ii) The random field $Z$ has the same properties as in Proposition~\ref{prop:finvol-coupling-bis}
  and is independent of $i,j,k$.

  \smallskip
  \noindent
  (iii) The random field  $\tilde \varphi_{i,j,k} $ is a random element in $C^r(\rho)$, and there are real-valued random variables $X_{i,j,k}$
  and a random element $\tilde\varphi$ in $C^r(\rho)$ such that
  \begin{equation}
    \lim_{k\to\infty}\lim_{j\to\infty}\lim_{i\to\infty}  \|\tilde \varphi_{i,j,k} - \tilde \varphi  + X_{i,j,k}\|_{C^r(\rho)} \to 0,
  \end{equation}
  almost surely,
  as well as
  \begin{equation}
    \sup_{i,j,k}\E\qa{\|\tilde\varphi_{i,j,k}\|_{C^r(\rho)}^p} \leq C_{p,r}(\beta, z).
  \end{equation}

  \smallskip\noindent
  (iv)
  In particular, the law of $\tilde\varphi$ agrees with that of $\nu^{\SG(\beta,z)}$ modulo constants.
\end{proposition}

The main idea compared to Section~\ref{sec:coupling-finvol} is to combine the moment bound
\eqref{e:SG-moments-bis}
with a priori estimates from the SDE \eqref{e:SDE} in order to obtain tightness bounds in suitable function spaces.

\begin{remark}
  Using that the constants in \eqref{e:SG-moments-bis} can be taken to be as in Corollary~\ref{cor:superexp-moments},
  the constant $C_{p,r}(\beta,z)$ in Theorem~\ref{prop:coupling-infvol-bis} can also be taken to be of the
  form
  \begin{equation}
    C_r(\beta,z)^p (p\log p)^{p/2}.
  \end{equation}
  In particular, exponential moments of all orders exist.
\end{remark}

\subsection{Uniform estimates on the renormalized potential}

The following proposition extends Proposition~\ref{prop:Vt} by estimates that are valid
uniformly as $\Lambda \to \R^2$ and $m \to 0$. The statements for $t\leq t_0$ are essentially contained
in the proof of Proposition~\ref{prop:Vt} which includes a convergent series expansion in this regime.
The extension to $t>t_0$ is non-perturbative and essentially a consequence of the maximum principle.

\begin{proposition} \label{prop:Vt-infvol}
  There exists $t_0=t_0(\beta,z)>0$ such that the following
  uniform in volume estimates hold for $\beta\in(0,6\pi)$ and $z\in \R$. %

  \smallskip\noindent (i)
  For $t\leq t_0$,
  uniformly in $\Lambda \subset \R^2$ bounded and $m>0$,
  and uniformly in $\varphi : \R^2 \to \R$,
  \begin{equation}  \label{e:nablaV-bd-smallt}
    \|(\sqrt{t}\nabla)^ke^{\Delta_x t}\nabla_\varphi V_t(\varphi,z,x|\Lambda,m)\|_{L^\infty_x} \leq C_k(\beta,z)t^{-\beta/8\pi},
  \end{equation}
  and for any weight $\rho(x)=(1+|x|^2)^{-\sigma/2}$ with $\sigma \geq 0$,
  \begin{multline}\label{e:nablaV-Lipschitz-smallt-rho}
    \|(\sqrt{t}\nabla)^ke^{\Delta_x t}\nabla_\varphi V_{t}(\varphi,z,x|\Lambda,m)-(\sqrt{t}\nabla)^k e^{\Delta_x t}\nabla_\varphi V_{t}(\varphi',z,x|\Lambda,m)\|_{L^\infty_x(\rho)}
    \\
    \leq C_{\sigma,k}(\beta,z) t^{-\beta/8\pi} \|\varphi-\varphi'\|_{L^\infty_x(\rho)}.
  \end{multline}

  \smallskip\noindent (ii)
  For $t \geq t_0$,
  the following estimates hold uniformly in $\Lambda \subset \R^2$ bounded and $m>0$:
  \begin{equation} \label{e:gradVt-maxprinciple}
    \|e^{\Delta_x t_0}\nabla_\varphi V_{t}(\varphi,z,x|\Lambda,m)\|_{L^\infty_xL^\infty_\varphi} \leq C(\beta,z).
  \end{equation}
  In particular, for any $k\geq 0$ and $t \geq 2t_0$,
  \begin{equation} \label{e:gradVt-maxprinciple-deriv}
    \|\nabla_x^ke^{\Delta_x t}\nabla_\varphi V_{t}(\varphi,z,x|\Lambda,m)\|_{L^\infty_xL^\infty_\varphi} \leq C_k(\beta,z) t^{-k/2}. %
  \end{equation}
\end{proposition}

For the proof of the proposition, we need slight improvements
of the estimates of Lemma~\ref{lem:tildeVn} stated in the following lemma, namely the inclusion of a growing weight.
The weighted estimates follow from exactly the same proof.
For completeness, in Appendix~\ref{app:fracrenorm}, we included the weight
in the more complicated setting of the renormalized potential with fractional charge insertions (which is however not needed here),
and the lemma is the special case $k=0$ of Proposition~\ref{pr:vtnorm}.

\begin{lemma} \label{lem:tildeVn-weight}
  Under the assumption of Lemma~\ref{lem:tildeVn} and with the same notation $h_t^n$,
  for $t\leq 1$, one also has the weighted estimate
  \begin{equation}
    \label{e:tildeVn-bd-rho}
    \sup_{\xi_1} \int d\xi_2\cdots d\xi_n\, h^n_t(\xi_1,\xi_2,\dots,\xi_n) e^{|x_1-x_2|} \leq n^{n-2} t^{-1} (C_\beta t^{1-\beta/8\pi})^n, \qquad (n\neq 2),
  \end{equation}
  and analogously for the charged part $\tilde V^2(\xi_1,\xi_2)\1_{\sigma_1+\sigma_2\neq 0}$ of $\tilde V^2$.
  The neutral part of $\tilde V_t^2$ satisfies
  \begin{equation}
        \label{e:tildeV2-bd-rho}
    \sup_{\xi_1}\int d\xi_2 \, \frac{|x_1-x_2|}{\sqrt{t}}h_t^2(\xi_1,\xi_2)\1_{\sigma_1+ \sigma_2=0} e^{|x_1-x_2|}
    \leq  t^{-1} (C_\beta t^{1-\beta/8\pi})^2.  
  \end{equation}
\end{lemma}

\begin{proof}[Proof of Proposition~\ref{prop:Vt-infvol}]
Let us first comment on the choice of $t_0(\beta,z)$.
As indicated above, the reasoning will be to control the renormalized potential by the series expansion for $t\leq t_0(\beta,z)$.
The bounds for the series expansion come from Lemma~\ref{lem:tildeVn-weight}, 
and as seen in the proof of Proposition~\ref{prop:Vt} (note that we have $\eta=0$ now), the series can be controlled as long as
\begin{equation}
|z|C_\beta t_0(\beta,z)^{1-\frac{\beta}{8\pi}}<1,
\end{equation}
and $t_0(\beta,z) \leq 1$ to apply  Lemma~\ref{lem:tildeVn-weight}.
Having fixed $t_0(\beta,z)$, let us turn to the actual statements.

\subproof{item (i)}
For $t \leq t_0$ the proof of the $L_x^\infty$ bound \eqref{e:nablaV-bd-smallt} follows from the series expansion as in the proof of Proposition~\ref{prop:Vt}.
The Lipschitz bound \eqref{e:nablaV-Lipschitz-smallt-rho} in the weighted norms requires the improved integrability of $\tilde V_t$ stated in \eqref{e:tildeVn-bd-rho} and \eqref{e:tildeV2-bd-rho}.
Indeed, by \eqref{e:nablaV-def},
\begin{multline} \label{e:nablaV-def-section3}
  \nabla_\varphi V_t(\varphi,z,x)-  \nabla_\varphi V_t(\varphi',z,x)
  = \sum_{n=1}^\infty \frac{z^n}{(n-1)!} \sum_{\sigma_1} \int_{(\Lambda \times \{\pm 1\})^{n-1}}d\xi_2 \cdots d\xi_n\, \\
  \tilde V_t^n((\sigma_1,x),\xi_2,\dots,\xi_n) i\sqrt{\beta}\sigma_1 (e^{i\sqrt{\beta}\sum \sigma_j \varphi(x_j)}-e^{i\sqrt{\beta}\sum \sigma_j \varphi'(x_j)}),
\end{multline}
where we are omitting the $\Lambda,m$ from the notation.
Using that
\begin{align}
  |e^{i\sqrt{\beta}\sum \sigma_j \varphi(x_j)}-e^{i\sqrt{\beta}\sum \sigma_j \varphi'(x_j)}|
  &\leq \sqrt{\beta}\sum_j |\varphi(x_j)-\varphi'(x_j)|
    \nnb
  & \leq \sqrt{\beta}\sum_j \rho(x_j)^{-1} \|\varphi-\varphi'\|_{L^\infty_x(\rho)}
\end{align}
and $\rho(x)/\rho(y) \leq C\rho(x-y)^{-1}$ so that
\begin{equation}
  \rho(x)\rho(x_j)^{-1}
  \lesssim \rho(x_1) \rho(x-x_1)^{-1} \rho(x_j)^{-1}
  \lesssim \rho(x-x_1)^{-1}\rho(x_j-x_1)^{-1}
\end{equation}
we obtain
\begin{align}
  &\rho(x) \int d\xi_1 \cdots d\xi_n\, e^{\Delta_x t}(x,x_1) |\tilde V_t^n(\xi_1,\dots,\xi_n)| \rho(x_j)^{-1}
    \nnb
  &  \lesssim  \int d\xi_1 \cdots d\xi_n \, e^{\Delta_x t}(x,x_1)\rho(x-x_1)^{-1} |\tilde V_t^n(\xi_1,\dots,\xi_n)| \rho(x_j-x_1)^{-1}
    \nnb
  &\lesssim n^{n-2} t^{-1}(C t^{1-\beta/8\pi})^n.
\end{align}
The same estimate holds with $e^{\Delta_x t}(x,x_1)$ replaced by $(\sqrt{t}\nabla_x)^k e^{\Delta_x t}(x,x_1)$.
The $n\neq 2$ terms multiplied by the weight $\rho(x)$ are therefore bounded by
\begin{multline}
  \rho(x)
  \int_{(\Lambda \times \{\pm 1\})^{n}}d\xi_1 \cdots d\xi_n\, (\sqrt{t}\nabla_x)^k e^{\Delta_x t}(x,x_1)\tilde V_t^n(\xi_1,\xi_2,\dots,\xi_n) i\sqrt{\beta}\sigma_1 \sum_{j=1}^n\rho(x_j)^{-1} \|\varphi-\varphi'\|_{L^\infty(\rho)}
  \\\
  \lesssim n^{n-1} t^{-1} (Ct^{1-\beta/8\pi})^n \|\varphi-\varphi'\|_{L^\infty(\rho)}
  .
\end{multline}
The explicit neutral $n=2$ term is similar. Indeed, bounding the weight as for $n\neq2$
and the resulting integral as in \eqref{e:nablaVbd}--\eqref{e:nablaV2bd}, we obtain
\begin{align}
  &\rho(x)\int_{\Lambda^2}dx_1\, dx_2\, |\tilde V_t^2((1,x_1),(-1,x_2))||e^{\Delta_x t}(x,x_1)-e^{\Delta_x t}(x,x_2)|
    \rho(x_1)^{-1}
  \lesssim t^{1-\frac{\beta}{4\pi}},
\end{align}
which is the needed estimate for $k=0$, and again the same estimate holds after replacing $e^{\Delta_x t}(x,\cdot)$ by $(\sqrt{t}\nabla_x)^k e^{\Delta_x t}(x,\cdot)$.

\subproof{item (ii)}
That the uniform bound \eqref{e:gradVt-maxprinciple} continues to hold for $t\geq t_0(\beta,z)$
is a consequence of the maximum principle (explained for example in \cite[Lemma~3.11]{MR4798104}).
Explicitly,
\begin{equation}
  e^{\Delta_x t_0}\nabla V_t(\varphi,z) =
  \frac{\E_{C_t-C_0}\qB{e^{\Delta_x {t_0}}\nabla V_{t_0}(\varphi+\zeta,z) e^{-V_{t_0}(\varphi,z)}}}
  {\E_{C_t-C_0}\qB{e^{-V_{t_0}(\varphi,z)}}}.
\end{equation}
Thus
\begin{equation}
  \|e^{\Delta_x t_0}\nabla V_t(\varphi,z)\|_{L^\infty_x L^\infty_\varphi}
  \leq C(\beta,z) t_0^{-\beta/8\pi} = C'(\beta,z).
\end{equation}
The derivative bounds \eqref{e:gradVt-maxprinciple-deriv} then follow from standard estimates of
derivatives of the heat kernel:
\begin{equation}
  \|\nabla^k e^{\Delta_x (t-t_0)} f\|_{L_x^\infty} \lesssim (t-t_0)^{-k/2} \|f\|_{L_x^\infty}
\end{equation}
which is a direct consequence of the formula $e^{\Delta_x t}f(x) = \int \frac{e^{-|x-y|^2/4t}}{4\pi t} f(y) \, dy$ on $\R^2$.
\end{proof}

\subsection{Decomposition in infinite volume -- proof of Proposition~\ref{prop:coupling-infvol-bis}}

By Propositions~\ref{prop:SDE-existence}--\ref{prop:SDE-Ito}, the sine-Gordon field with large scale regularization $\Lambda \subset \R^2$
bounded and $m>0$ can be
constructed as the solution $\Phi^{\SG(\beta,z|\Lambda,m)}_0$ to the SDE \eqref{e:SDE}:
\begin{align} \label{e:SDE-bis}
  \Phi_t^{\SG(\beta,z|\Lambda,m)}
  &= -\int_t^\infty e^{-m^2 s} e^{\Delta_x s} \nabla_\varphi V_s(\Phi_s^{\SG(\beta,z|\Lambda,m)},z,x|\Lambda,m)\, ds + \Phi^{\GFF(m)}_t
  \nnb
  &= \Phi_t^{\Delta(\beta,z|\Lambda,m)} + \Phi_t^{\GFF(m)}.
\end{align}
Choosing $t_0>0$ be as in Proposition~\ref{prop:Vt-infvol}, we decompose \eqref{e:SDE-bis} according to
\begin{equation}
  \Phi_0^{\SG} %
  = \Phi^\Delta_{0,t_0} + \Phi^{\GFF}_{0,t_0} +\Phi^{\SG}_{t_0}, \label{e:PhiSG0-decomp}
\end{equation}
where
\begin{equation}
  \Phi^{\Delta}_{0,t_0} = \Phi_0^{\Delta}-\Phi^{\Delta}_{t_0},
  \qquad
  \Phi^{\GFF}_{0,t_0} = \Phi_0^{\GFF}-\Phi^{\GFF}_{t_0}.
\end{equation}
Our goal now is to pass to the massless infinite-volume limit in this decomposition.

For the proof of Proposition~\ref{prop:coupling-infvol-bis},
we consider the three terms on the right-hand side of \eqref{e:PhiSG0-decomp} separately and show that they are tight in suitable weighted spaces,
defined in Section~\ref{sec:notation}.
The embeddings $C^{r'}(\rho) \subset C^{r}(\rho')$ are compact when $r<r'$ and
$\rho'(x)/\rho(x) \to 0$ as $|x|\to\infty$, see for example \cite[Proposition~A.7]{2504.08606}
for the case $-1<r'<0$ and the Arzelà--Ascoli theorem for $r'>0$.  %

\begin{lemma} \label{lem:tight}
  Let $r\in (0,1)$, $s>0$, and $\rho(x)=(1+|x|^2)^{-\sigma/2}$ for suitable $\sigma>0$. Then
  assuming $p\geq 1$ is such that \eqref{e:SG-moments-bis} holds, we have
  \begin{equation} \label{Phit0-nocenter-tightCk-bis}
  \E \qbb{
    \|P_T\Phi^{\SG(\beta,z|\Lambda,m)}_{t_0} \|_{C^r(\rho)}^p 
    + \|P_T\Phi^{\Delta(\beta,z|\Lambda,m)}_{0,t_0}\|_{C^{r}(\R^2)}^p +
    \|P_T\Phi^{\GFF(m)}_{0,t_0}\|_{C^{-s}(\rho)}^p
  } \leq C_{r,s,p}(\beta,z),
\end{equation}
uniformly in $T \geq 1$, $\Lambda\subset\R^2$ bounded, and $m\in(0,1]$.
In particular,
  \begin{equation} \label{e:tight}
    \pa{P_T\Phi^{\SG(\beta,z|\Lambda,m)}_{t_0}, P_T\Phi^{\Delta(\beta,z|\Lambda,m)}_{0,t_0}, P_T\Phi^{\GFF(m)}_{0,t_0}} \in C^{r}(\rho) \times C^{r}(\rho)\times C^{-s}(\rho)
  \end{equation}
  where $\Lambda\subset \R^2$ bounded, $m\in (0,1]$, and $T \in [1,\infty)$ is a tight family of random variables.
\end{lemma}

The main estimate is proved in Section~\ref{sec:infvol-tightness}; the remainder of the proof
is completed as follows.

\begin{proof}
By the compactness of the embeddings $C^{r}(\rho) \subset C^{r'}(\rho')$ if $r'<r$ and
$\rho'/\rho \to 0$, it suffices to show \eqref{Phit0-nocenter-tightCk-bis}.
Tightness then follows after decreasing the exponents and weights slightly.
  
The bound on the first term in  \eqref{Phit0-nocenter-tightCk-bis} is shown in Section~\ref{sec:infvol-tightness} below,
and the bound on the third term is a standard Gaussian estimate for which details are included in Appendix~\ref{sec:coupling-infvol-Gauss}.

For the second term, it is immediate from
the representation of $\Phi^\Delta$ in \eqref{e:SDE-bis} and
the uniform estimates on
the renormalized potential from Proposition~\ref{prop:Vt-infvol} that
\begin{equation}
  \|\Phi^{\Delta(\beta,z|\Lambda,m)}_{0,t_0}\|_{C^{r}(\rho)}
  \leq
  \|\Phi^{\Delta(\beta,z|\Lambda,m)}_{0,t_0}\|_{C^{r}(\R^2)}
  \lesssim 1,
\end{equation}
where we have interpolated the derivative estimates to H\"older estimates as in \eqref{e:Vt-Holder}.
In particular, from the second inequality in the line above,
\begin{equation}
  (\eta_T,\Phi_{0,t_0}^{\Delta(\beta,z|\Lambda,m)}) \leq \|\eta_T\|_{L^1}\|\Phi_{0,t_0}^{\Delta(\beta,z|\Lambda,m)}\|_{L^\infty} \lesssim 1,
\end{equation}
so that also $\|P_T\Phi^{\Delta(\beta,z|\Lambda,m)}_{0,t_0}\|_{C^{r}(\rho)} \leq \|P_T\Phi^{\Delta(\beta,z|\Lambda,m)}_{0,t_0}\|_{C^{r}(\R^2)}  \lesssim 1$.
These bounds on $\Phi^\Delta$ are actually deterministic.
\end{proof}

The above tightness lemma implies Proposition~\ref{prop:coupling-infvol-bis} as follows.

\begin{proof}[Proof of Proposition~\ref{prop:coupling-infvol-bis}]
  By Lemma~\ref{lem:tight}, we can find a sequences $m_i\to 0$, $\Lambda_j \to\R^2$, $T_k\to\infty$
  such that the triple in \eqref{e:tight} converges in distribution as first $m_i\to 0$ then $\Lambda_j \to \R^2$ and finally $T_k \to \infty$.
  Denote a random variable on the space $C^{r}(\rho)\times C^{r}(\rho)\times C^{-s}(\rho)$ with the corresponding limiting distribution by
  \begin{equation}
    ( \Phi^{\SG(\beta,z)}_{t_0}, \Phi^{\Delta(\beta,z)}_{0,t_0},\Phi^{\GFF(0)}_{t,t_0}).
  \end{equation}
By construction and the results of Section~\ref{sec:coupling-finvol},
\begin{equation}
  \Phi^{\SG(\beta,z|\Lambda,m)}_0 = \Phi^{\Delta(\beta,z|\Lambda,m)}_{0,t_0} + \Phi^{\GFF(m)}_{0,t_0} + \Phi_{t_0}^{\SG(\beta,z|\Lambda,m)} \sim \nu^{\SG(\beta,z|\Lambda,m)}.
\end{equation}
By the  construction of $\nu^{\SG(\beta,z)}$, see Section~\ref{sec:mixing},
in distribution as $m\to 0$  then $\Lambda \to \R^2$  then $T \to \infty$,
\begin{equation}
  P_T \Phi^{\SG(\beta,z|\Lambda,m)}_0 \to \nu^{\SG(\beta,z)}.
\end{equation}
On the other hand, by the above, if $(T,\Lambda,m)=(T_k,\Lambda_j,m_i)$ and $i\to\infty$ then $j\to\infty$ then $k\to\infty$,
then in distribution (by the mapping theorem for weak convergence \cite[Theorem 2.7]{MR1700749}):
\begin{equation}
  P_T \Phi^{\SG(\beta,z|\Lambda,m)}_0 \to \Phi^{\Delta(\beta,z)}_{0,t_0} + \Phi^{\GFF(0)}_{0,t_0} + \Phi_{t_0}^{\SG(\beta,z)}.
\end{equation}
The right-hand side provides the desired decomposition of $\varphi \sim \nu^{\SG(\beta,z)}$ into $\varphi = Z+\tilde\varphi$ where
\begin{equation}
  Z = \Phi^{\GFF(0)}_{0,t_0},
\end{equation}
which is essentially a massive GFF,
and the H\"older continuous field
\begin{equation}
  \tilde\varphi = \Phi^{\Delta(\beta,z)}_{0,t_0}+\Phi_{t_0}^{\SG(\beta,z)} \in C^{r}(\rho).
\end{equation}
This completes the proof of items (i) and (ii).
The moment bound of item (iii) holds because $\Phi^{\Delta}$ is deterministically bounded
and the moments of $\Phi^{\SG(\beta,z)}_{t_0}$ are bounded by Lemma~\ref{lem:tight}.
\end{proof}

\subsection{Decomposition and convergence in infinite volume -- proof of Proposition~\ref{prop:coupling-infvol-convergence-bis}}

By an application of the Skorokhod representation theorem \cite[Theorem~6.7]{MR1700749},
the weak convergence (along a subsequence)
implied by the tightness of Lemma~\ref{lem:tight} 
also implies that there is some probability space on which random variables with the same distributions as 
\eqref{e:tight} are defined such that the convergence is pointwise.
Unfortunately, this only provides a coupling in which
\begin{equation}
  \|\Phi^{\GFF(m)}_{0,t_0}-\Phi^{\GFF(0)}_{0,t_0}\|_{C^{-s}(\rho)} \to 0,
\end{equation}
while on the original probability space this convergence actually takes place in the space $C^{r}(\rho)$,
for suitable positive $r>0$.
Sufficiently high regularity is important to understand the convergence of the corresponding IMC.
To prove Proposition~\ref{prop:coupling-infvol-convergence},
we will therefore construct a coupling with improved regularity for the convergence of the difference.

To construct the desired probability space, 
we again start from the Skorokhod  representation theorem,
but now applied only to the tight family of random variables
$P_{T}\Phi_{t_0}^{\SG(\beta,z|\Lambda,m)} \in C^k(\rho)$.
The Skorokhod representation theorem implies that these random variables can be simultaneously defined on some probability space such that, almost surely,
\begin{equation} \label{e:infvol-prob-space1}
  P_{T}\Phi_{t_0}^{\SG(\beta,z|\Lambda,m)} \to \Phi^{\SG(\beta,z)}_{t_0} \quad \text{in $C^k(\rho)$,}
\end{equation}
where here (and implicitly from now on) the limit means that $(T,\Lambda,m) = (T_k,\Lambda_j,m_i)$ with $i\to \infty$ then $j\to \infty$ and then $k\to\infty$.
We extend the probability space by an independent cylindrical Brownian motion $(W_t)$ on $L^2(\R^2)$, as in Section~\ref{sec:coupling},
and also define, for $t<t_0$,
\begin{equation} \label{e:infvol-prob-space2}
  \Phi^{\GFF(m)}_{t,t_0} = \int_t^{t_0} e^{-\frac12 m^2 u}e^{\frac12 \Delta u} \, dW_u.
\end{equation}
As stated in Lemma~\ref{lem:GaussXrho-bis} below and proved in Appendix~\ref{sec:coupling-infvol-Gauss} by standard arguments,
these processes take values in the space $X^{-s}(\rho) \subset C((0,t_0], C^0(\rho))$, where for $s>0$, we set
\begin{equation}
  \norm{\chi}_{X^{-s}(\rho)} = \sup_{t\in (0,t_0]} t^{s} \|\chi_t\|_{C^0(\rho)},
\end{equation}
and define $X^{-s}(\rho)$ as the closure under this norm of the compactly supported smooth functions on $(t,x)\in (0,t_0] \times \R^2$.
This definition ensures $X^{-s}(\rho)$ is a separable Banach space. For $\delta>0$, we also define a H\"older version with norm
\begin{equation}
  \norm{\chi}_{X^{-s,\delta}(\rho)} = \|\chi\|_{X^{-s}(\rho)} + \sup_{t \in (0,t_0]} t^{s} \sup_{1/2 <r<1} \frac{\|\chi_{rt}-\chi_{t}\|_{C^{\delta}(\rho)}}{(\log r)^\delta}
  .
\end{equation}
The Arzelà--Ascoli theorem implies that the embedding $X^{-s',\delta'}(\rho') \subset X^{-s}(\rho)$
is compact if $s'<s$, $\rho/\rho' \to 0$, and $\delta>0$, see Appendix~\ref{sec:coupling-infvol-Gauss}.
The next lemma (which is also proved in Appendix~\ref{sec:coupling-infvol-Gauss})
shows that \eqref{e:infvol-prob-space2} indeed takes values in this space and converges as $m\to 0$ in a suitable sense.

\begin{lemma} \label{lem:GaussXrho-bis}
  For any $s>0$, $\delta>0$ small, and $\rho(x)=(1+|x|^2)^{-\sigma/2}$ with $\sigma>0$, almost surely,
  \begin{equation}
    \sup_{m \geq 0}\norm{\Phi^{\GFF(m)}}_{X^{-s}(\rho)}\leq \sup_{m\geq 0}     \norm{\Phi^{\GFF(m)}}_{X^{-s,\delta}(\rho)} < \infty,
  \end{equation}
  and
  \begin{equation}
    \lim_{m\to 0}\norm{\Phi^{\GFF(0)}-  \Phi^{\GFF(m)}}_{X^{-s}(\rho)}= 0.
  \end{equation}
  Moreover, for any $r<1$, with convergence in $L^p$ and (and thus almost surely along a subsequence),
  \begin{equation} \label{e:Phi0m-conv}
    \lim_{m\to 0} \norm{\Phi^{\GFF(0)}_{0,t_0}-\Phi^{\GFF(m)}_{0,t_0}}_{C^{r}(\rho)} = 0.
  \end{equation}
\end{lemma}

The final ingredient that we need is that,
given $\Phi_{t_0}^\SG$ and the decomposed free field $(\Phi^{\GFF}_{t,t_0})_{t\in (0,t_0]}$ above,
we can solve the SDE \eqref{e:SDE} from $t_0$ to $t$ and that the solution is uniformly continuous in $(\Phi^\GFF_{t,t_0})$ in $X^{-s}(\rho)$.
To simplify notation, for $t\in (0,t_0]$ and $\psi \in C^0(\rho)$, define
\begin{equation} \label{e:B-Phi}
  B_t(\psi) = e^{-m^2 t} e^{\Delta_x t} \nabla_\varphi V_t(\psi + \Phi^{\SG(\beta,z|\Lambda,m)}_{t_0}|\Lambda,m).
\end{equation}
Since $\Phi^{\SG}_{t_0}$ is random, also $B$ is random (and it also depends on the parameters $\Lambda$ and $m$),
but  Proposition~\ref{prop:Vt-infvol} guarantees the following  deterministic estimates uniformly in $\psi$:
\begin{equation} \label{e:B-nablaB}
  \|(\sqrt{t}\nabla_x)^kB_t(\psi)\|_{L_x^\infty} \leq C t^{-\beta/8\pi}
\end{equation}
and
\begin{equation} \label{e:B-Lip}
  \|(\sqrt{t}\nabla_x)^k [B_t(\psi)-B_t(\psi')]\|_{L_x^\infty(\rho)} \leq C t^{-\beta/8\pi}\|\psi-\psi'\|_{L_x^\infty(\rho)}.
\end{equation}
In particular, interpolating the derivative estimate \eqref{e:B-nablaB} with $k=0$ and $k=1$,
it follows that for any $r \in (0,1)$,
\begin{equation} \label{e:B-Cr}
  \|B_t(\psi)\|_{C^r_x} \leq C t^{-\beta/8\pi-r/2},
\end{equation}
and analogously for the Lipschitz in $\psi$ estimate.

\begin{lemma} \label{lem:H0}
  Assume that $\beta/8\pi+r/2+s <1$ with $r>0 $ and $s>0$.
  Given any $\varphi^G = (\varphi_{t,t_0}^G) \in X^{-s}(\rho)$ there is a unique
  $\varphi^\Delta =  (\varphi^\Delta_{t,t_0}) \in C([0,t_0], C^{r}(\rho))$ solving
  \begin{equation} \label{e:H0-ODE}
    \varphi^{\Delta}_{t,t_0} = -\int_t^{t_0} B_u(\varphi^G_{u,t_0}+\varphi^\Delta_{u,t_0}) \, du, \qquad \varphi^\Delta_{t_0,t_0} = 0.
  \end{equation}
  The solution map $H_0: X^{-s}(\rho) \to C^{r}(\rho)$  defined such that $\varphi^\Delta_{0,t_0} = H_0(\varphi^G)$ satisfies
  \begin{equation} \label{e:H0-bounds}
    \|H_0(\varphi^G)\|_{C^r(\R^2)} \leq C,
    \qquad
    \|H_0(\varphi^G)-H_0(\tilde\varphi^G)\|_{C^r(\rho)} \leq C \norm{\varphi^G-\tilde\varphi^G}_{X^{-s}(\rho)},
  \end{equation}
  where the constant $C$ is independent of $\Lambda \subset \R^2$ and $m>0$.
\end{lemma}

\begin{proof}
  Given any $\varphi^G \in C((0,t_0],C^0(\R^2))$ which is   not required to be bounded as $t\to 0$,
  a fixed point argument analogous to that in the proof of  Proposition~\ref{prop:SDE-existence} implies 
  there exists a uniformly bounded solution $\varphi^\Delta \in C([0,t_0],C^0(\R^2))$.
  Indeed, one can use the same argument with weight
  \begin{equation}
     w(t) = \exp\qa{\int_t^{t_0} 2Cs^{-\beta/8\pi}\, ds} \lesssim 1.
  \end{equation}
  Uniqueness also follows from the estimates \eqref{e:B-Lip}:
  if $\varphi^\Delta$ and $\tilde\varphi^\Delta$ are both solutions to \eqref{e:H0-ODE} with the same $\varphi^G$, then \eqref{e:B-Lip} implies
  \begin{equation}
    \|\varphi^\Delta_{t,t_0} -\tilde\varphi^\Delta_{t,t_0}\|_{L^\infty_x} \lesssim \int_t^{t_0} u^{-\beta/8\pi} \|\varphi^\Delta_{u,t_0} -\tilde\varphi^\Delta_{u,t_0}\|_{L_x^\infty} \, du.
  \end{equation}
  Using that $\|\varphi^\Delta_{t_0,t_0} - \tilde \varphi^\Delta_{t_0,t_0}\|_{L^\infty_x}= 0$,
  a Gronwall argument implies that $\varphi^\Delta_{t,t_0} = \tilde\varphi^\Delta_{t,t_0}$ for all $t\in [0,t_0]$.
  Denote the solution by  $\varphi^\Delta_{t,t_0} = H_t(\varphi^G)$. For any $r>0$ such that $\beta/8\pi+r/2<1$,
  by \eqref{e:B-Cr},
  \begin{equation}
    \|H_t(\varphi^G)\|_{C^r(\R^2)} \lesssim \int_t^{t_0} u^{-\beta/8\pi-r/2} \, du \lesssim 1.
  \end{equation}
  In particular, the first estimate in \eqref{e:H0-bounds} holds.
  Similarly, if $\beta/8\pi +r/2+s < 1$,
  \begin{align}
    &\|H_{t}(\varphi^G)- H_{t}(\tilde\varphi^G)\|_{C_x^r(\rho)}
      \nnb
      &\lesssim
      \int_t^{t_0} u^{-\beta/8\pi-r/2} \|\varphi^G_u-\tilde\varphi^G_u\|_{L_x^\infty(\rho)} \, du
        +
    \int_t^{t_0} u^{-\beta/8\pi-r/2}  \|H_{u}(\varphi^G)- H_{u}(\tilde\varphi^G)\|_{L_x^\infty(\rho)} \, du
      \nnb
    &\lesssim
      \|\varphi^G-\tilde\varphi^G\|_{X^{-s}(\rho)}
      +
    \int_t^{t_0} u^{-\beta/8\pi-r/2}  \|H_{u}(\varphi^G)- H_{u}(\tilde\varphi^G)\|_{L_x^\infty(\rho)} \, du
      .
  \end{align}
  Therefore, by another application of Gronwall's inequality,
  \begin{align}
    \|H_{t}(\varphi^G)- H_{t}(\tilde\varphi^G)\|_{C_x^r(\rho)}
    &\lesssim
    \|\varphi^G-\tilde\varphi^G\|_{X^{-s}(\rho)} \exp\qa{C\int_t^{t_0} u^{-\beta/8\pi-r/2} \, du}
    \nnb
    &\lesssim \|\varphi^G-\tilde\varphi^G\|_{X^{-s}(\rho)} .
  \end{align}
  This concludes the proof.
\end{proof}

\begin{proof}[Proof of Proposition~\ref{prop:coupling-infvol-convergence-bis}]
  On the probability space defined around
  \eqref{e:infvol-prob-space1}--\eqref{e:infvol-prob-space2}
  above, we define the solution maps $H_0(\cdot |\Lambda,m) \in C(X^{-s}(\rho),C^r(\rho))$
  in terms of $\Phi^{\SG(\beta,z|\Lambda,m)}_{t_0}$ as in Lemma~\ref{lem:H0} and set,  for bounded $\Lambda\subset \R^2$, $m>0$,
  \begin{equation} \label{e:PhiDelta-H0-finvol}
    \Phi^{\Delta(\beta,z|\Lambda,m)}_{0,t_0} = H_0(\Phi^{\GFF(m)}|\Lambda,m).
  \end{equation}
  Thus $\Phi^{\Delta(\beta,z|\Lambda,m)}_{0,t_0}$ is the solution to the SDE \eqref{e:SDE} (or more precisely the version for the difference field),
  and the results of Section~\ref{sec:coupling-finvol} imply 
  \begin{equation}
     \Phi^{\Delta(\beta,z|\Lambda,m)}_{0,t_0} + \Phi^{\GFF(m)}_{0,t_0} + \Phi^{\SG(\beta,z|\Lambda,m)}_{t_0} \sim \nu^{\SG(\beta,z|\Lambda,m)}.
  \end{equation}

  Given deterministic $\Lambda\subset \R^2$ and $m>0$,
  we would like to view $H_0(\cdot|\Lambda,m)$ as random variable in the space $C(X^{-s}(\rho), C^r(\rho))$
  with the topology of compact convergence. 
  Unfortunately, since $X^{-s}(\rho)$ is not compact, the space $C(X^{-s}(\rho), C^r(\rho))$ does not need to be separable
  (which is important for standard properties of Banach space valued random variables, see \cite[Chapter~1]{MR3236753}).
  However, by \cite[Theorem~2.4.3]{MR1619545}, the space $C(E,F)$ is a Polish space if $E$ is a compact metric space and $F$ is a Polish space
  (where a Polish space is a complete separable metric space).
  We will therefore replace $X^{-s}(\rho)$ by a sequence of compact subsets $K_N$ which  the decomposed free field $\Phi^{\GFF(m)}$ takes values in.
  Let $s'<s$, $r'>r$, $\delta'>0$, and $\rho/\rho' \to 0$ and define,
  for a constant $C>0$ fixed large in terms of the implicit constants of Lemma~\ref{lem:H0}, and all $N\in \N$, 
  \begin{align}
    D &= \overline{\{ \psi \in C^{r}(\rho) : \|\psi\|_{C^{r'}(\rho')} \leq C \}},
    \\
    K_N &= \overline{\{ \varphi^{G} \in X^{-s}(\rho): \|\varphi^G\|_{X^{-{s',\delta'}}(\rho')} \leq N\} }.
  \end{align}
  By the Arzelà--Ascoli theorem (see Section~\ref{sec:coupling-infvol-Gauss} for the formulation in terms of the $X^{-s}$ spaces),
  the spaces $D$ and $K_N$ are compact for any $N<\infty$
  and the solution maps of the last lemma are in $C(K_N,D)$.
  Since the space $C(K_N,D)$ is a separable Banach space,
  we can now view $H_0(\cdot|\Lambda,m)$ as a Banach space valued random variable in $C(K_N,D)$ for any $N$.
  
  Lemma~\ref{lem:H0} further shows that the collection $(H_0(\cdot|\Lambda,m)) \subset C(K_N,D)$ %
  is equicontinuous and that,
  for each $\varphi^G$, the images $H_0(\varphi^G|\Lambda,m)$ take values in the compact space $D$. %
  By the Arzelà--Ascoli theorem in the version \cite[(7.5.7)]{MR349288} now applied to the former family of maps in $C(K_N,D)$
  there are subsequences such that,
  uniformly on $\varphi^G \in K_N$,
  \begin{equation}
    \lim_{j\to\infty}\lim_{i \to 0}H_0(\varphi^{G}|\Lambda_j,m_i) \to H_0(\varphi^G).
  \end{equation}
  In terms of the limiting $H_0$, define the $C^r(\rho)$-valued random variables
  \begin{equation}
    \Phi^{\Delta(\beta,z)}_{0,t_0} = H_0(\Phi^{\GFF(0)}).
  \end{equation}
  Together with \eqref{e:PhiDelta-H0-finvol}, it then follows that, almost surely in $C^r(\rho)$,
  \begin{equation}
    \Phi^{\Delta(\beta,z|\Lambda,m)}_{0,t_0}
    \to \Phi^{\Delta(\beta,z)}_{0,t_0}.
  \end{equation}
  Indeed, since $\Phi^{\GFF(m)} \to \Phi^{\GFF(0)}$ in $X^{-s'}(\rho') = \cup_N K_N$ and $H_0$ converges in $C(K_N,D)$ for every $N$
  and $(H_0(\cdot|\Lambda,m))$ is uniformly continuous in $C(K_N,D)$,
  \begin{align}
    \Phi^{\Delta(\beta,z|\Lambda,m)}_{0,t_0}
    &=
      H_0(\Phi^{\GFF(m)} |\Lambda,m)
      \nnb
    &=
    H_0(\Phi^{\GFF(0)} |\Lambda,m)
      + \qB{ H_0(\Phi^{\GFF(m)}|\Lambda,m)-H_0(\Phi^{\GFF(0)}|\Lambda,m) }
      \nnb
      &\to H_0(\Phi^{\GFF(0)}) =     \Phi^{\Delta(\beta,z)}_{0,t_0}.
  \end{align}
  In summary, together with \eqref{e:infvol-prob-space1} for the first convergence below
  and \eqref{e:Phi0m-conv} for the second convergence below,
  we have obtained a probability space on which, along suitable sequence $m_i\to 0$ and then $\Lambda_j \to \R^2$
  (possibly passing to another subsequence of $m_i \to 0$ to deduce the almost sure convergence from the $L^p$ convergence in \eqref{e:Phi0m-conv}), almost surely
  \begin{gather}
    \|P_T\Phi^{\SG(\beta,z|\Lambda_j,m_i)}_{t_0} - \Phi^{\SG(\beta,z)}_{t_0}\|_{C^k(\rho)} \to 0,
    \\
    \|\Phi^{\GFF(0)}_{0,t_0}-\Phi^{\GFF(m_i)}_{0,t_0}\|_{C^{r}(\rho)} \to 0,
    \\
    \|\Phi^{\Delta(\beta,z|\Lambda_j,m_i)}_{0,t_0}  -    \Phi^{\Delta(\beta,z)}_{0,t_0}\|_{C^{r}(\rho)} \to 0.
  \end{gather}
  To finish the proof, we observe that
  \begin{equation} \label{e:infvol-coupling-final}
    P_T\Phi^{\SG(\beta,z|\Lambda,m)}_0 = \Phi^{\Delta(\beta,z|\Lambda,m)}_{0,t_0} + \Phi^{\GFF(m)}_{0,t_0} + P_T\Phi^{\SG(\beta,z|\Lambda,m)}_{t_0} + X^{T,\Lambda,m}
  \end{equation}
  where
  \begin{equation}
    - X^{T,\Lambda,m} =   (\eta_T, \Phi^{\Delta(\beta,z|\Lambda,m)}_{0,t_0}) +  (\eta_T, \Phi^{\GFF(m)}_{0,t_0}).
  \end{equation}
  Thus
  with $Z = \Phi^{\GFF(0)}_{0,t_0}$ and
  \begin{equation}
    \tilde\varphi_{i,j,k} =  \Phi^{\Delta(\beta,z|\Lambda_j,m_i)}_{0,t_0} + (\Phi^{\GFF(m_i)}_{0,t_0}-\Phi^{\GFF(0)}_{0,t_0}) + P_{T_k}\Phi^{\SG(\beta,z|\Lambda_j,m_i)}_{t_0} + X^{T_k,\Lambda_j,m_i}
  \end{equation}
  it follows that $Z+\tilde\varphi_{i,j,k} \sim (P_{T_k})_\#\nu^{\SG(\beta,z|\Lambda_j,m_i)}$ and with $X_{i,j,k} = X^{T_k,\Lambda_j,m_i}$,
  \begin{equation}
    \tilde\varphi_{i,j,k} + X_{i,j,k} \to \tilde\varphi =  \Phi^{\Delta(\beta,z)}_{0,t_0} + \Phi^{\SG(\beta,z)}_{t_0} + X \qquad \text{in $C^r(\rho)$.}
  \end{equation}
  where $X$ is the limit of $X_{i,j,k}$.
  The distribution of the left-hand side of \eqref{e:infvol-coupling-final} converges to $\nu^{\SG(\beta,z)}$,
  which shows that $Z+\tilde\varphi$ also has this distribution as desired.
\end{proof}

\subsection{Tightness of the large scale field}
\label{sec:infvol-tightness}

As the main step of the proof of Lemma~\ref{lem:tight},
we will show the moment bound on the large scale field $\Phi^{\SG}_{t_0}$ stated in the following Proposition~\ref{prop:Phit0-tightCk}.
We always assume that $\Lambda\subset \R^2$ is bounded and $m>0$ but all estimates are uniform in these parameters.
To simplify the notation, we write
\begin{equation}
  \Phi^\SG = \Phi^{\SG(\beta,z|\Lambda,m)}, \qquad \Phi^\Delta = \Phi^{\Delta(\beta,z|\Lambda,m)}
\end{equation}
in this section.
As $m\to 0$ the zero mode of $\Phi_{t_0}^\SG$ diverges in finite volume and $\Phi_{t_0}^\SG$ is not tight,
but  we will prove that
$P_T\Phi_{t_0}^\SG = \Phi_{t_0}^\SG-(\eta_T,\Phi_{t_0}^\SG)$ is tight in $C^k(\rho)$,
provided that the assumption \eqref{e:SG-moments-bis}  holds (which we proved if $\beta=4\pi$).
Essentially the same proof also shows that the difference $\Phi_{t_0}^{\SG}-\Phi^{\SG}_{t_0}(0)$ is tight
in $C^k(\rho)$ for all $\beta \in (0,6\pi)$.
We will not actually use this result, but include it as it follows from the same argument and could
be of some independent interest.

\begin{proposition} \label{prop:Phit0-tightCk}
  Let $\beta\in(0,6\pi)$ and $z\in \R$, and assume  that \eqref{e:SG-moments-bis} holds.
  Then
  \begin{equation} \label{Phit0-nocenter-tightCk}
    \sup_{T\geq 1}
    \E \qa{ \|P_T\Phi^{\SG}_{t_0} \|_{C^k(\rho)}^p } \leq C_{k,p}(\beta,z).
  \end{equation}
  Moreover, without the assumption \eqref{e:SG-moments-bis},
  \begin{equation} \label{Phit0-tightCk}
    \E \qa{ \| \Phi^{\SG}_{t_0} -\Phi_{t_0}^{\SG}(0)\|_{C^k(\rho)}^p } \leq C_{k,p}(\beta,z).
  \end{equation}
\end{proposition}

The first lemma below shows that the third  derivative (and the same applies to higher derivatives) of the difference field are controlled uniformly in the volume.
Let
\newcommand{\Tay}{{\sf T}}
\begin{equation}
  \Tay_0f(x) = f(0)+\nabla f(0)\cdot x + \frac12 x^T \He f(0)x
\end{equation}
be the second order Taylor approximation of $f:\R^2 \to \R$ at $0$.

\begin{lemma} \label{lem:thirdderivative}
  Let $\beta\in (0,6\pi)$ and $z\in \R$. For any $k>2$,
  there are deterministic constants $C_k(\beta,z)$, independent of $\Lambda$ and $m$, such that
  for all $\Lambda\subset \R^2$ bounded, $m\in (0,1]$, almost surely,
  \begin{equation} \label{e:Phit0-thirdderivative}
    \|\nabla^k\Phi^{\Delta}_{t_0}\|_{L^\infty(\R^2)} \leq C_k(\beta,z).
  \end{equation}
  Therefore $\Phi^{\SG}_{t_0} \in C^\infty(\R^2)$ almost surely and, uniformly in $\Lambda\subset \R^2$ bounded, $m\in(0,1]$, for any $p>0$,
  \begin{equation} \label{e:1TPhit0-tightCk}
    \E \qa{ \|\Phi^{\SG}_{t_0}-\Tay_0\Phi_{t_0}^\SG\|_{C^k(\rho)}^p }
    \leq C_{k,p}(\beta,z)  %
  \end{equation}
  for a polynomial weight $\rho(x) = (1+|x|^2)^{-\sigma/2}$ with $\sigma > 3$.
\end{lemma}

\begin{proof}
  By the definition in \eqref{e:SDE-bis},
  \begin{equation}
    \Phi_t^\Delta = -\int_t^\infty e^{-m^2 s} e^{\Delta_x s} \nabla_\varphi V_s(\Phi_s^\SG,z,x|\Lambda,m)\, ds.
  \end{equation}
  The estimate \eqref{e:Phit0-thirdderivative} therefore follows from \eqref{e:gradVt-maxprinciple-deriv} in Proposition~\ref{prop:Vt-infvol} and $k>2$:
  \begin{align}
  \|\nabla_x^k \Phi^\Delta_{t_0}\|_{L^\infty_x(\R^2)}
    &\leq \int_{t_0}^\infty  \|\nabla_x^k e^{t\Delta_x}\nabla_\varphi V_t(\Phi^{\SG}_t,x|\Lambda,m)\|_{L^\infty_x(\R^2)} \, dt
      \nnb
  &\leq C_k(\beta,z) \int_{t_0}^\infty t^{-k/2} \, dt \leq C_k'(\beta,z)
\end{align}
with deterministic constants.

By Taylor's theorem, provided that the weight $\rho(x)=(1+|x|^2)^{-\sigma/2}$ is defined with $\sigma> 3$,
it also follows that (again with deterministic constants)
\begin{equation}
  \|\Phi^{\Delta}_{t_0}-\Tay_0\Phi_{t_0}^\Delta\|_{C^k(\rho)} \leq C_k''(\beta,z).
\end{equation}
The bound \eqref{Phit0-tightCk} is completed by the analogous moment bound for the Gaussian field $\Phi^{\GFF}_{t_0}$
with smooth covariance. This follows from standard arguments, for example, by estimating high weighted Sobolev norms, see
Lemma~\ref{lem:Gauss-Phit0} in Appendix~\ref{app:Gauss} for details.
\end{proof}

The following lemma shows that the large scale part $\Phi^{\SG}_{t_0}= \Phi^{\GFF}_{t_0}+\Phi^{\Delta}_{t_0}$ in the decomposition has bounded fluctuations
as $m\to 0$ and $\Lambda\to\R^2$.
We emphasize that individually $\Phi^{\GFF}_{t_0}$ and $\Phi^{\Delta}_{t_0}$ do not have bounded fluctuations.

\begin{lemma} \label{lem:Phit0-moment}
  Let $\beta \in (0,6\pi)$ and $z\in \R$ and assume the
  moment bound \eqref{e:SG-moments-bis} with norm $\|\cdot\|$ for some $p \geq 2$.
  Then for any $g\in C_c^\infty(\R^2)$ with $\hat g(0) = 0$, uniformly in $\Lambda\subset\R^2$ bounded, $m>0$,
  \begin{equation} \label{e:Phit0-moment-p}
    \E \qB{ \absa{(\Phi_{t_0}^{\SG},g)}^p} < C_p(\beta,z)(\|g\|_{L^1} + \|g\|)^p.
  \end{equation}
  In particular, without the assumption $\hat g(0)=0$,
  \begin{equation} \label{e:Phit0-moment-p-kappa}
    \sup_{T\geq 1}
    \E \qB{ |(P_T\Phi_{t_0}^{\SG},g)|^p}
    \leq C_p(\beta,z) (\|g\|_{L^1} + \|g\|)^p
    .
  \end{equation}
\end{lemma}

\begin{remark} \label{rk:Phit0-moment-0}
  The bound %
  \eqref{e:Phit0-moment-p} also holds with $\|g\|_{\dot H^{-1}}$ on the right-hand sides
  without imposing the assumption \eqref{e:SG-moments-bis} (by instead using the bound \eqref{e:SG-variance-0} which is known also for $\beta \neq 4\pi$).
  However, the extension to $\hat g(0)\neq 0$ as in \eqref{e:Phit0-moment-p-kappa} then does not exist.
\end{remark}

\begin{proof}
  Using the SDE~\eqref{e:SDE-bis}, the massive sine-Gordon field $\Phi_0^{\SG}\sim \avg{\cdot}_{\SG(\beta,z|\Lambda,m)}$ with $|\Lambda|<\infty$, $m>0$
  can be expressed as
  \begin{equation}
    \Phi^{\SG}_0 = \Phi^{\SG}_{0,t_0} + \Phi^{\SG}_{t_0}
    .
  \end{equation}
  By the Jensen inequality one therefore has
  \begin{equation}
    \E\qB{ |(\Phi_{t_0}^{\SG},g)|^p}
    =
    \E\qB{ |(\Phi_{0}^{\SG}-\Phi_{0,t_0}^{\SG},g)|^p}
    \leq 2^{p-1}\qB{ \E |(\Phi^{\SG}_0,g)|^p + \E|(\Phi^{\SG}_{0,t_0},g)|^p }.
  \end{equation}
  The first term on the right-hand side is the moment of the full sine-Gordon field bounded by the
  moment assumption \eqref{e:SG-moments-bis} by $C_p(\beta,z) \|g\|^p$ provided that $\hat g(0)=0$
  or $P_T$ is applied.

  The second term is the variance of a massive sine-Gordon field with external mass essentially $1/\sqrt{t_0}$ and external field $h=C_{t_0}^{-1}\Phi_{t_0}^{\SG}$
  whose covariance can be  bounded as an operator uniformly in $h$, see \cite{MR4303014}.
Since we do not need this optimal bound,
  it is more convenient to derive an estimate with $\|g\|_{L^1}^2+\|g\|_{L^2}^2$ instead of $\|g\|_{L^2}^2$ on the right-hand side using that,
  deterministically,
  \begin{equation}
    (\Phi^\Delta_{0,t_0}, g) \leq \|\Phi^\Delta_{0,t_0}\|_{L^\infty(\R^2)} \|g\|_{L^1(\R^2)} \lesssim \|g\|_{L^1(\R^2)},
  \end{equation}
  and that the bound for the Gaussian term is elementary:
  \begin{equation}
    \E\qB{ |(\Phi^{\GFF(m)}_{0,t_0},g)|^p } \leq (Cp)^{p/2} \qa{\int_{0}^{t_0} e^{-m^2 t} (g, e^{\Delta t} g)}^{p/2} \leq (Cp)^{p/2} \|g\|_{L^2(\R^2)}^{p}.
  \end{equation}
  Since $\Phi^\SG_{0,t_0} = \Phi^\Delta_{0,t_0} + \Phi^{\GFF}_{0,t_0}$ this completes the proof of  %
  \eqref{e:Phit0-moment-p}.
  The bound \eqref{e:Phit0-moment-p-kappa} is an immediate consequence since
  \begin{equation}
    \|g-\hat g(0)\eta\|_{L^1} \lesssim \|g\|_{L^1}, \qquad \|g-\hat g(0) \eta_T \|_{L^2} \lesssim \|g\|_{L^1} + \|g\|_{L^2},
  \end{equation}
  where we have used $|\hat g(0)| \leq \|g\|_{L^1}$ and $\|\eta_T\|_{L^2} \lesssim 1$ since $T \geq 1$ in the second inequality.
\end{proof}

\begin{proof}[Proof of Proposition~\ref{prop:Phit0-tightCk}]
The moment bound \eqref{e:1TPhit0-tightCk} implies that for any $g\in C_c^\infty(\R^2)$ there is a constant $C_p(g)<\infty$ such that 
\begin{equation}
  \E \qB{\absa{((1-\Tay_0)\Phi_{t_0}^{\SG},g)}^p }
  \leq C_p(g),
\end{equation}
whereas Lemma~\ref{lem:Phit0-moment} implies that for any $g\in C_c^\infty(\R^2)$,
\begin{equation} \label{e:Phit0-moment}
  \sup_{T\geq 1}\E \qB{ \absa{(P_T\Phi_{t_0}^{\SG},g)}^p } \leq C_p(g).
\end{equation}
Together, using that $(1-\Tay_0)\Phi_{t_0}^\SG = (1-\Tay_0)P_T\Phi_{t_0}^\SG$, therefore 
\begin{equation} \label{e:TPhit0-moment}
  \sup_{T\geq 1}\E \qB{ \absa{(\Tay_0P_T\Phi_{t_0}^{\SG},g)}^p } \leq C_p(g).
\end{equation}

Moreover, without the assumption \eqref{e:SG-moments-bis},
using Remark~\ref{rk:Phit0-moment-0} with Lemma~\ref{lem:Phit0-moment},
it still follows that for all $g\in C_c^\infty(\R^2)$ with $\hat g(0)=0$,
\begin{equation} \label{e:TPPhit0-moment}
  \E \qB{ \absa{(\Tay_0\Phi_{t_0}^{\SG},g)}^p } \leq C_p(g).
\end{equation}

By choosing suitable $g \in C_c^\infty(\R^2)$, all coefficients of the random quadratic polynomial
$\Tay_0P_T\Phi_{t_0}^\SG$
respectively $\Tay_0\Phi_{t_0}^{\SG}-\Phi_{t_0}^\SG(0)$ have bounded second moments.
Explicitly, there are random variables $Y_n$ indexed by multi-indices $n=(n_1,n_2)$ with $|n| = n_1+n_2 \leq 2$ such that
  \begin{equation}
  \Tay_0P_T\Phi_{t_0}^\SG(x) = \sum_{|n|\leq 2} Y_n \frac{x^n}{n!}
\end{equation}
where $x^n= x_1^{n_1}x_2^{n_2}$ and $n!= n_1!n_2!$. Thus \eqref{e:TPhit0-moment} can be written as
\begin{equation} \label{e:Tphit0-moment-Y}
  \sup_{T\geq 1}\E\qa{ \absa{\sum_n Y_n \int \frac{x^n}{n!} g(x) \, dx}^p } \leq C_p(g),
\end{equation}
and the analogous bound holds for $\Tay_0\Phi_{t_0}^{\SG}$ if $\hat g(0)=0$.
The test functions $g \in C_c^\infty(\R^2)$ that extract the coefficients can be constructed as follows
(using a trick that we learned from \cite{MR3925533}).
Fix some function $\psi \in C_c^\infty(\R^2)$ with $\int \psi(x) \,dx =1$.
For any multi-indices $m=(m_1,m_2)$ and $n=(n_1,n_2)$, then set
\begin{equation}
  A_{nm}
  = \int \frac{x^n}{n!} (-\partial)^m \psi(x) \, dx
  = \int (\partial^m \frac{x^n}{n!}) \psi(x) \, dx.
\end{equation}
Then $A_{nm} = \1_{n=m}$ for $|m| \geq |n|$.
Thus for any  fixed integers $N_1 \geq N_0\geq 0$, the matrix $A= (A_{nm})_{N_0\leq |n|,|m|\leq N_1}$ is triangular with all diagonal entries (eigenvalues) $1$ and therefore invertible.
Thus given a multi-index $n$ with $N_0\leq |n|\leq N_1$ we can find $a = a^{(n)}$ such that $(Aa)_m = \1_{n=m}$ for $N_0 \leq |m|\leq N_1$.
We set
\begin{equation}
  g^{(n)} = \sum_{m: N_0\leq |m|\leq N_1} a_{m}^{(n)}(-\partial)^m \psi \in C_c^\infty(\R^2).
\end{equation}
Then for a multi-index $k$ with $N_0\leq |k|\leq N_1$,
\begin{equation}
  \int \frac{x^n}{n!} g^{(k)}(x) \, dx
  = \sum_m a_{m}^{(k)} \int \frac{x^n}{n!} (-\partial)^m \psi(x) \, dx
  = \sum_m A_{nm} a_m^{(k)} = \1_{n=k}.
\end{equation}
From now on, choose $N_0=0$ and $N_1=2$  and substitute the corresponding $g=g^{(n)}$ into  \eqref{e:Tphit0-moment-Y}
to conclude that, for all $0\leq |n|\leq 2$,
\begin{equation}
  \sup_{T\geq 1}\E[|Y_n|^p] \leq C_p,
\end{equation}
which is the claim.
In particular, if the weight $\rho(x)$ has $\sigma>2$, it follows that
\begin{equation}
  \sup_{T\geq 1}\E\qa{\|\Tay_0P_T\Phi^{\SG}_{t_0}\|_{C^k(\rho)}^p} \leq
  C_{k,p}(\beta,z)
\end{equation}
Together with \eqref{e:1TPhit0-tightCk}, where again $\Phi_{t_0}^\SG$ can be replaced $P_T\Phi_{t_0}^\SG$, the claimed bound \eqref{Phit0-nocenter-tightCk} follows.

For $\Tay_0\Phi_{t_0}^\SG-\Phi_{t_0}^\SG(0)$ an analogous argument applies (with the restriction
$\hat g(0)=0$). Taking $N_0=1$ and using $\hat g^{(n)}(0) =0$ the rest of the argument is identical.
\end{proof}

\section{Massless Bosonization of twisted free fermions}
\label{sec:massless-bosonization}

In this section, we define the Green's function of the massless twisted Dirac operator and define the corresponding fermionic correlation functions.
We then prove the Bosonization identities relating massless fermionic correlation functions with suitable GFF correlation functions.
The massless case considered here is well known,
but it provides the foundation of the massive Bosonization identities we prove in Section~\ref{sec:bosonization}
and also sets up the notation we need subsequently.

\subsection{Massless twisted free fermions}

To define massless twisted free fermions, we first introduce the Green's function for the massless Dirac operator on $\C \cong \R^2$,
\begin{equation} \label{e:Dirac-def}
  \Dirac = \begin{pmatrix} 0 & 2\bar\partial \\ 2\partial & 0 \end{pmatrix}
\end{equation}
with \emph{branching} 
around distinct points $x_1,\dots, x_n$ of the complex plane with a prescribed branching structure
with windings $\alpha_1,\dots, \alpha_n \in \R$ which we assume to satisfy $\sum_i \alpha_i =0$.
The branching structure is specified in terms of the function $\rho: \C \setminus \Gamma \to \C$ defined in \eqref{eq:rho}.

\begin{definition} \label{def:S0}
  The Green's function $S_0^\rho=S_0$ for the massless Dirac operator on $\C$ with branching described by $\rho$ is given
  for $z\notin \Gamma\cup\{w\}$  by
\begin{equation}\label{eq:S0}
S_0(z,w)=\begin{pmatrix}
0 & \frac{\overline{\rho(z)}}{\overline{\rho(w)}}\frac{1}{2\pi}\frac{1}{\bar z-\bar w}\\
\frac{\rho(w)}{\rho(z)}\frac{1}{2\pi}\frac{1}{z-w} & 0
\end{pmatrix}.
\end{equation}
\end{definition}

It is well known (and easy to verify) that the above Green's function $S_0^\rho$ can be regarded as the inverse to $\Dirac_\rho$
in the following sense. It %
satisfies
\begin{equation}
  \Dirac S_0^\rho(z,w) = \delta(z-w) \id_{2\times 2} \qquad \text{on $\C \setminus \Gamma$},
\end{equation}
where $\Dirac$ acts on the $z$ variable, $\lim_{z\to \infty} S_0^\rho(z,w)=0$ for every $w\in \C\setminus \Gamma$, and
\begin{equation}
  \begin{pmatrix}
    \frac{1}{\overline{\rho(z)}} & 0 \\
    0 & \rho(z)
  \end{pmatrix}
  S^\rho_0(z,w)
\end{equation}
has a continuous extension to $\C\setminus \{w\}$.

Instead of defining (Euclidean, Dirac) fermionic correlation functions as being analytic continuations of time-ordered correlation functions of fermionic operators on a Hilbert space or being given by a suitable infinite dimensional Grassmann integral, we take an operational approach.
Namely, given a kernel $S(z,w)=(S_{ij}(z,w))_{i,j=1}^2$, we define for distinct points $z_1,\dots,z_k\in \C$ and $w_1,\dots,w_k\in \C$ and labels $a_1,\dots,a_k,b_1,\dots,b_k\in \{1,2\}$,
\begin{equation}\label{eq:fermicor1}
\avg{\prod_{j=1}^k \bar \psi_{a_j}(z_j)\psi_{b_j}(w_j)}_{\FF(S)}=\det(S_{a_ib_j}(z_i,w_j))_{i,j=1}^k.
\end{equation}

In this section, $S$ is always the Green's function of the massless Dirac operator $S_0^\rho$ defined above and we then write $\FF_\rho(0)$ for $\FF(S_0^\rho)$.
We will be particularly interested in the situation where $z_i=w_i$.
The correlation functions are in general not well defined since $S$ has a singularity on the diagonal, but
\emph{truncated correlation functions} (cumulants) are well defined also for noncoincident points
(equivalently normal ordered fields have well-defined correlation functions).
Indeed, the following definition can be motivated through %
Grassmann integration (see for example \cite[Appendix~A]{MR4767492}),
but we can take it simply as a definition.

\begin{definition}
For distinct points $z_1,\dots,z_n\in \C$, define
\begin{equation}\label{eq:fermicor2}
\avg{\bar\psi_{a_1}\psi_{b_1}(z_1);\dots ;\bar \psi_{a_k}\psi_{b_k}(z_k)}_{\FF(S)}^\mathsf T=(-1)^{k+1} \sum_{\pi\in C_k}\prod_{i=1}^k S_{a_{\pi^i(1)},b_{\pi^{i+1}(1)}}(z_{\pi^i(1)},z_{\pi^{i+1}(1)}),
\end{equation}
where $C_k$ denotes the set of cyclic permutations on $\{1,\dots,k\}$.
\end{definition}

In particular, if $a_i=b_i$ for all $i$, the truncated correlation functions can be related to the untruncated ones (if they are finite) by
\begin{multline} \label{eq:fermcor3}
  \avg{\bar\psi_1\psi_1(z_1);\dots;\bar\psi_1\psi_1(z_k); \bar\psi_2\psi_2(z_{k+1});\dots;\bar\psi_2\psi_2(z_{k+k'})}_{\FF_\rho(0)}^\mathsf T
  \\
    =\sum_{\pi\in \frP_{k+k'}}(-1)^{|\pi|-1}(|\pi|-1)!\prod_{B\in \pi}\avg{\prod_{j\in B}\bar \psi_{a_j}\psi_{a_j}(z_j)}_{\FF_\rho(0)},
\end{multline}
where $a_j=1$ for $j\leq k$ and $a_j=2$ for $j>k$,
and $\frP_{k+k'}$ denotes the set of partitions of $\{1,\dots,k+k'\}$ and $|\pi|$ denotes the number of parts in such a partition.
We used that for the massless Dirac operator (with or without branching),  $S_{0,ii}^\rho=0$ and hence the untruncated correlation function
\begin{equation}
  \avg{\prod_{j=1}^k\bar \psi_{a_j}\psi_{a_j}(z_j)}_{\FF(S)} = \det(S_{0,a_ia_j}^\rho(z_i,z_j))_{i,j=1}^k
\end{equation}
is well defined for distinct points $z_i$ (consistent with \eqref{eq:fermicor1}).
The general formula for cumulants in terms of moments follows for example by
iterating \cite[Lemma~A.1]{MR4767492}.

\medskip
The partition function of the free fermions with kernel $S$ should formally be given by $\det(S^{-1})$.
In our case of interest, $S = \Dirac_\rho^{-1}$ and the corresponding determinant is not naively defined.
In Appendix~\ref{app:det} we however argue that the ratio $\det(\Dirac_\rho)/\det(\Dirac)$ is naturally interpreted as follows.

\begin{remark} \label{rk:det-rho1}
As explained in Appendix~\ref{app:det}, one should interpret
\begin{equation} \label{e:ren-det-massless}
  \frac{\det(\Dirac_\rho)}{\det(\Dirac)}
  \propto \prod_{1\leq r<s<n} |x_r-x_s|^{2\alpha_r\alpha_s}.
\end{equation}
For our purposes, it suffices to take the right-hand side as the definition of the left-hand side.
\end{remark}

\subsection{Massless free bosonic field with charge insertions}\label{sec:GFFinsert}

We first recall the standard Gaussian free field (without branching).
Denote by $\avg{\cdot}_{\GFF(m)}$ the expectation of the Gaussian free field with mass $m>0$ on $\R^2$,
and by $\avg{\cdot}_{\GFF(0)}$ the expectation of the massless Gaussian free field on $\R^2$ defined modulo constants.
We also consider versions of these expectations regularized at small scales. The precise choice of regularisation is not important,
but to be concrete we define
\begin{equation}
  \avg{F(\varphi)}_{\GFF(m,\epsilon)} = \E\qa{F(\Phi_{\epsilon^2,\infty}^{\GFF(m)})}
\end{equation}
with the notation $\Phi^\GFF$ from Section~\ref{sec:Gauss},
i.e., the above is the expectation of the Gaussian field with covariance
\begin{equation}
  \int_{\epsilon^2}^\infty e^{-m^2 t} e^{\Delta t} \, dt
\end{equation}
where $e^{\Delta t}$ is the heat kernel on $\R^2$, and the %
corresponding measure is supported on $C^\infty(\R^2)$.
The notation $\avg{\cdot}_{\GFF(m,\epsilon)}$ except that we switched the order of the parameters $m$ and $\epsilon$
is the one used in \cite{MR4767492} to which we refer for some statements.

\begin{definition}
Given the points $x_1,\dots,x_n$ and $\alpha_1,\dots,\alpha_n$ encoded by $\rho$ as in \eqref{eq:rho}, the ``expectation'' of the
Gaussian free field with charge insertions encoded by $\rho$ is defined by
\begin{equation} \label{e:GFF-rho}
  \avg{F(\varphi)}_{\GFF_\rho(0)}
  =\frac{\avg{ \prod_i \wick{e^{i\sqrt{4\pi}\alpha_i \varphi(x_i)}}F(\varphi)}_{\GFF(0)}}{\avg{\prod_i \wick{e^{i\sqrt{4\pi}\alpha_i \varphi(x_i)}}}_{\GFF(0)}}.
\end{equation}
More precisely, we understand the normal ordering (for suitable $F$)
as in Section~\ref{sec:Gauss-corr}, namely as defined with the heat kernel regularization as
\begin{equation}
  \avg{\prod_i \wick{e^{i\sqrt{4\pi}\alpha_i \varphi(x_i)}} F(\varphi) }_{\GFF(0)}
  = \lim_{m\to 0}\lim_{\epsilon \to 0}
  \avg{\prod_i \wick{e^{i\sqrt{4\pi}\alpha_i \varphi(x_i)}}_\epsilon F(\varphi)}_{\GFF(m,\epsilon)},
\end{equation}
and thus $\avg{F(\varphi)}_{\GFF_\rho(0)}$ is the limit of
\begin{equation} \label{e:GFF-rho-m-eps}
  \avg{F(\varphi)}_{\GFF_\rho(m,\epsilon)}
  = \frac{\avg{\prod_i \wick{e^{i\sqrt{4\pi}\alpha_i \varphi(x_i)}}_\epsilon F(\varphi)}_{\GFF(m,\epsilon)}}
  {\avg{\prod_i \wick{e^{i\sqrt{4\pi}\alpha_i \varphi(x_i)}}_\epsilon}_{\GFF(m,\epsilon)}}
\end{equation}
where
$\wick{e^{i\sqrt{4\pi}\alpha \varphi(x)}}_\epsilon = \epsilon^{-\alpha^2} e^{i\sqrt{4\pi}\alpha \varphi(x)}$.
\end{definition}

The expectation $\avg{\cdot}_{\GFF_\rho(0)}$ is not the expectation of a probability measure, but %
it is convenient to use the expectation notation nonetheless. The expectation of the normal ordered fields in the following equation is defined as above.

\begin{lemma} \label{lem:GFF-expi}
  For distinct $x_1,\dots, x_n\in\C$ and $\alpha_1,\dots,\alpha_n\in \R$ with $\sum_i \alpha_i =0$,
  \begin{equation}
    \avg{\prod_{i=1}^n \wick{e^{i\sqrt{4\pi}\alpha_i \varphi(x_i)}} }_{\GFF(0)}
    = (2e^{-\gamma/2})^{\sum_i \alpha_i^2} \prod_{1\leq r< s\leq n} |x_r-x_s|^{2\alpha_r\alpha_s}
  \end{equation}
\end{lemma}

\begin{proof}
  The proof is an elementary Gaussian computation, see Lemma~\ref{le:gfffrac} for details.
\end{proof}

\begin{lemma} \label{lem:GFFrho}
  For all $w_1,\dots, w_p,w_1',\dots,w_p' \in \C \setminus \Gamma$ distinct,
  \begin{equation}
    \avg{\prod_{l=1}^p \wick{e^{i\sqrt{4\pi}\varphi(w_l)}}\prod_{l'=1}^p \wick{e^{-i\sqrt{4\pi}\varphi(w_{l'}')}}}_{\GFF_\rho(0)}
    = (4e^{-\gamma})^{p}
    \prod_{l=1}^p \frac{|\rho(w_l)|^2}{|\rho(w_l')|^2}\frac{\prod_{1\leq r<s\leq p}|w_r-w_s|^2 |w_r'-w_s'|^2}{\prod_{r,s=1}^p |w_r-w_s'|^2}.
  \end{equation}
\end{lemma}

\begin{proof}
  It suffices to compute the numerator and denominators in the definition of $\avg{\cdot}_{\GFF_\rho(0)}$.
  Both are instances of Lemma~\ref{lem:GFF-expi} which gives
\begin{align}
&\avg{\prod_{j=1}^n\wick{e^{i\sqrt{4\pi}\alpha_j\varphi(x_j)}}\prod_{l=1}^p \wick{e^{i\sqrt{4\pi}\varphi(w_l)}}\prod_{l'=1}^p \wick{e^{-i\sqrt{4\pi}\varphi(w_{l'}')}}}_{\GFF(0)}\nnb
  &=(4e^{-\gamma})^{p}
    \avg{\prod_{j=1}^n\wick{e^{i\sqrt{4\pi}\alpha_j\varphi(x_j)}}}_{\GFF(0)} \prod_{j=1}^n \prod_{l=1}^p \frac{|x_j-w_l|^{2\alpha_j}}{|x_j-w_{l}'|^{2\alpha_j}}\frac{\prod_{1\leq r<s\leq p}|w_r-w_s|^2 |w_r'-w_s'|^2}{\prod_{r,s=1}^p |w_r-w_s'|^2}.
\end{align}
By the definition of $\rho$, see \eqref{eq:rho}, we see that 
\begin{equation}
\prod_{j=1}^n \prod_{l=1}^p \frac{|x_j-w_l|^{2\alpha_j}}{|x_j-w_{l}'|^{2\alpha_j}}=\prod_{l=1}^p \frac{|\rho(w_l)|^2}{|\rho(w_l')|^2}
\end{equation}
and hence the claim follows.
\end{proof}

\subsection{Massless Bosonization}

As an important preliminary ingredient of our proof of massive Bosonization (in Section~\ref{sec:bosonization}), we need a massless version relating
fermionic correlation functions of bilinear fields as in \eqref{eq:fermicor2} and with the massless Dirac Green's function $S=S_0 = S_0^\rho$ from \eqref{eq:S0} to suitable GFF correlation functions.
We recall that the corresponding expectation is written $\avg{\cdot}_{\FF_\rho(0)}=\avg{\cdot}_{\FF(S_0^\rho)}$.

\begin{proposition}\label{le:boso0}
  For $x_1,\dots,x_n\in \C$ distinct points, $\alpha_1,\dots,\alpha_n\in (-\frac{1}{2},\frac{1}{2})$ satisfying $\sum_{j=1}^n \alpha_j=0$, $z_1,\dots,z_{k+k'}\in \C\setminus \{x_1,\dots,x_n\}$ distinct points,
\begin{equation}\label{eq:momboso}
  \avg{\prod_{l=1}^k \bar\psi_1\psi_1(z_l)\prod_{l'=1}^{k}\bar \psi_2\psi_2(z_{k+l'})}_{\FF_\rho(0)}
  = (\frac{e^{\gamma/2}}{4\pi})^{2k}
  \avg{\prod_{l=1}^k\wick{e^{i\sqrt{4\pi}\varphi(z_l)}}\prod_{l'=1}^{k}\wick{e^{-i\sqrt{4\pi}\varphi(z_{k+l'})}}}_{\GFF_\rho(0)},
\end{equation}
and both sides vanish if the number of $\bar\psi_1\psi_1$ and $\bar\psi_2\psi_2$ factors do not agree.
In particular,
\begin{align}
  &\avg{\bar \psi_{1}\psi_{1}(z_1);\dots;\bar \psi_{1}\psi_{1}(z_k);\bar \psi_2\psi_2(z_{k+1});\dots;\bar\psi_2\psi_2(z_{k+k'})}_{\FF_\rho(0)}^\mathsf T
    \nnb
  &= (\frac{e^{\gamma/2}}{4\pi})^{k+k'}
    \sum_{\pi\in \frP_{k+k'}}(-1)^{|\pi|-1}(|\pi|-1)! \prod_{B\in \pi}
\avg{\prod_{p\in B}\wick{e^{i\sigma_p \sqrt{4\pi}\varphi(z_p)}}}_{\GFF_\rho(0)}
,
\end{align}
where $\sigma_p=1$ if $p\leq k$ and $\sigma_p=-1$ if $p>k$, $\frP_{k+k'}$ denotes the set of partitions of $\{1,\dots,k+k'\}$ and $|\pi|$ denotes the number of parts in such a partition.
In fact, both sides vanish unless $k=k'$.
\end{proposition}

\begin{proof}
By \eqref{eq:fermcor3}, for $z_1, \dots, z_{k+k'}$ distinct,
\begin{equation}
  \avg{\bar \psi_{1}\psi_{1}(z_1);\dots;\bar \psi_{2}\psi_{2}(z_{k+k'})}_{\FF_\rho(0)}^\mathsf T
  =\sum_{\pi\in \frP_{k+k'}}(-1)^{|\pi|-1}(|\pi|-1)!\prod_{B\in \pi}\avg{\prod_{j\in B}\bar \psi_{a_j}\psi_{a_j}(z_j)}_{\FF_\rho(0)},
\end{equation}
where $a_j=1$ for $j\leq k$ and $a_j=2$ for $j>k$.
The bilinear correlation functions on the right-hand side are well defined since the diagonal entries to the Green's function vanish on the diagonal
and the $z_i$ are distinct.
Thus it is sufficient to prove \eqref{eq:momboso}, i.e., that for distinct points $w_1,\dots,w_p,w_1',\dots,w_{q}'\in \C\setminus \{x_1,\dots,x_n\}$,
\begin{equation}\label{eq:momboso-pf}
  \avg{\prod_{l=1}^p \bar\psi_1\psi_1(w_l)\prod_{l'=1}^{q}\bar \psi_2\psi_2(w_{l'}')}_{\FF_\rho(0)}=
  (\frac{e^{\gamma/2}}{4\pi})^{p+q}
  \avg{\prod_{l=1}^p\wick{e^{i\sqrt{4\pi}\varphi(w_l)}}\prod_{l'=1}^{q}\wick{e^{-i\sqrt{4\pi}\varphi(w_{l'}')}}}_{\GFF_\rho(0)}
  .
\end{equation}
As in  \cite[Lemma 2.1]{MR4767492}, both sides vanish if $p\neq q$, so assume from now on that $p=q$.
Moreover, arguing as in the same lemma, we can use anticommutativity and that $S_{11}=S_{22}=0$ to see that
\begin{align}
  \avg{\prod_{l=1}^p\bar\psi_{1}\psi_{1}(w_l)\prod_{l'=1}^p \bar \psi_2\psi_2(w_{l'}')}_{\FF_\rho(0)}
  &=(-1)^p \avg{\prod_{l=1}^p \bar\psi_1(w_l)\psi_2(w_{l}')}_{\FF_\rho(0)}\avg{\prod_{l=1}^p \bar\psi_2(w_l')\psi_1(w_{l})}_{\FF_\rho(0)}\nnb
  &=(-1)^p \det\left(S_{0,12}(w_l,w_{l'}')\right)_{l,l'=1}^p
    \nnb &\qquad\qquad \times \det\left(S_{0,21}(w_{l'}',w_{l})\right)_{l,l'=1}^p.
\end{align}
By \eqref{eq:S0}, we have $S_{0,12}(z,w)=\frac{\overline{\rho(z)}}{\overline{\rho(w)}}\frac{1}{2\pi}\frac{1}{\bar z-\bar w}$ and thus
\begin{align}
\det\left(S_{0,12}(w_l,w_{l'}')\right)_{l,l'=1}^p=\frac{1}{(2\pi)^p}\prod_{l=1}^p \frac{\overline{\rho(w_l)}}{\overline{\rho(w_{l}')}}\det\left(\frac{1}{\bar w_l-\bar w_{l'}'}\right)_{l,l'=1}^p.
\end{align}
Similarly,
\begin{align}
\det\left(S_{0,21}(w_{l'}',w_{l})\right)_{l,l'=1}^p=\frac{1}{(2\pi)^p}\prod_{l=1}^p \frac{\rho(w_l)}{\rho(w_{l}')}\det\left(\frac{1}{w_{l'}'- w_{l}}\right)_{l,l'=1}^p.
\end{align}
Thus we find that 
\begin{equation}
\avg{\prod_{l=1}^p\bar\psi_{1}\psi_{1}(w_l)\prod_{l'=1}^p \bar \psi_2\psi_2(w_{l'}')}=\frac{1}{(2\pi)^{2p}}\prod_{l=1}^p \frac{|\rho(w_l)|^2}{|\rho(w_l')|^2}\left|\det\left(\frac{1}{w_l-w_{l'}'}\right)_{l,l'=1}^p\right|^2.
\end{equation}
On the other hand, the right-hand side of \eqref{eq:momboso} is given by Lemma~\ref{lem:GFFrho}
and agrees by the Cauchy-Vandermonde identity
\begin{equation}
\frac{\prod_{1\leq r<s\leq p}|w_r-w_s|^2 |w_r'-w_s'|^2}{\prod_{r,s=1}^p |w_r-w_s'|^2}=\left|\det\left(\frac{1}{w_l-w_{l'}'}\right)_{l,l'=1}^p\right|^2.
\end{equation}
This concludes the proof.
\end{proof}

\begin{remark} \label{rk:det-rho}
By Remark~\ref{rk:det-rho1} and Lemma~\ref{lem:GFF-expi},
\begin{equation} \label{e:ren-det-massless-bis}
  \frac{\det(\Dirac_\rho)}{\det(\Dirac)} %
  \propto \avg{\prod_{j=1}^n \wick{e^{i\sqrt{4\pi}\alpha_j\varphi(x_j)}}}_{\GFF(0)}.
\end{equation}
Moreover, one can show that the Bosonization dictionary \eqref{e:corr1}--\eqref{e:corr4} holds if bosonic correlations
are weighted by complex exponentials and fermionic correlations have the corresponding twisting, i.e.,
bosonic ones are evaluated with respect to $\avg{\cdot}_{\GFF_\rho(0)}$ and fermionic ones with respect to $\avg{\cdot}_{\FF_\rho(0)}$.
We will not directly use this massless twisted Bosonization dictionary in this paper (so do not include the details).
\end{remark}

\section{Twisted Dirac operator with finite-volume mass}\label{sec:Green}

In this section, we introduce the Green's function for the \emph{twisted Dirac operator}
with finite-volume mass term $\mu\chi$:
\begin{equation}
  \Dirac_\rho+ \mu\chi = \begin{pmatrix} \mu\chi & 2\rho^{-1}\bar\partial\rho \\ 2\bar\rho\partial\bar\rho^{-1} & \mu \chi \end{pmatrix},
\end{equation}
where $\mu \in \R$ and where we assume through this section that
\begin{equation}
  \label{e:chi}
  \begin{gathered}
  \chi \in L^\infty(\C), \quad \chi \geq 0, \\
  \supp(\chi) \text{ is a simply connected compact set with smooth boundary}.
\end{gathered}
\end{equation}

The assumption on $\chi$ could probably be weakened to some degree, but it is more general than what we need for our
main results (which are indicator functions of disks).

The specification of the branching structure $\rho$ is as in \eqref{eq:rho}:
given distinct points $x_1,\dots,x_n\in\C$ 
and $\alpha_1,\dots,\alpha_n\in (-\frac{1}{2},\frac{1}{2})$ satisfying $\sum_{j=1}^n \alpha_j=0$,
\begin{equation}\label{eq:rho-bis}
\rho(z)=\prod_{j=1}^n (z-x_j)^{\alpha_j}=\prod_{j=1}^n e^{\alpha_j\log(z-x_j)},
\end{equation}
and the exact choice of the branches for the logarithm is not important.
The union of all branch cuts is denoted $\Gamma$ %
(and can be chosen as around \eqref{eq:rho} if the imaginary parts are distinct but does not have to be).
We also introduce the notation %
\begin{equation}
P(z)=P_\rho(z) =\begin{pmatrix} \rho(z) & 0 \\ 0 & 1/\overline{\rho(z)}\end{pmatrix},\label{BranchingMat}
\end{equation}
so that
\begin{equation}
\Dirac_\rho+\mu\chi=P^{-1}\Dirac \overline{P}^{-1}+\mu\chi I,
\end{equation}
where $I$ is the identity operator and $\overline{M}$ of a matrix $M$ denotes complex conjugation of the matrix elements.
Observe that
\begin{equation}
  \slashed\partial_\rho f(z)=\slashed\partial f(z) \quad \text{for all $f\in C_c^\infty(\C\setminus \Gamma;\C^2)$}.
\end{equation}

The massless Green's function $S_0=S_0^\rho$, i.e., the Green's function for the above twisted Dirac operator with
$\mu=0$, is explicitly defined in \eqref{eq:S0}.
For $\mu \neq 0$ the Green's function is however not explicit
and this section establishes the existence, uniqueness, and various bounds for the Green's function
of $\Dirac_\rho + \mu \chi$.

We begin with a theorem establishing existence and uniqueness.
For $\mu \neq 0$ we always assume $x_1,\dots,x_n\in\mathrm{supp}(\chi)$.

\begin{proposition}\label{pr:diracexist}
Let $\mu\in \R$, assume $\chi$ satisfies \eqref{e:chi}, %
and let $x_1,\dots,x_n\in \mathrm{supp}(\chi)$
and $\alpha_1,\dots,\alpha_n\in (-\frac{1}{2},\frac{1}{2})$ with $\sum_{j=1}^n \alpha_j=0$.
Then for every $w\in \C\setminus \Gamma$ (where $\Gamma$ is the branch cut of $\rho$ from \eqref{eq:rho-bis})
there exists a unique complex $2\times 2$ matrix-valued function $S_{\chi\mu}^\rho(\cdot,w)=S(\cdot,w)=(S_{ij}(\cdot,w))_{i,j=1}^2$
defined on $(\C\setminus \Gamma) \setminus \{w\}$
which satisfies the following three conditions:

\smallskip\noindent (i)
On $\C\setminus \Gamma$ we have, 
\begin{equation}\label{eq:dirac}
\begin{pmatrix}
\mu\chi(z) & 2\bar\partial_z \\
2\partial_z & \mu\chi(z)
\end{pmatrix}\begin{pmatrix}
S_{11}(z,w) & S_{12}(z,w)\\
S_{21}(z,w) & S_{22}(z,w)
\end{pmatrix}=\begin{pmatrix}
\delta(z-w) & 0\\
0 & \delta(z-w)
\end{pmatrix}.
\end{equation}

\smallskip\noindent (ii)
For each $w\in \C\setminus \Gamma$, 
\begin{equation}\label{eq:infty}
\lim_{z\to \infty}S(z,w)=0.
\end{equation}

\smallskip\noindent (iii)
For each $w\in \C\setminus \Gamma$, 
\begin{equation}\label{eq:continuity}
z\mapsto \begin{pmatrix}
\frac{1}{\overline{\rho(z)}}S_{11}(z,w) & \frac{1}{\overline{\rho(z)}}S_{12}(z,w) \\
\rho(z)S_{21}(z,w) & \rho(z)S_{22}(z,w)
\end{pmatrix}
\end{equation}
has a continuous extension to $\C\setminus \{w\}$.
\end{proposition}

\begin{definition} \label{def:diracmass}
  The function $S^\rho_{\mu\chi} = S$ from Proposition~\ref{pr:diracexist} is the Green's function of $\Dirac_\rho+\mu\chi$,
  or in other words the Green's function of $\Dirac+\mu\chi$ with branching structure $\rho$.
\end{definition}

The next proposition provides a short-distance expansion of the massive Green's function, i.e., a representation
in terms of a singular contribution (the massless Green's function) and more regular remainder terms $\Delta$.
We only consider the first column as we will show in Proposition~\ref{prop:Greensym} below that the first column determines the second.

\begin{proposition} \label{prop:Green-Delta}
The Green's function $S^\rho=S^\rho_{\mu\chi}$ has the representation
\begin{align}
  S_{11}^\rho(z,w)&=\mu\frac{\overline{\rho(z)}\rho(w)}{(2\pi)^2}\int_{\C}du\frac{\chi(u)}{|\rho(u)|^2}\frac{1}{\bar z-\bar u}\frac{1}{u-w}+\overline{\rho(z)}\rho(w)\Delta_{11}^\rho(z,w)\label{S11definition-bis},
  \\
S_{21}^\rho(z,w)&=\frac{1}{2\pi}\frac{\rho(w)}{\rho(z)}\frac{1}{z-w}+\frac{\rho(w)}{\rho(z)}\Delta_{21}^\rho(z,w)\label{S21definition-bis}
\end{align}
where $\Delta_{11}^\rho$ and $\Delta_{21}^\rho$ are joint continuous in $\C\times \C$,
depend only on $|\rho|$ (and are thus independent of the branch cut $\Gamma$), decay at infinity, and satisfy
\begin{align}
  2\partial_z \Delta_{11}^\rho(z,w)&=-\mu\chi(z)\frac{1}{|\rho(z)|^2}\Delta_{21}^\rho(z,w) \label{e:Deltadefinition1-bis}
                                \\
2\bar \partial_z \Delta_{21}^\rho(z,w)&=-\mu^2 \chi(z)\frac{|\rho(z)|^2}{(2\pi)^2}\int_{\C}du\, \frac{\chi(u)}{|\rho(u)|^2}\frac{1}{\bar z-\bar u}\frac{1}{u-w}-\mu \chi(z)|\rho(z)|^2 \Delta_{11}^\rho(z,w). \label{e:Deltadefinition2-bis}
\end{align}
\end{proposition}

Another important ingredient that we need is analyticity of $S_{\chi \mu}$ as a function of $\mu$,
stated in the next proposition.
The condition that $\chi$ has compact support is important for analyticity.

\begin{proposition}\label{pr:analyt}
There exists a $\chi$-dependent complex neighborhood $I$
of the real axis into which $\mu\mapsto S^\rho_{\mu\chi}(z,w)$ has an analytic continuation (pointwise for each $z,w\in \mathrm{supp}(\chi)\setminus \Gamma$ and $z\neq w$).
The complex neighborhood does not depend on $z,w$, and it is uniform on compact subsets of $x_1,\dots,x_n$ distinct (and
is permitted to depend on the $\alpha_i$).
This continuation satisfies,
for some constant $C = C(U,\chi,\mu)>0$ where $U$ is any bounded subset of $\C\setminus \Gamma$ and $z,w\in U$, $z\neq w$,
\begin{equation}
  |S^\rho_{\mu\chi}(z,w)|\leq \frac{C}{|\tilde{\rho}(z)\tilde{\rho}(w)||z-w|}\label{tempineq12-bis},
\end{equation}
where $\tilde\rho$ is defined as $\rho$ with each $\alpha_j$ replaced by $|\alpha_j|$:
\begin{equation}\label{eq:tilderho}
\tilde\rho(z)=\prod_{j=1}^n (z-x_j)^{|\alpha_j|}.
\end{equation}
Moreover, there exists a $\chi$-dependent neighborhood of $\mu=0$ in which for each $z\neq w$ and $z,w\notin \Gamma$, 
\begin{equation}\label{eq:Spert}
S_{\mu\chi}^\rho(z,w)=\sum_{n=0}^\infty (-1)^n \mu^n \int_{\C^n}du\,  S_0^\rho(z,u_1)S_0^\rho(u_1,u_2) \cdots S_0^\rho(u_{n-1},u_n)S_0^\rho(u_n,w)\prod_{j=1}^n \chi(u_j),
\end{equation}
where $S_0^\rho$ is the massless Green's function \eqref{eq:S0}.
\end{proposition}

The next proposition gives a factorization identity for derivatives of the Green's function with respect to the branch points.
\begin{proposition} \label{prop:factorization}
For $(z,w)\in (\C\setminus\Gamma)^2$ with $z\neq w$ and $k=1,\dots,n$ the Green's function $S^\rho = S^\rho_{\mu\chi}$
is differentiable in the distinct branch points $x_1,\dots,x_n$ and
\begin{equation}\label{derivfactorisation}
\frac{1}{2\pi}\partial_{x_k} S^\rho(z,w)=\alpha_k \lim_{z',w'\rightarrow x_k}\frac{\rho(z')}{\rho(w')}\begin{pmatrix} S_{11}^\rho(z,w')\\ S_{21}^\rho(z,w')\end{pmatrix}\otimes  \begin{pmatrix} S_{21}^\rho(z',w) & S_{22}^\rho(z',w)\end{pmatrix}
\end{equation}
and
\begin{equation}
\frac{1}{2\pi}\bar\partial_{x_k}S^\rho(z,w)=-\alpha_k \lim_{z',w'\rightarrow x_k}\frac{\overline{\rho(w')}}{\overline{\rho(z')}}\begin{pmatrix} S_{12}^\rho(z,w')\\ S_{22}^\rho(z,w')\end{pmatrix}\otimes  \begin{pmatrix} S_{11}^\rho(z',w) & S_{12}^\rho(z',w)\end{pmatrix}.
\end{equation}
\end{proposition}

The next proposition summarizes various symmetries of the Green's function.

\begin{proposition} \label{prop:Greensym}
The following symmetries hold for $\mu \in \R$ and $z,w \in \C$:
\begin{equation}
  S_{\mu\chi,22}^\rho(z,w)=\overline{S_{\mu\chi,11}^{1/\rho}(z,w)},
  \qquad S_{\mu\chi,12}^\rho(z,w)=\overline{S_{\mu\chi,21}^{1/\rho}(z,w)}
  .
  \label{Greensym}
\end{equation}
Moreover, with $A^{*}$ denoting the conjugate transpose of a matrix $A$, 
\begin{equation} \label{symmetriesinmass0}
S^\rho_{\mu\chi}(z,w)^{*}=-S^\rho_{-\mu\chi}(w,z),
\end{equation}
and  $S_{\mu\chi}^\rho$ is an odd function of $\mu$ on its diagonal and an even function of $\mu$ on the off-diagonal:
\begin{alignat}{2}
S_{\mu\chi,11}^\rho(z,w)&=-S_{-\mu\chi,11}^\rho(z,w), &\qquad S_{\mu\chi,22}^\rho(z,w)&=-S_{-\mu\chi,22}^\rho(z,w),\label{symmetriesinmass1}\\
 S_{\mu\chi,21}^\rho(z,w)&=S_{-\mu\chi,21}^\rho(z,w), &\qquad S_{\mu\chi,12}^\rho(z,w)&=S_{-\mu\chi,12}^\rho(z,w).\label{symmetriesinmass2}
\end{alignat}
In particular,
$\overline{S_{21}^\rho(z,w)}=-S_{12}^\rho(w,z)$,   $\overline{S_{ii}^\rho(z,w)}=S_{ii}^\rho(w,z)$,
and
$\overline{\Delta_{ii}^\rho(z,w)}=\Delta_{ii}^\rho(w,z)$ and therefore $\Delta^\rho_{ii}(z,z)$ is real.
\end{proposition}

\begin{proof}
To see \eqref{Greensym}, note that the first column of \eqref{eq:dirac} states that the solution should satisfy 
\begin{align}
2\bar \partial_z S_{21}(z,w)&=-\mu \chi(z)S_{11}(z,w)+\delta(z-w)\\
2\partial_z S_{11}(z,w)&=-\mu \chi(z)S_{21}(z,w),
\end{align}
whereas the second column of \eqref{eq:dirac} reads, after taking complex conjugates,
\begin{align}
2\bar\partial_z \overline{S_{12}}(z,w)&=-\mu\chi(z)\overline{S_{22}}(z,w)+\delta(z-w)\\
2\partial_z \overline{S_{22}}(z,w)&=-\mu\chi(z)\overline{S_{12}}(z,w).
\end{align}
The assertion thus follows from uniqueness after checking the branching \eqref{eq:continuity}.

To see \eqref{symmetriesinmass0} note that the symmetry $S^\rho_0(z,w)^{*}=-S^\rho_0(w,z)$ extends via \eqref{eq:Spert}  to $S_{\mu \chi}^\rho(z,w)^*=-S_{-\bar\mu \chi}^\rho(w,z)$ in a small $\mu$-neighbourhood of the origin. Both sides are antiholomorphic in a complex neighbourhood of the real axis so this symmetry extends to the full real line.
Since $S_0^\rho$ is zero on the diagonal, the expansion \eqref{eq:Spert} also implies that $S_{\mu\chi}^\rho$ is an odd function of $\mu$ on its diagonal, and an even function of $\mu$ on the off-diagonal, i.e., \eqref{symmetriesinmass1} and \eqref{symmetriesinmass2} hold.
Indeed, these hold in a neighborhood of $0$ and then extend to a neighborhood of $\R$ by analyticity.

The symmetries \eqref{symmetriesinmass0}, \eqref{symmetriesinmass1}, \eqref{symmetriesinmass2} imply
$\overline{S_{21}^\rho(z,w)}=-S_{12}^\rho(w,z)$ and   $\overline{S_{ii}^\rho(z,w)}=S_{ii}^\rho(w,z)$.
Using  \eqref{S11definition-bis}, in particular  $\overline{\Delta_{ii}^\rho(z,w)}=\Delta_{ii}^\rho(w,z)$ and $\Delta_{ii}^\rho(z,z)$ is real.
\end{proof}

As a final ingredient that we need,
we introduce the infinite-volume version $S_{\mu}^\rho$ of $S_{\mu \chi}^\rho$
corresponding to the limit $\chi \to 1$.
The following existence and uniqueness statement
is proven in \cite{MR1233355}, which we have brought forward in our notation.
The translation to our notation is explained in detail in Section~\ref{sec:Palmer-notation} below.

\begin{proposition}
\label{pr:ivlim}
For $\mu \in\R$,
there is a unique $2\times 2$ matrix valued function $S_\mu^\rho(z,w)$ %
such that for all $g\in C_c^\infty(\C\setminus \Gamma; \C^2)$ the function
\begin{equation}
  f(z)=\int S_\mu^\rho(z,w)g(w) \, dw
\end{equation}
satisfies
\begin{equation}
\begin{pmatrix} \frac{1}{\overline{\rho(z)}} & 0 \\ 0 & \rho(z)\end{pmatrix}f(z)\in H^1(\C;\C^2)
\end{equation}
and
\begin{equation}
(\slashed\partial_\rho+\mu)f(z)=g(z).
\end{equation}
\end{proposition}

\begin{definition} \label{def:diracmass-infvol}
  The function $S^\rho_{\mu}$ from Proposition~\ref{pr:ivlim} is the Green's function of $\Dirac_\rho+\mu$,
  or in other words the Green's function of $\Dirac+\mu$ with branching structure $\rho$.
\end{definition}

For convergence of $S_{\mu\chi}^\rho$ to $S_\mu^\rho$, as $\chi\to 1$, we find it convenient to focus on the case where $\chi$ is an indicator function of the disk (and the radius of disk tends to infinity):
\begin{equation} \label{e:chi1Lambda}
  \chi = \1_{\Lambda_L}, \qquad \Lambda_L= B_L(0).
\end{equation}

\begin{proposition}\label{finiteGreentoinfGreen}
Let $\chi=\1_{\Lambda}$ denote the indicator function of the disk as in \eqref{e:chi1Lambda}, and let $\mu>0$.
Then
\begin{align}
S_{\mu\chi}^\rho(z,w)\rightarrow S_\mu^\rho(z,w)
\end{align}
in the limit $L \rightarrow \infty$, for fixed $z,w$ in $\C\setminus \Gamma$ with $z\neq w$.

Moreover, the symmetries \eqref{Greensym} hold for $\chi = 1$.
\end{proposition}

Since we need a slightly stronger version of the proposition in Section~\ref{sec:Palmer} (Proposition~\ref{diffatbpts})
we delay the proof of Proposition~\ref{finiteGreentoinfGreen} until the end of Section~\ref{ResolventIdentitysection}.
It follows from the resolvent identity
given in Proposition~\ref{prop-resolventidentity}, together with the exponential decay of $S_\mu^\rho$
and volume-dependent bounds on $S_{\mu\chi}^\rho$
provided by the next proposition.

\begin{proposition} \label{prop:Ldepenbounds}
Let $\chi=\1_{\Lambda_L}$ denote the indicator function of the disk as in \eqref{e:chi1Lambda}, and let $\mu \in \R$, $\mu \neq 0$.
Uniformly on compact subsets of $x_1,\dots, x_n \in \C$ distinct and of $w\in \C$,
\begin{equation} \label{e:Ldepenbounds}
\int_{\Lambda_{L}^c}\big(|\Delta_{11}^\rho(z',w)|+|\Delta_{21}^\rho(z',w)|\big)e^{-|\mu||z'|}\, dz' = o(1/L)
\end{equation}
in the limit $L\rightarrow \infty$.
\end{proposition}

\subsection{Existence and uniqueness of the solution -- proof of Propositions~\ref{pr:diracexist}
and~\ref{prop:Green-Delta}}

In this section, we will show existence and uniqueness of the finite-volume Green's function.
Before turning to the proof of existence, we record some basic facts about the Cauchy and Beurling transforms, as they will play a role in our proof.
Let $T$ be the Cauchy transform
\begin{equation} \label{Cauchytrans}
(T\varphi)(z)=\frac{1}{\pi}\int_{\C}du\, \frac{\varphi(u)}{z-u},
\end{equation}
and let $\mathcal{S}$ be the Beurling transform
\begin{equation}   \label{Beurlingtrans}
  (\mathcal{S}\varphi)(z)=-\frac{1}{\pi}\lim_{\varepsilon\rightarrow 0}\int_{|z-u|>\varepsilon}du\, \frac{\varphi(u)}{(z-u)^2} 
  .
\end{equation}
The space of continuous functions vanishing at infinity is denoted $C_0(\C)$. (In the conventions of Section~\ref{sec:notation},
it coincides with $C^0(\C)$.)

\begin{lemma}\label{le:cauchy}
The Cauchy transform $T$ and the Beurling transform $\mathcal{S}$ have the following properties.
\begin{itemize}
\item For $\varphi\in L^2(\C)$, the weak derivative $\bar \partial (T\varphi)$ exists and equals $\varphi$.
\item Let $1<q<2<p<\infty$ satisfy $\frac{1}{p}+\frac{1}{q}=1$. Then $T:L^p(\C)\cap L^q(\C)\to C_0(\C)$ is a bounded linear mapping and 
\begin{equation}\label{eq:cauchy1}
\|T\varphi\|_\infty\leq \frac{1}{\sqrt{2-q}}(\|\varphi\|_p+\|\varphi\|_q).
\end{equation}
\item For $p>2$, $\alpha=1-\frac{2}{p}$, and $\varphi\in L^p(\C)$, 
\begin{equation}\label{eq:cauchy2}
\sup_{z\neq w}\frac{|(T\varphi)(z)-(T\varphi)(w)|}{|z-w|^\alpha}\leq \frac{12p^2}{p-2}\|\varphi\|_p.
\end{equation}
\item  For all $1<p<2$,  $T: L^p(\C)\rightarrow L^{2p/(2-p)}(\C)$ is bounded. Moreover,
\begin{equation}
  \|T \varphi\|_{\frac{2p}{2-p}}\leq \frac{C}{(p-1)(2-p)}\|\varphi\|_p 
\end{equation}
for some absolute constant $C>0$.
\item For $1<p<\infty$, the Beurling transform $\mathcal{S}$ defines a bounded linear mapping on $L^p(\C)$ with
\begin{equation}
\|\mathcal{S}\varphi\|_p\leq C_p\|\varphi\|_p,
\end{equation}
where $C_2=1$. %
\item For $\varphi \in L^2(\C)$, we have
\begin{align}
\partial (T\varphi) =\mathcal{S}\varphi
\end{align}
in the sense of weak derivatives.
\end{itemize}
\end{lemma} 

\begin{proof}
See \cite[Theorem~4.0.10, Theorem~4.3.8, Theorem~4.3.10, Theorem~4.3.11, Theorem~4.3.13, Theorem~4.5.3]{MR2472875} and the discussion at the beginning of 
 \cite[Chapter~4]{MR2472875}.
\end{proof}

With these tools, we can prove the existence of the Green's function. We will focus on the first column; the second column can be treated analogously.
Our first remark is uniqueness of the Green's function.

\begin{lemma}\label{le:homogdirac}
Let $\re\mu\neq 0$.
Let the function $u=\begin{pmatrix}
u_1\\
u_2
\end{pmatrix}$  satisfy:
\begin{itemize} 
\item For $z\in \C\setminus \Gamma$,
\begin{equation}
\begin{pmatrix}
\mu \chi(z) & 2\bar \partial_z \\
2 \partial_z & \mu \chi(z)
\end{pmatrix}\begin{pmatrix}
u_1(z)\\
u_2(z)
\end{pmatrix}=0,
\end{equation}
in the weak sense.
\item $\lim_{z\to \infty}u(z)=0$.
\item The functions $z\mapsto\overline{\rho(z)}^{-1}u_1(z)$ and $z\mapsto \rho(z) u_2(z)$ have continuous extensions to $\C$.
\end{itemize}
Then $u=0$. In particular the Green's function is unique.

\end{lemma}

\begin{proof}
Let us begin by deducing that actually, $\rho(z)u_2(z), \overline{\rho(z)}^{-1}u_1(z)\in W_{\rm loc}^{1,p}(\C)$, for some $p>2$. Indeed, we can write
\begin{align} 
\bar\partial( \rho(z)u_2(z))=-\mu\chi(z)|\rho(z)|^2\overline{\rho(z)}^{-1}u_1(z),\label{tempw5et|}\\
\partial( \overline{\rho(z)}^{-1} u_1(z))=-\mu\chi(z)|\rho(z)|^{-2}\rho(z)u_2(z)
\end{align} 
off of $\Gamma$, in the weak sense. We extend the above equations to all of $\C$, by approximating an arbitrary test function in $C_c^\infty(\C)$ by functions in $C_c^\infty(\C\setminus \Gamma)$ and using continuity of the functions on the left-hand sides and the dominated convergence theorem on the right-hand sides. The first equation \eqref{tempw5et|} (and the second after taking the complex conjugate) is now of the form $\bar\partial f=g$, weakly on $\C$,  where $f$ is continuous and vanishes at infinity and $g\in L^{r}(\C)$ for some  $r>2$ and has compact support. 

By Lemma~\ref{le:cauchy}, $\bar\partial (f-Tg)=0$ weakly on $\C$ and so, by Weyl's lemma \cite[Proposition~1.9]{MR1419088},
there is an analytic function which is equal to $f-Tg$ a.e. Again using Lemma \ref{le:cauchy}, we see $f, Tg$ are both continuous functions so $f-Tg$ is actually equal to this analytic function pointwise, and since they both vanish at infinity, we see that $f= Tg$ by Liouville's theorem. Using Lemma \ref{le:cauchy}, $\partial f =\mathcal{S}g$ holds in the weak sense, where $\mathcal{S}$ is the Beurling transform (here we actually use that $g\in L^2(\C)$, which is easily checked). Recalling that $\mathcal{S}$ is a bounded linear mapping on $L^r(\C)$, we see $\mathcal{S}g\in L^{r}(\C)$. Hence both weak derivatives $\partial f, \bar\partial f$ lie in $L^{r}(\C)$. Using this it follows that $\rho(z)u_2(z), \overline{\rho(z)}^{-1}u_1(z)\in W_{\rm loc}^{1,p}(\C)$ for $p=r$.

Now notice that $\overline{u_1}u_2$ extends to a continuous function on $\C$. By the product rule for functions in $W_{\rm loc}^{1,p}(\C)$, we have
\begin{align}
2\bar\partial(\overline{u_1(z)}u_2(z))&=- \bar \mu \chi(z)|u_2(z)|^2-\mu\chi(z)|u_1(z)|^2\\
2\partial(u_1(z)\overline{u_2(z)})&=-\mu \chi(z)|u_2(z)|^2-\bar \mu \chi(z)|u_1(z)|^2.
\end{align}
In particular,
\begin{equation}
\mathrm{Re}(\mu)\chi(z)(|u_1(z)|^2+|u_2(z)|^2)=-(\bar \partial(\overline{u_1(z)}u_2(z))+\partial(u_1(z)\overline{u_2(z)})) = -2\mathrm{Re}(\bar\partial(\overline{u_1(z)}u_2(z)).
\end{equation}
Thus by Green's theorem in the form $\int_\Lambda \bar \partial_z f(z)\, dz=\frac{1}{2i}\oint_{\partial \Lambda}f(z)\, dz$ where $\partial\Lambda$ is oriented in the counter-clockwise manner,
if we take $R>0$ so large that $\mathrm{supp}(\chi)\subset B(0,R)$, then 
\begin{align}
&\mathrm{Re} \mu \int_{\C}\chi(z)(|u_1(z)|^2+|u_2(z)|^2)\, dz=i \mathrm{Re}\oint_{|z|=R}\overline{u_1(z)}u_2(z)\, dz.
\end{align}
(Note that if we had not interpreted $2\bar\partial(\overline{u_1}u_2)=-\mu\chi(|u_1|^2+|u_2|^2)$ as holding on $\Gamma$, we would have gotten here also integrals along $\Gamma$, but again by continuity of $\overline{u_1}u_2$, these would have vanished).

Directly from the equation $u$ satisfies, we see that $u_1$ is antiholomorphic in $(\mathrm{supp}(\chi))^\mathsf c\setminus \Gamma$ while $u_2$ is holomorphic here. Thus (again using continuity), $\overline{u_1}u_2$ is holomorphic in $(\mathrm{supp}(\chi))^\mathsf c$. Moreover, it vanishes at least as fast as $z^{-2}$ at infinity, so by Cauchy's integral theorem
\begin{equation}
\oint_{|z|=R}\overline{u_1(z)}u_2(z)\, dz=0.
\end{equation}
Since $\re\mu\neq 0$, we conclude that $u_1$ and $u_2$ are zero almost everywhere in $\mathrm{supp}(\chi)$, and by continuity, they vanish identically in $\mathrm{supp}(\chi)$.
Thus by continuity $u_1,u_2|_{\partial(\mathrm{supp}(\chi))}=0$, and again making use of continuity and analyticity, $u_1,u_2=0$ in all of $\C$ provided that $\mathrm{supp}(\chi)$ is sufficiently smooth which we have imposed in \eqref{e:chi}.

Uniqueness of the Green's function follows immediately from this since if we had two solutions, $S^{(1)}$ and $S^{(2)}$, then their difference $u=S^{(1)}-S^{(2)}$ would satisfy the assumptions of this lemma.
\end{proof}

Uniqueness will also play a key part in our proof of existence of the Green's function.
Again, we focus on the first column. The basic idea of the proof is to get rid of the branching, remove by hand the worst singularities from the Green's function, recast the problem as an integral equation and argue via the Fredholm alternative.

Let us write out explicitly the first column of \eqref{eq:dirac}. It says that the solution should satisfy 
\begin{align}
2\bar \partial_z S_{21}(z,w)&=-\mu \chi(z)S_{11}(z,w)+\delta(z-w)\\
2\partial_z S_{11}(z,w)&=-\mu \chi(z)S_{21}(z,w).
\end{align}
This suggests that the most singular behavior should be in $S_{21}$ and it should be such that $2\bar\partial_z$ acting on it produces the $\delta$-function. Keeping in mind the correct branching structure, this suggests that we should write 
\begin{equation}
S_{21}(z,w)=\frac{1}{2\pi}\frac{\rho(w)}{\rho(z)}\frac{1}{z-w}+\frac{\rho(w)}{\rho(z)}\Delta_{21}(z,w)\label{S21definition}
\end{equation}
for some function $z\mapsto \Delta_{21}(z,w)$ which is continuous in $\C$ (for each fixed $w$) and decays at infinity.
This in turn suggests that we write 
\begin{equation}
S_{11}(z,w)=\mu\frac{\overline{\rho(z)}\rho(w)}{(2\pi)^2}\int_{\C}du\frac{\chi(u)}{|\rho(u)|^2}\frac{1}{\bar z-\bar u}\frac{1}{u-w}+\overline{\rho(z)}\rho(w)\Delta_{11}(z,w)\label{S11definition},
\end{equation}
where $z\mapsto \Delta_{11}(z,w)$ is continuous in $\C$ (for each fixed $w$) and decays at infinity. We have been a bit sneaky in factoring out the $\rho(w)$ as well. Note that our equations now become
\begin{align}
2\bar \partial_z \Delta_{21}(z,w)&=-\mu^2 \chi(z)\frac{|\rho(z)|^2}{(2\pi)^2}\int_{\C}du\frac{\chi(u)}{|\rho(u)|^2}\frac{1}{\bar z-\bar u}\frac{1}{u-w}-\mu \chi(z)|\rho(z)|^2 \Delta_{11}(z,w)\\
2\partial_z \Delta_{11}(z,w)&=-\mu\chi(z)\frac{1}{|\rho(z)|^2}\Delta_{21}(z,w).
\end{align}

Our goal is to solve these equations, but we will be a bit indirect about it. First of all, since we are looking for functions decaying at infinity, Liouville's theorem implies that it should be equivalent to find solutions to the problem

\begin{multline}
  \Delta_{21}(z,w)+\frac{\mu}{2\pi}\int_{\C}du\,  \frac{\chi(u)|\rho(u)|^2\Delta_{11}(u,w)}{z-u}
  \\
  =-\frac{\mu^2}{2\pi}\int_{\C}du\, \frac{\chi(u)|\rho(u)|^2}{z-u}\int_{\C}dv \frac{\chi(v)}{|\rho(v)|^2}\frac{1}{(2\pi)^2}\frac{1}{\bar u-\bar v}\frac{1}{v-w}
\end{multline}
and
\begin{equation}
\Delta_{11}(z,w)+\frac{\mu}{2\pi}\int_{\C}du\, \frac{\chi(u)}{|\rho(u)|^2} \frac{\Delta_{21}(u,w)}{\bar z-\bar u}=0.
\end{equation}
This means that we are trying to find a solution $\varphi\in C_0(\C,\C^2)$ (two component complex valued continuous functions vanishing at infinity) 
\begin{equation} \label{e:IKphi}
(I-K)\varphi=h, 
\end{equation}
where
\begin{equation}
\left(K\begin{pmatrix}
\varphi_1\\
\varphi_2
\end{pmatrix}\right)(z)=-\frac{\mu}{2\pi}\int_\C du\,\chi(u)\begin{pmatrix}
|\rho(u)|^{-2}\frac{\varphi_2(u)}{\bar z-\bar u}\\
|\rho(u)|^2 \frac{\varphi_1(u)}{z-u}
\end{pmatrix}\label{e:Kdefinition}
\end{equation}
and 
\begin{equation}
h(z)=\begin{pmatrix}
0 \\
-\frac{\mu^2}{2\pi}\int_{\C}du\, \frac{\chi(u)|\rho(u)|^2}{z-u}\int_{\C}dv \frac{\chi(v)}{|\rho(v)|^2}\frac{1}{(2\pi)^2}\frac{1}{\bar u-\bar v}\frac{1}{v-w}
\end{pmatrix}.\label{hzwdefinition}
\end{equation}

\begin{lemma}
  $K:C_0(\C,\C^2)\to C_0(\C, \C^2)$ is compact.
\end{lemma}

\begin{proof}
We use a version of the Arzel\`a--Ascoli theorem for $C_0(\C,\C^2)$.
For a bounded subset $B$ in $C_0(\C,\C^2)$, denote $\mathcal{F}=K(B)$. We will show that $\mathcal{F}$ is relatively compact. It is enough
to show that $\mathcal{F}$ is equicontinuous; for each $x\in\C$, $\mathcal{F}(x)=\{(K\varphi)(x): \varphi\in B\}$ is relatively compact; and $\mathcal{F}$ tends to zero uniformly at infinity (i.e., $(K\varphi)(x)\to 0$ as $x\to\infty$ uniformly in $\varphi\in B$).

By definition of $K$, for $\varphi=\begin{pmatrix} \varphi_1\\\varphi_2\end{pmatrix}\in B$,
\begin{equation}
  [K\varphi]_1(z)=T\tilde \varphi_2(z)
\end{equation}
where
\begin{equation}
  \tilde\varphi_2(u):=-\frac{\mu}{2}\chi(u)\frac{\varphi_2(u)}{|\rho(u)|^2}
\end{equation}
and $\tilde\varphi_2\in L^p(\C)$ for $p>2$ small enough (depending on $\{\alpha_j\}$). Hence, with $\beta=1-2/p$, Lemma~\ref{le:cauchy} implies that
\begin{equation}
\sup_{z\neq w}\frac{|[K\varphi]_1(z)-[K\varphi]_1(w)|}{|z-w|^\beta}\leq \frac{12p^2}{p-2}\|\tilde\varphi_2\|_p\leq C\sup_u|\varphi_2(u)|\leq C_B
\end{equation}
where $C_B=C\sup_{\varphi\in B}\|\varphi\|_{C_0}$. A similar argument for $[K\varphi]_2$ shows that $K\varphi$ is uniformly H\"older continuous for some exponent $\beta'>0$. We thus have equicontinuity.

We also have that $\mathcal{F}(x)=\{K\varphi(x):\varphi\in B\}\subset \C^2$ is bounded, hence relatively compact, since $K$ is a bounded operator.

Finally, take $R>0$ large enough that supp$(\chi)\subset B_R(0)$, then for all $|z|>R$
\begin{equation}
|[K\varphi]_1(z)|\leq \frac{C}{|z|-R}\int_{B_R(0)}du\frac{\varphi(u)}{|\rho(u)|^2}\leq \frac{C_B}{|z|-R}\rightarrow 0
\end{equation}
uniformly for $\varphi\in B$, hence $|K\varphi(z)|$ tends to zero uniformly for $\varphi\in B$ as $z\rightarrow \infty$.
\end{proof}

Also it follows readily that $h\in C_0(\C,\C^2)$, where $h$ is defined in \eqref{hzwdefinition} (depending on $w \in \C$).
In fact,  the following is true.

\begin{lemma}\label{hzwcontinuity}
The function $h(z,w)$ defined by the right-hand side of \eqref{hzwdefinition} is jointly continuous in $(z,w)\in \C^2$ and vanishes at infinity.
\end{lemma}

\begin{proof}
Observe that 
\begin{equation}
h_2(z,w) = T[\chi(\cdot)|\rho(\cdot)|^2[\overline{T\varphi_w}](\cdot)](z),\qquad \text{with } \quad\varphi_w(v):=\frac{-\mu^2}{8\pi}\frac{\chi(v)}{|\rho(v)|^2}\frac{1}{\bar v-\bar w}
\end{equation} and where $h_2$ is the second component of $h$. We now prove joint continuity by showing that the arguments of the Cauchy transforms appearing in $h_2$ have sufficient regularity. Notice that the singularities of $1/|\rho|^2$ in $\varphi_w$ correspond to zeros in the factor of $|\rho|^2$ appearing in $h_2$ and vice versa. 

For $j=1,\dots,n$, take small balls $U_j$ centred at $x_j$ with radii less than $\min_{i \neq j}(|x_i-x_j|/3)$. Since $\chi$ has compact support, when $w\notin \{x_j\}$ it's clear that $\varphi_w\in L^p(\C)$, for any $1\leq p<2$,  (just note that the worst singularity in $\varphi_w$ is at $v=w$ which is locally almost $L^2$ integrable). Furthermore, we can see that $\varphi_{w=x_j}\in L^p(U_j)$ for $1\leq p<\frac{2}{2\alpha_j+1}$ and $\varphi_{w=x_j}\in L^p(U_j^c)$ for $1\leq p<2$. 

Denote $U_0=(\cup_j U_j)^c$. We now show that 
\begin{equation}
\chi(\cdot)|\rho(\cdot)|^2[\overline{T\1_{U_j}\varphi_w}](\cdot)\in L^p(\C)\label{temp34rgedv}
\end{equation} for some $p>2$ and all $j\in \{0,\dots,n\}$, all $w\in \C$. 

By the dominated convergence theorem, $[T\1_{U_j}\varphi_w](z)$ is continuous everywhere except at $z=x_j$, so $\chi(\cdot)|\rho(\cdot)|^2[T\1_{U_j}\varphi_w](\cdot)$ is in $L^p(U_j^c)$, for $2<p<\min\{\frac{-1}{\alpha_i}: i \neq j, \alpha_i<0\}$. Since we already argued that $\varphi_w\in L^p(U_j)$ for $1\leq p<\min(\frac{2}{2+\alpha_j+1},2)$, Lemma \ref{le:cauchy} implies  that $[T\1_{U_j}\varphi_w](\cdot)\in L^{2p/(2-p)}(U_j)$, for $p\in[1,p^*)\cap [1,2)$, and $w\in \C$, where $p^*=\frac{2}{2\alpha_j+1}$. If $\alpha_j\geq 0$ then $\chi(\cdot)|\rho(\cdot)|^2$ is continuous in $U_j$ and so letting $q=2p/(2-p)$, we have $\chi(\cdot)|\rho(\cdot)|^2[T\1_{U_j}\varphi_w](\cdot)\in L^q(U_j)$, $q\in (2,\frac{2p^*}{2-p*})=(2,1/\alpha_j)$, for all $w\in \C$. If instead $\alpha_j<0$, then $\chi(\cdot)|\rho(\cdot)|^2\in L^q(U_j)$, $2<q<-1/\alpha_j$ and $[T\1_{U_j}\varphi_w]\in L^{p}(U_j), 2<p<\infty$. Now using a variant of Hölder's inequality ($\|fg\|_r\leq \|f\|_p\|g\|_q$, where $1/r=1/p+1/q$) we find that $\chi(\cdot)|\rho(\cdot)|^2[\overline{T\chi_{j}\varphi_w}](\cdot)\in L^r(\C)$, $2<r<-1/\alpha_j$ for $\alpha_j<0$. 

Summarizing, we have shown that \eqref{temp34rgedv} holds with $2<p<\min\{\frac{1}{|\alpha_i|},i\in \{1,\dots,n\}\}$ for $j\in \{1,\dots,n\}$, $w\in\C$. The case $j=0$ holds by a similar argument. Now using that $\chi$ has compact support, \eqref{temp34rgedv} holds for $1\leq p<\min\{\frac{1}{|\alpha_i|},i\in \{1,\dots,n\}\}$. Hence by Lemma~\ref{le:cauchy} we have that $h_2(\cdot,w)\in C_0(\C)$ for all $w\in \C$.
We can see the joint continuity of $h(z,w)$ by writing $|h(z,w)-h(z_0,w_0)|\leq |h(z,w)-h(z,w_0)|+|h(z_0,w_0)-h(z,w_0)|$ and taking the limit $(z,w)\rightarrow (z_0,w_0)$,
where the first term goes to zero by continuity of the Cauchy transform and the second by continuity of $h(\cdot,w_0)$. This also shows that $h(z,w)$ vanishes at infinity, since $\varphi_w(v)\rightarrow 0$ in any $L^p(\C)$, $1\leq p <2$ as $w\rightarrow \infty$.
\end{proof}

\begin{proof}[Proof of Propositions~\ref{pr:diracexist} and~\ref{prop:Green-Delta}]
  To deduce the existence of a solution $\varphi$ to \eqref{e:IKphi},
  namely
  \begin{equation}
    \varphi = \begin{pmatrix}
                \Delta_{11}\\
                \Delta_{21}
              \end{pmatrix},
            \end{equation}
we can rely on the Fredholm alternative if we can show that the homogeneous equation $(I-K)\varphi=0$ does not have non-trivial solutions.

Let us assume that a non-trivial solution would exist. We note that it follows from Lemma~\ref{le:cauchy} that $\varphi=\begin{pmatrix}
\varphi_1\\ \varphi_2\end{pmatrix}$ satisfies the differential equation 
\begin{align}
2\partial \varphi_1(z)+\mu\frac{\chi(z)}{|\rho(z)|^2}\varphi_2(z)&=0\\
2\bar \partial\varphi_2(z)+\mu \chi(z)|\rho(z)|^2 \varphi_1(z)&=0.
\end{align}
Define then 
\begin{equation}
u(z)=\begin{pmatrix}
u_1(z)\\
u_2(z)
\end{pmatrix}=\begin{pmatrix}
\overline{\rho(z)}\varphi_1(z)\\
\frac{1}{\rho(z)}\varphi_2(z)
\end{pmatrix}.
\end{equation}
Now off of $\Gamma$, the differential equation for $\varphi$ becomes 
\begin{align}
2\partial u_1(z)+\mu \chi(z)u_2(z)=0\\
2\bar \partial u_2(z)+\mu \chi(z)u_1(z)=0.
\end{align}
Since $\varphi\in C_0(\C,\C^2)$ and $\lim_{z\to\infty}\rho(z)=1$, we are exactly in the setting of Lemma \ref{le:homogdirac}. This means that $u=0$ and $\varphi=0$. Thus by the Fredholm alternative, $(I-K)^{-1}$ exists and is bounded on $C_0(\C, \C^2)$. This means that we have found $\Delta_{11},\Delta_{21}$ as elements of $C_0(\C)$, which in turn means that we have constructed $S_{11},S_{21}$. This proves Proposition \ref{pr:diracexist}.

Let $\Delta(z,w)=(\Delta_{11}(z,w),\Delta_{21}(z,w))$ so that 
\begin{equation}
\Delta(z,w)=\big(I-K\big)^{-1}\big [z'\mapsto h(z',w)](z),\label{temp5265}
\end{equation}
where we have included the dependence on $w$ of the function $h$ appearing in \eqref{hzwdefinition}. We now show that $\Delta:\C^2\rightarrow \C^2$ is jointly continuous. By the triangle inequality,
\begin{align}
|\Delta(z,w)-\Delta(z_0,w_0)|&\leq |\big(I-K\big)^{-1}(h(\cdot,w)-h(\cdot,w_0))(z)|\label{tempg45}
\\&\quad+|\big(I-K\big)^{-1}(h(\cdot,w_0))(z)-\big(I-K\big)^{-1}(h(\cdot,w_0))(z_0)|.\nonumber
\end{align}
The first term on the right-hand side of the previous inequality is bounded above by 
\begin{equation}
\sup_{z'}|\big(I-K\big)^{-1}(h(\cdot,w)-h(\cdot,w_0))(z')|\leq C\sup_{z'}|h(z',w)-h(z',w_0)|\label{tempfr42}
\end{equation}
for some $C>0$, since $(I-K)^{-1}$ is a bounded linear map. Taking the limit $w\rightarrow w_0\in\C$,
we can see that \eqref{tempfr42} tends to zero since $h\in C_0(\C^2,\C^2)$.
Indeed, if it did not, there would be a sequence $(z_n,w_n)$ with $w_n \to w_0$ such that $|h(z_n,w_n)-h(z_n,w_0)|$ is bounded below by a positive constant.
Thus if $\{z_n\}$ is unbounded then we have a contradiction because $h$ vanishes at infinity,
whereas if $\{z_n\}$ is bounded we could take a convergent subsequence and again get a contradiction (by continuity). Therefore  \eqref{tempfr42} must tend to zero.
The second term in \eqref{tempg45} tends to zero just by continuity in $z$. This proves Proposition~\ref{prop:Green-Delta}.
\end{proof}

\subsection{Short-distance bounds on the Green's function -- proof of Proposition~\ref{pr:analyt} part (i)}

\begin{proposition}\label{pr:bounds22}
Let $S= S^\rho_{\mu\chi}$ be the Green's function from Proposition \ref{pr:diracexist}.

\smallskip\noindent (i)
$S(z,w)$ is jointly continuous on the set $(\C\setminus \Gamma)^2\setminus \{z=w\}$.

\smallskip\noindent (ii)
For any bounded $U \subset (\C\setminus \Gamma)^2\setminus\{z=w\}$ there is a constant $C=C(U,\chi,\mu)>0$ such that
\begin{equation}
|S(z,w)|\leq \frac{C}{|\tilde{\rho}(z)\tilde{\rho}(w)||z-w|}\label{tempineq12}
\end{equation}
holds uniformly for $(z,w)\in U$, where the matrix norm is the Frobenius norm and $\tilde{\rho}$ is defined as $\rho$ with each $\alpha_j$ replaced by $|\alpha_j|$.
\end{proposition}

\begin{proof}
For the item (i), consider the first column of $S$ given by equations \eqref{S21definition}, \eqref{S11definition}.
One can see that the claim follows from the fact that $\Delta_{11}, \Delta_{21}$ appearing there are jointly continuous.

For the item (ii), it suffices to show that $|\tilde{\rho}(z)\tilde{\rho}(w)||z-w||S_{i,j}(z,w)|$ is bounded on $U$, for $i,j\in\{1,2\}$. For $i=2$, $j=1$, using \eqref{S21definition}, this follows by the triangle inequality and joint continuity of $\Delta_{21}(z,w)$. For $i=j=1$ this follows by joint continuity of $\Delta_{11}(z,w)$ and the bound \eqref{upperbound3edf} of the integral appearing in \eqref{S11definition} given in the following lemma.
\end{proof}

\begin{lemma}\label{finerbound}
Introduce for $z\neq w$,
\begin{align}
C_{x,\alpha}(z,w)=\begin{cases} \frac{1}{|z-w|^{2\alpha}}, \quad \alpha>0 \text{ and }  (z,w)\in B_\epsilon(x)^2\\
1+|\log|z-w||, \quad \alpha\leq 0 \text{ or } (z,w)\notin B_\epsilon(x)^2 
\end{cases}\label{Cxalpha}
\end{align}
where $\epsilon>0$ is fixed less than $\min\{\frac12 |x_i-x_j|;\; i\neq j\}$.
Then
\begin{align}
|\int du\frac{\chi(u)}{|\rho(u)|^2}\frac{1}{\bar z-\bar u}\frac{1}{u-w}|\leq C\sum_{j=0}^n C_{x_j,\alpha_j}(z,w)\label{upperbound3edf}
\end{align}
where $C_{x_0,\alpha_0}(z,w):= 1+|\log|z-w||$ and $C>0$ is a $\chi$- and $\rho$-dependent constant.
\end{lemma}

\begin{proof}
To see that this upper bound holds, observe that, for $0<\alpha<1/2$,
\begin{align}
&|\int_{B_1(0)} du\, \frac{1}{|u|^{2\alpha}}\frac{1}{\bar z-\bar u}\frac{1}{u-w}|\nnb
&=|\int_{B_1(-w)} du\, \frac{1}{|u+w|^{2\alpha}}\frac{1}{\bar z-\bar w-\bar u}\frac{1}{u}|\nnb
&\leq \frac{1}{|z-w|^{2\alpha}} |\int_{B_{1/|z-w|}(-w)} du\, |z-w|\frac{1}{|u+\frac{w}{|z-w|}|^{2\alpha}}\frac{1}{\bar z-\bar w-\bar u |z-w|}\frac{1}{u}|\nnb
&\leq \frac{1}{|z-w|^{2\alpha}}\int_\C  du\, \frac{1}{|u+w'|^{2\alpha}}\frac{1}{|u-p|}\frac{1}{|u|},
\end{align}
where $|p|=|\frac{\bar z-\bar w}{|z-w|}|=1$ and $w'=w/|z-w|$.
This last integral is bounded for $w'\in \C$ and $|p|=1$.
For the case $\alpha\leq 0$ or $(z,w)\notin B_{\epsilon}(x_j)^2$, because $|\rho(u)|^{-2}$ is bounded in a neighborhood of $x_j$ whenever $\alpha\le 0$, the integrand is
dominated (up to a constant that depends on~$\chi$) by $|z-u|^{-1}|u-w|^{-1}$. Changing variables $u \rightarrow u-w$ and using polar coordinates gives the logarithmic upper bound.
Taking a partition of unity subordinate to small balls covering the branch points $\{x_j\}$, as in the proof of Lemma~\ref{hzwcontinuity}, it is straightforward from here to show \eqref{upperbound3edf}.
\end{proof}

\subsection{Analyticity -- proof of Proposition~\ref{pr:analyt} part (ii)}

In this subsection we prove that the finite-volume Green's function $S(z,w)$, as a function of $\mu$,
has an analytic extension into a neighborhood $I$ of the real line in the complex plane.
Introduce the notation
\begin{align} \label{e:Scdef}
  S_c(z,w)
  &= \overline{P_\rho}(z)^{-1} S^\rho_{\mu\chi}(z,w) P_\rho(z)
    \nnb
  &=\begin{pmatrix} \frac{\mu}{(2\pi)^2}\int du \frac{\chi(u)}{|\rho(u)|^2}\frac{1}{\bar z-\bar u}\frac{1}{u-w}+\Delta_{11}^\rho(z,w) &\frac{1}{2\pi}\frac{1}{\bar z-\bar w}+\overline{\Delta_{21}^{1/\rho}}(z,w)\\ \frac{1}{2\pi}\frac{1}{ z-w}+\Delta_{21}^\rho(z,w) &  \frac{\mu}{(2\pi)^2}\int du \chi(u)|\rho(u)|^2\frac{1}{ z- u}\frac{1}{\bar u-\bar w}+\overline{\Delta^{1/\rho}_{11}}(z,w) 
\end{pmatrix}
\end{align}
where the second equality follows from Proposition~\ref{prop:Green-Delta} and Proposition~\ref{prop:Greensym},
and $S_c$ is continuous in $(z,w)\in \C^2\setminus\{z=w\}$ (thus across the branch cuts of $\rho$).
To be precise, while the proof of Proposition~\ref{prop:Greensym} uses analyticity which we now wish to prove,
we only used  \eqref{Greensym} whose proof does not rely on analyticity -- so the argument is not circular.
We also need the following lemma.

\begin{lemma} Consider\label{lemmatemp42}
\begin{equation}
\int_\C du\, \chi(u) |P(z)|^{2}S_c(z,u)|P(u)|^{2}S_c(u,w)\label{temp234nf}
\end{equation}
where here $|A|$ denotes the element-wise absolute value of a matrix $A$.
Then \eqref{temp234nf} is continuous in $(z,w)\in ( (\C\setminus \{x_j\})\times \C)\setminus\{z=w\}$, and, as a function of $z$, in $L^q_{\loc}(\C)$ for some $q>2$.
\end{lemma}
\begin{proof}
We expand the matrix multiplications in \eqref{temp234nf} and consider the element (1,1), this reads
\begin{align}
&\int du \chi(u)\Big[|\rho(z)|^2\frac{1}{2\pi}\frac{1}{\bar z-\bar u}|\rho(u)|^{-2}\frac{1}{2\pi}\frac{1}{u-w}\label{temp43rfpn} \\
&+\frac{\mu}{(2\pi)^2}\int du_1\frac{|\rho(z)|^2}{|\rho(u_1)|^2}\chi(u_1)\frac{1}{\bar z-\bar u_1}\frac{1}{u_1-u}\frac{\mu}{(2\pi)^2}\int du_2\frac{|\rho(u)|^2}{|\rho(u_2)|^2}\chi(u_2)\frac{1}{\bar u-\bar u_2}\frac{1}{u_2-w}\Big]\label{temprefdrnr}
\end{align}
plus terms involving $\Delta_{11}, \Delta_{21}$ which are integrable in $u$ and thus continuous for all $(z,w)\in \C^2$. By essentially the same argument as that which appears in Lemma~\ref{hzwcontinuity}, we see the integral in \eqref{temp43rfpn} is $L^q_{\loc}(\C)$ for $2<q<\min_i\{\frac{1}{|\alpha_i|}\}$, and continuity just follows from the
dominated convergence theorem. Next observe that by Lemma \ref{finerbound}, the integrand \eqref{temprefdrnr} is bounded above by
\begin{equation}
C^2\sum_{j_1,j_2=0}^n |\rho(z)|^2C_{x_{j_1},\alpha_{j_1}}(z,u)|\rho(u)|^2C_{x_{j_2},\alpha_{j_2}}(u,w).
\end{equation}
If $\alpha_{j_2}>0$ then there is a zero in $|\rho(u)|^2\sim |u-x_{j_2}|^{2\alpha_{j_2}}$ at $x_{j_2}$ and by \eqref{Cxalpha} the singularity in $C_{x_{j_1},\alpha_{j_1}}(z,u)|\rho(u)|^2C_{x_{j_2},\alpha_{j_2}}(u,w)$ is integrable in $u$ for any $z,w\in \C$. Similarly, if $\alpha_{j_2}\leq 0$ then the singularity in $C_{x_{j_1},\alpha_{j_1}}(z,u)|\rho(u)|^2C_{x_{j_2},\alpha_{j_2}}(u,w)$ is at worst $\sim |z-u|^{2\alpha_{j_1}}|u-x_{j_2}|^{2\alpha_{j_2}}(1+|\log|u-w|)$ which is locally integrable for any $z,w$, so multiplying by $\chi(u)$ and integrating we see this term is continuous in $z,w$, then multiplying by $|\rho(z)|^2 \in L^q_{\loc}(\C)$ for some $q>2$, we see the lemma holds for the term \eqref{temprefdrnr} and thus holds for the element (1,1) of \eqref{temp234nf}. The remaining elements are similar.
\end{proof}

\begin{proposition}\label{pr:analyt12}
Let $S=S^\rho$ be as in Proposition \ref{pr:diracexist}.
There exists a $\chi$-dependent (complex) neighborhood of the real axis into which $\mu\mapsto S(z,w)$ has an analytic continuation.
The complex neighborhood does not depend on $z,w$, and it is uniform on compact subsets of $x_1,\dots,x_n$ distinct (and
is also permitted to depend on the $\alpha_i$).
This continuation also satisfies \eqref{tempineq12}.
Moreover, there exists a $\chi$-dependent neighborhood $O\subset \C$ of the origin in which for each $z\neq w$ and $z,w\notin \Gamma$, 
\begin{equation}\label{eq:Spert-bis}
S_{\mu\chi}(z,w)=\sum_{n=0}^\infty (-1)^n \mu^n \int_{\C^n}du\,  S_0(z,u_1)S_0(u_1,u_2) \cdots S_0(u_{n-1},u_n)S_0(u_n,w)\prod_{j=1}^n \chi(u_j),
\end{equation}
where $S_0$ is the Green's function of the massless twisted Dirac operator \eqref{eq:S0} and $\mu\in O$.
\end{proposition}

\begin{proof} Note that we can locally integrate over both coordinates of $S$ by Proposition~\ref{pr:bounds22}.
  First, we aim to show that for each $w\notin \Gamma$, the series
\begin{equation}
\tilde{S}_\delta(z,w)=\sum_{n=0}^\infty (-1)^n \delta^n \int_{\C^n}du\, S_{\mu\chi}^\rho(z,u_1)S_{\mu\chi}^\rho(u_1,u_2)\cdots S_{\mu\chi}^\rho(u_{n-1},u_n)S_{\mu\chi}^\rho(u_n,w)\prod_{j=1}^n \chi(u_j)\label{tildeSdelta}
\end{equation}
converges absolutely, for all $\delta\in \C$ such that $|\delta|$ is sufficiently small, independently of $z\notin \Gamma\cup\{w\}$. 
In terms of the definition \eqref{e:Scdef}, we rewrite this sum as
\begin{multline}
  \sum_{n=0}^\infty (-1)^n \delta^n \int_{\C^n}du\, \overline{P}(z)^{-1}S_c(z,u_1)|P(u_1)|^{-2}
  \\\times S_c(u_1,u_2)\cdots S_c(u_{n-1},u_n)|P(u_n)|^{-2}S_c(u_n,w)P(w)^{-1}\prod_{j=1}^n \chi(u_j)
  .
\end{multline}
Consider the last two integrals of each term $n\geq 2$:
\begin{equation}
  \int_{\C}d u_{n-1}\int_{\C}du_{n}\,  S_c(u_{n-2},u_{n-1})|P(u_{n-1})|^{ -2}S_c(u_{n-1},u_{n})|P(u_{n})|^{-2}S_c(u_{n},w)\chi(u_{n-1})\chi(u_n)
  .
  \label{tempgr3}
\end{equation}
By Lemma~\ref{lemmatemp42}, Lemma \ref{le:cauchy}, and H\"older's inequality $\|fg\|_r\leq \|f\|_p\|g\|_q$, where $1/r=1/p+1/q$,
we see \eqref{tempgr3} is continuous in the variables $u_{n-2},w\in \C$. Hence, it is bounded above by some constant $C_1$ uniformly for $u_{n-2}\in \text{supp}(\chi)$ and, since $S_c$  also vanishes at infinity in its second coordinate, uniformly for all $w\in \C$. 

The modulus of each element of the matrix $\tilde{S}_\delta$ is bounded above by the Frobenius norm of $\tilde{S}_\delta$. Hence by submultiplicativity, the second statement of Proposition \ref{pr:bounds22} and our bound on the last two integrals of the terms $n \geq2$, it suffices to show the series
\begin{equation}
\sum_{n=0}^\infty  C_1C^{n+1}\delta^n \int_{\text{supp}(\chi)^{n-2}}du\, \frac{1}{|z-u_1|}\prod_{j=1}^{n-3}\frac{1}{|\tilde\rho(u_j)|^2|u_j-u_{j+1}|}\frac{1}{|\tilde \rho(u_{n-2})|^2|u_{n-3}-u_{n-2}|}\label{temp45gd}
\end{equation}
converges uniformly for all $z\in \C$.
Now we use H\"older's inequality
\begin{equation}
  \sup_{z\in \C} \int_{\mathrm{supp}(\chi)} du \frac{1}{|\tilde \rho(u)|^2}\frac{1}{|z-u|}\leq \|\tilde\rho^{-2}\mathrm{1}_{\mathrm{supp}(\chi)}\|_p\sup_{z}\||z-\cdot|^{-1}\mathrm{1}_{\mathrm{supp}(\chi)}\|_q<\infty,
\end{equation}
with $1/p+1/q=1$ and $p$ just barely greater than 2 depending on the $\alpha_i \in (-\frac12,\frac12)$,
iteratively on each of the integrals over $u_{n-j}$, $j=2,\dots,n-2$ (ignoring the finite $u_1$ integral).
We see that each integral is bounded above \emph{by the same constant} $C_2>0$
and thus that the sum in \eqref{temp45gd} converges absolutely for all $\delta\in \C$ with $|\delta|< \epsilon$ where $ \epsilon<1/C_2C$ and this radius of convergence does not depend on $z$ or $w$.
Note that it depends on the windings $\{\alpha_i\}$ and branch points $\{x_i\}$, but is uniform on compact subsets of $x_1,\dots,x_n$ distinct.
 
Next we show that $\tilde{S}_\delta$, $\delta\in \R$ such that $|\delta|\leq \epsilon$,  satisfies the conditions of Proposition \ref{pr:diracexist} and hence equals $S_{(\mu+\delta)\chi}$.
The third condition \eqref{eq:continuity} follows by multiplying $\tilde{S}_\delta$
by the diagonal matrix $\overline{P}(z)^{-1} = \text{diag}(1/\overline{\rho(z)},\rho(z)$) on the left,
using the same condition for $S_{\mu\chi}$ and noting that we can swap limits and summation by the previous argument.
The second condition \eqref{eq:infty} follows by inspection on the terms $n=0,1,2$ in \eqref{tildeSdelta} and observing that the sum for $n>2$ is bounded above by
\begin{align}
C_4\int_{\text{supp}(\chi)}du_1\frac{1}{|\tilde{\rho}(z)||z-u_1||\tilde{\rho}(u_1)|^2}
\end{align}
for some $C_4>0$ by the previous argument, which tends to zero as $z\rightarrow \infty$.
Finally, we check the first condition \eqref{eq:dirac},
\begin{align}
  &([\Dirac + (\mu+\delta)\chi(\cdot) ]\tilde{S}_\delta(\cdot,w),\phi)\nnb
  &= ([\Dirac + \mu\chi(\cdot) ]\tilde{S}_\delta(\cdot,w),\phi)+([\delta\chi(\cdot)]\tilde S_\delta(\cdot,w),\phi) \nnb
  &=\phi(w)+([\delta\chi(\cdot)]\tilde S_\delta(\cdot,w),\phi)\nnb
  &\quad+(\sum_{n=1}^\infty \delta^n(-1)^n\chi(\cdot)\int_{\C^{n-1}}du\, S_{\mu\chi}^\rho(\cdot,u_1)\cdots S_{\mu\chi}^\rho(u_{n-1},w)\prod_{j=1}^{n-1}\chi(u_j),\phi)
\end{align}
by \eqref{eq:dirac}, where $\phi\in C_c^\infty(\C\setminus \Gamma)$, and derivatives are understood as weak derivatives.  
Note we can move the weak derivative inside the integral and sum by the above bounds and Fubini's theorem. Reindexing the last sum, we see these last two terms cancel. This finishes the proof.

Finally, we observe that the analytic continuation also satisfies the bound \eqref{tempineq12}.
For the $n=0$ terms this is clear and the $n\geq 2$ terms have been shown to be uniformly bounded above.
Thus it suffices to consider the $n=1$ term.
It is bounded as claimed by Lemma~\ref{finerbound}.
\end{proof}

\subsection{Factorization identity -- proof of Proposition~\ref{prop:factorization}}
Next we prove that the derivative of our Green's function with respect to one of the branch points satisfies a useful factorisation property.

\begin{proof}[Proof of Proposition~\ref{prop:factorization}]
By the symmetries \eqref{Greensym} it suffices to prove \eqref{derivfactorisation}.
First we check the case that $\mu=0$. A straightforward calculation with \eqref{eq:S0} yields
\begin{equation}
\frac{1}{2\pi}\partial_{x_k}S_0(z,w)=\begin{pmatrix} 0 & 0 \\ \alpha_k\frac{\rho(w)}{\rho(z)}\frac{1}{(2\pi)^2}\frac{1}{(z-x_k)(x_k-w)}&0\end{pmatrix},
\end{equation}
so we see that our formula holds for $\mu=0$. Now assume \eqref{derivfactorisation} holds for some $\mu\in \R$, we will prove that then the same identity holds for $\mu+\delta$ in a ball contained within in the region of analyticity of $S(z,w)$ in Proposition \ref{pr:analyt12}.
By using the series expansion for $S_{(\mu+\delta)\chi}$ as in \eqref{tildeSdelta},
\begin{multline}\label{derivseriesexp}
  \frac{1}{2\pi}\partial_{x_k}S_{(\mu+\delta)\chi}(z,w)
  \\
  =\sum_{n=0}^\infty (-1)^n \delta^n \int_{\C^n}du\,\sum_{j=0}^n S_{\mu\chi}(z,u_1)\cdots \frac{1}{2\pi}\partial_{x_k}S_{\mu\chi}(u_j,u_{j+1})\cdots S_{\mu\chi}(u_n,w)\prod_{j=1}^n \chi(u_j)
\end{multline}
where we identify $u_0=z,u_{n+1}=w$.
The justification for moving the derivative past the sum is a consequence of the following manipulations, which yield absolute convergence of the series (as the product of the two absolutely convergent series \eqref{temp43erdf}) and moving it past the integral is due to integrability of the singularities, which follows by \eqref{tempineq12-bis}, and Leibniz's integral rule. 

Now, recall the identities that for $n\times n$ matrices $A,B$ and column (row) vector $v$ $(w)$ of dimension $n$, we have $A(v\otimes w)= (Av)\otimes w$ and $(v\otimes w)B= v \otimes (wB)$. Using these identities we insert \eqref{derivfactorisation} into the integrand and find, with $S=S_{\mu\chi}$,
\begin{align}
&S(z,u_1)\cdots \frac{1}{2\pi}\partial_{x_k}S(u_j,u_{j+1})\cdots S(u_n,w)\nnb
  &=\alpha_k\lim_{z',w'\rightarrow x_k}\frac{\rho(z')}{\rho(w')} S(z,u_1)\cdots S(u_{j-1},u_j)\begin{pmatrix} S_{11}(u_j,w')\\ S_{21}(u_j,w')\end{pmatrix}\nnb
  &\qquad\otimes  \begin{pmatrix} S_{21}(z',u_{j+1}) & S_{22}(z',u_{j+1})\end{pmatrix}S(u_{j+1},u_{j+2})\cdots S(u_n,w)\nnb
  &=\alpha_k\lim_{z',w'\rightarrow x_k}\frac{\rho(z')}{\rho(w')} \Big(S(z,u_1)\cdots S(u_{j-1},u_j)\begin{pmatrix} S_{11}(u_j,w')\\ S_{21}(u_j,w')\end{pmatrix}\Big)\nnb
  &\qquad\otimes \Big( \begin{pmatrix} S_{21}(z',u_{j+1}) & S_{22}(z',u_{j+1})\end{pmatrix}S(u_{j+1},u_{j+2})\cdots S(u_n,w)\Big)\nnb
  &=\alpha_k\lim_{z',w'\rightarrow x_k}\frac{\rho(z')}{\rho(w')}\begin{pmatrix}[S(z,u_1)\cdots S(u_j,w')]_{11}\\ [S(z,u_1)\cdots S(u_j,w')]_{21}\end{pmatrix}\nnb
  &\qquad\otimes \begin{pmatrix}[S(z',u_{j+1})\cdots S(u_n,w)]_{21}&[S(z',u_{j+1})\cdots S(u_n,w)]_{22}\end{pmatrix}
\end{align}
where the third line follows by just multiplying out the last matrix in each of the matrix chain products on the right-hand side and looking at the appropriate index. 

Next we reinsert the above expression back into the series expansion \eqref{derivseriesexp}, reindex the double sum as $n=i+j$, then relabel the variables $u_{j+1}\rightarrow u_1',\dots, u_{i+j}\rightarrow u_i'$, this yields
\begin{align}
  &\alpha_k\sum_{i=0}^\infty\sum_{j=0}^\infty(-1)^{i+j}\delta^{i+j}\int_{\C^i}du'\int_{\C^j}du\lim_{z',w'\rightarrow x_k}\frac{\rho(z')}{\rho(w')}\begin{pmatrix}[S(z,u_1)\cdots S(u_j,w')]_{11}\\ [S(z,u_1)\cdots S(u_j,w')]_{21}\end{pmatrix}
    \nnb
  &\quad\quad\otimes \begin{pmatrix}[S(z',u_{1}')\cdots S(u_i',w)]_{21}&[S(z',u_{1}')\cdots S(u_i',w)]_{22}\end{pmatrix}\prod_{k=1}^i\chi(u_k')\prod_{k'=1}^j\chi(u_{k'})\label{temp43erdf}
  \nnb
&=\alpha_k \lim_{z',w'\rightarrow x_k}\frac{\rho(z')}{\rho(w')}\begin{pmatrix} S_{(\mu+\delta)\chi,11}(z,w')\\ S_{(\mu+\delta)\chi,21}(z,w')\end{pmatrix}\otimes  \begin{pmatrix} S_{(\mu+\delta)\chi,21}(z',w) & S_{(\mu+\delta)\chi,22}(z',w)\end{pmatrix}
\end{align}
by bilinearity of the tensor product. Note that in the above, the double sum splits as two independent sums, we can then move the limits outside of the integrals by using the fact that the $\rho(z')/\rho(w')$ factors cancel the singularities from the $\rho$ in the Green's function, so using Proposition \ref{prop:Green-Delta}, we get bounds similar to that of \eqref{tempineq12} and use a dominated convergence argument. We can move the limit outside of the integral because the geometric convergence of the series did not depend on $z,w$ in Proposition \ref{pr:analyt12}.

To see that this argument is sufficient to prove \eqref{derivfactorisation} for all $\mu\in \R$, note that for any fixed $\mu$ we can find a finite horizontal strip $\{\mu': \mu_1\leq \text{Re}[\mu']\leq \mu_2, -h<\text{Im}[\mu']<h\}$, where $(\mu_1,\mu_2)$ covers the interval from 0 to $\mu$ and its height $h$ is small enough that the strip is completely contained within the region of analyticity of $S(z,w)$. Place $N+1$ balls of radius $r$ less than $h$ at positions $\frac{k}{N}\mu,$ $k=0,\dots,N$ inside the strip, then fix $N$ sufficiently large so that the centre of each ball lies inside an adjacent ball. The above argument shows that \eqref{derivfactorisation} holds in each of these balls, and thus holds on the real line segment between 0 and $\mu$. Since this works for any $\mu\in \R$ our formula \eqref{derivfactorisation} holds for any $\mu\in\R$.
\end{proof}

\subsection{Volume-dependent bound on $\Delta$ -- proof of Proposition~\ref{prop:Ldepenbounds}}

Heuristically, when $\chi=\1_{\Lambda_L}$, the components of $\Delta$ are holomorphic or antiholomorphic on the exterior of the disk and vanish at infinity,
so their maximum modulus occurs on the boundary of the disk. Hence the integral \eqref{e:Ldepenbounds} is essentially controlled by the behaviour of $\Delta$ at the boundary $\partial \Lambda_L$. By Cauchy-Schwarz (see \eqref{temp24ehbc}), we instead use the $L^2$ norm of $\Delta$ on the boundary of the disk to bound the integral, which, in turn, we show grows at most polynomially in $L$ by an argument akin to the Sobolev trace theorem and analysing the operator which defines $\Delta$ in a suitable weighted $L^2$ space.

We recall that
$\begin{pmatrix}
\Delta_{11}(\cdot,w)\\
\Delta_{21}(\cdot, w)
\end{pmatrix}$
is the unique continuous function vanishing at infinity that satisfies 
\begin{equation}\label{eq:delta}
\left((I-K)\begin{pmatrix}
\Delta_{11}(\cdot,w)\\
             \Delta_{21}(\cdot, w)
\end{pmatrix}\right)(z)= h(z,w)
\end{equation}
for $z\in \C$, where the operator $K$ is defined in \eqref{e:Kdefinition} and the function $h$ in \eqref{hzwdefinition}.
Our goal is to show that this function satisfies
\begin{equation}
\|\Delta(\cdot,w)\|_{L^2(\partial {\Lambda_L})}\leq C(L)\|h(\cdot,w)\|_{W^{1,2}({\Lambda_L})}
\end{equation}
with some constant $C(L)$ that may grow mildly with $L$.

The idea will be to mimic the proof of the Sobolev trace theorem and get suitable Sobolev-type estimates for $\Delta$. We first prove a weighted $L^2$-bound for $\Delta$ inside $\Lambda_L$.

\subsubsection{A weighted $L^2$-bound for $\Delta$}

Let us consider the Hilbert space 
\begin{equation}
\H_\rho=\H_\rho(\Lambda_L)=\left\{\begin{pmatrix}
\varphi_1\\
\varphi_2
\end{pmatrix}: |\rho|\varphi_1,|\rho|^{-1}\varphi_2\in L^2(\Lambda_L)\right\}
\end{equation}
with the inner product
\begin{equation}
\langle \psi,\varphi\rangle_{\rho}=\int_{\Lambda_L}dz\, \pB{ |\rho(z)|^2 \overline{\psi_1(z)}\varphi_1(z)+|\rho(z)|^{-2}\overline{\psi_2(z)}\varphi_2(z) }.
\end{equation}
As a direct sum of two weighted $L^2$-spaces, this is naturally a Hilbert space. Note that since $|\rho|^{\pm 2}$ are integrable, $\Delta$ and $h$, being bounded functions on $\Lambda_L$, are elements of this space. We wish to study the operator $K$ on this space. We will show that it is a bounded skew adjoint operator on this space and use this to argue that
\begin{equation}
  \|\Delta(\cdot,w)\|_\rho\leq \|h(\cdot,w)\|_\rho.
\end{equation}

Let us begin by showing that $K$ is bounded. 
\begin{lemma}\label{le:KonHrhobounded}
$K:\H_\rho\to \H_\rho$ is a bounded operator.  
\end{lemma}
\begin{proof}
Let us focus on the first component $(K\varphi)_1$. To study integrability, we write
\begin{align}
|\rho(z)|(K\varphi)_1(z)&=-\frac{\mu}{2\pi}\int_{\Lambda_L} du \, \left(\frac{|\rho(z)|}{|\rho(u)|}-1\right)\frac{1}{\bar z-\bar u}|\rho(u)|^{-1}\varphi_2(u)\nnb
&\quad -\frac{\mu}{2\pi}\int_{\Lambda_L} du \, \frac{1}{\bar z-\bar u}|\rho(u)|^{-1}\varphi_2(u)\nnb
&=: (T_1\varphi_2)(z)+(T_2\varphi_2)(z).
\end{align}
To prove boundedness (of the first component), it is sufficient for us to show that
\begin{equation}
  \|T_j\varphi_2\|_{L^2(\Lambda_L)}\leq C_j \||\rho|^{-1}\varphi_2\|_{L^2(\Lambda_L)}.
\end{equation}
For $T_2$, this follows immediately from \cite[Theorem 4.3.12]{MR2472875}: it implies that  
\begin{equation}
\|T_2\varphi_2\|_{L^2(\Lambda_L)}\leq 6\mu L \||\rho|^{-1}\varphi_2\|_{L^2(\Lambda_L)}.
\end{equation}
Let us thus focus on $T_1\varphi_2$. We first note that by Cauchy-Schwarz, 
\begin{align}
|(T_1\varphi_2)(z)|\leq \frac{\mu}{2\pi}\left(\int_{\Lambda_L} du \left|\frac{|\rho(z)|}{|\rho(u)|}-1\right|^2 \frac{1}{|z-u|^2}\right)^{1/2}\||\rho|^{-1}\varphi_2\|_{L^2(\Lambda_L)}.
\end{align}
Thus if we can get an integrable (in $z$) bound for the $u$-integral here, we will be done. 

For this purpose, let us choose $\epsilon=\frac{1}{4}(\min_{i\neq j}|x_i-x_j|\wedge \min_i d(x_i, \partial \Lambda_L))$ and write $\Lambda_L=\bigcup_{j=1}^n B(x_j,\epsilon) \cup A_\epsilon$ where $A_\epsilon=\Lambda_L\setminus (\bigcup_{j=1}^n B(x_j,\epsilon))$. Consider first $z\in A_\epsilon$ and consider the integral over $A_\epsilon$. In this regime, the integrand is actually bounded by some constant depending on $L$,$x_i$, and $\alpha_i$. Now if we integrate over $B(x_j,\epsilon)$, we write 
\begin{align}
\left|\frac{|\rho(z)|}{|\rho(u)|}-1\right|^2\frac{1}{|z-u|^2}&=\left|\frac{|\rho(z)|}{|\rho(u)|}-\prod_{k\neq j}\frac{|z-x_k|^{\alpha_k}}{|u-x_k|^{\alpha_k}}+\prod_{k\neq j}\frac{|z-x_k|^{\alpha_k}}{|u-x_k|^{\alpha_k}}-1\right|^2\frac{1}{|z-u|^2}\nnb
&\leq 2 \left(\prod_{k\neq j}\frac{|z-x_k|^{2\alpha_k}}{|u-x_k|^{2\alpha_k}}\left|\frac{|z-x_j|^{\alpha_j}}{|u-x_j|^{\alpha_j}}-1\right|^2+\left|\prod_{k\neq j}\frac{|z-x_k|^{\alpha_k}}{|u-x_k|^{\alpha_k}}-1\right|^2\right)\nnb
&\quad \times \frac{1}{|z-u|^2}.
\end{align} 
As before, the quantity 
\begin{equation}
\left|\prod_{k\neq j}\frac{|z-x_k|^{\alpha_k}}{|u-x_k|^{\alpha_k}}-1\right|^2 \frac{1}{|z-u|^2}
\end{equation}
is actually bounded for $z\in A_\epsilon$ and $u\in B(x_j,\epsilon)$ (by a constant depending on $L,x_i,\alpha_i$). Let us thus focus on the first term. Here the prefactor $\prod_{k\neq j}\frac{|z-x_k|^{2\alpha_k}}{|u-x_k|^{2\alpha_k}}$ is bounded (depending on the same quantities as before), so our task is to bound the quantity 
\begin{equation}
\int_{B(x_j,\epsilon)}du \left|\frac{|z-x_j|^{\alpha_j}}{|u-x_j|^{\alpha_j}}-1\right|^2\frac{1}{|z-u|^2}.
\end{equation}
For this, we make the change of variables, $u=x_j+(z-x_j)w$ (or perhaps more precisely, first $u=x_j+|z-x_j|w$ and then make use of rotation invariance), so that we find
\begin{align}
\int_{B(x_j,\epsilon)}du \left|\frac{|z-x_j|^{\alpha_j}}{|u-x_j|^{\alpha_j}}-1\right|^2\frac{1}{|z-u|^2}&=\int_{B(0,\frac{\epsilon}{|z-x_j|})}du \left|\frac{1}{|w|^{\alpha_j}}-1\right|^2 \frac{1}{|1-w|^2}.
\end{align}
The singularities here are integrable in two dimensions and $\frac{\epsilon}{|z-x_j|}\leq 1$ since $z\in A_\epsilon$. Thus we see that for $z\in A_\epsilon$, the integral in question is actually a bounded function of $z$.

Let us then consider the behavior of the integral for $z\in B(x_j,\epsilon)$. We consider now three cases:
(i) the integral over $B(x_k,\epsilon)$ with $k\neq j$,
(ii) the integral over $A_\epsilon$, and
(iii) the integral over $B(x_j,\epsilon)$.
In the first case, $|z-u|$ is bounded from below by $2\epsilon$ and since $|\rho(u)|^{-2}$ is locally integrable, we see that the contribution from case (i) is a constant (depending on the same parameters as before), times $|\rho(z)|^2$.
For case (ii), we use the same trick as before: we write 
\begin{align}
\left|\frac{|\rho(z)|}{|\rho(u)|}-1\right|^2\frac{1}{|z-u|^2}&=\left|\frac{|\rho(z)|}{|\rho(u)|}-\prod_{k\neq j}\frac{|z-x_k|^{\alpha_k}}{|u-x_k|^{\alpha_k}}+\prod_{k\neq j}\frac{|z-x_k|^{\alpha_k}}{|u-x_k|^{\alpha_k}}-1\right|^2\frac{1}{|z-u|^2}\nnb
&\leq 2 \left(\prod_{k\neq j}\frac{|z-x_k|^{2\alpha_k}}{|u-x_k|^{2\alpha_k}}\left|\frac{|z-x_j|^{\alpha_j}}{|u-x_j|^{\alpha_j}}-1\right|^2+\left|\prod_{k\neq j}\frac{|z-x_k|^{\alpha_k}}{|u-x_k|^{\alpha_k}}-1\right|^2\right)\nnb
&\quad \times \frac{1}{|z-u|^2}\nnb
&\leq C_1 \left|\frac{|z-x_j|^{\alpha_j}}{|u-x_j|^{\alpha_j}}-1\right|^2 \frac{1}{|z-u|^2}+C_2
\end{align} 
for some constants $C_i=C_i(L,x,\alpha)$. By choosing $R$ large enough (depending on $L$), it is enough to study the integral 
\begin{equation}
\int_{\epsilon\leq |u-x_j|\leq R}du \left|\frac{|z-x_j|^{\alpha_j}}{|u-x_j|^{\alpha_j}}-1\right|^2 \frac{1}{|z-u|^2}.
\end{equation}
Using the same change of variables as before, this becomes 
\begin{equation}
\int_{\frac{\epsilon}{|z-x_j|}\leq |w|\leq \frac{R}{|z-x_j|}}du \left||w|^{-\alpha_j}-1\right|^2\frac{1}{|w-1|^2}\leq \int_{1\leq |w|\leq \frac{R}{|z-x_j|}}du \left||w|^{-\alpha_j}-1\right|^2\frac{1}{|w-1|^2}.
\end{equation}
The singularity at $w=1$ is integrable.
Thus it is only the long range behavior of the integral that could cause an issue.
However, if $\alpha_j>0$, then the integrand behaves like $|w|^{-2}$ for large $w$, and we get a $\log \frac{1}{|z-x_j|}$ bound, while if $\alpha_j<0$, the integrand behaves like $|w|^{-2\alpha_j-2}$ for large $w$, and we get a $|z-x_j|^{2\alpha_j}$ bound.

Finally for case (iii), we argue as in case (ii), but now we end up with the integral  
\begin{equation}
\int_{|w|\leq \frac{\epsilon}{|z-x_j|}}du \left||w|^{-\alpha_j}-1\right|^2\frac{1}{|w-1|^2}\leq \int_{1\leq |w|\leq \frac{R}{|z-x_j|}}du \left||w|^{-\alpha_j}-1\right|^2\frac{1}{|w-1|^2}
\end{equation}
and we get the same type of bound as in case (ii).

To summarize, for some constant $C$ (possibly different from before and depending on $L,x,\alpha$),
\begin{equation}
|(T_1\varphi_2)(z)|\leq C\mu \left(1+\sum_{j=1}^n (|z-x_j|^{2\alpha_j}+|\log |z-x_j||)\right)^{1/2} \||\rho|^{-1}\varphi_2\|_{L^2(\Lambda_L)}.
\end{equation}
This implies that for some constant $C$ (depending once again on the same parameters), 
\begin{equation}
\|T_1\varphi_2\|_{L^2(\Lambda_L)}\leq \mu C\||\rho|^{-1}\varphi_2\|_{L^2(\Lambda_L)}
\end{equation}
and our operator $T_1$ is bounded.
The argument for the second component is identical except the roles of $\alpha_j$ and $-\alpha_j$ are swapped.
\end{proof}

Next we show that $K$ is skew adjoint. 
\begin{lemma}
$K:\H_\rho\to \H_\rho$ satisfies $K^*=-K$. 
\end{lemma}
\begin{proof}
By density of bounded continuous functions in $\H_\rho(\Lambda_L)$ and boundedness of $K$, it is enough to check that 
\begin{equation}
\langle \psi,K\varphi\rangle_\rho=-\langle K\psi,\varphi\rangle_\rho
\end{equation}
when $\psi,\varphi$ are bounded and continuous on $\Lambda_L$.

Using the definition, we have 
\begin{align}
\langle \psi,K\varphi\rangle_\rho &= -\frac{\mu}{2\pi}\int_{\Lambda_L}dz|\rho(z)|^2\overline{\psi_1(z)}\int_{\Lambda_L}du \frac{1}{|\rho(u)|^2}\frac{\varphi_2(u)}{\bar z-\bar u}\nnb
&\quad  -\frac{\mu}{2\pi}\int_{\Lambda_L}dz\frac{1}{|\rho(z)|^2}\overline{\psi_2(z)}\int_{\Lambda_L}du |\rho(u)|^2\frac{\varphi_1(u)}{z-u}.
\end{align}
E.g.~by Young's convolution inequality (and using the boundedness of $\psi,\varphi$), we can use Fubini to interchange the order of integrals and we can write this as 
\begin{align}
\langle \psi,K\varphi\rangle_\rho &= \frac{\mu}{2\pi}\int_{\Lambda_L}du\frac{1}{|\rho(u)|^2}\varphi_2(u)\overline{\int_{\Lambda_L}dz |\rho(z)|^2 \frac{\psi_1(u)}{u-z}}\nnb
&\quad  +\frac{\mu}{2\pi}\int_{\Lambda_L}du|\rho(u)|^2\varphi_1(u)\overline{\int_{\Lambda_L}dz \frac{1}{|\rho(z)|^2}\frac{\psi_2(z)}{\bar u-\bar z}}\nnb
&=-\langle K\psi,\varphi\rangle_\rho.
\end{align}
Thus $K^*=-K$.
\end{proof}

We can now start studying $\Delta(\cdot,w)$ as an element of $\H_\rho$.
\begin{lemma}
  $I-K$ is invertible on $\H_\rho$, $h(\cdot,w)\in \H_\rho$ for every $w$,
  and $\Delta(\cdot,w)=(I-K)^{-1}h(\cdot,w)$ where we understand $(I-K)^{-1}$ as the inverse in $\H_\rho$. Moreover, 
\begin{equation}
\|\Delta(\cdot,w)\|_\rho\leq \|h\|_\rho.
\end{equation}
\end{lemma}
\begin{proof}
Invertibility of $I-K$ follows from $K$ being skew adjoint (implying that its spectrum is on the imaginary axis) and the spectral theorem (and the bounded inverse theorem). The fact that $h(\cdot,w)\in \H_\rho$ follows from $h(\cdot,w)$ being continuous on $\C$ and hence bounded on $\Lambda_L$. 

$u(\cdot,w):=(I-K)^{-1}h(\cdot,w)$ is the unique element of $\H_\rho$ for which $(I-K)u(\cdot,w)=h(\cdot,w)$ since $I-K$ is invertible. We know that $\Delta(\cdot,w)$ satisfies this, so we must have $\Delta(\cdot,w)=(I-K)^{-1}h(\cdot,w)$ when we view $K$ as an operator on $\H_\rho$.

For the norm estimates for $\Delta(\cdot,w)$, we note that since $K$ is skew adjoint (and thus normal), the resolvent $(1-K)^{-1}$ is bounded and normal, so its operator norm equals its spectral radius. Hence we have 
\begin{equation}
\|(I-K)^{-1}\|_{\H_\rho\to \H_\rho}=\sup_{\lambda\in \sigma(K)}\frac{1}{|1-\lambda|}= \frac{1}{d(1,\sigma(K))}\leq 1
\end{equation}
since $\sigma(K)$, the spectrum of $K$, is on the imaginary axis. We thus have
\begin{equation}
\|\Delta(\cdot,w)\|_\rho\leq \|(I-K)^{-1}\|_{\H_\rho\to \H_\rho}\|h\|_\rho\leq \|h\|_\rho
\end{equation}
as claimed.
\end{proof}

We will now translate this into a $L^2(\partial \Lambda_L)$-bound for $\Delta$.

\subsubsection{Boundary estimates for $\Delta$}

As mentioned, our approach for the boundary estimates of $\Delta$ mimic a proof of the Sobolev trace theorem. We need some preparations for this. First we record a simple fact about the regularity of partial derivatives of $\Delta$. 
\begin{lemma}\label{DeltainW12}
For every $w\in \C$ and $j\in \{1,2\}$,
\begin{equation}
z\mapsto \partial_z \Delta_{j1}(z,w) \in L^2(\Lambda_L), \qquad z\mapsto \bar \partial_z \Delta_{j1}(z,w)\in L^2(\Lambda_L),
\end{equation}
where the derivatives are understood in the weak sense.
\end{lemma}

\begin{proof}
Using the fact that $(I-K)\Delta=h$, i.e., $\Delta = K\Delta +h$
and  that $K\Delta$ and $h$ can be expressed in terms of the Cauchy transform of $L^2$ functions,
see \eqref{e:Kdefinition} and \eqref{hzwdefinition},
it follows from Lemma \ref{le:cauchy} that  all of the relevant weak derivatives exist.
Indeed, for $K\Delta$ this follows from $\frac{1}{|\rho|^2}\Delta_{21}\in L^2(\Lambda_L)$ and $|\rho|^2 \Delta_{11}\in L^2(\Lambda_L)$ by continuity of $\Delta$ and that $h$ is the Cauchy transform of an $L^2$ function follows from Lemma~\ref{finerbound}. Moreover,
\begin{align}
\label{eq:grad1}2\partial_z \Delta_{11}(z,w)&=-\mu \frac{1}{|\rho(z)|^2}\Delta_{21}(z,w)\\
\label{eq:grad2}2\bar \partial_z \Delta_{21}(z,w)&=-\mu |\rho(z)|^2 \Delta_{11}(z,w)-\mu^2 H(z,w),
\end{align}
where 
\begin{equation}
H(z,w)=|\rho(z)|^2\int_{\Lambda_L}du \frac{1}{|\rho(u)|^2}\frac{1}{(2\pi)^2}\frac{1}{\bar z-\bar u}\frac{1}{u-w}.\label{CapH}
\end{equation}
This same lemma implies that 
\begin{align}
\label{eq:grad3}2\bar \partial_z \Delta_{11}(z,w)&=-\mu \left(\overline{\mathcal S \left(\1_{\Lambda_L}\frac{1}{|\rho|^2}\overline{\Delta_{21}(\cdot,w)}\right)}\right)(z)\\
\label{eq:grad4}2\partial_z \Delta_{21}(z,w)&=-\mu (\mathcal S(\1_{\Lambda_L}|\rho|^2\Delta_{11}(\cdot,w)))(z)-\mu^2 \left(\mathcal S (\1_{\Lambda_L} H(\cdot,w))\right)(z),
\end{align}
where $\mathcal S$ is the Beurling transform. It follows from Lemma \ref{finerbound} that $z\mapsto H(z,w)\in L^2(\Lambda_L)$.
Since the Beurling transform is a bounded operator on $L^2(\C)$ (with norm one),
see Lemma~\ref{le:cauchy},
we see the question is to show that $\frac{1}{|\rho|^2}\Delta_{21}(\cdot,w),|\rho|^2\Delta_{11}(\cdot,w)\in L^2(\Lambda_L)$. This follows from the boundedness of $\Delta$ on $\Lambda_L$, so we are done.
\end{proof}

Let us now start looking at the boundary values of $\Delta$. For this purpose, let $R>0$ be large enough that $x_1,\dots,x_n\in B(0,R)$ and $\epsilon>0$ be small enough that $\overline{B(0,R+2\epsilon)}\subset \Lambda_L$. Let us choose some increasing smooth function $\eta:[0,L]\to [0,1]$ such that $\eta(L)=1$, $\mathrm{supp}(\eta)\subset [R+\epsilon,L]$ and $\|\eta'\|_\infty \leq \frac{1}{\epsilon}$.  We first record the following simple fact which is essentially the proof of the simplest variant of the Sobolev trace theorem.
\begin{lemma}
For $f\in W^{1,2}(\Lambda_L)$, we have for $R+2\epsilon<L$ and $0<\epsilon<1$
\begin{align}
\int_{\partial \Lambda_L}|dz| |f(z)|^2\leq   \frac{4}{R \epsilon}\left(\int_{R+\epsilon\leq |u|\leq L}du (|f(u)|^2+|\nabla f(u)|^2)\right)\label{sobtracethm}
\end{align}
where we interpret the $f$ appearing on the left-hand side as the trace of $f$ on the boundary $\partial \Lambda_L$.
\end{lemma}

\begin{proof}
We first prove the inequality for $f\in C^{\infty}(\overline{\Lambda_L})$.
We write $(r,\theta)$ for the polar coordinates on $\Lambda_L$.
Noting that $\eta(R+\epsilon)=0$, we use the fundamental theorem of calculus in the radial coordinate to write
\begin{align}
\int_{\Lambda_L} d\theta \, r\, dr\,  \frac{1}{r}\partial_r(\eta(r) |f(r,\theta)|^2)=\int_0^{2\pi}d\theta \,  |f(L,\theta)|^2.
\end{align}
On the other hand, computing the derivative, we can bound the left hand side as 
\begin{align}
&\int_{\Lambda_L} d\theta\, r\, dr\, \frac{1}{r}\partial_r(\eta(r) |f(r,\theta)|^2)\nnb
&=\int_{R+\epsilon \leq |r|\leq L} d\theta\, dr\, (\eta'(r)|f(r,\theta)|^2+\eta(r)(\partial_r \overline{f(r,\theta)})f(r,\theta)+\eta(r)\overline{f(r,\theta)}\partial_r f(r,\theta))\nnb
&\leq \frac{1}{R \epsilon} \int_{R+\epsilon\leq |u|\leq L}du \, |f(u)|^2+\frac{2}{R}\left(\int_{R+\epsilon\leq |u|\leq L} du\,|\nabla f(u)|^2\int_{R+\epsilon\leq |u|\leq L}du\, |f(u)|^2\right)^{1/2},
\end{align}
where in the last step we used Cauchy-Schwarz. Finally using that fact that for $a,b\geq 0$, 
\begin{equation}
(ab)^{1/2}\leq  \frac{1}{\sqrt{2}}(a+b),
\end{equation}
we find that 
\begin{equation}
\int_0^{2\pi}d\theta \, |f(L,\theta)|^2\leq \frac{4}{R \epsilon}\left(\int_{R+\epsilon\leq |u|\leq L}du \, (|f(u)|^2+|\nabla f(u)|^2)\right),
\end{equation}
which was the claim for $f\in C^\infty(\overline{\Lambda_L})$.
Now we use that for any $f\in W^{1,2}(\Lambda_L)$ there is a sequence $\{f_n\}\subset C^\infty(\overline{\Lambda_L})$ such that $\|f_n-f\|_{W^{1,2}(\Lambda_L)}\rightarrow 0$
(see for example \cite[Theorem~3 in Section 5.3.3]{MR1625845}),
which implies that $\|f_n\|_{W^{1,2}(A)}\rightarrow \|f\|_{W^{1,2}(A)}$ where $A=\{u: R+\epsilon \leq |u|\leq L\}\subset {\Lambda_L}$,  using this together with continuity of the trace operator
(see for example \cite[Theorem~1 in Section 5.5]{MR1625845}),
allows us to take the limit in the inequality, proving \eqref{sobtracethm}.
\end{proof}
We now have all the tools we need for the boundary estimates. 

\begin{proposition}\label{prop:polyboundonL2int}
With the same notation as above, we have for $w\in B_R(0)$, 
\begin{equation}
\int_{\partial \Lambda_L}|dz||\Delta_{j1}(z,w)|^2\leq C(\epsilon,x,\alpha,w,\mu)(\|h(\cdot,w)\|_\rho^2+\|H(\cdot,w)\|_{L^2(\Lambda_L)}^2)\label{temp56yt}
\end{equation}
where $H$ is defined \eqref{CapH} and $h$ in \eqref{hzwdefinition}.
Moreover, the right-hand side of \eqref{temp56yt} grows at most polynomially in $L$, for all $L$ sufficiently large.
\end{proposition}

\begin{proof} From Lemma \ref{DeltainW12}, we see that the components $\Delta_{j,1}$ are both in $W^{1,2}(\Lambda_L)$.
Using the previous lemma, we see that our task is to bound 
\begin{equation}
\int_{R+\epsilon\leq u\leq L}du \, ((|\Delta_{j1}(u,w)|^2+|\nabla_u \Delta_{j1}(u,w)|^2).
\end{equation}
For the $L^2$-part, we can write $|\Delta_{j1}(u,w)|^2=|\rho(u)|^2 |\rho(u)|^{-2}|\Delta_{j1}(u,w)|^2$, and use the fact that for $|u|\geq R+\epsilon$, we have 
\begin{equation}
|\rho(u)|=\prod_{j=1}^n \left|1-\frac{x_j}{u}\right|^{\alpha_j}
\end{equation}
which can be bounded from above by a constant depending only on $n$, $\epsilon$. One has the same bound for $|\rho(u)|^{-1}$. This means that we can bound the $L^2$-norm of $\Delta$ on the annulus by this constant times the $\H_\rho$-norm, which we can bound by $\|h\|_\rho$.

Let us then look at the gradient term. We have that 
\begin{align}
\int_{R+\epsilon\leq |u|\leq L} du \, |\nabla_u \Delta_{j,1}(u,w)|^2=2\int_{R+\epsilon\leq |u|\leq L} du\, (|\partial  \Delta_{j,1}(u,w)|^2+|\bar\partial  \Delta_{j,1}(u,w)|^2).
\end{align} 
Using \eqref{eq:grad1}, we see that 
\begin{align}
  \int_{R+\epsilon\leq u\leq L} du\, |2\partial_u \Delta_{11}(u,w)|^2&\leq \mu^2 \sup_{R+\epsilon\leq |u|\leq L}|\rho(u)|^{-2}\int_{\Lambda_L}du\, |\rho(u)|^{-2}|\Delta_{21}(u,w)|^2
                                                                 \nnb
&\leq \mu^2 \sup_{R+\epsilon\leq |u|\leq L}|\rho(u)|^{-2} \|h\|_\rho^2.
\end{align}

Next we note from \eqref{eq:grad3} that for $R+\epsilon\leq |u|\leq L$,  
\begin{align}
|2\bar\partial_u\Delta_{11}(u,w)|&=\mu \left|\int_{\Lambda_L} dv\frac{1}{|\rho(v)|^2}\overline{\Delta_{21}(v,w)}\frac{1}{(u-v)^2}\right|\nnb
  &\leq \frac{4\mu}{\epsilon^2} \left|\int_{|v|\leq R+\frac{\epsilon}{2}}dv\frac{1}{|\rho(v)|^2}|\Delta_{21}(v,w)|\right|
    \nnb
  &\qquad\qquad +\mu \left|\int_{R+\frac{\epsilon}{2}\leq |v|\leq L}dv\frac{1}{|\rho(v)|^2}\overline{\Delta_{21}(v,w)}\frac{1}{(u-v)^2}\right|
    .
\end{align}
For the first part, we can use Cauchy-Schwarz to bound it by a constant times $\|h\|_\rho$ while for the second part, we can use the fact that the Beurling transform is bounded on $L^2(\C)$ to bound its $L^2$ norm by the $L^2$-norm of 
\begin{equation}
\1_{R+\frac{\epsilon}{2}\leq |v|\leq L}\frac{1}{|\rho(v)|^2}|\Delta_{21}(v,w)|
\end{equation}
which we can bound by a constant times the $L^2$-norm of $\frac{1}{|\rho|}|\Delta_{21}(\cdot,w)|$, which we bound by $\|h\|_\rho$. Thus we get a bound by a constant (depending on the relevant quantities) times $\|h\|_\rho$. 

Similarly from \eqref{eq:grad2} and the inequality $|a+b|^2\leq 2|a|^2+2|b|^2$, we see that 
\begin{align}
\int_{R+\epsilon\leq u\leq L}|2\bar\partial_u \Delta_{21}(u,w)|^2&\leq 2 \mu^2\left(\sup_{R+\epsilon\leq |u|\leq L}|\rho(u)|^{2}\|h\|_\rho^2+\mu^2\int_{R+\epsilon\leq |u|\leq L}|H(u,w)|^2\right).
\end{align} 
For the remaining estimate for $2\partial_u \Delta_{21}(u,w)$ we get
\begin{align}
  &\int_{R+\epsilon\leq|u|\leq L}|2\partial_u \Delta_{21}(u,w)|^2
  \nnb
  &\leq 2\mu^2\int_{R+\epsilon\leq |u|\leq L}du\Big[\Big|\mathcal{S}\Big( \1_{\Lambda_L}|\rho|^2\Delta_{11}(\cdot,w)\Big)(u)\Big|^2+\mu^2\Big |\mathcal{S}\Big(\1_{\Lambda_L}H(\cdot,w)\Big)(u)\Big|^2\Big]
\end{align}
which can be bounded above by a similar argument to that of the $2\bar\partial_u\Delta_{11}(u,w)$ term.

To see that $\|H(\cdot,w)\|^2_{L^2(\Lambda_L)}$ grows at most polynomially in $L$, split up region of integration into $\Lambda_{R+1}$ and $\Lambda_L\setminus\Lambda_R$. On $\Lambda_{R+1}$ use \eqref{finerbound} to see it contributes a finite integral, and then on $\Lambda_L\setminus \Lambda_{R+1}$ use that $|H(z,w)|\leq C_R/|z-R|$ for $|z|>R+1$. To see that $\|h(\cdot,w)\|_\rho^2=\int_{\Lambda_L}\rho^{-2}|h(\cdot,w)|^2$ also grows at most polynomially in $L$,
by Lemma \ref{hzwcontinuity}, $h(z,w)$ is continuous and vanishes at infinity, so it is bounded, and $\rho^{-2}$ is bounded outside of $\Lambda_{R}$.
\end{proof}

\begin{proof}[Proof of Proposition~\ref{prop:Ldepenbounds}]\label{proofLdepenbounds}
From Proposition~\ref{prop:polyboundonL2int} we know that
\begin{equation}
\|\Delta(\cdot,w)\|_{L^2(\partial \Lambda_L)}\leq C(x,\alpha,L,\mu),
\end{equation}
where the constant grows at most polynomially in $L$.
We consider the integral
\begin{equation}
\int_{|u|\geq L} e^{-|\mu| |u|}|\Delta_{11}(u,w)| \, du \label{temp46tgrf}
\end{equation}
as the integral involving $\Delta_{21}$ is similar.
We know that $\Delta_{11}(\cdot,w)$ is anti-holomorphic in the exterior of $\Lambda_L$, and vanishes at infinity so  we have the Laurent expansion
\begin{equation}
\Delta_{11}(u,w)=\sum_{n=1}^\infty \beta_n \frac{1}{\bar u^n}
\end{equation}
in the exterior of the disk, and this holds on $\partial \Lambda_L$ by continuity. By Cauchy-Schwarz, we have 
\begin{equation}\label{temp24ehbc}
  \left(\int_{|u|\geq L} e^{-|\mu| |u|}|\Delta_{11}(u,w)|\, du\right)^2
  \leq \int_{|u|\geq L} du \, e^{-|\mu| |u|}\int_{|u|\geq L} du\, e^{-|\mu||u|}\left|\sum_{n=1}^\infty \beta_n \frac{1}{\bar u^n}\right|^2. 
\end{equation}
The first integral is
\begin{align}
\int_{|u|\geq L}du\,  e^{-|\mu||u|}=\frac{2\pi}{\mu^2}(|\mu|L+1)e^{-|\mu|L}
\end{align}
and by Parseval's identity on the circle, the second integral is
\begin{align}
  \int_L^\infty dr \int _0^{2\pi}d\theta \ e^{-|\mu|r}r \left|\sum_{n=1}^\infty \beta_n \frac{e^{in\theta}}{r^n}\right|^2
&=2\pi\sum_{n=1}^\infty |\beta_n|^2 \int_L^\infty dr e^{-|\mu| r} r^{-2n+1}\nnb
&\leq 2\pi L^2 \sum_{n=1}^\infty |\beta_n|^2 L^{-2n}\int_{1}^\infty dr\, e^{-|\mu| L r}\nnb
&=\frac{2\pi}{|\mu|} L e^{-|\mu| L} \sum_{n=1}^\infty |\beta_n|^2 L^{-2n}\nnb
&=\frac{1}{|\mu|} L e^{-|\mu| L} \|\Delta_{11}(\cdot,w)\|_{L^2(\partial \Lambda_L)}. 
\end{align}
Now by Proposition \ref{prop:polyboundonL2int}, we know the $L^2$ norm on $\partial \Lambda_L$ grows at most polynomially in $L$, so indeed \eqref{temp46tgrf} tends exponentially fast to  zero as $L\rightarrow \infty$.
\end{proof}

\section{Finite-volume tau functions and massive Bosonization}\label{sec:bosonization}

\subsection{Massive Bosonization identity}

In this section we extend the Bosonization identities of Section~\ref{sec:massless-bosonization} from the massless situation to the case with finite-volume mass.
More precisely, on the fermionic side, we consider finite-volume mass term $\mu {\bf 1}_{\Lambda}$ as in Section~\ref{sec:Green},
whereas on the bosonic side the massless free field is replaced
by the sine-Gordon measure with finite-volume interaction $2z \int_\Lambda \wick{\cos(\sqrt{4\pi} \varphi)} \, dx$ defined in Section~\ref{sec:mixing}.

Let $x_1,\dots,x_n\in\C$ be distinct points, $\alpha_1,\dots,\alpha_n \in (-\frac12,\frac12)$ with $\sum_i \alpha_i=0$,
and recall the definition of $\rho$ from \eqref{eq:rho}.
By \eqref{e:ren-det-massless} (see also Appendix~\ref{app:det}),  the fractional correlation functions of the massless GFF have the interpretation of
a renormalized partition function of massless free fermions with winding, i.e.,
\begin{equation} \label{e:ren-det-massless-Zrho}
  \avg{\prod_{j=1}^n \wick{e^{i\sqrt{4\pi}\alpha_j\varphi(x_j)}}}_{\GFF(0)}
  \propto Z_\rho(0)
\end{equation}
where $Z_\rho(0)$ may formally be interpreted as
\begin{equation}
  Z_\rho(0) = \text{``}  \frac{\det(\slashed{\partial}_\rho)}{\det(\slashed\partial)} \text{''} .
\end{equation}
To extend \eqref{e:ren-det-massless-Zrho} to the massive situation,
we first give a suitable interpretation to the renormalized partition function
\begin{equation} \label{e:Zrho-bis}
  Z_\rho(\mu{\bf 1}_\Lambda) = \text{``} \frac{\det(\Dirac_\rho+\mu{\bf 1}_\Lambda)}{\det(\Dirac + \mu{\bf 1}_\Lambda)} \text{''} ,
\end{equation}
in Definition~\ref{def:Zrho} below. The infinite-volume limit of these functions will be shown 
in Section~\ref{sec:Palmer} to agree with the tau functions of Sato--Miwa--Jimbo \cite{MR555666} and Palmer \cite{MR1233355}
(at least up to a normalization constant),
so our renormalized partition functions may be considered finite-volume versions of the tau functions.

Besides the definition of these finite-volume tau functions and the derivation of their important properties and regularity,
the main result of this section then is the following finite-volume Bosonization identity,
identifying the fractional sine-Gordon correlation functions with the former finite-volume tau functions.

\begin{theorem} \label{thm:massive-bosonization}
  Let $\Lambda \subset \R^2$ be a bound and simply connected with smooth boundary, and let $z \in \R$.
  Then for any $\alpha_1,\dots,\alpha_n \in (-\frac12,\frac12)$ with $\sum_i\alpha_i=0$
  and any $f_1,\dots,f_n\in C_c^\infty(\C)$ with disjoint supports, 
\begin{equation} \label{e:massive-bosonization}
  \avg{\prod_{j=1}^n M_{\alpha_j}(f_j)}_{\SG(4\pi, z|\Lambda,0)}
  \propto \int_{\C^n}dx\prod_{j=1}^n f_j(x_j) Z_\rho(\mu{\bf 1}_\Lambda)
\end{equation}
where $\mu = Az$ with $A$ as in Theorem~\ref{thm:main-tau},
and $Z_\rho(\mu \1_\Lambda)$ is defined in Definition~\ref{def:Zrho} below.
\end{theorem}

\begin{remark}
  Explicitly, the proportionality constant on the right-hand side is
  \begin{equation}
    \pb{4\pi e^{-\gamma/2}}^{n} (2e^{-\gamma/2}C_\eta)^{\sum_j \alpha_j^2} 
  \end{equation}
  where the mollifier-dependent constant $C_\eta$ is as in \eqref{e:Ceta}.
\end{remark}

The strategy for the identification in finite volume is their identification when $z=\mu=0$ provided by Section~\ref{sec:massless-bosonization},
together with analyticity of both sides proved in Section~\ref{sec:finvol-massless} for the sine-Gordon side and
established in this section for the right-hand side (as a consequence of Section~\ref{sec:Green}).

\subsection{The finite-volume tau functions}
\label{sec:Z-renorm}

Our goal is to define \eqref{e:Zrho-bis},
and to this end we start with formal manipulations of the right-hand side to motivate our rigorous definition
in Definition~\ref{def:Zrho}.
In the next few equations discussing these manipulations,
we drop the quotation marks from the right-hand side of \eqref{e:Zrho-bis} and write:
\begin{equation}
  Z_\rho(\mu\chi) =
\frac{\det(\slashed{\partial}_\rho+\mu\chi)}{\det(\slashed{\partial}+\mu\chi)},\label{Zrhoform}
\end{equation}
where we recall that the subscript $\rho$ indicates that in the numerator, the operator acts on functions with branching determined by $\rho$, while in the denominator, this is simply the Dirac operator with no branching. Continuing with formal manipulations, we can write this as 
\begin{equation}
  Z_\rho(\mu\chi) =
  \frac{\det(\slashed{\partial}_\rho)}{\det(\slashed{\partial})}\frac{\det(\slashed{\partial}_\rho+\mu\chi)}{\det(\slashed{\partial}_\rho)}\frac{\det(\slashed{\partial})}{\det(\slashed{\partial}+\mu\chi)}.
\end{equation}
Comparing with the formal Grassmann integral representation of the determinant as
\begin{equation}
  \det(\Dirac_\rho + \mu\chi)
  =
  \int d_\psi d_{\bar\psi} e^{-\int_{\C} du\, \psi \Dirac_\rho\bar\psi - \mu \int_\C du\, \chi \psi\bar\psi},
\end{equation}
and the corresponding formal expression for the fermionic expectation
\begin{equation}
  \avg{F}_{\FF_\rho(\mu)}
  =
  \frac{\int d_\psi d_{\bar\psi} e^{-\int_{\C} du\, \psi \Dirac_\rho\bar\psi - \mu \int_\C du\, \chi \psi\bar\psi}\, F}{\int d_\psi d_{\bar\psi} e^{-\int_{\C} du\, \psi \Dirac_\rho\bar\psi - \mu \int_\C du\, \chi \psi\bar\psi}},
\end{equation}
see, e.g., \cite[Appendix~A]{MR4767492},
it is natural to interpret the renormalized partition function as
\begin{equation}
  Z_\rho(\mu\chi) =
\frac{\det(\slashed{\partial}_\rho)}{\det(\slashed\partial)}\frac{\avg{e^{\mu \int_{\C}du\chi(u)(\bar \psi_1(u)\psi_1(u)+\bar\psi_2(u)\psi_2(u))}}_{\FF_\rho(0)}}{\avg{e^{\mu \int_{\C}du\chi(u)(\bar \psi_1(u)\psi_1(u)+\bar\psi_2(u)\psi_2(u))}}_{\FF(0)}},
\end{equation}
where $\FF_\rho(0)$ indicates the free fermion expectation with Green's function $S_0^\rho$ as defined in \eqref{eq:S0}
and $\FF(0)$ that with $S_0$ but with $\rho=1$ (no winding).
As discussed above, we interpret the first factor as in \eqref{e:ren-det-massless-bis}.
Still formally, the second factor should satisfy
\begin{align}
&\partial_\mu \log \frac{\avg{e^{\mu \int_{\C}du\chi(u)(\bar \psi_1(u)\psi_1(u)+\bar\psi_2(u)\psi_2(u))}}_{\FF_\rho(0)}}{\avg{e^{\mu \int_{\C}du\chi(u)(\bar \psi_1(u)\psi_1(u)+\bar\psi_2(u)\psi_2(u))}}_{\FF(0)}}\nnb
&=\int_{\C}du\, \chi(u)\pB{\avg{\bar\psi_1(u)\psi_1(u)+\bar \psi_2(u)\psi_2(u)}_{\FF_\rho(\mu)}-\avg{\bar\psi_1(u)\psi_1(u)+\bar \psi_2(u)\psi_2(u)}_{\FF(\mu)}}.
\end{align}
The individual terms in the integral do not make sense (are formally infinite), but we can define a well-defined renormalized partition function by so-called ``point-splitting''
(and the infinities then cancel in the difference).
This is the content of the following definition.
As mentioned previously,
a further justification for this definition is given a posteriori in Section~\ref{sec:Palmer} where we show that it is consistent with
the usual tau functions.

\begin{definition}\label{def:Zrho}
Given $\chi$ as in \eqref{e:chi} %
and $\rho$ as in \eqref{eq:rho} with windings $\alpha_1,\dots,\alpha_n \in (-\frac12,\frac12)$ satisfying $\sum_i \alpha_i = 0$ and
branch points $x_1,\dots,x_n \in \C$,
define the renormalized partition function (or finite-volume tau function) as
\begin{equation}\label{Zrho}
  Z_{\rho}(\mu\chi) = Z_\rho(0)\times \tilde Z_{\rho}(\mu\chi)
\end{equation}
where $\tilde Z_{\rho}(\mu\chi)$ is defined as
\begin{equation}\label{Ztilderho}
  \exp\qa{\int_0^\mu ds
    \int_{\C}du\, \chi(u)\,
    \lim_{h\to 0} \Big(S_{s\chi,11}^\rho(u_h,u)-S_{s\chi,11}^1(u_h,u)+S_{s\chi,22}^\rho(u_h,u)-S_{s\chi,22}^1(u_h,u)\Big)},
\end{equation}
where $u_h = u+h$,
and  $Z_\rho(0)$ is defined precisely as the right-hand side of \eqref{e:ren-det-massless}:
\begin{equation} \label{Zrho0}
  Z_\rho(0)=\prod_{1\leq r<s<n} |x_r-x_s|^{2\alpha_r\alpha_s}.
\end{equation}
\end{definition}

Important properties of our finite-volume tau functions, which follow from properties of the twisted Green's function
from Section~\ref{sec:Green}, are summarized in the following proposition.

\begin{proposition} \label{prop:Zrho}
  The definition makes sense, i.e., the limit \eqref{Ztilderho} exists and is finite.
  Moreover, $Z_\rho(\mu\chi)$ is independent of the branch cuts of $\rho$, $Z_\rho(\mu\chi)>0$, 
  and
  \begin{equation} \label{e:Zrho-symmetries}
    Z_\rho(\mu\chi)=Z_\rho(-\mu\chi), \qquad
    Z_\rho(\mu\chi) = Z_{1/\rho}(\mu\chi).
  \end{equation}
\end{proposition}

\begin{proof}
The statements follow from the short-distance expansion of Proposition~\ref{prop:Green-Delta}, using which
\begin{align}
  &   \int_\C du \, \chi(u) \,\lim_{h\to 0} \Big(S^\rho_{\mu\chi,11}(u+h,u)-S_{\mu\chi,11}^1(u+h,u)\Big)\label{temp76uyth}\nnb
&= \int_{\C}du \, \chi(u)\, \left[\frac{\mu}{(2\pi)^2}\int_{\C}dv \, \chi(v)\, \left(\frac{|\rho(u)|^2}{|\rho(v)|^2}-1\right)\frac{1}{|u-v|^2}+|\rho(u)|^2\Delta_{11}^\rho(u,u)-\Delta_{11}^1(u,u)\right],
\end{align}
and analogously for the $22$ component,
where $\Delta_{ii}^\rho$ and $\Delta_{ii}$ depend on $\mu$. %

To see that the integral on the right-hand side is convergent (and therefore that the definition makes sense),
by a partition of unity argument as in Lemma~\ref{le:KonHrhobounded},
we can reduce the existence of the integral to
\begin{equation}
  \int_{B_1(x)^2}du\, dv\,  |\frac{|u-x|^\alpha}{|v-x|^\alpha}-1||u-v|^{-2} < \infty.
\end{equation}
By shifting and scaling, this convergence is equivalent to the convergence of the integral
\begin{equation}
  \int_{B_1(0)}du \int_{B_{1/|u|}(0)}dv \, ||v|^{-\alpha}-1||1-v|^{-2}.
\end{equation}
The inner integral is $O(1+\log |u|^{-1})$ so we have convergence.

The positivity $Z_\rho(\mu\chi) > 0$ follows from $\Delta_{ii}(z,z)\in \R$ by Proposition~\ref{prop:Greensym}, which implies that
the term inside the exponential in the definition is real.

To see that $Z_\rho(\mu\chi)$ is independent of the branch cuts, it suffices to observe that it only depends on $|\rho|^2$.
This is clear from \eqref{temp76uyth} and the fact that $\Delta^\rho$ only depends on $|\rho|^2$, see \eqref{e:Deltadefinition1-bis}--\eqref{e:Deltadefinition2-bis}.

The symmetry  $\overline{Z_\rho(\mu\chi)} = Z_{1/\rho}(\mu\chi)$ follows from the 
symmetries of the Green's function \eqref{Greensym}. Together with the positivity this gives the second symmetry in \eqref{e:Zrho-symmetries}.

Similarly, the symmetry $Z_\rho(\mu\chi)=Z_\rho(-\mu\chi)$ in the mass follows from
the symmetry  of the diagonal components of the finite-volume Green's function \eqref{symmetriesinmass1}.
\end{proof}

\begin{remark}
We expect that we could have alternatively defined the formal determinant \eqref{Zrhoform} as the renormalized determinant
(see \cite{MR2154153} for the definition of $\det_4$):
\begin{equation}
\frac{\det_4(I+\mu\chi S^\rho_0)}{\det_4(I+\mu\chi S^1_0)} \times \text{``traces"},
\end{equation}
after showing that $\chi S^\rho_0$ defines an integral operator $\chi S^\rho_0:L^2(\supp\chi)\oplus L^2(\supp \chi)\rightarrow L^2(\supp\chi)\oplus L^2(\supp \chi)$ in the fourth Schatten class, i.e., $\text{Tr}((\chi S^\rho_0)^*\chi S^\rho_0)^2<\infty$.
The ``traces'' above refer to terms missing in $\det_4$ compared to a usual determinant (and would be divergent).
To make sense of the divergent traces (and thus obtain a version that correspond to the usual determinant) would require regularizing these
as in \eqref{temp76uyth}, and then extending to an analytic function of $\mu$.

As our more explicit definition of the regularized determinant is convenient for our analysis
(which is to connect to the sine-Gordon correlation functions and tau functions in the infinite-volume limit),
we do not discuss this further.
\end{remark}

\subsection{Regularity of the finite-volume tau function}

As preparation for the proof of Theorem~\ref{thm:massive-bosonization} we first establish some
regularity and analyticity properties of the finite-volume tau function.

\begin{lemma}\label{le:pfanalyticity}
  For each $\chi$ as in \eqref{e:chi},
  \begin{equation}
    \mu\rightarrow \log \tilde Z_\rho(\mu\chi)
  \end{equation} has an analytic extension to a $\chi$-dependent neighborhood of the real axis.
  The neighborhood also depends on $\rho$ but is uniform on compact subsets of $x_1,\dots,x_n$ distinct.
  The $\mu$-derivatives are
  \begin{align} \label{e:pfanalyticity}
    &\partial_\mu^{p}|_{\mu=0} \log \tilde Z_\rho(\mu\chi)\nnb&=\int du\prod_{j=1}^{p}\chi(u_j)\Big( \langle\bar\psi_1\psi_1(u_1)+\bar\psi_2\psi_2(u_1);\cdots;\bar\psi_1\psi_1(u_{p})+\bar\psi_2\psi_2(u_{p})\rangle^{\sf T}_{\FF_\rho(0)}\nnb&\qquad- \langle\bar\psi_1\psi_1(u_1)+\bar\psi_2\psi_2(u_1);\cdots;\bar\psi_1\psi_1(u_{p})+\bar\psi_2\psi_2(u_{p})\rangle^{\sf T}_{\FF(0)}\Big)
  \end{align}
  for $p\in \N$, which is equal to zero when $p$ is odd.
  
  As a consequence, for $f_1,\dots,f_n\in C_c^\infty(\C)$ with disjoint supports, the function
  \begin{equation} \label{e:Zrho-analyticmu}
    \mu \mapsto \int_{\C^n}dx\prod_{j=1}^n f_j(x_j) Z_\rho(\mu\chi)
  \end{equation}
  is analytic in a neighborhood of the real axis, and the $\mu$-derivatives are given by
  \begin{equation} \label{e:Zrho-derivmu}
    \partial_\mu^p|_{\mu=0} \int_{\C^n}dx \prod_{j=1}^n f_j(x_j)Z_\rho(\mu \chi)= \int_{\C^n}dx \prod_{j=1}^n f_j(x_j)\partial_\mu^p|_{\mu=0}Z_\rho(\mu \chi).
  \end{equation}
\end{lemma}

\begin{proof}
From Definition~\ref{def:Zrho}, recall that $\log \tilde{Z}_\rho(\mu\chi)$ is given by
\begin{equation}\label{recalllogtildeZ}
   \int_0^\mu ds  \int_{\C}du\, \chi(u)  \lim_{h\to 0}\Big((S_{s\chi,11}^\rho(u+h,u)-S_{s\chi,11}^1(u+h,u))+(S_{s\chi,22}^\rho(u+h,u)-S_{s\chi,22}^1(u+h,u))\Big).
\end{equation}
We will show that this has an analytic extension into the same region of analyticity of $S^\rho_{\mu\chi}(z,w)$ as in Proposition~\ref{pr:analyt}.
By symmetry in the components
and the fact that the integral of an analytic function is analytic, it is sufficient to show analyticity for \eqref{temp76uyth},
i.e., without the $s$ integral in  \eqref{recalllogtildeZ}.
Further note that using similar upper bounds as in the proof of Proposition~\ref{prop:Zrho} and a partition of unity,
we can find an integrable function that dominates $|S_{s\chi,11}^\rho(u+h,u)-S_{s\chi,11}^1(u+h,u)|$ for all $h$ small,
so that by dominated convergence we can interchange the limit in $h$ and the integral in $u$ in \eqref{recalllogtildeZ} and \eqref{temp76uyth}.

First using \eqref{tildeSdelta} we write
 \begin{align}
 S^\rho_{(\mu+\delta)\chi}(z,w)-S_{(\mu+\delta)\chi}^1(z,w)=\sum_{n=0}^\infty \delta^n a_n^\mu(z,w)
 \end{align}
 where
 \begin{align}
 a_0^\mu(z,w)=S_{\mu\chi}^\rho(z,w)-S_{\mu\chi}^1(z,w)
 \end{align}
 and 
 \begin{align}
 a_n^\mu(z,w)=(-1)^n\int_{\C^n} \prod_{j=1}^n\chi(u_j)(S_{\mu\chi}^\rho(z,u_1)\cdots S_{\mu\chi}^\rho(u_n,w)-S_{\mu\chi}^1(z,u_1)\cdots S_{\mu\chi}^1(u_n,w)).
 \end{align}
To show that \eqref{temp76uyth} is analytic, we observe that by the proof of Proposition~\ref{pr:analyt12}, 
\begin{align}
\int_{\C^n}du_1\cdots du_n \prod_{j=1}^n\chi(u_j)S_{\mu\chi}^\rho(z,u_1)\cdots S_{\mu\chi}^\rho(u_n,w)
\end{align} is jointly continuous in $z,w\notin \Gamma$ for $n\geq 2$. Hence recalling the series expansion of $S$ in \eqref{tildeSdelta}, we see that $\sum_{n=2}^\infty \delta^n a_n(z,w)$ is absolutely convergent for $\delta$ small enough, uniformly in $z,w\in\text{supp}(\chi)$. This also shows that $\overline{P}(z)^{-1}a_n(z,w)P(w)^{-1}$ is joint continuous for $(z,w)\in\C^2$, $n\geq 2$ and these two facts together allow us to move the integral and $h$-limit past the series:
\begin{align}
  \lim_{h\to 0}\int_\C du\,\chi(u)\sum_{n=0}^\infty \delta^na_n(u+h,u)
  =\sum_{n=2}^\infty \delta^n\int_\C du\chi(u)a_n(u,u)\nnb
  +\lim_{h\to 0}\int_\C du\,\chi(u)\, a_0(u+h,u)
  +\delta \lim_{h\to 0}\int_\C du\, \chi(u)\, a_1(u+h,u)
\end{align}
The term $\lim_{h\to 0}\int_\C du\, \chi(u) \, a_0(u+h,u)$ is equal to \eqref{temp76uyth} and
the limit in $\delta \lim_{h\to 0}\int_\C du\, \chi(u) \, a_1(u+h,u)$ exists by an argument similar to Lemma~\ref{lemmatemp42}.
This yields analyticity of $\log \tilde Z_\rho(\mu\chi)$.

From the above, we have the following series expansion for $\mu$ sufficiently small,
\begin{multline}
  \log\tilde Z_{\rho}(\mu\chi)
  =\sum_{n \in 2\N+1}\int du(-1)^n \frac{\mu^{n+1}}{n+1}\text{Tr}\Big(S_0^\rho(u_1,u_2)\cdots S_0^\rho(u_{n+1},u_1)\\
  -S_0^1(u_1,u_2)\cdots S_0^1(u_{n+1},u_1)\Big)\prod_{j=1}^{n+1}\chi(u_j),
\end{multline}
so, with $C_p$ the set of cyclic permutations, by cyclicity of the trace  we have
\begin{align}
  &\partial_\mu^{p+1}|_{\mu=0} \log \tilde Z_\rho(\mu\chi)\nnb 
  &=p!(-1)^p\int du \ \text{Tr}\Big(S_0^\rho(u_1,u_2)\cdots S_0^\rho(u_{p+1},u_1)-S_0^1(u_1,u_2)\cdots S_0^1(u_{p+1},u_1)\Big)\prod_{j=1}^{p+1}\chi(u_j),\nnb
&=(-1)^p\sum_{\sigma\in C_{p+1}}\int du\, \text{Tr}\Big(S_0^\rho(u_{\sigma(1)},u_{\sigma^2(1)})\cdots S_0^\rho(u_{\sigma^{p+1}(1)},u_{\sigma(1)})
  \nnb &\qquad\qquad\qquad\qquad\qquad-S_0^1(u_{\sigma(1)},u_{\sigma^2(1)})\cdots S_0^1(u_{\sigma^{p+1}(1)},u_{\sigma(1)}\Big)\prod_{j=1}^{p+1}\chi(u_j),
\end{align}
when $p\in 2\N+1$ and the $p \in 2\N$ derivatives are zero.
Now observe that
\begin{align}
&\sum_{\sigma\in C_{p+1}}\text{Tr}\Big(S_0^\rho(u_{\sigma(1)},u_{\sigma^2(1)})\cdots S_0^\rho(u_{\sigma^{p+1}(1)},u_{\sigma(1)})\Big)\nnb
&=\sum_{\sigma\in C_{p+1}}\sum_{(a_i)\in\{1,2\}^{p+1}}S^\rho_{a_1,a_2}(u_{\sigma(1)},u_{\sigma^2(1)})\cdots S^\rho_{a_{p+1},a_1}(u_{\sigma^{p+1}(1)},u_{\sigma(1)})\nnb
&=\sum_{\sigma\in C_{p+1}}\sum_{(a_i)\in\{1,2\}^{p+1}}S^\rho_{a_{\sigma(1)},a_{\sigma^2(1)}}(u_{\sigma(1)},u_{\sigma^2(1)})\cdots S^\rho_{a_{\sigma^{p+1}(1)},a_{\sigma(1)}}(u_{\sigma^{p+1}(1)},u_{\sigma(1)})\nnb
&=(-1)^p\langle\bar\psi_1\psi_1(u_1)+\bar\psi_2\psi_2(u_1);\cdots;\bar\psi_1\psi_1(u_{p+1})+\bar\psi_2\psi_2(u_{p+1})\rangle^{\mathsf T}_{\FF_\rho(0)}\label{derivlogZcumulant}
\end{align}
where the second equality holds because $a_i\leftrightarrow a_{\sigma^i(1)}$ is a bijection of $p+1$ tuples.
Note that the only nonzero contributions are from $(a_i)=(1,2,1,2,\dots)$ and $(2,1,2,1,\dots)$. The last equality holds by multilinearity and definition of the cumulant
(see the discussion around \cite[(1.10)]{MR4767492}).

For \eqref{e:Zrho-analyticmu}, we observe that 
by taking the exponential, $Z_\rho(\mu\chi)$ is analytic in the same region as $\log \tilde Z_\rho(\mu\chi)$.
By differentiablility in the branch points $x_i$ given by Lemma~\ref{differentiabilitylogZ} stated below the proof,
which implies continuity,
we can consider any closed, rectifiable curve $\gamma$ in the same region of analyticity to get 
\begin{align}
\int_{\gamma}d\mu\int_{\C^n}dx\prod_{j=1}^n f_j(x_j) Z_\rho(\mu\chi)=\int_{\C^n}dx\prod_{j=1}^n f_j(x_j)\int_\gamma d\mu \, Z_\rho(\mu\chi)=0
\end{align}
by Fubini's and Morera's theorems.
To integrate over the test functions $f_1,\dots,f_n \in C_c^\infty(\C)$ we used that the neighborhood
of analyticity is uniform on compact subsets $K$ of $x_1,\dots,x_n$ distinct.

Finally, to show \eqref{e:Zrho-derivmu}, let $K\subset \{(x_1,...,x_n)\in \C^n: x_i\neq x_j \text{ for }i\neq j\}$ be compact and fix $r>0$ small so that $\mu\mapsto Z_\rho(\mu\chi)$ is analytic in the disk $B_{2r}(0)$, and $\sup_{x\in K,\,|\mu|=r}|\tilde Z_\rho(\mu\chi)|<\infty$ (such an $r$ exists by the uniformity in $x\in K$ in the domain of analyticity). Then define
$a_k(x)=\frac{1}{2\pi i}\int_{|\mu|=r} \tilde Z_\rho(\mu\chi)\mu^{-(k+1)}\,d\mu$ where we recall that the right-hand sides depend on $x=(x_1,\dots,x_n)$
through $\rho$.
By Cauchy's integral formula $a_k(x)=\partial_\mu^k|_{\mu=0}\tilde Z_\rho(\mu \chi)$, and by analycity, $\tilde Z_\rho(\mu\chi)=\sum_{k\ge0} a_k(x)\mu^k$ with convergence uniform on $|\mu|\le r$ and $x\in K$.
Thus, by Fubini's theorem and the dominated convergence theorem,
\begin{equation}
\int_{\C^n} \, dx \prod_j f_j(x_j)\,Z_\rho(\mu\chi)=\sum_{k\ge0}\mu^k \int_{\C^n} dx \prod_j f_j(x_j)\, Z_\rho(0)a_k(x),
\end{equation}
and termwise differentiation at $\mu=0$ yields \eqref{e:Zrho-derivmu} since $Z_\rho(0)a_k(x)=Z_\rho(0)\partial^k_\mu|_{\mu=0}\tilde Z_\rho(\mu\chi)=\partial_\mu^k|_{\mu=0}Z_\rho(\mu \chi)$.
\end{proof}

\begin{lemma}\label{differentiabilitylogZ}
  For each $\chi$ as in \eqref{e:chi} and $\mu \in \R$, %
  $\log \tilde Z_\rho(\mu\chi)$ is differentiable in the distinct branch points $x_1,\dots, x_n$ (in particular continuous in this region),
  and $\mu \mapsto \partial_{x_j}\log \tilde Z_\rho(\mu\chi)$ has an analytic extension into the same neighborhood of the real axis as $S^\rho_{\mu\chi}$ does.
\end{lemma}

\begin{proof}
To simplify notation we write $S$ in place of $S^\rho_{\mu\chi}$ thoughout the proof.
By \eqref{derivfactorisation} we have
\begin{equation}
\frac{1}{2\pi}\partial_{x_j}S_{11}(z,w)=\alpha_j\lim_{z',w'\rightarrow x_j}\frac{S_{11}(z,w')}{\rho(w')}S_{21}(z',w)\rho(z').
\end{equation}
Using  \eqref{S11definition-bis}, \eqref{S21definition-bis} and Lemma~\ref{finerbound}, we see the bound
\begin{equation}
|\partial_{x_j}S_{11}(u+h,u)|\leq C \frac{1}{|x_j-u-h|^{|\alpha_j|}|x_j-u|^{1+|\alpha_j|}}
\end{equation}
for $u$ in a small ball in which $x_j$ lies and $h$ small. This is integrable for $x_j$ varying in this small ball, for fixed $h$ small. Hence Leibniz's integral rule gives
 \begin{equation}
 \partial_{x_j}\int_\C du\, \chi(u)\, (S^\rho_{s\chi,11}(u+h,u)-S_{s\chi,11}^1(u+h,u))=\int_\C du\, \chi(u)\, \partial_{x_j}S^\rho_{s\chi,11}(u+h,u).\label{tempdfgcxeq}
 \end{equation}
Now noting that \eqref{tempdfgcxeq} is bounded uniformly for $x_j$ in the small ball and all $h$ small, and noting analyticity, we have
\begin{align}  \label{eq:logZdervic}
  \partial_{x_j}\log \tilde Z_\rho(\mu\chi)
  &=\partial_{x_j}\int_0^\mu ds\lim_{h\to 0}\int_\C du\, \chi(u)\, (S^\rho_{s\chi,11}(u+h,u)-S_{s\chi,11}^1(u+h,u))+(1\leftrightarrow 2)\nnb
  &=\int_0^\mu ds\int_\C du\, \chi(u)\, (\partial_{x_j}S^\rho_{s\chi,11}(u,u)+\partial_{x_j}S^\rho_{s\chi,22}(u,u)).
\end{align}
Since the argument argument holds also for the antiholomorphic derivatives,
this shows differentiability of $\log \tilde Z_\rho(\mu\chi)$.
That the derivative has an analytic extension follows by analyticity of the derivative of the Green's function, which in turn follows by the factorization identity
(Proposition~\ref{prop:factorization})
and analyticity of the Green's function itself (Proposition~\ref{pr:analyt12}).
\end{proof}

\subsection{Proof of Theorem~\ref{thm:massive-bosonization}}

The proof of the theorem is by extending the left- and right-hand sides analytically to complex $\mu=Az$ and uniqueness of analytic continuation.

\begin{proposition}\label{prop:boso00}
  For $f_1,\dots,f_n\in C_c^\infty(\C)$ with disjoint supports and
  $\chi$ as in \eqref{e:chi},
  \begin{align}
    &\avg{\wick{e^{i\alpha_1\sqrt{4\pi}\varphi}}(f_1)\cdots \wick{e^{i\alpha_n \sqrt{4\pi}\varphi}}(f_n); \wick{\cos(\sqrt{4\pi}\varphi)}(\chi);\cdots; \wick{\cos(\sqrt{4\pi}\varphi)}(\chi)}_{\GFF(0)}^\mathsf T
      \nnb
    &\qquad =  \pb{2e^{-\gamma/2}}^{\sum_j \alpha_j^2} \pb{2\pi e^{-\gamma/2}}^{p} \int_{\C^{n}}dx\prod_{j=1}^n f(x_j)\partial_\mu^p|_{\mu=0}Z_\rho(\mu\chi),
  \end{align}
  where there are $p$ terms $\wick{\cos(\sqrt{4\pi}\varphi)}(\chi)$ on the left-hand side
  and we emphasize that there is no truncation between the $\wick{e^{i\alpha_j\sqrt{4\pi} \varphi}}$ terms.
\end{proposition} 

\begin{proof}
From \eqref{e:GFF-rho}, recall the definition of the expectation of the Gaussian free field with charge insertions encoded by $\rho$:
\begin{equation} \label{e:GFF-rho-bis}
  \avg{F(\varphi)}_{\GFF_\rho(0)}
  =\frac{\avg{ \prod_i \wick{e^{i\sqrt{4\pi}\alpha_i \varphi(x_i)}}F(\varphi)}_{\GFF(0)}}{\avg{\prod_i \wick{e^{i\sqrt{4\pi}\alpha_i \varphi(x_i)}}}_{\GFF(0)}}.
\end{equation}
Let us begin by noting that we have by \eqref{e:pfanalyticity},  the moments to cumulants formula, and multilinearity of the cumulants that
\begin{align}
  &\partial_\mu^p|_{\mu=0} \log Z_\rho(\mu\chi)\nnb
  &=\int_{\C}du\, \chi(u_1)\cdots \chi(u_p)\sum_{\pi\in \frP_p}(-1)^{|\pi|-1}(|\pi|-1)!\nnb
  &\qquad \times \left[\prod_{B\in \pi}\avg{\prod_{l\in B}(\psi_1\bar\psi_1(u_l)+\psi_2\bar\psi_2(u_l))}_{\FF_\rho(0)}-\prod_{B\in \pi}\avg{\prod_{l\in B}(\psi_1\bar\psi_1(u_l)+\psi_2\bar\psi_2(u_l))}_{\FF(0)}\right],
\end{align}
where $\frP_p$ denotes the set of  partitions of $\{1,\dots, p\}$.
By Proposition~\ref{le:boso0} (and $e^{ix}+e^{-ix}=2\cos x$), this equals
\begin{align}
  &\pB{\frac{e^{\gamma/2}}{4\pi}}^{p}\int_{\C}du\, \chi(u_1)\cdots \chi(u_p)\sum_{\pi\in \frP_p}(-1)^{|\pi|-1}(|\pi|-1)!\nnb
&\qquad \times \left[\prod_{B\in \pi}\avg{\prod_{l\in B}2\wick{\cos(\sqrt{4\pi}\varphi(u_l))}}_{\GFF_\rho(0)}-\prod_{B\in \pi}\avg{\prod_{l\in B}2\wick{\cos(\sqrt{4\pi}\varphi(u_l))}}_{\GFF(0)}\right].
\end{align}
Thus, by Faà di Bruno's formula,
\begin{align}
  \partial_\mu^p|_{\mu=0}Z_\rho(\mu\chi)
  &= Z_\rho(0) \pB{\frac{e^{\gamma/2}}{2\pi}}^p 
    \int_{\C^p}du\, \chi(u_1)\cdots \chi(u_p)\sum_{\pi\in \frP_p}\prod_{B\in \pi}\sum_{\tau\in \frP_B}(-1)^{|\tau|-1}(|\tau|-1)!\nnb
  &\quad\times \left[\prod_{C\in \tau}\avg{\prod_{l\in C}\wick{\cos(\sqrt{4\pi}\varphi(u_l))}}_{\GFF_\rho(0)}-\prod_{C\in \tau}\avg{\prod_{l\in C}\wick{\cos(\sqrt{4\pi}\varphi(u_l))}}_{\GFF(0)}\right],
\end{align}
where $\frP_B$ denotes the set of partitions of the finite set $B$.
Since, by the definition of $Z_\rho(0)$, see \eqref{Zrho0} and Lemma~\ref{lem:GFF-expi},
\begin{equation}
  Z_\rho(0) = (\frac12 e^{\gamma/2})^{\sum_j \alpha_j^2} \avg{\wick{e^{i\sqrt{4\pi}\alpha_1\varphi(x_1)}}\cdots \wick{e^{i\sqrt{4\pi}\alpha_n\varphi(x_n)}}}_{\GFF(0)},
\end{equation}
it suffices to show that
\begin{align}
&\avg{\wick{e^{i\sqrt{4\pi}\alpha_1\varphi(x_1)}}\cdots \wick{e^{i\sqrt{4\pi}\alpha_n\varphi(x_p)}}; \wick{\cos(\sqrt{4\pi}\varphi(u_1))};\cdots; \wick{\cos(\sqrt{4\pi}\varphi(u_n))}}_{\GFF(0)}^\mathsf T\nnb
&=\avg{\wick{e^{i\sqrt{4\pi}\alpha_1\varphi(x_1)}}\cdots \wick{e^{i\sqrt{4\pi}\alpha_n\varphi(x_n)}}}_{\GFF(0)}\sum_{\pi\in \frP_p}\prod_{B\in \pi}\sum_{\tau\in \frP_B}(-1)^{|\tau|-1}(|\tau|-1)!\nnb
&\quad\times \left[\prod_{C\in \tau}\avg{\prod_{l\in C}\wick{\cos(\sqrt{4\pi}\varphi(u_l))}}_{\GFF_\rho(0)}-\prod_{C\in \tau}\avg{\prod_{l\in C}\wick{\cos(\sqrt{4\pi}\varphi(u_l))}}_{\GFF(0)}\right].
\end{align}
Since  both sides are defined by the $\epsilon,m\to 0$ limit of the corresponding regularized objects, it is sufficient for us to prove that
\begin{align}
&\avg{\wick{e^{i\sqrt{4\pi}\alpha_1\varphi(x_1)}}_\epsilon \cdots \wick{e^{i\sqrt{4\pi}\alpha_n\varphi(x_n)}}_\epsilon; \wick{\cos(\sqrt{4\pi}\varphi(u_1))}_\epsilon;\dots; \wick{\cos(\sqrt{4\pi}\varphi(u_n))}_\epsilon}_{\GFF(m,\epsilon)}^\mathsf T\nnb
&=\avg{\wick{e^{i\sqrt{4\pi}\alpha_1\varphi(x_1)}}_\epsilon\cdots \wick{e^{i\sqrt{4\pi}\alpha_n\varphi(x_n)}}_\epsilon}_{\GFF(m,\epsilon)}\sum_{\pi\in \frP_p}\prod_{B\in \pi}\sum_{\tau\in \frP_B}(-1)^{|\tau|-1}(|\tau|-1)!\nnb
&\quad\times \left[\prod_{C\in \tau}\avg{\prod_{l\in C}\wick{\cos(\sqrt{4\pi}\varphi(u_l))}_\epsilon}_{\GFF_\rho(m,\epsilon)}-\prod_{C\in \tau}\avg{\prod_{l\in C}\wick{\cos(\sqrt{4\pi}\varphi(u_l))}_\epsilon}_{\GFF(m,\epsilon)}\right],
\end{align}
where  $\GFF_\rho(m,\epsilon)$ is defined in \eqref{e:GFF-rho-m-eps}.
Note that all the Wick ordered objects above are bounded random variables for fixed $\epsilon,m$, so we are dealing with correlation functions of honest (even bounded) random variables. Let us introduce some notation to simplify the proof. Let us write 
\begin{equation}
X=\frac{\wick{e^{i\sqrt{4\pi}\alpha_1\varphi(x_1)}}_\epsilon\cdots \wick{e^{i\sqrt{4\pi}\alpha_n\varphi(x_n)}}_\epsilon}{\avg{\wick{e^{i\sqrt{4\pi}\alpha_1\varphi(x_1)}}_\epsilon\cdots \wick{e^{i\sqrt{4\pi}\alpha_n\varphi(x_n)}}_\epsilon}_{\GFF(m,\epsilon)}}
\end{equation}
and $Y_j=\wick{\cos(\sqrt{4\pi}\varphi(u_j))}_\epsilon$ and note that $\avg{X}=1$. With this notation, our claim becomes 
\begin{align}
\avg{X;Y_1;\dots;Y_p}^\mathsf T=\sum_{\pi\in \frP_p}\prod_{B\in \pi}\sum_{\tau\in \frP_B}(-1)^{|\tau|-1}(|\tau|-1)!\left[\prod_{C\in \tau}\avg{X\prod_{l\in C}Y_l}-\prod_{C\in \tau}\avg{\prod_{l\in C}Y_l}\right],
\end{align}
where we dropped the subscript $\GFF(m,\epsilon)$ from $\avg{\cdots}$. 

This is actually a general identity for cumulants.
To see why it is true for bounded random variables, introduce the function 
\begin{equation}
L(t_1,\dots,t_p)=\log \frac{\avg{X e^{\sum_{j=1}^p t_j Y_j}}}{\avg{e^{\sum_{j=1}^p t_j Y_j}}}.
\end{equation}
Since our random variables are bounded and $\avg{X}=1$, this is analytic in some neighborhood of the origin and by the definition of cumulants,
\begin{equation}
\avg{X;Y_1;\dots;Y_p}^\mathsf T=\partial_{t_1}\cdots \partial_{t_p}|_{t=0} e^{L(t_1,\dots,t_p)}.
\end{equation}
On the other hand, by the Faà di Bruno formula, 
\begin{equation}
\partial_{t_1}\cdots \partial_{t_p}|_{t=0} e^{L(t_1,\dots,t_p)}=\sum_{\pi\in \frP_p}\prod_{B\in \pi}\prod_{j\in B}\partial_{t_j}|_{t=0}L(t_1,\dots,t_p).
\end{equation}
Comparing with our goal, we see that it suffices to prove that for each $B\subset \{1,..,p\}$ 
\begin{equation}
\prod_{j\in B}\partial_{t_j}|_{t=0}L(t_1,\dots,t_p)=\sum_{\tau\in \frP_B}(-1)^{|\tau|-1}(|\tau|-1)!\left[\prod_{C\in \tau}\avg{X\prod_{l\in C}Y_l}-\avg{\prod_{l\in C}Y_l}\right].
\end{equation}
In fact, it is even sufficient to prove that 
\begin{equation}
\prod_{j\in B}\partial_{t_j}|_{t=0}\log \avg{Xe^{\sum_{j=1}^p t_jY_j}}=\sum_{\tau\in \frP_B}(-1)^{|\tau|-1}(|\tau|-1)!\prod_{C\in \tau}\avg{X\prod_{l\in C}Y_l}
\end{equation}
since the denominator corresponds to setting $X=1$. This is now a straightforward application of Faà di Bruno and the fact that $\frac{d^n}{dx^n}\log x=(-1)^{n-1}(n-1)!x^{-n}$ combined with the fact that $\avg{X}=1$. This completes the proof.
\end{proof}

\begin{proof}[Proof of Theorem~\ref{thm:massive-bosonization}]
  By Theorem~\ref{thm:finvol-massless-corr} and Lemma~\ref{le:pfanalyticity}, both sides of \eqref{e:massive-bosonization} are analytic functions of $\mu=Az$ in a neighborhood of the real axis.
  It is thus sufficient to verify that they agree in a neighborhood of the origin, or equivalently that all of their derivatives at zero match. For the left-hand side, we know from Theorem~\ref{thm:finvol-massless-corr} that 
\begin{align}
&\pa{\prod_{i=1}^n C_\eta^{-\alpha_i^2}} \frac{1}{2^p} \ddp{^p}{z^p}\Big|_{z=0}\avg{\prod_{j=1}^n M_{\alpha_j}(f_j) }_{\SG(\sqrt{4\pi},z | 0,\Lambda)}\nnb
&= \avg{\wick{e^{i\alpha_1\sqrt{4\pi}\varphi}}(f_1)\cdots \wick{e^{i\alpha_p \sqrt{4\pi}\varphi}}(f_n); \wick{\cos(\sqrt{4\pi}\varphi)}(\chi);\dots ; \wick{\cos(\sqrt{4\pi}\varphi)}(\chi)}_{\GFF(0)}^\mathsf T.
\end{align}
From Proposition~\ref{prop:boso00} we know that this agrees (up to a constant)
with the corresponding derivative on the right-hand side.
Thus the theorem is proved.
\end{proof}

\section{Convergence to infinite-volume tau function}
\label{sec:Palmer}

In this section we show that the renormalized partition function $Z_\rho(\mu\chi)$, defined in Definition~\ref{def:Zrho},
converges in the infinite-volume limit $\chi \to 1$
to the tau function $\tau_\rho(\mu)$ defined in \cite[(4.9)]{MR1233355}.

This definition of the tau functions depends on monodromy parameters
$\alpha_1,\dots,\alpha_n \in (-\frac12,\frac12)$ and puncture points $x_1,\dots,x_n\in\C$ (both encoded by $\rho$)
and mass $\mu \neq 0$.
As in Section~\ref{sec:introferm}, even though they do not depend on the choice of branch cuts encoded by $\rho$,
we write $\tau_\rho(\mu)$.
The translation of our convention to that of \cite{MR1233355} is discussed in detail in Section~\ref{sec:Palmer-notation} below.

For the proof of the following theorem, all we need to know about the tau functions is that they are characterized (up to a multiplicative constant)
by Proposition~\ref{prop:logderivtaufnformula} below (and one might even consider this as their definition).

\begin{theorem} \label{thm:convergence-to-palmer}
  Let $\alpha_1,\dots,\alpha_n \in (-\frac12,\frac12)$ with $\sum_i\alpha_i=0$, and let $x_1,\dots, x_n \in \C$ be distinct.
  Let  $\mu \in \R$ and  $\chi(z)=\1_{\Lambda_L}(z)$ where $\Lambda_L$ is the open disk of radius $L$ centred at the origin. Then,
  in the limit $L\rightarrow \infty$,
\begin{equation}
  \partial_{x_j}\log Z_\rho(\mu\chi)\rightarrow \partial_{x_j}\log \tau_\rho(\mu),
\end{equation}
an analogous statement holds for the antiholomorphic derivatives,
and the convergence is uniform on compact subsets of $x_1,\dots, x_n \in \C$ distinct.
\end{theorem}

\begin{remark}
  The independence of $\tau_\rho(\mu)$ of the branch cuts does not seem to be explicitly stated in~\cite{MR1233355},
 but it now follows from our Proposition~\ref{prop:Zrho} for finite-volume versions
 of the tau functions and their convergence to the infinite-volume ones
 --  at least up to a multiplicative constant which is all that is relevant for us.
\end{remark}

The starting points for the proof of Theorem~\ref{thm:convergence-to-palmer} are the following two analogous formulas for the logarithmic derivatives
of the finite-volume renormalized partition function $Z_\rho(\mu\chi)$ from Definition~\ref{def:Zrho} and the tau functions from \cite{MR1233355}, proved in the remainder of this section.

\begin{proposition}\label{prop:logderivpartfnformula}
  Let the $x_i$ and $\alpha_i$ be as in Theorem~\ref{thm:convergence-to-palmer}.
  Then, for any $\chi$ as in \eqref{e:chi},
  \begin{align}\label{eq:logderivpartfnformula}
    \partial_{x_j}\log \tilde{Z}_\rho(\mu\chi)
    =-2\pi\alpha_j\lim_{z\rightarrow x_j}\lim_{w\rightarrow x_j}\frac{\rho(z)}{\rho(w)}(S_{21,\mu\chi}^\rho(z,w)-S_{21,0}^\rho(z,w)),
  \end{align}
  and $\bar\partial_{x_j}\log \tilde{Z}_\rho(\mu\chi) = \overline{\partial_{x_j}\log \tilde{Z}_{1/\rho}(\mu\chi)}$, and
  where $S^\rho_{\mu\chi}$ denotes the twisted Dirac Green's function with finite-volume mass from
  Definition~\ref{def:diracmass}. 
\end{proposition}

\begin{proposition}\label{prop:logderivtaufnformula}
  Let the $x_i$ and $\alpha_i$ be as in Theorem~\ref{thm:convergence-to-palmer}. Then for $\mu \in \R$,
 \begin{equation}
   \partial_{x_j}\log \tau_\rho(\mu)
   =
   -2\pi\alpha_j\lim_{z\rightarrow x_j}\lim_{w\rightarrow x_j}\frac{\rho(z)}{\rho(w)}(S_{21,\mu}^\rho(z,w)-S_{21,0}^\rho(z,w))
   +\alpha_j\sum_{i\neq j}\alpha_i\frac{1}{x_j-x_i},
   \label{holologderiv}
 \end{equation}
 and $\bar\partial_{x_\nu}\log\tau_\rho(\mu)= \overline{\partial_{x_\nu}\log\tau_{1/\rho}(\mu)}$, and
 where $S^\rho_{\mu}$ denotes the infinite-volume twisted Dirac Green's function from  Definition~\ref{def:diracmass-infvol}.
 The limits in $z,w$ are taken such that $z,w\in \C\setminus \Gamma$.
\end{proposition}

The resolvent identity together with the crude growth bound of Proposition~\ref{prop:Ldepenbounds}
implies that our finite-volume version of the Green's function converges to the infinite-volume one, including on the diagonal
if this is interpreted suitably as follows.

\begin{proposition}\label{diffatbpts}
Let $\chi=\1_{\Lambda_L}$
where $\Lambda_L$ is the open disk of radius $L$ centred at the origin. Then we have that 
\begin{equation}
\lim_{z\rightarrow x_j}\lim_{w\rightarrow x_j}\frac{\rho(z)}{\rho(w)}\Big(S^\rho_{\mu\chi}(z,w)-S^{\rho}_\mu(z,w)\Big)_{21}\rightarrow 0
\end{equation}
as $L\rightarrow \infty$, uniformly in compact subsets of $x_1,\dots, x_n$ distinct.
The limits in $z,w$ are taken such that $z,w\in \C\setminus \Gamma$.
\end{proposition}

By combining the above three propositions (proved in the remainder of this section), we first deduce 
Theorem~\ref{thm:convergence-to-palmer}.

\begin{proof}[Proof of Theorem~\ref{thm:convergence-to-palmer}]
By the symmetries under complex conjugation, it suffices to consider the holomorphic derivative and show that
\begin{equation}
  \partial_{x_j}\log Z_{\rho}(\mu\chi)\rightarrow \partial_{x_j}\log \tau_\rho(\mu).
\end{equation}
From Propositions~\ref{prop:logderivpartfnformula} and~\ref{prop:logderivtaufnformula} (and recalling Definition \ref{def:Zrho}) we have
\begin{align}
  \partial_{x_j}\log Z_{\rho}(\mu\chi)&=-2\pi\alpha_j\lim_{z\rightarrow x_j}\lim_{w\rightarrow x_j}\frac{\rho(z)}{\rho(w)}(S^\rho_{21,\mu\chi}(z,w)-S^\rho_{21,0}(z,w))+\partial_{x_j}\log Z_{\rho}(0)\\
  \partial_{x_j}\log \tau_{\rho}(\mu)&=-2\pi\alpha_j\lim_{z\rightarrow x_j}\lim_{w\rightarrow x_j}\frac{\rho(z)}{\rho(w)}(S^\rho_{21,\mu}(z,w)-S^\rho_{21,0}(z,w))+\partial_{x_j}\log Z_{\rho}(0)                                                ,\label{temp24weds}
\end{align}
where we used that
\begin{equation}
\partial_{x_j}\log Z_{\rho}(0)= \alpha_j\sum_{i\neq j}\alpha_i\frac{1}{x_j-x_i}
\end{equation}
Thus
\begin{equation}
  \partial_{x_j}\log Z_{\rho}(\mu\chi)-\partial_{x_j}\log \tau_\rho(\mu)
  =-\alpha_j2\pi\lim_{z\rightarrow x_j}\lim_{w\rightarrow x_j}\frac{\rho(z)}{\rho(w)}\Big(S^\rho_{21,\mu\chi}(z,w)-S^\rho_{21,\mu}(z,w)\Big),\label{temp34nf}
\end{equation}
which tends to $0$ as $L\to\infty$,
uniformly on compact subsets of $x_1,\dots,x_n$ distinct, by Proposition~\ref{diffatbpts}.
\end{proof}

\subsection{Logarithmic derivatives in finite volume -- proof of Proposition~\ref{prop:logderivpartfnformula}}

\begin{proof}[Proof of Proposition~\ref{prop:logderivpartfnformula}]
Since $\chi$ is fixed,
we write $S_s^\rho$ in place of $S_{s\chi}^\rho$ throughout the proof to simplify notation.
First observe that, as in \eqref{eq:logZdervic},
\begin{equation}
\partial_{x_j}\log \tilde{Z}_\rho(\mu\chi)=\int_0^\mu ds\int_{\C}du\, \chi(u)\,\Big(\partial_{x_j}(S_{s,11}^\rho(u,u)-S_{s,11}^1(u,u))+\partial_{x_j}(S_{s,22}^\rho(u,u)-S_{s,22}^1(u,u))\Big)
\end{equation}
 where we understand $S^{\rho}_{s,ii}(u,u)-S_{s,ii}^1(u,u)$ as $\lim_{h\to 0}(S^\rho_{s,ii}(u,u+h)-S_{s,ii}^1(u,u+h))$ for $i=1,2$. 
 By Proposition \ref{pr:analyt}, for $|s|$ sufficiently small, we have the expansion \eqref{eq:Spert}:
 \begin{align} \label{e:S11-expand-bis}
   S^\rho_{s,11}(z,w)&=-\overline{\rho(z)}\rho(w)\sum_{k=0}^\infty s^{2k+1}\int_{\C^{2k+1}} du\prod_{i=1}^k\frac{|\rho(u_{2i})|^2}{|\rho(u_{2i-1})|^2}\frac{1}{|\rho(u_{2k+1})|^2}
                  \frac{1}{(2\pi)^{2k+2}}
                                                                                                                                                                             \nnb
   &\qquad\qquad
     \times \frac{1}{\bar z-\bar u_1}                                                                                                                                                    \frac{1}{u_1-u_2}   \cdots\frac{1}{u_{2k+1}-w}\prod_{i=1}^{2k+1}\chi(u_i)
 \end{align}
 and, for later, we also need
 \begin{align}
   S^\rho_{s,21}(z,w)=\frac{\rho(w)}{\rho(z)}\sum_{k=0}^\infty s^{2k}\int_{\C^{2k}} du\prod_{i=1}^k\frac{|\rho(u_{2i-1})|^2}{|\rho(u_{2i})|^2}&\frac{1}{(2\pi)^{2k+1}}\frac{1}{z- u_1}\frac{1}{\bar u_1-\bar u_2}\nnb
   &\dots\frac{1}{u_{2k}-w}\prod_{i=1}^{2k}\chi(u_i).\label{S21expanstemp}
 \end{align}
Next note the calculation
\begin{align}
\partial_{x_j}\frac{\rho(u_{2i})}{\rho(u_{2i-1})}&=\alpha_{j}\frac{u_{2i}-u_{2i-1}}{(u_{2i}-x_j)(u_{2i-1}-x_j)}\frac{\rho(u_{2i})}{\rho(u_{2i-1})}.\label{temp40334}
\end{align}
So we see that $\partial_{x_j}(S_{s,11}^\rho(u,u)-S_{s,11}^1(u,u))$ is equal to
\begin{align}
-\sum_{k=0}^\infty s^{2k+1}\int_{\C^{2k+1}} du\sum_{q=1}^{k+1}\alpha_j \frac{u_{2q}-u_{2q-1}}{(u_{2q}-x_j)(u_{2q-1}-x_j)}\prod_{i=1}^{k+1}\frac{|\rho(u_{2i})|^2}{|\rho(u_{2i-1})|^2}\frac{1}{(2\pi)^{2k+2}}\nnb
\prod_{i=1}^{2k+1}\chi(u_i)\prod_{i=1}^{k+1}\frac{1}{ u_{2i-1}- u_{2i}}\frac{1}{\bar u_{2i-2}-\bar u_{2i-1}}
\end{align}
where we let $u_0=u_{2k+2}=u$,
and the interchange of the derivative and the sum and integral are justified the geometric convergence of the series following an
argument similar to
that used in the proof of Proposition~\ref{pr:analyt12}.
If we multiply by $\chi(u)$ and integrate (using Fubini whose application is again justified similarly), we see that
\begin{equation}
  \int du\, \chi(u)\, \partial_{x_j}(S_{s,11}^\rho(u,u)-S_{s,11}^1(u,u))
\end{equation}
equals
\begin{align}
-\sum_{k=0}^\infty s^{2k+1}\int_{\C^{2k+2}} du\,\sum_{q=1}^{k+1}\alpha_j \frac{u_{2q}-u_{2q-1}}{(u_{2q}-x_j)(u_{2q-1}-x_j)}\prod_{i=1}^{k+1}\frac{|\rho(u_{2i})|^2}{|\rho(u_{2i-1})|^2}\frac{1}{(2\pi)^{2k+2}}\nnb
\prod_{i=1}^{2k+2}\chi(u_i)\prod_{i=1}^{k+1}\frac{1}{ u_{2i-1}- u_{2i}}\frac{1}{\bar u_{2i-2}-\bar u_{2i-1}}\nnb
=\sum_{k=0}^\infty s^{2k+1}\sum_{q=1}^{k+1}\int_{\C^{2k+2}} du\alpha_j \frac{1}{(u_{2q}-x_j)(u_{2q-1}-x_j)}\prod_{i=1}^{k+1}\frac{|\rho(u_{2i})|^2}{|\rho(u_{2i-1})|^2}\frac{1}{(2\pi)^{2k+2}}\label{temp8459}\nnb
\frac{1}{\bar u_{2q-2}-\bar u_{2q-1}}\prod_{i=1}^{2k+2}\chi(u_i)\prod_{i\neq q}\frac{1}{ u_{2i-1}- u_{2i}}\frac{1}{\bar u_{2i-2}-\bar u_{2i-1}}.
\end{align}
Relabel the variables in the previous expression by $u_i\rightarrow u_{i'}$ where $i'=i-(2q-1)$ mod $2k+2$. Note that this changes $\prod_{i=1}^{k+1}\frac{|\rho(u_{2i})|^2}{|\rho(u_{2i-1})|^2}\rightarrow \prod_{i=1}^{k+1}\frac{|\rho(u_{2i-1})|^2}{|\rho(u_{2i})|^2}$. Then \eqref{temp8459} is equal to
\begin{equation}
-\alpha_j \sum_{k=0}^\infty s^{2k+1} (k+1)\int_{\C^{2k+2}} du\prod_{i=1}^{2k+2}\chi(u_i)\prod_{i=1}^{k+1}\frac{|\rho(u_{2i-1})|^2}{|\rho(u_{2i})|^2}\frac{1}{(2\pi)^{2k+2}}\frac{1}{x_j-u_1}\frac{1}{\bar u_1-\bar u_2}\dots \frac{1}{u_{2k+2}-x_j}
\end{equation}
which, using \eqref{S21expanstemp} and the geometric convergence of the sums as before, we see is equal to 
\begin{equation}
-2\pi\frac{\alpha_j}{2}\lim_{z \rightarrow x_j}\lim_{w\rightarrow x_j}\partial_s\frac{\rho(z)}{\rho(w)}S_{21,s}^\rho(z,w).
\end{equation}
So we have 
\begin{multline}
  \int_0^\mu ds\int_{\C}du\chi(u)\partial_{x_\nu}(S_{s,11}^\rho(u,u)-S_{s,11}^1(u,u))\\
  =-2\pi\frac{\alpha_j}{2}\lim_{z\rightarrow x_j}\lim_{w\rightarrow x_j}(\frac{\rho(z)}{\rho(w)}S_{21,\mu}^\rho(z,w)-\frac{\rho(z)}{\rho(w)}S_{21,0}^\rho(z,w)).
\end{multline}\\
Next, using the symmetry $S_{22}^\rho(z,w)=\overline{S_{11}^{1/\rho}}(z,w)$ (from Proposition \ref{prop:Greensym}),  for $|s|$ sufficiently small, we have the expansion
 \begin{align}
   S^\rho_{s,22}(z,w)=-\frac{1}{\rho(z)\overline{\rho(w)}}\sum_{k=0}^\infty s^{2k+1}\int_{\C^{2k+1}} du|\rho(u_{1})|^2\prod_{i=1}^k\frac{|\rho(u_{2i+1})|^2}{|\rho(u_{2i})|^2}&\frac{1}{(2\pi)^{2k+2}}\frac{1}{ z- u_1}\frac{1}{\bar u_1-\bar u_2}\nnb
   &\dots\frac{1}{\bar u_{2k+1}-\bar w}\prod_{i=1}^{2k+1}\chi(u_i).
 \end{align}
 Now use the formula \eqref{temp40334} but with variables $u_{2i}\rightarrow u_{2i+1}$, $u_{2i-1}\rightarrow u_{2i}$. We see that
 \begin{equation}
   \int du \chi(u)\partial_{x_j}(S^\rho_{s,22}(u,u)-S^1_{s,22}(u,u))
 \end{equation}
 has the expansion 
\begin{align}
-\sum_{k=0}^\infty s^{2k+1}\int_{\C^{2k+1}} du\sum_{q=0}^{k}\alpha_\nu \frac{u_{2q+1}-u_{2q}}{( u_{2q}-x_j)(u_{2q+1}- x_j)}\prod_{i=0}^{k}\frac{|\rho(u_{2i+1})|^2}{|\rho(u_{2i})|^2}\frac{1}{(2\pi)^{2k+2}}\nnb
\prod_{j=0}^{2k+1}\chi(u_i)\prod_{i=0}^{k}\frac{1}{  u_{2i}-  u_{2i+1}}\frac{1}{\bar u_{2i+1}-\bar u_{2i+2}}
\end{align}
where we let $u_0=u_{2k+2}=u$. Like before, take the sum over $q$ out of the integral, but this time relabel $u_i\rightarrow u_{i'}$ where $i'=i-2q$ mod $2k+2$. Notice that this does not flip the factors of $\rho$ here.
Now following the same argument as before we obtain
\begin{align}
-2\pi\frac{\alpha_j}{2} \lim_{z\rightarrow x_j}\lim_{w\rightarrow x_j}\partial_s\frac{\rho(z)}{\rho(w)}S_{21,s}^\rho(z,w).
\end{align}
Integrating in $s$ we see the result holds for all $|\mu|$ sufficiently small.  

The left- and right-hand sides of \eqref{eq:logderivpartfnformula} are both analytic in the same neighborhood of the real axis,
and we proved that they agree in a ball, so they agree everywhere in this neighborhood.
The analyticity of the left-hand side is Lemma~\ref{differentiabilitylogZ},
and that of the right-hand side follows from Proposition~\ref{pr:analyt},
that $\partial_s (\rho(z)/\rho(w)) S_{21,s}^\rho(z,w) = \partial_s \Delta_{21,s}^\rho(z,w)$ is continuous in $(z,w)$ by Proposition~\ref{prop:Green-Delta},
and Montel's theorem to take the limit $(z,w) \to (x_j,x_j)$.

The symmetry under complex conjugation is immediate from \eqref{e:Zrho-symmetries}.
\end{proof}

\subsection{Palmer's notation}
\label{sec:Palmer-notation}

Here we recall the notation and conventions used by Palmer \cite{MR1233355} to relate them to ours.
In \cite[Section 1]{MR1233355}, the massless Dirac operator is denoted
\begin{equation}
D =\begin{pmatrix} 0 &-\partial \\ -\bar\partial & 0\end{pmatrix},
\end{equation}
which relates to our definition \eqref{e:Dirac-def} of the massless Dirac operator $\Dirac$ by $\Dirac = -2\bar D$.
Our branch points $\{x_j\}$ and associated monodromy parameters $\{\alpha_j\}$
correspond in Palmer's notation to $\{a_j\}$ and $\{-\lambda_j\}$.
We will write $\{x_j\}$ rather than $\{a_j\}$ but distinguish between $\{\alpha_j\}$
and $\{\lambda_j\}$ to avoid confusion.
For $\sum_j \lambda_j=0$, Palmer defines his massive twisted Dirac operator
by\footnote{Note that on \cite[p.267]{MR1233355} the Dirac operator is defined with $mI$ instead of $\frac{m}{2}I$. This appears to
  be a typo and all of the remainder of the paper uses the convention $\frac{m}{2}$, see, e.g., \cite[p.276]{MR1233355}.}  
\begin{equation} \label{e:Palmer-Dirac}
D_{x,\lambda}(m) =\overline{P_\lambda} D P_\lambda+\frac{m}{2} I, \quad m>0,
\end{equation}
where $P_\lambda$ 
is defined as in \eqref{BranchingMat} except that $\alpha$ in $\rho$ from \eqref{eq:rho}
is replaced by $\lambda$ (and not $-\lambda$).
The domain of $D_{x,\lambda}(m)$ is
\begin{equation} 
\mathscr{D}_{x,\lambda}=\{f:\C\rightarrow \C^2; P_\lambda f\in H^1(\C;\C^2)\}
\end{equation}
where $H^1$ is the standard Sobolev space of square-integrable functions with square-integrable first derivatives.
Note that this slightly differs from the definition in \cite{MR1233355}, but is easily seen to be equivalent under the neutrality condition $\sum_j\lambda_j=0$.

Palmer defines his version of the Green's function $G^{x,\lambda}(z,w)$ of $D_{x,\lambda}(m)$ such that, for any function $g\in C_c^\infty(\C\setminus \Gamma;\C^2)$, 
\begin{equation} \label{e:Palmer-green}
f(z)=2 \int_{\C\setminus \Gamma}G^{x,\lambda}(z,w)Jg(w) \, dw
\end{equation} 
lies in $\mathscr{D}_{x,\lambda}$ and satisfies 
\begin{equation}
D_{x,\lambda}(m)f(z)=g(z),\label{Palmergreenisinv}
\end{equation}
where
\begin{equation}
J=\begin{pmatrix} 0 & i\\ -i & 0\end{pmatrix}.
\end{equation}
We emphasize the additional factor $2$ on the right-hand side of \eqref{e:Palmer-green} so that in our notation:
\begin{equation} \label{e:Dirac-identification}
  \Dirac_\rho+\mu
  =
  -2\overline{D_{x,\lambda}(m)}
  =
  -2\overline{D_{x,-\alpha}(m)}
\end{equation}
with
\begin{equation} \label{e:Palmer-m}
  \mu =- m, \qquad \lambda_j=-\alpha_j, \quad j=1,...,n,
\end{equation}
and
\begin{equation} \label{e:Green-identification}
  S_\mu^\rho(z,w) = (\Dirac_\rho + \mu)^{-1}(z,w)
  = (-2\overline{ D_{x,\lambda}(m)})^{-1}(z,w)
  =  \overline{G^{x,\lambda}(z,w)}J
  =  \overline{G^{x,-\alpha}(z,w)}J,
\end{equation}
and the left-hand sides correspond to our notation from Section~\ref{sec:Green}.
Formally,
\begin{equation}
  Z_\rho(\mu\chi) =  \det(\Dirac_\rho+\mu\chi), \qquad \tau^{\lambda}(x) = \det(D_{x,\lambda}(m)),
\end{equation}
and since %
$\Dirac_\rho+\mu=-2\overline{D_{x,-\alpha}(m)}$ we expect that, 
as $L\rightarrow \infty$,
\begin{equation}
  \partial_{x_j}\log Z_{\rho}(\mu\chi)\rightarrow \partial_{x_j}\log \tau_\rho(\mu) = \partial_{x_j}\log \overline{\tau^{\lambda}(x)}
\end{equation}
where $\tau^\lambda(x)$ refers to Palmer's notation in \cite{MR1233355}.
We make the definition
\begin{equation}
  \tau_\rho(\mu) =\overline{\tau^{\lambda}(x)}
  , \label{taufnidentifier}
\end{equation}
and we will see that, at least up to a constant, the right-hand side actually equals
\begin{equation}
  \tau^{-\lambda}(x) = \tau^{\alpha}(x).
\end{equation}

\begin{remark}  
 The main formula that connects our work to the tau function in \cite{MR1233355} is the logarithmic derivative of $\tau^\lambda(x)$, see below \eqref{e:tau-a}. We explain this in Section~\ref{subsectionlogderivtaufn}.
For the reader's orientation in Palmer's paper \cite{MR1233355},
we summarize his introduction of the tau function as a regularized determinant of $D_{x,\lambda}(m)$ as follows.
Since we do not use this, this can be skipped.

Palmer localises $D_{x,\lambda}(m)$ outside a collection of thin, disjoint horizontal strips $\{S_j\}$ that cover the branch cuts $\{\Gamma_j\}$. The localisations of different $D_{x,\lambda}$ 
(with $x$ varying in a small ball $B_\epsilon(x_0)$) are parametrized by subspaces of boundary conditions $W_x\subset H^{1/2}(\cup \partial S_j)$. The family $x\rightarrow W_x$ of subspaces lie inside a Grassmannian manifold which serves as a base for a holomorphic line bundle det* of the type introduced by Segal and Wilson \cite{MR783348},
see also the book \cite{MR900587}.
Palmer defines his tau function $\tau(x,x_0)$ as a ratio of certain sections in the det* line bundle, see \cite[Equation~4.9]{MR1233355}. This has the drawback of depending on the branch points $x_0$ for which Palmer uses to define the Grassmannian, and a principal result of Palmer \cite{MR1233355} is an expression for the logarithmic derivative of $\tau(x,x_0)$ with respect to the branch points that does not depend on $x_0$ (in fact, $\tau(x,x_0)$ actually factorises as $\tau(x)/\tau(x_0)$, see \cite[Equation~4.16]{MR1233355} where $\tau(x)$ is an explicit Hilbert-Schmidt determinant). Below, we show that the logarithmic derivative of $\tau(x)$ has an expression in terms of the Green's function of $D_{x,\lambda}(m)$, which we relate to the logarithmic derivative of our renormalized partition function.
\end{remark}

\subsection{Logarithmic derivatives of tau functions -- proof of Proposition~\ref{prop:logderivtaufnformula}}
\label{subsectionlogderivtaufn}

Using Palmer's notation \eqref{e:Green-identification} and the symmetries
of the Green's function stated in Proposition~\ref{prop:Greensym} (which also hold in the infinite-volume limit
by Proposition~\ref{finiteGreentoinfGreen}),
\begin{equation}
  S^\rho_{\mu,21}=\overline{S_{\mu,12}^{1/\rho}}
  =-iG_{11}^{x,-\lambda}
  =-iG_{11}^{x,\alpha}
\end{equation}
we have  that the claim \eqref{holologderiv} reads
\begin{align} 
  \partial_{x_j}\log \overline{\tau^\lambda(x)}
  =\partial_{x_j}\log\tau^{-\lambda}(x)
  &=-\lambda_j\lim_{z\to x_j}\lim_{w\to x_j}\left(2\pi i\rho(z)\rho(w)^{-1}G_{11}^{x,-\lambda}(z,w)+\frac{1}{z-x_j}\right)
    \nnb&\qquad+ \lambda_j\sum_{i\neq j}\lambda_i\frac{1}{x_j-x_i}.
          \label{holologderiv-Palmer}
\end{align}
Thus
\begin{equation}
  \partial_{x_j}\log \tau_\rho(\mu)
  = \partial_{x_j}\log\overline{\tau^{\lambda}(x)}
  =\partial_{x_j}\log \tau^{-\lambda}(x).
\end{equation}

\begin{proof}[Proof of Proposition~\ref{prop:logderivtaufnformula}]
We will prove \eqref{holologderiv-Palmer} and write $\lambda_j=-\alpha_j$.
It is shown in \cite[(2.6)]{MR1233355} that $j$-th column of Palmer's Green's function $G_{j,\cdot}^{x,\alpha}$ has the following expansion: for $z$ close to $x_\nu$, we have
\begin{align}\label{eq:Gasy}
G^{x,\alpha}_{j,\cdot}(z,w)=\sum_{\substack{k\in \Z+\frac{1}{2}:\\ k>0}}\left(a_{k,j}^\nu(z)w_{k+\alpha_\nu}(w-x_\nu)+b_{k,j}^\nu(z)w_{k-\alpha_\nu}^*(w-x_\nu)\right),
\end{align}
where $a_{k,j}^\nu$ and $b_{k,j}^\nu$ are currently unknown coefficient functions (we have a slightly different notation than in \cite{MR1233355}: Palmer writes $\alpha$ and $\beta$ for these coefficients and in \cite{MR1233355}, $a$ and $b$ are actually coefficients one obtains when one expands for $z$ close to $x_\nu$ instead) and for $z=re^{i\theta}$ we have
\begin{align}
w_l(z)=\begin{pmatrix}
e^{i(l-\frac{1}{2})\theta}I_{l-\frac{1}{2}}(m r)\\
e^{i(l+\frac{1}{2})\theta}I_{l+\frac{1}{2}}(m r)
\end{pmatrix} \qquad \text{and} \qquad 
w_l^*(z)=\begin{pmatrix}
e^{-i(l+\frac{1}{2})\theta}I_{l+\frac{1}{2}}(m r)\\
e^{-i(l-\frac{1}{2})\theta}I_{l-\frac{1}{2}}(m r)
\end{pmatrix},
\end{align}
where $I_\beta$ is a modified Bessel function of the first kind.

Then at the very end of \cite[Section 3]{MR1233355}, it is argued that for each $\nu\in \{1,2,\dots,n\}$, there exists a solution to the twisted Dirac equation, which is called $W_{\nu}(z,\alpha)$, and that close to $x_{\nu'}$ we have an expansion of the form 
\begin{equation} \label{eq:Wnu}
  W_{\nu}(z,\alpha)=\delta_{\nu,\nu'}w_{-\frac{1}{2}+\alpha_\nu}(z-x_{\nu'})+\sum_{\substack{k\in\Z+\frac{1}{2}:\\ k>0}}(\mathsf a_{k,\nu}^{\nu'}(x,\alpha)w_{k+\alpha_{\nu'}}(z-x_{\nu'})
  +\mathsf b_{k,\nu}^{\nu'}(x,\alpha)w_{k-\alpha_{\nu'}}^*(z-x_{\nu'})).
\end{equation}
Currently the only constraint for the coefficients $\mathsf a,\mathsf b$ is that $W_\nu$ needs to solve the twisted Dirac equation. Note that $W_\nu$ is a two-dimensional vector (since $w_l$ is).
On \cite[p.~315]{MR1233355} it is then argued that
\begin{align}\label{eq:Wanda}
\begin{pmatrix}
a_{\frac{1}{2},1}^\nu(z)\\
a_{\frac{1}{2},2}^\nu(z)
\end{pmatrix}=-\frac{im}{4 \sin \pi\alpha_\nu}W_\nu(z,-\alpha),
\end{align}
where $a_{\frac{1}{2},j}^\nu$ is the $k=\frac{1}{2}$-term from \eqref{eq:Gasy}.

The point of all this is that in \cite[Theorem 3]{MR1233355}, it is proven that the tau function satisfies
\begin{equation} \label{e:tau-a}
  \frac{\partial}{\partial x_\nu} \log \tau^\alpha(x)
  =\frac{m}{2}\mathsf a_{\frac{1}{2},\nu}^\nu(x,-\alpha).
\end{equation}
We now relate this to our setting. The first remark here is that since 
\begin{equation}
I_\beta(r)=(1+o(1))\frac{1}{\Gamma(1+\beta)}\left(\frac{r}{2}\right)^\beta,
\end{equation}
as $r\to 0$, we have as $z\to 0$,
\begin{align}
w_l(z)&=(1+o(1))\begin{pmatrix}
\frac{1}{\Gamma(l+\frac{1}{2})}(\frac{m}{2} z)^{l-\frac{1}{2}}\\
\frac{1}{\Gamma(l+\frac{3}{2})}(\tfrac{m}{2} z)^{l+\frac{1}{2}}
                \end{pmatrix}
        \\
  w_l^*(z)&=(1+o(1))\begin{pmatrix}
\frac{1}{\Gamma(l+\frac{3}{2})}(\tfrac{m}{2} \bar z)^{l+\frac{1}{2}}\\
\frac{1}{\Gamma(l+\frac{1}{2})}(\tfrac{m}{2} \bar z)^{l-\frac{1}{2}}
\end{pmatrix} .
\end{align}
This means that we have %
\begin{equation}
\mathsf a_{\frac{1}{2},\nu}^\nu(x,-\alpha)=\Gamma(1-\alpha_\nu)(\tfrac{m}{2})^{\alpha_\nu}\lim_{z\to 0} z^{\alpha_\nu}(W_{\nu,1}(x_\nu+z,-\alpha)-w_{-\frac{1}{2}-\alpha_\nu,1}(z)).
\end{equation}
We reformulate this expression as reading something off from the Green's function. For this, we go to \eqref{eq:Wanda}. From there, we have
\begin{align}
W_{\nu,1}(x_\nu+z,-\alpha)=\frac{4i \sin\pi\alpha_\nu}{m}a_{\frac{1}{2},1}^\nu(x_\nu+z).
\end{align}
Then recalling \eqref{eq:Gasy}, along with the asymptotic behavior of $w_l,w_l^*$, we see that 
\begin{equation}
a_{\frac{1}{2},1}^\nu(z)=\Gamma(1+\alpha_\nu)(\tfrac{m}{2})^{-\alpha_\nu}\lim_{w\to 0}w^{-\alpha_\nu}G_{11}^{x,\alpha}(x_\nu+z,x_\nu+w).
\end{equation}
This means that we have
\begin{align}
  \mathsf a_{\frac{1}{2},\nu}^\nu(x,-\alpha)
  &=\lim_{z\to 0}z^{\alpha_\nu}\bigg(\frac{4i\sin(\pi \alpha_\nu)\Gamma(1-\alpha_\nu)\Gamma(1+\alpha_\nu)}{m}\lim_{w\to 0}w^{-\alpha_\nu}G_{11}^{x,\alpha}(x_\nu+z,x_\nu+w)\nnb
  &\qquad \qquad -\Gamma(1-\alpha_\nu)(\tfrac{m}{2})^{\alpha_\nu}w_{-\frac{1}{2}-\alpha_\nu,1}(z)\bigg).
\end{align}
Using the identity 
\begin{equation}
\Gamma(1-\alpha_\nu)\Gamma(1+\alpha_\nu)=\frac{\pi\alpha_\nu}{\sin(\pi \alpha_\nu)} 
\end{equation}
along with the fact (coming from asymptotics of Bessel functions) that 
\begin{align}
z^{\alpha_\nu}\Gamma(1-\alpha_\nu)(\tfrac{m}{2})^{\alpha_\nu}w_{-\frac{1}{2}-\alpha_\nu,1}(z)&=z^{\alpha_\nu}\Gamma(1-\alpha_\nu)(\tfrac{m}{2})^{\alpha_\nu}\left(\frac{1}{\Gamma(-\alpha_\nu)}(\tfrac{zm}{2})^{-1-\alpha_\nu}+O(|z|^{1-\alpha_\nu})\right)\nnb
&=-\alpha_\nu \frac{2}{zm}+o(1),
\end{align}
we find that 
\begin{equation}
  \mathsf a_{\frac{1}{2},\nu}^\nu(x, -\alpha)
  =2\frac{\alpha_\nu}{m}\lim_{z,w\to x_\nu}\left(2\pi i (z-x_\nu)^{\alpha_\nu}(w-x_\nu)^{-\alpha_\nu}G_{11}^{x,\alpha}(z,w)+\frac{1}{z}\right),
\end{equation}
or in other words, 
\begin{equation}
\partial_{x_j}\log \tau^\alpha(x)=\alpha_j\lim_{z,w\to x_j}\left(2\pi i\rho(z)\rho(w)^{-1}G_{11}^{x,\alpha}(z,w)+\frac{1}{z-x_j}\right)+\alpha_j\sum_{i\neq j}\alpha_i\frac{1}{x_j-x_i}\label{tempt3frzed}
\end{equation}
where the second term in \eqref{tempt3frzed} comes from 
\begin{align}
\lim_{z,w\rightarrow x_j}\frac{1}{z-x_j}\Big(\frac{\rho(z)}{\rho(w)}\frac{(w-x_j)^{\alpha_j}}{(z-x_j)^{\alpha_j}}-1\Big)=\sum_{i\neq j}\alpha_i\frac{1}{x_j-x_i}.
\end{align}
This proves the first equality \eqref{holologderiv}. From \cite[Theorem~4.3]{MR1233355} and \eqref{e:tau-a} we also have
\begin{equation}
  \bar\partial_{x_\nu}\log\tau^\alpha(x)=\frac{m}{2}\overline{\mathsf a_{\frac{1}{2},\nu}^\nu(x, \alpha)}
  = \overline{\partial_{x_\nu}\log\tau^{-\alpha}(x)}.\label{temp5tgzbn}
\end{equation}
Now \eqref{holologderiv-Palmer} follows from \eqref{tempt3frzed} since by definition \eqref{taufnidentifier} we have
\begin{equation}
  \partial_{x_j}\log\overline{\tau^{-\alpha}(x)}
  =\overline{\overline{\partial_{x_j}}\log\tau^{-\alpha}(x)}
  =\partial_{x_j}\log \tau^\alpha(x)
\end{equation}
where the last equality is  by \eqref{temp5tgzbn}.
\end{proof}

\subsection{Resolvent identity and the proof of Proposition~\ref{diffatbpts}}
 \label{ResolventIdentitysection}
We have the following resolvent type identity involving our Green's function $S$ of the Dirac operator with branching.

\begin{proposition}\label{prop-resolventidentity}
  Let $\chi(z)=\1_{\Lambda_L}(z)$ %
  where $\Lambda_L$ is the open disk of radius $L$ centred at the origin.
  Then, for $z,w\in \C\setminus \Gamma, z\neq w$,
 \begin{equation}
   S^\rho_{\mu\chi}(z,w)-S^\rho_\mu(z,w) = \mu\int_{\Lambda_L^c}S^\rho_{\mu\chi}(z,z') S^{\rho}_\mu(z',w) \, dz'.
   \label{resolventidentity}
 \end{equation}
 for all $L$ sufficiently large (depending on $x,w$).
\end{proposition}

\begin{proof}
In Palmer's notation for the infinite-volume Green's function, \eqref{resolventidentity} reads
\begin{equation}
  S^\rho_{\mu\chi}(z,w)-\overline{G^{x,-\alpha}}(z,w)J= \mu\int_{\Lambda_L^c}S^\rho_{\mu\chi}(z,z') \overline{G^{x,-\alpha}}(z',w)J \, dz',
 \label{resolventidentity-Palmer}
\end{equation}
and we use this notation in the proof. 
To simplify notation we also write $S$ instead of $S^\rho_{\mu\chi}$ thoughout the proof.
For $g\in C_c^\infty (\C\setminus \Gamma;\R\times \R)$, $z\in \C\setminus \Gamma$, define 
\begin{equation}
H_g(z)=\mu\int_{\Lambda_L^c}S(z,z')\int \overline{G^{x,-\alpha}}(z',y)Jg(y)\,  dy\, dz'.
\end{equation}

To see that the integral over $z'$ converges, observe that for $L>0$ sufficiently large, $v_\alpha(z)=\int G^{x,\alpha}(z,y)J g(y)\, dy$ is a solution to the Dirac equation $D_{x,\alpha}(m)v_\alpha(z)=0$.
By G1 in \cite[p.278]{MR1233355} and $\sum_j\alpha_j=0$ we have $v_\alpha(e^{2\pi i}z)=v_\alpha(z)$ for all $|z|\geq L$.
By G2 in \cite[p.279]{MR1233355}, $\sup_{|z|\geq L}|v_\alpha(z)|<\infty$,
and so \cite[Proposition 3.1.5]{MR555666} %
gives $|v_\alpha(z)|\leq C e^{-m|z|}$ for some constant $C>0$ for all $|z|\geq L$. 
By Proposition~\ref{prop:Greensym}, we see that $S(z,\cdot)$ also vanishes at infinity. Thus $|H_g(z)|\leq |\mu| \int_{\Lambda_L^c}|S(z,z')||v_{-\alpha}(z')|dz'$ converges. 
Next we see the following facts:
\begin{align}
(\slashed \partial+\mu\chi)H_g(z)=\mu \1_{\Lambda_L^c}(z)\int \overline{G^{x,-\alpha}}(z,y)Jg(y)dy,
\end{align}
and \begin{align}
\lim_{z\rightarrow \infty}H_g(z)=0.
\end{align}
The latter of which follows from using \eqref{tempineq12-bis} to obtain the bound
$|H_g(z)|\leq C\int_{\Lambda_L^c}\frac{1}{|z-z'|}|v_{-\alpha}(z')|dz'$ for $|z|$ large and the exponential decay of $v$. Furthermore, we see that the second row of \eqref{Palmergreenisinv} can be rewritten
\begin{equation}\label{temp54tuyj|}
\bar\partial(\rho(z)f_1(z))=\frac{m}{2}|\rho(z)|^2\frac{f_2(z)}{\bar\rho(z)}-\rho(z)g_2(z),
\end{equation}
where we recall that $\mu=-m$.
Since $\rho(z)f_1(z), f_2(z)/\bar\rho(z)\in H^1(\C)$ and $|\rho(z)|^2\in L^{2+\epsilon}_{\loc}(\C)$ for small enough $\epsilon>0$,
it follows from the above and Sobolev's embedding that $f_2(z)/\bar\rho(z)\in L^q_{\loc}(\C)$ for all finite $q$ which together with \cite[Proposition 3.1.5]{MR555666} and that $f$ solves the Dirac equation outside of a large ball, we get $f_2(z)/\bar\rho(z)\in L^q(\C)$. Then by H\"older's inequality, the right-hand side of \eqref{temp54tuyj|}
is in $L^r(\C)$ for some $r>2$.
Hence the solution $\rho(z)f_1(z)$ is given by the Cauchy transform following a similar argument as found in Lemma~\ref{le:homogdirac}, and thus is H\"older continuous with parameter $\alpha=1-2/r$ by \eqref{eq:cauchy2}. A similar argument on the first row yields that $P_\alpha(z)f(z)$ is continuous.
Hence $\int\overline{P}_{-\alpha}(\cdot)\overline{G^{x,-\alpha}}(\cdot,y)g(y)dy$ is actually continuous on $\C$. A short calculation gives
\begin{align} 
& (\slashed\partial +\mu\chi)\int (S(z,y)-\overline{G^{x,-\alpha}}(z,y)J)g(y)dy
\nnb
&= (\slashed\partial +\mu\chi-(-2(\overline D+\frac{m}{2})))\int -\overline{G^{x,-\alpha}}(z,y)Jg(y)dy
\nnb
&=\mu \1_{\Lambda_L^c}(z)\int \overline{G^{x,-\alpha}}(z,y)Jg(y)dy,
\end{align}
which holds off of $\Gamma$ and where we used that $\slashed \partial+\mu=-2(\overline D+m/2)$.
Thus we are in the setting of Lemma \ref{le:homogdirac}, which gives us
\begin{equation}
  H_g(z)=\int (S(z,y)-\overline{G^{x,-\alpha}}(z,y)J)g(y)dy.
\end{equation}
Which in turn gives, for $z,w\in\C\setminus \Gamma$, $z\neq w$ with $|w|<L-1$,
 \begin{equation}
S(z,w)-\overline{G^{x,-\alpha}}(z,w)J= \mu\int_{\Lambda_L^c\setminus \Gamma}S(z,z') \overline{G^{x,-\alpha}}(z',w)Jdz'.
 \end{equation}
 as claimed.
\end{proof}

\begin{proof}[Proof of Proposition~\ref{diffatbpts}]
 By taking the complex conjugate and flipping the signs of the monodromy parameters in \eqref{resolventidentity}, we see that it is equivalent to show 
 \begin{align}
 \lim_{z,w\rightarrow x_j}\frac{\rho(z)}{\rho(w)}\int_{\Lambda_L^c}\Big(\overline{S^{1/\rho}_{11}}(z,z') (G^{x,\alpha}(z',w)\overline{J})_{11}+\overline{S_{12}^{1/\rho}}(z,z')(G^{x,\alpha}(z',w)\overline{J})_{21}\Big)dz'\label{temp35bgre}
 \end{align}
 tends to zero as $L\rightarrow \infty$. For $i\in \{1,2\}$, observe that via the expansion \eqref{eq:Gasy},
\begin{align}
  \lim_{w\rightarrow x_j}\frac{\Big(G^{x,\alpha}(z',w)\overline{J}\Big)_{i,2}}{(w-x_j)^{\alpha_j}}
  &=\lim_{w\rightarrow x_j}\frac{-iG^{x,\alpha}_{i,1}(z',w)}{(w-x_j)^{\alpha_j}}
  \nnb
  &=-i(\frac{m}{2})^{\alpha_j}\frac{a_{\frac{1}{2},i}^{j}(z')}{\Gamma(1+\alpha_j)}
  =-\frac{m^{1+\alpha_j}W_{j,i}(z',-\alpha)}{2^{\alpha_j}4\sin(\pi\alpha_j)\Gamma(1+\alpha_j)}
\end{align}
where the last equality is \eqref{eq:Wanda}.
Here $W_j(z,\alpha)$ are multivalued solutions to the Dirac equation which are square-integrable at infinity.
By \cite[Proposition~1.2]{MR1233355}, we have
\begin{align}
W_j(z,\alpha)=O(e^{-m|z|})\label{expboundWjz}
\end{align}
for $|z|$ sufficiently large.
As we prove in Lemma~\ref{le:wjreg} in Appendix~\ref{app:SMJ},
this exponential decay is uniform in compact subsets of the branch points $x_1,\dots,x_n$ distinct.  We use the symmetries $\overline{S_{11}(z,w)}=S_{11}(w,z)$ and $\overline{S_{21}(z,w)}=-S_{12}(w,z)$
from Proposition~\ref{prop:Greensym} in \eqref{temp35bgre} and then use Proposition~\ref{prop:Green-Delta} to write  $S$ in terms of its singular part and its regular part $\Delta$. Hence, using \eqref{expboundWjz},
we see that for $L$ large enough, the modulus of \eqref{temp35bgre} is bounded above by
\begin{align}
  C\int_{\Lambda_L^c}\Big(|\rho(z')|(\frac{1}{|z'-x_j|}+|\Delta_{21}^{1/\rho}(z',x_j)|)+\frac{|\mu|}{(2\pi)^2|\rho(z')|}|\int du \chi(u)|\rho(u)|^2\frac{1}{x_j-u}\frac{1}{\bar u -\bar z'}|
  \nnb
  +\frac{1}{|\rho(z')|}|\Delta_{11}^{1/\rho}(z',x_j)|\Big)e^{-m|z'|}\, dz'
  .\label{temp76uxck}
\end{align}
Now note that $|\rho(z')|$ is bounded for $|z'|$ large (due to the condition $\sum_{j}\alpha_j=0$).
It is also easy to prove the (cheap) bound
\begin{equation}
|\int du \, \chi(u)|\rho(u)|^2\frac{1}{x_j-u}\frac{1}{\bar u -\bar z'}|\leq\int du \, \frac{\chi(u)|\rho(u)|^2}{|x_j-u|}
\leq C'L
\end{equation}
for some $C'>0$, for all $|z'|>L+1$, so we see its contribution in \eqref{temp76uxck} tends to zero.
So for $L$ large enough, the modulus of \eqref{temp35bgre} is bounded above by
\begin{equation}
C''L\int_{\Lambda_{L}^c}\big(|\Delta_{21}^{1/\rho}(z',x_j)|+|\Delta_{11}^{1/\rho}(z',x_j)|\big)e^{-m|z'|}\, dz'
\end{equation}
for some $C''>0$ fixed in $L$. By Proposition \ref{prop:Ldepenbounds}, this tends to zero as $L\rightarrow \infty$.
\end{proof}

\begin{proof}[Proof of Proposition~\ref{finiteGreentoinfGreen}]
Just as in the previous proof, we use the resolvent identity \eqref{resolventidentity} and show that 
\begin{equation}
 \mu\int_{\Lambda_L^c}S^\rho_{\mu\chi}(z,z') S^{\rho}_\mu(z',w) \, dz'\label{temp54thg}
\end{equation}
tends to zero in the limit $L \rightarrow \infty$. Note that Palmer's Green's function is a multivalued solution to the Dirac equation which is square-integrable at infinity, so again by \cite[Proposition~1.2]{MR1233355},
for $|z'|$ sufficiently large, we have
\begin{equation}
G^{x,\alpha}(z',w)=O(e^{-m|z'|}).
\end{equation}
Using the symmetries in Proposition \ref{prop:Greensym}, we see \eqref{temp54thg} is bounded above for large enough $L$ by 
\begin{equation}
  C\int_{\Lambda_L^c}\begin{pmatrix} |S_{11}^\rho(z,z')| & |\overline{S_{21}^{1/\rho}(z,z')}|\\ |S_{21}^\rho(z,z')| & |\overline{S_{11}^{1/\rho} (z,z')}|
                     \end{pmatrix}e^{-m|z|}\begin{pmatrix} 1 & 1\\ 1 & 1\end{pmatrix} \, dz'.\label{tem346uyhxc}
\end{equation}
Now we use Proposition \ref{prop:Green-Delta} and observe that, because $\rho(z)$ is bounded, after multiplying the matrices in \eqref{tem346uyhxc}, each element in the resulting matrix is bounded above by an integral of the form \eqref{temp76uxck}. Each of these integrals then tends to zero by the same argument following \eqref{temp76uxck}.

The symmetry \eqref{Greensym} for $\chi=1$ just follows by convergence.
\end{proof}

\section{Proof of main results}
\label{sec:main-proofs}

\subsection{Proof of Theorems~\ref{thm:main-tau} and~\ref{thm:Mexists}}

The existence of the IMC with respect to $\nu^{\SG(4\pi,z)}$ stated in Theorem~\ref{thm:Mexists}
immediately follows from  Theorem~\ref{thm:sg-infvol}.
It also follows from Theorem~\ref{thm:sg-infvol} that
\begin{equation} \label{e:main-pf1}
  \avg{\prod_{i=1}^n M_{\alpha_i}(f_i)}_{\SG(4\pi,z)}
  = \lim_{\Lambda\to\R^2} \avg{\prod_{i=1}^n M_{\alpha_i}(f_i)}_{\SG(4\pi,z|\Lambda,0)},
\end{equation}
where the limit $\Lambda \to \R^2$ is along a suitable sequence of $\Lambda$ as in the statement of
Theorem~\ref{thm:sg-infvol}.
By Theorem~\ref{thm:massive-bosonization}, the right-hand side is equal to
\begin{equation} \label{e:main-pf2}
\avg{\prod_{i=1}^n M_{\alpha_i}(f_i)}_{\SG(4\pi,z|\Lambda,0)} = \int dx_1 \cdots dx_n\; f_1(x_1) \cdots f_n(x_n) \, Z_\rho(\mu \1_{\Lambda}).
\end{equation}
By Theorem~\ref{thm:convergence-to-palmer}, uniformly on nonoverlapping compact subsets of $x_i \in \C$, %
and $\alpha_i \in (-\frac12,\frac12)$ with $\sum_i \alpha_i=0$, as $L\to \infty$,
\begin{equation}
  \nabla \log  Z_\rho(\mu\1_{\Lambda_L})  \to \nabla \log \tau_\rho(\mu),
\end{equation}
where $\nabla = (\nabla_{x_1}, \cdots, \nabla_{x_n})$. It follows that
\begin{align}
 Z_\rho(\mu\1_{\Lambda_L}) \times   \frac{\tau_{\rho_0}(\mu)}{Z_{\rho_0}(\mu \1_{\Lambda_L})}  \to \tau_\rho(\mu)
\end{align}
where $\rho_0$ corresponds to some fixed distinct punctures $x_1^0, \dots, x_n^0 \in \C$, %
and the convergence is also uniform on compact sets.
In particular,
\begin{equation}
  \frac{\tau_{\rho_0}(\mu)}{Z_{\rho_0}(\mu \1_{\Lambda_L})}  \avg{\prod_{i=1}^n M_{\alpha_i}(f_i)}_{\SG(4\pi,z|\Lambda_L,0)} \to
  \int dx_1 \cdots dx_n\; f_1(x_1) \cdots f_n(x_n) \, \tau_\rho(\mu).\label{temp65bvt}
\end{equation}
Choosing $f_1,\dots,f_n$ so that $\avg{\prod_{i=1}^n M_{\alpha_i}(f_i)}_{\SG(4\pi,z)}  \neq 0$ and so that the right-hand side above is nonzero,
by using \eqref{e:main-pf1},
we also conclude that $\tau_{\rho_0}(\mu)/Z_{\rho_0}(\mu \1_{\Lambda_L})$ converges to a constant $C(\alpha,\mu)$.
Putting everything together, %
\begin{equation}
  \avg{\prod_{i=1}^n M_{\alpha_i}(f_i)}_{\SG(4\pi,z)} = C(\alpha,\mu)   \int dx_1 \cdots dx_n\; f_1(x_1) \cdots f_n(x_n) \, \tau_\rho(\mu),
\end{equation}
concluding the proof.

\subsection{Proof of Corollary~\ref{cor:BasorTracy}: Basor--Tracy asymptotics}
\label{sec:BasorTracy}

\begin{proof}[Proof of Corollary~\ref{cor:BasorTracy}]
The logarithm of the Fredholm determinant is given by
\begin{align}
&\log\det(I-K)\nnb
&=-\sum_{n=1}^\infty \frac{1}{n} \mathrm{Tr}K^n\nnb
&=-\sum_{n=1}^\infty \frac{(-1)^n}{n}\left(\frac{\sin^2(\pi \alpha)}{\pi^2}\right)^n \int_{[0,\infty)^n}da \int_{[0,\infty)^n}du \frac{e^{-|x-y|\sum_{j=1}^n \omega(a_j)-|x-y|\sum_{j=1}^n \omega(u_j)}}{\prod_{i=1}^n(a_i+u_i)(u_i+a_{i+1})}\prod_{i=1}^n \left(\frac{a_i a_{i+1}}{u_i^2}\right)^\alpha 
\end{align}
with the interpretation that $a_{n+1}=a_1$.
If we rename our integration variables as follows:
\begin{equation}
v_{2i-1}=a_i, \qquad
v_{2i}=u_i,
\end{equation}
then we can write this as 
\begin{align}
&\log\det(I-K)\nnb
&=-\sum_{n=1}^\infty \frac{(-1)^n}{n}\left(\frac{\sin^2(\pi \alpha)}{\pi^2}\right)^n \int_{[0,\infty)^{2n}}dv \frac{e^{-|x-y|\sum_{j=1}^{2n} \omega(v_j)}}{\prod_{i=1}^{2n}(v_i+v_{i+1})}\left(\frac{\prod_{i=1}^n v_{2i-1}}{\prod_{i=1}^n v_{2i}}\right)^{2\alpha} 
\end{align}
with the understanding that $v_{2n+1}=v_1$.
Basor--Tracy \cite{MR1187544} have studied the $|x-y|\to 0$ asymptotics of this Fredholm determinant.
Indeed, they study  the function %
\begin{align}
\tau(t,\theta,\lambda)=\exp\left(-\sum_{k=1}^\infty \frac{\lambda^{2k}}{k}c_{2k}(t,\theta)\right),
\end{align}
where 
\begin{align}
  c_{2k}(t,\theta)=\int_{[0,\infty)^{2k}}dv \prod_{j=1}^{2k}\frac{\exp(-\frac{t}{2}(v_j+v_j^{-1})}{v_j+v_{j+1}}\left(\frac{v_1 v_3\cdots v_{2k-1}}{v_2v_4\cdots v_{2k}}\right)^\theta
  .
\end{align}
Therefore $\det(I-K)=\tau(t,\theta,\lambda)$ with the choices 
\begin{equation}
t=\mu  |x-y|, \qquad
\theta =2\alpha, \qquad
\lambda^2=-\frac{\sin^2(\pi\alpha)}{\pi^2},
\end{equation}
so $\lambda$ is purely imaginary.
The main result of \cite{MR1187544} is as follows: assume that $2\pi^2 \lambda^2-\cos \pi\theta<1$ (they actually allow for complex parameters and the only assumption is that this is not a real number greater or equal to $1$) and let $\sigma$ be defined as the solution to 
\begin{equation}
\pi^2 \lambda^2=\sin \left(\frac{\pi}{2}(\sigma+\theta)\right)\sin \left(\frac{\pi}{2}(\sigma-\theta)\right)
\end{equation}
with the constraint $\mathrm{Re}(\sigma)\in[0,1)$, then as $t\to 0$, 
\begin{align}
\tau(t,\theta,\lambda)=(1+o(1))\tau_0(\theta,\lambda)t^{\frac{1}{2}(\sigma^2-\theta^2)},
\end{align}
where 
\begin{align}
\tau_0(\theta,\lambda)=2^{-2(p^2-q^2)}\frac{G(1+p+q)G(1+p-q)G(1-p+q)G(1-p-q)}{G(1+2p)G(1-2p)},
\end{align}
where 
\begin{align}
2p =\sigma \qquad \text{and} \qquad 2q=\theta.
\end{align}

Note that in our setting, since $\lambda^2<0$, the constraint $2\pi^2\lambda^2-\cos\pi\theta<1$ is automatically satisfied.  Moreover, the equation defining $\sigma$ becomes 
\begin{align}
-\sin^2(\pi\alpha)&=\sin \left(\frac{\pi}{2}(\sigma+2\alpha)\right)\sin \left(\frac{\pi}{2}(\sigma-2\alpha)\right)\nnb
&=\frac{1}{2}(\cos (2\pi \alpha)-\cos(\pi \sigma))\nnb
&=\frac{1}{2}(-1+2\cos^2(\pi\alpha)-\cos(\pi\sigma)),
\end{align}
from which we find 
\begin{align}
-1=-\frac{1}{2}(1+\cos(\pi\sigma)),
\end{align}
or in other words
\begin{align}
\cos(\pi\sigma)=1.
\end{align}
Thus $\sigma=0$ is our only solution. We conclude that as $|x-y|\to 0$, 
\begin{equation}
\det(I-K)=(1+o(1))\left(\mu |x-y|\right)^{-2\alpha^2}2^{2\alpha^2}G(1+\alpha)^2G(1-\alpha)^2.
\end{equation}
\end{proof}

\subsection{Proof of Corollary~\ref{cor:main-diffeq}: Bernard--LeClair differential equation}
\label{sec:BernardLeclair}

\begin{proof}[Proof of Corollary~\ref{cor:main-diffeq}]
In \cite[(5.37) and (5.40)]{MR1233355}, it is shown that
$\tau = \tau^\alpha$ with $|x_1-x_2|=r$ and
$\alpha=(-\frac12\lambda,\frac12 \lambda)$, where we note that $\lambda = \lambda_2-\lambda_1$ explaining the factor $\frac12$ (see \cite[below (5.27)]{MR1233355}),
\begin{align}
  \label{e:palmer-diffeq1}
  \ddp{}{r} \log \tau &= - \frac{1}{2r} \pa{ r^2 \psi'(r)^2 - \lambda^2 \tan^2(\psi) - m^2 r^2 \sin^2(\psi)}
                        \intertext{where}
                        \label{e:palmer-diffeq2}
                        r \ddp{}{r} (r \psi'(r))  &= \lambda^2 \sec^2(\psi) \tan(\psi) + \frac{m^2r^2}{2} \sin (2\psi).
\end{align}
We rewrite the differential equations %
as stated in Corollary~\ref{cor:main-diffeq}.
Using that $\Delta = \ddp{^2}{r^2} + \frac{1}{r}\ddp{}{r}$ on radially symmetric functions, the second equation reads
\begin{equation}
  \Delta\psi
  = \frac{m^2}{2} \sin(2\psi) + \frac{\lambda^2}{r^2} \sec^2(\psi)\tan(\psi)
  = \frac{m^2}{2} \sin(2\psi) + \frac{\lambda^2}{r^2}\tan(\psi)(1+\tan^2(\psi))
  ,
\end{equation}
The first equation gives
\begin{align}
  (\log \tau)'' = -\frac{1}{r} (\log \tau) &- \frac{1}{2r}\pa{2r \psi'(r)^2 + 2r^2 \psi'(r)\psi''(r) - \lambda^2 (\tan^2(\psi))' \psi'} \nnb
  &-\frac{1}{2r}\pa{ - 2m^2 r \sin^2(\psi) - 2m^2 r^2 \sin(\psi)\cos(\psi)\psi'}
\end{align}
so
\begin{align}
  \Delta (\log \tau)
  &=
    -\psi'(r)^2 - r \psi'(r)\psi''(r) + \frac{\lambda^2}{2r} (\tan^2(\psi))' \psi' + m^2 \sin^2(\psi) + m^2 r \sin(\psi)\cos(\psi)\psi'
    \nnb
  &=
    -r\psi'(r) (\Delta \psi - \frac{\lambda^2}{2r^2} (\tan^2(\psi))' - m^2 \sin(\psi)\cos(\psi)) + m^2 \sin^2(\psi)
    \nnb
  &=
    m^2 \sin^2(\psi)
    = \frac{m^2}{2}(1-\cos(2\psi)).
\end{align}
Since Palmer's $m$ is related to our $\mu$ by $m=-\mu$, see \eqref{e:Palmer-m}, this implies the claim.
\end{proof}

\appendix
\section{Determinants of massless twisted Dirac operators}
\label{app:det}

In this appendix, we give a formal argument for the identity 
\begin{equation} \label{e:det-free-identity}
\text{``}\frac{\det (\slashed \partial_\rho)}{\det(\slashed \partial)}\text{''} =\prod_{1\leq r<s\leq n}|x_r-x_s|^{2\alpha_r\alpha_s},
\end{equation}
where the twisted Dirac operator $\slashed \partial_\rho$ and the singularities $x_1,\dots,x_n$ and windings $\alpha_1,\dots,\alpha_n$ are as in Section~\ref{sec:introferm}.
As a rigorous interpretation or a proof is not necessary for us, we simply give a heuristic argument of the flavor
one might find in the theoretical physics literature.

For a rigorous version on the lattice, where $\det(\Dirac)$ becomes the partition function of the Dimer model,
we also refer to Dubédat~\cite{MR3369909}.

\subsection{Branching as gauge field}

We first argue that one can view $\slashed\partial_\rho$ as a massless Dirac operator coupled to an external singular gauge field. In our convention for the Dirac operator, coupling to an external gauge field $A=A_0-iA_1$ is equivalent to replacing
the massless Dirac operator $\slashed \partial$ by 
\begin{equation}
\slashed \partial_A=\slashed \partial-i\slashed A=\slashed \partial +\begin{pmatrix}
0 & A_0-iA_1\\
-A_0-iA_1 & 0
\end{pmatrix}=:\begin{pmatrix}
0 & 2\bar \partial + A\\
2\partial-\bar{A} & 0
\end{pmatrix}.
\end{equation}
The correlation functions are then given in terms of the Green's function $S_A$ of $\slashed\partial_A$, which is the solution of the problem 
\begin{equation}
\slashed \partial_A S_A(z,w)=\begin{pmatrix}
\delta(z-w) & 0\\
0 & \delta(z-w)
\end{pmatrix} \qquad \text{and} \qquad \lim_{z\to \infty} S_A(z,w)=0.
\end{equation}
For nice enough $A$, one readily checks that the solution to this is
\begin{equation}
S_A(z,w)=\begin{pmatrix}
0 & \frac{1}{2\pi}\frac{1}{\bar z-\bar w}e^{\frac{1}{2\pi}\int_{\R^2} du\, (\frac{\bar{A}(u)}{\bar z-\bar u}-\frac{\bar{A}(u)}{\bar w-\bar u})}\\
 \frac{1}{2\pi}\frac{1}{z-w}e^{-\frac{1}{2\pi}\int_{\R^2} du\, (\frac{A(u)}{z-u}-\frac{A(u)}{w-u})} & 0
\end{pmatrix}.
\end{equation}
Thus for correlation functions of massless twisted free fermions to agree with correlation functions of massless free fermions in an external gauge field $A$, we must have $S_A=(\slashed\partial_\rho)^{-1}$, which in view of \eqref{eq:S0} means that we should have 
\begin{equation}
\frac{1}{2\pi}\int_{\R^2}du \, \frac{A(u)}{z-u}=\log \rho(z),
\end{equation}
or in other words, 
\begin{equation}
  A(z)=2\bar \partial \log \rho(z).
\end{equation}
Note that off of the branch cut $\Gamma$, the right hand side is zero, so $A$ would have to be a distribution supported on $\Gamma$, and thus a rather singular object.
To understand it better, let us try to understand the action of $2\bar\partial \log \rho$ on a test function $f\in C_c^\infty(\R^2)$.
By our convention for the choice of the branch of $\rho$ from Section \ref{sec:introferm}, a simple application of Green's theorem shows that 
\begin{equation}
\int_{\R^2}dz\, 2\bar\partial \log \rho(z) f(z)=-2\sum_{j=1}^n \alpha_j \int_{\R^2}dz\, \log (z-x_j)\bar \partial f(z)=2\pi \sum_{j=1}^n \alpha_j \int_0^\infty du f(u+x_j),
\end{equation}
see also Figure~\ref{fig:branches}.
This means that we should understand $A$ as 
\begin{equation}
  A(z)=2\bar\partial \log \rho(z)=2\pi \sum_{j=1}^n \alpha_j \delta_{\Gamma_j}(z),
\end{equation}
where $\Gamma_j$ denotes the branch cut of $(z-x_j)^{\alpha_j}$, so also $\bar{A}(z)=A(z)$, and the action of $\delta_{\Gamma_j}$ on a test function is really shorthand for the integral appearing above.

\subsection{Determinant with gauge field}

Having established that massless twisted free fermions are equivalent to massless free fermions coupled to this singular external gauge field, let us move on to computing the determinant we are after.
Formally, if $A$ depends on $x_i$ one has
\begin{equation}
  \partial_{x_i} \log \det(\Dirac_A)
  = \partial_{x_i} \tr \pa{  \log\Dirac_A }
  = \tr \pa{ (\partial_{x_i} \Dirac_A) \Dirac_A^{-1} }.
\end{equation}
Since
\begin{equation}
  \partial_{x_i}\Dirac_A
  = \partial_{x_i} \begin{pmatrix}
      0 & 2 \bar\partial +A \\
      2 \partial - \bar A & 0
  \end{pmatrix}
  = \begin{pmatrix}
      0 & \partial_{x_i}A \\
       - \partial_{x_i}\bar A & 0
  \end{pmatrix},
\end{equation}
this gives
\begin{equation}
  \partial_{x_i} \log \det(\Dirac_A)
  = \int_{\R^2} du \, \partial_{x_i}A(u) S_{A,21}(u,u)
    -  \int_{\R^2} du \, \partial_{x_i}\bar A(u) S_{A,12}(u,u).
  \end{equation}
This is still problematic since $S_{A,21}(u,u)$ and $S_{A,12}(u,u)$ are both infinite.

\subsection{Determinant for twisted Dirac operator}

We note however that if $A$ is the gauge potential formally corresponding to the branching $\rho$,
for $\sum_{j=1}^n \alpha_j=0$, we have, for any constant $c$,
\begin{equation}
  \int_{\R^2}du \, \partial_{x_j} A(u)c=2\pi \partial_{x_j}\int_0^\infty du \sum_{j=1}^n \alpha_j c=0.
\end{equation}
Thus one can formally regularize the above expression as
\begin{align}\label{eq:partialdiag}
  \partial_{x_j}\log \frac{\det(\slashed\partial_A)}{\det(\slashed\partial)}
  &=\int_{\R^2}du \, \partial_{x_j}A(u)[S_{A,21}(u,u)-S_{21}(u,u)]\nnb
    &\quad-\int_{\R^2} du \, \partial_{x_j}\bar{A}(u)[S_{A,12}(u,u)-S_{12}(u,u)],
\end{align}
where $S_{ij} =(\Dirac)^{-1}_{ij}$ and $S_{A,21}(u,u)-S_{21}(u,u)$ is interpreted suitably, for example by point splitting.
Indeed, using that
\begin{align}
  \lim_{v\to u}[S_{A,21}(u,v)-S_{21}(u,v)]
  &=
  \lim_{v\to u}\frac{1}{2\pi(v-u)}\left(\frac{\rho(v)}{\rho(u)}-1\right) = \frac{1}{2\pi} \partial_u \log \rho(u)
  \\
  \lim_{v\to u}[S_{A,12}(u,u)-S_{12}(u,u)]
  &=
    \lim_{v\to u}\frac{1}{2\pi(\bar v-\bar u)}\left(\frac{\overline{\rho(u)}}{\overline{\rho(v)}}-1\right)
    = \frac{1}{2\pi} \bar\partial_u \overline{\log \rho(u)},
\end{align}
we find
\begin{align}
\partial_{x_j}\log \frac{\det(\slashed\partial_\rho)}{\det(\slashed\partial)}
&=\frac{1}{2\pi}\int_{\R^2}du\,\partial_{x_j} A(u)\partial_u \log\rho(u)+\frac{1}{2\pi}\int_{\R^2}du \, \partial_{x_j}\bar{A}(u)\bar \partial_u \overline{\log \rho(u)}\nnb
  &=\frac{1}{2\pi}\int_{\R^2}du \, \partial_{x_j} A(u)\sum_{k=1}^n \alpha_k \frac{1}{u-x_k}+\frac{1}{2\pi}\int_{\R^2} du \, \partial_{x_j}\bar{A}(u)\sum_{k=1}^n \alpha_k \frac{1}{\bar u-\bar x_k}.
\end{align}
If we had a nice test function $f$, then the natural interpretation of $\int du \, \partial_{x_j}A(u)f(u)$ would be 
\begin{align}
2\pi \alpha_j \partial_{x_j}\int_0^\infty du \, f(u+x_j)=2\pi \alpha_j\int_0^\infty du \, \partial f(u+x_j).
\end{align}
However, for our setting $f(u)=\sum_{k=1}^n \alpha_k \frac{1}{u-x_k}$ and the $k=j$ term is so singular at $u=x_j$ that $\int_0^\infty du \, \partial f(u+x_j)$ does not exist. To remedy this, we will regularize $A(u)$ and argue that as we remove the regularization, there exists a finite limit.

For this regularization, let $A_\epsilon(u)=(\eta_\epsilon*A)(u)$, where $\eta_\epsilon(u)=\epsilon^{-2}\eta(\epsilon^{-1}u)$ and $\eta\in C_c^\infty(\R^2)$ satisfies $\eta(-u)=\eta(u)$, $\eta\geq 0$ and $\int_{\R^2}\eta=1$ (so that $A_\epsilon\to A$ as $\epsilon\to 0$).
Then one readily checks that %
\begin{equation}
A_\epsilon(v)=2\pi \sum_{j=1}^n \alpha_j \int_0^\infty du\, \eta_\epsilon(v-u-x_j).
\end{equation}
Thus 
\begin{equation}
\partial_{x_j} A_\epsilon(v)=-2\pi \alpha_j \int_0^\infty du\,  \partial \eta_\epsilon(v-u-x_j)
\end{equation}
and 
\begin{align}
&\frac{1}{2\pi}\int_{\R^2}du\, \partial_{x_j}A_\epsilon(u)\sum_{k=1}^n \alpha_k \frac{1}{u-x_k}\nnb
&=-\alpha_j \int_{\R^2}du \int_0^\infty dv \partial\eta_\epsilon(u-v-x_j)\sum_{k=1}^n \alpha_k \frac{1}{u-x_k}\nnb
&=-\alpha_j \int_{\R^2}du\,  \partial\eta_\epsilon(u)\int_0^\infty dv \sum_{k=1}^n \alpha_k \frac{1}{u+v+x_j-x_k}.
\end{align}
Since 
\begin{equation}
\int_0^\infty dv \int_{\R^2}du |\partial \eta_\epsilon(u)|\left|\sum_{k=1}^n \alpha_k \frac{1}{u+v+x_j-x_k}\right|<\infty,
\end{equation}
we can replace the integral over $\R^2$ by one over $\R^2\setminus \R$ in which case (due to the constraint $\sum_{k=1}^n \alpha_k=0$), we have 
\begin{equation}
\int_0^\infty dv \sum_{k=1}^n \alpha_k \frac{1}{u+v+x_j-x_k}=-\sum_{k=1}^n \alpha_k \log (u+x_j-x_k),
\end{equation}
and 
\begin{equation}
\frac{1}{2\pi}\int_{\R^2}du\, \partial_{x_j}A_\epsilon(u)\sum_{k=1}^n \alpha_k \frac{1}{u-x_k}=-\alpha_j \sum_{k=1}^n \alpha_k \int_{\R^2}du\, \partial \eta_{\epsilon}(u)\log(u+x_j-x_k). 
\end{equation}
A similar calculation shows that 
\begin{equation}
\frac{1}{2\pi}\int_{\R^2}du\, \partial_{x_j}\bar{A}_\epsilon(u)\sum_{k=1}^n \alpha_k \frac{1}{\bar u-\bar x_k}=-\alpha_j \sum_{k=1}^n \alpha_k \int_{\R^2}du\,  \partial \eta_{\epsilon}(u)\log(\bar u+\bar x_j-\bar x_k). 
\end{equation}
Thus (for small enough $\epsilon$)
\begin{align}
&\frac{1}{2\pi}\int_{\R^2}du\, \partial_{x_j} A_\epsilon(u)\sum_{k=1}^n \alpha_k \frac{1}{u-x_k}+\frac{1}{2\pi}\int_{\R^2}du\, \partial_{x_j}\bar{A}_\epsilon(u)\sum_{k=1}^n \alpha_k \frac{1}{\bar u-\bar x_k}\nnb
&=\alpha_j \sum_{k:k\neq j}\alpha_k \int_{\R^2} du\, \eta_\epsilon(u)\frac{1}{u+x_j-x_k}-\alpha_j^2 \int_{\R^2}du\, \partial \eta_\epsilon(u)2\log |u|\nnb
&=\alpha_j \sum_{k:k\neq j}\alpha_k \int_{\R^2} du\, \eta_\epsilon(u)\frac{1}{u+x_j-x_k}+\alpha_j^2 \int_{\R^2}du\, \eta_\epsilon(u)\frac{1}{u}.
\end{align}
Since $\eta$ is even, the last integral vanishes while the first one tends to $\alpha_j \sum_{k: k\neq j}\frac{\alpha_k}{x_j-x_k}$. 

Putting everything together, the above suggests that the correct interpretation is that 
\begin{equation}
\partial_{x_j}\log \frac{\det(\slashed \partial_\rho)}{\det(\slashed\partial)}=\alpha_j \sum_{k:k\neq j}\frac{\alpha_k}{x_j-x_k}=\partial_{x_j}\log \prod_{1\leq r<s\leq n}|x_r-x_s|^{2\alpha_r\alpha_s},
\end{equation}
and a similar argument shows that 
\begin{align}
\bar \partial_{x_j}\log \frac{\det(\slashed \partial_\rho)}{\det(\slashed\partial)}=\bar\partial_{x_j}\log \prod_{1\leq r<s\leq n}|x_r-x_s|^{2\alpha_r\alpha_s}.
\end{align}
Thus we conclude that one should interpret that for some constant $C$ (independent of $x_1,\dots,x_n$, but possibly depending on $n$ and $\alpha_1,\dots,\alpha_n$), we should have
\begin{equation}
  \log \frac{\det(\slashed \partial_\rho)}{\det(\slashed\partial)}= \log \prod_{1\leq r<s\leq n}|x_r-x_s|^{2\alpha_r\alpha_s}+C.
\end{equation}
This constant is likely to be non-universal (depending on exactly how one regularizes the determinant).
For our purposes, this constant does not matter and the most natural interpretation of the ratio of determinants is
\eqref{e:det-free-identity}.

\section{Gaussian multiplicative chaos and free field correlations}
\label{app:GMCGFF}

In this appendix, we collect some estimates on the Gaussian free field and its associated IMC.

\subsection{Covariance estimates and fractional correlations of the free field}
\label{app:GMCGFF-cov}

The first lemma below records some basic covariance estimates for the GFF with convolution and heat kernel regularization.
From Section~\ref{sec:Gauss}, recall that the heat kernel regularized Gaussian free field with covariance
\begin{equation} \label{e:appGFFeps}
  \int_{\epsilon^2}^\infty e^{-m^2 t} e^{\Delta_x t}\,dt
\end{equation}
is denoted  $\Phi^{\GFF(m)}_{\epsilon^2}$ whereas
$\Phi^{\GFF(m),\epsilon}_{0} = \eta_\epsilon *\Phi^{\GFF(m)}_{0}$ will be the convolution regularized field.

\begin{lemma}\label{le:hkcov}
Let $K_0$ be the zeroth modified Bessel function of the second kind and $\gamma$ be the Euler-Mascheroni constant. The following covariance estimates hold.

\smallskip\noindent (i)
For fixed $m>0$, as $\epsilon\to 0$
\begin{align}
&\lim_{\epsilon\to 0}\left(\E[\Phi_0^{\GFF(m),\epsilon}(x)^2]-\frac{1}{2\pi}\log \epsilon^{-1}\right)\\
&\qquad =\frac{1}{2\pi}\left(\log \frac{1}{m}-\gamma+\log 2+\int_{\R^2\times \R^2}dudv\, \eta(u)\eta(v)\log \frac{1}{|u-v|}\right)
\end{align}
uniformly in $x\in \R^2$. %

\smallskip\noindent (ii)
For fixed $m>0$ and $(x,y)\in K\subset \{(x,y)\in \R^2\times \R^2: x\neq y\}$ compact, as $\epsilon\to 0$
\begin{equation}
\E[\Phi_0^{\GFF(m),\epsilon}(x)\Phi_0^{\GFF(m),\epsilon}(y)]=\frac{1}{2\pi}K_0(m|x-y|)+o(1),
\end{equation}
where %
the error is uniform in $(x,y)\in K$.

\smallskip\noindent (iii)
For fixed $m>0$ and $K\subset \R^2$ compact we have as $\epsilon\to 0$
\begin{equation}
\E[\Phi_0^{\GFF(m),\epsilon}(x)\Phi_0^{\GFF(m),\epsilon}(y)]= \frac{1}{2\pi}\log (\epsilon^{-1}\wedge |x-y|^{-1})+O(1)
\end{equation}
uniformly in $x,y\in K$.

\smallskip\noindent (iv)
For fixed $m>0$, as $\epsilon\to 0$ 
\begin{equation}
\lim_{\epsilon\to 0}\left(\E[\Phi_{\epsilon^2}^{\GFF(m)}(x)^2]-\frac{1}{2\pi}\log \epsilon^{-1}\right)=\frac{1}{2\pi}\log m^{-1}-\frac{\gamma}{4\pi}
\end{equation}
uniformly in $x\in \R^2$.

\smallskip\noindent (v)
For fixed $m>0$ and $(x,y)\in K\subset \{(x,y)\in \R^2\times \R^2: x\neq y\}$ compact, as $\epsilon\to 0$
\begin{equation}
\E[\Phi_0^{\GFF(m),\epsilon}(x)\Phi_{\epsilon^2}^{\GFF(m)}(y)]=\frac{1}{2\pi}K_0(m|x-y|)+o(1),
\end{equation}
where %
the error is uniform in $(x,y)\in K$.

\smallskip\noindent (vi)
For fixed $m>0$ and $(x,y)\in K\subset \{(x,y)\in \R^2\times \R^2: x\neq y\}$ compact, as $\epsilon\to 0$
\begin{equation}
\E[\Phi_{\epsilon^2}^{\GFF(m)}(x)\Phi_{\epsilon^2}^{\GFF(m)}(y)]=\frac{1}{2\pi}K_0(m|x-y|)+o(1),
\end{equation}
where %
the error is uniform in $(x,y)\in K$.

\smallskip\noindent (vii)
For fixed $m>0$ and $K\subset \R^2$ compact, we have as $\epsilon\to 0$,
\begin{equation}
\E[\Phi_0^{\GFF(m),\epsilon}(x)\Phi_{\epsilon^2}^{\GFF(m)}(y)]= \frac{1}{2\pi}\log(\epsilon^{-1}\wedge |x-y|^{-1})+O(1)
\end{equation}
uniformly in $x,y\in K$.

\smallskip\noindent (viii)
For fixed $m>0$ and $K\subset \R^2$ compact, we have as $\epsilon\to 0$,
\begin{equation}
\E[\Phi_{\epsilon^2}^{\GFF(m)}(x)\Phi_{\epsilon^2}^{\GFF(m)}(y)]= \frac{1}{2\pi}\log(\epsilon^{-1}\wedge|x-y|^{-1})+O(1)
\end{equation}
uniformly in $x,y\in K$.
\end{lemma}

\begin{proof}
\subproof{item (i)}
  By definition, 
\begin{align}
\E[\Phi_0^{\GFF(m),\epsilon}(x)^2]&=\int_{\R^2\times \R^2}dudv\, \eta_\epsilon(x-u)\eta_\epsilon(x-v)\int_0^\infty ds\, e^{-m^2 s}\frac{e^{-\frac{|u-v|^2}{4s}}}{4\pi s}\nnb
&=\int_{\R^2\times \R^2}\eta(u)\eta(v)\int_{0}^\infty e^{-m^2 s}\frac{e^{-\frac{\epsilon^2|u-v|^2}{4s}}}{4\pi s}\nnb
&=\int_{\R^2\times \R^2}dudv\, \eta(u)\eta(v)\frac{1}{2\pi}K_0(m\epsilon |u-v|), 
\end{align}
where $K_0$ is a modified Bessel function of the second kind (given by $K_0(r)=\int_0^\infty \frac{ds}{2s} e^{-s-\frac{r^2}{4s}}$ for $r>0$). The basic fact about this function we need is that as $r\to 0$
\begin{equation}\label{eq:K0asy}
K_0(r)=\log \frac{1}{r}-\gamma+\log 2+O(r),
\end{equation}
where $\gamma$ is the Euler-Mascheroni constant.  As $\eta\in C_c^\infty(\R^2)$ and $\int_{\R^2}dx\, \eta(x)=1$, this means that as $\epsilon\to 0$
\begin{align}
\E[\Phi_0^{\GFF(m),\epsilon}(x)^2]&=\int_{\R^2\times \R^2}dudv\, \eta(u)\eta(v)\frac{1}{2\pi}\left(\log \frac{1}{\epsilon}+\log \frac{1}{m}+\log \frac{1}{|u-v|}-\gamma+\log 2\right)+o(1)\nnb
&=\frac{1}{2\pi}\left(\log \frac{1}{\epsilon}+\log \frac{1}{m}-\gamma+\int_{\R^2\times \R^2}dudv\, \eta(u)\eta(v)\log \frac{1}{|u-v|}\right)+o(1),
\end{align}
which was the claim. Uniformity in $x$ is trivial since the variance is independent of $x$.

\subproof{item (ii)}
As in the previous claim, we find 
\begin{align}
\E[\Phi_0^{\GFF(m),\epsilon}(x)\Phi_0^{\GFF(m),\epsilon}(y)]=\int_{\R^2\times \R^2}dudv\, \eta(u)\eta(v)\frac{1}{2\pi}K_0(m|x-y+\epsilon(u-v)|).
\end{align}
Since $(x,y)\in K\subset \{(x,y)\in \R^2\times \R^2: x\neq y\}$ compact, $|x-y|$ is bounded away from zero. Moreover, as the support of $\eta$ is compact, for small enough $\epsilon$, $|x-y+\epsilon(u-v)|$ is bounded away from zero. We thus find that as $\epsilon\to 0$ 
\begin{align}
K_0(m|x-y+\epsilon(u-v)|)\rightarrow K_0(m|x-y|)
\end{align}
uniformly in $(x,y)\in K$ and $u,v\in \mathrm{supp}(\eta)$. This means that as $\epsilon\to 0$ 
\begin{equation}
\E[\Phi_0^{\GFF(m),\epsilon}(x)\Phi_0^{\GFF(m),\epsilon}(y)]\to \int_{\R^2\times \R^2}dudv\, \eta(u)\eta(v)\frac{1}{2\pi}K_0(m|x-y|)=\frac{1}{2\pi}K_0(m|x-y|)
\end{equation}
uniformly in $(x,y)\in K$, which was the claim.

\subproof{item (iii)}
Arguing as before, we have 
\begin{align}
\E[\Phi_0^{\GFF(m),\epsilon}(x)\Phi_0^{\GFF(m),\epsilon}(y)]&=\int_{\R^2\times \R^2}dudv\, \eta(u)\eta(v)\frac{1}{2\pi}K_0(m|x-y+\epsilon(u-v)|).
\end{align}
Using \eqref{eq:K0asy}, we see that it is sufficient to prove the claim for 
\begin{align}
\int_{\R^2\times \R^2}dudv\, \eta(u)\eta(v)\frac{1}{2\pi}\log |x-y+\epsilon(u-v)|^{-1}.
\end{align}
Consider two cases: (i) $|x-y|\geq 3\epsilon$ and (ii) $|x-y|<3\epsilon$. In case (i), we note that since $\mathrm{supp}(\eta)\subset B(0,1)$, we have 
\begin{align}
\log |x-y+\epsilon(u-v)|^{-1}=\log |x-y|^{-1}+O(1)
\end{align}
uniformly in $|x-y|\geq 3\epsilon$ and $u,v\in \mathrm{supp}(\eta)$. Thus the claim is fine in this case. For $|x-y|\leq 3\epsilon$, we have on the other hand 
\begin{align}
&\int_{\R^2\times \R^2}dudv\, \eta(u)\eta(v)\frac{1}{2\pi}\log |x-y+\epsilon(u-v)|^{-1}\nnb
&=\int_{\R^2}dv\, \eta(v)\left(\frac{1}{2\pi}\log \epsilon^{-1}+\int_{\R^2}du\, \rho(u)\frac{1}{2\pi}\log |u-v+\epsilon^{-1}(x-y)|^{-1}\right).
\end{align}
Since $-v+\epsilon^{-1}(x-y)\in B(0,4)$ under our assumptions and 
\begin{align}
\sup_{z\in B(0,4)}\int_{\R^2}du\, \rho(u)\left|\log(|u+z|)\right|<\infty, 
\end{align}
we find that for $|x-y|<3\epsilon$ 
\begin{align}
\int_{\R^2\times \R^2}dudv\, \eta(u)\eta(v)\frac{1}{2\pi}\log |x-y+\epsilon(u-v)|^{-1}=\frac{1}{2\pi}\log \epsilon^{-1}+O(1)
\end{align}
where the error term has the required uniformity. The claim follows by noting that for $\epsilon\leq |x-y|<3\epsilon$, $\log |x-y|^{-1}=\log \epsilon^{-1}+O(1)$ (with the required uniformity).

\subproof{item (iv)}
First of all, we see that $\E[\Phi_{\epsilon^2}^{\GFF(m)}(x)^2]$ is independent of $x$, so uniformity will follow from pointwise convergence. Pointwise convergence is proven in \cite[Lemma 2.4]{MR4767492}.

\subproof{item (v)}
We begin with the representation
\begin{align}
\E[\Phi_0^{\GFF(m),\epsilon}(x)\Phi_{\epsilon^2}^{\GFF(m)}(y)]&=\int_{\R^2}du \eta_\epsilon(x-u)\int_{\epsilon^2}^\infty ds\, e^{-m^2 s} \frac{e^{-\frac{|y-u|^2}{4s}}}{4\pi s}\nnb
&=\int_{\R^2}du\, \eta(u)\int_{\epsilon^2}^\infty ds\, e^{-m^2 s}\frac{e^{-\frac{|x-y-\epsilon u|^2}{4s}}}{4\pi s}.
\end{align}
As $|x-y|$ is bounded away from zero, one readily checks that the $\epsilon\to 0$ limit can be taken inside of the integrals and we have convergence to $\frac{1}{2\pi}K_0(m|x-y|)$ with the appropriate uniformity.

\subproof{item (vi)}
This follows from the proof of \cite[Lemma 2.4]{MR4767492}.

\subproof{item (vii)}
We again start by writing 
\begin{align}
\E[\Phi_{0}^{\GFF(m),\epsilon}(x)\Phi_{\epsilon^2}^{\GFF(m)}(y)]&=\int_{\R^2}du\, \rho(u)\int_{\epsilon^2}^\infty ds\, e^{-m^2 s} \frac{e^{-\frac{|x-y-\epsilon u|^2}{4s}}}{4\pi s}
\end{align}
We write the $s$-integral as 
\begin{align}
&\int_{\epsilon^2}^\infty ds\, e^{-m^2 s}\frac{e^{-\frac{|x-y-\epsilon u|^2}{4s}}}{4\pi s}\nnb
&=\int_{1}^\infty ds\, e^{-m^2 s}\frac{e^{-\frac{|x-y-\epsilon u|^2}{4s}}}{4\pi s}+\int_{\epsilon^2}^1 ds (e^{-m^2 s}-1)\frac{e^{-\frac{|x-y-\epsilon u|^2}{4s}}}{4\pi s}+\int_{\epsilon^2}^1 ds \frac{e^{-\frac{|x-y-\epsilon u|^2}{4s}}}{4\pi s}.
\end{align}
For fixed $m$, the first two integrals on the right hand side are bounded uniformly in $x,y,u\in \R^2$ and $0<\epsilon<1$, so our task is to estimate the last integral. For this purpose, we note that for $r>0$
\begin{align}
\int_{\epsilon^2}^1ds \frac{e^{-\frac{r^2}{4s}}}{4\pi s}&=\int_{r^2}^{r^2/\epsilon^2}ds \frac{1}{4\pi s} e^{-s/4}.
\end{align}
Consider now two cases: (a) $r^2\leq \epsilon^2$ and (b) $r^2>\epsilon^2$. In case (a), we have 
\begin{align}
\int_{r^2}^{r^2/\epsilon^2}ds \frac{1}{4\pi s}+O(1)=\frac{1}{2\pi}\log \epsilon^{-1}+O(1), 
\end{align}
where the implied constant is universal. 
In case (b), $r\geq \epsilon^2$, we have 
\begin{align}
\int_{r^2}^{r^2/\epsilon^2}ds \frac{1}{4\pi s}e^{-s/4}=\int_{r^2}^1 \frac{ds}{4\pi s}+O(1)=\frac{1}{2\pi}\log r^{-1}+O(1)
\end{align}
where again the implied constant is absolute.
We conclude that 
\begin{align}
\E[\Phi_{0}^{\GFF(m),\epsilon}(x)\Phi_{\epsilon^2}^{\GFF(m)}(y)]&=\int_{\R^2}du\rho(u)\frac{1}{2\pi}\log(\epsilon^{-1}\wedge |x-y-\epsilon u|^{-1})+O(1),
\end{align}
where the implied constant is uniform in the relevant $x,y$.

To conclude, consider now two cases (1) $|x-y|\geq 2\epsilon$ and (2) $|x-y|<2\epsilon$. In case (1), we have $\log(\epsilon^{-1}\wedge |x-y-\epsilon u|^{-1})=\log |x-y-\epsilon u|^{-1}=\log |x-y|^{-1}+O(1)$ (with the error uniform in $x,y,u,\epsilon$) so in case (1)
\begin{equation}
\E[\Phi_{0}^{\GFF(m),\epsilon}(x)\Phi_{\epsilon^2}^{\GFF(m)}(y)]=\frac{1}{2\pi}\log |x-y|^{-1}+O(1).
\end{equation}
In case (2), we have $\log (\epsilon^{-1}\wedge |x-y-\epsilon u|^{-1})=\log \epsilon^{-1}+O(1)$ and 
\begin{equation}
\E[\Phi_{0}^{\GFF(m),\epsilon}(x)\Phi_{\epsilon^2}^{\GFF(m)}(y)]=\frac{1}{2\pi}\log \epsilon^{-1}+O(1)
\end{equation}
with the appropriate uniformity. The claim follows by noting that for $\epsilon\leq |x-y|\leq 2\epsilon$, we have $\log |x-y|^{-1}=\log \epsilon^{-1}+O(1)$ (with the relevant uniformity). This concludes the proof of this item

\subproof{item (viii)}
In this case, we argue as in the previous case and write 
\begin{align}
\E[\Phi_{\epsilon^2}^{\GFF(m)}(x)\Phi_{\epsilon^2}^{\GFF(m)}(y)]&=\int_{\epsilon^2}^1 ds\, \frac{e^{-\frac{|x-y|^2}{4s}}}{4\pi s}+O(1)
\end{align}
(with the appropriate uniformity). For the remaining integral, we already argued in the previous case that it is 
\begin{align}
\int_{\epsilon^2}^1 ds\, \frac{e^{-\frac{|x-y|^2}{4s}}}{4\pi s}=\frac{1}{2\pi}\log(\epsilon^{-1}\wedge |x-y|^{-1})+O(1).
\end{align}
This concludes the proof.
\end{proof}

Next we deduce the existence of the massless fractional correlation functions 
\begin{equation}
\avg{\wick{e^{i\alpha_1\sqrt{4\pi}\varphi(x_1)}}\cdots \wick{e^{i\alpha_n\sqrt{4\pi}\varphi(x_n)}}}_{\GFF(0)}
\end{equation}
that were discussed in Section \ref{sec:Gauss-corr}.
The main claim is recorded in the following lemma.

\begin{lemma}\label{le:gfffrac}
For fixed $x_1,\dots,x_n\in \R^2$ distinct and fixed $\alpha_1,\dots,\alpha_n\in \C$, we have 
\begin{align}
&\avg{\wick{e^{i\alpha_1\sqrt{4\pi}\varphi(x_1)}}\cdots \wick{e^{i\alpha_1\sqrt{4\pi}\varphi(x_1)}}}_{\GFF(0)}\nnb
&\quad:=\lim_{m\to 0}\lim_{\epsilon\to 0}\E\left[\wick{e^{i\alpha_1\sqrt{4\pi}\Phi_{\epsilon^2}^{\GFF(m)}(x_1)}}_\epsilon\cdots \wick{e^{i\alpha_n\sqrt{4\pi}\Phi_{\epsilon^2}^{\GFF(m)}(x_n)}}_\epsilon\right]\nnb
&\quad\: =\begin{cases}
0, & \sum_{j=1}^n \alpha_j\neq 0\\
\left(2 e^{-\frac{\gamma}{2}}\right)^{\sum_{j=1}^n \alpha_j^2}\prod_{1\leq j<k\leq n}|x_j-x_k|^{2\alpha_j \alpha_k}, & \sum_{j=1}^n \alpha_j=0
\end{cases}.
\end{align}
The notation $\wick{e^{i\sqrt{4\pi}\alpha\Phi_{\epsilon^2}^{\GFF(m)}(x)}}_\epsilon = \epsilon^{-\alpha^2}e^{i\sqrt{4\pi}\alpha\Phi_{\epsilon^2}^{\GFF(m)}(x)}$ is defined in \eqref{eq:fracwick}.
\end{lemma}

\begin{proof}
Using the formula for the characteristic function of Gaussian random vector, we find that 
\begin{align}
&\E\left[\wick{e^{i\alpha_1\sqrt{4\pi}\Phi_{\epsilon^2}^{\GFF(m)}(x_1)}}_\epsilon\cdots \wick{e^{i\alpha_n\sqrt{4\pi}\Phi_{\epsilon^2}^{\GFF(m)}(x_n)}}_\epsilon\right]\nnb
&=\prod_{j=1}^n \left(\epsilon^{-\alpha_j^2} e^{-2\pi \alpha_j^2 \E[\Phi_{\epsilon^2}^{\GFF(m)}(x_j)^2]}\right)\prod_{1\leq j<k\leq n} e^{-4\pi \alpha_j \alpha_k \E[\Phi_{\epsilon^2}^{\GFF(m)}(x_j)\Phi_{\epsilon^2}^{\GFF(m)}(x_k)]}.
\end{align}
From item (iv) and item (vi) of Lemma \ref{le:hkcov}, we see that for fixed and distinct $x_1,\dots,x_n$, we have 
\begin{align}
&\lim_{\epsilon\to 0}\E\left[\wick{e^{i\alpha_1\sqrt{4\pi}\Phi_{\epsilon^2}^{\GFF(m)}(x_1)}}_\epsilon\cdots \wick{e^{i\alpha_n\sqrt{4\pi}\Phi_{\epsilon^2}^{\GFF(m)}(x_n)}}_\epsilon\right]\nnb
&=\prod_{j=1}^n \left( e^{-2\pi \alpha_j^2 (\frac{1}{2\pi}\log m^{-1}-\frac{\gamma}{4\pi})}\right)\prod_{1\leq j<k\leq n} e^{-2 \alpha_j \alpha_k K_0(m|x_j-x_k|)}.
\end{align}
From \eqref{eq:K0asy}, we deduce that as $m\to 0$, 
\begin{align}
&\prod_{j=1}^n \left( e^{-2\pi \alpha_j^2 (\frac{1}{2\pi}\log m^{-1}-\frac{\gamma}{4\pi})}\right)\prod_{1\leq j<k\leq n} e^{-2 \alpha_j \alpha_k K_0(m|x_j-x_k|)}\nnb
&=(1+o(1))(2m e^{\gamma})^{(\sum_{j=1}^n \alpha_j)^2} \left(2 e^{-\frac{\gamma}{2}}\right)^{\sum_{j=1}^n \alpha_j^2} \prod_{1\leq j<k\leq n}|x_j-x_k|^{2\alpha_j \alpha_k}.
\end{align}
Thus unless $\sum_{j=1}^n \alpha_j=0$, the prefactor tends to zero as $m\to 0$ and if $\sum_{j=1}^n \alpha_j=0$, we have 
\begin{align}
  &\lim_{m\to 0}\lim_{\epsilon\to 0}\E\left[\wick{e^{i\alpha_1\sqrt{4\pi}\Phi_{\epsilon^2}^{\GFF(m)}(x_1)}}_\epsilon\cdots \wick{e^{i\alpha_n\sqrt{4\pi}\Phi_{\epsilon^2}^{\GFF(m)}(x_n)}}_\epsilon\right]
    \nnb
    &=\left(2 e^{-\frac{\gamma}{2}}\right)^{\sum_{j=1}^n \alpha_j^2}\prod_{1\leq j<k\leq n}|x_j-x_k|^{2\alpha_j \alpha_k}
\end{align}
as claimed.
\end{proof}

\subsection{Gaussian imaginary multiplicative chaos -- proofs of Proposition~\ref{prop:GMC}--\ref{prop:hkimc}}
\label{app:GMC}

Proposition~\ref{prop:GMC} was essentially shown in \cite[Theorem 3.16]{MR4149524} and our proof below
therefore focuses on the differences compared to our setting.
As a preliminary,
the definition of the Besov norm used in \cite{MR4149524} does not exactly match \eqref{e:Besov-def}.
However, it is not difficult to see that the proof is straightforward to adapt to the definition \eqref{e:Besov-def},
and, alternatively, one can use that the norm \eqref{e:Besov-def}
is equivalent to the standard definition.\footnote{For example,
  the proof of \cite[Proposition~A.1]{2504.08606} shows that the norm \eqref{e:Besov-def} is equivalent to 
\begin{equation} \label{e:Besov-heat}
  \sup_{t\leq 1} t^{s/2}\|e^{\Delta t} f\|_{L^\infty(\R^2)} . 
\end{equation}
Indeed, without weight, the support restriction on $\varphi$ in that proof is not used and one apply
the conclusion with $(g,\varphi)=(\Psi,p_1)$ where $p_1(x) = e^{\Delta}(0,x)$ and with the role of the two functions reversed.
For a proof that \eqref{e:Besov-heat} is in turn equivalent to the standard $B^{-s}_{\infty,\infty}$ norm, defined in terms of Littlewood-Paley blocks, see
\cite[Theorem~2.34]{MR2768550}.}

\begin{proof}[Proof of Proposition~\ref{prop:GMC}]
As discussed above we focus on the differences to \cite[Theorem 3.16]{MR4149524}.

The moment bound  \eqref{e:GMC-moments} is not explicitly stated in \cite{MR4149524}, but it is part of the proof of
item~(i) of  \cite[Theorem~3.16]{MR4149524}.
Below we will also need the generalization that the convolution regularized imaginary chaos \eqref{e:Meps-bis}
satisfies the same moment bound uniformly in $\epsilon>0$, i.e.,
\begin{equation} \label{e:GMC-moments-eps}
  \sup_{\epsilon>0}\E\qa{\|\chi M_\alpha^\epsilon\|_{C^{-s}}^p} < \infty.
\end{equation}
This follows from a small modification of the proof in \cite{MR4149524}. Namely, one simply needs to replace
the application of item (i) of \cite[Proposition~3.6]{MR4149524} by that of item (ii). The additional term
$N^2$ only changes the constant of the moment bound (since $N$ is related to $p$ above).

The convergence  in $C^{-s}_\loc$, $s>\alpha^2$ was also not explicitly derived in \cite{MR4149524}.
It was only stated that $M^\epsilon_\alpha$ converges in $H^{-s}(\R^2)$, $s>d/2=1$ and that the limit
$M_\alpha$ is supported in $C^{-s}_\loc$, $s>\alpha^2$.
That the convergence also takes places in $C^{-s}_\loc$, $s>\alpha^2$ follows from \eqref{e:GMC-moments-eps}.
Indeed,
the embedding $C^{-s}_\loc \subset C^{-s'}_\loc$ is compact if $s'>s$,
see, e.g., \cite[Proposition~A.7]{2504.08606},
and since the moment bound can be applied with $s'$ instead of $s$ (which is arbitary as long as $s>\alpha^2$),
this implies tightness in $C^{-s}_\loc$
and hence the convergence also takes place in that space.

Finally, in \cite[Theorem 3.16]{MR4149524}, the regularized IMC is normalized with the standard Wick ordering
$e^{\frac12 4\pi \alpha^2 \var(\eta_\epsilon * \varphi(0))}$ instead of $\epsilon^{-\alpha^2}$ in our definition \eqref{e:Meps-bis}.
That the two agree up to a deterministic multiplicative constant (which can depend on the mollifier) follows from Lemma~\ref{le:hkcov}.
\end{proof}

Using the covariance estimates given in Lemma~\ref{le:hkcov} the equivalence of the convolution and heat-kernel regularized
IMC can be established as follows.

\begin{proof}[Proof of Proposition~\ref{prop:hkimc}]
We know from \cite[Theorem~1.1, Lemma~2.8, Lemma~3.10, and Corollary~3.11]{MR4149524} along with Vitali's convergence theorem that 
\begin{equation}
\lim_{\epsilon\to 0} \epsilon^{-\alpha^2}\int_{\R^2}dx\, e^{i\sqrt{4\pi}\alpha\Phi_0^{\GFF(m),\epsilon}(x)}f(x)=M_\alpha(f)
\end{equation}
with convergence in $L^p$ for any $p\in[1,\infty)$. Moreover, we claim that for each $p\in[1,\infty)$
\begin{equation}
\E\left[\left|\epsilon^{-\alpha^2}\int_{\R^2}dx\, e^{i\sqrt{4\pi}\alpha\Phi_{\epsilon^2}^{\GFF(m)}(x)}f(x)\right|^p\right]
\end{equation}
is bounded in $\epsilon$.
This follows for example from Theorem~\ref{th:fracapp} (which however considers a much more complicated setting).
Thus by Vitali's convergence theorem, it is sufficient for us to prove convergence in probability, which in turn is implied by convergence in $L^2$, namely that 
\begin{equation}
\lim_{\epsilon\to 0}\E\left[\left|\epsilon^{-\alpha^2}\int_{\R^2}dx\, e^{i\sqrt{4\pi}\alpha\Phi_0^{\GFF(m),\epsilon}(x)}f(x)-\epsilon^{-\alpha^2}C_{\eta}^{\alpha^2}\int_{\R^2}dx\, e^{i\sqrt{4\pi}\alpha\Phi_{\epsilon^2}^{\GFF(m)}(x)}f(x)\right|^2\right]=0.
\end{equation}
We can compute this expectation by a routine Gaussian calculation:
\begin{align}
&\E\left[\left|\epsilon^{-\alpha^2}\int_{\R^2}dx\, e^{i\sqrt{4\pi}\alpha\Phi_0^{\GFF(m),\epsilon}(x)}f(x)-\epsilon^{-\alpha^2}C_{\eta,\alpha}\int_{\R^2}dx\, e^{i\sqrt{4\pi}\alpha\Phi_{\epsilon^2}^{\GFF(m)}(x)}f(x)\right|^2\right]\nnb
  &=\epsilon^{-2\alpha^2}\int_{\R^2\times \R^2}dxdy\,f(x)f(y)\nnb
    &\quad \times \Bigg[ e^{-2\pi\alpha^2 \E[\Phi_0^{\GFF(m),\epsilon}(x)^2]-2\pi\alpha^2 \E[\Phi_0^{\GFF(m),\epsilon}(y)^2]}e^{4\pi \alpha^2 \E[\Phi_0^{\GFF(m),\epsilon}(x)\Phi_0^{\GFF(m),\epsilon}(y)]}\nnb
&\qquad -2C_{\eta}^{\alpha^2} e^{-2\pi\alpha^2 \E[\Phi_0^{\GFF(m),\epsilon}(x)^2]-2\pi\alpha^2 \E[\Phi_{\epsilon^2}^{\GFF(m)}(y)^2]}e^{4\pi \alpha^2 \E[\Phi_0^{\GFF(m),\epsilon}(x)\Phi_{\epsilon^2}^{\GFF(m)}(y)]}\nnb
&\qquad +(C_{\eta}^{\alpha^2})^2 e^{-2\pi\alpha^2 \E[\Phi_{\epsilon^2}^{\GFF(m)}(x)^2]-2\pi\alpha^2 \E[\Phi_{\epsilon^2}^{\GFF(m)}(y)^2]}e^{4\pi \alpha^2 \E[\Phi_{\epsilon^2}^{\GFF(m)}(x)\Phi_{\epsilon^2}^{\GFF(m)}(y)]}\Bigg].
\end{align}
We will argue that we can take the $\epsilon\to 0$ limit under the integral.
From Lemma~\ref{le:hkcov}, we see that 
\begin{align}
\epsilon^{-2\alpha^2} e^{-2\pi \alpha^2\E[\Phi_0^{\GFF(m),\epsilon}(x)^2]-2\pi \alpha^2\E[\Phi_0^{\GFF(m),\epsilon}(y)^2]}
\end{align} 
is bounded in $\epsilon$ uniformly in $x,y\in \mathrm{supp}(f)$. The same holds for the other variance terms as well. Moreover, we see from Lemma \ref{le:hkcov} that all of the exponentials of the covariance terms can be bounded by a constant time $|x-y|^{-2\alpha^2}$. Thus in the support of $(x,y)\mapsto f(x)f(y)$, we can bound our integrand by $C |x-y|^{-2\alpha^2}$ for a suitable constant $C$ (which may depend on $m$). This is locally integrable in $\R^2\times\R^2$ as $\alpha^2<1$, so by the dominated convergence theorem, we can take the $\epsilon\to 0$ limit under the integral. Using Lemma \ref{le:hkcov}, we thus find 
\begin{align}
&\lim_{\epsilon\to 0}\E\left[\left|\epsilon^{-\alpha^2}\int_{\R^2}dx\, e^{i\sqrt{4\pi}\alpha\Phi_0^{\GFF(m),\epsilon}(x)}f(x)-\epsilon^{-\alpha^2}C_{\eta}^{\alpha^2}\int_{\R^2}dx\, e^{i\sqrt{4\pi}\alpha\Phi_{\epsilon^2}^{\GFF(m)}(x)}f(x)\right|^2\right]\nnb
&=\int_{\R^2\times \R^2}dxdy\, f(x)f(y) e^{2\alpha^2 K_0(m|x-y|)}\bigg[e^{-2\alpha^2(\log m^{-1}-\gamma+\log 2+\int_{\R^2\times \R^2}dudv\eta(u)\eta(v)\log \frac{1}{|u-v|})}\nnb
&\quad -2C_{\eta}^{\alpha^2} e^{-\alpha^2(\log m^{-1}-\gamma+\log 2+\int_{\R^2\times \R^2}dudv\eta(u)\eta(v)\log \frac{1}{|u-v|})} e^{-\alpha^2(\log m^{-1}-\frac{\gamma}{2})}\nnb
&\quad +(C_{\eta}^{\alpha^2})^2 e^{-2\alpha^2(\log m^{-1}-\frac{\gamma}{2})}\bigg].
\end{align}
The quantity in the brackets vanishes by definition of $C_{\eta}^{\alpha^2}$.
\end{proof}

\section{Gaussian moment and continuity estimates}
\label{app:Gauss}

In this appendix, we collect the proofs of various moment and continuity estimates on the GFF and its regularized and
scale decomposed version. All arguments follow from standard methods.

\subsection{Regularity estimate from Section~\ref{sec:mixing}}

The following proposition implies the Besov regularity estimate for the sine-Gordon measure
stated in \eqref{e:SG-Besov}.

\begin{proposition} \label{prop:GFF-Besov}
  Let $\nu$ be a probability measure on $\cS'(\R^2)/\text{constants}$
  with expectation $\avg{\cdot}$ satisfying for any $g \in \cS(\R^2)$ with $\int g\, dx = 0$:
  \begin{equation} \label{e:IRB-variance-0}
  \avg{e^{(\varphi,g)}} \leq
  e^{\frac12 \|g\|_{\dot H^{-1}(\R^2)}^2}, \qquad \|g\|_{\dot H^{-1}(\R^2)} = (g, (-\Delta)^{-1}g)^{1/2}.
\end{equation}
  Then $\nu$ is supported on $C^{-s}(\rho)/\text{constants}$ where $\rho(x)=(1+|x|^2)^{-\sigma/2}$ with any $\sigma>0$ and any $s>0$.
  More precisely, for any $\eta \in \cS(\R^2)$ with $\int \eta \, dx=1$ and any $p\geq 1$ one has
  \begin{equation} \label{e:IRB-Besov}
    \avg{\|\varphi-(\varphi,\eta)\|_{C^{-s}(\rho)}^p} \leq C^p p^{p/2}.
  \end{equation}
  Moreover, if $\|g\|_{\dot H^{-1}(\R^2)}^2$ in \eqref{e:IRB-variance-0} is replaced by
  $(g,(-\Delta+m^2)^{-1}g)$ with fixed $m>0$ or by
  the covariance of $\Phi^{\GFF(0)}_{0,t_0}$ with fixed $t_0<\infty$,
  an analogous estimate holds   without subtracting $(\varphi,\eta)$.
\end{proposition}

\begin{proof}
  By Jensen's inequality, it suffices to prove the bound for $p$ an even integer.
  Let $\Psi \in C_c^\infty(\R^2)$ and $\Psi_R(x) = R^{-d}\Psi(x/R)$ be as in the definition of the Besov norm in \eqref{e:Besov-def}, where throughout the proof $d=2$, but we sometimes write $d$ for emphasis.
  By the embedding of $B^{-s}_{\infty,\infty}(\rho)$ into $B^{-s+d/s}_{p,p}(\rho)$ one has
  (see, for example, \cite[Proposition~A.6]{2504.08606}):
  \begin{equation}
    \|f\|_{C^{-s}(\rho)}^p \leq C^p \int_0^1 R^{ps-d} \|\Psi_R*f\|_{L^p(\rho)}^p \, \frac{dR}{R}.
  \end{equation}
  Denoting  $\Psi_R^x = \Psi_R(x-\cdot)$ so that $\Psi_R*f(x) = (\Psi_R^x,f)$ and $(\Psi_R^x,1) = 1$,
  the left-hand side of \eqref{e:IRB-Besov} is thus bounded by
  \begin{equation}
    C^p \int_0^1 R^{ps-d} \int_{\R^2} \avg{(\Psi_R^x-\eta,\varphi)^p} \rho(x)^p \, dx \frac{dR}{R}.
  \end{equation}
  By \eqref{e:IRB-variance-0},
  \begin{equation}
    \avg{(\Psi_R^x-\eta,\varphi)^p}
    \leq
    \frac{p!}{t^p} \avg{e^{t(\Psi_R^x-\eta,\varphi)}}
    \leq
    \frac{p!}{t^p} e^{\frac12 t^2 \|\Psi_R^x-\eta\|_{\dot H^{-1}}^2}
    .
  \end{equation}
  The exponent on the right-hand side is bounded by
  \begin{equation} \label{e:PsietaHminus1}
    \|\Psi_R^x-\eta\|_{\dot H^{-1}}^2
    \leq C+ C \log((1+|x|)/R).
    \end{equation}
  Choosing $t^2=2\delta p/C$ we obtain
  \begin{equation}
    \avg{(\Psi_R^x-\eta,\varphi)^p}
    \leq C_\delta^p p^{p/2} (1+|x|)^{\delta p}R^{-\delta p}
   \end{equation}
   It follows that
   \begin{equation}
     \avg{\|\varphi-(\varphi,\eta)\|_{C^{-s}(\rho)}^p} \leq C_\delta^p p^{p/2} \int_0^1 R^{ps-d} R^{-\delta p} \frac{dR}{R} \int_{\R^2} (1+|x|)^{\delta p} \rho(x)^p \, dx
   \leq C^p p^{p/2}.
  \end{equation}
  where the integral is finite if $\delta<\min\{s,\sigma\}$ and $p>0$ is chosen sufficiently large.

  Finally, for completeness, we include the proof of \eqref{e:PsietaHminus1}.
  We start from (using that $\hat\Psi$ and $\hat\eta$ both have rapid decay in the last inequality below and that $1/R \geq 1$):
    \begin{align}
      \|\Psi_R^x-\eta\|_{\dot H^{-1}}^2
    &= \frac{1}{(2\pi)^2}\int_{\R^2} |\hat\Psi_R^x(p)-\hat\eta(p)|^2  \frac{dp}{|p|^2}
      \nnb
    &= \frac{1}{(2\pi)^2}\int_{\R^2} |e^{ip\cdot x}\hat\Psi(Rp)-\hat\eta(p)|^2  \frac{dp}{|p|^2}
      \nnb
    &\leq C +  \frac{1}{(2\pi)^2}\int_{B_{1/R}(0)} |e^{ip\cdot x}\hat\Psi(Rp)-\hat\eta(p)|^2  \frac{dp}{|p|^2}.
    \end{align}
    Using $\hat \Psi(0)=\hat\eta(0)=1$ and $(a+b)^2 \leq 2a^2+2b^2$ the integral on the right-hand side is less than
    \begin{equation}
      \frac{2}{(2\pi)^2}\int_{B_{1/R}(0)} |e^{ip\cdot x}\hat\Psi(Rp)-\hat\Psi(0)|^2 \frac{dp}{|p|^2} +
      \frac{2}{(2\pi)^2}\int_{B_{1/R}(0)} |\hat\eta(p)-\hat\eta(0)|^2  \frac{dp}{|p|^2}.
    \end{equation}
    The second term is bounded by $C\log(1/R)$ and the first term equals
    (using  $|\hat \Psi(p)|\leq \|\Psi\|_{L^1} =1$):
    \begin{align}
      &\frac{2}{(2\pi)^2} \int_{B_{1}(0)} |e^{ip\cdot x/R}\hat\Psi(p)-\hat\Psi(0)|^2 \frac{dp}{|p|^2} 
      \nnb
      &
      \leq \frac{4}{(2\pi)^2}
        \int_{B_{1}(0)} |e^{ip\cdot x/R}-1|^2 \frac{dp}{|p|^2}
        +
        \frac{4}{(2\pi)^2}
        \int_{B_{1}(0)} |\hat\Psi(p)-\hat\Psi(0)|^2 \frac{dp}{|p|^2}.
    \end{align}
    The second term on the last right-hand side is uniformly bounded whereas the first term equals
    \begin{align}
      \int_{B_{1}(0)} (2-2\cos(p\cdot x/R)) \frac{dp}{|p|^2}
      &=
       2\pi \int_0^1 (2-2\cos(r |x|/R)) \frac{dr}{r}
        \nnb
      &=
        2\pi \int_0^{|x|/R} (2-2\cos (r)) \frac{dr}{r} \leq C \log(1+|x|/R),
    \end{align}
    up to a multiplicative constant, as claimed.

    Finally, the estimate under the assumption with $m>0$ or $t_0<\infty$ is analogous
    (but slightly simpler since the subtraction of $(\eta,\varphi)$ is not needed),
    and we omit the details.
\end{proof}

\subsection{Gaussian estimates from Section~\ref{sec:coupling-finvol}}
\label{app:gaussianest}

In this appendix, we prove Lemma~\ref{le:gaussianest}.
The lemma can be deduced most immediately from the estimates stated in \cite[Section~5.1]{MR4767492}, in particular (5.17),
as follows. An argument using Sobolev embedding akin to the one in the next subsection would also be possible.

\begin{proof}[Proof of Lemma~\ref{le:gaussianest}]
  Throughout the proof, probabilities and expectations refer to the Gaussian measure with covariance $C_t-C_{t_0}$.
  An application of Dudley's theorem and the Borell--TIS inequality show
  (see \cite[Section~5.1]{MR4767492}, in particular (5.17))
  that there are $c>0$ and $C(\Lambda)>0$ such that:
  \begin{equation}
    \P\qa{\sup_{x\in\Lambda} \partial_{\mu}\zeta(x) > u}\leq e^{-c u^2} \qquad  \text{for $u>C(\Lambda)$ and $\mu\in \{0,1\}$},
  \end{equation}
  where we write $x=(x_0,x_1)\in \R^2$.
  By a union bound, possibly decreasing $c$ and increasing $C(\Lambda)$,
  \begin{equation}\label{eq:nablatail}
    \P\qB{\|\nabla\zeta\|_{L^\infty(\Lambda)} > u}\leq e^{- cu^2} \qquad \text{for $u>C(\Lambda)$.}
  \end{equation}
  In particular,
  \begin{align}
    \E\qa{e^{-p\|\nabla \zeta\|_{L^\infty(\Lambda)}}}
    &= \int_0^\infty pe^{-pu} \P\qB{\|\nabla \zeta\|_{L^\infty(\Lambda)} \leq u} \, du
      \nnb
    &\geq \int_{C(\Lambda)}^\infty pe^{-pu}(1-e^{-cu^2}) \, du
      \nnb
    &\geq \frac12 \int_{C(\Lambda)}^\infty pe^{-pu} \, du
      = \frac12 e^{-pC(\Lambda)}.
  \end{align}
  Similarly,
  \begin{align}
    \E\qa{e^{+p\|\nabla \zeta\|_{L^\infty(\Lambda)}}\1_{\|\nabla \zeta\|_{L^\infty(\Lambda)>T}}}
    &= \int_T^\infty pe^{pu} \P\qB{\|\nabla \zeta\|_{L^\infty(\Lambda)} > u} \, du - e^{pT}\P(\|\nabla \zeta\|_{L^\infty(\Lambda)}>T)%
  \end{align}
  For $T>C(\Lambda)$ therefore,
  \begin{align}
    \E\qa{e^{+p\|\nabla \zeta\|_{L^\infty(\Lambda)}}\1_{\|\nabla \zeta\|_{L^\infty(\Lambda)>T}}}
    &\leq \int_T^\infty pe^{pu} e^{-cu^2} \, du
      \nnb
    &\leq e^{-\frac12 c T^2} \int_{-\infty}^\infty e^{u} e^{-\frac12 cu^2/p^2} \, du
      \nnb
    &= e^{-\frac12 c T^2}(2\pi p^2/c)^{1/2}e^{\frac{p^2}{2c}}
      \leq e^{-\frac12 cT^2} e^{Cp^2}.
  \end{align}
  The third estimate is similar. 
\end{proof}

\subsection{Gaussian estimates from Section~\ref{sec:coupling-infvol}}
\label{sec:coupling-infvol-Gauss}

In this section, we prove the Gaussian estimates used in Section~\ref{sec:coupling-infvol},
i.e., the bound of the third term of Lemma~\ref{lem:tight}, Lemma~\ref{lem:GaussXrho-bis}
and an estimate for the Gaussian term in the proof of Lemma~\ref{lem:thirdderivative}.
We also state the Arzelà--Ascoli theorem
in the form applying to the $X^{-s}(\rho)$ spaces that we will use.

\begin{lemma}
  Suppose that $(f_i) \subset X^{-s}(\rho)$.
  Then $(f_i)$ is relatively compact if there are $\delta'>0$ and $s'<s$ and $\rho/\rho'\to 0$ such that
  $\norm{f_i}_{X^{-s',\delta'}(\rho')} \leq 1$.
\end{lemma}
\begin{proof}
  This follows from the Arzelà--Ascoli theorem. Indeed, to simplify notation, assume $t_0=1$. Then
  \begin{equation}
    \|f\|_{X^{-s}(\rho)} = \sup_{\tau \in [0,\infty)}\sup_{x\in \R^d} e^{-\tau s} \rho(x) |f_{e^{-\tau}}(x)|
  \end{equation}
  is actually a weighted $C^0$ norm on $[0,\infty) \times \R^d$ with weight $\tilde \rho(\tau,x)=e^{-\tau s}\rho(x)$.
  Similarly, the $X^{-s',\delta'}(\rho')$ norm is a weighted $C^{\delta'}$ norm with weight $\tilde \rho'(\tau,x)=e^{-\tau s'}\rho'(x)$.
  The assumption implies that $\tilde\rho/\tilde\rho' \to 0$ as $\tau+|x|\to\infty$ so the embedding is compact by the
  Arzelà--Ascoli theorem.
\end{proof}

\begin{lemma} \label{lem:GaussXrho}
  For any $s>0$, $\delta>0$ small, and $\rho(x)=(1+|x|^2)^{-\sigma/2}$ with $\sigma>0$, almost surely,
  \begin{equation}
    \sup_{m \geq 0}\norm{\Phi^{\GFF(m)}}_{X^{-s}(\rho)}\leq \sup_{m\geq 0}     \norm{\Phi^{\GFF(m)}}_{X^{-s,\delta}(\rho)} < \infty,
  \end{equation}
  and
  \begin{equation}
    \lim_{m\to 0} \norm{\Phi^{\GFF(0)}-  \Phi^{\GFF(m)}}_{X^{-s}(\rho)}= 0.
  \end{equation}
  Moreover, for any $r<1$,  with convergence in $L^p$ and (and thus almost surely along a subsequence),
  \begin{equation}
    \lim_{m\to 0} \norm{\Phi^{\GFF(0)}_{0,t_0}-\Phi^{\GFF(m)}_{0,t_0}}_{C^{r}(\rho)} = 0.
  \end{equation}
\end{lemma}

\begin{proof}
  For notational simplicity, again assume $t_0=1$.
  Throughout the proof we write $\kappa$ in place of $s$ so that $s$ can be used as an integration variable.
  We will first show that
  \begin{equation} \label{e:GFFproof}
    \|\Phi^{\GFF(0)}\|_{X^{-\kappa,\delta}(\rho)} < \infty.
  \end{equation}
  This follows from a version of Kolmogorov's continuity theorem: Denote $\Phi_t=\Phi^{\GFF(0)}_{t,1}$. Then 
  \begin{equation}
    \E[\Phi_t(x)\Phi_s(y)] %
    = \int_{s\vee t}^1 p_u(x-y) \, du, \qquad p_u(x) = \frac{e^{-|x|^2/4u}}{4\pi u},
  \end{equation}
  and, for any small $\gamma>0$, with $s=e^{-\sigma}$ and $t=e^{-\tau}$ and $|\sigma-\tau|\leq 1$ and $|x-y|\leq 1$, $s >t$,
  \begin{align} \label{e:decomp-cov-pf}
    \E[(\Phi_t(x)-\Phi_{s}(y))^{2}]
    &= \int_{t}^{s} p_u(0) \, du  + 2 \int_s^1 [p_u(0)-p_u(x-y)] \, du
      \nnb
    &= \frac{1}{4\pi} \log(s/t) + 2 \int_{s}^1 \frac{1-e^{-|x-y|^2/4u}}{4\pi u} \, du
      \nnb
      &\lesssim |\sigma-\tau|  +  t^{-2\gamma} |x-y|^{4\gamma}
        \lesssim t^{-2\gamma} |\xi-\eta|^{4\gamma}
        \lesssim e^{2\gamma\tau} |\xi-\eta|^{4\gamma}
    \end{align}
    where $\xi=(\tau,x)$ and $\eta=(\sigma,y)$. Also write $\Phi(\xi)=\Phi_{e^{-\tau}}(x)$ and $\Phi(\eta)=\Phi_{e^{-\sigma}}(y)$.
    
  The Kolmogorov continuity theorem for $B \subset \R^{n}$ with $n=d+1=3$ given in \cite[Theorem~3.5]{MR3236753}, or more precisely its proof, gives
  that there is a version of $\Phi$ such that
  \begin{equation} \label{e:KolmogorovO}
    \E\qa{\sup_{\xi,\eta \in B} \frac{|\Phi(\xi)-\Phi(\eta)|^p}{|\xi-\eta|^{\delta p}}}
    \leq C^p \int_{B\times B} \frac{\E|\Phi(\xi)-\Phi(\eta)|^p}{|\xi-\eta|^{\delta p+2n}} \, d\xi\, d\eta.
  \end{equation}
  By partitioning $\R^n$ into unit balls $B$, for example, one can deduce the following weighted estimate for the infinite space $\R^n$:
  \begin{equation} \label{e:Kolmogorovrho}
    \E\qa{ [\Phi]_{X^{-\kappa,\delta}(\rho)}^p}
    = \E\qa{ [\Phi]_{C^\delta(\tilde\rho)}^p}
    \leq C^p \int_{\R^3} \tilde\rho(\xi)^p \int_{B_4(0)}  \frac{\E|\Phi(\xi)-\Phi(\xi+\zeta)|^p}{|\zeta|^{\delta p+2n}} \, d\xi\, d\eta.
  \end{equation}
  
  In detail, to see the weighted estimate, for any $\bar\xi \in \R^3$, the local estimate \eqref{e:KolmogorovO} implies
  \begin{equation}
    \E\qa{\sup_{\xi\in B_1(\bar \xi)}\sup_{\zeta \in B_1(0)} \frac{|\Phi(\xi)-\Phi(\xi+\zeta)|^p}{|\zeta|^{\delta p}}}
    \leq C^p \int_{B_2(\bar\xi) \times B_4(0)} \frac{\E|\Phi(\xi)-\Phi(\xi+\zeta)|^p}{|\zeta|^{\delta p+2n}} \, d\xi\, d\zeta,
  \end{equation}
  where we have not optimized the size of the balls on the right-hand side.
  Since the minimum and maximum of the
  weight $\tilde\rho(\xi) = \tilde\rho(\tau,x) =  e^{-\kappa\tau}\rho(x)$ are comparable on unit balls of $\xi = (\tau,x)$,
  i.e., $\tilde\rho(\xi) \lesssim \tilde\rho(\bar\xi) \lesssim \tilde\rho(\xi)$, it also follows that
  (possibly with a different constant $C$)
  \begin{equation}
    \E\qa{\sup_{\xi\in B_1(\bar \xi)}\tilde\rho(\xi)^p\sup_{\zeta \in B_1(0)} \frac{|\Phi(\xi)-\Phi(\xi+\zeta)|^p}{|\zeta|^{\delta p}}}
    \leq C^p \int_{B_2(\bar\xi) \times B_4(0)} \tilde\rho(\xi)^p\frac{\E|\Phi(\xi)-\Phi(\xi+\zeta)|^p}{|\zeta|^{\delta p+2n}} \, d\xi\, d\zeta.
  \end{equation} 
  Using that for $f$ nonnegative and the sum over $\bar\xi$ below ranging over $(\frac12 \Z)^3$ for example,
  \begin{equation}
    \sup_{\xi\in\R^3} f(\xi)
    \leq  \sum_{\bar\xi} \sup_{\xi\in B_1(\bar\xi)} f(\xi),
    \qquad \sum_{\bar\xi} \int_{B_2(\bar\xi)} f(\xi) \, d\xi \lesssim \int_{\R^3} f(\xi) \, d\xi
  \end{equation}
  we obtain the claimed weighted estimate \eqref{e:Kolmogorovrho}.

  Since $\Phi$ is Gaussian (and hence the moments are controlled by the covariance), the right-hand side of \eqref{e:Kolmogorovrho} is bounded using \eqref{e:decomp-cov-pf} by
  \begin{equation}
    (Cp)^{p/2} \int_{\R^3} \tilde\rho(\xi)^p e^{\gamma\tau p} \, d\xi \int_{B_4(0)}  \frac{|\zeta|^{2\gamma p}}{|\zeta|^{\delta p+6}} \, d\eta
    \leq (\tilde Cp)^{p/2},
  \end{equation}
  provided that $\gamma<\kappa$ and $\sigma>0$ in the definition of $\tilde\rho$, that $\delta<2\gamma$, and that $p$ is sufficiently large.
  Since $\Phi_1(x)=0$ it also follows that 
  \begin{align}
    \|\Phi\|_{X^{-\kappa}(\rho)}
    &= \sup_{t \in (0,1]} t^{\kappa} \|\Phi_t-\Phi_1\|_{C^0(\rho)}
      \nnb
    &= \sup_{t \in (0,1]} \sup_{x\in\R^2} t^{\kappa}\rho(x) |\Phi_t(x)-\Phi_1(x)|
      \nnb
    &\leq \sup_{t\in(0,1]} t^{\kappa} |\log t|^\delta t^{-\kappa'} [\Phi]_{X^{-\kappa',\delta}(\rho)}
    \lesssim [\Phi]_{X^{-\kappa',\delta}(\rho)},
  \end{align}
  provided $\delta>0$ and $\kappa'<\kappa$.
  This completes the proof that \eqref{e:GFFproof} holds for some $\kappa>0$ and $\delta>0$.

  We now deduce the first two estimates stated in the lemma. Indeed, by definition and Ito's formula,
  \begin{equation}
    \Phi_{t,1}^{\GFF(m)}
    = \int_{t}^1 e^{-\frac12 m^2 s} e^{\frac12 \Delta s} dW_s
    = -\int_{t}^1 e^{-\frac12 m^2 s} d\Phi^{\GFF(0)}_{s,1}.
  \end{equation}
  Thus
  \begin{equation}
    \Phi_{t,1}^{\GFF(m)}
    = -\frac12 m^2  \int_t^1 e^{-\frac12 m^2 s} \Phi^{\GFF(0)}_{s,1}  \, ds + e^{-\frac12 m^2 t}\Phi_{t,1}^{\GFF(0)}.
  \end{equation}
  In particular,
  \begin{equation}
    \Phi_{t,1}^{\GFF(0)}- \Phi_{t,1}^{\GFF(m)}
    =
    \frac12 m^2  \int_t^1 e^{-\frac12 m^2 s} \Phi^{\GFF(0)}_{s,1}  \, ds + (1-e^{-\frac12 m^2 t}) \Phi_{t,1}^{\GFF(0)}.
  \end{equation}
  Therefore, for any $\kappa \in (0,1)$,
  \begin{align}
    \|\Phi_{t,1}^{\GFF(0)}- \Phi_{t,1}^{\GFF(m)}\|_{C^0(\rho)}
    &\leq \qa{ \frac12 m^2 \int_t^1 s^{-\kappa} \, ds + \frac12 m^2 t^{1-\kappa}} \|\Phi^{\GFF(0)}\|_{X^{-\kappa}(\rho)}
      \nnb
    &\lesssim m^2 \|\Phi^{\GFF(0)}\|_{X^{-\kappa}(\rho)}.
  \end{align}
  
  Finally, the third estimate for the $C^r$ norm of $\Phi_{0,t_0}^{\GFF(0)}-\Phi_{0,t_0}^{\GFF(m)}$ can be
  obtained by computing its covariance and using Sobolev embedding. Again assuming that $t_0=1$, we have
  \begin{equation}
    \Phi_{0,1}^{\GFF(0)}-\Phi_{0,1}^{\GFF(m)}
    =
    \int_0^1 (1-e^{-\frac12 m^2 s}) e^{\frac12 \Delta s} \, dW_s
  \end{equation}
  so the covariance of it is
  \begin{equation}
    C(x-y) = \int_0^1 (1-e^{-\frac12 m^2 s})^2 e^{\Delta s}(x,y) \, ds
    = \int_0^1 O(m^4 s^2) e^{\Delta s}(x,y) \, ds.
  \end{equation}
  In particular, for $n<2$,
  \begin{equation}
    |\nabla^{2n}C(0)| \lesssim m^4 \int_0^1 s^{2-n} \frac{ds}{s} < \infty,
  \end{equation}
  and the Sobolev embedding gives that for $r=1-2/p$ and $p$ sufficiently  large,
  \begin{equation}
    \E\qa{\|\Phi^{\GFF(0)}_{0,1}-\Phi^{\GFF(m)}_{0,1}\|_{C^r(\rho)}^p}
    \lesssim 
    (Cp)^{p/2} \sum_{n=0}^1 |\nabla^{2n}C(0)|^p \int  \rho(x)^p \, dx  \leq (Cp)^{p/2}
    .     
  \end{equation}
  This completes the proof.
\end{proof}

The next lemma gives the bound on the third term in \eqref{Phit0-nocenter-tightCk-bis}.

\begin{lemma}
For $t_0\leq 1$, and $p\geq 1$
\begin{equation}
  \sup_{T \geq 1}\sup_{m\in(0,1]}\E\qB{ \|P_T\Phi^{\GFF(m)}_{0,t_0}\|_{C^{-s}(\rho)}^p }<\infty
  .
\end{equation}
\end{lemma}
\begin{proof}
  It suffices to estimate the expected norms of $\Phi^{\GFF(m)}_{0,t_0}$ and of $(\eta_T, \Phi^{\GFF(m)}_{0,t_0})$ separately.
  
  For the first term we observe that the covariance of $\Phi_{0,t_0}^{\GFF(m)}$
  is smaller than that of $\Phi_{0,t_0}^{\GFF(0)}$ which is effectively a massive free field with mass $1/\sqrt{t_0}$,
  and thus the estimate follows from Proposition~\ref{prop:GFF-Besov}.
  (Alternatively, 
  we could have used the decomposition into $\Phi^{\GFF(m)}_{0,t_0}-(\eta,\Phi^{\GFF(m)}_{0,t_0})$
  and $(\eta_T, \Phi^{\GFF(m)}_{0,t_0}) -(\eta, \Phi^{\GFF(m)}_{0,t_0})$ to apply the massless
  version of  Proposition~\ref{prop:GFF-Besov}.)
  
  The second term $(\eta_T,\Phi^{\GFF(m)}_{0,t_0})$ is a single Gaussian random variable with variance
  \begin{align}
    \int_0^{t_0} e^{-m^2 u} (\eta_T,e^{\Delta u}\eta_T) \, du
    &=
      \int_0^{t_0} e^{-m^2 u} \int |\hat\eta_T(p)|^2 e^{-|p|^2 u} \, \frac{dp}{(2\pi)^2} \, du
      \nnb
    &\leq
       \int_0^{t_0} e^{-m^2 u} \int |\hat\eta_T(p)|^2 \, \frac{dp}{(2\pi)^2} \, du
    \leq \|\eta_T\|_{L^2}^2.
  \end{align}
  Since $\|\eta_T\|_{L^2} = T^{-d/2}\|\eta_1\|_{L^2} \leq \|\eta\|_{L^2}$ for $T\geq 1$ (see Remark~\ref{rk:etaT}), this implies the claim.
\end{proof}

We conclude with an estimate used in the proof of Lemma~\ref{lem:thirdderivative}.
The bound could be derived in a similar way as the one in the previous subsection, but it now
seems more efficient to use Sobolev embeddings to reduce the computation to a covariance estimate.

\begin{lemma} \label{lem:Gauss-Phit0}
  For any $p>0$, $k>0$, $\sigma\geq 3$, uniformly in $m>0$,
  \begin{equation}
    \E\qa{ \|\Phi^{\GFF}_{t_0}-\Tay_0\Phi_{t_0}^\GFF\|_{C^k(\rho)}^p}
    \leq C_{k,p}.
  \end{equation}
\end{lemma}
  
\begin{proof}
  By Sobolev embedding, for any $k>0$ and $p \geq 1$, there is $l>0$ such that
  \begin{equation}
    \|f\|_{C^k(\rho)} \lesssim \|f\|_{W^{l,p}(\rho)}.
  \end{equation}
  Indeed, the usual Sobolev embedding gives that, for any ball $B$ of fixed radius,
  \begin{equation}
    \|f\|_{C^k(B)}^p \lesssim  \sum_{n=0}^l \int_B |\nabla^n f(x)|^p \, dx
  \end{equation}
  and thus the weighted estimate follows by covering $\R^2$  by unit balls:
  \begin{align}
    \|f\|_{C^k(\rho)}^p
    \lesssim \sup_{\bar x\in \Z^2} \rho(\bar x)^p \|f\|_{C^k(B_2(\bar x))}^p
    &\lesssim \sup_{\bar x\in \Z^2} \rho(\bar x)^p \sum_{n=0}^l \int_{B_2(\bar x)} |\nabla^n f(x)|^p \, dx
      \nnb
    &\lesssim \sum_{n=0}^l \int_{\R^2} \rho(x)^p |\nabla^n f(x)|^p \, dx = \|f\|_{W^{l,p}(\rho)},
  \end{align}
  where the implicit constants can depend on $k,l,p,\sigma$.
 
  For any fixed $n$, using that $\Tay_0$ is a Taylor polynomial of degree $2$,
  one can check that
  \begin{equation} \label{e:Phit0-cov}
    \E |\nabla^n\Phi^\GFF_{t_0}(x)-\nabla^n\Tay_0\Phi^\GFF_{t_0}(x)|^2 \lesssim (1+|x|)^4.
  \end{equation}
  Indeed, $\Phi^{\GFF}_{t_0}$ is smooth with covariance
  \begin{equation}
    \int_{t_0}^\infty e^{-m^2 s} e^{\Delta_x s} \, ds.
  \end{equation}
  Thus for $n>0$, one can bound the two terms in the difference separately since $\nabla^l\Phi^{\GFF}_{t_0}$
  has bounded variance for any $l>0$ uniformly in $m>0$;  we omit the   details. For $n=0$, one has
  \begin{equation}
    \E |\Phi^\GFF_{t_0}(x)-\Phi^\GFF_{t_0}(0)|^2 \lesssim 1+\log(1+|x|) \lesssim (1+|x|)^4
  \end{equation}
  since $\Phi^{\GFF}_{t_0}$ is a log-correlated at large distances as $m\to 0$ (and smooth at small distances).
  The remaining terms contributing to the left-hand side of \eqref{e:Phit0-cov} with $n=0$
  come from the other terms the derivative terms in $\Tay_0$ which are bounded in terms of the
  bounded variances of  $\nabla^m\Phi^{\GFF}_{t_0}(0)$.
  
  Using that $\Phi^\GFF_{t_0}$ is Gaussian to bound higher moments,
  it follows from the Sobolev embedding that
  \begin{align}
    \E\qa{\|\Phi^\GFF_{t_0}-\Tay_0\Phi^\GFF_{t_0}\|_{C^k(\rho)}^p}
    &\lesssim \sum_{n=0}^l \int \E |\nabla^n\Phi^\GFF_{t_0}(x)-\nabla^n\Tay_0\Phi^\GFF_{t_0}(x)|^p \rho(x)^p \, dx
      \nnb
    & \lesssim
      \int (1+|x|^4)^{p/2} \rho(x)^p \, dx
      \lesssim 1
      ,
  \end{align}
  where we also used that  $\int (1+|x|^4)^{p/2}  \rho(x)^p \, dx$ is uniformly bounded for $\sigma>2$ and $p$ sufficiently large.
\end{proof}

\section{Renormalized potential with fractional charge insertions}\label{app:fracrenorm}

This appendix provides input for  the proof of Theorem~\ref{th:fracapp} in Appendix~\ref{app:renormpart}
(which is a restatement of Theorem~\ref{th:fracapp-bis}),
but it may also be of independent interest as it provides the analysis of the renormalized potential
for potentials that are sums of periodic potentials with different periodicities and
multiple ultraviolet regularization scales. The renormalized potential has many applications,
see for example \cite{MR4798104}.
Our analysis builds on \cite[Section 4]{MR4767492} (which in turn is based on \cite{MR914427}),
which concerns the situation with a single scale.

\subsection{Renormalized potential with two scales}
\label{app:renormpot}

We begin by introducing the renormalized potential with two regularization scales $\delta<\epsilon$.
Let $\beta\in[4\pi,6\pi)$ and $\Sigma\subset (-\sqrt{6\pi},\sqrt{6\pi})$ be a finite set
of the form\footnote{The case $\beta<4\pi$ is conceptually simpler, but would require some changes in notation, and we are primarily interested in $\beta=4\pi$, so we focus on the case $\beta\in[4\pi,6\pi)$. In fact, the case $\beta<4\pi$ can eventually be deduced from our analysis by setting $\eta(x,\pm \sqrt{\beta})=0$, but we do not discuss this further.}
\begin{equation}
  \Sigma=\Sigma_\beta \cup \Sigma_\alpha,
  \quad \text{where} \quad
  \Sigma_\beta=\{-\sqrt{\beta},\sqrt{\beta}\}
  \quad \text{and} \quad
  \Sigma_\alpha=\{\sqrt{4\pi}\alpha_1, \dots,\sqrt{4\pi}\alpha_n\}
  ,
\end{equation}
where the $\alpha_i$ are distinct and satisfy $0<\alpha_i^2<2-\beta/4\pi$ (note that $2-\beta/4\pi\leq 1$ for $\beta\in[4\pi,6\pi)$). Moreover, let $\eta\in L_c^\infty(\R^2\times \Sigma,\C)$,
meaning that for each $\sigma\in \Sigma$, $x\mapsto \eta(x,\sigma)$ is a complex valued, compactly supported, measurable, and essentially bounded function on $\R^2$,
$m>0$, and $0<\delta<\epsilon$.
We write $\xi=(x,\sigma)\in \R^2\times \Sigma$, and $\int_{\R^2\times \Sigma_\alpha}d\xi=\int_{\R^2}dx\sum_{\sigma\in \Sigma_\alpha}$, and similarly for $\Sigma_\beta$.

For $\varphi^1,\varphi^2\in C(\R^2)$, we define the microscopic potential
\begin{equation}
  V_{\delta^2,\epsilon^2}(\varphi^1,\varphi^2)
  =\int_{\R^2\times \Sigma_\alpha}d\xi\, \eta(\xi)\wick{e^{i\sigma\varphi^1(x)}}_\delta+\int_{\R^2\times \Sigma_\beta}d\xi\, \eta(\xi) \wick{e^{i\sigma \varphi^2(x)}}_\epsilon,
\end{equation}
where
\begin{equation}
  \wick{e^{i\sigma\varphi^1(x)}}_\delta = \delta^{-\frac{\sigma^2}{4\pi}}e^{i\sigma\varphi^1(x)},\qquad
  \wick{e^{i\sigma \varphi^2(x)}}_\epsilon = \epsilon^{-\frac{\sigma^2}{4\pi}}e^{i\sigma \varphi^2(x)}
\end{equation}
Then in terms of a decomposed Gaussian free field as in Section~\ref{sec:Gauss-decomp},
typically omitting the dependence on $m$ from the notation in $\Phi$ and $\dot C$,
\begin{equation} \label{e:app-Phi}
  \Phi_{s,t} = \Phi_{s,t}^{\GFF(m)} = \int_{s}^t \sqrt{\dot C_u} \, dW_u = \int_{s}^t e^{-\frac12 m^2 u} e^{\frac12 u\Delta} \, dW_u
  ,
\end{equation}
we introduce the renormalized potential: for deterministic $\varphi_1,\varphi_2\in C(\R^2)$,
\begin{equation}
  V_{t}(\varphi^1,\varphi^2|m,\delta,\epsilon)
  = -\log \E\qb{e^{-V_{\delta^2,\epsilon^2}(\Phi_{\delta^2,t}+\varphi^1, \Phi_{\epsilon^2,\epsilon^2\vee t}+\varphi^2)}},
\end{equation}
where we use the principal branch of the logarithm.

Our goal is to expand the renormalized potential as a series in terms of coefficients $\tilde V_t=\tilde V_t^{k,n}(\xi_1,...,\xi_n|m,\delta,\epsilon)$ as
\begin{multline} \label{e:Vt-expand}
  V_t(\varphi^1,\varphi^2|m,\delta,\epsilon)
  =\sum_{n=1}^\infty \frac{1}{n!}
  \sum_{k=0}^n {n \choose k}\int_{(\R^2\times \Sigma_\alpha)^k\times (\R^2\times \Sigma_\beta)^{n-k}}d\xi_1\cdots d\xi_n\, \eta(\xi_1)\cdots \eta(\xi_n)
  \\
  \tilde V_t^{k,n}(\xi_1,\dots,\xi_n)
  e^{i\sum_{j=1}^k \sigma_j \varphi^1(x_j)+i\sum_{j=k+1}^n \sigma_j \varphi^2(x_j)}
  .
\end{multline}
We will first define the coefficients through a recursion that they satisfy (which corresponds to a variant of the Polchinski equation for $V_t$)
and then verify that they indeed expand $V_t$ as above.

\begin{definition}
For $n=1$ define
\begin{align}\label{eq:vtilk1}
  \tilde V_t^{k,1}(\xi|m,\delta,\epsilon)=\delta_{k,1}e^{-\frac{\sigma^2}{2}(\int_{\delta^2}^t ds\,\dot C_s(0)+\frac{1}{4\pi}\log \delta^2)}+\delta_{k,0}e^{-\frac{\sigma^2}{2}(\int_{\epsilon^2}^{\epsilon^2 \vee t} ds\,\dot C_s(0)+\frac{1}{4\pi}\log \epsilon^2)}.
\end{align}
For $n \geq 2$  and $\delta^2\leq t\leq \epsilon^2$, define  $\tilde V_t^n = \tilde V_t^{n,n}$ recursively by
\begin{align}\label{eq:vtiln}
&\tilde V_t^n(\xi_1,\dots,\xi_n|m,\delta,\epsilon)\nnb
&\quad =
\frac{1}{2}\sum_{I_1\dot \cup I_2=[n]}\sum_{i\in I_1,j\in I_2}\sigma_i \sigma_j \int_{\delta^2}^t ds\, \dot C_s(x_i-x_j)\tilde V_s^{|I_1|}(\xi_{I_1}|m,\delta,\epsilon)\tilde V_s^{|I_2|}(\xi_{I_2}|m,\delta,\epsilon)\nnb
&\qquad \qquad \qquad \qquad \qquad \qquad \times e^{-\frac{1}{2}\sum_{i,j=1}^n \sigma_i \sigma_j \int_s^t dr\, \dot C_r(x_i-x_j)}.
\end{align}
and set $\tilde V_t^{k,n} = 0$ if $n \geq 2$ and $k<n$.

For $n\geq 2$ and $t>\epsilon^2$, define recursively
\begin{align}\label{eq:vtilkn}
&\tilde V_t^{k,n}(\xi_1,\dots,\xi_n|m,\delta,\epsilon)\nnb
&\quad = \delta_{n,k}e^{-\frac{1}{2}\sum_{i,j=1}^n \sigma_i \sigma_j \int_{\epsilon^2}^t ds\, \dot C_s(x_i-x_j)}\tilde V_{\epsilon^2}^n(\xi_1,\dots,\xi_n|m,\delta,\epsilon)\nnb
&\qquad +\frac{1}{2}\sum_{I_1\dot \cup I_2=[n]}\sum_{i\in I_1,j\in I_2}\sigma_i \sigma_j \int_{\epsilon^2}^t ds\, \dot C_s(x_i-x_j)\tilde V_s^{|I_1\cap [k]|,|I_1|}(\xi_{I_1}|m,\delta,\epsilon)\tilde V_s^{|I_2\cap[k]|,|I_2|}(\xi_{I_2}|m,\delta,\epsilon)\nnb
&\qquad \qquad \qquad \qquad \qquad \qquad \times e^{-\frac{1}{2}\sum_{i,j=1}^n \sigma_i \sigma_j \int_s^t dr\, \dot C_r(x_i-x_j)}.
\end{align}
\end{definition}

The parameters $n$ respectively $k,n$ of $\tilde V^n_t$ and $\tilde V_t^{k,n}$ are determined by
the arguments of the functions, and for notational simplicity we will often drop them when the arguments are present.

The following bounds for the kernels $\tilde V_t$ %
will guarantee convergence of the expansion of the renormalized potential
uniformly in $\delta,\epsilon,m>0$.
For $n\geq 1$ and $f:(\R^2\times \Sigma)^n\to \C$, define %
\begin{equation}
\|f\|_n=\begin{cases}
\|f\|_{L^\infty(\R^2\times \Sigma)}, & n=1\\
\sup_{\xi_1\in \R^2\times \Sigma}\int_{(\R^2\times \Sigma)^{n-1}}d\xi_2\cdots d\xi_n |f(\xi_1,\dots,\xi_n)|, & n\geq 2.
\end{cases}
\end{equation}
\newcommand{\tdist}{\lambda}
In fact, as in \cite[Extension 2.3]{MR914427}, we can alternatively use the following stronger norms. Define
\begin{equation} \label{e:rho-def}
  \tdist(x_1,\cdots, x_n) = \inf_T \sum_{kl \in T} |x_k-x_l|
\end{equation}
where the infimum is over spanning trees $T$ on the complete graph with vertices $\{1,\dots,n\}$
and $kl\in T$ denotes that $kl=\{k,l\}$ is an edge in $T$.
In particular, $\tdist(x_1,x_2)=|x_1-x_2|$ and
\begin{equation} \label{e:rho-add}
  \tdist(x_k,x_{k+1}) + \tdist(x_1,\dots, x_k) + \tdist(x_{k+1}, \dots, x_n) \geq \tdist(x_1,\dots,x_n).
\end{equation}
Then the following proposition also holds with the following weighted norm:
\begin{equation}
  \|f\|_n = \|f\|_{n,\tdist} = \sup_{\xi_1\in \R^2\times \Sigma}\int_{(\R^2\times \Sigma)^{n-1}}d\xi_2\cdots d\xi_n |f(\xi_1,\dots,\xi_n)| e^{\tdist(x_1,\dots,x_n)}, \quad n\geq 2.
\end{equation}

\begin{proposition}\label{pr:vtnorm}
For $t>0$ and $n\neq 2$ or $n=2$ and $\sigma_1\sigma_2\neq -\beta$, and $0\leq k\leq n$, there exist functions $h_t^{k,n}:(\R^2\times \Sigma)^n\to [0,\infty]$ which are symmetric in their variables, independent of $m,\delta,\epsilon$ such that for $0<\delta^2<\epsilon^2<t<m^{-2}$, one has
\begin{equation}
|\tilde V_t^{k,n}(\xi_1,\dots,\xi_n|m,\delta,\epsilon)|\leq h_t^{k,n}(\xi_1,\dots,\xi_n)
\end{equation}
for all $\xi_1,\dots,\xi_n\in \R^2\times \Sigma$, and for $0<t<1$ 
\begin{equation}
\|h_t^{k,n}\|_n\leq n^{n-2}t^{-1}(C_\Sigma t^{1-\frac{\beta}{8\pi}})^n
\end{equation}
for some finite $C_\Sigma>0$ (which depends only on $\Sigma$).
\end{proposition} 

For Theorem~\ref{th:fracapp}, we only care about the situation where $t$ is of order one and $\epsilon$ is small, so we do not state bounds in the region $t\in[\delta^2,\epsilon^2]$ (though these also feature in the proof).

The proof is given in Section~\ref{app:vtbounds}.
It follows quite carefully \cite[Section 4]{MR4767492}. While the proof is quite lengthy, it is essentially  routine. %

\begin{proposition}\label{pr:rpexp}
  For $0<\delta^2<\epsilon^2<t<\min(1,m^{-2})$, $\varphi^1,\varphi^2\in C(\R^2)$ and $\eta\in L_c^\infty(\R^2\times \Sigma,\C)$ satisfying
  (with $C_\Sigma$ is as in Proposition~\ref{pr:vtnorm})
\begin{equation}\label{eq:etabound}
\|\eta\|_{L^\infty(\R^2\times \Sigma)}<\frac{1}{2 eC_\Sigma t^{1-\frac{\beta}{8\pi}}},
\end{equation}
the renormalized potential is given by the series in \eqref{e:Vt-expand}, i.e.,
  \begin{multline} \label{e:Vt-expand2}
    V_t(\varphi^1,\varphi^2|m,\delta,\epsilon)
    =\sum_{n=1}^\infty \frac{1}{n!}
    \sum_{k=0}^n {n \choose k}\int_{(\R^2\times \Sigma_\alpha)^k\times (\R^2\times \Sigma_\beta)^{n-k}}d\xi_1\cdots d\xi_n\, \eta(\xi_1)\cdots \eta(\xi_n)
    \\
    \tilde V_t(\xi_1,\dots,\xi_n)
    e^{i\sum_{j=1}^k \sigma_j \varphi^1(x_j)+i\sum_{j=k+1}^n \sigma_j \varphi^2(x_j)}
    .
  \end{multline}
  The same holds for complex-valued $\varphi^1$ and $\varphi^2$ with bounded imaginary part (with $C_\Sigma$ now depending on the bound on the imaginary part as well).
\end{proposition}

The following proposition states that the kernels $\tilde V_t$ are equivalent to cumulants. This
observation is a variant of the result of \cite{MR914427}.
For completeness, it is proved in Section~\ref{sec:cumulants}.

\begin{proposition} \label{prop:cumulants}
  The solution to the recursion for $\tilde V_t$ is given by cumulants:
  if $\delta^2\leq t \leq \epsilon^2$,
  \begin{align} \label{eq:tVtn}
    \tilde V_t(\xi_1,\dots, \xi_n|m,\delta,\epsilon)
  &=     (-1)^{n-1} \avg{\wick{e^{i\sigma_1\Phi_{\delta^2,t}(x_1)}}_\delta;\cdots ;\wick{e^{i\sigma_n\Phi_{\delta^2,t}(x_n)}}_\delta}^\mathsf T
    ,
\end{align}
and for $t\geq \epsilon^2$ as well as $\sigma_1,\dots,\sigma_k\in \Sigma_\alpha$ and $\sigma_{k+1},\dots,\sigma_n\in \Sigma_\beta$ (with $0\leq k\leq n$),
\begin{align}\label{eq:tVtkn}
\tilde V_t(\xi_1,\dots, \xi_n|m,\delta,\epsilon) 
  &=
  (-1)^{n-1} 
    \langle \wick{e^{i\sigma_1\Phi_{\delta^2,t}(x_1)}}_\delta;\cdots ;\wick{e^{i\sigma_k\Phi_{\delta^2,t}(x_k)}}_\delta; \nnb
 & \qquad\qquad\qquad \wick{e^{i\sigma_{k+1}\Phi_{\epsilon^2,t}(x_{k+1})}}_\epsilon; \cdots ; \wick{e^{i\sigma_{n}\Phi_{\epsilon^2,t}(x_{n})}}_\epsilon\rangle^\mathsf T.
\end{align}
\end{proposition}

\subsection{Proof of Proposition~\ref{pr:rpexp}}

We begin
with the expansion for fixed $\epsilon,\delta,m>0$.

\begin{lemma}\label{le:cumuexp}
For $t\geq\delta^2$ and $\varphi^1,\varphi^2\in C(\R^2)$, the function
\begin{equation}
  z\mapsto f_t(z)
  =     \log \E\qb{e^{-zV_{\delta^2,\epsilon^2}(\varphi^1+\Phi_{\delta^2,t},\varphi^2+\Phi_{\epsilon^2,\epsilon^2\vee t})}}
\end{equation}
is analytic in some disk containing the origin (which possibly depends on $\delta,\epsilon,t$).
In this disk, for $\delta^2\leq t\leq \epsilon^2$, it has the series expansion
\begin{align}
 f_t(z)&=
-z\int_{\R^2\times \Sigma_\beta}d\xi\, \eta(\xi)\wick{e^{i\sigma\varphi^2(x)}}_{\epsilon}\nnb
&\quad + \sum_{n=1}^\infty \frac{z^n}{n!}(-1)^n \int_{(\R^2\times \Sigma_\alpha)^n}d\xi_1\cdots d\xi_n\, \eta(\xi_1)\cdots \eta(\xi_n)e^{i\sum_{j=1}^n \sigma_j \varphi^1(x_j)}\nnb
&\qquad \qquad \qquad \qquad  \times \avg{\wick{e^{i\sigma_1\Phi_{\delta^2,t}(x_1)}}_\delta;\cdots ;\wick{e^{i\sigma_n\Phi_{\delta^2,t} (x_n)}}_\delta}^\mathsf T,
\end{align}
while for $t\geq \epsilon^2$, the series expansion is given by
\begin{align}
  f_t(z)&=\sum_{n=1}^\infty \frac{z^n}{n!}(-1)^n\sum_{k=0}^n {n \choose k}\int_{(\R^2\times \Sigma_\alpha)^k\times (\R^2\times \Sigma_\beta)^{n-k}}d\xi_1\cdots d\xi_n\, \eta(\xi_1)\cdots \eta(\xi_n) \nnb
          &\qquad \times e^{i\sum_{j=1}^k \sigma_j \varphi^1(x_j)+\sum_{j=k+1}^n \sigma_j \varphi^2(x_j)}\nnb
  &\qquad \times\avg{\wick{e^{i\sigma_1\Phi_{\delta^2,t}(x_1)}}_\delta;\cdots ;\wick{e^{i\sigma_k\Phi_{\delta^2,t}(x_k)}}_\delta; 
\wick{e^{i\sigma_{k+1}\Phi_{\epsilon^2,t}(x_{k+1})}}_\epsilon; \cdots ; \wick{e^{i\sigma_{n}\Phi_{\epsilon^2,t}(x_{n})}}_\epsilon}^\mathsf T
    .
\end{align}
\end{lemma}
\begin{proof}
  Let us begin with the claims for $\delta^2\leq t\leq \epsilon^2$. Here 
  $V_{\delta^2,\epsilon^2}(\cdot,\cdot)$
  is deterministically bounded so that the function 
\begin{equation}
  z\mapsto
  \E\qb{e^{-zV_{\delta^2,\epsilon^2}(\varphi^1+\Phi_{\delta^2,t},\varphi^2)}}
\end{equation}
is an entire function.
This function takes the value $1$ at the origin $z=0$,
so there exists some disk containing the origin in which it does not vanish. This means that $f_t$, as the logarithm of this function, is also analytic in this disk.
In this disk of convergence, $f_t$ is of course (up to the minus sign) the cumulant generating function of this random variable, so we have 
\begin{equation}
  f_t(z)
  =\sum_{n=1}^\infty \frac{z^n}{n!}(-1)^n \avg{V_{\delta^2,\epsilon^2}(\varphi^1+\Phi_{\delta^2,t},\varphi^2);\cdots;V_{\delta^2,\epsilon^2}(\varphi^1+\Phi_{\delta^2,t},\varphi^2)}^\mathsf T
  .
\end{equation}
Using multilinearity of the cumulant, Fubini's theorem, and the fact that cumulants of order greater than one vanish if one of the entries is a constant, we find the desired series expansion in the case $\delta^2\leq t\leq \epsilon^2$.

The proof of analyticity in the second case is identical. Also the expansion is derived in a similar manner, but now with the differences that when we use multilinearity and Fubini, we have to sum over all possible subsets $I\subset \{1,\dots,n\}$ of indices $i$ for which $\sigma_i\in \Sigma_\alpha$ (and $\sigma_j\in \Sigma_\beta$ for $j\notin I$). By renaming our integration variables and using the fact that cumulants are permutation invariant, we see that it is only the cardinality of the set $I$ that matters. Since the number of subsets $I$ of cardinality $k$ is ${n\choose k}$, this argument produces the claimed series expansion.
\end{proof}

Using the bounds on $\tilde V_t^{k,n}$ from Proposition~\ref{pr:vtnorm} together with the identification with cumulants
from Proposition~\ref{prop:cumulants},
we can now argue that our cumulant expansion actually converges in a larger domain and we can use it to control the renormalized potential up to a scale $t$ which is of order one (depending only on $\Sigma$ and $\eta$).

\begin{proof}[Proof of Proposition~\ref{pr:rpexp}]
Consider the function $f_t$ from Lemma \ref{le:cumuexp}, observe
$V_t(\varphi^1,\varphi^2|m,\delta,\epsilon)=-f_t(1)$, so it is sufficient for us to show that $-f_t(1)$ is given by this series expansion.

First of all, we note since  that since $e^{f_t(z)}$ is an entire function, %
using that cumulants are equivalent to $\tilde V_t$ by Proposition~\ref{prop:cumulants},
for $z$ in any neighborhood of the origin where the series converges (absolutely), we have 
\begin{align}
-f_t(z)&=\sum_{n=1}^\infty \frac{z^n}{n!}\sum_{k=0}^n {n\choose k}\int_{(\R^2\times \Sigma_\alpha)^k\times (\R^2\times \Sigma_\beta)^{n-k}}d\xi_1\cdots d\xi_n \eta(\xi_1)\cdots \eta(\xi_n)\nnb
&\qquad\qquad\qquad\qquad \times e^{i\sum_{j=1}^k \sigma_j\varphi^1(x_j)+\sum_{j=k+1}^n \sigma_j \varphi^2(x_j)}\tilde V_t^{k,n}(\xi_1,\dots,\xi_n|m,\delta,\epsilon).
\end{align}
Our task is to show that under the restriction \eqref{eq:etabound}, this series converges absolutely for $z=1$.
Note that the terms  with $n \leq 2$ do not affect convergence of this series, so for the sum of the $n\geq 3$ terms, we have by Proposition \ref{pr:vtnorm} and the inequality $\frac{n^n}{n!}\leq e^n$, the bound 
\begin{align}
&\sum_{n=3}^\infty \frac{|z|^n}{n!}\sum_{k=0}^n {n\choose k}\|\eta\|_{L^1(\R^2\times \Sigma)}\|\eta\|_{L^\infty(\R^2\times \Sigma)}^{n-1} \|\tilde V_t^{k,n}\|_n\nnb
&\leq C_\Sigma  \|\eta\|_{L^1(\R^2\times \Sigma)} t^{-\frac{\beta}{8\pi}} \sum_{n=3}^\infty \frac{|z|^n}{n!}\left(C_\Sigma t^{1-\frac{\beta}{8\pi}}\|\eta\|_{L^\infty(\R^2\times \Sigma)}\right)^{n-1} n^{n-2}\sum_{k=0}^n 1\nnb
&\leq  2|z|C_\Sigma \|\eta\|_{L^1(\R^2\times \Sigma)}t^{-\frac{\beta}{8\pi}}\sum_{n=3}^\infty \frac{1}{n^2} \left(2 e |z| C_\Sigma t^{1-\frac{\beta}{8\pi}}\|\eta\|_{L^\infty(\R^2\times \Sigma)}\right)^{n-1}.
\end{align}
This series converges for 
\begin{align}
|z|\leq \frac{1}{2e C_\Sigma t^{1-\frac{\beta}{8\pi}}\|\eta\|_{L^\infty(\R^2\times \Sigma)}},
\end{align}
so under our assumption \eqref{eq:etabound}, the radius of convergence is greater than one. In particular, we have a convergent series expansion at $z=1$. This concludes the proof.
\end{proof}

\subsection{Cumulants and the recursion -- proof of Proposition~\ref{prop:cumulants}}
\label{sec:cumulants}

The proposition is essentially the equivalence of the Polchinski equation
in the representation \eqref{e:Vt-expand} with the Mayer expansion, observed in \cite{MR914427} in particular  \cite[Lemma 3.3]{MR914427}.
For completeness and because our setting is more general, we include a detailed proof.

\begin{lemma}\label{le:gexpcumu}
Let $X_1,\dots,X_n$ be centered and jointly Gaussian random variables. We then have 
\begin{align}
\avg{e^{iX_1};\cdots;e^{iX_n}}^\mathsf T=\sum_{G\in \mathcal C_n}\prod_{\{i,j\}\in G}(e^{-\avg{X_iX_j}}-1)\prod_{i=1}^n e^{-\frac{1}{2}\avg{X_i^2}},
\end{align}
where $\mathcal C_n$ denotes the set of connected graphs with vertex set $\{1,\dots,n\}$ and $\{i,j\}\in G$ indicates that $\{i,j\}$ is an edge in the graph $G$. 
\end{lemma}
\begin{proof}
Let us begin with a general statement about the relationship of cumulants and moments. Let us assume that we are given (for simplicity, bounded) random variables $Y_1,\dots,Y_n$. We claim that if there exist numbers $(\kappa_I)_{I\subset[n]}$ which satisfy for each $I\subset [n]$
\begin{equation}
\avg{\prod_{i\in I}Y_i}=\sum_{\pi\in \frP_I}\prod_{B\in \pi}\kappa_B,
\end{equation}
where $\frP_I$ denotes the set of partitions of $I$, then 
\begin{equation}\label{eq:kappacumu}
\kappa_{[n]}=\avg{Y_{1};\cdots; Y_{n}}^\mathsf T.
\end{equation}
To see why this is true, note that using the definition of cumulants as a sum over partitions, and our assumption about expressing moments in terms of the numbers $\kappa_B$, we have 
\begin{align}
\avg{Y_1;\cdots;Y_n}^\mathsf T=\sum_{\pi\in \frP_n}(-1)^{|\pi|-1}(|\pi|-1)!\prod_{B\in \pi}\sum_{\tau_B\in P_B}\prod_{C\in \tau_B}\kappa_{C}.
\end{align}
Next we note that we can join all of the partitions $\tau_B$ into a single partition of $[n]$ and this partition, let us call it $\tau$, is a refinement of $\pi$ in the sense that each block of $\tau$ is a subset of a block of $\pi$. If we write $\tau\leq \pi$ for this relation, then we have in fact
\begin{align}
\avg{Y_1;\cdots;Y_n}^\mathsf T=\sum_{\tau\in \frP_n}\prod_{C\in \tau}\kappa_C\sum_{\pi\in \frP_n}\mathbf 1\{\tau\leq \pi\}(-1)^{|\pi|-1}(|\pi|-1)!.
\end{align}
From \cite[Section 7, Example 1]{MR174487}, we know that 
\begin{align}
\sum_{\pi\in \frP_n}\mathbf 1\{\tau\leq \pi\}(-1)^{|\pi|-1}(|\pi|-1)!=\delta_{\tau,\{[n]\}},
\end{align}
so we see that 
\begin{align}
\kappa_{[n]}=\avg{Y_1;\cdots ;Y_n}^\mathsf T
\end{align}
as claimed.

To apply this to our setting, we take $Y_j=e^{iX_j}$ and write for each $I\subset [n]$
\begin{align}
\avg{\prod_{j\in I}e^{iX_j}}=e^{-\frac{1}{2}\avg{(\sum_{j\in I}X_j)^2}}&=\prod_{j\in I} e^{-\frac{1}{2}\avg{X_j^2}}\prod_{\substack{j,k\in I:\\ j<k}}e^{-\avg{X_jX_k}}\nnb
&=\prod_{j\in I} e^{-\frac{1}{2}\avg{X_j^2}}\prod_{\substack{j,k\in I:\\ j<k}}(e^{-\avg{X_jX_k}}-1+1)\nnb
&=\prod_{j\in I} e^{-\frac{1}{2}\avg{X_j^2}}\sum_{G\subset \{(j,k)\in I^2: j<k\}}\prod_{(j,k)\in G}(e^{-\avg{X_jX_k}}-1).
\end{align}
Note that any $G\subset \{(j,k)\in I^2: j<k\}$ can be identified with a graph with vertex set $I$ and edges between vertices $j$ and $k$ (with $j<k$) if and only if $(j,k)\in G$. If we write $\mathcal G_I$ for the set of all graphs with vertex set $I$, then we have shown that 
\begin{align}
\avg{\prod_{j\in I}e^{iX_j}}=\prod_{j\in I} e^{-\frac{1}{2}\avg{X_j^2}}\sum_{G\in \mathcal G_I}\prod_{\{j,k\}\in G}(e^{-\avg{X_jX_k}}-1),
\end{align} 
where $\{j,k\}\in G$ indicates that $\{j,k\}$ is an edge in $G$.

Now each graph $G\in \mathcal G_I$ can be decomposed into connected components, and each component is a graph on some subset of $I$. Summing over all graphs on $I$ is equivalent to summing over all partitions of $I$ and summing over all connected graphs on each block of the partition. So we have 
\begin{align}
\avg{\prod_{j\in I}e^{iX_j}}&=\prod_{j\in I} e^{-\frac{1}{2}\avg{X_j^2}}\sum_{\pi\in P_I}\prod_{B\in \pi}\sum_{G\in \mathcal C_B}\prod_{\{j,k\}\in G}(e^{-\avg{X_jX_k}}-1)\nnb
&=\sum_{\pi\in P_I}\prod_{B\in \pi}\sum_{G\in \mathcal C_B}\prod_{\{j,k\}\in G}(e^{-\avg{X_jX_k}}-1)\prod_{j\in B} e^{-\frac{1}{2}\avg{X_j^2}}
\end{align} 
where $\mathcal C_I$ denotes the set of connected graphs with vertex set $I$. This is precisely of the form 
\begin{align}
\avg{\prod_{j\in I}e^{iX_j}}=\sum_{\pi\in P_I}\prod_{B\in \pi}\kappa_B,
\end{align}
with 
\begin{align}
\kappa_B=\sum_{G\in \mathcal C_B}\prod_{\{j,k\}\in G}(e^{-\avg{X_jX_k}}-1)\prod_{j\in B} e^{-\frac{1}{2}\avg{X_j^2}}.
\end{align}
Thus the claim follows by \eqref{eq:kappacumu}.
\end{proof}

We now prove that the cumulants satisfy essentially the same recursions as $\tilde V_t$ (up to a conventional sign).
To simplify notation slightly, we denote the cumulants as follows.
For $\delta^2\leq t\leq \epsilon^2$ and $\sigma_1,\dots,\sigma_n\in \Sigma_\alpha$,
\begin{align}
  U_t(\xi_1,\dots,\xi_n|m,\delta,\epsilon)
  &= U_t^n(\xi_1,\dots,\xi_n|m,\delta,\epsilon)\nnb
  &= \avg{\wick{e^{i\sigma_1\Phi_{\delta^2,t}(x_1)}}_\delta;\cdots ;\wick{e^{i\sigma_n\Phi_{\delta^2,t}(x_n)}}_\delta}^\mathsf T
    ,
\end{align}
and for $t\geq \epsilon^2$ as well as $\sigma_1,\dots,\sigma_k\in \Sigma_\alpha$ and $\sigma_{k+1},\dots,\sigma_n\in \Sigma_\beta$ (with $0\leq k\leq n$),
\begin{align}\label{eq:utkn}
  U_t(\xi_1,\dots,\xi_n|m,\delta,\epsilon)
  &=U_t^{k,n}(\xi_1,\dots,\xi_n|m,\delta,\epsilon) \nnb
 &=\langle \wick{e^{i\sigma_1\Phi_{\delta^2,t}(x_1)}}_\delta;\cdots ;\wick{e^{i\sigma_k\Phi_{\delta^2,t}(x_k)}}_\delta; \nnb
 & \qquad\qquad\qquad \wick{e^{i\sigma_{k+1}\Phi_{\epsilon^2,t}(x_{k+1})}}_\epsilon; \cdots ; \wick{e^{i\sigma_{n}\Phi_{\epsilon^2,t}(x_{n})}}_\epsilon\rangle^\mathsf T.
\end{align}
We extend this definition to arbitrary $\xi_1,\dots,\xi_n\in \Sigma$ satisfying $|\{i: \xi_i\in \Sigma_\alpha\}|=k$ by imposing permutation invariance in the arguments (recall that cumulants are permutation invariant).
As before (and indicated in the notation), the parameters $n$ respectively $k,n$ of  the functions $U^n_t$ and $U_t^{k,n}$ are determined by the arguments of the functions, and for notational simplicity we will therefore often drop them
when the arguments are present.

We will now characterise these functions in terms of a differential recursion. The arguments we present are essentially variants of \cite[Lemma 3.3]{MR914427}. We focus first on the case $t\in[\delta^2,\epsilon^2]$.

\begin{lemma}\label{le:utdif1}
For $t\in [\delta^2,\epsilon^2]$, the sequence of functions $(U_t^n)_{n=1}^\infty$ is the unique solution to the following problem:
\begin{enumerate}
\item For $n\geq 1$ and $\xi_1,\dots,\xi_n\in \Sigma_\alpha$, the function $t\mapsto U_t^n(\xi_1,\dots,\xi_n|m,\delta,\epsilon)$ is continuous on $[\delta^2,\epsilon^2]$ and continuously differentiable on $(\delta^2,\epsilon^2)$.
\item For $n\geq 1$ and $\xi_1,\dots,\xi_n\in \Sigma_\alpha$, we have 
\begin{equation}
U_{\delta^2}^n(\xi_1,\dots,\xi_n|m,\delta,\epsilon)=\delta_{n,1}\delta^{-\frac{\sigma^2}{4\pi}}.
\end{equation}
\item For $n=1$, $\xi\in \Sigma_\alpha$, and $\delta^2\leq t\leq \epsilon^2$ we have 
\begin{align}
U_t^1(\xi|m,\delta,\epsilon)=e^{-\frac{\sigma^2}{2}(\int_{\delta^2}^t ds\dot C_s(0)+\frac{1}{4\pi}\log \delta^2)}.
\end{align}
\item For $n\geq 2$, $\xi_1,\dots.,\xi_n\in \Sigma_\alpha$,  and $t\in(\delta^2,\epsilon^2)$ we have 
\begin{align}
  &\partial_t U_t^n(\xi_1,\dots,\xi_n|m,\delta,\epsilon)\nnb
  &=-\frac{1}{2}\sum_{I_1\dot\cup I_2=[n]}\sum_{i\in I_1,j\in I_2}\sigma_i \sigma_j \dot C_t(x_i-x_j)U_t^{|I_1|}(\xi_{I_1}|m,\delta,\epsilon)U_t^{|I_2|}(\xi_{I_2}|m,\delta,\epsilon)\nnb
&\quad -\frac{1}{2}\sum_{i,j=1}^n \sigma_i\sigma_j \dot C_t(x_i-x_j)U_t^n(\xi_1,\dots,\xi_n|m,\delta,\epsilon),
\end{align}
where, as before, $[n]=\{1,\dots,n\}$, and the sum is over non-empty sets $I_1,I_2\subset [n]$ which are disjoint and $I_1\cup I_2=[n]$. We have also written for $I=\{i_1,\dots,i_k\}$, $\xi_I=\{\xi_{i_1},\dots,\xi_{i_k}\}$. 
\end{enumerate}
\end{lemma}
\begin{proof}
\subproof{uniqueness}
Let us first prove that if the problem has a solution, it is unique. We prove this by induction. Clearly $U_t^1$ is uniquely specified by item (iii). Let us then assume that uniqueness holds for $n\leq k$ with $k\geq 1$, and that for $n=k+1$ we have two solutions $U_t^{k+1}$ and $\widetilde U_t^{k+1}$. Writing $\Delta_t=U_t^{k+1}(\xi_1,\dots,\xi_{k+1}|m,\delta,\epsilon)-\widetilde U_t^{k+1}(\xi_1,\dots,\xi_{k+1}|m,\delta,\epsilon)$, we see from our induction hypothesis and item (iv) that $\Delta_t$ satisfies the equation 
\begin{align}
\partial_t \Delta_t=-\frac{1}{2}\sum_{i,j=1}^n \sigma_i\sigma_j \dot C_t(x_i-x_j)\Delta_t
\end{align}
with the initial condition (coming from item (ii) since $k+1>1$)
\begin{equation}
\Delta_{\delta^2}=0.
\end{equation} 
By uniqueness of solutions to linear first order differential equations, we have $\Delta_t=0$, so we see that if our problem has a solution it is unique.

\subproof{existence -- item (i)}
Using for example the definition of cumulants, this boils down to the corresponding claims for the covariance $\avg{\Phi_{\delta^2,t}(x)\Phi_{\delta^2,t}(y)}=\int_{\delta^2}^t ds\, \dot C_s(x-y)$ which is a smooth function.

\subproof{existence -- item (ii)}
Note that $\Phi_{\delta^2,\delta^2}(x)=0$ so 
\begin{align}
U_{\delta^2}^n(\xi_1,\dots,\xi_n|m,\delta,\epsilon)=\avg{\delta^{-\frac{\sigma_1^2}{4\pi}};\cdots; \delta^{-\frac{\sigma_n^2}{4\pi}}}^\mathsf T.
\end{align}
Cumulants of order two or greater involving constants vanish (as one readily checks e.g. from the cumulant generating function), so this vanishes for $n\neq 1$. For $n=1$, the cumulant is simply the expectation, and we find our claim.

\subproof{existence -- item (iii)}
Again, the order $n=1$ cumulant is simply the expectation so we have 
\begin{align}
  U_t^1(\xi|m,\delta,\epsilon)
  =\E\qb{\delta^{-\frac{\sigma^2}{4\pi}}e^{i\sigma \Phi_{\delta^2,t}(x)}}=\delta^{-\frac{\sigma^2}{4\pi}}e^{-\frac{\sigma^2}{2}\E[\Phi_{\delta^2,t}(x)^2]}
\end{align} 
and the claim follows from \eqref{e:app-Phi}. %
\medskip

\subproof{existence -- item (iv)}
This is the most involved part of the proof. For this, we use multilinearity of the cumulants and Lemma \ref{le:gexpcumu} to write 
\begin{multline}
  U_t^n(\xi_1,\dots,\xi_n|m,\delta,\epsilon)=\delta^{-\frac{1}{4\pi}\sum_{j=1}^n \sigma_j^2} e^{-\frac{1}{2}\sum_{j=1}^n \sigma_j^2 \int_{\delta^2}^tds\, \dot C_s(0)}
  \\
  \times \sum_{G\in \mathcal C_n}\prod_{\{i,j\}\in G}\left(e^{-\sigma_i \sigma_j \int_{\delta^2}^t ds\, \dot C_s(x_i-x_j)}-1\right) .
\end{multline}
Differentiating, we find 
\begin{align}
&\partial_t U_t^n(\xi_1,\dots,\xi_n|m,\delta,\epsilon)\nnb
&=-\frac{1}{2}\sum_{j=1}^n \sigma_j^2 \dot C_t(0)U_t^n(\xi_1,\dots,\xi_n|m,\delta,\epsilon)\nnb
&\quad -\delta^{-\frac{1}{4\pi}\sum_{j=1}^n \sigma_j^2} e^{-\frac{1}{2}\sum_{j=1}^n \sigma_j^2 \int_{\delta^2}^tds\, \dot C_s(0)} \sum_{G\in \mathcal C_n}\sum_{\{i,j\}\in G}\sigma_i\sigma_j \dot C_t(x_i-x_j)\nnb
  &\qquad \times e^{-\sigma_i\sigma_j\int_{\delta^2}^t ds\, \dot C_s(x_i-x_j)} \prod_{\substack{\{p,q\}\in G:\\ \{p,q\}\neq \{i,j\}}}\left(e^{-\sigma_p\sigma_q\int_{\delta^2}^t ds\, \dot C_s(x_p-x_q)}-1\right).
\end{align}
By writing
\begin{equation}
  e^{-\sigma_i\sigma_j\int_{\delta^2}^t ds\, \dot C_s(x_i-x_j)}=e^{-\sigma_i\sigma_j\int_{\delta^2}^t ds\, \dot C_s(x_i-x_j)}-1+1
\end{equation}
and using that
\begin{equation}
  \sum_G\sum_{\{i,j\}\in G} f(G)=\frac{1}{2}\sum_{i,j\in [n], i\neq j}\sum_G\mathbf 1\{\{i,j\}\in G\} f(G),
\end{equation}
we obtain
\begin{align}
&\partial_t U_t^n(\xi_1,\dots,\xi_n|m,\delta,\epsilon)\nnb
&=-\frac{1}{2}\sum_{i,j=1}^n \sigma_i\sigma_j \dot C_t(x_i-x_j)U_t^n(\xi_1,\dots,\xi_n|m,\delta,\epsilon)\nnb
&\quad +\delta^{-\frac{1}{4\pi}\sum_{j=1}^n \sigma_j^2} e^{-\frac{1}{2}\sum_{j=1}^n \sigma_j^2 \int_{\delta^2}^tds\, \dot C_s(0)}\sum_{1\leq i<j\leq n} \sigma_i\sigma_j\dot C_t(x_i-x_j)\nnb
&\quad \times\Bigg[\sum_{\substack{G\in \mathcal C_n:\\ \{i,j\}\notin G}}\prod_{\{p,q\}\in G}\left(e^{-\sigma_p\sigma_q\int_{\delta^2}^t ds\, \dot C_s(x_p-x_q)}-1\right) \nnb&\qquad\qquad\qquad-\sum_{\substack{G\in \mathcal C_n:\\ \{i,j\}\in G}}\prod_{\substack{\{p,q\}\in G:\\ \{p,q\}\neq \{i,j\}}}\left(e^{-\sigma_p\sigma_q\int_{\delta^2}^t ds\, \dot C_s(x_p-x_q)}-1\right)\Bigg],
\end{align}

To compute the term in the square brackets, we decompose the set of graphs $\{G\in \mathcal C_n: \{i,j\}\in G\}$ into two sets: 
\begin{align}
\mathcal C_n^{(1)}&=\{G\in \mathcal C_n: \{i,j\}\in G, G\setminus \{i,j\} \text{ is connected}\}\\
\mathcal C_n^{(2)}&=\{G\in \mathcal C_n: \{i,j\}\in G, G\setminus \{i,j\} \text{ is disconnected}\},
\end{align}
where $G\setminus\{i,j\}$ denotes the graph obtained by removing the edge $\{i,j\}$ (but not the vertices).
Then the sum over $\mathcal C_n^{(1)}$ cancels with the first sum in the square brackets, so we have 
\begin{align}
&\Bigg[\sum_{\substack{G\in \mathcal C_n:\\ \{i,j\}\notin G}}\prod_{\{p,q\}\in G}\left(e^{-\sigma_p\sigma_q\int_{\delta^2}^t ds\, \dot C_s(x_p-x_q)}-1\right) -\sum_{\substack{G\in \mathcal C_n:\\ \{i,j\}\in G}}\prod_{\substack{\{p,q\}\in G:\\ \{p,q\}\neq \{i,j\}}}\left(e^{-\sigma_p\sigma_q\int_{\delta^2}^t ds\, \dot C_s(x_p-x_q)}-1\right)\Bigg]\nnb
&=-\sum_{G\in \mathcal C_n^{(2)}}\prod_{\substack{\{p,q\}\in G:\\ \{p,q\}\neq \{i,j\}}}\left(e^{-\sigma_p\sigma_q\int_{\delta^2}^t ds\, \dot C_s(x_p-x_q)}-1\right).
\end{align}
Removing the edge $\{i,j\}$ from $G\in \mathcal C_n^{(2)}$ produces two connected components $G_1$ and $G_2$ which do not share vertices but for which the union of the vertex sets is $[n]$, and say $i\in G_1$ and $j\in G_2$. Thus summing over $\mathcal C_n^{(2)}$ is equivalent to summing over $I_1,I_2$ satisfying $i\in I_1$, $j\in I_2$, $I_1\dot \cup I_2=[n]$, and $G_1$ a connected graph with vertex set $I_1$ and $G_2$ a connected graph with vertex set $I_2$. Swapping the order of the sum over $i<j$ and the sum over $I_1,I_2$, we find (using again Lemma \ref{le:gexpcumu}) that 
\begin{align}
&\delta^{-\frac{1}{4\pi}\sum_{j=1}^n \sigma_j^2} e^{-\frac{1}{2}\sum_{j=1}^n \sigma_j^2 \int_{\delta^2}^tds\, \dot C_s(0)}\sum_{1\leq i<j\leq n} \sigma_i \sigma_j\dot C_t(x_i-x_j)\nnb
&\quad \times \Bigg[\sum_{\substack{G\in \mathcal C_n:\\ \{i,j\}\notin G}}\prod_{\{p,q\}\in G}\left(e^{-\sigma_p\sigma_q\int_{\delta^2}^t ds\, \dot C_s(x_p-x_q)}-1\right) \nnb&\qquad\qquad\qquad-\sum_{\substack{G\in \mathcal C_n:\\ \{i,j\}\in G}}\prod_{\substack{\{p,q\}\in G:\\ \{p,q\}\neq \{i,j\}}}\left(e^{-\sigma_p\sigma_q\int_{\delta^2}^t ds\, \dot C_s(x_p-x_q)}-1\right)\Bigg]\nnb
&=-\frac{1}{2}\sum_{I_1\dot\cup I_2=[n]}\sum_{i\in I_1,j\in I_2}\sigma_i\sigma_j \dot C_t(x_i-x_j)U_t^{|I_1|}(\xi_{I_1}|m,\delta)U_t^{|I_2|}(\xi_{I_2}|m,\delta).
\end{align}
This concludes the proof item (iv), and thus the proof of the lemma.
\end{proof}

Let us now turn to the case $t\geq \epsilon^2$.
\begin{lemma}\label{le:utdif2}
For $t\geq \epsilon^2$, the sequence of functions $(U_t^{k,n})_{n\geq 1, 0\leq k\leq n}$ is the unique solution to the following problem:
\begin{enumerate}
\item For $n\geq 1$ and $0\leq k\leq n$, $\xi_1,\dots,\xi_k\in \Sigma_\alpha$ and $\xi_{k+1},\dots,\xi_n\in \Sigma_\beta$, the function $t\mapsto U_t^{k,n}(\xi_1,\dots,\xi_n|m,\delta,\epsilon)$ is continuous on $[\epsilon^2,\infty)$ and continuously differentiable on $(\epsilon^2,\infty)$.
\item For $n\geq 1$ and $0\leq k\leq n$, $\xi_1,\dots,\xi_k\in \Sigma_\alpha$ and $\xi_{k+1},\dots,\xi_n\in \Sigma_\beta$, we have 
\begin{equation}
U_{\epsilon^2}^{k,n}(\xi_1,\dots,\xi_n|m,\delta,\epsilon)=\delta_{n,1}\delta_{k,0}\epsilon^{-\frac{\sigma^2}{4\pi}}+\delta_{n,k}U_{\epsilon^2}^n(\xi_1,\dots,\xi_n|m,\delta,\epsilon),
\end{equation}
where,  on the right-hand side, $U_{\epsilon^2}^n$ is the kernel from the $[\delta^2,\epsilon^2]$ interval.
\item For $n=1$, $\xi\in \Sigma$, and $t\geq \epsilon^2$ we have 
\begin{align}
U_t^{k,1}(\xi|m,\delta,\epsilon)=\delta_{k,1}e^{-\frac{\sigma^2}{2}(\int_{\delta^2}^t ds\,\dot C_s(0)+\frac{1}{4\pi}\log \delta^2)}+\delta_{k,0}e^{-\frac{\sigma^2}{2}(\int_{\epsilon^2}^t ds\,\dot C_s(0)+\frac{1}{4\pi}\log \epsilon^2)}.
\end{align}
\item For $n\geq 2$, $0\leq k\leq n$, $\xi_1,\dots.,\xi_k\in \Sigma_\alpha$, $\xi_{k+1},\dots,\xi_n\in \Sigma_\beta$ and $t>\epsilon^2$ we have 
\begin{align}
\partial_t U_t^{k,n}(\xi_1,\dots,\xi_n|m,\delta,\epsilon)&=-\frac{1}{2}\sum_{I_1\dot\cup I_2=[n]}\sum_{i\in I_1,j\in I_2}\sigma_i \sigma_j \dot C_t(x_i-x_j)U_t^{|I_1\cap[k]|,|I_1|}(\xi_{I_1}|m,\delta,\epsilon)\nnb
&\qquad\qquad  \times U_t^{|I_2\cap[k]|,|I_2|}(\xi_{I_2}|m,\delta,\epsilon)\nnb
&\quad -\frac{1}{2}\sum_{i,j=1}^n \sigma_i\sigma_j \dot C_t(x_i-x_j)U_t^{k,n}(\xi_1,\dots,\xi_n|m,\delta,\epsilon),
\end{align}
where $[n]=\{1,\dots,n\}$, and the sum is over non-empty sets $I_1,I_2\subset [n]$ which are disjoint and $I_1\cup I_2=[n]$. As before, for $I=\{i_1,\dots,i_j\}$, we have written $\xi_I=\{\xi_{i_1},\dots,\xi_{i_j}\}$.
\end{enumerate}
\end{lemma}
\begin{proof}
The proof is very similar to that of Lemma \ref{le:utdif1}. We again begin with uniqueness.

\subproof{uniqueness}
This is the essentially the same induction argument as in the proof of Lemma \ref{le:utdif1}. Uniqueness for $n=1$ is manifest for both $k=0$ and $k=1$. As an induction hypothesis, one assumes that uniqueness holds for each pair $(n,k)$ with $0\leq k\leq n$, and $1\leq n\leq N$. One then finds that if we had two solutions for some $0\leq K\leq N+1$, then by the induction hypothesis and item (iv), their difference $\Delta_{K,N}(t)$ would satisfy the equation $\partial_t \Delta_{K,N}(t)=-\frac{1}{2}\sum_{i,j=1}^{N+1} \sigma_i \sigma_j \dot C_t(x_i-x_j)\Delta_{K,N}(t)$ with zero initial data (by item (ii)). So again, uniqueness follows from uniqueness of the solution to this first order linear differential equation.

\subproof{existence --  item (i)}
This again follows from the corresponding properties for the covariances involving $\Phi_{\delta^2,t}$ and $\Phi_{\epsilon^2,\epsilon^2\vee t}$, and eventually boils down to smoothness of $\dot C_s(x-y)$.

\subproof{existence -- item (ii)}
Since $\Phi_{\epsilon^2,\epsilon^2}(x)=0$, and cumulants of order two or more vanish if one of the random variables is constant, we see that the cumulant vanishes unless $k=n$ or $n=1$ and $k=0$, and in these cases, the cumulant is given by the claim in item (ii).

\subproof{existence -- item (iii)}
This again follows immediately from from computing the variance of $\Phi_{\delta^2,t}(x)$ and $\Phi_{\epsilon^2,t}(x)$.

\subproof{existence -- item (iv)}
Again, this is the most complicated part of the proof, but it is very similar to the corresponding part of Lemma \ref{le:utdif1}, so we omit most of the details. The starting point is to use Lemma \ref{le:gexpcumu} to write  
\begin{align}
&U_t^{k,n}(\xi_1,\dots,\xi_n|m,\delta,\epsilon)\nnb
&=\delta^{-\frac{1}{4\pi}\sum_{j=1}^k\sigma_j^2}\epsilon^{-\frac{1}{4\pi}\sum_{j=k+1}^n \sigma_j^2}e^{-\frac{1}{2}\sum_{j=1}^n \sigma_j^2\int_{\delta^2}^t ds\, \dot C_s(0)} e^{-\frac{1}{2}\sum_{j=k+1}^n \sigma_j^2\int_{\epsilon^2}^t ds\, \dot C_s(0)}\nnb
&\quad \times \sum_{G\in \mathcal C_n}\prod_{\{i,j\}\in G}\left(e^{-\sigma_i \sigma_j \1\{i>k \text{ or } j>k\}\int_{\epsilon^2}^tds\, \dot C_s(x_i-x_j)-\sigma_i\sigma_j \1\{i,j\leq k\}\int_{\delta^2}^t ds\, \dot C_s(x_i-x_j)}-1 \right).
\end{align}
The next step is again to differentiate with respect to $t$. Note that we have for example
\begin{align}
&\partial_t \delta^{-\frac{1}{4\pi}\sum_{j=1}^k\sigma_j^2}\epsilon^{-\frac{1}{4\pi}\sum_{j=k+1}^n \sigma_j^2}e^{-\frac{1}{2}\sum_{j=1}^n \sigma_j^2\int_{\delta^2}^t ds\, \dot C_s(0)} e^{-\frac{1}{2}\sum_{j=k+1}^n \sigma_j^2\int_{\epsilon^2}^t ds\, \dot C_s(0)}\nnb
  &=-\frac{1}{2}\sum_{j=1}^n \sigma_j^2 \dot C_t(0) \delta^{-\frac{1}{4\pi}\sum_{j=1}^k\sigma_j^2}\epsilon^{-\frac{1}{4\pi}\sum_{j=k+1}^n \sigma_j^2}
    \nnb
    &\qquad \qquad \times e^{-\frac{1}{2}\sum_{j=1}^n \sigma_j^2\int_{\delta^2}^t ds\, \dot C_s(0)} e^{-\frac{1}{2}\sum_{j=k+1}^n \sigma_j^2\int_{\epsilon^2}^t ds\, \dot C_s(0)}
\end{align}
and
\begin{align}
&\partial_t \left(e^{-\sigma_i \sigma_j \1\{i>k \text{ or } j>k\}\int_{\epsilon^2}^tds\, \dot C_s(x_i-x_j)-\sigma_i\sigma_j \1\{i,j\leq k\}\int_{\delta^2}^t ds\, \dot C_s(x_i-x_j)}-1 \right)\nnb
&=-\sigma_i\sigma_j \dot C_t(x_i-x_j)e^{-\sigma_i \sigma_j \1\{i>k \text{ or } j>k\}\int_{\epsilon^2}^tds\, \dot C_s(x_i-x_j)-\sigma_i\sigma_j \1\{i,j\leq k\}\int_{\delta^2}^t ds\, \dot C_s(x_i-x_j)},
\end{align}
so the arguments from the proof of Lemma~\ref{le:utdif1}, we find 
\begin{align}
&\partial_t U_t^{k,n}(\xi_1,\dots,\xi_n|m,\delta,\epsilon)\nnb
&=-\frac{1}{2}\sum_{i,j=1}^n \sigma_i\sigma_j \dot C_t(x_i-x_j)U_t^{k,n}(\xi_1,\dots,\xi_n|m,\delta,\epsilon)\nnb
&\quad -\delta^{-\frac{1}{4\pi}\sum_{j=1}^k \sigma_j^2} \epsilon^{-\frac{1}{4\pi}\sum_{j=k+1}^n \sigma_j^2} e^{-\frac{1}{2}\sum_{j=1}^k \int_{\delta^2}^t ds\, \dot C_s(0)-\frac{1}{2}\sum_{j=k+1}^n \sigma_j^2\int_{\epsilon^2}^t ds\, \dot C_s(0)} \nnb &\qquad \times \sum_{1\leq i<j\leq n}\sigma_i \sigma_j \dot C_t(x_i-x_j)\nnb
&\qquad \times \sum_{G\in \mathcal C_n^{(2)}}\prod_{\substack{\{p,q\}\in G:\\ \{p,q\}\neq \{i,j\}}}\left(e^{-\sigma_p \sigma_q \1\{p>k \text{ or } q>k\}\int_{\epsilon^2}^tds\, \dot C_s(x_p-x_q)-\sigma_p\sigma_q \1\{p,q\leq k\}\int_{\delta^2}^t ds\, \dot C_s(x_p-x_q)}-1 \right),
\end{align}
from which the claim in item (iv) follows in the same way as in Lemma \ref{le:utdif1}. This concludes the proof of the lemma.
\end{proof}

The connection of $\tilde V_t^n$ and $\tilde V_t^{k,n}$ with the cumulants $U_t^n$ and $U_t^{k,n}$ is given by the next lemma,
which proves Proposition~\ref{prop:cumulants}.

\begin{lemma}\label{le:utvt}
We have for $t\in[\delta^2,\epsilon^2]$
\begin{equation}
\tilde V_t^n(\xi_1,\dots,\xi_n|m,\delta,\epsilon)=(-1)^{n-1}U_t^n(\xi_1,\dots,\xi_n|m,\delta,\epsilon),
\end{equation}
and for $t\geq \epsilon^2$ and $0\leq k\leq n$,
\begin{equation}
\tilde V_t^{k,n}(\xi_1,\dots,\xi_n|m,\delta,\epsilon)=(-1)^{n-1}U_t^{k,n}(\xi_1,\dots,\xi_n|m,\delta,\epsilon).
\end{equation}
\end{lemma}

\begin{proof}
  We use the uniqueness of the solutions to the problems described in Lemma \ref{le:utdif1} and Lemma \ref{le:utdif2}.
  We first consider the case  $\delta^2\leq t\leq \epsilon^2$ and check that $(-1)^{n-1}\tilde V_t^{n}(\xi_1,\dots,\xi_n|m,\delta,\epsilon)$ satisfies the conditions of item (i) -- item (iv).

\subproofx{Item (i) of Lemma \ref{le:utdif1}}
By the recursive definition (and an induction argument if one wants to be very precise), we see that $(-1)^{n-1}\tilde V_t^n(\xi_1,\dots,\xi_n|m,\delta,\epsilon)$ is continuous on $[\delta^2,\epsilon^2]$ and continuously differentiable on $(\delta^2,\epsilon^2)$, so this condition is satisfied.

\subproofx{Item (ii) of Lemma \ref{le:utdif1}}
From the recursion, we see that $\tilde V_{\delta^2}^n(\xi_1,\dots,\xi_n|m,\delta,\epsilon)=0$ for $n\geq 2$, while for $n=1$, we have 
\begin{equation}
\tilde V_{\delta^2}^1(\xi|m,\delta,\epsilon)=\delta^{-\frac{\sigma^2}{4\pi}}
\end{equation}
so also item (ii) is satisfied.

\subproofx{Item (iii) of Lemma \ref{le:utdif1}}
This is satisfied by definition of $\tilde V_t^1(\xi|m,\delta,\epsilon)$.

\subproofx{Item (iv) of Lemma \ref{le:utdif1}}
Differentiating the recursion defining $\tilde V_t^n$ we have
\begin{align}
\partial_t &(-1)^{n-1}\tilde V_t^n(\xi_1,\dots,\xi_n|m,\delta,\epsilon)\nnb
&=-\frac{1}{2}\sum_{I_1\dot \cup I_2=[n]}\sum_{i\in I_1,j\in I_2}\sigma_i \sigma_j \dot C_t(x_i-x_j))(-1)^{|I_1|-1}\tilde V_t^{|I_1|}(\xi_{I_1}|m,\delta,\epsilon)(-1)^{|I_2|-1}\tilde V_t^{|I_2|}(\xi_{I_2}|m,\delta,\epsilon)\nnb
&\quad -\frac{1}{2}\sum_{i,j=1}^n \sigma_i \sigma_j \dot C_t(x_i-x_j)(-1)^{n-1}\tilde V_t^n(\xi_1,\dots,\xi_n|m,\delta,\epsilon)
\end{align}
so also item (iv) is satisfied.

Thus by uniqueness of Lemma \ref{le:utdif1}, we conclude the $(-1)^{n-1}\tilde V_t^n=U_t^n$ as claimed. 

\smallskip

We consider now the case $t\geq \epsilon^2$. Our task is now to check that $(-1)^{n-1}\tilde V_t^{k,n}$ satisfies the conditions of Lemma \ref{le:utdif2}.

\subproofx{Item (i) of Lemma \ref{le:utdif2}}
This follows from the recursive definition of $\tilde V_t^{k,n}$ (and an induction argument).

\subproofx{Item (ii) of Lemma \ref{le:utdif2}}
This follows from the definition of $\tilde V_t^{k,n}$ and the result for $t\in[\delta^2,\epsilon^2]$.

\subproofx{Item (iii) of Lemma \ref{le:utdif2}}
This follows immediately from the definition of $\tilde V_t^{k,1}(\xi|m,\delta,\epsilon)$.

\subproofx{Item (iv) of Lemma \ref{le:utdif2}}
This is essentially the same calculation as in the case $t\in[\delta^2,\epsilon^2]$.
\end{proof}

\subsection{Bounds -- proof of Proposition~\ref{pr:vtnorm}}
\label{app:vtbounds}

The proof follows \cite[Section 4]{MR4767492}  quite carefully.
Since it is still quite lengthy, we split the proof into several parts.

\subsubsection{Proof of Proposition \ref{pr:vtnorm} for $k=n$}

One of the reasons that the proof is simpler for $k=n$ is that, by combining \eqref{eq:vtiln} and \eqref{eq:vtilkn}, we have for $t\geq \epsilon^2$ and $n\geq 2$
\begin{align}\label{eq:nnrec}
\tilde V_t^{n,n}&(\xi_1,\dots,\xi_n|m,\delta,\epsilon)\nnb
&=\frac{1}{2}\sum_{I_1\dot \cup I_2=[n]}\sum_{i\in I_1,j\in I_2}\sigma_i \sigma_j \int_{\delta^2}^t ds\, \dot C_s(x_i-x_j)\tilde V_s^{|I_1|,|I_1|}(\xi_{I_1}|m,\delta,\epsilon)\tilde V_s^{|I_2|,|I_2|}(\xi_{I_2}|m,\delta,\epsilon)\nnb
&\qquad \qquad \qquad \qquad \qquad \qquad \times e^{-\frac{1}{2}\sum_{i,j=1}^n \sigma_i \sigma_j \int_s^t dr\, \dot C_r(x_i-x_j)},
\end{align}
where we understand that $\tilde V_s^{|I|,|I|}(\xi_I|m,\delta,\epsilon)=\tilde V_s^{|I|}(\xi_I|m,\delta,\epsilon)$ for $s\leq \epsilon^2$.
In fact, with this interpretation, the claim is also true for $\delta^2\leq t\leq \epsilon^2$ (by \eqref{eq:vtiln}). This is sufficient preparation for the proof of Proposition \ref{pr:vtnorm} for $k=n$. 
\begin{proof}[Proof of Proposition \ref{pr:vtnorm} for $k=n$]
We will prove a slightly stronger version of the statement. We will show that the claim holds for $\delta^2\leq t\leq 1$ and we will show that there exist constants $C_\Sigma$ and $C$ (the former depending only on $\Sigma$ and the latter being universal) so that 
\begin{equation} \label{e:hnn-stronger}
\|h_t^{n,n}\|_n \leq n^{n-2} t^{-1}C_\Sigma^{n-1}\left(C t^{1-\frac{\max_{\sigma\in \Sigma_\alpha}(\sigma^2)}{8\pi}}\right)^n.
\end{equation}
Note that for $t<1$, this upper bound is smaller than the one in the statement of Proposition \ref{pr:vtnorm} since $\beta\geq 4\pi>\max_{\sigma\in \Sigma_\alpha}(\sigma^2)$.

First of all, we note from %
\eqref{eq:vtilk1} that for $t\geq \delta^2$ (recalling our convention that $\tilde V_t^{1,1}=\tilde V_t^1$ for $t\leq \epsilon^2$ and that $m^2 t<1$)
\begin{align}
0\leq \tilde V_t^{1,1}(\xi|m,\delta,\epsilon)&=e^{-\frac{\sigma^2}{2}(\int_{\delta^2}^t ds\, \dot C_s(0)+\frac{1}{4\pi}\log \delta^2)}\nnb
&=e^{-\frac{\sigma^2}{2}(\int_{\delta^2}^tds\,  \frac{e^{-m^2 s}-1}{4\pi s}+\frac{1}{4\pi}\log t) }\nnb
&\leq t^{-\frac{\sigma^2}{8\pi}} e^{\frac{\sigma^2}{8\pi}m^2 t}
\leq e t^{-\frac{\sigma^2}{8\pi}}=:h_t^{1,1}(\xi). 
\end{align}
Since $t^{-\frac{\sigma^2}{8\pi}}\leq t^{-\frac{\max_{\sigma\in \Sigma_\alpha}(\sigma^2)}{8\pi}}$ for $t < 1$,
we have verified \eqref{e:hnn-stronger} for $n=1$.

We now make the induction hypothesis that \eqref{e:hnn-stronger} %
holds for $n<N$ for some $N\geq 2$. We prove that it holds for $N$.
First note that $\sum_{i,j=1}^{N} \sigma_i \sigma_j \int_s^t dr\, \dot C_r(x_i-x_j)\geq 0$ since $\dot C_r(x-y)$ is a covariance (i.e., of positive type). Feeding our induction hypothesis into \eqref{eq:nnrec}, we thus find
\begin{align}
|\tilde V_t^{N,N}(\xi_1,\dots,\xi_N|m,\delta,\epsilon)|&\leq 2\pi \sum_{I_1\dot \cup I_2=[N]}\sum_{i\in I_1,j\in I_2}\int_0^t ds \frac{e^{-\frac{|x_i-x_j|^2}{4s}}}{4\pi s} h_s^{|I_1|,|I_1|}(\xi_{I_1}) h_s^{|I_2|,|I_2|}(\xi_{I_2})\nnb
&=:h_t^{N,N}(\xi_1,\dots,\xi_N).
\end{align}
The sum over $I_1\dot \cup I_2$ symmetrizes this function so $h_t^{N,N}$ is symmetric in its arguments by construction. Using this symmetry in arguments, our induction hypothesis, and that, for $s \leq 1$,
\begin{equation}
  \int_{\R^2}du\, \frac{e^{-\frac{|u|^2}{4s}}}{4\pi s}
  e^{\tdist(0,u)} \lesssim 1,
\end{equation}
and \eqref{e:rho-add},
we find that
\begin{equation}
\|h_t^{N,N}\|_N\leq 2\pi \sum_{I_1\dot \cup I_2=[N]}|I_1|^{|I_1|-1}|I_2|^{|I_2|-1} C_\Sigma^{N-2} C^N \int_0^t ds s^{-2+N(1-\frac{\max_{\sigma\in \Sigma_\alpha}(\sigma^2)}{8\pi})}.
\end{equation}
Recalling that $\max_{\sigma\in \Sigma_\alpha}(\sigma^2)<4\pi$, we see that this integral converges for $N\geq 2$. Moreover,
\begin{align}
\sum_{I_1\dot \cup I_2=[N]}|I_1|^{|I_1|-1}|I_2|^{|I_2|-1}=\sum_{k=1}^{N-1}k^{k-1}(N-k)^{N-k-1}=2(N-1)N^{N-2},
\end{align}
see, for example, \cite[Proof of Proposition 4.1]{MR4767492}.
Thus we have 
\begin{align}
\|h_t^{N,N}\|_N \leq \frac{4\pi (N-1)}{N(1-\frac{\max_{\sigma\in \Sigma_{\alpha}}(\sigma^2)}{8\pi})-1} N^{N-2} t^{-1}C_\Sigma^{N-2}(Ct^{1-\frac{\max_{\sigma\in \Sigma_\alpha}(\sigma^2)}{8\pi}})^N.
\end{align}
For $N\geq 2$, this prefactor is bounded by some $\Sigma$-dependent quantity, so by possibly adjusting $C_\Sigma$ so that it is larger than this prefactor for all $N$, we recover the stronger claim, and as mentioned, this implies Proposition \ref{pr:vtnorm} for the case $n=k$.
\end{proof}

For $k<n$ we may encounter integrals of the type $\int_0^t ds\, s^{-2+n(1-\frac{\beta}{8\pi})}$. As $\beta\in [4\pi,6\pi)$, this integral may not converge for small $n$. Thus for the $k<n$ case, we need to consider the small $n$ terms separately. After this, we can run essentially the same induction argument.

\subsubsection{Proof of Proposition \ref{pr:vtnorm} for $0\leq k<n\leq 4$}

The case of $n=1$ (and $k=0$) is simple since from \eqref{eq:vtilk1} we have (with the same calculation as in the $n=k=1$ case) for $\epsilon^2\leq t<\min(1,m^{-2})$
\begin{align}
0\leq \tilde V_t^{0,1}(\xi|m,\delta,\epsilon)\leq e t^{-\frac{\beta}{8\pi}}=: h_t^{0,1}(\xi).
\end{align}
For $n=2$, our starting point is that, from \eqref{eq:vtilk1} and \eqref{eq:vtilkn},
\begin{align}\label{eq:V2explicit1}
\tilde V_t^{0,2}(\xi_1,\xi_2|m,\delta,\epsilon)&=\sigma_1\sigma_2\int_{\epsilon^2}^t ds\, \dot C_s(x_1-x_2) e^{-\frac{\sigma_1^2+\sigma_2^2}{2}(\int_{\epsilon^2}^s dr\, \dot C_r(0)+\frac{1}{4\pi}\log \epsilon^2)}e^{-\frac{\sigma_1^2+\sigma_2^2}{2}\int_s^t dr\, \dot C_r(0)} \nnb
&\qquad \qquad \times e^{-\sigma_1\sigma_2 \int_s^tdr\, \dot C_r(x_1-x_2)}\nnb
&=\tilde V_t^{0,1}(\xi_1|m,\delta,\epsilon)\tilde V_t^{0,1}(\xi_2|m,\delta,\epsilon)\left(1-e^{-\sigma_1\sigma_2\int_{\epsilon^2}^t dr\, \dot C_r(x_1-x_2)}\right)
\end{align}
and similarly 
\begin{equation}\label{eq:V2explicit2}
  \tilde V_t^{1,2}(\xi_1,\xi_2|m,\delta,\epsilon)=\tilde V_t^{1,1}(\xi_1|m,\delta,\epsilon)\tilde V_t^{0,1}(\xi_2|m,\delta,\epsilon)\left(1-e^{-\sigma_1\sigma_2\int_{\epsilon^2}^t dr\, \dot C_r(x_1-x_2)}\right).
\end{equation}
Using these explicit formulas, we can prove Proposition~\ref{pr:vtnorm} for $n=2$.

\begin{proof}[Proof of Proposition~\ref{pr:vtnorm} for $n=2$]
We first of all note that 
\begin{equation}
\left|1-e^{-\sigma_1\sigma_2\int_{\epsilon^2}^t dr\, \dot C_r(x_1-x_2)}\right|\leq \left|1-e^{-\sigma_1\sigma_2 \int_0^t ds\, \dot C_s(x_1-x_2)}\right|.
\end{equation}
In case $\sigma_1\sigma_2\geq 0$, we simply define (using that fact that for $x\geq 0$, $|1-e^{-x}|\leq x$ and noting that for $t<1$, $h_t^{1,1}(\xi)\leq h_t^{0,1}(\xi)$)
\begin{equation}
  h_t^{k,2}(\xi_1,\xi_2)=h_t^{0,1}(\xi_1)h_t^{0,1}(\xi_2)\sigma_1\sigma_2 \int_0^t ds\, \frac{1}{4\pi s} e^{-\frac{|x_1-x_2|^2}{4s}},
\end{equation}
and note that for $0<t<1$ (by Fubini and shifting the $x_2$ integration variable by $x_1$)
\begin{equation}
  \|h_t^{k,2}\|_{2}
  \leq C t^{-\frac{\beta}{8\pi}}  t^{-\frac{\beta}{8\pi}} \sigma_1\sigma_2 \int_0^t ds\, \int_{\R^2} dx_2 \frac{1}{4\pi s} e^{-\frac{|x_2|^2}{4s}} e^{\tdist(0,x_2)}
  \leq C 6\pi  t^{1-\frac{\beta}{4\pi}}, 
\end{equation}
which is the desired bound. 

We thus look at $\sigma_1\sigma_2<0$. Here we note that 
\begin{align}
\left|1-e^{-\sigma_1\sigma_2 \int_0^t ds\, \dot C_s(x_1-x_2)}\right|=|\sigma_1\sigma_2|\int_0^t ds\, \dot C_s(x_1-x_2) e^{|\sigma_1\sigma_2|\int_0^s dr\, \dot C_r(x_1-x_2)}.
\end{align}
By \cite[Lemma 4.4]{MR4767492}, we have for $s<m^{-2}$ (which is satisfied in our situation), 
\begin{equation}
\int_0^s dr\, \dot C_r(x_1-x_2)=-\frac{1}{2\pi}\log \left(\frac{|x_1-x_2|}{\sqrt{s}}\wedge 1\right)+O(1),
\end{equation}
where the implied constant is universal. We thus define for a suitable universal $C>0$ 
\begin{align}
h_t^{k,2}(\xi_1,\xi_2)=C h_t^{0,1}(\xi_1)h_t^{0,1}(\xi_2)\int_0^t ds\, \frac{1}{4\pi s}e^{-\frac{|x_1-x_2|^2}{4s}}\left(\frac{|x_1-x_2|}{\sqrt{s}}\wedge 1\right)^{-\frac{|\sigma_1\sigma_2|}{2\pi}}
\end{align}
Recall from the statement of Proposition \ref{pr:vtnorm} that we are considering only the situation where $\sigma_1\sigma_2\neq -\beta$. This means that $\sigma_1=\sqrt{4\pi}\alpha_1\in \Sigma_\alpha$ so by the assumption $\alpha_1^2<2-\frac{\beta}{4\pi}$, we see that 
\begin{align}
|\sigma_1\sigma_2| < 4\pi  \sqrt{\left(2-\frac{\beta}{4\pi}\right)\frac{\beta}{4\pi}}\leq 4\pi.
\end{align}
Thus (with a change of variables)
\begin{align}
\|h_t^{k,2}\|_2 &\leq C t^{-\frac{\beta}{4\pi}}\int_0^t ds \int_{\R^2} du\, \frac{1}{4\pi s} e^{-\frac{|u|^2}{4s}}\left(\frac{|u|}{\sqrt{s}}\wedge 1\right)^{-\frac{|\sigma_1\sigma_2|}{2\pi}} e^{\tdist(0,u)}\nnb
&\leq C t^{1-\frac{\beta}{4\pi}}\int_{\R^2} du\, \frac{1}{4\pi} e^{-\frac{|u|^2}{4}}(|u|\wedge 1)^{-\frac{|\sigma_1\sigma_2|}{2\pi}} e^{|u|}.
\end{align}
Since $|\sigma_1\sigma_2|<4\pi$, this last integral is just some $\Sigma$-dependent constant. Thus we have the bound claimed in Proposition \ref{pr:vtnorm} also in this case. This concludes the proof for $0\leq k<n=2$.
\end{proof}

We move on to $n=3$. For $\tilde V_t^{k,3}$, we find from \eqref{eq:vtilkn} and the explicit formulas for $n=2$ that
\begin{align}\label{eq:vt03}
\tilde V_t^{0,3}(\xi_1,\xi_2,\xi_3|m,\delta,\epsilon)&=\prod_{j=1}^3\tilde V_t^{0,1}(\xi_j|m,\delta,\epsilon) \int_{\epsilon^2}^tds\, (\sigma_1 \sigma_2\dot C_2(x_1-x_2)+\sigma_1\sigma_3 \dot C_s(x_1-x_3))\nnb
&\quad \times \left(1-e^{-\sigma_2\sigma_3\int_{\epsilon^2}^s dr\, \dot C_r(x_2-x_3)}\right)e^{-\sum_{1\leq i<j\leq 3}\sigma_i\sigma_j \int_{s}^t dr\, \dot C_r(x_i-x_j)}\nnb
&\quad +\text{permutations},
\end{align}
where ``permutations'' refers to permutations of the indices under the $s$ integral. Similarly
\begin{align}\label{eq:vt13}
\tilde V_t^{1,3}(\xi_1,\xi_2,\xi_3|m,\delta,\epsilon)&=\tilde V_t^{1,1}(\xi_1|m,\delta,\epsilon) \prod_{j=2}^3 \tilde V_t^{0,1}(\xi_j|m,\delta,\epsilon) \nnb
&\quad \times \int_{\epsilon^2}^tds\, (\sigma_1 \sigma_2\dot C_2(x_1-x_2)+\sigma_1\sigma_3 \dot C_s(x_1-x_3))\nnb
&\quad \times \left(1-e^{-\sigma_2\sigma_3\int_{\epsilon^2}^s dr\, \dot C_r(x_2-x_3)}\right)e^{-\sum_{1\leq i<j\leq 3}\sigma_i\sigma_j \int_{s}^t dr\, \dot C_r(x_i-x_j)}\nnb
&\quad +\text{permutations}
\end{align}
and
\begin{align}\label{eq:vt23}
\tilde V_t^{2,3}(\xi_1,\xi_2,\xi_3|m,\delta,\epsilon)&=\prod_{j=1}^2\tilde V_t^{1,1}(\xi_j|m,\delta,\epsilon)  \tilde V_t^{0,1}(\xi_3|m,\delta,\epsilon) \nnb
&\quad \times \int_{\epsilon^2}^tds\, (\sigma_1 \sigma_2\dot C_2(x_1-x_2)+\sigma_1\sigma_3 \dot C_s(x_1-x_3))\nnb
&\quad \times \left(1-e^{-\sigma_2\sigma_3\int_{\epsilon^2}^s dr\, \dot C_r(x_2-x_3)}\right)e^{-\sum_{1\leq i<j\leq 3}\sigma_i\sigma_j \int_{s}^t dr\, \dot C_r(x_i-x_j)}\nnb
&\quad +\text{permutations},
\end{align}
where now in ``permutations'' one needs to be slightly more careful as for the permutation for which $\sigma_2\sigma_3$ gets replaced by $\sigma_1\sigma_2$, we need to replace $1-e^{-\sigma_2\sigma_3\int_{\epsilon^2}^s dr\, \dot C_r(x_2-x_3)}$ by 
\begin{equation}
1-e^{-\sigma_1\sigma_2\int_{\delta^2}^s dr\, \dot C_r(x_1-x_2)}
\end{equation}
(since in this case, we have the term $\tilde V_s^{2,2}(\xi_1,\xi_2|m,\delta,\epsilon)$ appearing in our recursion).

To bound $\tilde V_t^3$, we record the following lemma (which is a generalization of \cite[Lemma 4.5]{MR4767492}).
\begin{lemma}\label{le:n3bound1}
For $s>0$, there exists a function $F_s:(\R^2\times \Sigma)^3\to [0,\infty]$ which is invariant under permutations of its arguments and is independent of $\delta,m$ such that for $\xi_1,\xi_2,\xi_3\in \R^2\times \Sigma$ and $0<\epsilon^2\leq s\leq m^{-2}$, we have 
\begin{align}
\left| (\sigma_1\sigma_2\dot C_s(x_1,x_2)+\sigma_1\sigma_3\dot C_s(x_1,x_3))
\left(1-e^{-\sigma_2\sigma_3\int_{\epsilon^2}^s dr\, \dot C_r(x_2-x_3)}\right)\right|\leq F_s(\xi_1,\xi_2,\xi_3)
\end{align}
and 
\begin{align}
\|F_s\|_3\leq C_\Sigma s
\end{align}
for some finite $C_\Sigma>0$ depending only on $\Sigma$. The same bound holds if we replace in the exponential $\epsilon^2$ by $\delta^2$ with $\delta<\epsilon$.
\end{lemma}
\begin{proof}
There are essentially two cases to consider: i) $\sigma_2\sigma_3=-\beta$ and ii) $\sigma_2\sigma_3\neq -\beta$. In the first case, we have 
\begin{align}
|\sigma_1\sigma_2\dot C_s(x_1-x_2)+\sigma_1\sigma_3\dot C_s(x_1-x_3)|=|\sigma_1|\sqrt{\beta}|\dot C_s(x_1-x_2)-\dot C_s(x_1-x_3)|.
\end{align}
Apart from the multiplicative prefactor, which does not affect the bound we are after, this is exactly the situation that was dealt with in \cite[Proof of Lemma 4.5, case 1]{MR4767492}. We can thus focus on case ii).
In case ii), we note that due to the condition that $\alpha_i^2<2-\frac{\beta}{4\pi}$, we have 
\begin{align}
\sigma_2\sigma_3>-\gamma_\Sigma,
\end{align}
where $\gamma_\Sigma<4\pi$. To be more precise, if $\sigma_2=\sqrt{4\pi}\alpha_i$ and $\sigma_3=\sqrt{4\pi}\alpha_j$, then this follows from the condition that $\alpha_i^2<1$ (since $\beta\geq 4\pi$), while if $\sigma_2=\sqrt{4\pi}\alpha_i$ and $\sigma_3=\pm\sqrt{\beta}$, then as $\alpha_i^2<2-\frac{\beta}{4\pi}$, we have $|\sigma_2\sigma_3|<\sqrt{\beta(8\pi-\beta)}\leq 4\pi$ (and the case $\sigma_2=\sigma_3=\sqrt{\beta}$ is obvious). We thus see that 
\begin{align}
&\left| (\sigma_1\sigma_2\dot C_s(x_1,x_2)+\sigma_1\sigma_3\dot C_s(x_1,x_3))
\left(1-e^{-\sigma_2\sigma_3\int_{\epsilon^2}^s dr\, \dot C_r(x_2-x_3)}\right)\right|\nnb
&\leq 6\pi(\dot C_s(x_1-x_2)+\dot C_2(x_1-x_3))\left(e^{\gamma_\Sigma\int_0^{s}dr\, \dot C_r(x_2-x_3)}-1\right) \nnb
&=6\pi \gamma_\Sigma(\dot C_s(x_1-x_2)+\dot C_s(x_1-x_3))\int_0^s dr\, \dot C_r(x_2-x_3)e^{\gamma_\Sigma C_r(x_2-x_3)}.  
\end{align}
The same bound holds if $\epsilon^2$ is replaced by $\delta^2$.

Note that this is symmetric in $x_2$ and $x_3$, so it is enough to control only one of the terms appearing in the sum. By \cite[Lemma 4.4]{MR4767492}, we have for $r<s\leq m^{-2}$ 
\begin{align}
C_r(x_2-x_3)=-\frac{1}{2\pi}\log\left(\frac{|x_2-x_3|}{\sqrt{r}}\wedge 1\right)+O(1),
\end{align}	
where the implied constant is universal, so we have 
\begin{align}
&\dot C_s(x_1-x_2)\int_0^s dr\, \dot C_r(x_2-x_3)e^{\gamma_\Sigma C_r(x_2-x_3)}\nnb
&\leq C \frac{1}{4\pi s} e^{-\frac{|x_1-x_2|^2}{4s}}\int_0^s dr\, \frac{1}{4\pi r} e^{-\frac{|x_2-x_3|^2}{4r}}\left(\frac{|x_2-x_3|}{\sqrt{r}}\wedge 1\right)^{-\frac{\gamma_\Sigma}{2\pi}}. 
\end{align} 
By changing variables, we see that the $\|\cdot\|_3$-norm of this quantity is bounded by 
\begin{align}
  &C\int_{\R^2}da\, \frac{1}{4\pi s}e^{-\frac{|a|^2}{2s}} \int_0^s dr\, \int_{\R^2} db \frac{1}{4\pi r} e^{-\frac{|b|^2}{4r}}\left(\frac{|b|}{\sqrt{r}}\wedge 1\right)^{-\frac{\gamma_\Sigma}{2\pi}}
    e^{\tdist(0,a,b)}
    \nnb
  &\leq Cs  \int_{\R^2} da\, \frac{1}{4\pi} e^{-\frac{|a|^2}{2}} e^{|a|} \int_{\R^2} db \frac{1}{4\pi} e^{-\frac{|b|^2}{4}}\left(|b|\wedge 1\right)^{-\frac{\gamma_\Sigma}{2\pi}} e^{|b|}.
\end{align}
This last integral is finite due to the condition that $\gamma_\Sigma<4\pi$, and we have the bound we are after. So in this case, we find our function $F_s$ by symmetrizing the function 
\begin{align}
6\pi \gamma_\Sigma C \frac{1}{4\pi s}e^{-\frac{|x_1-x_2|^2}{4s}} \int_0^s dr \frac{1}{4\pi r}e^{-\frac{|x_2-x_3|^2}{4r}} \left(\frac{|x_2-x_3|}{\sqrt{r}}\wedge 1\right)^{-\frac{\gamma_\Sigma}{2\pi}}.
\end{align}
over the permutations of $\{1,2,3\}$.
\end{proof}

This allows us to prove Proposition \ref{pr:vtnorm} in the $n=3$ case (where again the proof mimics that of \cite[Lemma 4.7]{MR4767492}).
\begin{proof}[Proof of Proposition \ref{pr:vtnorm} for $n=3$]
Given Lemma \ref{le:n3bound1} and our explicit formulas for $\tilde V_t^{k,3}$, our task is to upper bound the quantity 
\begin{align}
e^{-\sum_{1\leq i<j\leq 3}\sigma_i\sigma_j\int_s^t dr\, \dot C_r(x_i-x_j)}.
\end{align}
Indeed, we will prove that for $s\leq t$  
\begin{equation}\label{eq:n3energy}
\exp\left(-\sum_{1\leq i<j\leq 3}\sigma_i\sigma_j \int_s^t dr\, \dot C_r(x_i-x_j)\right)\leq C_\Sigma \left(\frac{t}{s}\right)^{\frac{\beta}{4\pi}}
\end{equation}
For some finite $C_\Sigma>0$ depending only on $\Sigma$.

If all of the $\sigma_i$ have the same sign, the exponential is simply bounded by one, so we can focus on the case where the sign of one is different from the other two. By symmetry, we can assume that this is $\sigma_3$. Again by symmetry, we can assume that the worst case scenario is that $|\sigma_1\sigma_2|\leq |\sigma_1\sigma_3|,|\sigma_2\sigma_3|$ and $|x_1-x_2|\leq |x_1-x_3|,|x_2-x_3|$.

By the triangle inequality, either $|x_1-x_3|\geq \frac{|x_1-x_2|}{2}$ or $|x_2-x_3|\geq \frac{|x_1-x_2|}{2}$.
Let us assume that it is the first case we are dealing with (the latter one is dealt with in a similar manner).
Then
\begin{align} \label{eq:n3energy-bis}
&\sigma_1\sigma_2\int_s^t dr\, \dot C_r(x_1-x_2)+\sigma_1\sigma_3\int_s^t dr\, \dot C_r(x_1-x_3)+\sigma_2\sigma_3\int_s^t dr\, \dot C_r(x_2-x_3)\nnb
&\geq |\sigma_1\sigma_2|\int_s^t dr\,  \dot C_r(x_1-x_2)-|\sigma_1\sigma_3|\int_s^t dr \, \dot C_r((x_1-x_2)/2) -|\sigma_2\sigma_3| \int_s^t dr\, \dot C_r(0)\nnb
&\geq |\sigma_1\sigma_2|\int_s^t dr\,  \dot C_r(x_1-x_2)-|\sigma_1\sigma_3|\int_s^t dr \, \dot C_r((x_1-x_2)/2) -\frac{|\sigma_2\sigma_3|}{4\pi}\log \frac{t}{s}\nnb
&\geq |\sigma_1\sigma_2|\int_s^t dr\, (\dot C_r(x_1-x_2)-\dot C_r((x_1-x_2)/2))-\frac{|\sigma_1\sigma_3|-|\sigma_1\sigma_2|+|\sigma_2\sigma_3|}{4\pi}\log \frac{t}{s}.
\end{align}
We begin with the first term on the right-hand side and show that it is uniformly bounded.
Using that the integrand is negative here (by definition of $\dot C_r$) we have in fact
\begin{align}
\int_s^t dr\, (\dot C_r(x_1-x_2)-\dot C_r((x_1-x_2)/2))&\geq \int_0^t dr\, (\dot C_r(x_1-x_2)-\dot C_r((x_1-x_2)/2))\nnb
&=C_t(x_1-x_2)-C_t((x_1-x_2)/2). 
\end{align} 
Now \cite[Lemma 4.4]{MR4767492} states that for $t<m^{-2}$, 
\begin{equation}
C_t(u)=-\frac{1}{2\pi}\log \left(\frac{|u|}{\sqrt{t}}\wedge 1\right)+O(1),
\end{equation}
where the implied constant is again universal.
Thus going over the cases $|x_1-x_2|\leq \sqrt{t}$, $\sqrt{t}<|x_1-x_2|\leq 2\sqrt{t}$, $|x_1-x_2|>2\sqrt{t}$, one readily sees that 
\begin{equation}
C_t(x_1-x_2)-C_t((x_1-x_2)/2)=O(1)
\end{equation}
where the implied constant is universal.
To bound the second term on the right-hand side of \eqref{eq:n3energy-bis} and deduce \eqref{eq:n3energy},
it is sufficient for us to prove that 
\begin{equation}
|\sigma_1\sigma_3|-|\sigma_1\sigma_2|+|\sigma_2\sigma_3|\leq \beta.
\end{equation}
Recall that we have the conditions $\sigma_1\sigma_2>0$, $\sigma_1\sigma_3<0$ and $|\sigma_1\sigma_2|\leq |\sigma_1\sigma_3|,|\sigma_2\sigma_3|$.
The bound follows by considering the different possible cases.
Namely, there are three different cases: (1) $|\sigma_i|=\sqrt{\beta}$ for all $i$,
(2) $|\sigma_i|=\sqrt{\beta}$ for exactly two $i$, and (3) $|\sigma_i|=\sqrt{\beta}$ for exactly one $i$
(note that $|\sigma_i|\neq \sqrt{\beta}$ for all $i$ is not possible since we have $k\neq 3$).

In case (1), we see that $|\sigma_1\sigma_3|-|\sigma_1\sigma_2|+|\sigma_2\sigma_3|=\beta$, and the claim is fine. 
In case (2), the worst case is $|\sigma_1|\neq \sqrt{\beta}$ (or by symmetry $|\sigma_2|\neq \sqrt{\beta}$). In any event, we see that $|\sigma_1\sigma_3|-|\sigma_1\sigma_2|+|\sigma_2\sigma_3|=\beta$ (since two terms cancel and the third one is $\beta$).
In case (3), the worst case is $|\sigma_1|,|\sigma_2|\neq \sqrt{\beta}$. Since the function 
\begin{equation}
f(x,y)=\sqrt{\beta}x-xy+\sqrt{\beta}y
\end{equation}
is increasing in both $x,y\in [0,\sqrt{4\pi}]$, we see that in this case, 
\begin{equation}
|\sigma_1\sigma_3|-|\sigma_1\sigma_2|+|\sigma_2\sigma_3|\leq \sqrt{4\pi \beta}-4\pi +\sqrt{4\pi \beta}=-(\sqrt{\beta}-\sqrt{4\pi})^2 +\beta\leq \beta.
\end{equation}
Thus we have proven \eqref{eq:n3energy}.

To conclude our proof, we combine \eqref{eq:n3energy} with Lemma \ref{le:n3bound1} and \eqref{eq:vt03}--\eqref{eq:vt23}, and define 
\begin{equation}
h_t^{k,3}(\xi_1,\xi_2,\xi_3)=3C_\Sigma h_t^{0,1}(\xi_1) h_t^{0,1}(\xi_2)h_t^{0,3}(\xi_3)\int_0^t ds\, F_s(\xi_1,\xi_2,\xi_3)\left(\frac{t}{s}\right)^{\frac{\beta}{4\pi}}.
\end{equation}
With this definition, we see that $|\tilde V_t^{k,3}|\leq h_t^{k,3}$ and from Lemma \ref{le:n3bound1}, we see with a change of variables that 
\begin{align}
\|h_t^{k,3}\|_3\leq C_\Sigma t^{-3\frac{\beta}{8\pi}}t^2 \int_0^1 ds\, s^{1-\frac{\beta}{4\pi}}
\end{align}
for some possibly different constant $C_\Sigma$ (still depending only on $\Sigma$). Since this integral is convergent, this is precisely the claim of Proposition \ref{pr:vtnorm} for $n=3$.
\end{proof}

Now consider the $n=4$ case. From \eqref{eq:vtilkn}, we see that (up to permutations) there are two types of terms that we encounter when computing $\tilde V_t^{k,4}$ with $0\leq k<4$:
\begin{align}
\int_{\epsilon^2}^t ds\, \sum_{j=2}^4 \sigma_1 \sigma_j \dot C_s(x_1-x_j)\tilde V_s^{k_1,1}(\xi_1|m,\delta,\epsilon)\tilde V_s^{k_2,3}(\xi_2,\xi_3,\xi_4|m,\delta,\epsilon) e^{-\frac{1}{2}\int_s^t ds\, \dot C_r(x_i-x_j)}
\end{align}
and 
\begin{align}
\int_{\epsilon^2}^t ds\, \sum_{i\in \{1,2\},j\in \{3,4\}} \sigma_i \sigma_j \dot C_s(x_i-x_j)\tilde V_s^{k_1,2}(\xi_1,\xi_2|m,\delta,\epsilon)\tilde V_s^{k_2,2}(\xi_3,\xi_4|m,\delta,\epsilon) e^{-\frac{1}{2}\int_s^t ds\, \dot C_r(x_i-x_j)},
\end{align}
where $k_1+k_2=k$.

The first case will be simple to bound based on our bound for $n=3$, but the second case will require some more care. Using the explicit form of $\tilde V_t^{k,2}$, we see that it becomes important to bound terms of the form 
\begin{align}
\sum_{i\in \{1,2\},j\in\{3,4\}}\sigma_i\sigma_j \dot C_s(x_i-x_j)(1-e^{-\sigma_1\sigma_2\int_{\epsilon^2}^s dr\, \dot C_r(x_1-x_2)})(1-e^{-\sigma_3\sigma_4\int_{\epsilon^2}^s dr\, \dot C_r(x_3-x_4)}),
\end{align}
(and terms where $\int_{\epsilon^2}^s dr\, \dot C_r(x_1-x_2)$ is replaced by $\int_{\delta^2}^s dr\, \dot C_r(x_1-x_2)$ in the $k=2$ and $k=3$ cases).
The next lemma (which is a generalization of \cite[Lemma 4.6]{MR4767492}) takes care of this bound.

\begin{lemma}\label{le:n4bound}
For $s>0$ there exists a function $G_s:(\R^2\times \Sigma)^4\to [0,\infty]$ which is independent of $\delta,m$, symmetric in its arguments, such that for $0<\epsilon^2<s<m^{-2}$
\begin{align}
&\left|\sum_{i\in \{1,2\},j\in\{3,4\}}\sigma_i\sigma_j \dot C_s(x_i-x_j)(1-e^{-\sigma_1\sigma_2\int_{\epsilon^2}^s dr\, \dot C_r(x_1-x_2)})(1-e^{-\sigma_3\sigma_4\int_{\epsilon^2}^s dr\, \dot C_r(x_3-x_4)})\right|\nnb
&\qquad \qquad \leq G_s(\xi_1,\xi_2,\xi_3,\xi_4)
\end{align}
and 
\begin{align}
\|G_s\|_4\leq C_\Sigma s^2
\end{align}
for some finite $C_\Sigma$ which depends only on $\Sigma$. The same bound holds if we replace $\int_{\epsilon^2}^\infty ds\, C_r(x_1-x_2)$ by $\int_{\delta^2}^s dr\, \dot C_r(x_1-x_2)$ for any $\delta<\epsilon$.
\end{lemma}
\begin{proof}
We split the proof into two cases: case 1), where $\sigma_1\sigma_2=-\beta$ and $\sigma_3\sigma_4=-\beta$, case 2) $\sigma_1\sigma_2\neq -\beta$ or $\sigma_3\sigma_4\neq -\beta$ (or both).

Case 1) is dealt with\footnote{The fact that $\int_{\epsilon^2}^sdr\, \dot C_r(x_1-x_2)$ can be replaced by $\int_{\delta^2}^s dr\, \dot C_r(x_1-x_2)$ for any $\delta<\epsilon$ is not explicitly mentioned in \cite{MR4767492}, but it follows immediately from the proof there since the proof is based on bounding $|1-e^{-\sigma_1\sigma_2\int_{\epsilon^2}^s dr\, \dot C_r(x_1-x_2)}|$ by $|1-e^{-\sigma_1\sigma_2\int_{0}^s dr\, \dot C_r(x_1-x_2)}|$.} in \cite[Lemma 4.6]{MR4767492}, so we can focus on case 2). By symmetry, we can assume that $\sigma_1\sigma_2\neq -\beta$, and use Lemma \ref{le:n3bound1} to write 
\begin{align}
&\left|\sum_{i\in \{1,2\},j\in\{3,4\}}\sigma_i\sigma_j \dot C_s(x_i-x_j)(1-e^{-\sigma_1\sigma_2\int_{\epsilon^2}^s dr\, \dot C_r(x_1-x_2)})(1-e^{-\sigma_3\sigma_4\int_{\epsilon^2}^s dr\, \dot C_r(x_3-x_4)})\right|\nnb
&\leq |\sigma_1|\left|\sum_{j\in \{3,4\}}\sigma_j \dot C_s(x_1-x_j)\right| \left|1-e^{-\sigma_3\sigma_4\int_{\epsilon^2}^s dr\, \dot C_r(x_3-x_4)}\right| \left|1-e^{-\sigma_1\sigma_2\int_{\epsilon^2}^s dr\, \dot C_r(x_1-x_2)}\right|\nnb
&\quad +|\sigma_2|\left|\sum_{j\in \{3,4\}}\sigma_j \dot C_s(x_2-x_j)\right| \left|1-e^{-\sigma_3\sigma_4\int_{\epsilon^2}^s dr\, \dot C_r(x_3-x_4)}\right| \left|1-e^{-\sigma_1\sigma_2\int_{\epsilon^2}^s dr\, \dot C_r(x_1-x_2)}\right|\nnb
&\leq (F_s(\xi_1,\xi_3,\xi_4)+F_s(\xi_2,\xi_3,\xi_4))\left|1-e^{-\sigma_1\sigma_2\int_{\epsilon^2}^sdr\, \dot C_r(x_1-x_2)}\right|.
\end{align}
We now have two subcases: either i) $\sigma_1\sigma_2\geq 0$ or ii) $\sigma_1\sigma_2<0$. In the latter case, the constraint $\alpha_i^2<2-\frac{\beta}{4\pi}$ implies that we have $\sigma_1\sigma_2>-\gamma_\Sigma$ for some $0<\gamma_\Sigma<4\pi$.

Let us look at the subcase i) first. Here we have 
\begin{align}
\left|1-e^{-\sigma_1\sigma_2\int_{\epsilon^2}^sdr\, \dot C_r(x_1-x_2)}\right|\leq \left|1-e^{-\sigma_1\sigma_2\int_{0}^sdr\, \dot C_r(x_1-x_2)}\right|\leq 6\pi \int_0^s dr\, \dot C_r(x_1-x_2).
\end{align}
If we define 
\begin{align}
G_t^{4,\mathrm{i}}(\xi_1,\dots,\xi_4)=\frac{3}{2}(F_s(\xi_1,\xi_3,\xi_4)+F_s(\xi_2,\xi_3,\xi_4))\int_0^s dr\, e^{-\frac{|x_1-x_2|^2}{4r}}\frac{dr}{r},
\end{align}
then this dominates the quantity we are interested in, and we have by Lemma \ref{le:n3bound1}
\begin{equation}
\|G_t^{4,\mathrm{i}}\|_4\leq C_\Sigma s\int_0^sdr=C_\Sigma s^2.
\end{equation}
Symmetrizing over the arguments yields the claim in this case. 

In case ii), we have 
\begin{align}
  \left|1-e^{-\sigma_1\sigma_2\int_{\epsilon^2}^sdr\, \dot C_r(x_1-x_2)}\right|
  &\leq e^{\gamma_\Sigma \int_0^s dr\, \dot C_r(x_1-x_2)}-1\nnb
  &=\gamma_\Sigma\int_0^s dr\, \dot C_r(x_1-x_2) e^{\gamma_\Sigma \int_0^r du\, \dot C_u(x_1-x_2)}.
\end{align}
From \cite[Lemma 4.4]{MR4767492}, we know that for $0<r<m^{-2}$ (as in our situation since $0<r<s<m^{-2}$)
\begin{align}
\int_0^r du\, \dot C_u(x_1-x_2)=-\frac{1}{2\pi}\log \left(\frac{|x_1-x_2|}{\sqrt{r}}\wedge 1\right)+O(1),
\end{align} 
where the implied constant is universal. We thus define for a large enough constant $C>0$,
\begin{align}
G_t^{4,\mathrm{ii}}(\xi_1,\dots,\xi_4)=C(F_s(\xi_1,\xi_3,\xi_4)+F_s(\xi_2,\xi_3,\xi_4))\int_0^s \frac{dr}{r}\, e^{-\frac{|x_1-x_2|^2}{4r}}\left(\frac{|x_1-x_2|}{\sqrt{r}}\wedge 1\right)^{-\frac{\gamma_\Sigma}{2\pi}}. 
\end{align}
For $C$ large enough (but universal), this dominates the quantity we are interested in, and by Lemma~\ref{le:n3bound1}, we have
\begin{align}
\|G_t^{4,\mathrm{ii}}\|_4\leq C_\Sigma s \int_0^s dr\, \int_{\R^2} dy \frac{1}{r}e^{-\frac{|y|^2}{4r}}\left(\frac{|y|}{\sqrt{r}}\wedge 1\right)^{-\frac{\gamma_\Sigma}{2\pi}} e^{\tdist(0,y)}.
\end{align}
By scaling $y$ by $\sqrt{r}$ and using that $\tdist(0,\sqrt{r}y) \leq \tdist(0,y)$,
we see that this last integral is bounded independently of $r$, and convergent since $\gamma_\Sigma<4\pi$.
We thus have the bound we want, and symmetrizing $G_t^{4,\mathrm{ii}}$ with respect to its arguments produces our object. 

Finally, we note that everywhere in the proof, it was irrelevant whether we had $\delta^2$ or $\epsilon^2$ as the lower limit in the integral. This concludes the proof.
\end{proof}

The final ingredient we need for small $n$ is the proof of Proposition \ref{pr:vtnorm} for $n=4$. We do this now.
\begin{proof}[Proof of Proposition \ref{pr:vtnorm} for $n=4$]
For $n=4$, it turns out that we can drop the exponential term $e^{-\frac{1}{2}\sum_{i,j=1}^4 \sigma_i\sigma_j \int_s^t dr\, \dot C_r(x_i-x_j)}$ -- this is less than one since $\dot C_r(x_i-x_j)$ is a covariance (positive-definite). We thus use the $n=1$ and $n=3$ cases of Proposition \ref{pr:vtnorm} and Lemma \ref{le:n4bound} to define $h_t^{k,4}$ to be the symmetrization of 
\begin{equation}
\frac{1}{2}\int_0^t ds\, \sum_{j=2}^4 |\sigma_1\sigma_j|\frac{1}{4\pi s}e^{-\frac{|x_1-x_j|^2}{4s}} h_s^{0,1}(\xi_1)h_s^{k,3}(\xi_2,\xi_3,\xi_4)+\int_0^t ds\, G_s(\xi_1,\xi_2,\xi_3,\xi_4)\prod_{j=1}^4 h_s^{0,1}(\xi_j).
\end{equation}
Using the bound from the $n=1$ and $n=3$ cases of Proposition~\ref{pr:vtnorm} and Lemma~\ref{le:n4bound} we find 
\begin{align}
\|h_t^{k,4}\|_4&\leq C_\Sigma \int_0^t \int_{\R^2}dy \frac{1}{4\pi s} e^{-\frac{|y|^2}{4s}} e^{\tdist(0,y)} \|h_s^{0,1}\|_1 \|h_s^{k,3}\|_3  +C_\Sigma \int_0^t ds\, \|h_s^{0,1}\|_1^4 \|G_s\|_4\nnb
&\leq \widetilde C_\Sigma \int_0^t ds s^{-2} (s^{1-\frac{\beta}{8\pi}})^4+\widetilde C_\Sigma\int_0^t ds\, s^2 s^{-\frac{\beta}{2\pi}}\nnb
&=2\widetilde C_\Sigma \int_0^t ds\, s^{2-\frac{\beta}{2\pi}}\nnb
&= \widehat C_\Sigma t^{3-\frac{\beta}{2\pi}}\nnb
&= \widehat C_\Sigma t^{-1} t^{(1-\frac{\beta}{8\pi})4},
\end{align}
for suitable constants $C_\Sigma,\widetilde C_\Sigma, \widehat C_\Sigma$ depending only on $\Sigma$. This is exactly the bound we are after. Note that the integrals are convergent as $\beta<6\pi$. This concludes our proof for $n=4$.
\end{proof}

\subsubsection{Proof of the remaining cases of Proposition \ref{pr:vtnorm} for general $n$}

In this subsection, we carry out the induction argument for proving Proposition \ref{pr:vtnorm} in the remaining cases. The proof does not require further preparation. 

\begin{proof}[Proof of Proposition \ref{pr:vtnorm} for $n\geq 5$ and $0\leq k<n$] 
This is very similar to the induction argument in the proof of \cite[Proposition 4.1]{MR4767492}, and in fact quite similar to the induction argument we presented in the proof of the $k=n$ case of this proposition.

As for the $k=n$ case, we make the induction hypothesis that $|\tilde V_t^{k,n}|\leq h_t^{k,n}$ (as in the statement of the proposition)
\begin{align}
\|h_t^{k,n}\|_n \leq n^{n-2} t^{-1}C_\Sigma^{n-1}\left(C t^{1-\frac{\beta}{8\pi}}\right)^n
\end{align}
for $0<k\leq n\leq N-1$ for some $N\geq 5$. 

As opposed to the $k=n$ case, we need to be slightly more careful in analysing the recursion \eqref{eq:vtilkn}. The issue is that we can have terms of the form $\tilde V_t^{0,2}$ which cannot be bounded by integrable functions. We thus treat separately the cases where (i) $|I_1|,|I_2|\neq 2$ and (ii) $|I_1|=2$ or $|I_2|=2$.

In case (i), we can proceed as $k=n$ and define
\begin{align}
h_t^{k,N,\mathrm{i}}(\xi_1,\dots,\xi_n)=\frac{1}{8\pi}\sum_{\substack{I_1\dot \cup I_2=[N]: \\ |I_1|,|I_2|\neq 2}}\sum_{i\in I_1,j\in I_2}|\sigma_i \sigma_j|\int_0^t ds\, \frac{e^{-\frac{|x_i-x_j|^2}{4s}}}{4\pi s} h_s^{|I_1\cap [k]|,|I_1|}(\xi_{I_1})h_s^{|I_2\cap[k]|,|I_2|}(\xi_{I_2}).
\end{align}
The contribution to $\tilde V_t^{k,|N|}$ from the $|I_1||I_2|\neq 2$ terms is bounded by this and using our induction hypothesis, the fact that the heat kernel integrates to one, and symmetry in the variables, we see as before that 
\begin{align}
\|h_t^{k,N,\mathrm{i}}\|_N&\leq \frac{6\pi}{8\pi}\sum_{\substack{I_1\dot \cup I_2=[N]: \\ |I_1|,|I_2|\neq 2}}|I_1|^{|I_1|-1} |I_2|^{|I_2|-1} C_\Sigma^{N-2} C^N \int_0^t ds\, s^{-2} \left(s^{1-\frac{\beta}{8\pi}}\right)^N\nnb
&\leq \frac{3}{2}\frac{N-1}{-1+N(1-\frac{\beta}{8\pi})} N^{N-2}C_\Sigma^{N-2}   t^{-1}(Ct^{1-\frac{\beta}{8\pi}})^N.
\end{align}
Since $\beta<6\pi$, the integral is convergent for $N\geq 5$, and the prefactor is bounded in $N$ so possibly adjusting $C_\Sigma$ yields a bound of the type we are after.

It remains to treat the terms with $|I_1|=2$ or $|I_2|=2$. Using similar arguments as before (namely our recursion, positive definiteness of $\dot C_r$ etc) as well as Lemma \ref{le:n3bound1}, we can bound the contributions of such terms to $\tilde V_t^{k,N}$ by 
\begin{align}
\sum_{1\leq a<b\leq N}&\int_0^t ds\, \left|\sum_{j\in [N]\setminus \{a,b\}}(\sigma_a\sigma_j \dot C_s(x_a-x_j)+\sigma_b\sigma_j\dot C_s(x_b-x_j))\right|\left|1-e^{-\sigma_a\sigma_b\int_{\nu_{\delta,\epsilon}}^s dr\, \dot C_r(x_a-x_b)}\right| \nnb
&\qquad \times e^2 s^{-\frac{\beta}{4\pi}}h_s^{|[k]\setminus \{a,b\}|,N-2}(\xi_{[n]\setminus \{a,b\}})\nnb
&\leq e^2\sum_{1\leq a<b\leq N}\sum_{j\in [n]\setminus \{a,b\}}\int_0^t ds\, F_s(\xi_a,\xi_b,\xi_j)s^{-\frac{\beta}{4\pi}}h_s^{|[k]\setminus \{a,b\}|,N-2}(\xi_{[N]\setminus \{a,b\}})\nnb
&=: h_t^{k,N,\mathrm{ii}}(\xi_1,\dots,\xi_N),
\end{align}
where $\nu_{\delta,\epsilon}$ is either $\delta^2$ or $\epsilon^2$, depending on $(\sigma_a,\sigma_b)$.

Using our induction hypothesis, Lemma \ref{le:n3bound1}, and symmetry in the variables, we have (possibly adjusting $C$)
\begin{align}
\|h_t^{k,N,\mathrm{ii}}\|_N &\leq e^2 \frac{N(N-1)}{2}(N-2) C_\Sigma \int_0^tds\, s s^{-\frac{\beta}{4\pi}} s^{-1}(N-2)^{N-4}C_\Sigma^{N-3} C^{N-2}(s^{1-\frac{\beta}{8\pi}})^{N-2}\nnb
&\leq C^N N N^{N-2}C_\Sigma^{N-2}\int_0^t ds\, s^{-2} (s^{1-\frac{\beta}{8\pi}})^N,
\end{align}
which yields the type of bound we want with a similar argument as in the $|I_1|,|I_2|\neq 2$ case.

Thus defining $h_t^{k,n}$ to be (the symmetrization of) $h_t^{k,n,\mathrm{i}}+h_t^{k,n,\mathrm{ii}}$ yields our claim.
\end{proof}

\section{Uniform integrability of fractional free field correlation functions}
\label{app:renormpart}

In this appendix, we use Appendix~\ref{app:fracrenorm} to prove Theorem~\ref{th:fracapp-bis} which is restated as the following theorem.
The roles of $n$ and $p$ are switched in this appendix,
for consistency with
Appendix~\ref{app:fracrenorm} in which the roles are switched to be consistent with \cite{MR4767492}.
Our argument generalizes \cite[Section 5]{MR4767492}.

\begin{theorem}\label{th:fracapp}
  Let $\beta \in (0,6\pi)$ and $\alpha_1,\dots,\alpha_p \in \R$ with $\alpha_i^2 < \min\{1,2-\beta/4\pi\}$. Then for $n \geq 0$,
  \begin{equation}
    \avg{\wick{e^{i\alpha_1 \sqrt{4\pi} \varphi(x_1)}}\cdots \wick{e^{i\alpha_p \sqrt{4\pi} \varphi(x_p)}}; \wick{\cos(\sqrt{\beta}\varphi(y_1))}; \cdots; \wick{\cos(\sqrt{\beta}\varphi(y_n))}}_{\GFF(0)}^{\mathsf T}\label{e:finvol-corfunc-deriv-explicit-app}
  \end{equation}
  is locally integrable in $x_1,\dots, x_p, y_1, \dots, y_n\in \C$, and
  the regularized version \eqref{eq:gffregcf} is uniformly integrable in $m,\delta,\epsilon\in(0,1)$
  and as $\delta\to 0$  then $\epsilon \to 0$ then $m\to 0$ the regularized correlation function converges
  to its limiting version, see   \eqref{eq:gffregcfconv}.
\end{theorem}

\subsection{The renormalized partition function}

We use the notation and set-up of Appendix~\ref{app:fracrenorm}.
Our main goal is to understand $m,\delta,\epsilon\to 0$ asymptotics of the generalized partition function 
\begin{align}
  Z(\eta|m,\delta,\epsilon)
  &=\E\qB{e^{-\int_{\R^2\times \Sigma_\alpha}d\xi\, \eta(\xi)\delta^{-\frac{\sigma^2}{4\pi}}e^{i\sigma\Phi_{\delta^2,\infty}(x)}-\int_{\R^2\times \Sigma_\beta}d\xi\, \eta(\xi)\epsilon^{-\frac{\sigma^2}{4\pi}}e^{i\sigma \Phi_{\epsilon^2,\infty}(x)})}}
    \nnb
  &= \E\qB{e^{-V_{\delta^2,\epsilon^2}(\Phi_{\delta^2,\infty},\Phi_{\epsilon^2,\infty})}}
  ,
\end{align}
where $\Phi_{s,t}$ depends on $m$ (suppressed in the notation).
It will use useful to express this partition function in terms of the renormalized potential:
for any $t \in [0,+\infty]$,
\begin{equation}\label{eq:partfunc}
  Z(\eta|m,\delta,\epsilon)
  =
  \E\qb{e^{-V_{t}(\Phi_{t,\infty},\Phi_{\epsilon^2\vee t,\infty}|m,\delta,\epsilon)}}
  .
\end{equation}
For $\beta \geq 4\pi$, the generalized partition function is divergent in the $\epsilon\to 0$ limit,
and to control this divergence, we renormalize it by a $\epsilon$-dependent counterterm.
More precisely, we introduce
\begin{equation}
A(\xi_1,\xi_2|m,\epsilon)=\E\qa{ \epsilon^{-\frac{\beta}{4\pi}}e^{i\sigma_1\Phi_{\epsilon,\infty}(x_1)}; \epsilon^{-\frac{\beta}{4\pi}}e^{i\sigma_2\Phi_{\epsilon,\infty}(x_2)}}^\mathsf T,
\end{equation}
and define
\begin{multline}
  \mathcal Z(\eta|m,\delta,\epsilon)=
  \E\qa{\prod_{j=1}^p \delta^{-\alpha_j^2} e^{i\sqrt{4\pi}\alpha_j \Phi_{\delta^2,\infty}}(f_j) e^{2z \epsilon^{-\frac{\beta}{4\pi}}\int_{\R^2} dx\, \eta(x) \cos(\sqrt{\beta}\Phi_{\epsilon^2,\infty}(x))}}
  \\
  e^{-\frac{z^2}{2}\int_{(\Lambda\times \Sigma_\beta)^2}d\xi_1d\xi_2 A(\xi_1,\xi_2|m,\epsilon)}.
\end{multline}
Since we are dealing with entire functions for $\delta>0$ and $\epsilon>0$ (due to our random variables being bounded),
one readily checks that the smeared correlation function of interest can be written as 
\begin{equation}
\left.\partial_{u_1}\cdots \partial_{u_p}\right|_{u=0}\mathcal Z(\eta_{u,z}|m,\delta,\epsilon),
\end{equation}
divided by $\mathcal Z(\eta_{u,z}|m,\delta,\epsilon)|_{u=0}$, where 
\begin{equation}\label{eq:etauz}
\eta_{u,z}(\xi)=-\sum_{j=1}^pu_j \delta_{\sigma,\sqrt{4\pi}\alpha_j}f_j(x)-z\1_{\sigma\in \Sigma_\beta}\1_\Lambda(x).
\end{equation}
To control the derivatives, we study the Taylor coefficients of the entire function $z\mapsto \mathcal Z(z\eta|m,\delta,\epsilon)$
for a given $\eta\in L_c^\infty(\R^2\times \Sigma)$.
In the next lemma we express these in terms of the GFF correlation functions. In the lemma following it we will instead
represent them in terms of the renormalized potential.

\begin{lemma}\label{le:mncal1}
For $z\in \C$ and $\eta\in L_c^\infty(\R^2\times \Sigma)$, we have  
\begin{equation} \label{e:mncal1-1}
\mathcal Z(z\eta|m,\delta,\epsilon)= \sum_{n=0}^\infty \frac{z^n}{n!}\mathcal M_n(\eta|m,\delta,\epsilon),
\end{equation}
where 
\begin{equation} \label{e:mncal1-2}
\mathcal M_n(\eta|m,\delta,\epsilon)=\int_{(\R^2\times \Sigma)^n}d\xi_1\cdots d\xi_n \, \eta(\xi_1)\cdots \eta(\xi_n) \, \mathcal M_n(\xi_1,\dots,\xi_n|m,\delta,\epsilon)
\end{equation}
with $\mathcal M_n(\xi_1,\dots,\xi_n|m,\delta,\epsilon)$ an explicit symmetric function given by symmetrizing
the integrals in \eqref{e:cZ-explicit} in the proof.
This function is the unique continuous function on $(\R^2\times \Sigma)^n$ for which
\begin{multline}
  \frac{1}{n!}\partial_{u_1}\cdots \partial_{u_n}\mathcal M(u_1 f_1+\cdots+u_n f_n|m,\delta,\epsilon)
  \\
  =\int_{(\R^2\times \Sigma)^n}d\xi_1\cdots d\xi_n f_1(\xi_1)\cdots f_n(\xi_n) \mathcal M_n(\xi_1,\dots,\xi_n|m,\delta,\epsilon)
\end{multline}
for all $f_1,\dots,f_n\in C_c^\infty(\R^2\times \Sigma)$ (meaning that for each $\sigma\in \Sigma$, $x\mapsto f_j(x,\sigma)\in C_c^\infty(\R^2)$).
\end{lemma}
\begin{proof}
This is analogous to parts of \cite[Lemma 5.5 and Lemma 5.6]{MR4767492}. We begin by writing out explicitly the definition of $\mathcal Z(z\eta|m,\delta,\epsilon)$:
\begin{align}
\mathcal Z(z\eta|m,\delta,\epsilon)&=\avg{e^{-z\int_{\R^2\times \Sigma_\alpha}d\xi\, \eta(\xi)\delta^{-\frac{\sigma^2}{4\pi}}e^{i\sigma \Phi_{\delta^2,\infty}}}e^{-z\int_{\R^2\times \Sigma_\beta}d\xi\, \eta(\xi)\epsilon^{-\frac{\sigma^2}{4\pi}}e^{i\sigma \Phi_{\epsilon^2,\infty}}}}_{\GFF(m,\delta)} \nnb
&\quad \times e^{-\frac{z^2}{2}\int_{(\R^2\times \Sigma_\beta)^2}d\xi_1 d\xi_2\, A(\xi_1,\xi_2|m,\epsilon)}.
\end{align}
As the random variables appearing in the expectation are bounded random variables, we can expand both exponentials and interchange the order of summation and the expectation. Expanding also the last expectation, we find 
\begin{align}
&\mathcal Z(z\eta|m,\delta,\epsilon)\nnb
&=\sum_{j,k,l=0}^\infty \frac{(-1)^{j+k+l}}{2^l j!k!l!}z^{j+k+2l}\int_{(\R^2\times \Sigma_\alpha)^j \times (\R^2\times \Sigma_\beta)^{k+2l}}d\xi_1\cdots d\xi_j d\xi_1'\cdots d\xi_k' d\tilde\xi_1\cdots d\tilde \xi_{2l}\nnb
&\qquad  \times \eta(\xi_1)\cdots \eta(\xi_j)\eta(\xi_1')\cdots \eta(\xi_k') \eta(\tilde \xi_1)\cdots \eta(\tilde \xi_{2l})A(\tilde \xi_1,\tilde \xi_2|m,\epsilon)\cdots A(\tilde \xi_{2l-1},\tilde \xi_{2l}|m,\epsilon)\nnb
  &\qquad  \times \avg{\delta^{-\frac{\sigma_1^2}{4\pi}}e^{i\sigma_1 \Phi_{\delta^2,\infty}(x_1)}\cdots \delta^{-\frac{\sigma_j^2}{4\pi}}e^{i\sigma_j \Phi_{\delta^2,\infty}(x_j)}\nnb
    &\qquad\qquad \times \epsilon^{-\frac{(\sigma_1')^2}{4\pi}}e^{i\sigma_1' \Phi_{\epsilon^2,\infty}(x_1)}\cdots \epsilon^{-\frac{(\sigma_l')^2}{4\pi}}e^{i\sigma_l' \Phi_{\epsilon^2,\infty}(x_l')}}_{\GFF(m,\delta)}.
\end{align}
Writing $n=j+k+2l$ and renaming the integration variables ($\xi_i'=\xi_{j+i}$ and $\tilde \xi_i=\xi_{j+k+i}$), we can write this as 
\begin{align} \label{e:cZ-explicit}
&\mathcal Z(z\eta|m,\delta,\epsilon)\nnb
&=\sum_{n=0}^\infty \frac{z^n}{n!}\int_{(\R^2\times \Sigma)^n}d\xi_1\cdots d\xi_n \, \eta(\xi_1)\cdots \eta(\xi_n)\nnb
&\quad \times \sum_{ j,k,l=0}^n (-1)^{j+k+l}\frac{n!}{2^l j!k!l!}\1\{j+k+2l=n\}\1\{\sigma_1,\dots,\sigma_j\in \Sigma_\alpha,\sigma_{j+1},\dots,\sigma_n\in \Sigma_\beta\}\nnb
&\quad \times A(\xi_{j+k+1},\xi_{j+k+2}|m,\delta,\epsilon)\cdots A(\xi_{n-1},\xi_{n}|m,\delta,\epsilon)\nnb
  &\quad \times \E\Big[\delta^{-\frac{\sigma_1^2}{4\pi}}e^{i\sigma_1 \Phi_{\delta^2,\infty}(x_1)}\cdots \delta^{-\frac{\sigma_j^2}{4\pi}}e^{i\sigma_j \Phi_{\delta^2,\infty}(x_j)}\nnb
    &\qquad\qquad \times \epsilon^{-\frac{\sigma_{j+1}^2}{4\pi}}e^{i\sigma_{j+1} \Phi_{\epsilon^2,\infty}(x_{j+1})}\cdots \epsilon^{-\frac{\sigma_{j+l}^2}{4\pi}}e^{i\sigma_{j+l} \Phi_{\epsilon^2,\infty}(x_{j+l})}\Big].
\end{align}
For any integrable function $F(\xi_1,\dots,\xi_n)$ we have %
\begin{align}
  &\int_{(\R^2\times \Sigma)^n}d\xi_1\cdots d\xi_n \eta(\xi_1)\cdots \eta(\xi_n)F(\xi_1,\dots,\xi_n)
  \nnb
  &=\frac{1}{n!}\sum_{\tau\in S_n}\int_{(\R^2\times \Sigma)^n}d\xi_1\cdots d\xi_n \eta(\xi_1)\cdots \eta(\xi_n)F(\xi_{\tau_1},\dots,\xi_{\tau_n}).
\end{align}
This implies \eqref{e:mncal1-1}--\eqref{e:mncal1-2}. %
The uniqueness claim is proven exactly as in \cite[Lemma~5.6]{MR4767492}, and we therefore omit further details.
\end{proof}

We next express $\mathcal M_n(\xi_1,\dots,\xi_n|m,\delta,\epsilon)$ in terms of the renormalized potential.
\begin{lemma}\label{le:mnrp}
For %
$0<\delta^2<\epsilon^2<t<\min(1,m^{-2})$, the
kernels $\cM_n(\xi_1,\dots,\xi_n|m,\delta,\epsilon)$ can be written in terms of the renormalized potential as
as a sum over permutations of linear combinations (with possibly some constraints on the $\sigma$-variables) of the product of %
\begin{align}\label{eq:funct1}
  &\prod_{j=1}^p \tilde V_t^{k_j,n_j}(\xi_{\sum_{q=1}^{j-1}n_q+1},\dots,\xi_{\sum_{q=1}^j n_q}|m,\delta,\epsilon)
    \nnb
    &\qquad
    \times \E\qa{\prod_{j=1}^p(e^{i\sum_{r=\sum_{q=1}^{j-1}n_q+1}^{\sum_{q=1}^j n_q}\sigma_r \Phi_{t,\infty}(x_r)}-\delta_{k_j,0}\delta_{n_j,2}) }
\end{align}
with (possibly multiple) factors of
\begin{equation}\label{eq:funct2}
A(\xi_1,\xi_2|m,\epsilon)-\tilde V_t^{0,2}(\xi_1,\xi_2|m,\delta,\epsilon),
\end{equation}
and the variables $\xi_1,\xi_2$ in \eqref{eq:funct2} do not appear in the \eqref{eq:funct1} terms,
and $p$ denotes the number of $\xi_i \in \R^2 \times \Sigma_\alpha$.
Explicitly, $\mathcal M_n$ is given by \eqref{e:Mn-horrible}.
\end{lemma}

\begin{proof}
Let us begin by using \eqref{eq:partfunc} and Proposition \ref{pr:rpexp} to write (for $t$ satisfying the assumptions of the lemma and $|z|$ small enough that \eqref{eq:etabound} is satisfied with $\eta$ replaced by $z\eta$)
\begin{align}
&\mathcal Z(z\eta|m,\delta,\epsilon)\nnb
&=\E\qa{e^{-\sum_{n=1}^\infty \frac{z^n}{n!}\sum_{k=0}^n {n\choose k}\int_{(\R^2\times \Sigma_\alpha)^k \times(\R^2\times \Sigma_\beta)^{n-k}}d\xi_1\cdots d\xi_n \eta(\xi_1)\cdots \eta(\xi_n) e^{i\sum_{j=1}^n \sigma_j \Phi_{t,\infty}(x_j)}\tilde V_t^{k,n}(\xi_1,\dots,\xi_n|m,\delta,\epsilon)}}\nnb
&\quad \times e^{-\frac{z^2}{2}\int_{(\R^2\times \Sigma_\beta)^2}d\xi_1d\xi_2\, \eta(\xi_1)\eta(\xi_2) A(\xi_1,\xi_2|m,\epsilon)}\nnb
&=\E\qa{e^{-\sum_{n=1}^\infty \frac{z^n}{n!}\sum_{k=0}^n {n\choose k}\int_{(\R^2\times \Sigma_\alpha)^k \times(\R^2\times \Sigma_\beta)^{n-k}}d\xi_1\cdots d\xi_n \eta(\xi_1)\cdots \eta(\xi_n) (e^{i\sum_{j=1}^n \sigma_j \Phi_{t,\infty}(x_j)}-\delta_{k,0}\delta_{n,2})\tilde V_t^{k,n}(\xi_1,\dots,\xi_n|m,\delta,\epsilon)}}\nnb
&\quad \times e^{-\frac{z^2}{2}\int_{(\R^2\times \Sigma_\beta)^2}d\xi_1d\xi_2\, \eta(\xi_1)\eta(\xi_2)(\avg{\epsilon^{-\frac{\beta}{4\pi}}e^{i\sigma_1 \varphi(x_1)}; \epsilon^{-\frac{\beta}{4\pi}}e^{i\sigma_2\varphi(x_2)}}_{\GFF(\epsilon,m)}^\mathsf T-\tilde V_t^{0,2}(\xi_1,\xi_2|m,\delta,\epsilon))}.
\end{align}
Thus 
\begin{align}
&\mathcal M_n(\eta|m,\delta,\epsilon)\nnb
&=\left.\frac{d^n}{dz^n}\right|_{z=0}\mathcal Z(z\eta|m,\delta,\epsilon)\nnb
&=\sum_{l=0}^n {n\choose l}\left.\frac{d^l}{dz^l}\right|_{z=0}\nnb
&\quad \E\qa{e^{-\sum_{n=1}^\infty \frac{z^n}{n!}\sum_{k=0}^n {n\choose k}\int_{(\R^2\times \Sigma_\alpha)^k \times(\R^2\times \Sigma_\beta)^{n-k}}d\xi_1\cdots d\xi_n \eta(\xi_1)\cdots \eta(\xi_n) (e^{i\sum_{j=1}^n \sigma_j \Phi_{t,\infty}(x_j)}-\delta_{k,0}\delta_{n,2})\tilde V_t^{k,n}(\xi_1,\dots,\xi_n|m,\delta,\epsilon)}}\nnb
&\qquad \times \left.\frac{d^{n-l}}{dz^{n-l}}\right|_{z=0}e^{-\frac{z^2}{2}\int_{(\R^2\times \Sigma_\beta)^2}d\xi_1d\xi_2\, \eta(\xi_1)\eta(\xi_2)(A(\xi_1,\xi_2|m,\epsilon)-\tilde V_t^{0,2}(\xi_1,\xi_2|m,\delta,\epsilon))}.
\end{align} 
In the first derivative, we want to take the derivative inside of the expectation. To justify this, we note that if we choose $r$ small enough that \eqref{eq:etabound} is satisfied with $\eta$ replaced by $r\eta$, then the expectation is analytic in a domain containing the set $\{z\in \C:|z|\leq r\}$ and we have by Cauchy's integral formula for derivatives and Fubini
\begin{align}
&\left.\frac{d^l}{dz^l}\right|_{z=0}\big\langle\exp\big(-\sum_{n=1}^\infty \frac{z^n}{n!}\sum_{k=0}^n {n\choose k}\int_{(\R^2\times \Sigma_\alpha)^k \times(\R^2\times \Sigma_\beta)^{n-k}}d\xi_1\cdots d\xi_n \eta(\xi_1)\cdots \eta(\xi_n) \nnb
&\qquad \qquad \times (e^{i\sum_{j=1}^n \sigma_j \Phi_{t,\infty}(x_j)}-\delta_{k,0}\delta_{n,2})\tilde V_t^{k,n}(\xi_1,\dots,\xi_n|m,\delta,\epsilon)\big)\big\rangle_{\GFF(\sqrt{t},m)}\nnb
&=l!\big\langle\oint_{|z|=r}\frac{dz}{2\pi i z^{l+1}}\exp\big(-\sum_{n=1}^\infty \frac{z^n}{n!}\sum_{k=0}^n {n\choose k}\int_{(\R^2\times \Sigma_\alpha)^k \times(\R^2\times \Sigma_\beta)^{n-k}}d\xi_1\cdots d\xi_n \eta(\xi_1)\cdots \eta(\xi_n)\nnb
&\qquad \qquad \times  (e^{i\sum_{j=1}^n \sigma_j \Phi_{t,\infty}(x_j)}-\delta_{k,0}\delta_{n,2})\tilde V_t^{k,n}(\xi_1,\dots,\xi_n|m,\delta,\epsilon)\big)\big\rangle_{\GFF(\sqrt{t},m)}\nnb
&=\big\langle \left.\frac{d^l}{dz^l}\right|_{z=0}\exp\big(-\sum_{n=1}^\infty \frac{z^n}{n!}\sum_{k=0}^n {n\choose k}\int_{(\R^2\times \Sigma_\alpha)^k \times(\R^2\times \Sigma_\beta)^{n-k}}d\xi_1\cdots d\xi_n \eta(\xi_1)\cdots \eta(\xi_n) \nnb
&\qquad \times (e^{i\sum_{j=1}^n \sigma_j \Phi_{t,\infty}(x_j)}-\delta_{k,0}\delta_{n,2})\tilde V_t^{k,n}(\xi_1,\dots,\xi_n|m,\delta,\epsilon)\big)\big\rangle_{\GFF(\sqrt{t},m)}.
\end{align}
To compute this derivative, we first note that by Faà di Bruno's formula, for any $(a_n)_{n=1}^\infty$ such that $\sum_{n=1}^\infty \frac{a_n}{n!}z^n$ is analytic in some neighborhood of the origin,
\begin{equation}
\left.\frac{d^l}{dz^l}\right|_{z=0}e^{\sum_{n=1}^\infty \frac{a_n}{n!}z^n}=\sum_{\tau\in \frP_l}\prod_{B\in \tau}a_{|B|}.
\end{equation}
If we write $\tau\in \frP_l$ as $\tau=\{B_1,\dots,B_{|\tau|}\}$ and define $S_p(\tau)=\sum_{q=1}^p|B_q|$ (understanding that $S_0(\tau)=0$), we thus find by relabeling our integration variables 
\begin{align}
&\big\langle \left.\frac{d^l}{dz^l}\right|_{z=0}\exp\big(-\sum_{n=1}^\infty \frac{z^n}{n!}\sum_{k=0}^n {n\choose k}\int_{(\R^2\times \Sigma_\alpha)^k \times(\R^2\times \Sigma_\beta)^{n-k}}d\xi_1\cdots d\xi_n \eta(\xi_1)\cdots \eta(\xi_n) \nnb
&\qquad \times (e^{i\sum_{j=1}^n \sigma_j \Phi_{t,\infty}(x_j)}-\delta_{k,0}\delta_{n,2})\tilde V_t^{k,n}(\xi_1,\dots,\xi_n|m,\delta,\epsilon)\big)\big\rangle_{\GFF(\sqrt{t},m)}\nnb
&=\sum_{\tau\in \frP_l}(-1)^{|\tau|} \sum_{k_1=0}^{|B_1|}\cdots \sum_{k_{|\tau|}=0}^{|B_{|\tau|}|}\prod_{j=1}^{|\tau|}{|B_j|\choose k_j}\int_{\prod_{j=1}^{|\tau|}(\R^2 \times \Sigma_\alpha)^{k_j}\times (\R^2\times \Sigma_\beta)^{|B_j|-k_j}}d\xi_1\cdots d\xi_l\eta(\xi_1)\cdots \eta(\xi_l)\nnb
&\quad \times \prod_{j=1}^{|\tau|}\tilde V_t^{k_j,|B_j|}(\xi_{S_{j-1}+1},\dots,\xi_{S_j}|m,\delta,\epsilon)\avg{\prod_{j=1}^{|\tau|}(e^{i\sum_{p=S_{j-1}(\tau)+1}^{S_j(\tau)}\sigma_p \Phi_{t,\infty}(x_p)}-\delta_{k_j,0}\delta_{|B_j|,2})}_{\GFF(\sqrt{t},m)}\nnb
&=\int_{(\R^2\times \Sigma)^l}d\xi_1\cdots d\xi_l \eta(\xi_1)\cdots \eta(\xi_l)\sum_{\tau\in \frP_l}(-1)^{|\tau|}\sum_{k_1=0}^{|B_1|}\cdots \sum_{k_{|\tau|}=0}^{|B_{|\tau|}|}\prod_{j=1}^{|\tau|}{|B_j|\choose k_j}\nnb
&\quad \times \1\{\sigma_{S_{j-1}(\tau)+1},\dots,\sigma_{S_{j-1}(\tau)+k_j}\in \Sigma_\alpha, \sigma_{S_{j-1}(\tau)+k_j+1},\dots,\sigma_{S_j(\tau)}\in \Sigma_\beta \text{ for }1\leq j\leq |\tau|\}\nnb
&\quad \times \prod_{j=1}^{|\tau|}\tilde V_t^{k_j,|B_j|}(\xi_{S_{j-1}(\tau)+1},\dots,\xi_{S_{j}(\tau)}|m,\delta,\epsilon)\avg{\prod_{j=1}^{|\tau|}(e^{i\sum_{p=S_{j-1}(\tau)+1}^{S_j(\tau)}\sigma_p \Phi_{t,\infty}(x_p)}-\delta_{k_j,0}\delta_{|B_j|,2})}_{\GFF(\sqrt{t},m)}
\end{align}
For the $\left.\frac{d^{n-l}}{dz^{n-l}}\right|_{z=0}$-derivative, we note that 
\begin{align}
&\left.\frac{d^{n-l}}{dz^{n-l}}\right|_{z=0}e^{-\frac{z^2}{2}\int_{(\R^2\times \Sigma_\beta)^2}d\xi_1d\xi_2\, \eta(\xi_1)\eta(\xi_2)(\avg{\epsilon^{-\frac{\beta}{4\pi}}e^{i\sigma_1 \varphi(x_1)}; \epsilon^{-\frac{\beta}{4\pi}}e^{i\sigma_2\varphi(x_2)}}_{\GFF(\epsilon,m)}^\mathsf T-\tilde V_t^{0,2}(\xi_1,\xi_2|m,\delta,\epsilon))}\nnb
  &=\1\{n-l \text{ is even}\}(-1)^{\frac{n-l}{2}}2^{-\frac{n-l}{2}}{n-l\choose \frac{n-l}{2}}
    \nnb
    &\qquad\qquad \times \left(\int_{(\R^2\times \Sigma_\beta)^2}d\xi d\xi'\, \eta(\xi)\eta(\xi')\left(A(\xi,\xi'|m,\epsilon)-\tilde V_t^{0,2}(\xi,\xi'|m,\delta,\epsilon)\right)\right)^{\frac{n-l}{2}}.
\end{align}
Combining these identities, we find that 
\begin{align} \label{e:Mn-horrible}
&\mathcal M_n(\eta|m,\delta,\epsilon)\nnb
&=\int_{(\R^2\times \Sigma)^n}d\xi_1\cdots d\xi_n\, \eta(\xi_1)\cdots \eta(\xi_n) \sum_{l=0}^n {n\choose l}\1\{n-l \text{ is even}\}(-1)^{\frac{n-l}{2}}2^{-\frac{n-l}{2}}{n-l\choose \frac{n-l}{2}}\nnb
&\quad \times \sum_{\tau \in \frP_l}(-1)^{|\tau|}\sum_{k_1=0}^{|B_1|}\cdots \sum_{k_{|\tau|}=0}^{|B_{|\tau|}|}\prod_{j=1}^{|\tau|}{|B_j|\choose k_j}\nnb
&\quad \times \1\{\sigma_{S_{j-1}(\tau)+1},\dots,\sigma_{S_{j-1}(\tau)+k_j}\in \Sigma_\alpha, \sigma_{S_{j-1}(\tau)+k_j+1},\dots,\sigma_{S_j(\tau)}\in \Sigma_\beta \text{ for }1\leq j\leq |\tau|, \, \sigma_p\in \Sigma_{\beta} \text{ for } p>l\}\nnb
&\quad \times \prod_{j=1}^{|\tau|}\tilde V_t^{k_j,|B_j|}(\xi_{S_{j-1}(\tau)+1},\dots,\xi_{S_{j}(\tau)}|m,\delta,\epsilon)\avg{\prod_{j=1}^{|\tau|}(e^{i\sum_{p=S_{j-1}(\tau)+1}^{S_j(\tau)}\sigma_p \Phi_{t,\infty}(x_p)}-\delta_{k_j,0}\delta_{|B_j|,2})}_{\GFF(\sqrt{t},m)}\nnb
&\quad \times \prod_{p=0}^{\frac{n-l}{2}-1} \left(A(\xi_{n-2p-1},\xi_{n-2p}|m,\epsilon)-\tilde V_t^{0,2}(\xi_{n-2p-1},\xi_{n-2p}|m,\delta,\epsilon)\right).
\end{align}
This formula is of the desired form.
\end{proof}

\subsection{Proof of Theorem~\ref{th:fracapp}}
Note that for Theorem \ref{th:fracapp}, due to the moments to cumulants map, it is sufficient for us to prove the corresponding claims for
\begin{align}
\mathcal M_n(f_1,\dots,f_p|m,\epsilon,\delta)=\left.\frac{\partial^n}{\partial z^n}\right|_{z=0}\prod_{j=1}^p \left.\frac{\partial}{\partial u_j}\right|_{u=0} \mathcal Z(\eta_{u,z}|m,\epsilon,\delta),
\end{align}
where $\eta_{u,z}$ is as in \eqref{eq:etauz}.
As $\eta_{u,z}$ is linear in $u$ and $z$, this quantity can be extracted from the order $n+p$ term in the expansion of $\mathcal Z(w \eta_{u,z}|m,\epsilon,\delta)$ (as a function of $w$). In other words,
\begin{align}
\mathcal M_n(f_1,\dots,f_p|m,\epsilon,\delta)=\frac{\partial^n}{\partial z^n}\prod_{j=1}^p \frac{\partial}{\partial u_j}\mathcal M_{n+p}(\eta_{u,z}|m,\epsilon,z).
\end{align}
Noting that 
\begin{align}
\frac{\partial^n}{\partial z^n} z^n=\prod_{j=1}^n \frac{\partial}{\partial v_j}\left(\sum_{j=1}^n v_j\right)^n,
\end{align}
we see from Lemma \ref{le:mncal1}  that we can write this as
\begin{align}
&\mathcal M_n(f_1,\dots,f_p|m,\epsilon,\delta)\nnb
&=\prod_{j=1}^n \frac{\partial}{\partial v_j}\prod_{k=1}^p \frac{\partial}{\partial u_k}\mathcal M_{n+p}\left(\left.-\sum_{j=1}^n v_j\1_{\sigma\in \Sigma_\beta}\1_\Lambda-\sum_{k=1}^p u_k \delta_{\sigma,\sqrt{4\pi}\alpha_k}f_k\right|m,\epsilon,\delta\right)\nnb
  &=(n+p)!(-1)^{n+p}\int_{(\R^2\times \Sigma)^{n+p}}d\xi_1\cdots d\xi_{n+p}\nnb
    &\qquad \qquad \qquad\qquad \qquad \times \1_{\sigma_1,\dots,\sigma_n\in \Sigma_\beta}\1_\Lambda(x_1)\cdots \1_\Lambda(x_n)\delta_{\sigma_{n+1},\sqrt{4\pi}\alpha_1}\cdots \delta_{\sigma_{n+p},\sqrt{4\pi}} \nnb
&\qquad \qquad \qquad\qquad \qquad \times f_1(x_{n+1})\cdots f_p(x_{n+p})\mathcal M_{n+p}(\xi_1,\dots,\xi_{n+p}|m,\epsilon,\delta).
\end{align}
Thus Theorem \ref{th:fracapp} boils down to convergence and bounds for $\mathcal M_{n+p}(\xi_1,\dots,\xi_{n+p}|m,\epsilon,\delta)$.

\begin{proof}[Proof of Theorem \ref{th:fracapp}]
As just argued, our task is to show that $\mathcal M_{n+p}(\xi_1,\dots,\xi_{n+p}|m,\epsilon,\delta)$ converges as $m,\epsilon,\delta\to 0$
and that it satisfies the desired bounds that are uniform in $m,\epsilon,\delta$. We split the proof into two parts: convergence and bounds.

\subproof{convergence} The convergence follows from Lemma~\ref{le:mncal1} which gives the explicit representation
in terms of GFF correlations, namely as the symmetrization of the integrals \eqref{e:cZ-explicit},
together with the convergence of the GFF correlation functions,  similar to~Lemma~\ref{le:gfffrac}.

\subproof{bounds}
For the bounds, we apply Lemma~\ref{le:mnrp} which gives a representation of the kernels in terms of the renormalized potential.
More specifically, $\mathcal M_{n+p}(\xi_1,\dots,\xi_{n+p}|m,\epsilon,\delta)$ can be written
as a sum over permutations of linear combination (with possibly some constraints on the $\sigma$-variables) of products of functions of the form  \eqref{eq:funct1} and \eqref{eq:funct2}.

The term \eqref{eq:funct2} was already bounded in \cite[Lemma 5.4]{MR4767492}.
For \eqref{eq:funct1}, we note that if $(k_j,n_j)\neq (0,2)$, then we can bound the corresponding exponentials simply by one and the $\tilde V_t^{k_j,n_j}$-term can be bounded by $h_t^{k_j,n_j}$ from Proposition \ref{pr:vtnorm}. This also works if $(k_j,n_j)=(0,2)$ and $\sigma_{\sum_{q=1}^j n_q}=\sigma_{\sum_{q=1}^j n_q-1}$. If on the other hand we have $(k_j,n_j)=(0,2)$ and $\sigma_{\sum_{q=1}^j n_q}\neq \sigma_{\sum_{q=1}^j n_q-1}$, then we make use of the bound (for $x,y\in K$ with $K\subset \R^2$ compact)
\begin{align}
|e^{i\sqrt{\beta}(\Phi_{t,\infty}(x)-\Phi_{t,\infty}(y))}-1|\leq |x-y| \|\nabla \Phi_{t,\infty}\|_{L^\infty(K)}.
\end{align}
Thus the question boils down to bounding objects of the form
\begin{align}
\prod_{j=1}^p \tilde V_t^{0,2}(\xi_{2j-1},\xi_{2j}|m,\epsilon,\delta)|x_{2j-1}-x_{2j}|\E\qa{\|\nabla \Phi_{t,\infty}\|_{L_\infty(K)}^p}.
\end{align}
From \eqref{eq:V2explicit1} (and \cite[Lemma 4.4]{MR4767492}), one readily finds a bound of the form 
\begin{align}
|\tilde V_t^{0,2}(\xi_{2j-1},\xi_{2j}|m,\epsilon,\delta)|\leq C_t |x_1-x_2|^{-\frac{\beta}{2\pi}},
\end{align}
with $C_t$ independent of $m,\epsilon,\delta$, so due to the $|x_{2j-1}-x_{2j}|$ factor, we see that our quantity is locally integrable for $\beta<6\pi$. We only need to argue that the remaining expectation is bounded in $m$. This follows for example from \cite[Lemma 5.2]{MR4767492}. This concludes the proof. 
\end{proof}

\section{Regularity of solutions to the full plane twisted Dirac equation}
\label{app:SMJ}

In this appendix, we derive some regularity estimates for certain singular solutions (not in $L^2(\C)$)
to the twisted massive Dirac equation in the full plane -- referred to as wave functions in \cite{MR1233355}.
These estimates are important for estimates that are locally uniform in the branch points in Section~\ref{sec:Palmer}.
They are likely to be clear for experts in \cite{MR555666,MR1233355}, but for readers not familiar with these articles, we try to provide a review of the relevant concepts and results needed to derive these estimates.
For the review of the analysis of \cite{MR555666}, we follow the presentation of \cite[Section~6]{MR695532} (with minor notational changes),
where the results of \cite{MR555666} are reviewed in a concise manner.

The precise statement we wish to prove in this appendix is as follows.
The functions $W_j$ are the ones defined at the end of \cite[Section 3]{MR1233355}, see also \eqref{eq:Wnu} where they are denoted $W_\nu$.

\begin{lemma}\label{le:wjreg}
Let $K\subset \{(x_1,\dots,x_n)\in \C^n: x_i\neq x_j \text{ for } i\neq j\}$ be compact, let $m>0$, and let $\alpha_1,\dots,\alpha_n\in(-\frac{1}{2},\frac{1}{2})$ satisfy $\sum_{j=1}^n \alpha_j=0$. There exist constants $C_K,R_K>0$ (depending only on $K$) such that for $j=1,\dots,n$, $W_j(z)=W_j(z,\alpha,x)$ satisfies 
\begin{equation}\label{e:wjreg}
|W_j(z)|\leq C_K e^{-m|z|}
\end{equation}
for $x\in K$ and $|z|>R_K$.
\end{lemma}

\subsection{A review of some of the results of \cite{MR555666} following \cite{MR695532}} To introduce the main objects needed for proving Lemma \ref{le:wjreg}, we review briefly some central ideas from \cite{MR555666} following the presentation of \cite[Section 6]{MR695532}. Compared to our treatment, the conventions in \cite{MR555666,MR695532} are slightly different. In particular, they consider the Dirac operator 
\begin{equation}
\begin{pmatrix}
-\msmj &  \partial\\
\bar\partial & -\msmj
\end{pmatrix}
\end{equation}
instead of 
\begin{equation}
\begin{pmatrix}
\frac{m}{2} & -\partial\\
-\bar\partial & \frac{m}{2}
\end{pmatrix}
\quad \text{ or }\quad
\begin{pmatrix}
\mu & 2\bar \partial\\
2\partial & \mu
\end{pmatrix}
\end{equation}
as Palmer does in  \cite{MR555666} respectively as  we do in Section~\ref{sec:Green} and \cite{MR4767492},
see in particular Section~\ref{sec:Palmer-notation}. We choose to keep separate notation for all of the different conventions (i.e., to write $\msmj, m, \mu$).
The difference between $\partial$ and $2\bar\partial$ (along with the conjugates) can be dealt with by rescaling the spatial variables and interchanging the roles of $z$ and $\bar z$ (the latter being equivalent to choosing a different representation for the $\gamma$-matrices).
The difference in the sign of the mass term on the other hand is irrelevant by the symmetries of the Dirac operator.
Note that in the normalization of  \cite{MR555666}  the rate in \eqref{e:wjreg}  corresponds to the rate $\msmj$ in the conventions of \cite{MR555666,MR695532}. We will thus prove Lemma \ref{le:wjreg} in the setting of \cite{MR555666,MR695532} (and with a rate $2\msmj$).  %

Also the notation for the monodromy data or branching structure of the solutions is indexed slightly differently in \cite{MR555666,MR695532} than in \cite{MR1233355} or in our treatment in Sections~\ref{sec:massless-bosonization}--\ref{sec:Palmer};
we will come back to this issue when describing the results of \cite{MR1233355} concerning the existence of $W_j$.
The setting of \cite{MR695532} is as follows: we consider given distinct points (branch points) $a_1,\dots,a_n\in \C$ and constants (monodromy data) $l_1,\dots,l_n\in (-\frac{1}{2},\frac{1}{2})\setminus \{0\}$.
The restriction to $l_i\neq 0$ is not very important, but simplifies our discussion slightly -- see \cite{MR695532} for a more general setting.
(For the relation of the $l_j$ to the $\lambda_j$ in the convention of \cite{MR555666,MR695532} that we also use,
see \eqref{eq:lambda-l} below.)
We write $a=(a_1,\dots,a_n)$ and $l=(l_1,\dots,l_n)$. For $\msmj>0$, we are interested in multivalued solutions $w=\begin{pmatrix} w_+\\ w_-\end{pmatrix}$ to the Dirac equation on $\C\setminus \{a_1,\dots,a_n\}$:
\begin{equation}\label{eq:deq}
\begin{pmatrix}
-\msmj &  \partial\\
\bar\partial & -\msmj
\end{pmatrix}\begin{pmatrix}
w_+\\
w_-
\end{pmatrix}=0,
\end{equation}
which satisfy, for each $j=1,\dots,n$ and $|z-a_j|$ small enough,
\begin{equation}\label{eq:monod}
\lim_{\theta\to 2\pi^-}w(a_j+e^{i\theta}(z-a_j))=e^{2\pi i(l_j-\frac{1}{2})}w(z).
\end{equation}
\begin{remark}\label{rem:universalcover}
Specifying multivaluedness only locally near the branch points in \eqref{eq:monod} may feel a bit imprecise. There are various ways to make this more concrete. For example, as we do for the Green's function in Section \ref{sec:Green}, one could specify branch cuts, require existence of boundary values on these branch cuts, and that the boundary values are related by \eqref{eq:monod}.

Another, slightly more abstract, but also in many ways more useful, approach to this is taken in  \cite{MR555666}, where the equation is studied on the (simply connected) universal cover of $\C\setminus\{a_1,\dots,a_n\}$. This has the benefit of allowing the use of (local) elliptic regularity arguments for the solution of the equation, and one can readily prove for example that a solution must be a smooth function on the universal cover. The connection to elliptic equations comes from the fact that if we write out \eqref{eq:deq} in component form, we have from the first component $w_+=\frac{1}{\msmj}\partial w_-$ and plugging this into the second component,
\begin{equation}
0=\partial w_+-\msmj w_-=\frac{1}{4\msmj}\Delta w_--\msmj w_-,
\end{equation}
so $w_-$ satisfies a constant coefficient elliptic equation on the universal cover.

The precise sense in which this multivaluedness is realized does not matter much for our discussion,
though we will occasionally make use of smoothness of the solutions in which case we refer to elliptic regularity and the interpretation involving the universal cover. 
\end{remark}

While $W_j$ may not be in $L^2(\C)$, nevertheless, a fundamental building block in understanding solutions to this type of equations is to understand solutions to \eqref{eq:deq} and \eqref{eq:monod} that are globally in $L^2(\C)$. As we will see, $W_j$ can then be built from such solutions via differentiation. Note that due to \eqref{eq:monod}, $|w|^2=|w_+|^2+|w_-|^2$ is single valued and $\int_{\C}|w|^2$ is well defined. We introduce some notation for this $L^2$-space.
\begin{definition}
For $a=(a_1,\dots,a_n)\in \C^n$ (with $a_i$ distinct) and $l=(l_1,\dots,l_n)\in ((-\frac{1}{2},\frac{1}{2})\setminus \{0\})^n$, we write $W_a^l$ for the complex vector space of multivalued solutions to \eqref{eq:deq} and \eqref{eq:monod} which satisfy 
\begin{equation}
\|w\|_{L^2(\C)}^2=\int_{\C}(|w_+|^2+|w_-|^2)<\infty.
\end{equation}
\end{definition}

An important part of the analysis of \cite{MR555666} (and \cite{MR1233355,MR695532}) is that $W_a^l$ is $n$-dimensional, and there is a certain \emph{canonical} basis for it. To introduce this canonical basis, we need to introduce certain special multivalued functions. For $l\in \R\setminus (\Z\cup (\Z+\frac{1}{2}))$ and $z=|z|e^{i\theta}\in \C$, we (modify the notation of \cite{MR695532} slightly and)
define\footnote{Note these functions are essentially as the same ones as in \eqref{eq:Wnu}.
We use slightly different notation to be locally consistent with the references used, i.e., \cite{MR1233355} in Section~\ref{sec:Palmer}
and \cite{MR695532} in this appendix.}
\begin{align}
v_l(z)=\begin{pmatrix}
e^{il\theta}I_l(2\msmj |z|)\\
e^{i(l+1)\theta} I_{l+1}(2\msmj |z|)
\end{pmatrix} \qquad \text{and} \qquad \bar v_l(z)=\begin{pmatrix}
e^{-i(l+1)\theta}I_{l+1}(2\msmj |z|)\\
e^{-il\theta}I_l(2\msmj |z|)
\end{pmatrix},
\end{align}
where $I_l$ is a modified Bessel function of the first kind.
The next lemma records some properties of these functions. 

\begin{lemma}\label{le:vlprops}
We have
\begin{align}\label{eq:vlpart}
\partial v_l=\msmj v_{l-1}, \quad \partial \bar v_l=\msmj \bar v_{l+1}, \quad \bar\partial v_l=\msmj v_{l+1}, \quad \bar\partial \bar v_l=\msmj \bar v_{l-1}.
\end{align}
Moreover, as $z\to 0$
\begin{align}\label{eq:vasy}
v_l(z)=(1+o(1))\frac{1}{\Gamma(1+l)}(\msmj z)^{l}\begin{pmatrix}
1\\
\frac{1}{1+l}\msmj z 
\end{pmatrix}
\end{align}
and 
\begin{align}\label{eq:vbarasy}
\bar v_l(z)=(1+o(1))\frac{1}{\Gamma(1+l)}(\msmj \bar z)^{l}\begin{pmatrix}
\frac{1}{1+l}\msmj \bar z \\
1
\end{pmatrix}.
\end{align}
\end{lemma}

\begin{proof}
Expressing the complex derivatives in polar coordinates and using that the Bessel functions satisfy $I_l'(r)=\frac{1}{2}(I_{l-1}(r)+I_{l+1}(r))$ and $\frac{l}{r}I_l(r)=\frac{1}{2}(I_{l-1}(r)-I_{l+1}(r))$, one readily checks how the operators $\partial,\bar\partial$ act on $v_l$ and $\bar v_l$. 
For further details on this, see  \cite[Proposition~3.1.2]{MR555666}
(though note the difference in notation: our $v_l(z)$ and $\bar v_l(z)$ correspond to $w_{l+\frac{1}{2}}(z,\bar z)$ and $w^*_{l+\frac{1}{2}}(z,\bar z)$ in \cite{MR555666} and $z=\frac{1}{2}re^{i\theta}$ in \cite{MR555666}).

The small $|z|$ asymptotics of $v_l$ and $\bar v_l$
follow readily from the small $r$ asymptotics of $I_l(r)$,
namely $I_l(r)=(1+o(1))\frac{1}{\Gamma(1+l)}(r/2)^l$.
\end{proof}

One can readily check from \eqref{eq:vlpart} that $v_l$ and $\bar v_l$ satisfy \eqref{eq:deq} when we only have one branch point at the origin.
Indeed, one should think of these functions as a basis of multivalued local solutions to the Dirac equation. The precise statement (which is \cite[Theorem~6.0]{MR695532}) is as follows.

\begin{theorem}\label{th:localexp}
Let $w\in W_a^l$. Then for each $j=1,\dots,n$, there exist constants $(c_j(k,\pm))_{k=0}^\infty$ (depending on $a$ and $l$) such that in a small enough punctured neighborhood of $a_j$, we have the convergent expansion (actually, locally uniformly, and absolutely convergent due to regularity of $w$)
\begin{equation}\label{eq:localexp}
w(z)=\sum_{k=0}^\infty (c_j(k,+) v_{k+l_j-\frac{1}{2}}(z-a_j)+c_j(k,-) \bar v_{k-l_j-\frac{1}{2}}(z-a_j)).
\end{equation}
Moreover, if $w'$ is another element of $W_a^l$ (with expansion coefficients $c_j'(k,\pm)$), we have 
\begin{equation}\label{eq:l2ip}
\langle w,w'\rangle_{L^2(\C)}=\int_{\C}(w_+\bar w_+'+w_-\bar w_-')=-\frac{2}{\msmj ^2}\sum_{j=1}^n c_j(0,+)\overline{c_j'(0,-)}\cos(\pi l_j).
\end{equation}
Finally, there exist some $C,R>0$ (depending on $w,a,l$) such that for $|z|>R$, 
\begin{align}\label{eq:wexpdecay}
|w(z)|\leq  \frac{C}{\sqrt{|z|}}e^{-2\msmj |z|}.
\end{align}
\end{theorem}

\begin{remark}
To be precise, \eqref{eq:wexpdecay} is not explicitly stated in \cite[Theorem 6.0]{MR695532}. Instead, they argue the existence of an expansion in terms of modified Bessel functions of the second kind that holds outside of a large enough disk.
It is however a special case of \cite[Proposition 1.2]{MR1233355}, where it is pointed out that this is a consequence of \cite[Proposition 3.1.5]{MR555666}. 
\end{remark}

Note that \eqref{eq:localexp} and \eqref{eq:l2ip} imply in particular that if there exists a set $I\subset\{1,\dots,n\}$ such that $c_j(0,+)=0$ for $j\in I$ and $c_j(0,-)=0$ for $j\notin I$, then $w=0$. This suggests  that $c_j(0,\pm)$ play an important role in the analysis of such functions. We point out some basic properties of these quantities and introduce notation for them.

First of all, note that the coefficients $c_j(k,\pm)$ can be extracted from $w$ through a Fourier expansion. In particular, for $k=0$, one has for small enough $r>0$ 
\begin{align}\label{eq:coefint}
\int_0^{2\pi}\frac{d\theta}{2\pi} e^{-i(l_j-\frac{1}{2})\theta}w_+(a_j+re^{i\theta})=c_j(0,+)I_{l_j-\frac{1}{2}}(2\msmj r)+c_j(0,-)I_{-l_j+\frac{1}{2}}(2\msmj r),
\end{align}
so using asymptotics of $I_\nu(2r)$ as $r\to 0$ (namely $I_\nu(2r)=(1+o(1))\frac{1}{\Gamma(1+\nu)}r^\nu$), we find
\begin{align}
c_j(0,+)=\lim_{r\to 0}\frac{1}{I_{l_j-\frac{1}{2}}(2\msmj r)}\int_0^{2\pi}\frac{d\theta}{2\pi} e^{-i(l_j-\frac{1}{2})\theta}w_+(a_j+re^{i\theta}).
\end{align}
Therefore $w\mapsto (c_j(0,+))_{j=1}^n$ is a linear mapping from $W_a^l\to \C^n$ and by \eqref{eq:l2ip}, it is an injection. Thus the dimension of $W_a^l$ is at most $n$.
In \cite{MR555666}, it is proved that it is $n$-dimensional by constructing a particular basis for the space (the canonical basis mentioned earlier). The precise statement (see e.g. \cite[Theorem 6.1]{MR695532}) is as follows.
\begin{theorem}\label{th:canonical}
There exists a basis $(\mathsf w_j)_{j=1}^n$ for $W_a^l$ which is uniquely characterized by
\begin{equation}
  c_k^j(0,+)=
  \delta_{jk},
\end{equation}
where $c_k^j(0,+)$ is the expansion coefficient $c_k(0,+)$ for $\mathsf w_j$.
\end{theorem}

Given the importance of these expansion coefficients $c_j(0,\pm)$, we introduce some notation for them. To analyze $W_j$ in Lemma \ref{le:wjreg}, it is  convenient to extend the definition from $W_a^l$ to solutions that can be so singular at $a_j$ that they are not square integrable.
\begin{definition}\label{def:cj}
Let $w$ be a solution to \eqref{eq:deq} and \eqref{eq:monod} which for some nonnegative integer $p$ has for each $j=1,\dots,n$ the local expansion
\begin{align}\label{eq:localexp2}
w(z)=\sum_{k=-p}^\infty (c_j(k,+) v_{k+l_j-\frac{1}{2}}(z-a_j)+c_j(k,-) \bar v_{k-l_j-\frac{1}{2}}(z-a_j)).
\end{align}
Then define for each $j=1,\dots,n$ the linear mappings $w\mapsto \mathsf c_j(w,+)=c_j(0,+)$ and $w\mapsto \mathsf c_j(w,-)=c_j(0,-)$ where $c_j(0,\pm)$ are the expansion coefficients from \eqref{eq:localexp2}. 

\end{definition}

For the proof of Lemma~\ref{le:wjreg}, we will need the fact that the canonical basis depends nicely on $a_1,\dots,a_n$. The following is \cite[Corollary 3.3.11]{MR555666}. 
\begin{theorem}\label{th:wjreg}
The the function $(z,a)\mapsto (\mathsf w_j(z))_{j=1}^n$ is a continuous function of $z$ and $a_1,\dots,a_n$ as long as none of the points coincide ($z\neq a_i$ for all $i=1,\dots,n$ and $a_i\neq a_j$ for $i\neq j$). 

More precisely, for each $a^0\in \C^n$ with $a^0_i\neq a^0_j$ for $i\neq j$, there exists a neighborhood $U_0\subset \C^n$ of $a^0$ such that $(z,a)\mapsto (\mathsf w_j(z))_{j=1}^n$ is continuous on the universal cover of $\{(z, a)\in \C^{n+1}:  a\in U_0, \, z\neq  a_j \quad \text{for} \quad j=1,\dots,n\}$. 
\end{theorem}
\begin{remark}
\cite[Corollary~3.3.11]{MR555666} is actually a stronger statement. It says that the canonical basis $(\mathsf w_j)_{j=1}^n$ is analytic in the variables $a_1,\dots.,a_n,\bar a_1,\dots,\bar a_n$ (and $z,\bar z$) as long as the points do not coincide. Again, the precise statement means analyticity on a suitable universal cover and analyticity in variables $z,\bar z$ (or $a,\bar a$) means existence of a convergent expansion in powers of $z$ and $\bar z$. We refer to the discussion preceding \cite[Corollary~3.3.11]{MR555666}  and the proof of \cite[Corollary~3.3.11]{MR555666} for details. Nevertheless, such analyticity implies continuity, and continuity is all we need. 
\end{remark}

Having reviewed the main results of \cite{MR555666} that we need, we turn to results from \cite{MR1233355} concerning the existence of $W_j$.

\subsection{Existence of the wave functions $W_j$ following \cite{MR1233355}}

As mentioned, in \cite{MR1233355}, the setting is slightly different from \cite{MR555666}. One fact that is needed in \cite{MR1233355} is the existence, for given $a_1,\dots,a_n\in\C$ (distinct) and $\lambda_1,\dots,\lambda_n\in (-\frac{1}{2},\frac{1}{2})\setminus \{0\}$ multivalued solutions $w$ to \eqref{eq:deq} with \eqref{eq:monod} replaced by (for small enough $|z-a_j|$) 
\begin{equation}\label{eq:pmono}
\lim_{\theta\to 2\pi^-}w(a_j+e^{i\theta}(z-a_j))=e^{-2\pi i \lambda_j}w(z).
\end{equation}
Compared to the definition of $W_a^l$, the main difference is however that the solutions needed in \cite{MR1233355} may not be in $L^2$ near some $a_j$. Instead one searches for solutions $(W_j)_{j=1}^n$ (that are called wave functions in \cite{MR1233355}) that have a local expansion (for $|z-a_k|$ small enough) 
\begin{align}\label{eq:palmerasy}
W_j(z)=\delta_{j,k}v_{-1-\lambda_k}(z-a_k)+\sum_{p=1}^\infty (\alpha_{p}^{j,k}v_{-1+p-\lambda_k}(z-a_k)+\beta_p^{j,k}\bar v_{-1+p+\lambda_k}(z-a_k))
\end{align}
for some constants $\alpha_p^{j,k},\beta_p^{j,k}\in \C$. Note that Lemma \ref{le:vlprops}, such a $W_j$ would not be locally in $L^2$ in a neighborhood of $a_j$ if $\lambda_j>0$. Thus these are not necessarily elements of $W_a^l$.

The following fact is argued in \cite{MR1233355} (using results of \cite{MR555666}), but we will state (and prove) it as a separate result here as we will need some of the details of the proof for Lemma \ref{le:wjreg}. 
\begin{lemma}\label{le:palmexist}
For given $\lambda_1,\dots,\lambda_n\in (-\frac{1}{2},\frac{1}{2})\setminus \{0\}$ and distinct $a_1,\dots,a_n$, there exist unique multivalued solutions to \eqref{eq:deq},  $W_1,\dots,W_n$, satisfying \eqref{eq:pmono}, \eqref{eq:palmerasy}, and which are square integrable in a neighborhood of infinity.  
\end{lemma}

\begin{proof}
The sign of $\lambda_j$ will play a role in our discussion.
We begin by defining the sets $\mathbf P=\{j\in \{1,\dots,n\}: \lambda_j>0\}$ and $\mathbf N=\{j\in \{1,\dots,n\}: \lambda_j<0\}$. To transform \eqref{eq:pmono} into the form \eqref{eq:monod} (to make use of the formalism of \cite{MR555666,MR695532}), we set
\begin{equation} \label{eq:lambda-l}
l_j=\begin{cases}
-(\lambda_j+\frac{1}{2}), & j\in \mathbf N\\
-(\lambda_j-\frac{1}{2}), & j\in \mathbf P
\end{cases}.
\end{equation}
Note that $l_1,\dots,l_n\in (-\frac{1}{2},\frac{1}{2})\setminus \{0\}$, and $l_j<0$ for $j\in \mathbf N$ and $l_j>0$ for $j\in \mathbf P$.

In terms of $l$, the local expansion at $a_k$ becomes
\begin{equation}\label{eq:lexp}
W_j(z)=\begin{cases}
\delta_{jk}v_{-\frac{1}{2}+l_k}(z-a_k)+\sum_{p=0}^\infty (\alpha_{p+1}^{j,k}v_{p+l_k+\frac{1}{2}}(z-a_k)+\beta_{p+1}^{j,k}\bar v_{p-\frac{1}{2}-l_k}(z-a_k)), & k\in \mathbf N\\
\delta_{jk}v_{-\frac{3}{2}+l_k}(z-a_k)+\sum_{p=0}^\infty (\alpha_{p+1}^{j,k}v_{p-\frac{1}{2}+l_k}(z-a_k)+\beta_{p+1}^{j,k}\bar v_{p-l_k+\frac{1}{2}}(z-a_k)), & k\in \mathbf P
\end{cases}.
\end{equation}

We see that for $k\in \mathbf N$, this expansion is of the from \eqref{eq:localexp}, but for $k\in \mathbf P$, the quantity $v_{-\frac{3}{2}+l_k}$ is not allowed in \eqref{eq:localexp}.
We note however that if $j\in \mathbf N$, due to the $\delta_{jk}$-term, the $v_{-\frac{3}{2}+l_k}$-term is not present. Thus we should expect that for $j\in \mathbf N$, $W_j\in W_a^l$. For $j\in \mathbf P$, we will shortly see that we can however express $W_j$ in terms of functions in $W_a^l$ and suitable derivatives of them. Let us first however show that for each $j\in \mathbf N$, there exists a unique function $W_j$ with local expansion \eqref{eq:lexp}.

\subproofx{The case $j\in \mathbf N$}
To reiterate, we wish to show that there is a unique element $W_j\in W_a^l$ which has the local expansion
\begin{equation}\label{eq:lexp-}
W_j(z)=\begin{cases}
\delta_{jk}v_{-\frac{1}{2}+l_k}(z-a_k)+\sum_{p=0}^\infty (\alpha_{p+1}^{j,k}v_{p+l_k+\frac{1}{2}}(z-a_k)+\beta_{p+1}^{j,k}\bar v_{p-\frac{1}{2}-l_k}(z-a_k)), & k\in \mathbf N\\
\sum_{p=0}^\infty (\alpha_{p+1}^{j,k}v_{p-\frac{1}{2}+l_k}(z-a_k)+\beta_{p+1}^{j,k}\bar v_{p-l_k+\frac{1}{2}}(z-a_k)), & k\in \mathbf P
\end{cases}.
\end{equation}
Comparing with Definition \ref{def:cj}, we see that if such a $W_j$ would exist, it would have to satisfy  
\begin{align}
\begin{cases}
\mathsf c_k(W_j,+)=\delta_{jk}, & k\in \mathbf N\\
\mathsf c_k(W_j,-)=0, & k\in \mathbf P
\end{cases}.
\end{align}
The second constraint comes from the fact there there is no $\bar v_{-l_k-\frac{1}{2}}$-term in the expansion \eqref{eq:lexp-} for $k\in \mathbf P$. The claim is that there is a unique element of $W_a^l$ that satisfies this condition. To see why this is the case, consider the linear mapping $w\mapsto ((\mathsf c_{k}(w,+))_{k\in \mathbf N},(\mathsf c_k(w,-))_{k\in \mathbf P})$ from $W_a^l$ to $\C^n$. By \eqref{eq:l2ip}, this linear mapping is an injection (any element in the kernel has zero $L^2$-norm). Thus as the spaces have the same dimension, this mapping must be a bijection. The element $W_j$ is thus the unique one that maps to $((\delta_{jk})_{k\in \mathbf N},(0)_{k\in \mathbf P})$. Note that for $j\in \mathbf N$, any $W_j$ that satisfies \eqref{eq:deq}, \eqref{eq:pmono}, \eqref{eq:palmerasy} and is square integrable at infinity is in $W_a^l$
(by elliptic regularity away from $a_1,\dots,a_n,\infty$).
Thus we have argued the existence of the unique $W_j$ we are after for $j\in \mathbf N$. 

We will later on need some regularity estimates on $W_j$ in terms of $a$, so we will be slightly more explicit about what $W_j$ looks like in the basis $(\mathsf w_k)_{k=1}^n$. By definition of this basis and the local expansion of $W_j$, we see that we must be able to write 
\begin{align}\label{eq:jnegbas}
W_j=\mathsf w_j+\sum_{k\in \mathbf P}\gamma_k \mathsf w_k
\end{align}
for some complex constants $(\gamma_k)_{k\in \mathbf P}$. Indeed, these constants are the unique solution to 
\begin{align}\label{eq:jnegbas2}
\sum_{k'\in \mathbf P}\mathsf c_{k}(\mathsf w_{k'},-)\gamma_{k'}=-\mathsf c_k(\mathsf w_j,-)
\end{align}
This linear system of equations has a unique solution (by the existence and uniqueness of $W_j$ which we already argued) which means that the matrix $(\mathsf c_k(\mathsf w_{k'},-))_{k,k'\in \mathbf P}$ is invertible. We will later make use of the fact that $\mathsf c_k(\mathsf w_{k'},-)$ and $\mathsf c_k(\mathsf w_j,-)$ are continuous functions of $a$ (as long as the entries are distinct). Since the matrix is invertible, this will mean that also $\gamma_k$ are continuous in $a$. This will provide the regularity of $W_j$ that we need for Lemma \ref{le:wjreg}. But as mentioned, we will go into details later.

\subproofx{The case $j\in \mathbf P$}
Let us start by writing down the local expansions that $W_j$ would have if it existed. We would have 
\begin{align}
W_j(z)=\begin{cases}
\sum_{p=0}^\infty (\alpha_{p+1}^{j,k}v_{p+l_k+\frac{1}{2}}(z-a_k)+\beta_{p+1}^{j,k}\bar v_{p-\frac{1}{2}-l_k}(z-a_k)), & k\in \mathbf N\\
\delta_{jk}v_{-\frac{3}{2}+l_k}(z-a_k)+\sum_{p=0}^\infty (\alpha_{p+1}^{j,k}v_{p-\frac{1}{2}+l_k}(z-a_k)+\beta_{p+1}^{j,k}\bar v_{p-l_k+\frac{1}{2}}(z-a_k)), & k\in \mathbf P
\end{cases}.
\end{align}
Recalling our notation from Definition \ref{def:cj}, we see that if $W_j$ existed, it would have to satisfy 
\begin{equation}\label{eq:ckWj}
\begin{cases}
\mathsf c_k(W_j,+)=0, & k\in \mathbf N\\
\mathsf c_k(W_j,-)=0, & k\in \mathbf P
\end{cases}.
\end{equation}
Note that as $W_j$ is not assumed to be in $W_a^l$, \eqref{eq:l2ip} does not imply that it would have to be zero.

We would however like to make use of $W_a^l$ (and \eqref{eq:l2ip}) somehow to argue existence of $W_j$. The trick for this is that by Lemma \ref{le:vlprops}, $\partial v_l=\msmj v_{l-1}$. Thus one finds using the local expansions (and elliptic regularity and the behavior at infinity) that if $W_j$ existed, then we would have to have
\begin{align}
\tilde w_j:=W_j-\frac{1}{\msmj}\partial \mathsf w_j\in W_a^l
\end{align}
Moreover, from \eqref{eq:ckWj}, we would have to have 
\begin{align}
\begin{cases}
\mathsf c_k(\tilde w_j,+)=-\frac{1}{\msmj}\mathsf c_k(\partial \mathsf w_j,+), & k\in \mathbf N\\
\mathsf c_k(\tilde w_j,-)=-\frac{1}{\msmj}\mathsf c_k(\partial \mathsf w_j,-), & k\in \mathbf P
\end{cases}.
\end{align}
Once again, the mapping $w\mapsto ((\mathsf c_k(w,+))_{k\in \mathbf N},(\mathsf c_k(w,-))_{k\in \mathbf P})$ is a bijection from $W_a^l$ to $\C^n$ (the dimensions of the spaces match and the mapping is an injection by \eqref{eq:l2ip}). This means that $\tilde w_j$ is the unique element of $W_a^l$ that maps to $-\frac{1}{\msmj}((\mathsf c_k(\partial \mathsf w_j,+))_{k\in \mathbf N},(\mathsf c_k(\partial \mathsf w_j,-))_{k\in \mathbf P})$. We then define 
\begin{align}
W_j=\frac{1}{\msmj}\partial \mathsf w_j+\tilde w_j.
\end{align}
One readily checks that this $W_j$ satisfies all of the desired properties (for example, using that $\partial \mathsf w_j$ is also a solution to \eqref{eq:deq}). Uniqueness follows from the fact that the difference (say $\Delta$) of two such functions would be in $W_a^l$ with $\mathsf c_k(\Delta,+)=0$ for $k\in \mathbf N$ and $\mathsf c_k(\Delta,-)=0$ for $k\in \mathbf P$. Again, by \eqref{eq:l2ip}, this would mean that $\Delta=0$.

As for $j\in \mathbf N$, we will need to know something more precise about how to express $W_j$ in terms of $(\mathsf w_k)_{k=1}^n$. For this purpose, we express $\tilde w_j$ in the basis $(\mathsf w_k)$. Using the definition of the basis, we see that we can write 
\begin{align}
\tilde w_j=-\frac{1}{\msmj}\sum_{k\in \mathbf N}\mathsf c_k(\partial \mathsf w_j,+)\mathsf w_k+\sum_{k\in \mathbf P}\gamma_k \mathsf w_k
\end{align}
for some constants $\gamma_k$. Indeed, these constants are uniquely specified by the system of equations 
\begin{align}\label{eq:gammapos}
\sum_{k'\in \mathbf P}\gamma_{k'}\mathsf c_k(\mathsf w_{k'},-)=\frac{1}{\msmj}\sum_{k''\in \mathbf N}\mathsf c_{k''}(\partial \mathsf w_j,+)\mathsf c_k(\mathsf w_{k''},-).
\end{align}
The matrix on the left is invertible (as we argued in the case where $j\in \mathbf N$), so the coefficients $\gamma_{k'}$ can be solved uniquely from this system of equations. As we will argue later on, it follows from this that the $\gamma_{k'}$ are continuous in $a$.
\end{proof}

We can finally begin the proof of Lemma \ref{le:wjreg}.

\subsection{Regularity properties of the wave functions -- proof of Lemma \ref{le:wjreg}}

\begin{proof}[Proof of Lemma \ref{le:wjreg}]
We will prove that given $K\subset \{a\in \C^n: a_i\neq a_j \text{ for } i\neq j\}$ compact, there exists a $R_K$ such that $(z,a)\mapsto |W_j(z)|e^{2\msmj |z|}$ is a bounded continuous function on $\{(z,a)\in \C^{n+1}: |z|>R \text{ and } a\in K\}$. In fact, making use of an argument based on compactness, it is slightly more convenient to allow $z$ to take the value $\infty$ as well. So it is sufficient for us to prove that this is a continuous function on the compact set $\{(z,a)\in (\C\cup \{\infty\})\times \C^{n}: |z|>R \text{ and } a\in K\}$.

Given $K$, let us now take some $R_K>\sup_{a\in K}\max_{1\leq i\leq n}|a_i|$. Then we see from \eqref{eq:monod} that $z\mapsto |W_j(z)|$ is actually a single-valued function (since $\lambda_j\in \R$) and from the regularity on the universal cover, we see that it is continuous for $R_K<|z|<\infty$. Moreover, we see from \eqref{eq:wexpdecay}, that this function is also continuous at $\infty$. It thus remains to prove continuity in $a$.

For this purpose, we express $W_j$ in terms of the canonical basis as discussed in the proof of Lemma \ref{le:palmexist}. We know from Theorem \ref{th:wjreg} that the canonical basis is continuous in $a_1,\dots,a_n$, so it is enough for us to show that the expansion coefficients are continuous. We will make use of the notation from Lemma \ref{le:palmexist}.

We first look at the case $j\in \mathbf N$. From \eqref{eq:jnegbas}, we have the expression 
\begin{align}
W_j=\mathsf w_j+\sum_{k\in \mathbf P}\gamma_k\mathsf w_k
\end{align}
with (by \eqref{eq:jnegbas2}), $(\gamma_k)_{k\in \mathbf P}$ being the unique solution to 
\begin{align}
\sum_{k'\in \mathbf P}\mathsf c_k(\mathsf w_{k'},-)\gamma_{k'}=-\mathsf c_k(\mathsf w_j,-).
\end{align}
for $k\in \mathbf P$. As the matrix here is invertible for each of the relevant $a$s (we know a unique solution exists), it is enough to prove that the quantities $\mathsf c_k(\mathsf w_{k'},-)$ and $\mathsf c_k(\mathsf w_j,-)$ are continuous. For this, we recall \eqref{eq:coefint} which states that for a solution $w$, we have for small enough $r$
\begin{align}
\int_0^{2\pi}\frac{d\theta}{2\pi} e^{-i(l_k-\frac{1}{2})\theta}\mathsf w_{k',+}(a_k+re^{i\theta})=\mathsf c_k(\mathsf w_{k'},+)I_{l_k-\frac{1}{2}}(2\msmj r)+\mathsf c_k(\mathsf w_{k'},-)I_{-l_k+\frac{1}{2}}(2\msmj r).
\end{align}
To get a hold of continuity, let us first address which $r$ is small enough. For this, it is perhaps most convenient to fix some concrete branch cuts (for example horizontal ones as in Section \ref{sec:intro}). By elliptic regularity, the relevant series expansion is convergent as long as we don't hit any of the branch cuts. Thus if we take $r$ to be half of the smallest distance between the branch cuts (which is a continuous function of $a$), this identity is valid. Thus recalling the definition of the canonical basis, for such a choice of $r$, we have 
\begin{align}
\mathsf c_k(\mathsf w_{k'},-)=\frac{1}{I_{-l_k+\frac{1}{2}}(2\msmj r)}\left(\int_0^{2\pi}\frac{d\theta}{2\pi} e^{-i(l_k-\frac{1}{2})\theta}\mathsf w_{k',+}(a_k+re^{i\theta})-\delta_{kk'}I_{l_k-\frac{1}{2}}(2\msmj r)\right).
\end{align}
Using the regularity of the canonical basis and the choice of $r$, we see that this is a continuous function of $a$. The same argument works for $\mathsf c_k(\mathsf w_j,-)$. We conclude that $W_j$ is continuous for $j\in \mathbf N$.

It remains to consider the case $j\in \mathbf P$. Here we have 
\begin{align}
W_j=\frac{1}{\msmj}\partial \mathsf w_j-\frac{1}{\msmj}\sum_{k\in \mathbf N}\mathsf c_k(\partial \mathsf w_j,+)\mathsf w_k+\sum_{k\in \mathbf P}\tilde \gamma_k \mathsf w_j
\end{align}
where (by \eqref{eq:gammapos}) $(\tilde \gamma_k)_{k\in \mathbf P}$ is the unique solution to 
\begin{equation}
\sum_{k'\in \mathbf P}\tilde \gamma_{k'}\mathsf c_k(\mathsf w_{k'},-)=\frac{1}{\msmj}\sum_{k''\in \mathbf N}\mathsf c_{k''}(\partial \mathsf w_j,+)\mathsf c_k(\mathsf w_{k''},-).
\end{equation}
The task is thus to show that $\mathsf c_k(\partial \mathsf w_j,+)$ and $\mathsf c_k(\mathsf w_{k''},-)$ and the quantities $\tilde \gamma_{k'}$ are continuous functions. The only thing that is slightly different from the $j\in \mathbf N$ case, is that we need to show that $\mathsf c_k(\partial \mathsf w_j,+)$ is continuous.

For this purpose, let us first write in a punctured $r$-neighborhood of $a_k$ the local expansion for $\mathsf w_j$ 
\begin{align}
\mathsf w_j(z)=\sum_{p=0}^\infty (c_{j,k}(p,+)v_{p+l_k-\frac{1}{2}}(z-a_k)+c_{j,k}(p,-)\bar v_{p-l_k-\frac{1}{2}}(z-a_k)),
\end{align}
where $c_{j,k}(0,+)=\delta_{jk}$. 

Recalling from Lemma \ref{le:vlprops} that $\partial v_l=\msmj v_{l-1}$ and $\partial \bar v_l=\msmj\bar v_{l+1}$, so we have the expansion
\begin{align}
\frac{1}{\msmj}\partial\mathsf w_j(z)=\sum_{p=0}^\infty (c_{j,k}(p,+)v_{p+l_k-\frac{3}{2}}(z-a_k)+c_{j,k}(p,-)\bar v_{p-l_k-\frac{3}{2}}(z-a_k)),
\end{align}
again with $c_{j,k}(0,+)=\delta_{jk}$, which means that $\mathsf c_k(\partial \mathsf w_j,-)=\msmj c_{j,k}(1,-)$. The idea is now to solve the coefficients $c_{j,k}(p,\pm)$ using the Fourier expansion.

For the Fourier expansion, note that 
\begin{align}
v_{p+l_k-\frac{1}{2}}(z-a_k)=e^{i(l_k-\frac{1}{2})\theta}e^{ip\theta}\begin{pmatrix}
I_{p+l_k-\frac{1}{2}}(2\msmj |z-a_k|)\\
e^{i\theta}I_{p+1+l_k-\frac{1}{2}}(2\msmj |z-a_k|)
\end{pmatrix} 
\end{align}
\begin{align}
 \bar v_{p-l_k-\frac{1}{2}}(z-a_k)=e^{i(l_k-\frac{1}{2})\theta}e^{-ip\theta}\begin{pmatrix}
I_{p-l_k+\frac{1}{2}}(2\msmj |z-a_k|)\\
e^{i\theta}I_{p-l_k-\frac{1}{2}}(2\msmj |z-a_k|)
\end{pmatrix}
\end{align}
This means in particular that for our choice of $r$, 
\begin{align}
c_{j,k}(1,-)I_{l_k+\frac{1}{2}}(2\msmj r)=\int_0^{2\pi} \frac{d\theta}{2\pi} e^{-i(l_k-\frac{3}{2})\theta}\mathsf w_{j,+}(a_k+re^{i\theta}).
\end{align}
Continuity of this follows again from regularity of $\mathsf w_j$ and our choice of $r$. 

So to summarize, $(z,a)\mapsto |W_j(z)|e^{2\msmj |z|}$ is a continuous function on a compact set, so in particular, it is bounded. As already mentioned, this also implies the corresponding result in the conventions of \cite{MR1233355}. This concludes the proof.
\end{proof}

\section*{Acknowledgements}

R.B. and C.W. thank David Brydges for many helpful discussions on Bosonization in the context of a related project.

R.B. acknowledges funding from NSF grant DMS-2348045.

S.M. was supported by the Knut and Alice Wallenberg Grant KAW 2022.0295.

C.W. was supported by the Academy of Finland through the grant 348452 and ERC grant CONFSTAT funded by the European Union. Views and opinions expressed are however those of the authors only and do not necessarily reflect those of the European Union or ERC. Neither the European Union nor ERC can be held responsible for them.

\bibliography{all}
\bibliographystyle{plain}

\end{document}